\pdfoutput=1 
\documentclass[draft=false, leqno]{article}
\usepackage{amsmath,amsthm,mathrsfs}
\usepackage[normal]{boilerart1}
\usepackage{longtable}
\usepackage{booktabs}

\hypersetup{
    urlcolor = slate,
}

\crefname{section}{Chapter}{Chapters}
\crefname{subsection}{\S\!}{\S\S\!}
\Crefname{section}{Chapter}{Chapters}
\Crefname{subsection}{Section}{Sections}

\theoremstyle{definition}
\newtheorem{rmk}[equation]{Remark}

\newcommand{\interior}{\upiota}
\DeclareMathOperator{\invert}{\upvarepsilon}

\newcommand{\andeq}{\text{\qquad and\qquad}}
\newcommand{\period}{\rlap{\ .}}
\newcommand{\comma}{\rlap{\ ,}}
\newcommand{\semicolon}{\rlap{\ ;}}

\usepackage[style=british]{csquotes}
\newcommand{\defn}[1]{\emph{#1}}
\newcommand{\upperth}{\textsuperscript{th}\,}

\newcommand{\pifinite}{\xspace{$\uppi$-fi\-nite}\xspace}
\newcommand{\pifiniteness}{\xspace{$\uppi$-fi\-nite\-ness}\xspace}

\newcommand{\basechange}{\xspace{base\-change}\xspace}
\newcommand{\Basechange}{\xspace{Base\-change}\xspace}

\DeclareMathOperator{\upt}{t}
\newcommand{\tstructure}{\xspace{$\upt$-struc\-ture}\xspace}
\newcommand{\tstructures}{\xspace{$\upt$-struc\-tures}\xspace}
\newcommand{\texact}{\xspace{$\upt$-exact}\xspace}
\newcommand{\tboundedabove}{\xspace{$\upt$-bound\-ed-above}\xspace}

\newcommand{\proobject}{\xspace{pro\-öbject}\xspace}
\newcommand{\proobjects}{\xspace{pro\-öbjects}\xspace}

\newcommand{\category}{\xspace{$\infty$-cat\-e\-go\-ry}\xspace}
\newcommand{\categories}{\xspace{$\infty$-cat\-e\-gories}\xspace}
\newcommand{\categorical}{\xspace{$\infty$-cat\-e\-gor\-i\-cal}\xspace}
\newcommand{\Categorical}{\xspace{$\infty$-Cat\-e\-gor\-i\-cal}\xspace}
\newcommand{\acategory}{{an $\infty$-cat\-e\-go\-ry}\xspace}
\newcommand{\Acategory}{{An $\infty$-cat\-e\-go\-ry}\xspace}
\newcommand{\acategorical}{{an $\infty$-cat\-e\-gor\-i\-cal}\xspace}

\newcommand{\groupoid}{\xspace{$\infty$-group\-oid}\xspace}
\newcommand{\groupoids}{\xspace{$\infty$-group\-oids}\xspace}
\newcommand{\agroupoid}{{an $\infty$-group\-oid}\xspace}
\newcommand{\Agroupoid}{{An $\infty$-group\-oid}\xspace}

\newcommand{\topos}{\xspace{$\infty$-topos}\xspace}
\newcommand{\topoi}{\xspace{$\infty$-topoi}\xspace}
\newcommand{\toposic}{\xspace{$\infty$-topos\-ic}\xspace}
\newcommand{\atopos}{{an $\infty$-topos}\xspace}
\newcommand{\Atopos}{{An $\infty$-topos}\xspace}

\newcommand{\Topoi}{\xspace{$\infty$-Topoi}\xspace}

\newcommand{\pretopos}{\xspace{$\infty$-pre\-topos}\xspace}
\newcommand{\pretopoi}{\xspace{$\infty$-pre\-topoi}\xspace}

\newcommand{\apretopos}{{an $\infty$-pre\-topos}\xspace}
\newcommand{\Apretopos}{{An $\infty$-pre\-topos}\xspace}

\newcommand{\site}{\xspace{$\infty$-site}\xspace}
\newcommand{\sites}{\xspace{$\infty$-sites}\xspace}
\newcommand{\asite}{{an $\infty$-site}\xspace}
\newcommand{\Asite}{{An $\infty$-site}\xspace}

\newcommand{\presite}{\xspace{$\infty$-pre\-site}\xspace}

\newcommand{\apresite}{{an $\infty$-pre\-site}\xspace}

\newcommand{\ellbar}{\bar{\ell}}

\newcommand{\clowerstar}{c_{\ast}}

\newcommand{\elowerstar}{e_{\ast}}
\newcommand{\eupperstar}{e^{\ast}}
\newcommand{\ilowerstar}{i_{\ast}}
\newcommand{\iupperstar}{i^{\ast}}

\newcommand{\jlowerstar}{j_{\ast}}
\newcommand{\jupperstar}{j^{\ast}}

\newcommand{\fbar}{\bar{f}}
\newcommand{\flowerstar}{f_{\ast}}
\newcommand{\fupperstar}{f^{\ast}}
\newcommand{\fuppershriek}{f^{!}}
\newcommand{\flowershriek}{f_{!}}
\newcommand{\gbar}{\bar{g}}
\newcommand{\glowerstar}{g_{\ast}}
\newcommand{\gupperstar}{g^{\ast}}
\newcommand{\guppershriek}{g^{!}}

\newcommand{\hupperstar}{h^{\ast}}

\newcommand{\plowerstar}{p_{\ast}}
\newcommand{\pupperstar}{p^{\ast}}
\newcommand{\puppershriek}{p^{!}}
\newcommand{\plowershriek}{p_{!}}
\newcommand{\prupperstar}{\pr^{\ast}}
\newcommand{\qlowerstar}{q_{\ast}}
\newcommand{\rlowerstar}{r_{\ast}}
\newcommand{\rupperstar}{r^{\ast}}

\newcommand{\supperstar}{s^{\ast}}

\newcommand{\qupperstar}{q^{\ast}}
\newcommand{\xlowerstar}{x_{\ast}}
\newcommand{\xupperstar}{x^{\ast}}
\newcommand{\xuppershriek}{x^{!}}
\newcommand{\ylowerstar}{y_{\ast}}

\newcommand{\yuppershriek}{y^{!}}
\newcommand{\wlowerstar}{w_{\ast}}
\newcommand{\wupperstar}{w^{\ast}}

\newcommand{\zlowerstar}{z_{\ast}}
\newcommand{\zupperstar}{z^{\ast}}
\newcommand{\zuppershriek}{z^{!}}
\newcommand{\epsilonlowerstar}{\varepsilon_{\ast}}
\newcommand{\epsilonupperstar}{\varepsilon^{\ast}}

\newcommand{\phiupperstar}{\phi^{\ast}}
\newcommand{\philowershriek}{\phi_{!}}

\newcommand{\lambdahat}{\widehat{\uplambda}}
\newcommand{\piupperstar}{\pi^{\ast}}

\newcommand{\xiupperstar}{\xi^{\ast}}

\newcommand{\sigmahat}{\hat{\sigma}}

\newcommand{\Xtilde}{\widetilde{X}}
\newcommand{\Ytilde}{\widetilde{Y}}
\newcommand{\Ztilde}{\widetilde{Z}}

\newcommand{\ctop}{\operatorname{c}}
\DeclareMathOperator{\upC}{C}
\newcommand{\Chat}{\widehat{\upC}}

\newcommand{\ksep}{k^{\sep}}
\newcommand{\kbar}{\overline{k}}

\newcommand{\Gammalowerstar}{\Gammaup_{\ast}}
\newcommand{\Gammaupperstar}{\Gammaup^{\ast}}
\newcommand{\Gammauppershriek}{\Gammaup^{!}}
\newcommand{\Gammalowershriek}{\Gammaup_{!}}

\DeclareMathOperator{\Mup}{M}
\DeclareMathOperator{\Wup}{W}
\newcommand{\overleftW}{\overleftarrow{\Wup}}
\newcommand{\overleftrightW}{\overleftrightarrow{\Wup}}

\DeclareMathOperator{\sep}{sep}
\DeclareMathOperator{\Nis}{nis}
\DeclareMathOperator{\fg}{fg}
\DeclareMathOperator{\bcc}{bcc}
\DeclareMathOperator{\cart}{cart}

\DeclareMathOperator{\cons}{cons}

\DeclareMathOperator{\lc}{lc}

\DeclareMathOperator{\an}{an}
\DeclareMathOperator{\perf}{perf}

\DeclareMathOperator{\zar}{zar}
\DeclareMathOperator{\hyp}{hyp}
\DeclareMathOperator{\post}{post}
\DeclareMathOperator{\bdd}{b}
\DeclareMathOperator{\spec}{spec}
\DeclareMathOperator{\disc}{disc}
\DeclareMathOperator{\indisc}{indisc}
\DeclareMathOperator{\loc}{loc}
\DeclareMathOperator{\Stone}{Stn}
\DeclareMathOperator{\qc}{qc}

\DeclareMathOperator{\coh}{coh}
\DeclareMathOperator{\bc}{bc}
\DeclareMathOperator{\lex}{lex}
\DeclareMathOperator{\ft}{ft}
\DeclareMathOperator{\inv}{inv}
\DeclareMathOperator{\cts}{cts}

\DeclareMathOperator{\sober}{sob}
\DeclareMathOperator{\sh}{sh}
\DeclareMathOperator{\hens}{h}
\DeclareMathOperator{\alg}{alg}

\DeclareMathOperator{\toptextit}{top}
\DeclareMathOperator{\abs}{abs}
\DeclareMathOperator{\norm}{norm}

\DeclareMathOperator{\cc}{cc}
\DeclareMathOperator{\adm}{adm}
\DeclareMathOperator{\und}{und}
\DeclareMathOperator{\pyk}{pyk}
\DeclareMathOperator{\cg}{cg}

\DeclareMathOperator{\num}{ng}

\newcommand{\pt}{\ast}

\DeclareMathOperator{\source}{s}
\DeclareMathOperator{\target}{t}

\newcommand{\intersect}{\cap}
\newcommand{\colonequals}{\coloneq}
\newcommand{\tensor}{\otimes}
\newcommand{\coproduct}{\sqcup}
\newcommand{\cross}{\times}
\newcommand{\crosslimits}{\operatornamewithlimits{\cross}} 
\newcommand{\unionlimits}{\operatornamewithlimits{\cup}} 
\newcommand{\leftadjoint}{\shortlefttack}
\newcommand{\isomorphic}{\cong}
\newcommand{\orientedcup}{\mathbin{\overleftarrow{\cup}}}
\newcommand{\orientedcupbc}{\orientedcup_{\bc}}
\newcommand{\orientedtimes}{\mathbin{\overleftarrow{\times}}}
\newcommand{\of}{\circ}
\newcommand{\equivalent}{\simeq}
\newcommand{\paren}[1]{\left(#1\right)}

\renewcommand{\rhd}{\smalltriangleright}

\newcommand{\horn}[2]{\upLambda_{#1}^{#2}}

\def\reflectedop#1#2{\mathop{\reflectbox{$#1{#2}$}}}

\newcommand{\llbracket}{\lBrack} 
\newcommand{\rrbracket}{\rBrack} 

\newcommand{\unit}{u}
\newcommand{\counit}{c}

\DeclareMathOperator{\LMod}{LMod}

\newcommand{\Subpi}{\Sub_{\uppi}}

\newcommand{\Deltaop}{\DDelta^{\op}}

\newcommand{\Pos}{\categ{Pos}}
\newcommand{\Posfin}{\Pos^{\fin}}
\newcommand{\Setfin}{\Set^{\fin}}
\newcommand{\Preord}{\categ{Pord}}

\newcommand{\Comp}{\categ{Comp}}
\newcommand{\Stn}{\categ{Stn}}
\newcommand{\Proj}{\categ{Proj}}
\newcommand{\Projop}{\Proj^{\op}}

\newcommand{\Sch}{\categ{Sch}}
\newcommand{\Schft}{\Sch^{\ft}}
\newcommand{\SchftCC}{\Schft_{\CCup}}
\newcommand{\Schperf}{\Sch_{\perf}}

\newcommand{\TSpc}{\categ{TSpc}}

\newcommand{\TSpccg}{\TSpc^{\cg}}
\newcommand{\TSpcng}{\TSpc^{\num}}
\newcommand{\TSpcspec}{\TSpc^{\spec}}

\newcommand{\sSet}{\operatorname{s\Set}}

\newcommand{\preTop}{\categ{Pretop}_{\infty}}
\newcommand{\preTopbdd}{\preTop^{\bdd}}
\newcommand{\Topcoh}{\Top_{\infty}^{\coh}}
\newcommand{\Topbc}{\Top_{\infty}^{\bc}}
\newcommand{\Topbcadm}{\Top_{\infty}^{\bc,\adm}}
\newcommand{\Toploc}{\Top_{\infty}^{\loc}}
\newcommand{\Toppt}{\Top_{\infty,\ast}}
\newcommand{\TopStone}{\Top_{\infty}^{\Stone}}
\newcommand{\StrTop}{\categ{StrTop}}
\newcommand{\StrTopcc}{\StrTop^{\cc}}
\newcommand{\StrTopbcc}{\StrTop^{\bcc}}
\newcommand{\StrTopspec}{\StrTop^{\spec}}
\newcommand{\LocCocart}{\categ{LocCocart}}

\newcommand{\profincomp}{_{\uppi}^{\wedge}}
\renewcommand{\Space}{\categ{S}}
\newcommand{\Spaceu}{\Space^{\operatorname{u}}}
\newcommand{\Spacefin}{\Space_{\uppi}}
\newcommand{\Spacefinite}{\Space^{\fin}}
\newcommand{\Spacefinitept}{\Space_{\ast}^{\fin}}
\newcommand{\Spaceprofin}{\Space_{\uppi}^{\wedge}}
\newcommand{\Str}{\categ{Str}}
\newcommand{\Strat}{\categ{Str}}
\newcommand{\StrP}{\Str_{P}}

\newcommand{\Strfin}{\Str_{\uppi}}
\newcommand{\StrfinP}{\Str_{\uppi,P}}
\newcommand{\Stratfin}{\Str_{\uppi}}
\newcommand{\Strprofin}{\Str_{\uppi}^{\wedge}}
\newcommand{\Strprofinadm}{\Str_{\uppi}^{\wedge,\adm}}
\newcommand{\StrprofinS}{\Str_{\uppi,S}^{\wedge}}
\newcommand{\StrprofinP}{\Str_{\uppi,P}^{\wedge}}
\newcommand{\Stratprofin}{\Str_{\uppi}^{\wedge}}

\newcommand{\Lay}{\categ{Lay}}
\newcommand{\Laycart}{\Lay^{\cart}}
\newcommand{\Layfin}{\Lay_{\uppi}}
\renewcommand{\Stab}{\operatorname{Sp}}

\newcommand{\lambdapi}{\uplambda_{\uppi}}

\DeclareMathOperator{\Cons}{Cons}
\newcommand{\Aff}{\categ{Aff}}
\newcommand{\Conset}{\categ{Cons}_{\et}}
\newcommand{\PShet}{\categ{PSh}_{\et}}
\newcommand{\Shet}{\categ{Sh}_{\et}}

\DeclareMathOperator{\Dup}{D}

\renewcommand{\Perf}{\operatorname{Perf}}
\newcommand{\Dcons}{\Dup_{\!\cons}}

\DeclareMathOperator{\Eup}{E}

\DeclareMathOperator{\ho}{h}
\newcommand{\hoSpace}{{\ho_1\kern-0.15em\Space}}

\DeclareMathOperator{\Pyk}{Pyk}
\DeclareMathOperator{\CS}{CS}
\DeclareMathOperator{\CO}{CO}

\DeclareMathOperator{\LC}{LC}

\DeclareMathOperator{\PSh}{PSh}
\DeclareMathOperator{\Sh}{Sh}
\newcommand{\Shhyp}{\Sh^{\!\hyp}}
\newcommand{\Sheff}[1]{\Sh_{\eff}(#1)}
\newcommand{\Sheffhyp}[1]{\Sh_{\eff}^{\!\hyp\!}(#1)}

\DeclareMathOperator{\Vect}{Vect}
\renewcommand{\Rep}{\operatorname{Rep}}

\DeclareMathOperator{\FC}{FC}
\DeclareMathOperator{\UH}{UH}

\DeclareMathOperator{\Pup}{P}

\DeclareMathOperator{\Alex}{Alex}
\DeclareMathOperator{\Special}{S}

\newcommand{\trun}{\uptau}

\newcommand{\exrep}{\uprho}

\newcommand{\xtilde}{\widetilde{x}}
\newcommand{\stilde}{\tilde{s}}
\newcommand{\Ktilde}{\widetilde{K}}
\newcommand{\Stilde}{\widetilde{S}}
\newcommand{\Ttilde}{\widetilde{T}}
\newcommand{\Ptilde}{\widetilde{P}}
\newcommand{\Qtilde}{\widetilde{Q}}
\newcommand{\Mtilde}{\widetilde{M}}
\newcommand{\Utilde}{\widetilde{U}}
\newcommand{\Wtilde}{\widetilde{W}}

\newcommand{\Pitilde}{\widetilde{\Pi}}
\newcommand{\mbfPitilde}{\widetilde{\mbfPi}}

\DeclareMathOperator{\Isom}{Isom}

\DeclareMathOperator{\Gal}{Gal}
\newcommand{\GalDelta}{\Gal^{\upDelta}}

\DeclareMathOperator{\mat}{mat}

\DeclareMathOperator{\Pro}{Pro}
\DeclareMathOperator{\rk}{rk}
\DeclareMathOperator{\Sing}{Sing}
\DeclareMathOperator{\Exit}{Exit}
\newcommand{\ExittildeP}{\widetilde{\Exit}{}^{P}}

\DeclareMathOperator{\Pt}{Pt}

\DeclareMathOperator{\Nbd}{Nbd}

\DeclareMathOperator{\MOR}{Hom}
\newcommand{\Funlowerstar}{\Fun_{\ast}}
\newcommand{\Funupperstar}{\Fun^{\ast}}
\newcommand{\Funlex}{\Fun^{\lex}}
\newcommand{\sdop}{\sd^{\op}}
\DeclareMathOperator{\rank}{rk}

\newcommand{\Functs}{\Fun^{\cts}}

\newcommand{\Funcross}{\Fun^{\cross}}

\DeclareMathOperator{\Frac}{Frac}

\newcommand{\SC}{\upbeta}

\newcommand{\restrict}[2]{#1|_{#2}}
\newcommand{\Union}{\bigcup}
\DeclareMathOperator{\Lup}{L}

\DeclareMathOperator{\Nerve}{N}
\newcommand{\NNerve}{\NNup}
\renewcommand{\Bar}{\operatorname{Bar}}
\DeclareMathOperator{\Hup}{H}
\DeclareMathOperator{\Oup}{O}
\newcommand{\Ohens}{\Oup^{\hens}}
\newcommand{\Osh}{\Oup^{\sh}}
\DeclareMathOperator{\Tup}{T}

\DeclareMathOperator{\Frob}{Frob}
\DeclareMathOperator{\Mon}{Mon}

\DeclareMathOperator{\BC}{BC}

\newcommand{\cohbdd}{_{<\infty}^{\coh}}
\newcommand{\UUcoh}{\UU^{\coh}}
\newcommand{\WWcoh}{\WW^{\coh}}
\newcommand{\XXcoh}{\XX^{\coh}}
\newcommand{\YYcoh}{\YY^{\coh}}
\newcommand{\ZZcoh}{\ZZ^{\coh}}
\newcommand{\UUcohbdd}{\UUcoh_{<\infty}}
\newcommand{\WWcohbdd}{\WWcoh_{<\infty}}
\newcommand{\XXcohbdd}{\XXcoh_{<\infty}}
\newcommand{\YYcohbdd}{\YYcoh_{<\infty}}
\newcommand{\ZZcohbdd}{\ZZcoh_{<\infty}}

\newcommand{\XXhyp}{\XX^{\hyp}}
\newcommand{\YYhyp}{\YY^{\hyp}}
\DeclareMathOperator{\Pathop}{Path}
\newcommand{\Path}[1]{\Pathop(#1)}

\newcommand{\XXan}{(X/X^{\zar})_{\an}}
\newcommand{\XPan}{(X/P)_{\an}}

\newcommand{\XSan}{(X/S)_{\an}}
\newcommand{\XSet}{(X/S)_{\et}}

\newcommand{\YYan}{(Y/Y^{\zar})_{\an}}

\DeclareMathOperator{\lisse}{lisse}
\DeclareMathOperator{\locsys}{lc}

\newcommand{\Pcons}{P\textup{-}\!\operatorname{cons}}
\newcommand{\Qcons}{Q\textup{-}\!\operatorname{cons}}

\newcommand{\Scons}{S\textup{-}\!\operatorname{cons}}

\newcommand{\Tcons}{T\textup{-}\!\operatorname{cons}}
\newcommand{\Mcons}{M\textup{-}\!\operatorname{cons}}
\newcommand{\Ucons}{U\textup{-}\!\operatorname{cons}}

\newcommand{\Sfcons}{S\textup{-}\!\operatorname{fcons}}

\newcommand{\Sspec}{S\textup{-}\!\operatorname{spec}}
\newcommand{\nspec}{[n]\textup{-}\!\operatorname{spec}}
\newcommand{\ncoh}{n\text{-}\!\operatorname{coh}}

\newcommand{\DecTop}[1]{\Dec_{#1}(\Topbc)}
\newcommand{\DecTopStone}[1]{\Dec_{#1}(\TopStone)}
\newcommand{\DecSpace}[1]{\Dec_{#1}(\Space)}
\newcommand{\DecSpacefin}[1]{\Dec_{#1}(\Spacefin)}
\newcommand{\DecSpaceprofin}[1]{\Dec_{#1}(\Spaceprofin)}

\newcommand{\orientedpull}[3]{#1 \orientedtimes_{#2} #3}
\newcommand{\orientedpush}[3]{#1 \orientedcup^{#2} #3}
\newcommand{\commacat}[3]{#1 \mathbin{\downarrow}_{#2} #3}
\newcommand{\commacatdisplay}[3]{#1 \underset{#2}{\mathbin{\downarrow}} #3}
\newcommand{\Loc}[2]{{#1}_{(#2)}}

\newcommand{\near}{\upPsi}
\newcommand{\conear}{\upPsi}
\newcommand{\conearlowerstar}{\conear_{\ast}}

\newcommand{\AAup}{\mathbfup{A}}
\newcommand{\CCup}{\mathbfup{C}}
\newcommand{\FFup}{\mathbfup{F}}
\newcommand{\NNup}{\mathbfup{N}}
\newcommand{\NNrhd}{\NNup^{\rhd}}
\newcommand{\PPup}{\mathbfup{P}}
\newcommand{\QQup}{\mathbfup{Q}}
\newcommand{\UUup}{\mathbfup{U}}
\newcommand{\ZZup}{\mathbfup{Z}}
\newcommand{\FFp}{\FFup_{\!p}}
\newcommand{\FFq}{\FFup_{\!q}}
\newcommand{\QQell}{\QQup_{\el}}
\newcommand{\QQellbar}{\overline{\QQup}_{\el}}
\newcommand{\ZZell}{\ZZup_{\el}}
\newcommand{\ZZhat}{\widehat{\ZZup}}
\DeclareMathOperator{\BG}{BG}
\DeclareMathOperator{\BD}{BD}
\DeclareMathOperator{\BI}{BI}
\DeclareMathOperator{\Gup}{G}
\DeclareMathOperator{\upU}{U}
\newcommand{\Gk}{\Gup_{k}}
\newcommand{\GK}{\Gup_{K}}
\newcommand{\BGk}{\BG_{k}}
\newcommand{\BGK}{\BG_{K}}
\DeclareMathOperator{\Iup}{I}
\DeclareMathOperator{\Kup}{K}
\DeclareMathOperator{\Sup}{S}
\newcommand{\fan}{\wedgemidvert} 
\DeclareMathOperator{\Bup}{B\!}
\newcommand{\BZZ}{{\Bup\ZZup}}
\newcommand{\BZZhat}{{\Bup\ZZhat}}

\newcommand{\mfrak}{\mathfrak{m}}
\newcommand{\pfrak}{\mathfrak{p}}

\newcommand{\Shape}{\upPi_{\infty}}
\newcommand{\Shapen}{\upPi_{n}}
\newcommand{\Shapetrun}{\upPi_{<\infty}}
\newcommand{\Shapeprofin}{\widehat{\upPi}_{\infty}}
\newcommand{\Shapeet}{\Shape^{\et}}
\newcommand{\Shapeettrun}{\Shapetrun^{\et}}
\newcommand{\Shapeetprofin}{\Shapeprofin^{\et}}
\newcommand{\Shapeetprofinone}{\widehat{\upPi}_{1}^{\et}}
\newcommand{\Shapeethyp}{\Shape^{\et,\hyp}}
\newcommand{\Shapeethypprofin}{\Shapeprofin^{\et,\hyp}}

\newcommand{\StrShape}{\widehat{\upPi}_{(\infty,1)}}

\newcommand{\StrShapeet}{\StrShape^{\et}}

\newcommand{\pietext}{\uppi^{\et}}
\newcommand{\piet}{\hat{\uppi}^{\et}}
\newcommand{\Pihat}{\widehat{\upPi}}

\usepackage{todonotes} 
\usepackage{xargs}  
\newcommandx{\PHtodo}[2][1=]{\todo[linecolor=blue,backgroundcolor=blue!25,bordercolor=blue,#1]{#2}}

\newcommand{\enumref}[2]{(\hyperref[#1.#2]{\ref*{#1}.\ref*{#1.#2}})}

\numberwithin{equation}{subsection}

\defbibheading{references}[\refname]{%
	\section*{#1}%
	\addcontentsline{toc}{part}{References} %
	\markboth{#1}{#1}
	}

\indexsetup{toclevel=part} 
\makeindex[name=terminology,title={Index of Terminology},columns=2,intoc]
\makeindex[name=notation,title={Index of Notation},columns=2,intoc]

\title{Exodromy}

\author{Clark Barwick \and Saul Glasman \and Peter Haine}

\date{August 21, 2020}

\begin{document}

\maketitle

\begin{abstract} 
	Let $X$ be a quasicompact quasiseparated scheme.
	Write $\Gal(X)$ for the category whose objects are geometric points of $X$ and whose morphisms are specializations in the étale topology.
	We define a natural profinite topology on the category $ \Gal(X) $ that globalizes the topologies of the absolute Galois groups of the residue fields of the points of $ X $.
	One of the main results of this book is that $\Gal(X)$ variant of MacPherson's exit-path category suitable for the étale topology: we construct an equivalence between representations of $\Gal(X)$ and constructible sheaves on $ X $. 
	We show that this \textit{exodromy equivalence} holds with nonabelian coefficients and with finite abelian coefficients.
	More generally, by using the pyknotic/condensed formalism, we extend this equivalence to coefficients in the category of modules over profinite rings and algebraic extensions of $ \QQell $.
	As an `exit-path category', the topological category $ \Gal(X) $ also gives rise to a new, concrete description of the étale homotopy type of $ X $.

	We also prove a higher categorical form of Hochster Duality, which reconstructs the entire étale topos of a quasicompact and quasiseparated scheme from the topological category $\Gal(X)$.
	Appealing to Voevodsky's proof of a conjecture of Grothendieck, we prove the following reconstruction theorem for normal varieties over a finitely generated field $ k $ of characteristic $ 0 $: the functor $ \goesto{X}{\Gal(X)} $ from normal $ k $-varieties to topological categories with an action of $\Gk$ and equivariant functors that preserve minimal objects is \textit{fully faithful}.
\end{abstract}

\newpage 

\bookmark[page=2,level=1]{Contents}
\setcounter{tocdepth}{2}
\tableofcontents

\newpage


\setcounter{section}{-1}
\section{Introduction}


Let $X$ be a quasicompact quasiseparated scheme.

\begin{cnstr}\label{cnstr:GalXcat}
 	Define a category $\Gal(X)$ as follows.
	\begin{itemize}
 		\item An object is a \textit{geometric point} $\fromto{x}{X}$: a point $ \Spec\upkappa(x) \to X $ that exhibits $\upkappa(x) $ as a separable closure of the residue field $\upkappa(x_0)$ of its image $x_0\in X^{\zar}$.
 
 		\item Given geometric points $x\to X$ and $y\to X$, a morphism $\fromto{x}{y}$ in $ \Gal(X) $ is a \textit{specialization} $x\leftsquigarrow y$: a lift of the geometric point $ y \to X $ to a geometric point $y\to X_{(x)}$ of the strict localization \smash{$X_{(x)} \colonequals \Spec(\Osh_{X,x_0}) $}. 
	\end{itemize}
\end{cnstr}

The assignment $\goesto{x}{x_0}$ defines a functor from $ \Gal(X) $ to the \textit{specialization poset} of the Zariski topological space $X^{\zar}$ of $ X $: the poset of points of $ X^{\zar} $ in which $x_0\leq y_0$ if and only if $x_0$ lies in the closure of $y_0$. 
The fiber over a point $x_0$ is a connected groupoid in which the automorphism group of each object is the absolute Galois group $\Gup_{\upkappa(x_0)}$ of $\upkappa(x_0)$.
The category $\Gal(X)$ thus \textit{globalizes} the absolute Galois groups of the residue fields of the points of $ X $.

Accordingly, we synthesize the profinite topologies on these Galois groups into a global topology on the category $ \Gal(X) $:

\begin{cnstr}\label{cnstr:GalXtop}
	For any point $u\to X$ that is finite over its image $u_0\in X^{\zar}$, we form the unramified extension $A$ of the henselization $\Ohens_{X,u_0}$ with residue field the separable closure of $\upkappa(u_0)$ in $\upkappa(u)$, and we write $X_{(u)}\coloneq\Spec A$.
	If $ v \to X $ is finite over its image $v_0\in X^{\zar}$, then a \emph{specialization} $u\leftsquigarrow v$ is a point $v\to X_{(u)}$ of $X_{(u)}$ lying over $v\to X$.
	For any specialization $u\leftsquigarrow v$, we define the subset $U(u\leftsquigarrow v)$ of the set of morphisms of $\Gal(X)$ consisting of those specializations $x\leftsquigarrow y$ that lie over $u\leftsquigarrow v$. 
	We endow the morphisms of $\Gal(X)$ with the topology generated by the sets $U(u\leftsquigarrow v)$. 
\end{cnstr}

In this book, we prove three theorems about the topological category $\Gal(X)$: the \textit{Exodromy}, \textit{Homotopy}, and \textit{Reconstruction Theorems}.
The Exodromy Theorem is a classification theorem for constructible sheaves on $X$.
To state it, we consider the constructible derived category%
\footnote{More precisely, we work with the natural \categorical enhancements of these derived categories.}
$ \Dcons(X; \Lambda) $ in each of the following situations:
\begin{enumerate}[(1)]
	\item If $\Lambda$ is a finite ring, then $\Dcons(X;\Lambda)$ is the constructible derived category of complexes of $\Lambda$-modules with perfect stalks.

	\item If $\Lambda$ is a noetherian ring that is complete with respect to an ideal $ I \subset \Lambda $ such that the quotients $ \Lambda/I^n $ are finite, then $\Dcons(X;\Lambda)$ is the category of constructible complexes of $\Lambda$-modules (with perfect stalks) as constructed by Deligne \cite[\S1.1]{MR601520}, Ekedahl \cite{MR1106899}, and Bhatt--Scholze \cite{MR3379634}.

	\item If $\Lambda $ is an algebraic extension of \smash{$ \QQell $}, then $\Dcons(X;\Lambda)$ is the category of constructible complexes of $\Lambda$-vector spaces (with perfect stalks) as constructed by Deligne \cite{MR601520} and Bhatt--Scholze \cite{MR3379634}.
\end{enumerate}
Using the formalism of pyknotic/condensed mathematics \cite{pyknoticI,Scholze:Condensedtalk,Scholze:condensedtalknotes,Scholze:condensednotes}, one can make sense of `topologies' on (higher) categories and the continuity of functors between them.
In each of these cases, the natural topology on $\Lambda$ induces a `topology' on the category $\Perf(\Lambda)$ of perfect complexes of $ \Lambda $-modules; we show that for every constructible complex $F\in \Dcons(X;\Lambda)$, the formation of the stalks $ x \mapsto F_x$ defines a continuous functor
\begin{equation*}
	\exrep_F \colon \Gal(X) \to \Perf(\Lambda) \period
\end{equation*}
We prove:

\begin{thm}[{Exodromy; \Cref{thm:exodromyfinitering,thm:exodromyZell,thm:exodromyQellbar}}]\label{lede:exodromy}
	Let $ X $ be a scheme whose underlying topological space is noetherian.
	Then the assignment $F \mapsto \exrep_F $ is an equivalence of \categories
	\begin{equation*}
		\Dcons(X;\Lambda) \simeq \Functs(\Gal(X),\Perf(\Lambda)) \period
	\end{equation*}
\end{thm}

In fact, the coefficients in the Exodromy Theorem can be taken to be even more general, even nonabelian (see \Cref{cor:Exoforschemes,mainthm:exodromyfinitering,introthm:exodromyZell,introthm:exodromyQellbar} for more refined statements).
As a result, we are able to use Hoyois' description of the \textit{étale homotopy type} $\Shapeet(X)$ of Artin--Mazur--Friedlander \cites{MR0245577}[\S 4]{MR676809} via Lurie's shape theory \cites[\HTTsubsec{7.1.6}]{HTT}[\HAsec{A.1}]{HA}[\SAGsec{E.2}]{SAG}[Corollary 5.6]{MR3763287} to prove our \textit{Homotopy Theorem}.
This result shows that we can recover the étale homotopy type from $\Gal(X)$ in the following manner.
We may invert all the morphisms of $\Gal(X)$ in a way that respects the topology to obtain a \textit{classifying prospace} $\invert(\Gal(X))$.
We then construct a natural map
\begin{equation*}
	\theta_X \colon \Shapeet(X) \to \invert(\Gal(X)) \comma
\end{equation*}
and prove that $ \theta_X $ is an equivalence (in an appropriate sense):

\begin{thm}[Homotopy; \Cref{thm:mainAG}]\label{lede:homotopy}
	Let $ X $ be a quasicompact quasiseparated scheme.
	For each geometric point $x \to X$ and integer $n \geq 1$, the map $\theta_X$ induces an isomorphism of homotopy progroups
	\begin{equation*}
		\isomto{\pietext_n(X,x)}{\uppi_n(\Gal(X),x)} \period
	\end{equation*}
\end{thm}

\noindent Thus the category $\Gal(X)$ is a \textit{refinement} of the étale homotopy type of $ X $.
Moreover, the Homotopy Theorem provides a new, concrete description of the étale homotopy type.

Once again, more is true: $\Gal(X)$ recovers not only the étale homotopy type but the whole étale topos.
This fact, combined with an early result of Voevodsky \cite{MR1098621}, yields the following \textit{Reconstruction Theorem}.
We view this Reconstruction Theorem as a globalization of the Neukirch--Uchida Theorem \cites{MR0244211}{MR0258804}{MR0432593}: it provides situations in which the topological category $\Gal(X)$ is a complete invariant of the scheme $X$.

\begin{thm}[Reconstruction; \Cref{nul:mainthmonGal}]\label{lede:reconstruction} 
	Let $k$ be a finitely generated field of characteristic zero, and let $\Gk$ be the absolute Galois group of $ k $.
	Then the assignment
	\begin{equation*}
		\goesto{X}{\Gal(X)}
	\end{equation*}
	is fully faithful as a functor from normal $k$-varieties to topological categories over $\BGk$ and continuous functors over $\BGk$ that carry minimal objects to minimal objects.
\end{thm}

\noindent Together with a generalization of the Riemann Existence Theorem (\Cref{thm:headlinestratRiemannexistence}), this proves that normal $k$-varieties and morphisms between them can be reconstructed entirely from topological and group-theoretic data.

The Exodromy Theorem (\Cref{lede:exodromy}) justifies thinking of the category $\Gal(X)$ is an algebro-geometric analogue of MacPherson's exit-path category.
To further explain this point, we turn to a discussion of homotopy types and exit-path categories in topology and algebraic geometry.
We then give a more detailed overview of the contents of this book.


\subsection{Monodromy for topological spaces}

It is a truth universally acknowledged, that a locally constant sheaf of $\CCup$-vector spaces on a connected topological manifold $M$ is completely determined by its attached \textit{monodromy representation}: a choice of basepoint $ x \in M $ specifies an equivalence of categories
\begin{equation*}
	\Mon_x \colon \equivto{\LC(M;\Vect(\CCup))}{\Rep_{\CCup}(\uppi_1(M,x))} \period
\end{equation*}
To avoid selecting a point and to drop the connectivity hypothesis on $M$, we can combine the set of connected components and the various fundamental groups of $M$ to form the \textit{fundamental groupoid} $\upPi_1(M)$. 
Then the monodromy equivalence becomes a natural equivalence
\begin{equation*}
	\Mon \colon \equivto{\LC(M;\Vect(\CCup))}{\Fun(\upPi_1(M),\Vect(\CCup))} \period
\end{equation*}

An early insight of Kan was that, in a similar fashion, \textit{all} the homotopy groups and all the $k$ invariants of $M$ could be combined to form a single combinatorial object which knows everything about the homotopy type of $M$.
This combinatorial object a simplicial set $\Shape(M)$ known as the \textit{singular simplicial set} or the \textit{fundamental \groupoid} of $M$.
Perhaps the clearest formulation of this insight was that of Quillen, who showed that when topological spaces and simplicial sets are given the conventional choices of model structure, the functor
\begin{equation*}
	\Shape \colon \fromto{\TSpc}{\sSet}
\end{equation*}
is a right Quillen equivalence. 
Nowadays we go a step farther and think of $\Shape$ as an equivalence $\equivto{\Space}{\Gpd_{\infty}}$ between the underlying \textit{\category} of spaces and that of \groupoids.

The fundamental \groupoid of $M$ appears in derived versions of the monodromy equivalence: for instance, the monodromy of a locally constant sheaf of \textit{complexes} of $\CCup$-vector spaces is a functor from $\Shape(M)$ to complexes, and there is a monodromy equivalence of \categories
\begin{equation*}
	\Mon \colon \equivto{\LC(M;\Dup(\CCup))}{\Fun(\Shape(M),\Dup(\CCup))} \period
\end{equation*}
Here $ \Dup(\CCup) $ denotes the derived \category of $ \CCup $.
All of these monodromy equivalences follow from the universal example of locally constant sheaves of \textit{spaces} on $M$, which are known as \textit{parametrized homotopy types} in the homotopy theory literature \cites{MR3890766}{MR3252967}{MR2271789}. 
That is to say, there is a natural monodromy equivalence of \categories
\begin{equation}\label{eq:monmanifold}
	\Mon \colon \equivto{\LC(M;\Space)}{\Fun(\Shape(M),\Space)} 
\end{equation}
from \category of locally constant sheaves of spaces on $ M $ to space-valued representations of the \groupoid $ \Shape(M) $ \cite[Theorems \HAappthmlink{A.1.15} \& \HAappthmlink{A.4.19}]{HA}.


\subsection{Monodromy for schemes and topoi}

To replace the manifold in the monodromy story with a scheme $ X $, Grothendieck introduced the \textit{extended étale fundamental group} $ \pietext_1(X,x) $ of \cite[Exposé X, \S 6]{MR43:223b}. 
Here it is not the Zariski topological space of $ X $ that is relevant, but its \textit{étale topos}.
Moreover, since universal covers in algebraic geometry do not exist as étale covers, but only as \textit{proétale} covers, the extended étale fundamental group is not a group, but rather a \textit{progroup}.
In the case that $ X $ is connected and \textit{locally connected}, Grothendieck proved a monodromy equivalence between locally constant étale sheaves and representations of $ \pietext_1(X,x) $.

In the algebro-geometric setting, local connectedness is not automatic, and removing this hypothesis requires a tradeoff.
On the one hand, without the local connectedness hypothesis, the monodromy equivalence only holds for \textit{lisse} sheaves, i.e., locally constant sheaves of finite sets that can be trivialized on a finite cover.
On the other hand, to classify lisse sheaves, we can work with the (more simple) profinite completion of $\pietext_1(X,x) $: the usual \textit{profinite étale fundamental group} $ \piet_1(X,x) $.
In this setting, the monodromy equivalence states the following: if $ X $ is a connected scheme, then a choice of geometric point $ x \to X $ provides an equivalence between the category of lisse sheaves of sets on $ X $ and finite sets with a continuous action of the profinite group $ \piet_1(X,x) $.


\subsubsection{Monodromy for topoi}

Dubuc \cite[\S\S 5--6]{MR2440260} provided a simultaneous generalization of the fundamental groupoid of a topological space and the étale fundamental group of a scheme: given a topos $ \XX $, Dubuc defined the \textit{fundamental progroupoid} $\upPi_1(\XX)$ of $ \XX $.
In the two cases of interest, Dubuc's progroupoid recovers the existing notions: if $ \XX $ is the category of sheaves of sets on a manifold $ M $, then $\upPi_1(\XX)$ is the fundamental groupoid of $ M $, and if $ \XX $ is the étale topos of scheme $ X $, then $ \upPi_1(\XX) $ is a progroupoid whose automorphism groups at every point are the extended étale fundamental groups of $ X $.
In addition, if the topos $ \XX $ is \textit{locally connected} in a suitable sense, then there is a monodromy equivalence
\begin{equation*}
	\XX^{\locsys} \simeq \Fun(\upPi_1(\XX),\Set)
\end{equation*}
between the locally constant objects of $\XX$ and (continuous) $ \Set $-valued representations of the progroupoid $\upPi_{1}(\XX)$ (see \cite[Remark 2.14 \& Theorem 3.3]{MR3763287}).

Again, to remove the local connectedness hypothesis on $ \XX $, there's a tradeoff: without the local connectedness hypothesis, the monodromy equivalence only holds for \textit{lisse} objects.
On the other hand, to classify lisse sheaves, we can work with the (more simple) profinite completion of $ \Pihat_1(X) $ of $ \upPi_1(X) $: there is a natural monodromy equivalence
\begin{equation*}
	\XX^{\lisse} \simeq \Fun(\Pihat_1(\XX),\Setfin) \period
\end{equation*}


\subsubsection{Monodromy for higher topoi: \texorpdfstring{$\infty$}{∞}-Categorical Stone Duality}

Lurie provided a homotopical refinement of the Dubuc's fundamental progroupoid: given \atopos $ \XX $, Lurie constructed a pro-\groupoid $ \Shape(\XX) $ called the \textit{shape} of $ \XX $ \cites[\HTTsubsec{7.1.6}]{HTT}[\HAsec{A.1}]{HA}[\SAGsec{E.2}]{SAG}.
If $ \XX $ is \textit{locally contractible}%
\footnote{The terms \textit{locally of constant shape} and \textit{locally $ \infty $-connected} are also used for local contractibility \cites[\HAappthm{Definition}{A.1.5} \& \HAappthm{Proposition}{A.1.8}]{HA}[Definition 3.2]{MR3763287}}, 
then there is a natural monodromy equivalence
\begin{equation*}
	\XX^{\locsys} \simeq \Fun(\Shape(\XX),\Space)
\end{equation*}
between the locally constant objects of $\XX$ and (continuous) space-valued representations of the pro-\groupoid $ \Shape(\XX) $ \HAa{Theorem}{A.1.15}.
This theory is well-suited to the study of locally constant sheaves on manifolds: since manifolds are locally contractible, the \topos of sheaves on a manifold is locally contractible. 
Moreover, for any manifold $ M $, the shape of the \topos of sheaves of spaces on $ M $ coincides with the fundamental \groupoid $ \Shape(M) $ of $ M $; the monodromy equivalence for locally contractible \topoi is, in fact, how Lurie proves the monodromy equivalence \eqref{eq:monmanifold}.

The local contractability assumption becomes problematic in the algebro-geo\-metric setting: the étale topos of a scheme is rarely locally simply connected, never mind locally contractible.
In order to have a monodromy equivalence for étale sheaves, it is necessary that we work with the profinite completion of $ \Shape(X_{\et}) $ and \textit{lisse sheaves of spaces} on $ X $.
The setting of higher topos theory, the category of finite sets is replaced by the \category of \textit{\pifinite spaces}, and a \textit{lisse} sheaf is a locally constant sheaf of \pifinite spaces that can be trivialized on a finite cover.
The monodromy theorem in this context is as follows: for any \topos $ \XX $, there is a natural natural monodromy equivalence of \categories
\begin{equation}\label{eq:monshapeprofin}
	\XX^{\lisse} \simeq \Fun(\Shapeprofin(\XX),\Spacefin)
\end{equation}
between the lisse objects on $\XX $ and representations of $\Shapeprofin(\XX)$ valued in the \category $\Spacefin$ of \pifinite spaces \cite[\SAGthm{Proposition}{A.8.3.2}, \SAGthm{Theorem}{E.2.4.1}, \& \SAGthm{Theorem}{E.3.1.1}]{SAG}.
(See also \cite[Proposition 10.1]{MotivicNorms:BachmannHoyois} and \Cref{prop:Stonemonodromy}.)

\begin{exm}
	Hoyois showed that if $ X $ is a locally connected scheme, then the profinite shape $\Shapeprofin(X_{\et})$ of the étale \topos $X_{\et}$ of $ X $ coincides with the \textit{profinite étale homotopy type} $\Shapeetprofin(X)$ of Artin--Mazur--Friedlander \cite[Corollary 5.6]{MR3763287}.
	Thus Lurie's shape theory both provides a new perspective on the étale homotopy type and shows that the profinite étale homotopy type classifies lisse étale sheaves.
\end{exm}

The monodromy equivalence \eqref{eq:monshapeprofin} is a higher categorical form of \textit{Stone Duality}.
Recall that the classsical Stone Duality Theorem identifies profinite sets with totally disconnected compact Hausdorff topological spaces.
Lurie's \Categorical Stone Duality Theorem identifies profinite \textit{spaces} with a certain class of \topoi in the following manner.

\begin{thm}[{\Categorical Stone Duality; \SAG{Theorem}{E.2.4.1}}]\label{introthm:StoneDuality}
	The fully faithful functor $\incto{\Spacefin}{\Top_{\infty}}$ given by the assignment $\goesto{\Pi}{\Fun(\Pi,\Space)}$ extends along inverse limits to a fully faithful right adjoint
	\begin{align*}
		\Pro(\Spacefin) &\inclusion \Top_{\infty} \\ 
		\{\Pi_{\alpha}\}_{\alpha \in A} &\mapsto \lim_{\alpha \in A} \Fun(\Pi_{\alpha},\Space) 
	\end{align*}
	from profinite spaces to \topoi.
	Moreover, the left adjoint of this functor is the profinite shape $ \Shapeprofin \colon \fromto{\Top_{\infty}}{\Pro(\Spacefin)} $.
\end{thm}

The essential image of this embedding $ \incto{\Pro(\Spacefin)}{\Top_{\infty}} $ is spanned by a higher categorical version totally disconnected compact Hausdorff topological spaces; we will return to this point momentarily (see \cref{sec:Exodromyforhighertopoi}).


\subsection{Exodromy for stratified topological spaces}

A string of results has suggested the possibility that \textit{stratified spaces} and \textit{constructible sheaves} might be modeled in a combinatorial fashion similar to spaces and locally constant sheaves.
MacPherson proved that constructible sheaves of sets on a (suitably nice) stratified topological space $T$ over a poset $P$ determine and are determined by a functor from the \textit{exit-path category} $\Exit_1^P(T) $ of $T$.
The objects of $\Exit_1^P(T) $ are points of $T$ and the morphisms are stratified homotopy equivalence classes of \textit{exit paths} -- paths from a stratum $T_p$ to a stratum $T_q$ for $q\geq p$. 
We call this equivalence
\begin{equation*}
	\Ex^P \colon\equivto{\Sh(T;\Set)^{\Pcons}}{\Fun(\Exit_1^P(T),\Set)}
\end{equation*}
between $ P $-constructible sheaves of sets on $ T $ and set-valued representations of $ \Exit_1^P(T) $ the \textit{exodromy} equivalence.%
\footnote{\textgreek{ἔξω}: outer; \textgreek{δρόμος}: avenue.} 
Observe that $\Exit_1^P(T)$ is a category with a conservative functor to $P$:
for each point $p\in P$, the fiber of the functor $ \Exit_1^P(T) \to P $ over $p$ is the fundamental groupoid $\upPi_1(T_p)$ of the stratum $T_p$.

Treumann \cite{MR2575092} extended MacPherson's result to give an exodromy equivalence between constructible \textit{stacks} with groupoid-valued representations of the \textit{exit-path $2$-category} of $T$. 
Lurie \cite[\HAapp{A}]{HA} extended this result still further to give an exodromy equivalence
\begin{equation*}
	\Ex^P \colon \equivto{\Sh(T;\Space)^{\Pcons}}{\Fun(\Exit^P(T),\Space)}
\end{equation*}
between $ P $-constructible sheaves on $T$ with values in the \category of spaces and space-valued representations of an \textit{exit-path \category} $\Exit^P(T)$ of $ T $.
The objects of $ \Exit^P(T) $ are points of $T$, the morphisms are exit-paths, the $2$-morphisms are stratified homotopies, the $3$-morphisms are stratified homotopies of homotopies, etc., \textit{ad infinitum}.
Again, $\Exit^P(T)$ is \acategory with a conservative functor to $P$: over each point $p\in P$, the fiber of this functor is the fundamental \groupoid $\Shape(T_p)$ of the stratum $T_p$.


\subsubsection{Stratified spaces as \texorpdfstring{$\infty$}{∞}-categories with a conservative functor to a poset}

One is led to seek an analogue of the Kan--Quillen Theorem that states that the formation of the exit-path \category is an equivalence of homotopy theories between stratified topological spaces and suitable \categories.
A geometric form of this result was proven by Ayala--Francis--Rozenblyum \cite{MR3941460}, who showed that the exit-path \category construction is fully faithful from a homotopy theory of \textit{conically smooth} stratified spaces to \categories.

A still closer stratified analogue of the Kan--Quillen equivalence has now been provided by the simultaneous work of three teams:
Douteau \cites{Douteau:thesis}{Douteau:chapter7} (after Henriques \cites{Henriques:thesis}{Henriques:stratifiedmodel}), Nand-Lal--Woolf \cites{Nand-Lal:thesis}{NandLalWoolf}, and the third-named author \cite{Haine:stratifiedmodel}.
These papers each take a slightly different point of view, but for our purposes the salient point (expressed in \cite{Haine:stratifiedmodel}) is that the functor $\Exit^P$ is an equivalence between the following homotopy theories:
\begin{enumerate}[(1)]
	\item Topological spaces with a sufficiently nice stratification over $P$ -- in which a weak equivalence of such is a weak equivalence on strata and (homotopy) links.

	\item The \category of \categories equipped with a conservative functor to $P$.
\end{enumerate}

We thus refer to \categories with a conservative functor to a poset $P$ as \textit{$P$-stratified spaces}.
This makes it possible to port some of the ideas of stratified homotopy theory to the study of schemes, as we shall soon see.


\subsection{Exodromy for higher topoi: \texorpdfstring{$\infty$}{∞}-Categorical Hochster Duality}

With motivation from existing monodromy and exodromy results in place, in the remainder of this introduction we explain the results of this book in more detail.
The first goal of this book is to prove that for any coherent\footnote{Following the Grothendieck school we use the term `coherent scheme' synonymously with `quasicompact quasiseparated scheme' \cref{nul:coherentqcqs}.} scheme $ X $, there exists a `profinite' \category $ \StrShapeet(X) $ that classifies constructible sheaves of \textit{spaces} on $ X $.
That is, for which there is a natural \textit{exodromy equivalence}
\begin{equation}\label{introeq:desireexodromy}
	X_{\et}^{\cons} \equivalent \Fun(\StrShapeet(X), \Spacefin) \period
\end{equation}
Due to the failure of the étale \topos to be locally simply connected, here it is crucial that the term `constructible' is interpreted in the way it is usually used in algebraic geometry, rather than topology: a sheaf $ F $ is \textit{constructible} if there is a finite stratification such that the restriction of $ F $ to each stratum is lisse (not just locally constant).
The second goal is to identify the profinite \category $ \StrShapeet(X) $ with the category in profinite topological spaces $ \Gal(X) $ introduced in \Cref{cnstr:GalXcat,cnstr:GalXtop}.


\subsubsection{Hochster Duality}

In order to accomplish the first goal, we take a hint from Lurie's proof that the profinite shape of a \topos classifies lisse sheaves.
Through his proof of \Categorical Stone Duality (\Cref{introthm:StoneDuality}), Lurie identifies the essential image of the embedding
\begin{equation*}
	\Pro(\Spacefin) \inclusion \Top_{\infty}
\end{equation*} 
of profinite spaces into \topoi with those \topoi that are \textit{bounded coherent} in which the truncated coherent objects coincide with the \textit{lisse} sheaves \cite[\SAGsec{E.3}]{SAG}.
We call these \topoi \textit{Stone \topoi.}%
\footnote{Lurie calls these \textit{profinite \topoi}.
In this book we introduce a more general class of \topoi that could also reasonably be called `profinite \topoi', so we use the distinct term `Stone \topoi' to avoid confusion.}
\textit{Boundedness} is a technical condition that means that the \topos can be recovered by its truncated objects in a particular way; \textit{coherence} is the higher-toposic version of being quasicompact and quasiseparated and generalizes Groth\-endieck's notion of coherence for ordinary topoi \cite[Exposé VI]{MR50:7131}. 
The key example to keep in mind is that the étale \topos of a coherent scheme is bounded coherent.
However, in the case of the étale \topos, the truncated coherent objects coincide with the larger class of \textit{constructible} sheaves, rather than the lisse sheaves.

Thus in order to provide the desired equivalence \eqref{introeq:desireexodromy}, we aim to extend \Categorical Stone Duality to an equivalence between profinite \textit{stratified} spaces and a class of \textit{stratified} \topoi.
Hochster's thesis \cites{Hochster:thesis}{Hochster:primeideal} gives evidence that such a generalization should be possible: Hochster identifies the category of profinite posets with the category of \textit{spectral topological spaces}, i.e., those topological spaces that underlie coherent schemes. 
This functions as a simultaneous generalization of Alexandroff Duality (which identifies finite posets with finite $ \Tup_0 $ topological spaces) and Stone Duality. 

For our generalization of Hochster Duality, we need to identify the correct notion of \textit{\pifiniteness} for stratified spaces.
This turns out to be the following simple generalization of \pifiniteness for spaces: 

\begin{dfn}
	A stratified space $ \fromto{\Pi}{P} $ is \defn{\pifinite} if the poset $ P $ is finite, the \category $ \Pi $ has finitely many objects up to equivalence, and the mapping spaces of $ \Pi $ are \pifinite spaces.

	Write $ \Strfin $ for the \category of \pifinite stratified spaces.
	We call objects of the \category $ \Pro(\Strfin) $ of \proobjects of $ \Strfin $ \defn{profinite stratified spaces}.
\end{dfn}

Profinite stratified spaces are, by definition, stratified by profinite posets, equivalently, by spectral topological spaces. 
With this in mind, there is an obvious way to talk about stratified \topoi on the same footing.

\begin{dfn}
	Let $ S $ be a spectral topological space.
	An \defn{$ S $-stratified \topos} is \atopos $ \XX $ equipped with a geometric morphism $ \fromto{\XX}{\Sh(S)} $ to the \topos of sheaves of spaces on $ S $.
	We write $ \StrTop_{\infty} $ of the \category of stratified \topoi.
\end{dfn}

The following two examples are of particular importance to our results.

\begin{exm}
	Let $ P $ be a finite poset.
	Then there is a natural equivalence 
	\begin{equation*}
		\Sh(P) \equivalent \Fun(P,\Space)
	\end{equation*}
	between sheaves on the Alexandroff topological space attached to $ P $ and functors $ \fromto{P}{\Space} $.

	More generally, if $ S $ is a spectral topological space regarded as a profinite poset $ \{P_{\alpha}\}_{\alpha \in A} $, then
	\begin{equation*}
		\Sh(S) \equivalent \lim_{\alpha \in A} \Fun(P_{\alpha},S) \period
	\end{equation*} 
\end{exm}

\begin{exm}
	Let $ X $ be a coherent scheme.
	Then the natural geometric morphism 
	\begin{equation*}
		X_{\et} \to X_{\zar} = \Sh(X^{\zar})
	\end{equation*}
	is a stratification of the étale \topos of $ X $ by the Zariski topological space $ X^{\zar} $ of $ X $.
\end{exm}

To identify profinite stratified spaces with a class of stratified \topoi, we follow the paradigm of Lurie's \Categorical Stone Duality Theorem and extend the functor
\begin{align*}
	\Strfin &\to \StrTop_{\infty}
\intertext{given by the assignment}
	[\Pi \to P] &\mapsto [\Fun(\Pi,\Space) \to \Fun(P,\Space)]
\end{align*}
along inverse limits.
We also generalize the theory of constructible sheaves to stratified \topoi, and prove the following higher categorical refinement of Hochster Duality.

\begin{thm}[\Categorical Hochster Duality; \Cref{thm:inftyHochster}]\label{thm:headlineinftyHochster}
	Write
	\begin{equation*}
		\lambdahat \colon \Pro(\Strfin) \to \StrTop_{\infty}
	\end{equation*}
	for the functor that carries a profinite stratified space $ \mbfPi = \{\Pi_{\alpha} \}_{\alpha \in A} $ to the \topos 
	\begin{equation*}
		\Fun(\mbfPi,\Space) \colonequals \lim_{\alpha \in A} \Fun(\Pi_{\alpha},\Space) \period
	\end{equation*}
	Then the functor $ \lambdahat $ is fully faithful with its essential image those bounded coherent stratified \topoi in which the truncated coherent objects coincide with the constructible sheaves.
\end{thm}

\noindent We call these \topoi \textit{spectral \topoi} (\Cref{def:spectraltopos}).
This is partially justified by the fact that they are the natural higher categorical extension of Hochster's spectral topological spaces.


\subsubsection{Exodromy for higher topoi}\label{sec:Exodromyforhighertopoi}

The next thing we show is that for each spectral topological space $ S $, the fully faithful functor
\begin{equation*}
	\lambdahat_{S} \colon \incto{\Pro(\Strfin)_{S}}{\StrTop_{\infty,S}} 
\end{equation*}
from profinite $ S $-stratified space to $ S $-stratified \topoi admits a left adjoint
\begin{equation*}
	\StrShape^{S} \colon \fromto{\StrTop_{\infty,S}}{\Pro(\Strfin)_{S}} \period
\end{equation*}
Given the existence of this left adjoint, the following \textit{Exodromy Theorem} is an immediate consequence of \Categorical Hochster Duality.

\begin{thm}[Exodromy; \Cref{thm:exodromyforstrattopoi}]\label{mainthm:classifyconstr}
	Let $ S $ be a spectral topological space.
	For any $S$-stratified \topos $\XX$, the unit
	\begin{equation*}
		\fromto{\XX}{\Fun(\StrShape^{S}(\XX),\Space)}
	\end{equation*}
	of the adjunction to profinite stratified spaces restricts to an equivalence
	\begin{equation*}
		\Fun(\StrShape^{S}(\XX),\Spacefin) \simeq \XX^{\Scons}
	\end{equation*}
	between the \category of representations of $ \StrShape^{S}(\XX) $ valued in \pifinite spaces and $ S $-constructible sheaves $\XX$.
	We call this identification the \emph{exodromy equivalence} for stratified \topoi.
\end{thm} 

\noindent We call the profinite \category $ \StrShape^{S}(\XX)$ the \textit{profinite $S$-stratified shape of $\XX$}.

\begin{exm}[\Cref{exm:profinshapedeloc}]\label{exm:stratshapeisdelocofprofinshape}
	Let $ \XX $ be a spectral $ S $-stratified \topos.
	The profinite stratified shape is a \textit{refinement} of the usual profinite shape type of $\XX$: the profinite classifying space of \smash{$ \StrShape^{S}(\XX)$} is precisely \smash{$\Shapeprofin(\XX)$}.
\end{exm}

\begin{exm}[\Cref{lem:matofPiinfisPt}]\label{exm:stratshaperefinesPt}
	Let $ \XX $ be a spectral $ S $-stratified \topos.
	The profinite stratified shape is a \textit{refinement} of the \category
	\begin{equation*}
		\Pt(\XX) \colonequals \Funupperstar(\Space,\XX)
	\end{equation*}
	of \textit{points} of $ \XX $ in the following sense.
	The profinite stratified shape \smash{$ \StrShape^{S}(\XX) $} is a pro-\category, and taking the limit of a prosystem defining \smash{$ \StrShape^{S}(\XX) $} recovers the \category $ \Pt(\XX) $.
	We thus think of the profinite stratified shape as the \category $ \Pt(\XX) $ equipped with an a `profinite structure'.
\end{exm}


\subsection{Exodromy for schemes \& the Reconstruction Theorem}

Our interest in the profinite stratified shape is primarily due to the following example.

\begin{exm}\label{exm:stratetshapeisGal}
	If $ X $ is a coherent scheme, then the stratified \topos $ \fromto{X_{\et}}{X_{\zar}} $ is spectral.
	We call the profinite \category
	\begin{equation*}
		\StrShapeet(X) \coloneq \StrShape^{X^{\zar}}(X_{\et}) 
	\end{equation*}
	the \textit{stratified étale homotopy type} of $ X $.
	
	Since the étale \topos of $ X $ comes from an ordinary topos, the stratified étale homotopy type \smash{$ \StrShapeet(X) $} is simply a profinite $1$-category.
	In light of \Cref{exm:stratshaperefinesPt}, the $ 1 $-category obtained by taking the limit of a prosystem defining \smash{$ \StrShapeet(X) $} is the $ 1 $-category $ \Pt(X_{\et}) $ of points of the étale \topos of $ X $.
	The Grothendieck School showed that the $ 1 $-category $ \Pt(X_{\et}) $ is the $ 1 $-category $ \Gal(X) $ of geometric points of $ X $ introduced in \Cref{cnstr:GalXcat} \cite[Exposé VIII, Théorème 7.9]{MR50:7131}.

	We are able to identify stratified $1$-types with $1$-categories equipped with a suitable topology; under this correspondence, the stratified étale homotopy type of $ X $ agrees with the \textit{topological} category $\Gal(X)$ \Cref{cnstr:GalXcat,cnstr:GalXtop}.
	That is, there is a natural identification
	\begin{equation*}
		\StrShapeet(X) \equivalent \Gal(X) \period
	\end{equation*}
\end{exm}

\Cref{mainthm:classifyconstr} thus provides the following exodromy equivalence for schemes:

\begin{cor}[Exodromy for schemes]\label{cor:Exoforschemes}
	Let $ X $ be a coherent scheme.
	Then there is a natural exodromy equivalence
	\begin{equation*}
		\Fun(\Gal(X),\Spacefin) \simeq X_{\et}^{\cons} \period
	\end{equation*}
\end{cor}

Armed with this, the Reconstruction Theorem (\Cref{lede:reconstruction}) follows as soon as we know that the $ k $-schemes in question can be recovered from their étale \topoi. 
On this score, in his letter to Faltings, Grothendieck conjectured -- and Voevodsky proved \cite{MR1098621} -- that the assignment $\goesto{X}{X_{\et}}$ is a fully faithful functor from normal schemes of finite type over a finitely generated field $k$ of characteristic $0$ to \topoi with an action of the absolute Galois group $\Gk$ and `admissible' $\Gk$-equivariant morphisms. 
Combined with our results on the profinite stratified shape, we obtain our \Cref{lede:reconstruction}.

Whereas one only expects that the étale homotopy type is a complete invariant of varieties constructed iteratively from hyperbolic curves, the addition of the natural stratification on the étale homotopy type makes the stratified étale homotopy type a complete invariant of \textit{all} normal varieties.

In positive characteristic and for more general arithmetic schemes, the presence of inseparable extensions forces us to give a more careful formulation of Grothendieck's conjecture (\Cref{cnj:main}).
Both it and the analogue of \Cref{lede:reconstruction} remain open problems in this case.


\subsection{Extending exodromy: coefficients \& \texorpdfstring{$\el$}{ℓ}-adic sheaves}

The nonabelian exodromy equivalence readily implies the abelian version of the Exodromy Theorem (\Cref{lede:exodromy}) with coefficients in a finite ring.

\begin{thm}[{Exodromy with finite coefficients; \Cref{thm:exodromyfinitering}}]\label{mainthm:exodromyfinitering}
	Let $ X $ be a coherent scheme and let $ R $ be a finite ring.
	Then there is a natural equivalence of \categories
	\begin{equation*}
		\Dcons(X;R) \simeq \Fun(\Gal(X),\Perf(R)) \period
	\end{equation*}
\end{thm}

\noindent Thus the datum of a constructible sheaf $F$ of complexes of $ R $-modules on $ X $ is essentially the same information as that of a \textit{exodromy representation}
\begin{equation*}
	\exrep_F \colon \Gal(X) \to \Perf(R) \period
\end{equation*}

The claim that $ \Gal(X) $ classifies constructible $ \el $-adic sheaves (\Cref{lede:exodromy}), however, is not obvious from this, and we emphasize that the question a bit subtle.
Recall that Deligne's famous example of a curve of genus $ \geq 1 $ with two points identified shows that the extended étale fundamental group $ \pietext_1(X) $ is insufficient to reconstruct lisse \smash{$ \QQell $}-sheaves \cite[Example 7.4.9]{MR3379634}.
The scheme in Deligne's example is necessarily non-normal; for geometrically unibranch schemes, even the \textit{profinite} étale fundamental group $ \piet_1(X) $ is sufficient to reconstruct lisse \smash{$ \QQell $}-sheaves \cite[Lemmas 7.4.7 \& 7.4.10]{MR3379634}.
Nevertheless, even though $\Gal(X)$ is profinite, it is still capable of classifying constructible \smash{$ \QQell $}-sheaves.
What rescues us is that passage to a sufficiently fine stratification ensures that the strata are all normal.

In order to access coefficients like \smash{$ \ZZell $} and \smash{$ \QQell $} with \textit{topological} structure, we employ a small piece of the \textit{pyknotic} (\textsc{aka} \textit{condensed}) formalism, which was introduced simultaneously by the first- and third-named authors \cite{pyknoticI} and by Clausen--Scholze \cites{Scholze:Condensedtalk,Scholze:condensedtalknotes,Scholze:condensednotes}.
With this, we can speak of \textit{continuous functors} from $\Gal(X)$ into \acategory of perfect complexes over \smash{$ \ZZell $} or \smash{$ \QQell $} that incorporates their topologies into the definition.

Using this formalism, we first extend the exodromy equivalence to coefficients in profinite rings like the ring of integers in a nonarchimedean local field (e.g., \smash{$ \ZZell $} or $ \FFq\llbracket t \rrbracket $):

\begin{thm}[{Exodromy with profinite coefficients; \Cref{thm:exodromyZell}}]\label{introthm:exodromyZell}
	Let $ X $ be a coherent scheme, $\Lambda$ be a noetherian ring, and $ I \subset \Lambda $ an ideal.
	Assume that $ \Lambda $ is complete with respect to the $I$-adic topology and that for each integer $ n \geq 1 $, the quotient ring $ \Lambda/I^n $ is finite.
	Then there is a natural equivalence of \categories
	\begin{equation*}
		\Dcons(X;\Lambda) \simeq \Functs(\Gal(X),\Perf(\Lambda)) \period
	\end{equation*}
\end{thm}

We finally complete the proof of \Cref{lede:exodromy} by extending \Cref{introthm:exodromyZell} to coefficients in \smash{$ \QQell $} or \smash{$ \QQellbar $}.
For this extension, we need the underlying topological space of $ X $ to be noetherian in order to ensure that the usual notion of a constructible complex of \smash{$ \QQell $}-sheaves is equivalent to the requirement that the sheaf is lisse over a finte stratification (see \cite[\S6.6]{MR3379634}).

\begin{thm}[{Exodromy for $ \el $-adic sheaves; \Cref{thm:exodromyQellbar}}]\label{introthm:exodromyQellbar}
	Let $ X $ be a scheme whose underlying topological space is noetherian, $ \el $ be a prime number, and $ E $ be an algebraic field extension of \smash{$ \QQell $}.
	Then there is a natural equivalence of \categories
	\begin{equation*}
		\Dcons(X;E) \simeq \Functs(\Gal(X),\Perf(E)) \period
	\end{equation*}
\end{thm}


\subsection{Other roles of \texorpdfstring{$\Gal(X)$}{Gal(X)}}

We also prove a number of results about the stratified étale homotopy type that are not directly related to constructible sheaves.
In this section, we explain two such results. 


\subsubsection{New description of the étale homotopy type}

Let $ X $ be a cohrent scheme.
We have seen that the stratified étale homotopy type $ \Gal(X) $ of $ X $ is a delocalization of the profinite étale homotopy type $ \Shapeetprofin(X) $ (\Cref{exm:stratshapeisdelocofprofinshape,exm:stratetshapeisGal}).
We have also seen that the whole étale \topos of $ X $ can be reconstructed from the profinite $ 1 $-category $ \Gal(X) $.
Since the uncompleted étale homotopy type $ \Shapeet(X) $ of $ X $ only depends on the étale \topos of $ X $, and since $\Gal(X)$ recovers the étale topos in its entirely, we are assured that the étale homotopy type can be abstractly reconstructed from $\Gal(X)$.

Our Homotopy Theorem (\Cref{lede:homotopy}) is more precise:
it essentially states that the classifying prospace of $ \Gal(X) $ (with no profinite completion) coincides with the étale homotopy type of $ X $.
The Homotopy Theorem follows from the following more general result for spectral \topoi.

\begin{thm}[{Homotopy; \Cref{thm:protruncdeloc}}]
	Let $ \fromto{\XX}{\Sh(S)} $ be a spectral $ S $-stratified \topos, and write $ \invert \colon \fromto{\Pro(\Cat_{\infty})}{\Pro(\Space)} $ for the left adjoint to the inclusion.
	Then there is a natural morphism
	\begin{equation*}
		\theta_{\XX} \colon \fromto{\Shape(\XX)}{\invert(\StrShape^S(\XX))}
	\end{equation*}
	from the shape of $ \XX $ to the prospace obtained by inverting all morphisms in the profinite stratified shape $ \StrShape^S(\XX) $ of $ \XX $.
	Moreover, the morphism $ \theta_{\XX} $ induces an equivalence on protruncations.
\end{thm}

\begin{rmk}		
	A morphism of prospaces induces an equivalence on protruncations if and only if it induces an equivalence on all homotopy prosets.
	Moreover, is not clear what invariants of a prospace exist that are not detected by its protruncation.
	We emphasize that essentially all existing results in étale homotopy theory work with the étale homotopy type up to protruncation (cf. \Cref{rmk:ArtinMazurnatural,rmk:Isaksenmodelstructure}).
	Hence the fact that we prove that the morphism $ \theta_{\XX} $ is an equivalence on protruncations should not be seen as a problematic, but rather as the best result one could reasonably expect. 
\end{rmk}


\subsubsection{Stratified Riemann Existence}

Let $ X $ be a $\CCup$-scheme of finite type, and write $ X^{\an} $ for the topological space of complex points of $ X $ with its analytic topology.
A modern formulation of the Riemann Existence Theorem is that there is a natural comparison morphism 
\begin{equation*}
	\fromto{\Shape(X^{\an})}{\Shapeet(X)}
\end{equation*}
from the fundamental \groupoid of $ X^{\an} $ to the étale homotopy type of $ X $, and this morphism becomes 
an equivalence after profinite completion \cites[Theorem 12.9]{MR0245577}[Proposition 4.12]{Carchedi:higheretale}.
In this book, we prove the following \textit{stratified} refinement of the Riemann Existence Theorem.

\begin{thm}[Stratified Riemann Existence; \Cref{cor:stratRiemann}]\label{thm:headlinestratRiemannexistence}
	Let $ X $ be a $\CCup$-scheme of finite type.
	Then there is a natural equivalence
	\begin{equation*}
		\Gal(X) \simeq \lim_{X^{\zar} \to P} \Exit^{P}(X^{\an})\profincomp \comma
	\end{equation*}
	where the right hand side is limit over finite algebraic stratifications $ X^{\zar} \to P $ of the profinite completions of the exit-path \categories $ \Exit^{P}(X^{\an}) $.
\end{thm}

Combining Stratified Riemann Existence with the Reconstruction Theorem (\Cref{lede:reconstruction}) provides the following variant of the Reconstruction Theorem:

\begin{exm}[{\Cref{exm:StratRiemannfgfield}}]
	Let $ k $ be a finitely generated field of characteristic $ 0 $, and let \smash{$ \kbar $} be an algebraic closure of $ k $.
	Choose a complex embedding \smash{$ \incto{\kbar}{\CCup} $}.
	Then a normal $k$-variety $ X $ can be reconstructed from the stratified homotopy type of the topological space
	\begin{equation*}
		(X \times_{\Spec k} \Spec \kbar)^{\an}
	\end{equation*}
	along with its action of the absolute Galois group $\Gk$.

	In dimension $1$, for example, a connected, smooth, and complete curve over $k$ is uniquely specified by a genus $g$ and a suitable action of $\Gk$ on a diagram of free groups whose ranks depend on $g$ (see \cref{subsec:examplecurves}).
\end{exm}


\subsection{Technical overview}

This book consists of four parts.
\Cref{part:stratspace,part:topoi,part:strattopoi} reflect the three ingredients necessary to construct the profinite stratified shape and to prove the central \Categorical Hochster Duality Theorem (\Cref{thm:headlineinftyHochster}=\Cref{thm:inftyHochster}).
\Cref{part:stratetale} is then focused applying this machinery to the étale \topoi of schemes.

The first ingredient is a small (and quite elementary) piece of abstract homotopy theory in the study of stratified spaces and profinite stratified spaces.
Most of this work is relatively formal, but one important notion is that of a \textit{spatial décollage}, which is a presheaf on the subdivision of a poset satisfying a Segal condition (\Cref{def:spatialdecollage}).
We prove that the \category of stratified spaces is equivalent to that of spatial décollages via a nerve construction (\Cref{thm:nerveequiv}).
The upshot is that a stratified space can be recovered from its `unglued' form\footnote{whence the term `décollage'} -- a collection of strata and links, suitably organized.

On the toposic side, we need to be able to perform the same ungluing procedure, so that we can recover \atopos $\XX$ from the data of a closed subtopos $\ZZ$, its open complement $\UU$, and the gluing information in the form of the \textit{deleted tubular neighborhood} $\WW$ of $\ZZ$ in $\UU$.
This is the second major ingredient -- \textit{gluing squares} of \topoi, which are certain squares
	\begin{equation*}
		\begin{tikzcd}
			\WW \arrow[r, "q_{\ast}" above] \arrow[d, "p_{\ast}" left] & \UU \arrow[d, "\jlowerstar" right] \arrow[dl, phantom, "\scriptstyle \sigma" below right, "\Longleftarrow" sloped] \\ 
			\ZZ \arrow[r, "\ilowerstar" below] & \XX
		\end{tikzcd}
	\end{equation*}
of geometric morphisms with a noninvertible natural transformation $\sigma$. In order to make sense of this, there are three nontrivial tasks:
\begin{enumerate}[(1)]
	\item We must work -- systematically and \textit{ab initio} -- with \textit{bounded coherent} \topoi. This involves some care, particularly as boundedness and coherence are not stable under the formation of recollements.
	See \Cref{sec:rechighertopoi}.

	\item We must develop the higher categorical analogue of Deligne's \textit{oriented fiber product} \cites{MR3309086}{MR726426}{MR2249998}. 
	The \textit{tubular neighborhood} of $\ZZ$ in $\XX$ is the \textit{evanescent} \topos $\ZZ\orientedtimes_{\XX}\XX$, and the \textit{deleted tubular neighborhood} $\WW$ is then the open subtopos $\ZZ\orientedtimes_{\XX}\UU\subseteq\ZZ\orientedtimes_{\XX}\XX$.
	See \Cref{sec:orientedfiberprod}.

	\item Finally, and most crucially, we must prove a rather delicate \textit{Basechange Theorem for oriented fiber products} (\Cref{thm:BCfororientedfibs}), which ensures that the two gluing functors $\iupperstar\jlowerstar$ and $p_{\ast}q^{\ast}$ agree, at least on truncated objects.
	See \Cref{sec:localtopoi,section:BC}.
\end{enumerate}

We then define \textit{stratified \topoi} in a manner completely analogous to our definition of stratified topological spaces (\Cref{sec:strattopoi}).
Our study of gluing squares now permits us to prove that the \category of \textit{bounded coherent} stratified \topoi are equivalent to a \category of \textit{toposic décollages}, i.e., presheaves of \topoi on the subdivision of a poset that satisfy a kind of \textit{oriented Segal condition} (\Cref{thm:strattopanddecollages}).
This condition ensures that a chain $\{p_0 < \cdots < p_n\} \subset P $ is carried to the iterated oriented fiber product $\XX_{p_0}\orientedtimes_{\XX}\cdots\orientedtimes_{\XX}\XX_{p_n}$ of the strata.
Passing to \proobjects in the base permits us to contemplate stratified \topoi over spectral topological spaces.

Among the bounded coherent stratified \topoi are those in which the strata are Stone \topoi.
These are the \textit{spectral \topoi.}
We prove that these agree with those bounded coherent stratified \topoi in which the truncated coherent objects are exactly the \textit{constructible sheaves} -- i.e., those sheaves that restrict to a lisse sheaf on any stratum.
If $\mbfPi$ is a profinite stratified space, then the stratified \topos $ \Fun(\mbfPi,\Space) $ is spectral in this sense.
As in Lurie's \Categorical Stone Duality, there is a left adjoint to the functor $\goesto{\mbfPi}{\Fun(\mbfPi,\Space)}$, which carries a stratified \topos to its \textit{stratified homotopy type}.

The \Categorical Hochster Duality Theorem (\Cref{thm:headlineinftyHochster}=\Cref{thm:inftyHochster}) now follows from a sequence of three moves:
\begin{enumerate}[(1)]
	\item We reduce to the case of a finite poset $P$.
	This is formal.

	\item We then show that the stratified homotopy type of a spectral \topos stratified by a finite poset can be computed by ungluing to the toposic décollage, forming the homotopy type objectwise to get a spatial décollage, and then regluing to a profinite stratified space.

	\item We then appeal to Lurie's \Categorical Stone Duality Theorem.
\end{enumerate}


\subsection{Open problems}

There are a number of questions we have not answered in this book. 
Here are two. 

\begin{qst*}
	Our work here leaves \Cref{cnj:main} open.
	In effect, it predicts that a large class of \textit{absolute schemes} (see \Cref{dfn:absolute}) can be reconstructed from their stratified étale homotopy types.
\end{qst*}

\begin{qst*}
	We may ask whether one can recover an absolute scheme $ X $ from the profinite stratified space at a finite stage.
	That is, is there a finite constructible stratification $\fromto{X^{\zar}}{P}$ such that for any absolute scheme $Y$, the map
	\begin{equation*}
		\fromto{\Map_{\Sch_k}(X,Y)\simeq\Map_{\BGk}(\Gal(Y),\Gal(X))}{\uppi_0\Map_{\BGk}(\Gal(Y),\Gal(X/P))}
	\end{equation*}
	is a bijection?
	Here $ \Gal(X/P) $ denotes the profinite stratified shape of $ X_{\et} $ with respect to the stratification $\fromto{X^{\zar}}{P}$.
	One might expect that it suffices to choose stratification in which the strata in $ X $ are strongly hyperbolic Artin neighborhoods \cite[Exposé XI, §§2 \& 3]{MR50:7132}; at this point, we do not know.
\end{qst*}


\subsection{Acknowledgements}

The Université Montpellier has recently released a collection of notes of Groth\-endieck \cite{GrothendieckMontpelier}, including `Cote n\textsuperscript{o} 151: Espaces stratifiés', in which he develops some elements of stratified topos theory and some elements of an attached shape theory, to which he referred in his \emph{Esquisse d'un Programme} \cite[p. 36]{MR1483107}.
It is not clear to us how much of the work here he anticipated.

We have used the framework and results in Jacob Lurie's three big books \cites{HTT}{HA}{SAG} everywhere here. 
The impact of his ideas here is obvious and extensive.
We are also grateful to him for his very helpful answers to a number of technical questions we pelted him with over the course of this project.

Much of our understanding of stratified spaces is directly the result of our conversations with David Ayala, who independently developed the `spatial décollage' perspective on stratified spaces.
We are exceedingly grateful to him for sharing with us a portion of these ideas.
Ayala's thinking has had a tremendous influence on us, and conversations with him in the first phase of this project have been vital to our work.
Additionally, he has read some early editions of this monograph, spotted some mistakes, and helped us improve our writing.

Dustin Clausen and Jacob Lurie have each developed a portion of the topos-theoretic material here. In particular, it seems that they were each already aware of the \Categorical Hochster Duality Theorem.
We are grateful to them both for sharing their insights on this matter.

We are grateful to Marc Hoyois for carefully reading a draft of this monograph and for providing us with helpful feedback.

We are grateful to Ko Aoki and Harry Gindi for carefully reading this monograph and for spotting a number of typos and  pointing out some mistakes.

We thank Chang-Yeon Chough for helpful conversations on the relationship between the nonabelian proper base\-change theorem in algebraic geometry and our base\-change theorem for oriented fiber products.

Some of the ideas here were originally conceived during the Bourbon Seminar at MIT. 
We are grateful to our fellow participants -- Marc Hoyois, Denis Nardin, Sune Precht Reeh, and Jay Shah.

The first- and third-named authors are grateful to the Isaac Newton Institute in Cambridge and to the Mathematical Sciences Research Institute in Berkeley, whose hospitality they enjoyed while completing some of this monograph.

The third-named author gratefully acknowledges support from both the MIT Dean of Science Fellowship and the National Science Foundation Graduate Research Fellowship
under Grant \#112237.
They would also like to thank Alex Sear for her generous hospitality during his visits to the University of Edinburgh during which much of this monograph was written.

This text is partially based upon work supported by the National Science Foundation under Grant \#DMS-1440140, while the first- and third-named authors were in residence at the Mathematical Sciences Research Institute in Berkeley, California, during the Spring 2020 semester.
During this time, the first author was supported by the Simons Foundation.


\subsection{Terminology \& notations}


\subsubsection{Set theoretic conventions}

\begin{nul}
	Recall that if $\delta$ is a strongly inaccessible cardinal (which we always assume to be uncountable), then the set $\VV_{\delta}$ of all sets of rank strictly less than $\delta$ is a Grothendieck universe of rank and cardinality $\delta$ \cite[Exposé I, Appendix]{MR50:7130}.
	Conversely, if $\VV$ is a Grothendieck universe that contains an infinite cardinal, then $\VV=\VV_{\delta}$ for some strongly inaccessible cardinal $\delta$.

	In order to deal precisely and simply with set-theoretic problems arising from the consideration of `large' collections, we append to \textsc{zfc} the Axiom of Universes (\textsc{au}).
	This asserts that any cardinal is dominated by a strongly inaccessible cardinal.

	We write $\updelta_0$ for the smallest strongly inaccessible cardinal.%
	\index[notation]{delta@$\updelta_0, \updelta_1, \ldots$}
	Now \textsc{au} implies the existence of a hierarchy of strongly inaccessible cardinals
	\begin{equation*}
		\updelta_0 < \updelta_1 < \updelta_2 < \cdots \comma
	\end{equation*}
	in which for each ordinal $\alpha$, the cardinal $\updelta_{\alpha}$ is the smallest strongly inaccessible cardinal $\updelta_{\alpha}$ that dominates $\updelta_{\beta}$ for any $\beta<\alpha$.\footnote{Thus $\VV_{\updelta_{\alpha}}$ models \textsc{zfc} plus the axiom `the set of strongly inaccessible cardinals is order-isomorphic to $\alpha$'.}

	We certainly will not use the full strength of \textsc{au}; the existence of only $\updelta_0$, $\updelta_1$, and $\updelta_2$ suffices for our work here.
	At the cost of some circumlocutions, one could even get away with \textsc{zfc} alone.
\end{nul}

\begin{nul}
	We write $ \NNup $ for the poset of \defn{nonnegative integers}.%
	\index[notation]{N@$\NNup, \NNup^{\ast}, \NNrhd$}
	We write $ \NNup^{\ast} \coloneq \NNup \smallsetminus \{0\} $, and $ \NNrhd \coloneq \NNup \cup \{\infty \} $.
\end{nul}


\subsubsection{Higher categories}

Throughout this book we use the language and tools of \category theory, as defined by Boardman--Vogt and developed by Joyal \cites{MR1935979}{Joyal}{Joyal08} and Lurie \cites{HTT}{HA}{SAG}.
There are many places in this book where we do not know how construct the functors we're interested in without \category theory.
It would be prohibitively difficult to prove even the most basic properties about our constructions using a different framework.
In particular, we assume that the reader is familiar with the basics of \category theory, (co)limits \cite[\HTTsubsec{1.2.13} \& \HTTch{4}]{HTT}, adjunctions \cite[\HTTsec{5.2}]{HTT}, Kan extensions \cite[\HTTsec{4.13}]{HTT}, (co)cartesian fibrations \cite[\HTTch{2}]{HTT}, and \topoi \cite[\HTTch{6}]{HTT}.

We also use the language of \category theory in a mostly model-independent manner, which means that some standard terminology is used with a slightly different meaning than usual. 
For example, a full subcategory is always assumed to be closed under equivalences, a cartesian fibration is any functor that is equivalent to a cartesian fibration in the sense of \cite[\HTTsec{2.4}]{HTT}, etc. 
We also simply regard categories as \categories, namely those with discrete mapping spaces (up to equivalence).

\begin{nul}\label{nul:conventionshighercats}
	We will generally follow the terminological and notational conventions of Lurie's trilogy \cites{HTT,HA,SAG}.
	In particular:
	\begin{itemize}


		\item For each integer $ n \geq 0 $, we write $ [n] $\index[notation]{n@$[n]$} for the poset $ \{0 < \cdots < n \} $.

		\item Let $\delta$ be a strongly inaccessible cardinal.
		We say that a set, group, simplicial set, \category, ring, etc., is \textit{$\delta$-small}%
		\footnote{The adverb `essentially' is often deployed in this situation.}%
		\index[terminology]{delta-small@$\delta$-small}
		if it is equivalent (in whatever appropriate sense) to one that lies in $\VV_{\delta}$.
		We abbreviate \textit{$\updelta_0$-small} to \defn{small}.\index[terminology]{small}

		\item \Acategory $C$ is \defn{locally $\delta$-small}%
		\index[terminology]{delta-small@$\delta$-small!locally}	
		if and only if, for any objects $x,y\in C$, the mapping space $\Map_C(x,y)$ is $\delta$-small.
		We abbreviate \textit{locally $\updelta_0$-small} to \defn{locally small}.\index[terminology]{small!locally}

		\item Accessibility%
			\index[terminology]{accessible}
			\index[terminology]{category@\category!accessible}
		of \categories and functors and presentability of \categories will always refer to accessibility and presentability with respect to some $\updelta_0$-small cardinal.
			\index[terminology]{presentable}
			\index[terminology]{category@\category!presentable}
		Please observe that an accessible \category is always $\updelta_1$-small and locally $\updelta_0$-small.

		\item We will use the terms \emph{\groupoid} or \emph{space} interchangeably for \acategory in which every morphism is invertible.%
			\footnote{Clausen and Scholze have suggested the term \defn{anima} for this notion \cite[\S5.1.4]{CesnaviciusScholze}, following Beilinson's use of the term \defn{animation}.}
			\index[terminology]{groupoid@\groupoid}
			\index[terminology]{space}

		\item Let $\delta$ be a strongly inaccessible cardinal.
		Then we write $\Space_{\delta}$%
		\index[notation]{Space@$\Space, \Space_{\delta}$}
		for the \category of $\delta$-small spaces and $\Cat_{\infty,\delta}$ for the \category of $\delta$-small \categories.%
		\index[notation]{Cat@$\Cat_{\infty}, \Cat_{\infty,\delta}$}
		In particular, we shall write $\Space$ and $ \Cat_{\infty} $ for $\Space_{\updelta_0}$ and $ \Cat_{\infty,\updelta_0} $, respectively.

		\item Let $C$ be \acategory and $ W \subseteq \Mor(C) $ a set of morphisms of $C$.
		Then we write $W^{-1}C$\index[notation]{Winverse@$W^{-1}C$} for the result of inverting the morphisms of $W$.
		If $\delta$ is an inaccessible cardinal for which $C$ is $\delta$-small, then $W^{-1}C$ is $\delta$-small as well.
		This \category comes equipped with a functor $\fromto{C}{W^{-1}C}$ that, for any \category $D$, induces a fully faithful functor
		\begin{equation*}
			\incto{\Fun(W^{-1}C,D)}{\Fun(C,D)}
		\end{equation*}
		that identifies $\Fun(W^{-1}C,D)$ with the full subcategory spanned by those functors $\fromto{C}{D}$ that carry the morphisms of $W$ to equivalences in $D$.
		See \cite[\S7.1]{MR3931682}.
	\end{itemize}
\end{nul}

\begin{nul}\label{nul:nCat}
	For any $n\in\NNrhd$, write $\Cat_n \subseteq \Cat_{\infty} $\index[notation]{Catn@$\Cat_n$} for the full subcategory spanned by the $(n,1)$-categories;
	that is, \acategory $C$ lies in $\Cat_n$ if and only if for any $x,y\in C$, the \groupoid $\Map_C(x, y)$ is equivalent to an $(n-1)$-groupoid.
	In particular, $\Cat_0 \simeq \Pos $, the $1$-category of partially ordered sets.

	The inclusion $\Cat_n \subseteq \Cat_{\infty} $ admits a left adjoint $\ho_n$ \cite{MR4042793}.
	\index[notation]{ho@$\ho_n$}
	If $C$ is a \category, then the unit $C \to \ho_n(C)$ exhibits $\ho_n(C)$ as the $n$-categorical truncation, so that the objects of $\ho_n(C)$ are exactly those of $C$ and whose mapping spaces are defined by the condition that the map
	\begin{equation*}
		\Map_{C}(x, y) \to \Map_{\ho_n(C)}(x, y)
	\end{equation*}
	exhibits $\Map_{\ho_n(C)}(x, y)$ as the $(n-1)$-truncation of $\Map_{C}(x, y)$.
	The $1$-categorical truncation $\ho_1(C)$ is also known as the \defn{homotopy category} of $C$.
	The $0$-categorical truncation $\ho_0(C)$ is equivalent to the preorder whose elements are the equivalence classes of objects of $C$ in which $x\leq y$ if and only if there exists a morphism $x \to y$.
\end{nul}


\subsubsection{Proöbjects in higher categories}

\begin{nul}
	We say that a $\updelta_0$-small \category $A$ is \defn{inverse}\index[terminology]{category@\category!inverse}\index[terminology]{inverse!\category} if and only if its opposite \category $A^{\op}$ is filtered.
	Hence an \defn{inverse system} in \acategory $C$ is a functor $\fromto{A}{C}$ from an inverse \category $A$, and an \defn{inverse limit}\index[terminology]{limit!inverse}\index[terminology]{inverse!limit} is a limit of an inverse system.
\end{nul}

\begin{nul}
	Let $ C $ be \acategory.
	The \category $ \Pro(C) $\index[notation]{ProC@$ \Pro(C) $} of \defn{\proobjects in $C$}\index[terminology]{pröobject} and the Yoneda embedding%
	\footnote{The Hiragana character `\kern-0.6em{\begin{CJK}{UTF8}{min} \CJKfamily{goth} ょ \end{CJK}}\kern-0.3em' is pronounced `yo'.} 
	\begin{equation*}
		\yo\colon\incto{C}{\Pro(C)} \comma \index[notation]{yo@$\yo$}
	\end{equation*}
	are defined by the following universal property
	\begin{itemize}
		\item The \category $ \Pro(C) $ admits $ \updelta_0 $-small inverse limits.

		\item For any \category $ D $ with $ \updelta_0 $-small inverse limits, composition with $ \yo $ induces an equivalence 
		\begin{equation*}
			\equivto{\Fun^{\inv}(\Pro(C),D)}{\Fun(C,D)} \comma
		\end{equation*}
		where $ \Fun^{\inv}(\Pro(C),D) \subset \Fun(\Pro(C),D) $ full subcategory of functors that preserve $\updelta_0 $-small inverse limits.
	\end{itemize}

	The existence of $ \Pro(C) $ is a special case of (the dual of) \HTT{Proposition}{5.3.6.2}.
\end{nul}

\begin{nul}
	The formation of \proobjects is formally dual to the formation of indobjects: for any \category $ C $, we have a natural identification $\Pro(C)^{\op} \simeq \Ind(C^{\op})$.
\end{nul}

\begin{nul}
	For any accessible \category $C$ that with finite limits, the \category $ \Pro(C) $ admits an explicit description: $ \Pro(C) $ is equivalent to the full subcategory of $ \Fun(C,\Space)^{\op} $ spanned by the left exact accessible functors \SAG{Proposition}{A.8.1.6}.
	Under this identification, the Yoneda embedding $ \yo \colon \fromto{C}{\Pro(C)} $ is the restriction of the opposite of the Yoneda embedding $ \incto{C^{\op}}{\Fun(C,\Space)} $

	If $X\colon\fromto{A}{C}$ is an inverse system, then its limit in $\Pro(C)$ is the functor
	\begin{equation*}
		\goesto{Y}{\colim_{\alpha \in A^{\op}}\Map_C(X_{\alpha},Y)} \period
	\end{equation*}
	We abuse notation and denote this \proobject by $X=\{X_{\alpha}\}_{\alpha \in A}$.
	Every \proobject of $C$ can be exhibited in this manner, and for \proobjects $X=\{X_{\alpha}\}_{\alpha \in A}$ and $Y=\{Y_{\beta}\}_{\beta \in B}$ we obtain the familiar formula
	\begin{equation*}
		\Map_{\Pro(C)}(X,Y)\simeq\lim_{\beta \in B}\colim_{\alpha \in A^{\op}}\Map_C(X_{\alpha},Y_{\beta}) \period
	\end{equation*}
	We thus often speak of objects of $\Pro(C)$ as if they \textit{were} inverse systems.
	In particular, we call a \proobject $ X $ \defn{constant} if and only if $ X $ lies in the essential image of $\yo$;
	equivalently, $ X $ is constant if and only if, as a functor $\fromto{C}{\Space}$, $ X $ preserves inverse limits.
\end{nul}

\begin{nul}[proëxistent left adjoint]\label{nul:proexistentadjoint}
	Let $ \delta $ be an inaccessible cardinal, $C$ a locally $\delta$-small \category that admits all $\updelta_0$-small limits, $D$ an accessible \category that admits finite limits, and $u\colon\fromto{D}{C}$ a left exact functor.
	The functor $ u $ does not generaly admit a left adjoint, but passage to \proobjects often repairs this.
	Indeed, one may extend $u$ to a (unique) functor $U\colon\fromto{\Pro(D)}{C}$ that preserves inverse limits, and in the other direction, one may consider the composite
	\begin{equation*}
		F \coloneq u^{\ast} \circ \yo \colon C\to\Fun(C,\Space_{\delta})^{\op}\to\Fun(D,\Space_{\delta})^{\op}
	\end{equation*}
	of the Yoneda embedding $ \yo $ with the restriction along $ u $.
	The functor $F$ carries an object $ c \in C$ to the assignment $ \goesto{d}{\Map_C(c,u(d))}$.
	We have to make two set-theoretic assumptions:
	\begin{itemize}
		\item Assume that for any object $ c \in C$ and any object $ d \in D$, the space $\Map_C(c,u(d))$ is $\updelta_0$-small.

		\item Assume that for any object $ c \in C $, there exists a regular cardinal $ \delta' < \updelta_0 $ such that for any $ \delta' $-filtered diagram $ d_{\ast} \colon\fromto{A}{D} $, the natural map
		\begin{equation*}
			\fromto{\colim_{\alpha \in A} \Map_C(c,u(d_\alpha))}{\Map_C(c,\colim\nolimits_{\alpha \in A} u(d_\alpha))}
		\end{equation*}
		is an equivalence.
	\end{itemize}
	In this case, the functor $ F $ lands in $ \Pro(D) $, and $ F $ is left adjoint to $ U $.
	We shall call $ F $ the \defn{proëxistent left adjoint to $u$}.%
	\index[terminology]{proëxistent left adjoint}%
	\index[terminology]{left adjoint!proëxistent}
	If $u$ \textit{already} admits a left adjoint $ f $, then $ F $ lands in $ D $ and coincides with $ f $.
\end{nul}


\subsubsection{Recollements}

\begin{nul}[oriented fiber product of \categories]\label{nul:orientedfpinCat}
	Given functors $ F \colon \fromto{X}{Z} $ and $ G \colon \fromto{Y}{Z} $ between \categories, we write
	\begin{equation*}
		\commacat{X}{Z}{Y} \coloneq X \crosslimits_{\Fun(\{0\},Z)} \Fun([1],Z) \crosslimits_{\Fun(\{1\},Z)} Y
	\end{equation*}
	for the \defn{oriented fiber product}\index[terminology]{oriented fiber product!of categories@of \categories}\index[notation]{XZY@$\commacat{X}{Z}{Y}$} of \categories.
\end{nul}

\begin{nul}\label{nul:orientedfpinCatesssmall}
	Let $ X $ and $ Y $ be $ \updelta_0 $-small \categories, let $ Z $ be a locally $ \updelta_0 $-small \category, and let $ F \colon \fromto{X}{Z} $ and $ G \colon \fromto{Y}{Z} $ be functors.
	Write $ Z' \subset Z $ for the full subcategory spanned by those objects in the image of $ F $ or the image of $ G $.
	Then $ Z' $ is $ \updelta_0 $-small and the oriented fiber product $ \commacat{X}{Z}{Y} $ is equivalent to $ \commacat{X}{Z'}{Y} $, whence $ \commacat{X}{Z}{Y} $ is $ \updelta_0 $-small.
\end{nul}

\begin{nul}[see \protect{\cite[\HAsec{A.8}]{HA}}]\label{rec:recollement}
	Let $C$ be \acategory that admits finite limits. 
	Then two functors $\ilowerstar\colon\fromto{C_Z}{C}$ and $\jlowerstar\colon\fromto{C_U}{C}$ exhibit $C$ as a \defn{recollement}\index[terminology]{recollement} of $C_Z$ and $C_U$ if and only if the following conditions are satisfied.
	\begin{itemize}
		\item Both $\ilowerstar$ and $\jlowerstar$ are fully faithful.

		\item There are left exact left adjoints $\iupperstar$ and $\jupperstar$ to the functors $\ilowerstar$ and $\jlowerstar$, respectively.

		\item The functor $\jupperstar\ilowerstar$ is constant at the terminal object of $ C_U $.

		\item The functor $(\iupperstar,\jupperstar)\colon\fromto{C}{C_Z\times C_U}$ is conservative.
	\end{itemize}
	We refer to the \category $C_Z$ as the \defn{closed subcategory}, the \category $C_U$ as the \defn{open subcategory}, and the functor $\iupperstar\jlowerstar\colon\fromto{C_U}{C_Z}$ as the \defn{gluing functor}.

	If $C$ is the recollement of \categories $C_Z$ and $C_U$, then $C_Z$ is canonically equivalent to the \emph{kernel} of $\jupperstar$ (i.e., the full subcategory spanned by those objects $x$ such that $\jupperstar(x)$ is terminal in $C_U$).

	If $C_Z$ and $C_U$ are any \categories with finite limits, and $\phi\colon\fromto{C_U}{C_Z}$ is a left exact functor, then we write
	\begin{equation*}
		C_Z\orientedcup^{\phi}C_U\coloneq\commacat{C_Z}{C_Z}{C_U} \period
	\end{equation*}
	The projections
	\begin{equation*}
		\iupperstar\colon\fromto{C_Z\orientedcup^{\phi}C_U}{C_Z} \andeq \jupperstar\colon\fromto{C_Z\orientedcup^{\phi}C_U}{C_U}
	\end{equation*}
	admit right adjoints
	\begin{equation*}
		\ilowerstar\colon\fromto{C_Z}{C_Z\orientedcup^{\phi}C_U} \andeq \jlowerstar\colon\fromto{C_U}{C_Z\orientedcup^{\phi}C_U}
	\end{equation*}
	that together exhibit $C_Z\orientedcup^{\phi}C_U$ as a recollement of $C_Z$ and $C_U$. 
	Furthermore, every recollement is of this form, where $\phi$ is the gluing functor.

	If $C_Z$ contains an initial object, then $\jupperstar$ admits a further left adjoint $j_!$, so in this case we may also write $ j^{!} \coloneq \jupperstar$. If, moreover, $C$ contains a zero object (whence so do $C_Z$ and $C_U$), then $\ilowerstar$ admits a further right adjoint $i^{!}$, so in this case we may also write $i_{!}\coloneq \ilowerstar$.
\end{nul}

\begin{nul}\label{nul:truncatednessinrecoll}
	Let $ C $ be \acategory with finite limits and let $\ilowerstar\colon\incto{C_Z}{C}$ and $\jlowerstar\colon\incto{C_U}{C}$ be functors which exhibit $C$ as a recollement of $C_Z$ and $C_U$
	Then for any integer $ n \geq -2 $, since the left exact functor
	\begin{equation*}
		(\iupperstar,\jupperstar) \colon\fromto{C}{C_Z\times C_U}
	\end{equation*}
	is conservative, a morphism $ f $ of $ C $ is $ n $-truncated if and only if $ \iupperstar(f) $ and $ \jupperstar(f) $ are both $ n $-truncated.
\end{nul}


\subsubsection{Relative adjunctions}

\begin{nul}\label{rec:reladj}
	Given a commutative triangle of \categories
	\begin{equation*}
		\begin{tikzcd}
			C \arrow[dr, "p"'] & & D \arrow[dl, "q"] \arrow[ll, "G"'] \\ 
			 & E &
		\end{tikzcd}
	\end{equation*}
	where $ p $ and $ q $ are isofibrations, we say that $ G $ \defn{admits a left adjoint relative to $ E $}\index[terminology]{relative adjunction} if the following condition holds:
	\begin{itemize}
		\item There exists a functor $ F \colon \fromto{C}{D} $ and a natural transformation $ \unit \colon \fromto{\id_{C}}{GF} $ which exhibits $ F $ as a left adjoint to $ G $ such that $ p\unit \colon \fromto{p}{pGF \equivalent qF} $ is an equivalence in $ \Fun(C,E) $.
	\end{itemize}

	In this situation, given a functor $ \fromto{E'}{E} $, define $ C_{E'} \coloneq C \cross_E E' $, $ D_{E'} \coloneq D \cross_E E' $, and write $ G_{E'} \colon \fromto{D_{E'}}{C_{E'}} $ and $ F_{E'} \colon \fromto{C_{E'}}{D_{E'}} $ for the induced functors on pullbacks.
	Then the induced natural transformation $ \fromto{\id_{C_{E'}}}{G_{E'}F_{E'}} $ exhibits $ F_{E'} $ as a left adjoint to $ G_{E'} $ relative to $ E' $.
	See \HA{Proposition}{7.3.2.5}.

	If $ p $ and $ q $ are cartesian fibrations, $ G $ admits a left adjoint relative to $ E $ if and only if the following conditions hold:
	\begin{itemize}
		\item For every object $ e \in E $, the induced functor $ G_{e} \colon \fromto{D_{e}}{C_{e}} $ admits a left adjoint.

		\item The functor $ G $ carries $ p $-cartesian morphisms in $ D $ to $ q $-cartesian morphisms in $ C $.
	\end{itemize}
	See \HA{Proposition}{7.3.2.6}.
	In this case, if $f\colon a\to b$ is a morphism of $E$, then there is a natural equivalence
	\begin{equation*}
		f^{\ast}G_b\simeq G_af^{\ast} \period
	\end{equation*}

	Dually, if $ p $ and $ q $ are cocartesian fibrations, $ G $ admits a left adjoint relative to $ E $ if and only if the following (somewhat more complicated) conditions hold:
	\begin{itemize}
		\item For every object $ e \in E $, the induced functor $ G_{e} \colon \fromto{D_e}{C_{e}} $ admits a left adjoint $ F_{e} $.

		\item Let $ c \in C $ and $ \alpha \colon \fromto{e}{e'} $ be a morphism of $ e $ where $ e \equivalent p(c) $.
		Let $ \widetilde{\alpha} \colon \fromto{F_e(c)}{d} $ be a $ q $-cocartesian morphism in $ D $ lying over $ \alpha $, and let $ \beta \colon \fromto{c}{G(d)} $ be the composite $ \beta \coloneq G(\widetilde{\alpha}) \of \unit(c) $.
		Choose a factorization of $ \beta $ as
		\begin{equation*}
			\begin{tikzcd}[sep=1.5em]
				\beta \colon c \arrow[r, "\beta'"] & c' \arrow[r, "\beta''"] & G(d) \comma
			\end{tikzcd}
		\end{equation*}
		where $ \beta' $ is a $ p $-cocartesian morphism lifting $ \alpha $ and $ \beta'' $ is a morphism in $ C_{e'} $.
		Then $ \beta'' $ induces an equivalence $ \fromto{F_{e'}(c')}{d} $ in the \category $ D_{e'} $.
	\end{itemize}
	See \HA{Proposition}{7.3.2.11}.
	In this case, if $f\colon a\to b$ is a morphism of $E$, then there is a natural equivalence
	\begin{equation*}
		G_bf_! \simeq f_!G_a \period
	\end{equation*}
\end{nul}


\subsubsection{Schemes}

\begin{nul}\label{nul:coherentqcqs}
	Following the Grothendieck school \cites[Exposé VI, Exemples 1.22]{MR50:7131}[Exposé XVII, 0.12]{MR50:7132}{MR3309086}{MR2249998}, we say that scheme $ X $ is \defn{coherent}\index[terminology]{coherent!scheme}\index[terminology]{scheme!coherent} if and only if $ X $ is quasicompact and quasiseparated.
	All schemes considered in this book will be coherent.
\end{nul}


\newpage

\part{Stratified spaces}\label{part:stratspace}

In \Cref{sec:posetrec}, we recall the \textit{Alexandroff} topology on a poset.
Just as the category of profinite sets can be identified with that of Stone topological spaces, the category of profinite posets be identified with that of \textit{spectral} topological spaces.
These topological spaces allow one to define stratifications of topological spaces over finite and profinite posets.

The homotopy theory of spaces stratified over a poset $P$ is introduced in \Cref{sec:hothystratspaces}.
These can be described in two equivalent ways: as a \category with a conservative functor to a poset and as a \textit{spatial décollage} -- a diagram of spaces indexed by the subdivision of $P$ that satifies a Segal condition.
These descriptions are equivalent, and they both permit one to study the Postnikov tower of stratified spaces and identify finiteness conditions on them.


\section[Aide-mémoire on the topology of posets \& profinite posets]{Aide-mémoire on the topology of posets \& profinite \\ posets}\label{sec:posetrec}

In this short chapter we review the topologies on posets, and stratifications of topological spaces by posets.
We also recall Hochster's Theorem classifying spectral topological spaces in terms of \proobjects in finite posets (\Cref{thm:Hoschterthesis}).


\subsection{Alexandroff Duality}

We start by reviewing the relationship between topological spaces and preorders.
The first thing to note is that every topological space gives rise to a preorder.

\begin{dfn}\label{def:specializationpreorder}
	Let $ T $ be a topological space.
	The \defn{specialization preorder}\index[terminology]{specialization preorder} on $ T $ is the preorder on the underlying set of $ T $ with order relation $ x \leq y $ if and only if \smash{$ x \in \overline{\{y\}} $}.
	We denote the specialization preorder on $ T $ by $ \Special(T) $.
\end{dfn}

Every preorder also gives rise to a topological space.

\begin{dfn}\label{dfn:sievepreorder}\label{dfn:Alexandroff}
	Let $ P $ be a preorder.
	\begin{enumerate}[(\ref*{dfn:sievepreorder}.1)]
		\item We say that a subset $ U \subseteq P $ a \defn{cosieve}\index[terminology]{cosieve} if for any points $p,q\in P$ such that $p\leq q$, if $ p\in U $ then $ q\in U $.

		\item We say that a subset $ Z \subseteq P $ is a \defn{sieve}\index[terminology]{sieve} if for any points $p,q\in P$ such that $p\leq q$, if $q\in Z$ then $p\in Z$.

		\item We say that subset $ W \subseteq P $ is an \defn{interval}\index[terminology]{interval} if for any points $p,q,r\in P$ such that $p\leq q\leq r$, if $p,r\in W$ then $q\in W$.
	\end{enumerate}

	The \defn{Alexandroff topology}\index[terminology]{Alexandroff topology} on $ P $ is the topology on the underlying set of $ P $ in which a subset $ U \subseteq P $ is open if and only if $ U $ is a cosieve. 
	We write $ \Alex(P) $ or simply $ P $ for the set $ P $ equipped with the Alexandroff topology.
\end{dfn}

\begin{nul}
	Note that, a subset $Z\subseteq P$ is closed if and only if $ Z $ is a sieve, and subset $ W \subseteq P $ is locally closed if and only if $ W $ is an interval.
\end{nul}

Alexandroff topologies admit a well-known characterization.

\begin{prp} 
	The following are equivalent for a topological space $T$.
	\begin{itemize}
		\item The space $T$ is \emph{finitely generated}; that is, a subset $U\subseteq T$ is open if for any finite topological space $ F $ and continuous map $f\colon\fromto{F}{T}$, the inverse image $f^{-1}(U)$ is open.

		\item The union of any collection of closed subsets of $T$ is again closed.

		\item The topology on $T$ coincides with the Alexandroff topology attached to the specialization preorder on $T$.
	\end{itemize}
\end{prp}

\begin{nul}[{Alexandroff Duality\index[terminology]{Alexandroff Duality}}] 
	The formation of the Alexandroff topology defines an equivalence of categories
	\begin{equation*}
		\Alex \colon \equivto{\Preord}{\TSpc^{\fg}}
	\end{equation*}
	from the category of preorders to the category of finitely generated topological spaces.
	The inverse \smash{$ \Special \colon \equivto{\TSpc^{\fg}}{\Preord} $} given by taking the specialization preorder.
	In particular, the functors $ \Alex$ and $ \Special $ restrict to an equivalence between the category of finite preorders and the category of finite topological spaces.

	The functors $ \Alex $ and $ \Special $ also restrict to an equivalence between:
	\begin{itemize}
		\item the category of posets and the category of $ \Tup_0 $ finitely generated topological spaces,

		\item the category of \textit{noetherian} preorders (i.e., those for which every nonempty subset contains a maximal element) and the category of quasi-sober finitely generated topological spaces, and thus

		\item the category of noetherian posets and the category of sober finitely generated topological spaces.
	\end{itemize}
\end{nul}

\begin{ntn}
	Let $P$ be a preorder.
	For any subset $W\subseteq P$, we write $P_{\geq W}$ for the cosieve generated by $W$, which is the smallest open neighborhood of $W$.
	Dually, we write $P_{\leq W}$ for the sieve generated by $W$, which is the closure of $W$.

	In particular, we will refer to the sets of the form%
	\index[notation]{Pp@$P_{\geq p}, P_{\leq p}, P_{>p}, P_{<p}$}
	\begin{equation*}
		P_{\geq p} = \{q \in P \, | \, q \geq p \} 
	\end{equation*}
	for $p\in P$, which are sometimes called the \defn{principal open sets},
	and the sets of the form
	\begin{equation*}
		P_{\leq p} = \{q \in P \, | \, q \leq p \} \comma
	\end{equation*}
	which are sometimes called the \defn{principal ideals}.

	Similarly, we write
	\begin{equation*}
		P_{>p} \coloneq P_{\geq p} \smallsetminus \{p\} \andeq P_{<p} \coloneq P_{\leq p} \smallsetminus \{p\} \period
	\end{equation*}
\end{ntn}

\begin{nul}
	A poset is quasicompact in the Alexandroff topology if and only if its set of minimal elements is finite and limit-cofinal.
	A monotone map $f\colon\fromto{Q}{P}$ of posets is quasicompact if and only if, for any $p\in P$, the poset $f^{-1}(P_{\geq p})$ is quasicompact.
\end{nul}

\begin{ntn}
	Let $ P $ be a poset.
	We call a nonempty linearly ordered finite subset $ \Sigma \subset P $ a \defn{chain}\index[terminology]{chain} in $ P $.
	We write $ \sd(P) $\index[notation]{sd@$\sd$} for the \defn{subdivision}\index[terminology]{subdivision} of $ P $; that is, $ \sd(P) $ is the poset of chains $ \Sigma \subset P $ ordered by containment.
	
	Note that there is a natural forgetful functor $\fromto{\sd(P)}{\DDelta}$.

	Let $\Sigma\subset P$ be a chain.
	Then every closed subset $Z\subset\Sigma$ is again a chain, and we denote the inclusion $ Z \subset \Sigma $ by $i_{Z\subset\Sigma}$%
	\index[notation]{iZ@$i, i_{Z \subset\Sigma}$}
	(or simply $i$ if $Z$ and $\Sigma$ are clear from the context).
	Dually, every open subset $U\subset\Sigma$ is also a chain, and we denote the inclusion $ U \subset \Sigma $ by $j_{Z\subset\Sigma}$%
	\index[notation]{jU@$j, j_{U \subset\Sigma}$}
	(or simply $j$ if $U$ and $\Sigma$ are clear from the context).

	In more general situations, we write $e_{W\subset\Sigma} \colon \incto{W}{\Sigma}$%
	\index[notation]{eW@$e, e_{W \subset\Sigma}$}
	for an inclusion $W\subset\Sigma$ that is not known to be be either closed or open.
\end{ntn}


\subsection{Stratifications of topological spaces}

The theory of stratified topological spaces can now be neatly organized in terms of topological spaces equipped with a continuous map to a poset in the Alexandroff topology.

\begin{dfn}\label{dfn:strattopspace}
	A \defn{stratification}\index[terminology]{stratification!of a topological space} of a topological space $T$ by a poset $P$ is a continuous map $f\colon\fromto{T}{P}$.
	For any point $p\in P$, we write%
	\index[notation]{Tp@$T_{\geq p}, T_{>p}, T_{\leq p}, T_{<p}, T_p$}
	\begin{align*}
		T_{\geq p} &\coloneq f^{-1}(P_{\geq p}) \comma \\
		T_{>p} &\coloneq f^{-1}(P_{>p}) \comma \\
		T_{\leq p} &\coloneq f^{-1}(P_{\leq p}) \comma \\
		T_{<p} &\coloneq f^{-1}(P_{<p}) \comma \\
		T_p &\coloneq T_{\geq p}\cap T_{\leq p} \period
	\end{align*}
	The subspaces $T_{\geq p}$ and $T_{>p}$ are open in $T$, and $T_{\leq p}$ and $T_{<p}$ are closed in $T$.
	The subspace $ T_p \subset T $ is locally closed; we call $ T_p $ the \smash{$p$\upperth} \defn{stratum}\index[terminology]{stratum} of $ f \colon \fromto{T}{P} $.

	We say that the stratification $f\colon\fromto{T}{P}$ is \defn{nondegenerate}\index[terminology]{stratification!nondegenerate} if each stratum $T_p$ is nonempty, and for all $p,q\in P$ such that $p\leq q$, we have \smash{$T_p\subseteq\overline{T}_q$}.

	We say that a stratification is \defn{finite}\index[terminology]{stratification!finite} if and only if its base poset is finite.
	We say that the stratification $f\colon\fromto{T}{P}$ is \defn{constructible}\index[terminology]{stratification!constructible} if and only if, for all $p\in P$, the open subset $T_{\geq p}\subset T$ is retrocompact -- i.e., for every quasicompact open $ V\subset T$, the intersection $ V \intersect T_{\geq p} $ is quasicompact.
\end{dfn}


\subsection{Hochster Duality}\label{subsec:ordinaryHochsterdual}

In this section we recall Hochster's characterization of topological spaces that arise as the Zariski space of a coherent scheme in terms of \proobjects in the category of finite posets.
This characterization provides a convenient way to study stratifications of schemes.

\begin{ntn}\label{ntn:FC}
	For any topological space $T$, we write $\FC(T)$\index[notation]{FC@$ \FC $} for the $1$-category of finite, nondegenerate, constructible stratifications $\fromto{T}{P}$. 
	We abuse notation and write merely $P$ for an object $\fromto{T}{P}$ of this category, leaving the structure morphism implicit.
	The $1$-category $\FC(T)$ is, up to equivalence, an inverse poset in which $P\leq Q$ if and only if $P$ \textit{refines} $Q$.
	That is, $P\leq Q$ if and only if the structure morphism $\fromto{T}{Q}$ factors through the structure morphism $\fromto{T}{P}$.
\end{ntn}

\begin{dfn}
	A topological space $S$ is \defn{spectral}%
	\index[terminology]{topological space!spectral}\index[terminology]{spectral!topological space}.
	\footnote{Others call such topological spaces \textit{coherent}; see for example \cites[\SAGseclink{A.1}]{SAG}[Chapter III \S 3.4 \& p. 78]{MR861951}.
	We use Hochster's algebro-geometric terminology \cites{Hochster:thesis,Hochster:primeideal}.}
	if and only if $ S $ is the limit of its finite, nondegenerate, constructible stratifications; that is, if and only if the natural map
	\begin{equation*}
		\fromto{S}{\lim_{P\in \FC(S)} P}
	\end{equation*}
	is an isomorphism the $1$-category of topological spaces.
	We write $\TSpcspec \subset \TSpc $ for the subcategory of spectral topological spaces and \emph{quasicompact} continuous maps.
\end{dfn}

\begin{ntn}
	We write the $\Pos$\index[notation]{Pos@$\Pos, \Posfin$} for $1$-category of posets, and $\Posfin$ for the $1$-category of finite posets.
	We regard the $1$-category $\Pos$ as a full subcategory of $\Cat_{\infty}$; indeed one has $ \Pos \simeq \Cat_{0} $.

	Passing to \proobjects, we obtain the $1$-category $\Pro(\Pos)$ of proposets and the full subcategory $\Pro(\Posfin)$ of \proobjects in the category of finite posets -- which we call \defn{profinite posets}\index[terminology]{profinite!poset}\index[terminology]{poset!profinite}.
\end{ntn}

\begin{nul}
	The formation of the Alexandroff topology extends to an equivalence of $1$-categor\-ies
	\begin{equation*}
		\Alex \colon\equivto{\Pro(\Posfin)}{\TSpcspec} \period
	\end{equation*}
	We will therefore fail to distinguish between a spectral topological space and its corresponding profinite poset.
\end{nul}

Hochster's characterization of spectral topological spaces justifies their name:

\begin{thm}[{Hochster Duality\index[terminology]{Hochster Duality} \cites{Hochster:thesis,Hochster:primeideal}}]\label{thm:Hoschterthesis}
	The following are equivalent for a topological space $S$.
	\begin{itemize}
		\item The topological space $S$ is spectral.

		\item The topological space $S$ is sober, quasicompact, and quasiseparated; additionally, the set of quasicompact open subsets forms a basis for the topology of $S$.

		\item The topological space $S$ is homeomorphic to the underlying Zariski topological space of $\Spec R$ for some ring $R$.

		\item The topological space $S$ is homeomorphic to the underlying Zariski topological space of some coherent scheme $Y$.
	\end{itemize}
\end{thm}

\begin{nul}\label{nul:dualitycube} 
	On one hand, Alexandroff Duality characterizes posets as finitely generated topological spaces.
	On the other, \textit{Stone Duality} characterizes profinite sets as \textit{Stone topological spaces},%
	\index[terminology]{topological spaces!Stone}\index[terminology]{Stone!topological spaces}
	i.e., totally separated quasicompact topological spaces.
	Hochster Duality provides a common extension of each of these forms of duality.
	The situation is summarized in the cube
	\begin{equation*}
		\begin{tikzcd}[column sep={10ex,between origins}, row sep={8ex,between origins}]
			& \Setfin \arrow[dl, hooked'] \arrow[rr, "\sim"{yshift=-0.2em}] \arrow[dd, hooked, xshift=-0.25em]  & & \TSpc^{\fin,\disc} \arrow[dl, hooked'] \arrow[dd, hooked, xshift=-0.5em] \\
			\Pro(\Setfin) \arrow[dd, hooked] \arrow[rr, crossing over, "\sim"{yshift=-0.2em, near end}] & & \TSpc^{\Stone} \\
			& \Posfin \arrow[dl, hooked'] \arrow[rr, "\sim"{yshift=-0.2em, near start}] & & \TSpc^{\fin} \arrow[dl, hooked'] \\
			\Pro(\Posfin) \arrow[rr, "\sim"{yshift=-0.2em}] & & \TSpcspec \arrow[from=uu, crossing over, hooked] & \phantom{\TSpc^{\fin,\disc}} \comma
		\end{tikzcd}
	\end{equation*}
	where $ \TSpc^{\fin} $ denotes the $ 1 $-category of finite spectral topological spaces, and the horizontal functors marked `$\sim$' are equivalences of $1$-categories.

	One of the main technical results of this book -- the \Categorical Hochster Duality Theorem (\Cref{thm:headlineinftyHochster}=\Cref{thm:inftyHochster}) -- is an extension of this cube of dualities to one in which the $1$-category of finite sets is replaced with the \category of \pifinite \groupoids.
	Part of this extension is already established in the literature: Lurie proves \acategorical form of Stone Duality \cite[\SAGsec{E.3}]{SAG}.
	This \Categorical Stone Duality Theorem identifies the \category of profinite \groupoids with the \category of what we call \textit{Stone \topoi}.%
	\footnote{Lurie calls these \textit{profinite \topoi}.
	In \Cref{sec:spectraltopoi} we introduce a more general class of \topoi that could also reasonably be called `profinite \topoi', so we use the distinct term `Stone \topoi' to avoid confusion.}
\end{nul}


\subsection{Profinite stratifications}

The theory of stratifications also works well for profinite stratifications.

\begin{dfn}\label{def:profinstrat}
	A \defn{profinite stratification}\index[terminology]{profinite stratification!of a topological space}\index[terminology]{stratification!profinite!of a topological space} of a topological space $T$ is a spectral topological space $S$ and a continuous map $f\colon\fromto{T}{S}$.
	We say that $ f $ is \defn{constructible}\index[terminology]{profinite stratification!constructible}\index[terminology]{constructible!profinite stratification}\index[terminology]{stratification!constructible} if and only if, for every quasicompact open subset $U\subseteq S$, the inverse image $f^{-1}(U)\subseteq T$ is retrocompact.
\end{dfn}

\begin{nul}
	A profinite stratification with base $S$ is the same as a compatible family of stratifications with base $P$ for each nondegenerate, finite, constructible stratification $\fromto{S}{P}$.
\end{nul}

\begin{ntn}\label{ntn:Xzar}
	Let $ X $ be a scheme.
	We write $ X^{\zar} $ for the underlying Zariski topological space of $ X $.
\end{ntn}

\begin{exm}\label{exm:Xanprofinstrat}
	Let $ X $ be a scheme of finite type over the complex numbers.
	Write $ X^{\an} $ for the set $ X(\CCup) $ of complex points of $ X $ equipped with the complex analytic topology.
	Then the natural continuous map $ \fromto{X^{\an}}{X^{\zar}} $ is a profinite stratification of $ X^{\an} $ by $ X^{\zar} $.
\end{exm}

\newpage

\section{The homotopy theory of stratified spaces}\label{sec:hothystratspaces}

In this chapter we develop the homotopy theory of stratified spaces.
To start, \cref{subsec:stratisconsfunctor} explains how to think about the homotopy theory of stratified spaces in terms of \categories with a conservative functor to a poset.
\Cref{subsec:consrepair} explains how the \categories of stratified spaces relate as the poset varies.
\Cref{subsec:stratPostnikov} explains the correct notions of connectedness and truncatedness for stratified spaces.
\Cref{subsec:finitestratspaces} explains the analogue of \textit{\pifinite} spaces (i.e., truncated spaces with finite homotopy groups) in the stratified setting.
These \textit{\pifinite stratified spaces} are crucial in our formulation of one of the main results of this text: \Categorical Hochster Duality (\Cref{thm:headlineinftyHochster}).
\Cref{subsec:profinstratspaces} explains the theory of profinite stratified spaces.
\Cref{subsec:spatialdecollages,subsec:stratifiednerve} explain a complete Segal space style approach to stratified spaces that we'll use again and again throughout the text; the power of this approach is that it reduces many questions about stratifications over a general poset $ P $ to questions about \textit{strata} and \textit{links} (and essentially to the poset $ P = [1] $).
Finally, \cref{subsec:profintiedecollages} explains how the complete Segal space approach works in the profinite setting.


\subsection{Stratified spaces as \texorpdfstring{$\infty$}{∞}-categories with a conservative functor to a poset}\label{subsec:stratisconsfunctor}

The equivalence between the homotopy theory of topological spaces and that of simplicial sets -- supplied in the 1950s and 1960s by the work of Kan and Quillen \cites{MR90047}{MR0096210}{MR111032}{MR111033}{MR111036}{Quillen:homotopicalalgebra} -- justifies the treatment of the \category of Kan complexes as `the' homotopy theory of spaces.
Today, the theory of stratified spaces stands on similarly good footing.
Work of Ayala--Francis--Rozenblyum \cite{MR3941460}, Nand-Lal--Woolf \cites{Nand-Lal:thesis}{NandLalWoolf}, Douteau \cite{Douteau:chapter7}, and finally the third-named author \cite{Haine:stratifiedmodel} furnish an equivalence between the homotopy theory of stratified topological spaces and that of \categories with a conservative functor to a poset.

In this section we breifly recall this result and begin to develop the theory of stratified spaces as \categories with a conservative functor to a poset.
To state the core result, we need to fix a convenient category of topological spaces.

\begin{ntn}
	We write $ \TSpcng \subset \TSpc $\index[notation]{TSpcng@$\TSpcng$} for the full subcategory spanned by the \defn{numerically generated} topological spaces.%
	\index[terminology]{numerically generated topological space}\index[terminology]{topological space!numerically generated}
	The category $ \TSpcng $ is a convenient category of topological spaces and is presentable; see \cites{DuggerDelta}{MR2425555}{HaraguchiThesis}[\S3]{MR3289294}{MR3884529}.
	Moreover, every poset in the Alexandroff topology is a numerically generated topological space.
\end{ntn}

\begin{thm}[{\cite[\S3]{Haine:stratifiedmodel}}]\label{thm:layeredmodel}
	Let $P$ be a poset.
	There is a class $W$ of \emph{stratified weak equivalences}\index[terminology]{stratified weak equivalence}\index[terminology]{weak equivalence!stratified} in the category $\TSpcng_{/P}$ with the following properties.
	\begin{itemize}
		\item There is an equivalence of \categories
		\begin{equation}\label{eq:equivofstrspaces}
			\TSpcng_{/P}[W^{-1}] \equivalence \Cat_{\infty,/P}^{\cons} \comma
		\end{equation}
		where the target is the \category of \categories with a conservative functor to $ P $.

	\item Let $ T $ be a $ P $-stratified topological space whose stratification is \emph{conical}\index[terminology]{conical stratification}\index[terminology]{stratification!conical} in the sense of Lurie \HAa{Definition}{A.5.5}, e.g., $ T $ is a \emph{topologically stratified}\index[terminology]{stratification!topological}\index[terminology]{topological stratification} in the sense of Goresky--MacPherson \cite[\S1.1]{MR696691}.
		Then the equivalence \eqref{eq:equivofstrspaces} sends $ T $ to the \emph{exit path \category}\index[terminology]{exit path \category}\index[terminology]{category@\category!exit path} of $ T $.

		\item If $ T $ and $ T' $ are conically $ P $-stratified spaces, then a $ P $-stratified map $ f \colon T \to T' $ is a stratified weak equivalence if and only if $ f $ induces a weak homotopy equivalence on strata and links.
	\end{itemize}
\end{thm}

\begin{nul}
	\Acategory $ C $ admits a conservative functor to a poset if and only if every endomorphism of an object of $ C $ is an equivalence.
	In this case, the homotopy $ 0 $-category $ \ho_0(C) $ is a poset and the natural functor $ \fromto{C}{\ho_0(C)} $ is conservative. 
\end{nul}

We therefore give the following definition. 

\begin{dfn}\label{dfn:stratspaces} 
	We define the \category $\Str$\index[notation]{Str@$ \Str, \StrP $} as the full subcategory of $\Fun([1],\Cat_{\infty})$ spanned by those functors $f\colon \fromto{\Pi}{P}$ where $P$ is a poset and the functor $f$ is conservative.
	We call an object of $\Str$ a \defn{stratified space}.\index[terminology]{stratified space}\index[terminology]{space!stratified}

	Let $ P $ be a poset.
	We write $ \StrP $ for the fiber the target functor $ \target \colon \fromto{\Str}{\Pos} $ over $ P $.
	That is to say, $ \StrP \equivalent \Cat_{\infty,/P}^{\cons} $ is the \category of \categories with a conservative functor to $ P $.
	We call an object of $\StrP$ a \defn{$P$-stratified space}\index[terminology]{Pstratified space@$P$-stratified space}.
\end{dfn}

\begin{nul}
	The \category $ \StrP $ can also be described as the underlying \category of the third-named author's \emph{Joyal--Kan model category}\index[terminology]{model structure!Joyal--Kan}\index[terminology]{Joyal--Kan model structure} $ \sSet_{/P} $ \cite[Corollary 2.5.11]{Haine:stratifiedmodel}.
\end{nul}

\begin{ntn}
	Please observe that if $\Pi$ and $\Pi'$ are $P$-stratified spaces, then the \category $ \Fun_P(\Pi,\Pi') $ of functors $\fromto{\Pi}{\Pi'}$ over $P$ is \agroupoid. 
	Thus $ \Fun_P(\Pi,\Pi') $ coincides with the mapping space $ \Map_{\StrP}(\Pi,\Pi') $.
	To simplify notation we write\index[notation]{MapP@$\Map_P$}
	\begin{equation*} 
		\Map_{P}(\Pi,\Pi') \colonequals \Map_{\StrP}(\Pi,\Pi') \period
	\end{equation*}
\end{ntn}

\begin{dfn}
	Let $f\colon \fromto{\Pi}{P}$ be a $ P $-stratified space.
	For each point $p\in P$, we call the space\index[notation]{Pip@$\Pi_p$}
	\begin{equation*}
		\Pi_p \coloneq \Map_{P}(\{p\},\Pi) \equivalent \{p\} \cross_P \Pi
	\end{equation*}
	the \defn{$p$-th stratum}\index[terminology]{stratum} of $\Pi$.
	For each pair of points $p,q\in P$ with $p\leq q$, we call the space
	\begin{equation*}
		\Nerve_P(\Pi)\{p \leq q\} \coloneq \Map_{P}(\{p \leq q\},\Pi) \index[notation]{NPPipq@$ \Nerve_P(\Pi)\{p \leq q\} $}
	\end{equation*}
	the \defn{link}\index[terminology]{link}\footnote{Our link corresponds to what Frank Quinn and others called the \textit{homotopy link} or \textit{holink}.
	The significance of our chosen notation will become clear in \Cref{cnstr:nerveofstratifiedspace}.} from the $p$-th stratum to the $q$-th stratum.

	Please observe that the link comes equipped with source and target maps
	\begin{equation*}
		(\source,\target) \colon \fromto{\Nerve_P(\Pi)\{p \leq q\}}{\Pi_p \times \Pi_q} \comma
	\end{equation*}
	and the fiber of $ (\source,\target) $ over a point $ (x,y) \in \Pi_p \times \Pi_q $ is the mapping space $\Map_{\Pi}(x,y)$.
	When $ p = q $, each of $ \source $ and $ \target $ is an equivalence, whence $ (\source,\target) $ is equivalent to the diagonal
	\begin{equation*}
		\Pi_p \to \Pi_p \times \Pi_p \period
	\end{equation*}
\end{dfn}

\begin{nul}
	A morphism $ f \colon \fromto{\Pi'}{\Pi}$ of $\StrP$ is an equivalence if and only if, for every pair of points $p,q\in P$ with $p\leq q$, the map on links
	\begin{equation*}
		\fromto{\Nerve_P(\Pi')\{p \leq q\}}{\Nerve_P(\Pi)\{p \leq q\}}
	\end{equation*}
	is an equivalence (in particular, when $p=q$, the map on strata $\fromto{\Pi'_p}{\Pi_p}$ is an equivalence).
	That is, $ f $ is an equivalence in $ \StrP $ if and only if $ f $ induces and equivalence on all strata and links.
\end{nul}


\subsection{Functoriality in the poset}\label{subsec:consrepair}

In this section we explain how the \categories of stratified spaces relate as the poset varies.
Notice that if $ \phi \colon \fromto{P'}{P} $ is a morphism of posets, then the functor
\begin{equation*}
	\fromto{\Cat_{\infty,/P'}}{\Cat_{\infty,/P}}
\end{equation*}
given by postcomposition with $ \phi $ does not generally send $ P' $-stratified spaces to $ P $-strat\-ified spaces.
However, we can easily repair this by inverting all morphisms that lie over identities in $ P $.
To explain this point, let us first explain the left and right adjoints to the inclusion $ \StrP \subset \Cat_{\infty,/P} $.

To begin with, we fix some notation.

\begin{ntn}\label{ntn:inverteverything}
	We write $ \interior \colon \fromto{\Cat_{\infty}}{\Space} $\index[notation]{iota@$\interior$} for the right adjoint to the inclusion, given by sending \acategory $ C $ to the largest \groupoid $ \interior C\subseteq C $ contained in $ C $.
	We call $ \interior C $ the \defn{interior}\index[terminology]{interior of \acategory} of $C$.

	We write $ \invert \colon \fromto{\Cat_{\infty}}{\Space} $\index[notation]{epsilon@$\invert$} for the left adjoint to the inclusion.
	The functor $ \invert $ is given by sending \acategory $ C $ to the \groupoid $ \invert(C) $ obtained by inverting every morphism of $ C $.
	We call $ \invert(C) $ the \defn{classifying space}\index[terminology]{classifying space of \acategory}\index[terminology]{space!classifying} of $ C $.%
	\footnote{In simplicial sets the functor $ \invert $ can be modeled as Kan's $ \Ex^{\infty} $ functor.
	The notation $ \Bup C $ is often used for the classifying space of $ C $.
We use the notation $ \invert(C) $ for three reasons: to avoid conflict with the notation $ \Bup G $ for the $ 1 $-object groupoid with automorphism group $ G $, to pay homage to Kan's $ \Ex^{\infty} $ functor, and because `$\invert$' stands for \emph{everything} in the phrase \emph{invert everything}.}
	The \groupoid $ \invert(C) $ can be computed as the colimit $ \invert(C) \equivalent \colim_{C} 1_{\Space} $ of the constant diagram $ \fromto{C}{\Space} $ at the terminal object.
\end{ntn}

\begin{cnstr} 
	Let $P$ be a poset.
	Then inclusion $ \incto{\StrP}{\Cat_{\infty,/P}} $ admits a right adjoint $ \interior_P \colon \fromto{\StrP}{\Cat_{\infty,/P}} $.\index[notation]{iotaP@$\interior_P$}
	Indeed, if $ C $ is \acategory, and $ f\colon \fromto{C}{P}$ is any functor, we write $ \interior_P(C) \subset C $ for the largest subcategory of $ C $ with the property that the composite
	\begin{equation*}
		\begin{tikzcd}[sep=1.5em]
			\interior_P(C) \arrow[r, hooked] & C \arrow[r, "f"] & P
		\end{tikzcd}
	\end{equation*}
	is conservative.
	Concretely, $ \interior_P(C) \subset C $ is the subcategory containing all objects such that morphism $ e $ of $ C $ lies in $ \interior_P(C) $ if and only if $ e $ satisfies one of the following (disjoint) conditions:
	\begin{itemize}
		\item The morphism $ e $ is an equivalence.

		\item The morphism $ e $ is not sent to an identity morphism in $ P $.
	\end{itemize}

	Given a $ P $-stratified space $ \Pi $, every functor $ \fromto{\Pi}{C} $ over $ P $ factors through $ \interior_P(C) $.
	Hence the assignment $ \goesto{C}{\interior_P(C)} $ defines a right adjoint to the inclusion $ \StrP \subset \Cat_{\infty,/P} $.
\end{cnstr}

\begin{cnstr}
	Let $ P $ be a poset.
	Then the inclusion $ \incto{\StrP}{\Cat_{\infty,/P}} $ admits a left adjoint $ \invert_P \colon \fromto{\StrP}{\Cat_{\infty,/P}} $.
	\index[notation]{epsilonP@$\invert_P$}
	Indeed, if $ C $ is \acategory, and $f\colon \fromto{C}{P}$ is any functor, we can formally invert those morphisms of $ C $ that are sent to identities in $ P $ by forming the pullback 
	\begin{equation*}
		\invert_P(C) \colonequals \invert(C) \cross_{\invert(P)} P 
	\end{equation*}
	in $ \Cat_{\infty} $.
	Note that the second projection $ \fromto{\invert_P(C)}{P} $ is conservative: for each $ p \in P $ we have
	\begin{equation*}
		\invert_P(C) \cross_{P} \{ p \} \equivalent \invert(C) \cross_{\invert(P)} \{ p\} \comma
	\end{equation*}
	so that $ \invert_P(C) \cross_{P} \{ p \} $ is the fiber of a map between \groupoids.
	We regard $ \invert_P(C) $ as a $ P $-stratified space via the second projection $ \fromto{\invert_P(C)}{P} $.

	The functor $ f \colon \fromto{C}{P} $ and the unit $ \fromto{C}{\invert(C)} $ induce a natural functor $ \fromto{C}{\invert_P(C)} $.
	By construction, this functor exhibits $ \invert_P(C) $ as the localization of $ C $ at those morphisms that lie over identities in $ P $.
	Moreover, the natural functor $ \fromto{C}{\invert_P(C)} $ is the unit of the desired adjunction.
\end{cnstr}

Now we can describe the functionality of the construction $ \goesto{P}{\StrP} $.

\begin{nul}
	Let $ \phi \colon \fromto{P'}{P} $ be a morphism of posets.
	Since the pullback of a conservative functor is conservative, the pullback functor
	\begin{equation*}
		\phiupperstar \colonequals (-) \cross_P P' \colon \fromto{\Cat_{\infty,/P}}{\Cat_{\infty,/P'}} 
	\end{equation*}
	carries $ P $-stratified spaces to $ P' $-stratfied spaces.
	The pullback functor $ \phiupperstar \colon \fromto{\StrP}{\Str_{P'}} $ admits a left adjoint $ \philowershriek $ given by the composite
	\begin{equation*}
		\begin{tikzcd}[sep=1.5em]
			\Str_{P'} \arrow[r] & \Cat_{\infty,/P} \arrow[r, "\invert_P"] & \StrP
		\end{tikzcd}
	\end{equation*}
	of postcomposition with $ \phi $ followed by the left adjoint to the inclusion $ \StrP \subset \Cat_{\infty,/P} $.
\end{nul}

\begin{prp}
	The target functor $ \target \colon \fromto{\Str}{\Pos}$ is a bicartesian fibration.
\end{prp}

\begin{proof}
	Let $ \phi \colon \fromto{P'}{P}$ be a morphism of posets.
	If $ f \colon \fromto{\Pi}{P}$ is a $P$-stratified space, the resulting pullback square
	\begin{equation*}
		\begin{tikzcd}
			P' \times_{P} \Pi \arrow[r] \arrow[d, swap, "\phi^{\ast}(f)"] \arrow[dr, phantom, very near start, "\lrcorner", xshift=-0.5em, yshift=0.25em] & \Pi \arrow[d, "f"] \\
			P' \arrow[r, swap, "\phi"]& P
		\end{tikzcd}
	\end{equation*}
	is a $ \target $-cartesian morphism lying over $\phi$.

	In the other direction, let $ f' \colon \fromto{\Pi'}{P'}$ be a $P'$-stratified space.
	Then formally inverting those morphisms of $ \Pi' $ that are sent to identities by $ \phi \of f' $, we see that the resulting square
	\begin{equation*}
		\begin{tikzcd}
			\Pi' \arrow[r] \arrow[d, swap, "f'"]& \invert_P(\Pi') \arrow[d, "\philowershriek(f')"] \\
			P' \arrow[r, swap, "\phi"]& P
		\end{tikzcd}
	\end{equation*}
	is a $ \target $-cocartesian morphism of $ \Str $ lying over $\phi$.
\end{proof}  

We can use the pullbacks to describe limits in $ \Str $:

\begin{nul}
	To compute the limit of a diagram $ \goesto{\alpha}{\left[\Pi_{\alpha}\to P_{\alpha}\right]}$ in $ \Str $, we first form the limit $ P \coloneq \lim_{\alpha} P_{\alpha}$; then pulling back along the various projections $p_{\alpha}\colon \fromto{P}{P_{\alpha}}$, we obtain the diagram $\goesto{\alpha}{p_{\alpha}^{\ast}\Pi_{\alpha}}$ of $P$-stratified spaces.
	We then form the limit $ \Pi \coloneq \lim_{\alpha} \Pi_{\alpha}$ in $ \StrP $.
	If the diagram is \textit{connected}, then the limit $ \lim_{\alpha} \Pi_{\alpha} $ is computed in $ \Cat_{\infty} $.
\end{nul}


\subsection{The stratified Postnikov tower}\label{subsec:stratPostnikov}

In this section we investigate a Postnikov tower for stratified spaces.
Importantly, the correct notion of `$ n $-truncatedness' is \textit{not} the notion of $ n $-truncatedness interntal to the \category $ \Str_{P} $ (in the sense of \cite[\HTTsubsec{5.5.6}]{HTT}); rather it corresponds to the categorical level of the stratified space.\footnote{Recall that an $ n $-category is an $ n $-truncated object of $ \Cat_{\infty} $, but the converse is false.}
For this, recall that we write $ \ho_n \colon \fromto{\Cat_{\infty}}{\Cat_{n}} $ for the left adjoint to the inclusion $ \Cat_{n} \subset \Cat_{\infty} $ \Cref{nul:nCat}.

\begin{dfn}
	Let $P$ be a poset and $\Pi$ a $P$-stratified space.
	We call the tower of $P$-stratified spaces
	\begin{equation*}
		\Pi\to\cdots\to \ho_3\Pi\to \ho_2\Pi\to \ho_1\Pi\to \ho_0\Pi\to P \comma
	\end{equation*}
 	the \defn{stratified Postnikov tower}\index[terminology]{stratified Postnikov tower} of $ \Pi $.

	In particular, please observe that $\ho_0\Pi\to P$ is a morphism of posets.
\end{dfn}

\begin{nul}
	If $P=\{0\}$, then the stratified Postnikov tower coincides with the usual Postnikov tower of spaces.
\end{nul}

\begin{nul}\label{nul:truncatedstratspace}
	The following are equivalent for a poset $P$, a $P$-stratified space $f\colon\Pi\to P$, and a nonnegative integer $n\in\NNup$:
	\begin{itemize}
		\item the \category $\Pi$ is an $n$-category;

		\item the natural functor $\fromto{\Pi}{\ho_n\Pi}$ is an equivalence;

		\item for all objects $x,y\in\Pi$, the space $\Map_{\Pi}(x,y)$ is $(n-1)$-truncated;

		\item for all points $p,q\in P$ such that $p\leq q$, the source and target map
		\begin{equation*}
			(\source,\target)\colon \fromto{\Nerve_P(\Pi)\{p \leq q\}}{\Pi_p \times \Pi_q}
		\end{equation*}
		is $(n-1)$-truncated (in particular, when $p=q$, the stratum $\Pi_p$ is $n$-truncated).
	\end{itemize} 
\end{nul}

\begin{dfn}\label{def:truncatednessforstratspaces}
	Let $P$ be a poset and $n\in\NNup$.
	We say that a $P$-stratified space $\Pi$ is \defn{$n$-truncated}\index[terminology]{ntruncated@$ n $-truncated!$ P $-stratified space}\index[terminology]{Pstratified space@$ P $-stratified space!$ n $-truncated} if $ \Pi $ satisfies the equivalent conditions of \Cref{nul:truncatedstratspace}.
	We write $\Str_{P,\leq n}\subset\StrP$\index[notation]{StrPn@$ \Str_{P,\leq n} $} for the full subcategory spanned by the $n$-truncated $P$-stratified spaces. 

	We caution that an $n$-truncated $P$-stratified space is generally \emph{not} the same thing as an $n$-truncated object of the \category $\StrP$ in the sense of \HTT{Definition}{5.5.6.1}.
	Nor is it the same thing as a $P$-stratified space whose strata are $n$-truncated; truncatedness in our sense involves a condition on the links as well.
\end{dfn}

\begin{nul}\label{nul:connectivestratspace}
	Dually, the following are equivalent for a poset $P$, a $P$-stratified space $f\colon\Pi\to P$, and a nonnegative integer $n\in\NNup$:
	\begin{itemize}
		\item the natural functor $\fromto{\ho_n\Pi}{P}$ is an equivalence;
		\item for all objects $x,y\in\Pi$ such that $f(x)\leq f(y)$, the space $\Map_{\Pi}(x,y)$ is $(n-1)$-connected;
		\item the strata of $ \Pi $ are nonempty and for all points $p,q\in P$ such that $p\leq q$, the map
		\begin{equation*}
			(\source,\target)\colon \fromto{\Nerve_P(\Pi)\{p \leq q\}}{\Pi_p \times \Pi_q}
		\end{equation*}
		is $(n-1)$-connected (in particular, when $p=q$, the stratum $\Pi_p$ is $n$-connected).
	\end{itemize}
\end{nul}

\begin{dfn}
	Let $P$ be a poset and $n\in\NNup$.
	We say that a $P$-stratified space $\Pi$ is \defn{$n$-connected}\index[terminology]{n-connected@$ n $-connected!$ P $-stratified space}\index[terminology]{P-stratified space@$ P $-stratified space!$ n $-connected} if $ \Pi $ satisfies the equivalent conditions of \Cref{nul:connectivestratspace}.
	We write $ \Str_{P,\geq n} \subset \StrP $\index[notation]{StrPn@$ \Str_{P,\geq n} $} for the full subcategory spanned by the $n$-connected $P$-stratified spaces. 
\end{dfn}

We can easily identify the $ 0 $-connected stratified spaces.

\begin{dfn}
	We say that a $1$-category $ C $ is \defn{layered}%
	\footnote{Layered categories are often called \textit{EI} categories.}
	if and only if every endomorphism of an object of $ C $ is an isomorphism.
	We say that \acategory $\Pi$ is \defn{layered}\index[terminology]{layered!\category}\index[terminology]{category@\category!layered} if and only if its homotopy category $ \ho_1(\Pi) $ is a layered $ 1 $-category.
	This holds if and only if the natural functor $\fromto{\Pi}{\ho_0(\Pi)}$ is conservative.
	Thus a layered \category $\Pi$ is naturally an $\ho_0(\Pi)$-stratified space.

	We write $\Lay_{\infty}$\index[notation]{Lay@$\Lay_{\infty}$} for the full subcategory of $\Cat_{\infty} $ spanned by the layered \categories.
\end{dfn}

\begin{nul}
	The assignment $\goesto{[\Pi\to P]}{\Pi}$ defines a functor $\fromto{\Str}{\Lay_{\infty}}$ with a fully faithful left adjoint that carries $\Pi$ to the $\ho_0(\Pi)$-stratified space $\Pi$. Consequently, we obtain an identification
	\begin{equation*}
		\Lay_{\infty}\simeq\Str_{\geq 0} \period
	\end{equation*}
	Here $\Str_{\geq 0}\subset\Str$ is the full subcategory spanned by the $0$-connected stratified spaces.
\end{nul}


\subsection{\texorpdfstring{$\uppi$}{π}-finite stratified spaces}\label{subsec:finitestratspaces}

In this section we introduce the key finiteness condition that we impose on almost all of the stratified spaces we consider in this book.
This finiteness condition is the stratified version of \textit{\pifiniteness} for spaces.

\begin{rec}[{\SAG{Definition}{E.0.7.8}}]
	\Agroupoid $K$ is \defn{\pifinite}\index[terminology]{space!\pifinite}\index[terminology]{pifinite@\pifinite!space} if and only if the following conditions are satisfied.
	\begin{itemize}
		\item The set $\uppi_0(K)$ is finite.

		\item For every point $x\in K$ and any $i\geq 1$, the group $\uppi_i(K,x)$ is finite.

		\item The \groupoid $K$ is $n$-truncated for some $n\in\NNup$.
	\end{itemize}
	We write $ \Spacefin \subset \Space $\index[notation]{Spacepi@$ \Spacefin $} for the full subcategory spanned by the \pifinite \groupoids.
\end{rec}

\begin{wrn}
	We caution that a \pifinite space is not the same thing as what is normally called a \defn{finite space}\index[terminology]{space!finite} -- one obtained via finite colimits from the point. 
	In fact, the overlap between these two classes of spaces is essentially trivial: the spaces satisfying both of these conditions are exactly the discrete spaces with finitely many connected components.
\end{wrn}

We now define the analogous condition for a \emph{stratified} space.

\begin{dfn}\label{dfn:finstratspaces}
	We say that a stratified space $\fromto{\Pi}{P}$ is \defn{\pifinite}\index[terminology]{stratified space!\pifinite}\index[terminology]{pifinite@\pifinite!stratified space} if and only if the following conditions are satisfied.
	\begin{itemize}
		\item The poset $P$ is finite.

		\item For every point $p\in P$, the set $\uppi_0(\Pi_p)$ is finite.

		\item For all $ x,y \in \Pi $, the mapping space $ \Map_{\Pi}(x,y) $ is a \pifinite space.

		\item The \category $\Pi$ is an $n$-category for some $n\in\NNup$.
	\end{itemize}
	In particular, a nondegenerate stratified space $\fromto{\Pi}{P}$ is \pifinite if and only if $\Pi$ has finitely many objects up to equivalence and is \defn{locally \pifinite}\index[terminology]{category@\category!locally \pifinite}\index[terminology]{locally \pifinite!\category} in the sense that each mapping space $\Map_{\Pi}(x,y)$ is \pifinite.

	We write $\Strfin \subset \Str $\index[notation]{Strpi@$ \Strfin, \StrfinP $} for the full subcategory spanned by the \pifinite stratified spaces.
	Given a finite poset $P$, we write $\StrfinP \subset\StrP$ for the full subcategory spanned by the \pifinite $P$-stratified spaces.
\end{dfn}

\begin{nul}
	The target functor $ \target \colon \fromto{\Strfin}{\Posfin}$ is a cartesian fibration.
	However, it is not a cocartesian fibration because the pullback functor doesn't admit a left adjoint when restricted to \pifinite stratified spaces.
	To see this, note that the free pair of parallel arrows $ 0 \rightrightarrows 1 $ is \pifinite as a $ [1] $-stratified space, but its classifying space is equivalent to $ \BZZ $, which is not \pifinite.

	However, the pullback does preserve finite limits, hence has a proëxistent left adjoint; we will discuss this in \cref{subsec:profinstratspaces}.
\end{nul}

\begin{lem}\label{lem:Stratpiaccess}
	The full subcategory $\Strfin \subset \Str$ is an accessible subcategory that is closed under finite limits.
\end{lem}

\begin{proof}
	Finite limits of finite posets are finite, pullbacks of \pifinite stratified spaces along maps of finite posets are \pifinite, and finite limits of locally \pifinite \categories are locally \pifinite.
	Finally, $\Strfin$ is $\updelta_0$-small and idempotent complete.
\end{proof}


\subsection{Profinite stratified spaces}\label{subsec:profinstratspaces}

In light of \Cref{nul:proexistentadjoint}, \Cref{lem:Stratpiaccess} allows us to speak of profinite stratified spaces in terms of left exact accessible functors.
We now turn to stratifications over profinite posets (i.e., spectral topological spaces).

\begin{dfn}
	We call objects of the \category $\Pro(\Str)$ \defn{stratified prospaces};%
	\index[terminology]{stratified prospace}\index[terminology]{prospace!stratified}
	the target functor $ \target \colon \fromto{\Str}{\Pos}$ from stratified spaces to posets extends to a target functor
	\begin{equation*}
		\target \colon \fromto{\Pro(\Str)}{\Pro(\Pos)} \period
	\end{equation*}
	Let $ P $ be a poset, regarded as a constant proposet.
	The fiber $ \Pro(\Str)_{P} $ of $ \target $ over $ P $ can be identified with the \category $\Pro(\StrP)$ of \proobjects in $ \StrP $.
	We refer to objects of $ \Pro(\Str)_{P} $ as \defn{$P$-stratified prospaces}.%
	\index[terminology]{Pstratified@$P$-stratified!prospace}\index[terminology]{prospace!Pstratified@$P$-stratified}

	Similarly, if $\PP$ is a proposet, then we refer to the fiber $\Pro(\Str)_{\PP}$ of $ \target $ over $\PP$ as the \category of \defn{$\PP$-stratified prospaces}.
\end{dfn}

\begin{nul}
	A stratified prospace can be exhibited as an inverse system $\{\Pi_{\alpha}\to P_{\alpha}\}_{\alpha \in A}$ of stratified spaces.
	The target functor $ \target $ carries this stratified prospace to the proposet $\{P_{\alpha}\}_{\alpha \in A}$.
\end{nul}

\begin{exm}
	Our primary interest is in stratified prospaces stratified by a spectral topological space regarded as a profinite poset.
	Still, in order to reason effectively with these, it is occasionally necessary to deal with more general stratified prospaces.
\end{exm}

We now turn to the functoriality of the assignment $ \goesto{\PP}{\Pro(\Str)_{\PP}} $.
The key point is that the adjunctions relating the \categories $ \Str_P $ as $ P $ varies extend to stratified prospaces.

\begin{cnstr}[pullbacks of stratified prospaces]
	Please observe that the target functor
	\begin{equation*}
		\target \colon \fromto{\Pro(\Str)}{\Pro(\Pos)}
	\end{equation*}
	is a cartesian fibration.
	Indeed, let $ \phi \colon \fromto{\PP'}{\PP} $ be a morphism of proposets, and exhibit $ \PP' $ and $ \PP $ as inverse systems of proposets $ \{P'_{\alpha}\}_{\alpha \in A} $ and $ \{P_{\alpha}\}_{\alpha \in A} $, respectively.
	Given a $ \PP $-stratified prospace $ \mbfPi = \{\Pi_{\alpha}\to P_{\alpha}\}_{\alpha \in A} $, we define a $ \PP' $-stratified prospace $ \phiupperstar(\mbfPi) $ as the inverse system
	\begin{equation*}
		\phiupperstar(\mbfPi) \colonequals \{\Pi_{\alpha}\times_{P_{\alpha}}P'_{\alpha} \to P'_{\alpha} \}_{\alpha \in A} \period
	\end{equation*}
	The morphism $ \phiupperstar(\mbfPi) \to \mbfPi $ is a $ \target $-cartesian morphism lying over $ \phi $ and the assignment $ \goesto{\mbfPi}{\phiupperstar(\mbfPi)} $ defines a functor $ \phiupperstar \colon \fromto{\Pro(\Str)_{\PP}}{\Pro(\Str)_{\PP'}} $.
\end{cnstr}

We now describe the left adjoint to this functor.

\begin{cnstr}[Pushforwards of stratified prospaces]\label{cnstr:pushforwardofprostrat}
	Let $\eta\colon \fromto{\PP'}{P}$ a morphism of proposets where $ P $ is constant.
	Regarding $ \PP' $ as a left exact accessible functor
	\begin{equation*}
		\fromto{\Pos}{\Set} \comma
	\end{equation*}
	the morphism $ \eta $ defines an element $ \eta \in \PP'(P) $.
	For a $ \PP' $-stratified prospace $ \mbfPi' $, there exists a $ \target $-cocartesian edge $\fromto{\mbfPi'}{\eta_!\mbfPi'}$ covering $\eta$; indeed, for any $ P $-stratified space $ \Pi $, we have an equivalence
	\begin{equation*}
		(\eta_!\mbfPi')(\Pi) \simeq \mbfPi'(\Pi) \crosslimits_{\PP'(P)} \{\eta\} \period
	\end{equation*}
	Equivalently, if we exhibit $\mbfPi'$ as an inverse system $\{\fromto{\Pi'_{\alpha}}{P'_{\alpha}}\}_{\alpha \in A}$ in $\Str$, then the $ P $-stratified prospace $\eta_!\mbfPi'$ can be exhibited as the inverse system $ A \times_{\Pos}\Pos_{/Q}\to\StrP$ given by
	\begin{equation*}
		\goesto{\left(\alpha,\fromto{P'_{\alpha}}{P}\right)}{\invert_P(\Pi'_{\alpha})} \period
	\end{equation*}
	Note in particular that if $ \PP' $ and $ \mbfPi' $ are constant, then so is $ \eta_!\mbfPi' $.

	In the \category $\Pro(\Str)$, the inverse system $\fromto{\Pos_{\PP'/}}{\Str}$ given by $\goesto{\eta}{\eta_!\mbfPi'}$ is identified with $ \mbfPi' $ itself.

	Given \textit{any} morphism of proposets $ \phi \colon \fromto{\PP'}{\PP} $ and $\PP'$-stratified prospace $\mbfPi'$, the inverse system $\fromto{\Pos_{\PP/}}{\Str}$ given by the assignment
	\begin{equation*}
		\goesto{\eta}{(\eta\circ\phi)_!\mbfPi'}
	\end{equation*}
	defines a $ \PP $-stratified prospace $\philowershriek\mbfPi'$.
	As this notation suggests, the morphism
	\begin{equation*}
		\fromto{\mbfPi'}{\philowershriek\mbfPi'}
	\end{equation*}
	is a $ \target $-cocartesian edge over $\phi$.
	Thus $ \target \colon \fromto{\Pro(\Str)}{\Pro(\Pos)}$ is a cocartesian fibration.
\end{cnstr}

We thus combine the previous two points:

\begin{prp}
	The target functor $ \target \colon \fromto{\Pro(\Str)}{\Pro(\Pos)} $ is a bicartesian fibration.
\end{prp}

We now turn to \proobjects in \pifinite stratified spaces.

\begin{dfn}\label{dfn:profinitestratspace}
	A \defn{profinite stratified space}\index[terminology]{stratified space!profinite}\index[terminology]{profinite!stratified space} is a \proobject of the \category $\Strfin$.
	We write $\Strprofin\coloneq\Pro(\Strfin) $\index[notation]{Strhatpi@$\Strprofin$} for the \category of profinite stratified spaces.

	The target functor
	\begin{equation*}
		\target \colon \Strprofin \to \Pro(\Posfin) \equivalent \TSpcspec
	\end{equation*}
	is a cartesian fibration.
	Given a spectral topological space $S$, we write $\StrprofinS$ for the fiber of $ \target $ over $S$.
	We call $ \StrprofinS $ the \category of \defn{profinite $S$-stratified spaces}.%
	\index[terminology]{profinite!Sstratified space@$S$-stratified space}\index[terminology]{Sstratified space@$S$-stratified space!profinite}

	The inclusion $\incto{\Strfin}{\Str}$ extends to a fully faithful functor $\incto{\Strprofin}{\Pro(\Str)}$, which admits a left adjoint $\goesto{\mbfPi}{\mbfPi\profincomp}$ given by restriction of left exact accessible functors.
	We call the profinite stratified space $\mbfPi\profincomp$ the \defn{profinite completion}\index[terminology]{profinite completion} of $\mbfPi$.
\end{dfn}

\begin{nul} 
	The profinite completion functor $\goesto{\mbfPi}{\mbfPi\profincomp}$ is not a relative left adjunction over $\Pro(\Pos)$; however, the inclusion $\incto{\Strfin}{\Str}$ \emph{does} induce a fully faithful functor
	\begin{equation*}
		\incto{\Strprofin}{\Pro(\Str) \crosslimits_{\Pro(\Pos)}\TSpcspec} \comma
	\end{equation*}
	and profinite completion \emph{does} define a relative left adjoint over $\TSpcspec$.
	In particular, if $S$ is a spectral topological space and $\mbfPi$ is an $S$-stratified prospace, then $\mbfPi\profincomp$ is a profinite $S$-stratified space, and the morphism $\fromto{\mbfPi}{\mbfPi\profincomp}$ lies over $S$.
\end{nul}

\begin{cnstr}[Pushforwards of profinite stratified spaces]\label{cnstr:pushforwardofprofinstrat}
	Let $\phi\colon \fromto{S'}{S}$ be a quasicompact continuous map of spectral topological spaces, and let $\fromto{\mbfPi'}{S'}$ be a profinite $ S' $-stratified space.
	Then following \Cref{cnstr:pushforwardofprostrat}, we obtain an $S$-stratified prospace $\fromto{\philowershriek\mbfPi'}{S}$.
	Forming the profinite completion $\fromto{(\philowershriek\mbfPi')\profincomp}{S}$ of $ \philowershriek\mbfPi' $, we see that the map
	\begin{equation*}
		\fromto{\mbfPi'}{(\philowershriek\mbfPi')\profincomp}
	\end{equation*}
	is a cocartesian edge over $\phi$ for the target functor $ \target \colon \fromto{\Strprofin}{\TSpcspec}$.
\end{cnstr}

\noindent We thus obtain:

\begin{prp}
	The target functor $ \target \colon \fromto{\Strprofin}{\TSpcspec}$ is a bicartesian fibration.
\end{prp}

\begin{prp}\label{prp:profinstratSislimoverFS}
	Let $S$ be a spectral topological space. 
	Then the natural functor
	\begin{equation*}
		\fromto{\StrprofinS}{\lim_{P\in \FC(S)}\StrprofinP}
	\end{equation*}
	is an equivalence.
\end{prp}

\begin{proof}
	The formation of the limit in $\Strprofin$ is an inverse.
\end{proof}


\subsection{Complete Segal spaces \& spatial décollages}\label{subsec:spatialdecollages}

\begin{rec}
	\Acategory can be modeled as a simplicial \emph{space}.
	In effect, if $C$ is \acategory, then one may extract a functor $ \Nerve(C)\colon \fromto{\Deltaop}{\Space}$ in which $ \Nerve(C)_m$ is the \groupoid of functors $\fromto{[m]}{C}$ (the `moduli space of sequences of arrows in $C$').
	The simplicial space $ \Nerve(C)$ is what Charles Rezk \cite{MR1804411} called a \emph{complete Segal space}\index[terminology]{complete Segal space} -- i.e., a functor $D\colon \fromto{\Deltaop}{\Space}$ satisfying the following conditions.
	\begin{enumerate}[(1)]
		\item For all $m\in\NNup^{\ast}$, the natural map
		\begin{equation*}
			D_m\to D\{0 < 1\} \crosslimits_{D\{1\}} D\{1 < 2\} \crosslimits_{D\{2\}} \cdots \crosslimits_{D\{m-1\}}D\{m-1 < m\}
		\end{equation*}
		is an equivalence.

		\item Let $ I $ denote the unique contractible $1$-groupoid with two objects.
		Then the map
		\begin{equation*}
			D_0 \to \Map_{\Fun(\Deltaop,\Space)}(\Nerve(I),D)
		\end{equation*}
		induced by the projection $ \fromto{I}{\{0\}} $ is an equivalence.
	\end{enumerate}
	(This same formalism can be deployed to define category objects in any \category with finite limits; we exploit this in \Cref{subsec:categoryobjs}.)

	Joyal and Tierney \cite{MR2342834} showed that the assignment $\goesto{C}{\Nerve(C)}$ defines an equivalence from the \category $\Cat_{\infty} $ of \categories to the \category $\CSS$\index[notation]{CSS@$\CSS$} of complete Segal spaces.

	We can isolate the \groupoids in $\CSS$: \acategory $C$ is \agroupoid if and only if $ \Nerve(C)\colon \fromto{\Deltaop}{\Space}$ is constant. 
\end{rec}

In the remainder of this section and the next, we shall demonstrate that the homotopy theory of stratified spaces admits an analogous description.

\begin{ntn}
	For a poset $ P $, we write $ \sdop(P) \coloneq \sd(P)^{\op} $.\index[notation]{sdop@$\sd^{\op}$}
\end{ntn}

\begin{dfn}\label{def:spatialdecollage}
	Let $P$ be a poset.
	A functor $D\colon \fromto{\sdop(P)}{\Space}$ is  a \defn{spatial décollage (over $P$)}%
	\index[terminology]{spatial décollage}\index[terminology]{decollage@décollage!spatial}
	if and only if, for every string $\{p_0 < \cdots < p_m\}\subseteq P$, the map
	\begin{equation*}
		D\{p_0 < \cdots < p_m\}\to D\{p_0 < p_1\} \crosslimits_{D\{p_1\}}D\{p_1 < p_2\} \crosslimits_{D\{p_2\}} \cdots \crosslimits_{D\{p_{m-1}\}}D\{p_{m-1} < p_m\}
	\end{equation*}
	is an equivalence.
	We write%
	\index[notation]{Dec@$\DecSpace{}, \DecSpace{P}$}
	\begin{equation*}
		\DecSpace{P}\subseteq\Fun(\sdop(P),\Space)
	\end{equation*}
	for the full subcategory spanned by the spatial décollages.
\end{dfn}

\begin{exm}
	Let $ P $ be a poset of rank $ \leq 1 $.
	Then every functor $ \fromto{\sdop(P)}{\Space} $ automatically satisfies the décollage condition.
	So in this case, 
	\begin{equation*}
		\DecSpace{P} = \Fun(\sdop(P),\Space) \period
	\end{equation*}
\end{exm}

Now we turn to the functoriality of the assignment $ \goesto{P}{\DecSpace{P}}$:

\begin{cnstr}[functoriality of spatial décollages]\label{cnstr:DecS}
	Write 
	\begin{equation}\label{eq:cartfiboverPos}
		\fromto{\textstyle \int_{\Pos} \Fun(\sdop,\Space)}{\Pos} \index[notation]{zzzP@$ \textstyle \int_{\Pos} \Fun(\sdop,\Space) $}
	\end{equation}
	for the cartesian fibration classified by the functor $ \fromto{\Pos^{\op}}{\Cat_{\infty}} $ given by the assignment
	\begin{equation*}
		\goesto{P}{\Fun(\sdop(P),\Space)}
	\end{equation*}
	with functoriality given by right Kan extension \HTT{Corollary}{3.2.2.13}.
	Thus the objects of \smash{$ \int_{\Pos} \Fun(\sdop,\Space) $} consist of a pair $ (P,F) $ of a poset $ P $ and a functor
	\begin{equation*}
		F \colon \fromto{\sdop(P)}{\Space} \period
	\end{equation*}
	The fiber of \eqref{eq:cartfiboverPos} over a poset $P$ is the \category $ \Fun(\sdop(P),\Space) $.

	Let
	\begin{equation*}
		\DecSpace{} \subset \textstyle \int_{\Pos} \Fun(\sdop,\Space) 
	\end{equation*}
	denote the full subcategory spanned by the pairs $(P,D)$ in which $D$ is a spatial décollage. 
	Since $\DecSpace{}$ contains all the cartesian edges, the functor $\fromto{\DecSpace{}}{\Pos}$ is a cartesian fibration.
\end{cnstr}


\subsection{The nerve of a stratified space}\label{subsec:stratifiednerve}

We now show that the \category $\Str$ of stratified spaces and the \category $\DecSpace{}$ of décollages are equivalent over $\Pos$.

\begin{cnstr}[nerve of a stratified space]\label{cnstr:nerveofstratifiedspace} 
	Let $P$ be a poset. 
	Any chain contained in $P$ can be regarded as a $P$-stratified space via the inclusion map. 
	This assignment defines a functor $\fromto{\sd(P)}{\StrP}$. 
	For any $P$-stratified space $\Pi$, we define the \defn{nerve}\index[terminology]{nerve!of a stratified space} of $ \Pi $ to be the functor
	\begin{equation*}
		\Nerve_P(\Pi)\colon \fromto{\sdop(P)}{\Space}
	\end{equation*}
	given by the assignment $\goesto{\Sigma}{\Map_P(\Sigma,\Pi)}$. 
	(This is the moduli space of sections over $\Sigma$.) 
	An equivalence of $P$-stratified spaces is carried to an objectwise equivalence of functors; hence the nerve defines a functor
	\begin{equation*}
		\Nerve_P \colon \fromto{\StrP}{\Fun(\sdop(P),\Space)} \period
	\end{equation*}
	Furthermore, the assignment $\goesto{\left[\Pi\to P\right]}{(P,\Nerve_P(\Pi))}$ defines a functor
	\begin{equation*}
		\Nerve \colon \fromto{\Str}{\textstyle \int_{\Pos} \Fun(\sdop,\Space)} \period
	\end{equation*}
\end{cnstr}

\begin{exm}\label{exm:linkindecollage}
	For any poset $P$, $P$-stratified space $\Pi$, and points $p,q\in P$ such that $p\leq q$, the space
	\begin{equation*}
		\Nerve_P(\Pi)\{p \leq q\}\simeq\Map_P(\{p \leq q\},\Pi)
	\end{equation*}
	is the \emph{link} between the $p$-th and $q$-th strata of $\Pi$.
\end{exm}

Let us demonstrate that the functor $ \Nerve $ lands in the full subcategory
\begin{equation*}
	\DecSpace{}\subset \textstyle \int_{\Pos} \Fun(\sdop,\Space) \period
\end{equation*}

\begin{lem}
	For any poset $P$ and $P$-stratified space $\Pi$, the functor $\Nerve_P(\Pi)$ is a spatial décollage.
\end{lem}

\begin{proof}
	In $\Cat_{\infty,/P}$, for any chain $\{p_0 < \cdots < p_n\}\subseteq P$, there is an equivalence
	\begin{equation*}
		\equivto{\{p_0 < p_1\} \unionlimits^{\{p_1\}} \cdots \unionlimits^{\{p_{n-1}\}}\{p_{n-1} < p_n\}}{\{p_0 < \cdots < p_n\}} \comma
	\end{equation*}
	which induces an equivalence
	\begin{equation*}
		\equivto{\Map_P(\{p_0 < \cdots < p_n\},\Pi)}{\Map_P(\{p_0 < p_1\},\Pi) \crosslimits_{\Pi_{p_1}} \cdots \crosslimits_{\Pi_{p_{n-1}}} \Map_P(\{p_{n-1} < p_n\},\Pi)} \comma
	\end{equation*}
	as desired.
\end{proof}

\begin{thm}\label{thm:nerveequiv}
	The functor $ \Nerve \colon \fromto{\Str}{\DecSpace{}} $ is an equivalence of \categories over $\Pos$.
\end{thm}

\begin{proof}
	Let $P$ be a poset and write $ \DDelta_{/P} $ for the category of simplices of $ P $.
	The Joyal--Tierney theorem \cite{MR2342834} implies that the functor
	\begin{align*}
		\Nerve \colon \Cat_{\infty,/P}& \to \Fun(\Deltaop,\Space)_{/\Nerve(P)} \simeq \Fun(\Deltaop_{/P},\Space) \\ 
		C &\mapsto [\Sigma \mapsto \interior\Fun_{/P}(\Sigma,C)]
	\end{align*}
	is fully faithful, and the essential image $\CSS_{/\Nerve(P)}$ consists of those functors $\fromto{\Deltaop_{/P}}{\Space}$ that satisfy both the Segal condition and the completeness condition.
	Now notice that left Kan extension along the inclusion $ \incto{\sdop(P)}{\Deltaop_{/P}} $ defines a fully faithful functor $\incto{\DecSpace{P}}{\CSS_{/\Nerve(P)}}$ whose essential image consists of those complete Segal spaces $\fromto{C}{\Nerve(P)}$ such that for any $p\in P$, the complete Segal space $C_p$ is \agroupoid.
\end{proof}

\begin{nul}\label{nul:finitedecollages}
	Write $\DecSpacefin{} \subset \DecSpace{}$\index[notation]{DecSpi@$\DecSpacefin{}$} for the full subcategory spanned by those pairs $(P,D)$ where $P$ is a finite poset and $D$ is a spatial décollage over $P$ whose values are all \pifinite.
	Then nerve restricts to an equivalence of \categories $ \Nerve \colon \equivto{\Strfin}{\DecSpacefin{}}$.
\end{nul}


\subsection{Profinite spatial décollages}\label{subsec:profintiedecollages}

In this section we extend the theory of décollages to \proobjects and explain how to understand profinite $ P $-stratified spaces in terms of décollages valued in the \category of \textit{profinite spaces} (\Cref{cnstr:profinitespacesanddecollages}).

\begin{nul}\label{nul:pronerveequiv}
	We extend the nerve $ \Nerve \colon \equivto{\Str}{\DecSpace{}} $ to \proobjects to obtain an equivalence of \categories
	\begin{equation*}
		\Nerve \colon \equivto{\Pro(\Str)}{\Pro(\DecSpace{})}
	\end{equation*}
	over $\Pro(\Pos)$. 
\end{nul}

In order to understand profinite décollages, we recall some basic facts about profinite spaces and the product in the \category of profinite spaces.

\begin{rec}\label{rec:profinspace}
	We write $ \Spaceprofin\coloneq\Pro(\Spacefin) $\index[notation]{Spihat@$\Spaceprofin$} for the \category of \defn{profinite spaces}.%
	\index[terminology]{space!profinite}\index[terminology]{profinite!space}
	We regard the $ \Spaceprofin $ as a full subcategory of the \category $\Pro(\Space)$.
	Precomposition with the inclusion $\incto{\Spacefin}{\Space}$ defines a functor $(-)\profincomp \colon \fromto{\Pro(\Space)}{\Spaceprofin} $ that exhibits $ \Spaceprofin $ as a localization of $ \Pro(\Space) $. 
	Given a prospace $ X $, we call $ X\profincomp $ the \defn{profinite completion} of $ X $\index[terminology]{profinite completion}\index[notation]{Xpihat@$ X\profincomp $}.
\end{rec}

\begin{rec}[{monoidal structures on $ \Pro(\Space) $ \SAG{Remark}{E.2.1.2}}]\label{rec:SAG.E.2.1.2}
	In addition to the cartesian symmetric monoidal structure on $\Pro(\Space)$, there is a related `composition' monoidal structure: since the composition of two left exact accessible functors $ \fromto{\Space}{\Space} $ is again left exact and accessible, the composition monodial structure on $ \Fun(\Space,\Space)^{\op} $ restricts to a monoidal structure $ \goesto{(X,Y)}{X\circ Y} $\index[notation]{zzc@$\circ$} on $ \Pro(\Space) $.
	The identity functor is both the unit for $\circ$ and terminal object of $\Pro(\Space)$.
	Hence the universal property of the product provides is a comparison morphism
	\begin{equation*}
		c_{X,Y} \colon \fromto{X \circ Y}{X \times Y}
	\end{equation*}
	that is natural in $ X $ and $ Y $.
	However, this morphism is not generally an equivalence.

	Since the subcategory $ \Spaceprofin \subset \Pro(\Space) $ is closed under products, if $ X, Y \in \Spaceprofin $, then the morphism $ \fromto{X \circ Y}{X \times Y} $ induces a morphism 
	\begin{equation}\label{eq:profinprodcomparison}
		\fromto{(X\circ Y)\profincomp}{X\times Y} \period
	\end{equation}
	We claim that the morphism \eqref{eq:profinprodcomparison} is an equivalence.
	To see this, we show that the comparison morphism $ \fromto{X \circ Y}{X \times Y} $ becomes an equivalence after evaluation at any truncated space.\footnote{We are grateful to Jacob Lurie for this observation.}
\end{rec}

\begin{lem}\label{lem:compofprofiniteontrun}
	Let $ X $ be a profinite space and $ Y $ a prospace.
	Then for every truncated space $ K $, the morphism
	\begin{equation*}
		c_{X,Y} \colon \fromto{(X \circ Y)(K)}{(X \times Y)(K)}
	\end{equation*}
	is an equivalence in $ \Space^{\op} $.
\end{lem}

\begin{proof}
	Exhibit $ X $ an inverse system $\{X_{\alpha}\}_{\alpha \in A}$ of \textit{\pifinite} spaces and $ Y $ an inverse system $\{Y_{\beta}\}_{\beta \in B}$ of spaces.
	For each $ \alpha \in A $, the fact that the space $X_{\alpha} $ is \pifinite for implies that the functor corepresented by $ X_{\alpha} $ preserves colimits of filtered diagrams of uniformly truncated spaces \SAG{Corollary}{A.2.3.2}.
	Since the space $ K $ is truncated, the filtered diagram
	\begin{equation*}
		\goesto{\beta}{\Map_{\Space}(Y_{\beta},K)}
	\end{equation*}
	is uniformly truncated.
	Hence we have the following equivalences in $ \Space $:
	\begin{align*}
		(X\times Y)(K) &\simeq \colim_{(\alpha,\beta)\in A^{\op}\times B^{\op}} \Map_{\Space}(X_{\alpha} \cross Y_{\beta}, K)  \\
		&\equivalent \colim_{(\alpha,\beta)\in A^{\op}\times B^{\op}} \Map_{\Space}(X_{\alpha}, \Map_{\Space}(Y_{\beta}, K)) \\ 
		&\equivalence \colim_{\alpha \in A^{\op}} \Map_{\Space}(X_{\alpha},\colim_{\beta \in B^{\op}}\Map_{\Space}(Y_{\beta}, K)) \\ 
		&\simeq (X\circ Y)(K) \period \qedhere
	\end{align*}
\end{proof}

\begin{cor}\label{cor:compofprofiniteisprod}
	Let $ X $ and $ Y $ be profinite spaces.
	Then the natural morphism
	\begin{equation*}
		\fromto{(X \of Y)\profincomp}{X \cross Y}
	\end{equation*}
	is an equivalence in $ \Spaceprofin $.
\end{cor}

\begin{nul}
	\Cref{cor:compofprofiniteisprod} is helpful for describing fiber products in $\Spaceprofin$ as well: given morphisms of profinite spaces $ p \colon \fromto{X}{Z} $ and $ q \colon \fromto{Y}{Z} $, the pullback $ X \times_Z Y $ of $ p $ along $ q $ is given by a cobar construction:
	\begin{equation*}
		X \times_Z Y \simeq \lim_{[m] \in \DDelta}\left(X\circ Z^{\circ m}\circ Y\right)\profincomp\period
	\end{equation*}
\end{nul}

We are now ready to describe $ \StrprofinP $ in terms of profinite décollages.

\begin{cnstr}[profinite décollages]\label{cnstr:profinitespacesanddecollages}
	For any finite poset $ P $, write $\DecSpaceprofin{P}$\index[notation]{DecSpihat@$\DecSpaceprofin{}$} for the full subcategory of $\Fun(\sdop(P),\Spaceprofin)$ spanned by those functors
	\begin{equation*}
		D \colon \fromto{\sdop(P)}{\Spaceprofin}
	\end{equation*}
	such that for any chain $\{p_0 < \cdots < p_n\}\subseteq P$, the natural map
	\begin{equation*}
		D\{p_0 < \cdots < p_n\} \to D\{p_0 < p_1\} \crosslimits_{D\{p_1\}} \cdots \crosslimits_{D\{p_{n-1}\}} D\{p_{n-1} < p_n\}
	\end{equation*}
	is an equivalence of profinite spaces.
	We call objects of $\DecSpaceprofin{P}$ \emph{profinite décollages} over $ P $.%
	\index[terminology]{profinite!decollage@décollage}\index[terminology]{decollage@décollage!profinite}

	Combining the equivalence
	\begin{equation*}
		\equivto{\Pro(\Fun(\sdop(P),\Spacefin))}{\Fun(\sdop(P),\Spaceprofin)}
	\end{equation*}
	furnished by \HTT{Proposition}{5.3.5.15} with the equivalences \eqref{nul:finitedecollages} and \eqref{nul:pronerveequiv}, we obtain equivalences of \categories 
	\begin{equation}\label{eq:strprofindecequiv}
		\StrprofinP \equivalence \Pro(\DecSpacefin{P}) \equivalence \DecSpaceprofin{P} \period
	\end{equation}
	Specifically, the composite equivalence \eqref{eq:strprofindecequiv} is given by extending the nerve
	\begin{equation*}
		\Nerve_{P} \colon \equivto{\StrfinP}{\DecSpacefin{P} \subset \DecSpaceprofin{P}}
	\end{equation*}
	along inverse limits.
\end{cnstr}


\newpage

\part{Elements of higher topos theory}\label{part:topoi}

In this part we develop the higher-toposic tools that we'll need to state and prove our \Categorical Hochster Duality Theorem (\Cref{thm:headlineinftyHochster}=\Cref{thm:inftyHochster}).
In \cref{sec:rechighertopoi} we recall a number of important results from higher topos theory and develop the basic frameworks of \textit{(bounded) coherent \topoi} and \textit{(bounded) \pretopoi}.
The theories of coherent \topoi and bounded \pretopoi will be used heavily in the remainder of the text.
In \Cref{sec:shapetheory}, we describe the basic theorems of \textit{shape theory} for \topoi; the generalization of shape theory to stratified \topoi forms the foundation of our work.
\Cref{sec:orientedfiberprod} develops the basics of Deligne's oriented fiber product; this plays a fundamental role in our approach to stratified higher topos theory in \Cref{part:strattopoi}.
In \Cref{sec:localtopoi} we introduce the analogue of local rings in the context of higher topos theory.
In \Cref{section:BC} we generalize work of Moerdijk--Vermeulen and Illusie--Gabber by proving a key \basechange theorem for oriented fiber products of bounded coherent \topoi.
As with the proof of the proper basechange theorem in algebraic geometry, reduction to the local case plays a key role in our proof.
  

\section{Aide-mémoire on higher topoi}\label{sec:rechighertopoi}

In this chapter we recall a number of important results from higher topos theory (mostly from Lurie's \cite[Appendices \SAGapplink{A} \& \SAGapplink{E}]{SAG}), and develop some basic results that we'll use throughout the rest of the text.
This chapter is here mostly for ease of reference; much of the material is expository, and the original material mostly serves to fill small gaps in the existing foundations.

\Cref{subsec:highertopoi} sets our notational conventions for \topoi.
\Cref{subsec:boundedness} introduces the first of two finiteness conditions that we impose on almost all of the \topoi we consider in this text: \textit{boundedness}.
\Cref{subsec:coherence} introduces the second finiteness condition: \textit{coherence}.
\Cref{subsec:n-coherenceonlydependsonn-topos} studies the relationship between coherence and $ n $-topoi.
\Cref{subsec:coherenceforn-localic} provides a convenient reformulation of coherence for $ n $-localic \topoi.
\Cref{subsec:cohfor1localic} shows that a morphism of \textit{finitary \sites} induces a coherent geometric morphism on corresponding \topoi; along with the material from \cref{subsec:coherenceforn-localic}, this implies that the theory of $ 1 $-localic coherent \topoi recovers Grothendieck's theory of coherent ordinary topoi \cite[Exposé VI, Definition 2.3]{MR50:7131}.
\Cref{subsec:examplesofcoherent} uses the material from \cref{subsec:cohfor1localic} to provide examples of coherent \topoi and geometric morphisms coming from algebraic geometry.
\Cref{subsec:classificationofbctopoi} explains how bounded coherent \topoi are classified by their truncated coherent objects.
This perspective will be used extensively throughout our work.
\Cref{subsec:cohinverselimtis} recalls the fact that inverse limits of bounded coherent \topoi are again bounded coherent.
\Cref{subsec:coherenceandfilteredcolims} shows that the pushforward in a coherent geometric morphism commutes with filtered colimits of uniformly truncated diagrams.
\Cref{subsec:conceptual} discusses points of \topoi and hypercompletness, Deligne Completeness, and Lurie's Conceptual Completeness Theorem for bounded coherent \topoi.
\Cref{subsec:bases} explains how bases for Grothendieck topologies work for \topoi (which is more subtle than for ordinary topoi); the key example that we need is that the \topos of sheaves on a finite poset $ P $ is the functor \category $ \Fun(P,\Space) $.


\subsection{Higher topoi}\label{subsec:highertopoi}

We begin by setting our basic notational conventions for higher topoi. 

\begin{ntn}
	We use here the theory of \emph{$n$-topoi}\index[terminology]{ntopoi@$n$-topoi} for $n\in\NNrhd$; see \cite[\HTTch{6}]{HTT}.
	We write $\Top_n\subset\Cat_{\infty,\updelta_1}$\index[notation]{Topn@$\Top_n,\Top_{\infty}$} for the subcategory of $\updelta_1$-small $n$-topoi and geometric morphisms.
	Many of the examples in this paper will have $n\in\{0,1,\infty\}$.

	For any $ \updelta_0 $-small \category $ C $, we write $ \PSh(C) \coloneq \Fun(C^{\op},\Space) $ for the \topos of presheaves of spaces on $ C $.
\end{ntn}

\begin{exm}
	Recall that $0$-topoi are ($\updelta_0$-small) locales \HTT{Proposition}{6.4.2.5}, and $1$-topoi are topoi in the classical sense of Grothendieck \HTT{Remark}{6.4.1.3}. 
\end{exm}

\begin{exm}
	Let $m,n\in\NNrhd$ with $m\leq n$.
	An \emph{$m$-site}\index[terminology]{nsite@$n$-site}\index[terminology]{site} is a $\updelta_0$-small $m$-category $X$ equipped with a Grothendieck topology $\tau$.
	Given an $m$-site $ (X,\tau) $, we write $\Sh_{\tau}(X)_{\leq(n-1)}$\index[notation]{Shtau@$\Sh_{\tau}$} for the $ n $-topos of sheaves $\updelta_0$-small $(n-1)$-groupoids on $X$;
	when $n=\infty$, we will simply write $\Sh_{\tau}(X)$.

	It is not expected that all \topoi are of the form $\Sh_{\tau}(X)$ for an \site $ (X,\tau) $;
	however, if $n\in\NNup$, then every $n$-topos is of the form $\Sh_{\tau}(X)_{\leq(n-1)}$ for some $n$-site $(X,\tau)$ \HTT{Theorem}{6.4.1.5}.
\end{exm}

\begin{exm}\label{exm:widetilde}
	For any topological space $ W $, we write $ \Wtilde \coloneq \Sh(W) $ for the $ 0 $-localic \topos of sheaves of spaces on $W$.
\end{exm}

\begin{ntn}
	Let $n\in\NNrhd$ and let $\XX$ and $\YY$ be $n$-topoi.
	We write\index[notation]{Funstar@$\Funlowerstar,\Funupperstar$}
	\begin{equation*}
		\Funlowerstar(\XX,\YY) \subseteq \Fun(\XX,\YY)
	\end{equation*}
	for the full subcategory spanned by the \textit{geometric morphisms}. 
	We note that $\Funlowerstar(\XX,\YY)$ is accessible \HTT{Proposition}{6.3.1.13}.
	We write
	\begin{equation*}
		\Funupperstar(\YY,\XX)\subseteq\Fun(\YY,\XX)
	\end{equation*}
	for the full subcategory spanned by those functors that are left exact left adjoints, so that $ \Funupperstar(\YY,\XX) \equivalent \Funlowerstar(\XX,\YY)^{\op} $.
\end{ntn}

\begin{ntn}\label{ntn:globalsecpoints}
	Recall that the \topos $\Space$ of spaces is terminal in $\Top_{\infty}$.
	If $\XX$ is \atopos, then we write $ \Gammaup_{\XX,\ast} $ or $ \Gammaup_{\ast} $\index[notation]{Gammastar@$\Gammalowerstar$} for the unique geometric morphism $ \fromto{\XX}{\Space} $; the functor $ \Gammaup_{\XX,\ast} $ is corepresented by the terminal object $ 1_{\XX} \in \XX $.
	The geometric morphism $ \Gammaup_{\XX,\ast} $ is called the \defn{global sections}\index[terminology]{global sections} geometric morphism.
\end{ntn}

\begin{dfn}
	Let $\XX$ be \atopos.
	A \defn{point}\index[terminology]{point!of \atopos} of $\XX$ is a geometric morphism
	\begin{equation*}
		\xlowerstar\colon\fromto{\Space}{\XX} \comma
	\end{equation*}
	which is necessarily a section of $\Gammalowerstar$.
	We often write $\xtilde$\index[notation]{xtilde@$\xtilde$} for this copy of $\Space$, regarded as lying over $\XX$ via the geometric morphism $\xlowerstar$.
\end{dfn}

\begin{rec}[étale geometric morphisms]\label{HTT.6.3.5.6} 
	Let $\XX$ and $\YY$ be \topoi. 
	A geometric morphism $p_{\ast}\colon\fromto{\XX}{\YY}$ is \emph{étale}\index[terminology]{etale@étale!geometric morphism}\index[terminology]{geometric morphism!etale@étale} if $p^{\ast}$ admits a further left adjoint $p_!\colon\fromto{\XX}{\YY}$ that exhibits $\XX$ as the slice \topos \smash{$\YY_{/p_!(1_{\XX})}$.}
	In this case $ p_! $ is identified with the forgetful functor \smash{$ \fromto{\YY_{/p_!(1_{\XX})}}{\YY} $}.

	By \HTT{Corollary}{6.3.5.6}, the functor
	\begin{equation*}
		\fromto{\Funlowerstar(\ZZ,\XX)}{\Funlowerstar(\ZZ,\YY)}
	\end{equation*}
	is a right fibration whose fiber over a geometric morphism $\flowerstar\colon\fromto{\ZZ}{\YY}$ is the ($\updelta_0$-small) \groupoid $\Map_{\XX}(1_{\XX},f^{\ast}p_!(1_{\XX}))$.
\end{rec}

\begin{nul}
	If $\XX$ and $\YY$ are \topoi, the product $\XX\times \YY$\index[notation]{XY@$\XX \times \YY$} in $\Top_{\infty}$%
	\index[terminology]{product!of \topoi}\index[terminology]{topos@\topos!product}
	is \emph{not} the product of \categories;
	rather, the \topos $ \XX \cross \YY $ can be identified with the tensor product of presentable \categories.%
	\footnote{For this reason, Lurie writes $\XX\otimes \YY$ for the product in $\Top_{\infty}$.}

	Similarly, if $\flowerstar\colon\fromto{\XX}{\ZZ}$ and $\glowerstar\colon\fromto{\YY}{\ZZ}$ are geometric morphisms, then the pullback $\XX\times_{\ZZ}\YY$\index[notation]{XZY@$\XX\times_{\ZZ}\YY$} in $\Top_{\infty}$%
	\index[terminology]{fiber product!of \topoi}\index[terminology]{topos@\topos!fiber product}%
	\index[terminology]{pullback!of \topoi}\index[terminology]{topos@\topos!pullback}
	exists \HTT{Proposition}{6.3.4.6}, but the \topos $ \XX \cross_{\ZZ} \YY $ is not the pullback of \categories.

	In \Cref{sec:orientedfiberprod} we also study an \textit{oriented fiber product} of \topoi.
	Again, this oriented fiber product does not coincide with the oriented fiber product of \categories \Cref{nul:orientedfpinCat}. 
	We therefore endeavour to indicate clearly when a product, pullback, or oriented fiber product is formed in $\Top_{\infty}$ or \smash{$\Cat_{\infty,\updelta_1}$}.
\end{nul}

We repeatedly make use of the fact that inverse limits in $ \Top_{\infty} $ are computed in $ \Cat_{\infty,\updelta_1} $.
\index[terminology]{limit!inverse!of \topoi}\index[terminology]{inverse!limit!of \topoi}\index[terminology]{topos@\topos!inverse limit}

\begin{thm}[{\HTT{Theorem}{6.3.3.1}}]\label{thm:filteredlimsinRTop}
	The forgetful functor $ \fromto{\Top_{\infty}}{\Cat_{\infty,\updelta_1}} $ preserves inverse limits.
\end{thm}


\subsection{Boundedness}\label{subsec:boundedness}

We now turn to the first of two finiteness conditions that we impose on almost all of the \topoi we consider in this book.

\begin{ntn}\label{ntn:n-trun}
	Let $ C $ be a presentable \category.
	For each integer $ n \geq -2 $, write $ C_{\leq n} \subset C $\index[notation]{Cleqn@$C_{\leq n}, C_{<\infty}$} for the full subcategory spanned by the $ n $-truncated objects, and
	\begin{equation*}
		\trun_{\leq n} \colon \fromto{C}{C_{\leq n}} \index[notation]{tauleqn@$\trun_{\leq n}$}
	\end{equation*}
	for the $ n $-truncation functor, which is left adjoint to the inclusion $ C_{\leq n} \subset C $ \HTT{Proposition}{5.5.6.18}.
	Write $ C_{<\infty} \subset C $ for the full subcategory spanned by those objects which are $ n $-truncated for some integer $ n \geq -2 $.
\end{ntn}

\begin{ntn}
	If $m,n\in\NNrhd$ with $m < n$, then passage to $(m-1)$-truncated objects defines a functor
	\begin{equation*}
		(-)_{\leq m-1}\colon\fromto{\Top_n}{\Top_m} \period
	\end{equation*}
	We call a $ (-1) $-truncated object of an $ n $-topos $ \XX $ an \defn{open}\index[terminology]{open subtopos}\index[terminology]{subtopos!open} in $ \XX $ and write $ \Open(\XX) \colonequals \XX_{\leq -1} $.
	For any open $U$ of $\XX$, the étale geometric morphism $\XX_{/U} \inclusion \XX$ then exhibits $\XX_{/U}$ as an \defn{open subtopos} of $\XX$.
\end{ntn}

\begin{dfn}\label{cnstr:localictopoi}
	If $m,n\in\NNrhd$ with $m < n$, then the functor $ (-)_{\leq m-1}\colon\fromto{\Top_n}{\Top_m}$ admits a fully faithful right adjoint.
	Write $\Top_n^m\subseteq\Top_n$ for the essential image of this functor;
	this consists of those $n$-topoi $\XX$ such that, for every $n$-topos $\YY$, the functor
	\begin{equation*}
		\fromto{\Funlowerstar(\YY,\XX)}{\Funlowerstar(\YY_{\leq m-1},\XX_{\leq m-1})}
	\end{equation*}
	is an equivalence. 
	We call such $n$-topoi \defn{$m$-localic}\index[terminology]{topos@\topos!nlocalic@$n$-localic}\index[terminology]{localic} \cite[\HTTsubsec{6.4.5}]{HTT}.

	The inclusion $ \Top_{\infty}^n \subset \Top_{\infty} $ of the full subcategory of $ n $-localic \topoi admits a left adjoint
	\begin{equation*}
		\Lup_n \colon \fromto{\Top_{\infty}}{\Top_{\infty}^n} \index[notation]{Ln@$ \Lup_n $}
	\end{equation*}
	called \defn{$ n $-localic reflection}\index[terminology]{nlocalic@$n$-localic!reflection}\index[terminology]{localic!reflection}.
\end{dfn}

\begin{nul}\label{nul:nlocalicintermsofsites}
	If $n\in\NNup$, then the proof of \HTT{Proposition}{6.4.5.9} demonstrates that \atopos $\XX$ is $n$-localic if and only if $\XX \equivalent \Sh_{\tau}(X)$, where $(X,\tau)$ is a $\updelta_0$-small $n$-site \textit{with all finite limits}.
\end{nul}

\begin{exm}
	If $W$ is a topological space, then the \topos $\widetilde{W}$ of sheaves on $ W $ is $0$-localic.
\end{exm}

\begin{wrn}
	If $ (X,\tau) $ is an $ n $-site and the $ n $-category $ X $ does \textit{not} have finite limits, then the \topos $ \Sh_{\tau}(X) $ may not be $ N $-localic for any $ N \geq 0 $.
	See \SAG{Counterexample}{20.4.0.1} for a basis $ B $ for the topology on the Hilbert cube $ \prod_{i \in \ZZup} [0,1] $ for which the \topos of sheaves on $ B $ is not $ N $-localic for any $ N \geq 0 $.
\end{wrn}

\begin{exm}
	If $X$ is a scheme, then the \topos $X_{\et}$ of étale sheaves on the $1$-site of étale $X$-schemes is $1$-localic.
\end{exm}

\begin{exm}\label{exm:nlocalicslive}
	Let $ n \in \NNup $ and let $ \XX $ be an $ n $-localic \topos.
	Then \SAG{Lemma}{1.4.7.7} demonstrates that for an object $ U \in \XX $, the over \topos $ \XX_{/U} $ is $ n $-localic if and only if $ U $ is $ n $-truncated.
\end{exm}

\begin{dfn}
	We write $\Top_{<\infty}^{\wedge}$\index[notation]{Tophat@$\Top_{<\infty}^{\wedge}$} the inverse limit of \categories
	\begin{equation*}
		\Top_{<\infty}^{\wedge}\coloneq\lim_{n\in\NNup^{\op}}\Top_n
	\end{equation*}
	along the various truncation functors $ (-)_{\leq m-1}$.
	This is the \category of sequences $\{\XX_n\}_{n\in\NNup}$ in which each $\XX_n$ is an $n$-topos, along with identifications $ \equivto{(\XX_n)_{\leq m-1}}{\XX_m}$ whenever $m\leq n$.
	The truncation functors provide a functor
	\begin{equation*}
		\trun \colon\fromto{\Top_{\infty}}{\Top_{<\infty}^{\wedge}} \comma
	\end{equation*}
	which carries \atopos $\XX$ to the sequence $\{\XX_{\leq n-1}\}_{n\in\NNup}$.
\end{dfn}

\begin{cnstr}\label{cnstr:boundedtopoi}
	The functor $ \trun \colon\fromto{\Top_{\infty}}{\Top_{<\infty}^{\wedge}}$ admits a fully faithful right adjoint, which identifies $\Top_{<\infty}^{\wedge}$ with the full subcategory of $\Top_{\infty}$ spanned by the \defn{bounded \topoi}\index[terminology]{bounded!topos@\topos}\index[terminology]{topos@\topos!bounded} \SAG{Proposition}{A.7.1.5}.
	These are the \topoi that can be exhibited as inverse limits in $\Top_{\infty}$ of a diagram of localic \topoi.
	Equivalently, \atopos $ \XX $ is bounded if and only if the natural geometric morphism
	\begin{equation*}
		\fromto{\XX}{\lim_{n \in \NNup^{\op}} \Lup_n(\XX)}
	\end{equation*}
	is an equivalence. 
\end{cnstr}

Surprisingly, the functor the functor $ \trun \colon\fromto{\Top_{\infty}}{\Top_{<\infty}^{\wedge}}$ also admits a left adjoint; to state what the essential image of this left adjoint is, we first recall a bit about \textit{Postnikov completeness}.

\begin{dfn}\label{def:Postnikovstuff}
	Let $ C $ be a presentable \category.
	We say that:
	\begin{enumerate}[(\ref*{def:Postnikovstuff}.1)]
		\item\label{def:Postnikovstuff.1} \defn{Postnikov towers converge}\index[terminology]{Postnikov towers converge} in $ C $, if for every object $ X \in C $, the natural morphism $ \fromto{X}{\lim_{n \in \NNup^{\op}} \trun_{\leq n} X} $ is an equivalence in $ C $.

		\item\label{def:Postnikovstuff.2} The \category $ C $ is \defn{Postnikov complete}\index[terminology]{Postnikov complete category@Postnikov complete \category}\index[terminology]{category@\category!Postnikov complete} if the natural functor
		\begin{equation*}
			C \to \lim \bigg(
		\begin{tikzcd}[sep=1.5em]
			\cdots \arrow[r] & C_{\leq n+1} \arrow[r, "\trun_{\leq n}"] & C_{\leq n} \arrow[r] & \cdots \arrow[r, "\trun_{\leq 0}"] & C_{\leq 0} 
		\end{tikzcd}\bigg)
		\end{equation*}
		is an equivalence of \categories.
		(Here the limit is formed in $\Cat_{\infty,\updelta_1}$.)
	\end{enumerate}
\end{dfn}

\begin{nul}
	Note that Postnikov completeness implies the convergence of Postnikov towers, but not conversely.
\end{nul}

\begin{cnstr}
	The functor $ \trun \colon\fromto{\Top_{\infty}}{\Top_{<\infty}^{\wedge}}$ also admits a left adjoint, which is necessarily fully faithful.
	This identifies the \category $\Top_{<\infty}^{\wedge}$ with the full subcategory of $\Top_{\infty}$ spanned by the Postnikov complete \topoi \SAG{Corollary}{A.7.2.8}\index[terminology]{Postnikov complete topos@Postnikov complete \topos}\index[terminology]{topos@\topos!Postnikov complete}.

	We write \smash{$ (-)^{\post} $}\index[notation]{Xpost@$\XX^{\post}$} for the right adjoint to the inclusion of the full subcategory of \smash{$ \Top_{\infty} $} spanned by the Postnikov complete \topoi, and write \smash{$ (-)^{\bdd} $}\index[notation]{Xbdd@$\XX^{\bdd}$} for the left adjoint to the inclusion of the full subcategory of $ \Top_{\infty} $ spanned by the bounded \topoi, so that
	\begin{equation*}
		\XX^{\post} \coloneq \lim_{n \in \NNup^{\op}} \XX_{\leq n} \andeq \XX^{\bdd} \coloneq \lim_{n \in \NNup^{\op}} \Lup_n(\XX) \period
	\end{equation*}
	For \atopos $ \XX $, we call $ \XX^{\post} $ the \defn{Postnikov completion}\index[terminology]{Postnikov completion} of $ \XX $ and call $ \XX^{\bdd} $ the \defn{bounded reflection}\index[terminology]{bounded reflection} of $ \XX $.
\end{cnstr}

\begin{nul}
	The relationship between bounded \topoi and Postnikov complete \topoi is formally analogous to the relationship between $p$-nilpotent and $p$-complete abelian groups.
	Of course $p$-nilpotent and $p$-complete abelian groups form equivalent categories, but their embeddings into the category of all abelian groups differ.
\end{nul}


\subsection{Coherence}\label{subsec:coherence}

The second finiteness condition that we impose on almost all of the \topoi we consider is \emph{coherence}.
Coherence is the topos-theoretic analogue of being quasicompact and quasiseparated in the world of schemes.
Indeed, the étale \topos of a scheme $ X $ is coherent if and only if $ X $ is quasicompact and quasiseparated (see \Cref{prop:coherentschemetopos}).

\begin{dfn}[coherence]\label{rec:coherence}
	Let $ 0 \leq r \leq \infty $, and let $\XX$ be an $ r $-topos.
	We say that $\XX$ is \defn{$0$-coherent} if and only if the $0$-topos (=locale) $\Open(\XX)$ is quasicompact.
	Let $n\in\NNup^{\ast}$, and define $n$-coherence of $ r $-topoi and their objects recursively as follows.
	\begin{itemize}
		\item An object $U\in \XX$ is \defn{$n$-coherent}\index[terminology]{ncoherent@$ n $-coherent!topos@\topos}\index[terminology]{topos@\topos!n-coherent@$ n $-coherent}\index[terminology]{ncoherent@$ n $-coherent!object} if and only if the $ r $-topos $\XX_{/U}$ is $n$-coherent.

		\item The $ r $-topos $\XX$ is \defn{locally $n$-coherent}\index[terminology]{locally ncoherent@locally $ n $-coherent!topos@\topos}\index[terminology]{topos@\topos!locally ncoherent@locally $ n $-coherent} if and only if every object $U\in \XX$ admits a cover $\{\fromto{V_i}{U}\}_{i\in I}$ in which each $V_i$ is $n$-coherent.

		\item The $ r $-topos $\XX$ is \defn{$(n+1)$-coherent} if and only if $ \XX $ is locally $n$-coherent, and the $n$-coherent objects of $\XX$ are closed under finite products.
	\end{itemize}

	An $ r $-topos $\XX$ is \defn{coherent}\index[terminology]{coherent!topos@\topos}\index[terminology]{topos@\topos!coherent} if and only if $ \XX $ is $n$-coherent for every $n\in\NNup$, and an object $U$ of an $ r $-topos $\XX$ is \defn{coherent}\index[terminology]{coherent@!object} if and only if $\XX_{/U}$ is a coherent $ r $-topos.
	Finally, an $ r $-topos $\XX$ is \defn{locally coherent}\index[terminology]{locally coherent!topos@\topos}\index[terminology]{topos@\topos!locally coherent} if and only if every object $U\in \XX$ admits a cover $\{\fromto{V_i}{U}\}_{i\in I}$ in which each $V_i$ is coherent.
\end{dfn}

\begin{nul}
	In particular, if $\XX$ is locally $n$-coherent, then $U\in \XX$ is $(n+1)$-coherent if and only if $ U $ is $n$-coherent and for any pair $U', V\in \XX_{/U}$ of $n$-coherent objects, the fiber product $ U' \times_U V $ is $ n $-coherent.
\end{nul}

\begin{nul}
	We are mostly interested in coherence for \topoi, however we have introduced the notion for $ r $-topoi in general because \atopos $ \XX $ is $ n $-coherent if and only if its underlying $ n $-topos $ \XX_{\leq n-1} $ is $ n $-coherent (this is the content of \cref{subsec:n-coherenceonlydependsonn-topos}).
\end{nul}

\begin{ntn}
	Let $ r \in \NNrhd $, and let $\XX$ be an $ r $-topos. 
	Write $ \XXcoh \subset \XX $\index[notation]{Xcoh@$ \XXcoh $} for the full subcategory of $ \XX $ spanned by the coherent objects and $ \XXcohbdd \subset \XX $\index[notation]{Xbc@$ \XXcohbdd $} for the full subcategory of $ \XX $ spanned by the truncated coherent objects.
	For each integer $ n \geq 0 $, write $ \XX^{\ncoh} \subset \XX $\index[notation]{Xncoh@$\XX^{\ncoh}$} for the full subcategory spanned by the $ n $-coherent objects. 
\end{ntn}

\begin{exm}\label{exm:Spacecoh}
	A space $ K \in \Space $ is truncated coherent if and only if $ K $ is \pifinite.
	That is to say,
	\begin{equation*}
		\Space\cohbdd = \Spacefin \period
	\end{equation*}
\end{exm}

\begin{exm}\label{exm:boundedcohisloccoh}
	By \SAG{Proposition}{A.7.5.1}, if $ \XX $ is a bounded coherent \topos, then $ \XX $ is also locally coherent.
\end{exm}

\begin{nul}\label{nul:truncohslice}
	Let $ 0 \leq r \leq \infty $, let $ \XX $ be an $ r $-topos, and let $ U \in \XX $.
	Then for any integer $ n \geq 0 $, an object $ \fromto{U'}{U} $ of $ \XX_{/U} $ is $ n $-coherent if and only if $ U' $ is $ n $-coherent when viewed as an object of $ \XX $.
	Thus we have canonical identifications
	\begin{equation*}
		(\XX^{\ncoh})_{/U} = (\XX_{/U})^{\ncoh} \andeq (\XX^{\coh})_{/U} = (\XX_{/U})^{\coh}
	\end{equation*}
	as full subcategories of $ \XX_{/U} $.
	If $ U \in \XX_{<\infty} $ is a truncated object, then we have a canonical identification
	\begin{equation*}
		(\XXcohbdd)_{/U} = (\XX_{/U})_{<\infty}^{\coh}
	\end{equation*}
	as full subcategories of $ \XX_{/U} $.
\end{nul}

\begin{dfn}\label{def:coherentmors}
	Let $\XX$ and $\YY$ be \topoi.
	We say that a geometric morphism
	\begin{equation*}
		\flowerstar\colon\fromto{\XX}{\YY}
	\end{equation*}
	is \defn{coherent}\index[terminology]{coherent!geometric morphism}\index[terminology]{geometric morphism!coherent} if and only if, for any coherent object \smash{$ F \in \YYcoh $}, the object $f^{\ast}(F) \in \XX $ is coherent as well. 
	We write \smash{$\Topcoh$}\index[notation]{Topcoh@$ \Topcoh $} for the subcategory of $\Top_{\infty}$ whose objects are coherent \topoi and whose morphisms are coherent geometric morphisms.
\end{dfn}

We defer examples of coherent \topoi to \cref{subsec:examplesofcoherent}; we do this in order to put all of our examples from algebraic geometry on the same footing after developing the basic calculus of \textit{finitary sites} in this section and in \cref{subsec:cohfor1localic}.

\begin{dfn}\label{dfn:finitaryinftysite}
	\Asite $(X,\tau)$ is \defn{finitary}\index[terminology]{site@\site!finitary}\index[terminology]{finitary!site@\site} if and only if $ X $ admits all fiber products, and, for every object $ U \in X $ and every covering sieve $ S \subset X_{/U} $, there is a finite subset $\{U_i\}_{i\in I} \subset S $ that generates a covering sieve.

	Let $ (X,\tau_X) $ and $ (Y,\tau_Y) $ be finitary \sites.
	A morphism of \sites
	\begin{equation*}
		\fupperstar \colon \fromto{(Y,\tau_Y)}{(X,\tau_X)}
	\end{equation*}
	is a \defn{morphism of finitary \sites} if $ \fupperstar $ is preserves fiber products. 
\end{dfn}

\begin{prp}[\SAG{Proposition}{A.3.1.3}]\label{prop:SAG.A.3.1.3}
	Let $(X,\tau)$ be a finitary \site and write $ \yo_{\tau} \colon \fromto{X}{\Sh_{\tau}(X)} $\index[notation]{yotau@$\yo_{\tau}$} for the sheafified Yoneda embedding.
	Then:
	\begin{enumerate}[{\upshape (\ref*{prop:SAG.A.3.1.3}.1)}]
		\item The \topos $\Sh_{\tau}(X)$ locally coherent.

		\item For every object $x\in X$, the sheaf $ \yo_{\tau}(x)$ is a coherent object of $\Sh_{\tau}(X)$.
	
		\item If, in addition, $ X $ admits a terminal object, then $ \Sh_{\tau}(X) $ is coherent.
	\end{enumerate}
\end{prp}

An elementary way to construct a finitary \site is to make use of \acategorical analogue of the notion of pretopology on a $1$-category.

\begin{dfn}\label{dfn:presite}
	An \defn{\presite}\index[terminology]{presite@\presite} is a pair $(X,E)$ consisting of \acategory $X$ along with a subcategory $E\subseteq X$ satisfying the following conditions.
	\begin{itemize}
		\item The subcategory $E$ contains all equivalences of $X$.

		\item The \category $X$ admits finite limits, and $E$ is stable under base change.

		\item The \category $X$ admits finite coproducts, finite coproducts are universal in $ X $, and $E$ is closed under finite coproducts.
	\end{itemize} 
\end{dfn}

\begin{cnstr}[{\SAG{Proposition}{A.3.2.1}}]
	Let $(X,E)$ be \apresite.
	Then there exists a topology $\tau_E$ in which the $\tau_E$-covering sieves are generated by \textit{finite} families $\{y_i\to x\}_{i\in I}$ such that $\coprod_{i\in I}y_i\to x$ lies in $E$.
	The \site $(X,\tau_E)$ is finitary.
\end{cnstr}


\subsection{Coherence \& \texorpdfstring{$n$}{n}-topoi}\label{subsec:n-coherenceonlydependsonn-topos}

In this section we prove that the property that \atopos $ \XX $ be $ n $-coherent only depends on its underlying $ n $-topos $ \XX_{\leq n-1} $ of $ (n-1) $-truncated objects (\Cref{cor:n-coherenceonlydependsonn-topos}).\footnote{We are grateful to Jacob Lurie for conveying this observation.}
We begin with some preliminaries on the relationship between coherence and connectivity.

\begin{prp}[{\SAG{Proposition}{A.2.4.1}}]\label{prop:SAG.A.2.4.1}
	Let $ \XX $ be \atopos, let $ f \colon \fromto{U}{V} $ be a morphism in $ \XX $, and let $ n \in \NNup $.
	Then: 
	\begin{enumerate}[{\upshape (\ref*{prop:SAG.A.2.4.1}.1)}]
		\item If $ U $ is $ n $-coherent and $ f $ is $ n $-connective, then $ V $ is $ n $-coherent.

		\item If $ V $ is $ n $-coherent and $ f $ is $ (n+1) $-connective, then $ U $ is $ n $-coherent.
	\end{enumerate}
\end{prp}

\noindent Since the natural morphism from an object in \atopos to its $ n $-truncation is $ (n+1) $-connective, we deduce:

\begin{cor}\label{cor:cohifftrunctionis}
	Let $ \XX $ be \atopos and $ n \in \NNup $.
	An object $ U \in \XX $ is $ n $-coherent if and only if $ \trun_{\leq n-1}(U) $ is an $ n $-coherent object of $ \XX $.
\end{cor}

\noindent It is also easy to deduce the following.

\begin{cor}[{\SAG{Corollary}{A.2.4.4}}]\label{cor:SAG.A.2.4.4}
	Let $ \XX $ be a coherent \topos and $ n \in \NNup $.
	Then for any $ n $-coherent object $ U \in \XX $, the $ (n-1) $-truncation $ \trun_{\leq n-1}(U) $ of $ U $ is a coherent object of $ \XX $.
\end{cor} 

\begin{cor}
	Let $ \XX $ be a coherent \topos.
	Then an object $ U \in \XX $ is coherent if and only if for every $ n \in \NNup $, the $ (n-1) $-truncation $ \trun_{\leq n-1}(U) $ of $ X $ is a coherent object of $ \XX $.
\end{cor} 

\begin{cor}\label{cor:criterionforcoherence}
	Let $ \flowerstar \colon \fromto{\XX}{\YY} $ be a geometric morphism between coherent \topoi.
	Then $ \flowerstar $ is coherent if and only if $ \fupperstar $ carries $ \YYcohbdd $ to $ \XXcohbdd $.
\end{cor}

We also deduce that coherence of a geometric morphism between coherent \topoi is equivalent to the \textit{a priori} stronger condition that the pullback functor preserve $ n $-coherent objects for all $ n \geq 0 $:%
\footnote{This second notion is how Grothendieck and Verdier originally defined coherence for geometric morphisms between ordinary topoi \cite[Exposé VI, Définition 3.1]{MR50:7131}.}

\begin{cor}\label{cor:strongcoherenceofmorphisms}
	Let $ \flowerstar \colon \fromto{\XX}{\YY} $ be a geometric morphism between coherent \topoi.
	Then $ \flowerstar $ is coherent if and only if $ \fupperstar $ carries $ n $-coherent objects of $ \YY $ to $ n $-coherent objects of $ \XX $ for all $ n \in \NNup $.
\end{cor}

\begin{proof}
	It is immediate from the definition that if $ \fupperstar $ preserves $ n $-coherence for all $ n \geq 0 $, then $ \flowerstar $ is coherent.
	Conversely, assume that $ \flowerstar $ is coherent, and let $ U \in \YY $ be an $ n $-coherent object.
	Since $ \YY $ is coherent, \Cref{cor:SAG.A.2.4.4}=\allowbreak\SAG{Corollary}{A.2.4.4} shows that $ \trun_{\leq n-1}^{\YY}(U) $ is an $ n $-coherent object of $ \YY $.
	Since $ \flowerstar $ is coherent, we see that
	\begin{equation*}
		\fupperstar \trun_{\leq n-1}^{\YY}(U) \equivalent \trun_{\leq n-1}^{\XX}(\fupperstar(U))
	\end{equation*}
	is a coherent object of $ \XX $.
	\Cref{cor:cohifftrunctionis} then shows that $ \fupperstar(U) $ is an $ n $-coherent object of $ \XX $. 
\end{proof}

Before showing that $ n $-coherence only depends on the underlying $ n $-topos, we need two preliminary facts on $ m $-connective morphisms in \atopos.

\begin{lem}\label{lem:fibprodconnective}
	Let $ \XX $ be \atopos and $ m \geq 0 $ an integer. 
	Let $ W \in \XX $ and let $ u \colon \fromto{U'}{U} $ and $ v \colon \fromto{V'}{V} $ be morphisms in $ \XX_{/W} $.
	If $ u $ and $ v $ are $ m $-connective morphisms of $ \XX $, then the induced morphism $ u \cross_W v \colon \fromto{U' \cross_W V'}{U \cross_W V} $ is $ m $-connective.
\end{lem}

\begin{proof}
	First we treat the case where $ W = 1_{\XX} $ is the terminal object of $ \XX $.
	In this case, since $ \trun_{\leq m-1} \colon \fromto{\XX}{\XX} $ preserves finite products \HTT{Lemma}{6.5.1.2} and $ \trun_{\leq m-1}(u) $ and $ \trun_{\leq m-1}(v) $ are equivalences by assumption, we see that
	\begin{equation*}
		\trun_{\leq m-1}(u \cross v) \equivalent \trun_{\leq m-1}(u) \cross \trun_{\leq m-1}(v)
	\end{equation*}
	is an equivalence.

	Now we treat the general case.
	In the diagram 
	\begin{equation*}
      \begin{tikzcd}[column sep=4em, row sep=2.5em]
	       U' \cross_{W} V' \arrow[dr, phantom, very near start, "\lrcorner", xshift=-1em, yshift=0.25em] \arrow[d] \arrow[r, "u \cross_W v"] & U \cross_W V \arrow[dr, phantom, very near start, "\lrcorner", xshift=-1em, yshift=0.25em] \arrow[d] \arrow[r] & W \arrow[d, "\upDelta_W"] \\ 
	       U' \cross V' \arrow[r, "u \cross v"'] & U \cross V \arrow[r] & W \cross W
      \end{tikzcd}
    \end{equation*}
    both squares are pullbacks and $ u \cross v $ is $ m $-connective (by the preceding paragraph).
    This completes the proof since the class of $ m $-connective morphisms in \atopos is stable under pullback \HTT{Proposition}{6.5.1.16}.
\end{proof}

The following is a useful strenthening of the fact that $ n $-truncation commutes with basechange along a morphism between $ n $-truncated objects \cite[Lemma 1.8]{MR3641669}:

\begin{lem}\label{lem:truncationcommuteswithcertainpullbacks}
	Let $ \XX $ be \atopos and $ n \in \NNup $. 
	Let $ W \in \XX $ and let $ \fromto{U}{W} $ and $ \fromto{V}{W} $ be morphisms in $ \XX $.
	If $ W $ is $ n $-truncated, then the natural morphism
	\begin{equation*}
		\fromto{\trun_{\leq n}(U \cross_W V)}{\trun_{\leq n}(U) \cross_W \trun_{\leq n}(V)}
	\end{equation*}
	is an equivalence.
\end{lem}

\begin{proof}
	Since the natural morphisms $ \surjto{U}{\trun_{\leq n}(U)} $ and $ \surjto{V}{\trun_{\leq n}(V)} $ are $ (n+1) $-connective, by \Cref{lem:fibprodconnective} the natural morphism
	\begin{equation*}
		\phi \colon \fromto{U \cross_W V}{\trun_{\leq n}(U) \cross_W \trun_{\leq n}(V)}
	\end{equation*}
	is $ (n+1) $-connective.
	Since $ W $ is $ n $-truncated and the $ n $-truncated objects of \atopos are closed under limits, the object $ \trun_{\leq n}(U) \cross_W \trun_{\leq n}(V) $ is $ n $-truncated.
	By the uniqueness of the factorization of a morphism in \atopos into an $ (n+1) $-connective morphism followed by an $ n $-truncated morphism, we see that $ \phi $ exhibits $ \trun_{\leq n}(U) \cross_W \trun_{\leq n}(V) $ as the $ n $-truncation of $ U \cross_W V $.
\end{proof}

\begin{prp}\label{prp:n-coherenceonlydependsonn-topos}
	Let $ \XX $ be \atopos and $ n \in \NNup $.
	The following are equivalent for an $ (n-1) $-truncated object $ W \in \XX_{\leq n-1} $:
	\begin{enumerate}[{\upshape (\ref*{prp:n-coherenceonlydependsonn-topos}.1)}]
		\item\label{prp:n-coherenceonlydependsonn-topos.1} As an object of the \topos $ \XX $, the object $ W $ is $ n $-coherent.

		\item\label{prp:n-coherenceonlydependsonn-topos.2} As an object of the $ n $-topos $ \XX_{\leq n-1} $, the object $ W $ is $ n $-coherent.
	\end{enumerate}
\end{prp}

\begin{proof}
	Clearly \enumref{prp:n-coherenceonlydependsonn-topos}{1} implies \enumref{prp:n-coherenceonlydependsonn-topos}{2}.
	We prove that \enumref{prp:n-coherenceonlydependsonn-topos}{2} implies \enumref{prp:n-coherenceonlydependsonn-topos}{1} by induction on $ n $.
	The base case $ n = 0 $ is immediate from the definition of $ 0 $-coherence.

	For the induction step assume we have shown that an $ (n-1) $-truncated object of $ \XX $ is $ n $-coherent if it is $ n $-coherent as an object of the $ n $-topos $ \XX_{\leq n-1} $.
	Let $ W $ be an $ n $-truncated object of $ \XX $ that is $ (n+1) $-coherent as an object of the $ (n+1) $-topos $ \XX_{\leq n} $; we prove that $ W $ is $ (n+1) $-coherent as an object of the \topos $ \XX $.
	First we show that $ \XX_{/W} $ is locally $ n $-coherent.
	Let $ f \colon \fromto{U}{W} $ be a morphism in $ \XX $.
	Since $ W $ is $ n $-truncated, $ f $ factors as a composite
	\begin{equation*}
		U \surjection \trun_{\leq n}(U) \to W \period
	\end{equation*}
	Since $ \XX_{\leq n,/W} $ is locally $ n $-coherent by the inductive hypothesis, there exists a cover
	\begin{equation*}
		\{U_i \to \trun_{\leq n}(U)\}_{i \in I}
	\end{equation*}
	of $ \trun_{\leq n}(U) $ such that for each $ i \in I $, the object $ U_i \in \XX_{\leq n,/W} $ is an $ n $-coherent object of $ \XX_{\leq n,/W} $.
	Equivalently, each $ U_i $ is an $ n $-coherent object of $ \XX_{\leq n} $ \cref{nul:truncohslice}.
	Since the morphism $ \surjto{U}{\trun_{\leq n}(U)} $ is $ (n+1) $-connective, \Cref{prop:SAG.A.2.4.1}=\allowbreak\SAG{Proposition}{A.2.4.1} and the fact that $ (n+1) $-connective morphisms in \atopos are stable under pullback \HTT{Proposition}{6.5.1.16} show that the family 
	\begin{equation*}
		\{U_i \cross_{\trun_{\leq n}(U)} U \to U\}_{i \in I}
	\end{equation*}
	is a cover of $ U $ in $ \XX_{/W} $ by $ n $-coherent objects.
	That is, $ \XX_{/W} $ is locally $ n $-coherent.

	Now let us show that the $ n $-coherent objects of $ \XX_{/W} $ are stable under finite products.
	Let $ f \colon \fromto{U}{W} $ and $ g \colon \fromto{V}{W} $ be morphisms in $ \XX_{/W} $, where $ U $ and $ V $ are $ n $-coherent.
	Then since the $ n $-coherent objects of $ \XX_{\leq n,/W} $ are stable under finite products by the inductive hypothesis, we see that $ \trun_{\leq n}(U) \cross_{W} \trun_{\leq n}(V) $ is an $ n $-coherent object of $ \XX_{\leq n,W} $.
	By the inductive hypothesis and \Cref{cor:cohifftrunctionis}, the object $ \trun_{\leq n}(U) \cross_{W} \trun_{\leq n}(V) $  is an $ n $-coherent object of $ \XX_{/W} $.
	The claim now follows from the fact that the natural morphism
	\begin{equation*}
		\fromto{U \cross_W V}{\trun_{\leq n}(U) \cross_W \trun_{\leq n}(V)}
	\end{equation*}
	is $ (n+1) $-connective (\Cref{lem:truncationcommuteswithcertainpullbacks}) and \Cref{prop:SAG.A.2.4.1}=\allowbreak\SAG{Proposition}{A.2.4.1}.
\end{proof} 
	
\noindent Setting $ W = 1_{\XX} $ in \Cref{prp:n-coherenceonlydependsonn-topos} we deduce:

\begin{cor}\label{cor:n-coherenceonlydependsonn-topos}
	Let $ n \in \NNup $.
	The following are equivalent for \atopos $ \XX $:
	\begin{enumerate}[{\upshape (\ref*{cor:n-coherenceonlydependsonn-topos}.1)}]
		\item The \topos $ \XX $ is $ n $-coherent.

		\item The $ n $-topos $ \XX_{\leq n-1} $ is $ n $-coherent.
	\end{enumerate}
\end{cor}

For the next few results, please recall the notations of \Cref{cnstr:boundedtopoi}.

\begin{cor}\label{cor:n-coherencecanbecheckedonn-localicreflections}
	Let $ n \in \NNup $ and let $ \flowerstar \colon \fromto{\XX}{\YY} $ be a geometric morphism of \topoi.
	If $ \flowerstar $ induces an equivalence $ \equivto{\XX_{\leq n-1}}{\YY_{\leq n-1}} $, then $ \XX $ is $ n $-coherent if and only if $ \YY $ is $ n $-coherent.
	Equivalently, if $ \flowerstar $ induces an equivalence $ \equivto{\Lup_n(\XX)}{\Lup_n(\YY)} $ on $ n $-localic reflections, then $ \XX $ is $ n $-coherent if and only if $ \YY $ is $ n $-coherent.
\end{cor}

\noindent \Cref{cor:n-coherencecanbecheckedonn-localicreflections} shows that there are many different ways to check the $ n $-coherence of \atopos.

\begin{lem}\label{lem:equivalentconditionsforn-coherence}
	Let $ n \in \NNup $.
	The following are equivalent for \atopos $ \XX $:
	\begin{enumerate}[{\upshape (\ref*{lem:equivalentconditionsforn-coherence}.1)}]
		\item\label{lem:equivalentconditionsforn-coherence.1} The \topos $ \XX $ is $ n $-coherent.

		\item\label{lem:equivalentconditionsforn-coherence.2} The $ n $-localic reflection $ \Lup_n(\XX) $ of $ \XX $ is $ n $-coherent.

		\item\label{lem:equivalentconditionsforn-coherence.3} The hypercompletion $ \XX^{\hyp} $ of $ \XX $ is $ n $-coherent (see \Cref{def:hypercompleteness}).

		\item\label{lem:equivalentconditionsforn-coherence.4} The Postnikov completion $ \XX^{\post} $ of $ \XX $ is $ n $-coherent.

		\item\label{lem:equivalentconditionsforn-coherence.5} The bounded reflection $ \XX^{\bdd} $ of $ \XX $ is $ n $-coherent.
	\end{enumerate}
\end{lem}

\begin{proof}
	The equivalence of these statements follows from repeated application of \Cref{cor:n-coherencecanbecheckedonn-localicreflections}.
	The equivalence of \enumref{lem:equivalentconditionsforn-coherence}{1} and \enumref{lem:equivalentconditionsforn-coherence}{2} follows immediately from \Cref{cor:n-coherencecanbecheckedonn-localicreflections}.

	To see that \enumref{lem:equivalentconditionsforn-coherence}{1}$ \Leftrightarrow $\enumref{lem:equivalentconditionsforn-coherence}{3}, note that since truncated objects are hypercomplete, the natural fully faithful geometric morphism $ \incto{\XX^{\hyp}}{\XX} $ induces an equivalence on $ (n-1) $-truncated objects.

	To see that \enumref{lem:equivalentconditionsforn-coherence}{1}$ \Leftrightarrow $\enumref{lem:equivalentconditionsforn-coherence}{4}, note that by \SAG{Proposition}{A.7.3.7} the natural geometric morphism $ \fromto{\XX^{\post}}{\XX} $ is an equivalence when restricted to $ (n-1) $-truncated objects.

	To see that \enumref{lem:equivalentconditionsforn-coherence}{1}$ \Leftrightarrow $\enumref{lem:equivalentconditionsforn-coherence}{5}, note that since the $ n $-localic reflection functor $ \Lup_n \colon \fromto{\Top_{\infty}}{\Top_{\infty}} $ preserves inverse limits \SAG{Lemma}{A.7.1.4}, the natural geometric morphism
	\begin{equation*}
		\fromto{\XX}{\XX^{\bdd} \equivalent \lim_{k \in \NNup^{\op}} \Lup_k(\XX)}
	\end{equation*}
	induces an equivalence on $ n $-localic reflections.
\end{proof}

\begin{nul}[equivalent conditions for coherence of geometric morphisms]\label{nul:coherencecheckedaftercompletion}
	Let 
	\begin{equation*}
		\begin{tikzcd}
			\XX' \arrow[r, "\xlowerstar"] \arrow[d, swap, "\flowerstar'"] & \XX \arrow[d, "\flowerstar"] \\
			\YY' \arrow[r, "\ylowerstar"'] & \YY 
		\end{tikzcd}
	\end{equation*}
	be a commutative square of coherent \topoi in $ \Top_{\infty} $.
	\Cref{lem:equivalentconditionsforn-coherence,cor:criterionforcoherence} show that if $ \xlowerstar $ and $ \ylowerstar $ induce equivalences
	\begin{equation*}
		\equivto{\XX'_{<\infty}}{\XX_{<\infty}} \andeq \equivto{\YY'_{<\infty}}{\YY_{<\infty}}
	\end{equation*}
	on truncated objects, then $ \flowerstar $ is coherent if and only if $ \flowerstar' $ is coherent.

	In particular, the following are equivalent for a geometric morphism $ \flowerstar \colon \fromto{\XX}{\YY} $ between coherent \topoi:
	\begin{enumerate}[(\ref*{nul:coherencecheckedaftercompletion}.1)]
		\item The geometric morphism $ \flowerstar \colon \fromto{\XX}{\YY} $ is coherent.

		\item The induced geometric morphism $ \flowerstar^{\hyp} \colon \fromto{\XX^{\hyp}}{\YY^{\hyp}} $ on hypercompletions is coherent.

		\item The induced geometric morphism $ \flowerstar^{\post} \colon \fromto{\XX^{\post}}{\YY^{\post}} $ on Postnikov completions is coherent.

		\item The induced geometric morphism $ \flowerstar^{\bdd} \colon \fromto{\XX^{\bdd}}{\YY^{\bdd}} $ on bounded refelections is coherent.
	\end{enumerate}

	Thus the equivalence between Postnikov complete \topoi and bounded \topoi (\Cref{cnstr:boundedtopoi}) restricts to an equivalence between the subcategory of Postnikov complete coherent \topoi and coherent geometric morphisms and the subcategory of bounded coherent \topoi and coherent geometric morphisms.
\end{nul}


\subsection{Coherence of morphisms \& \texorpdfstring{$n$}{n}-localic \texorpdfstring{$\infty$}{∞}-topoi}\label{subsec:coherenceforn-localic}

In this section we prove that coherence for an $ n $-localic \topos is equivalent to $ (n+1) $-coherence, and may be checked on its underling $ n $-topos (\Cref{prop:coherenceforn-localic}).
First we'll need \toposic versions of a number of points from \cite[Exposé VI, \S\S 1--3]{MR50:7131}; these follow easily from \cite[\SAGsubsec{A.2.1}]{SAG}.

\begin{dfn}\label{def:relativecoh}
	Let $ n \in \NNup $ and let $ \XX $ be a locally $ n $-coherent \topos. 
A morphism $ \fromto{U}{V} $ in $ \XX $ is called \defn{relatively $ n $-coherent}\index[terminology]{morphism!relatively ncoherent@relatively $n$-coherent}\index[terminology]{relatively ncoherent@relatively $n$-coherent!morphism} if for every $ n $-coherent object $ V' \in \XX $ and every morphism $ \fromto{V'}{V} $, the fiber product $ U \cross_V V' $ is also $ n $-coherent.
\end{dfn}

\begin{exm}[\SAG{Example}{A.2.1.2}]\label{exm:SAG.A.2.1.2}
	Let $ \XX $ be a locally $ n $-coherent \topos and $ f \colon \fromto{U}{V} $ a morphism in $ \XX $.
	If $ U $ is $ n $-coherent and $ V $ is $ (n+1) $-coherent, then $ f $ is relatively $ n $-coherent.
\end{exm}

\begin{exm}\label{exm:n+1constronglyrelatn-coh}
	As a consequence of \Cref{prop:SAG.A.2.4.1}=\allowbreak\SAG{Proposition}{A.2.4.1} and the fact that the class of $ (n+1) $-connective morphisms in \atopos is stable under pullback \HTT{Proposition}{6.5.1.16}, the $ (n+1) $-connective morphism of \atopos are `relatively $ n $-coherent' in a very strong sense: they satisfy the condition of relative $ n $-coherence without the need of local $ n $-coherence assumptions on the \topos.
\end{exm}

\begin{lem}\label{lem:pullbackofreln-coh}
	Let $ n \in \NNup $ and let $ \XX $ be a locally $ n $-coherent \topos.
	Let $ u \colon \fromto{U'}{U} $ and $ v \colon \fromto{V'}{V} $ be relatively $ n $-coherent morphisms in $ \XX $, $ W \in \XX $ an object, and $ \fromto{U}{W} $ and $ \fromto{V}{W} $ be any morphisms.
	Then the induced morphism $ \fromto{U' \cross_W V'}{U \cross_W V} $ is relatively $ n $-coherent.
\end{lem}

\begin{proof}
	Let $ f \colon \fromto{X}{U \cross_W V} $ be a morphism in $ \XX $ where $ X $ is $ n $-coherent.
	Note that we have equivalences of iterated fiber products
	\begin{align*}
		X \operatornamewithlimits{\cross}_{U \cross_W V} (U' \cross_W V') &\equivalent (X \cross_U U') \cross_X (X \cross_V V') \\
		&\equivalent (X \cross_U U') \cross_V V' \period
	\end{align*} 
	First, since $ X \cross_U U' $ is the pullback of $ \pr_1 \of f \colon \fromto{X}{U} $ along the relatively $ n $-coherent morphism $ u $, the object $ X \cross_U U' $ is $ n $-coherent.
	Second, $ (X \cross_U U') \cross_V V' $ is the pullback of the morphism $ \fromto{X \cross_U U'}{V} $ induced by $ \pr_2 \of f \colon \fromto{X}{V} $ along the relatively $ n $-coherent morphism $ v $.
	Hence $ (X \cross_U U') \cross_V V' $ is an $ n $-coherent object of $ \XX $, as desired.
\end{proof}

\begin{lem}\label{lem:closureofn-cohunderprodu}
	Let $ \XX $ be \atopos and $ m \in \NNup $.
	Let $ \XX_0 \subset \XX $ be a full subcategory satisfying the following conditions:
	\begin{enumerate}[{\upshape (\ref*{lem:closureofn-cohunderprodu}.1)}]
		\item\label{lem:closureofn-cohunderprodu.1} The full subcategory $ \XX_0 \subset \XX $ is closed under finite products.

		\item\label{lem:closureofn-cohunderprodu.2} Every object of $ \XX_0 $ is $ m $-coherent.

		\item\label{lem:closureofn-cohunderprodu.3} For every object $ U \in \XX $, there exists an effective epimorphism $ \surjto{\coprod_{i \in I} U_i}{U} $ where $ U_i \in \XX_0 $ for each $ i \in I $.
	\end{enumerate}
	Then the $ m $-coherent objects of $ \XX $ are closed under finite products.
\end{lem}

\begin{proof}
	Let \smash{$ \XX'_0 \subset \XX $} denote the closure of $ \XX_0 $ under finite coproducts; then every object of \smash{$ \XX'_0 $} is $ m $-coherent.
	Since colimits in $ \XX $ are universal and $ \XX_0 $ is closed under finite products, \smash{$ \XX'_0 \subset \XX $} is closed under finite products.

	Let $ U, V \in \XX $ be $ m $-coherent objects; we show that $ U \cross V $ is $ m $-coherent.
	Since $ U $ and $ V $ are quasicompact, there exist effective epimorphisms $ u \colon \surjto{U'}{U} $ and $ v \colon \surjto{V'}{V} $ where $ U',V' \in \XX'_0 $. 
	By \Cref{exm:SAG.A.2.1.2}=\allowbreak\SAG{Example}{A.2.1.2} both $ u $ and $ v $ are relatively $ (m-1) $-coherent.
	\Cref{lem:pullbackofreln-coh} shows that
	\begin{equation*}
		u \cross v \colon \surjto{U' \cross V'}{U \cross V}
	\end{equation*}
	is a relatively $ (m-1) $-coherent effective epimorphism.
	Since $ U' \cross V' \in \XX'_0 $ is $ m $-coherent and $ \XX $ is locally $ m $-coherent, \SAG{Proposition}{A.2.1.3} shows that $ U \cross V $ is $ m $-coherent, as desired.
\end{proof}

\begin{prp}\label{prop:coherenceforn-localic}
	Let $ n \in \NNup $.
	The following are equivalent for an $ n $-localic \topos $ \XX $:
	\begin{enumerate}[{\upshape (\ref*{prop:coherenceforn-localic}.1)}]
		\item\label{prop:coherenceforn-localic.1} The $ n $-topos $ \XX_{\leq n-1} $ is $ (n+1) $-coherent.

		\item\label{prop:coherenceforn-localic.2} The \topos $ \XX $ is $ (n+1) $-coherent.

		\item\label{prop:coherenceforn-localic.3} The \topos $ \XX $ is coherent.

		\item\label{prop:coherenceforn-localic.4} The $ n $-topos $ \XX_{\leq n-1} $ is coherent.
	\end{enumerate}
\end{prp}

\begin{proof}
	Clearly \enumref{prop:coherenceforn-localic}{3}$ \Rightarrow $\enumref{prop:coherenceforn-localic}{4} and \enumref{prop:coherenceforn-localic}{4}$ \Rightarrow $\enumref{prop:coherenceforn-localic}{1}.

	First we show that \enumref{prop:coherenceforn-localic}{1}$ \Rightarrow $\enumref{prop:coherenceforn-localic}{2}.
	\Cref{cor:n-coherenceonlydependsonn-topos} shows that $ \XX $ is $ n $-coherent.
	First notice that every object of $ \XX $ admits a cover by $ (n-1) $-truncated $ n $-coherent objects (so, in particular, $ \XX $ is locally $ n $-coherent).
	This follows from the following observations:
	\begin{itemize}
		\item Since \topos $ \XX $ is $ n $-localic, every object of $ \XX $ admits a cover by $ (n-1) $-truncated objects.

		\item Since the $ n $-topos $ \XX_{\leq n-1} $ is locally $ n $-coherent, \Cref{prp:n-coherenceonlydependsonn-topos} shows that every $ (n-1) $-truncated object of $ \XX $ admits a cover by $ (n-1) $-truncated $ n $-coherent objects.
	\end{itemize} 
	Moreover, since the $ (n-1) $-truncated objects of \atopos are closed under limits and $ \XX_{\leq n-1} $ is $ (n+1) $-coherent, \Cref{prp:n-coherenceonlydependsonn-topos} shows that the $ (n-1) $-truncated $ n $-coherent objects of $ \XX $ are closed under finite products.
	\Cref{lem:closureofn-cohunderprodu} applied to the full subcategory $ \XX_0 \subset \XX $ spanned by the $ (n-1) $-truncated $ n $-coherent objects (so that $ m = n $ in the notation of \Cref{lem:closureofn-cohunderprodu}) now shows that the $ n $-coherent objects of $ \XX $ are closed under finite products. 

	Since an $ n $-localic \topos is $ N $-localic for all $ N \geq n $, to prove the implication \enumref{prop:coherenceforn-localic}{2}$ \Rightarrow $\enumref{prop:coherenceforn-localic}{3}, it suffices to prove that if $ \XX $ is $ (n+1) $-coherent, then $ \XX $ is $ (n+2) $-coherent.
	First we show that $ \XX $ is locally $ (n+1) $-coherent.
	We have already seen that every object of $ \XX $ admits a cover by a $ (n-1) $-truncated $ n $-coherent objects, and that the subcategory $ \XX_0 $ of $ (n-1) $-truncated $ n $-coherent objects is closed under finite products.
	Since $ \XX $ is $ (n+1) $-coherent, \SAG{Corollary}{A.2.4.3} shows that $ (n-1) $-truncated $ n $-coherent objects of $ \XX $ are automatically $ (n+1) $-coherent, immediately implying that $ \XX $ is locally $ (n+1) $-coherent.
	\Cref{lem:closureofn-cohunderprodu} applied to the subcategory $ \XX_0 $ of $ (n-1) $-truncated $ (n+1) $-coherent objects (so that $ m = n+1 $ in the notation of \Cref{lem:closureofn-cohunderprodu}) shows that the $ (n+1) $-coherent objects of $ \XX $ are closed under finite products.
\end{proof}


\subsection[Coherent geometric morphisms via sites \& coherent ordinary topoi]{Coherent geometric morphisms via sites \& coherent ordinary \\ topoi}\label{subsec:cohfor1localic}

In this section we explain the relationship between coherent ordinary topoi  and their corresponding $ 1 $-localic \topoi.%
\footnote{The contents of this section originally appeared in a (partially expository) preprint of the third-named author \cite{Haine:1-localic}.}
(See \cites{Lurie:CatLogic11}[Appendix C, \S\S 5--6]{Ultracategories} for an excellent accounts of coherent ordinary topoi.)
We show that the \category of coherent $ 1 $-localic \topoi is equivalent to the $ 2 $-category of coherent ordinary topoi.
In fact, the results of \cref{subsec:n-coherenceonlydependsonn-topos} allow us to show that the \category of coherent $ n $-localic \topoi is equivalent to the $ (n+1) $-category of coherent $ n $-topoi (\Cref{prop:coherent1localic}).

\begin{rec}\label{nul:coherent1topoi}
	A $ 1 $-topos $ \XX $ is \defn{coherent}\index[terminology]{coherent!1topos@$1$-topos}\index[terminology]{1topos@$1$-topos!coherent} in the sense of \cite[Exposé VI, Definition 2.3]{MR50:7131} if and only if $ \XX $ is $ 2 $-coherent in the sense of \Cref{rec:coherence}.
 	This is true if and only if $ \XX $ is equivalent to the $ 1 $-topos of sheaves of sets on a finitary $ 1 $-site $ (X,\tau) $ with a terminal object.
	\Cref{prop:coherenceforn-localic} shows that $ \XX $ is coherent if and only if its corresponding $ 1 $-localic \topos is coherent.

	A geometric morphism morphism of coherent $ 1 $-topoi $ \flowerstar \colon \fromto{\XX}{\YY} $ is \defn{coherent} \cite[Exposé VI, Definition 3.1]{MR50:7131} if and only if $ \flowerstar $ is induced by a morphism of finitary $ 1 $-sites $ \fupperstar \colon \fromto{(Y,\tau_Y)}{(X,\tau_X)} $.
\end{rec}

The content of the equivalence between coherent $ n $-topoi and coherent $ n $-localic \topoi reduces to showing that a coherent morphism of coherent $ n $-topoi induces a coherent morphism of corresponding $ n $-localic \topoi.
This follows from the fact that coherence of a geometric morphism between locally coherent \topoi can be checked on a generating set of coherent objects (\Cref{cor:checkcohongen}).
A particularly useful consequence is that morphisms of finitary \sites induce coherent geometric morphisms (\Cref{cor:morsitescoherent}).

First we need a few preliminary results.
For this, please recall the notion of \textit{relative $ n $-coherence} (\Cref{def:relativecoh}) introduced in \cref{subsec:coherenceforn-localic}.
	
\begin{lem}\label{lem:quotientofqcisqc}
	Let $ \XX $ be \atopos.
	If $ e \colon \surjto{U}{V} $ is an effective epimorphism in $ \XX $ and $ U $ is quasicompact, then $ V $ is quasicompact.
\end{lem}

\begin{proof}
	This is a special case of \SAG{Proposition}{A.2.1.3}, or, alternatively, \Cref{prop:SAG.A.2.4.1}=\allowbreak\SAG{Proposition}{A.2.4.1}.
\end{proof}

\begin{lem}\label{lem:relativecoh}
	Let $ n \geq 1 $ be an integer and $ \XX $ a locally $ (n-1) $-coherent \topos.
	Let $ U \in \XX $ and let $ e \colon \surjto{\coprod_{i \in I} U_i}{U} $ be a cover of $ U $ where $ I $ is finite and $ U_i $ is $ n $-coherent for each $ i \in I $.
	The following are equivalent: 
	\begin{enumerate}[{\upshape (\ref*{lem:relativecoh}.1)}]
		\item The effective epimorphism $ e $ is relatively $ (n-1) $-coherent.

		\item For all $ i,j \in I $, the object $ U_i \cross_{U} U_j $ is $ (n-1) $-coherent.

		\item The object $ U $ is $ n $-coherent.
	\end{enumerate}
\end{lem}

\begin{proof}
	If $ e $ is relatively $ (n-1) $-coherent, then since coproducts in $ \XX $ are universal, the fiber product
	\begin{equation*}
		\Big( {\textstyle \coprod_{i \in I} U_i} \Big) \, \crosslimits_U \, \Big( {\textstyle \coprod_{j \in I} U_j} \Big) \equivalent \coprod_{i,j \in I} U_i \cross_{U} U_j
	\end{equation*}
	is $ (n-1) $-coherent.
	Thus $ U_i \cross_{U} U_{j} $ is $ (n-1) $-coherent for all $ i,j \in I $ \SAG{Remark}{A.2.0.16}.

	If each $ U_i \cross_{U} U_j $ is $ (n-1) $-coherent, then since each $ U_i $ is $ n $-coherent we see that the pullback of $ e $ along itself
	\begin{equation*}
		\surjto{\coprod_{i,j\in I} U_i \cross_{U} U_j}{\coprod_{i \in I} U_i}
	\end{equation*}
	is relatively $ (n-1) $-coherent (\Cref{exm:SAG.A.2.1.2}=\allowbreak\SAG{Example}{A.2.1.2}).
	Applying \SAG{Corollary}{A.2.1.5} we deduce that $ e \colon \surjto{\coprod_{i \in I} U_i}{U} $ is relatively $ (n-1) $-coherent.

	To conclude, note that if $ e \colon \surjto{\coprod_{i \in I} U_i}{U} $ is relatively $ (n-1) $-coherent, then \SAG{Proposition}{A.2.1.3} shows that $ U $ is $ n $-coherent.
	On the other hand, if $ U $ is $ n $-coherent, then $ e $ is $ (n-1) $-coherent by \Cref{exm:SAG.A.2.1.2}=\allowbreak\SAG{Example}{A.2.1.2}.
\end{proof}

\begin{prp}\label{prop:gencoherenceofgeom}
	Let $ \flowerstar \colon \fromto{\XX}{\YY} $ be a geometric morphism of \topoi and $ n \in \NNup $.
	Assume that:
	\begin{enumerate}[{\upshape (\ref*{prop:gencoherenceofgeom}.1)}]
		\item\label{prop:gencoherenceofgeom.1} There exists a collection of $ n $-coherent objects $ \YY_0 \subset \Obj(\YY) $ of $ \YY $ such that for every $ n $-coherent object $ U \in \YY $ there exists a cover $ \surjto{\coprod_{i \in I} U_i}{U} $ where $ U_i \in \YY_0 $ for each $ i \in I $.

		\item\label{prop:gencoherenceofgeom.2} The pullback functor $ \fupperstar \colon \fromto{\YY}{\XX} $ takes objects of $ \YY_0 $ to $ n $-coherent objects of $ \XX $.

		\item\label{prop:gencoherenceofgeom.3} If $ n \geq 1 $, the \topoi $ \XX $ and $ \YY $ are locally $ (n-1) $-coherent and $ \fupperstar \colon \fromto{\YY}{\XX} $ takes $ (n-1) $-coherent objects of $ \YY $ to $ (n-1) $-coherent objects of $ \XX $.
	\end{enumerate}
	Then $ \fupperstar $ takes $ n $-coherent objects of $ \YY $ to $ n $-coherent objects of $ \XX $.
\end{prp}

\begin{proof}
	Let $ U \in \YY $ be an $ n $-coherent object; we show that $ \fupperstar(U) $ is $ n $-coherent.
	Since $ U $ is $ 0 $-coherent, by \enumref{prop:gencoherenceofgeom}{1} there exists a cover
	\begin{equation*}
		e \colon \surjto{\coprod_{i \in I} U_i}{U}
	\end{equation*}
	where $ U_i \in \YY_0 $ for each $ i \in I $ and $ I $ is finite.
	By \enumref{prop:gencoherenceofgeom}{2}, for all $ i \in I $ the object $ \fupperstar(U_i) $ is $ n $-coherent, so since $ n $-coherent objects are closed under finite coproducts \SAG{Remark}{A.2.0.16}, the object
	\begin{equation*}
		\fupperstar\paren{\textstyle \coprod_{i\in I} U_i} \equivalent \coprod_{i \in I} \fupperstar(U_i)
	\end{equation*}
	is $ n $-coherent.

	Note that 
	\begin{equation*}
		\fupperstar(e) \colon \surjto{\coprod_{i \in I} \fupperstar(U_i)}{\fupperstar(U)}
	\end{equation*}
	is an effective epimorphism in $ \XX $.
	If $ n = 0 $, this proves the claim (\Cref{lem:quotientofqcisqc}).
	If $ n \geq 1 $, then \Cref{lem:relativecoh} shows that it suffices to show that for all $ i ,j \in I $, the object 
	\begin{equation*}
		\fupperstar(U_i) \cross_{\fupperstar(U)} \fupperstar(U_j) \equivalent \fupperstar(U_i \cross_{U} U_j)
	\end{equation*}
	is $ (n-1) $-coherent.
	This follows from the fact that $ U_i \cross_{U} U_j $ is $ (n-1) $-coherent (by \Cref{lem:relativecoh}) and the assumption that $ \fupperstar $ sends $ (n-1) $-coherent objects of $ \YY $ to $ (n-1) $-coherent objects of $ \XX $.
\end{proof}

\Cref{prop:gencoherenceofgeom} shows that coherence of a geometric morphism between \textit{locally} coherent \topoi is equivalent to the \textit{a priori} stronger condition that the pullback functor preserve $ n $-coherent objects for all $ n \geq 0 $; see also \Cref{cor:strongcoherenceofmorphisms}.

\begin{cor}
	Let $ \flowerstar \colon \fromto{\XX}{\YY} $ be a geometric morphism between locally coherent \topoi.
	Then $ \flowerstar $ is coherent if and only if $ \fupperstar $ takes $ n $-coherent objects of $ \YY $ to $ n $-coherent objects of $ \XX $ for all $ n \geq 0 $.
\end{cor}

\Cref{prop:gencoherenceofgeom} also shows that coherence of a geometric morphism can be checked on a generating set of coherent objects.

\begin{cor}\label{cor:checkcohongen}
	Let $ \flowerstar \colon \fromto{\XX}{\YY} $ be a geometric morphism between locally coherent \topoi.
	Let $ \YY_0 \subset \Obj(\YYcoh) $ be a collection of coherent objects such that for every object $ U \in \YY $ there exists a cover $ \surjto{\coprod_{i \in I} U_i}{U} $ where $ U_i \in \YY_0 $ for each $ i \in I $.
	If for all $ U \in \YY_0 $ the object $ \fupperstar(U) $ is coherent, the geometric morphism $ \flowerstar \colon \fromto{\XX}{\YY} $ is coherent.
\end{cor}

For the next result, we need the following lemma.

\begin{lem}\label{lem:sitecoherence}
	Let
	\begin{equation*}
		\fupperstar \colon \fromto{(Y,\tau_Y)}{(X,\tau_X)}
	\end{equation*}
	be a morphism of \sites, and write $ \yo_{\tau_Y} \colon \fromto{Y}{\Sh_{\tau_Y}(Y)} $ for the sheafified Yoneda embedding.
	If the topology $ \tau_X $ is finitary, then
	\begin{equation*}
		\fupperstar \yo_{\tau_Y} \colon \fromto{Y}{\Sh_{\tau_X}(X)}
	\end{equation*}
	factors through $ \Sh_{\tau_X}(X)^{\coh} \subset \Sh_{\tau_X}(X) $.
\end{lem}

\begin{proof}
	We have a commutative square 
	\begin{equation*}
		\begin{tikzcd}
			Y \arrow[r, "\fupperstar"] \arrow[d, "\yo_{\!\tau_Y}" left] & X \arrow[d, "\yo_{\!\!\tau_X}" right] \\
			\Sh_{\tau_Y}(Y) \arrow[r, "\fupperstar" below] & \Sh_{\tau_X}(X)
		\end{tikzcd}
	\end{equation*}
	where the vertical functors are sheafified Yoneda embeddings.
	Since the topology $ \tau_X $ is finitary, the sheafified Yoneda embedding $ \yo_{\tau_X} \colon \fromto{X}{\Sh_{\tau_X}(X)} $ factors through $ \Sh_{\tau_X}(X)^{\coh} $ (\Cref{prop:SAG.A.3.1.3}=\allowbreak\SAG{Proposition}{A.3.1.3}). 
\end{proof}

\begin{cor}\label{cor:morsitescoherent}
	Let $ \fupperstar \colon \fromto{(Y,\tau_Y)}{(X,\tau_X)} $ be a morphism of finitary \sites.
	Then the induced geometric morphism 
	\begin{equation*}
		\flowerstar \colon \fromto{\Sh_{\tau_X}(X)}{\Sh_{\tau_Y}(Y)}
	\end{equation*}
	is coherent.
\end{cor}

\begin{proof}
	By \Cref{prop:SAG.A.3.1.3}, both $ \Sh_{\tau_X}(X) $ and $ \Sh_{\tau_Y}(Y) $ are locally coherent.
	The image $ \yo_{\tau_Y}(Y) $ of $ Y $ under the sheafified Yoneda embedding generates $ \Sh_{\tau_Y}(Y) $ under colimits, so by \Cref{cor:checkcohongen} it suffices to check that $ \fupperstar $ carries objects in $ \yo_{\tau_Y}(Y) $ to coherent objects of $ \XX $; this the content of \Cref{lem:sitecoherence}.
\end{proof}

\begin{nul}
	\Cref{prop:coherenceforn-localic,cor:checkcohongen,cor:morsitescoherent} together show that a geometric morphism of coherent $ 1 $-topoi is coherent in the sense of \cite[Exposé VI, Definition 3.1]{MR50:7131} if and only if the geometric morphism of corresponding of $ 1 $-localic \topoi is coherent if and only if the geometric morphism of coherent $ 1 $-topoi is coherent in the sense of \Cref{def:coherentmors}.
\end{nul}

We now turn to the equivalence between coherent $ n $-topoi and coherent $ n $-localic \topoi.

\begin{ntn}
	Let $ n \in \NNup $.
	Write
	\begin{equation*}
		\Top_{\infty}^{n,\coh} \subset \Topcoh
	\end{equation*}
	for the full subcategory spanned by the $ n $-localic coherent \topoi.
	Write
	\begin{equation*}
		\Top_{n}^{\coh} \subset \Top_{n}
	\end{equation*}
	for the subcategory of the $ (n+1) $-category of $ n $-topoi with objects coherent $ n $-topoi and morphisms coherent geometric morphisms.
	When $ n = 1 $, the $ 2 $-category $ \Top_{1}^{\coh} $ is the $ 2 $-category of ordinary coherent topoi and coherent geometric morphisms (both in the sense of \cite[Exposé VI]{MR50:7131}).
\end{ntn}

\noindent \Cref{prop:coherenceforn-localic,cor:checkcohongen} immediately imply the following:

\begin{prp}\label{prop:coherent1localic}
	Let $ n \in \NNup $.
	The equivalence of \categories $ (-)_{\leq n-1} \colon \equivto{\Top_{\infty}^{n}}{\Top_{n}} $ restricts to an equivalence
	\begin{equation*}
		(-)_{\leq n-1} \colon \equivto{\Top_{\infty}^{n,\coh}}{\Top_{n}^{\coh}}
	\end{equation*}
\end{prp}

\begin{cor}\label{cor:coherent1localicmors}
	Let $ n \in \NNup $.
	The following are equivalent for a geometric morphism $ \flowerstar \colon \fromto{\XX}{\YY} $ between $ n $-localic coherent \topoi:
	\begin{enumerate}[{\upshape (\ref*{cor:coherent1localicmors}.1)}]
		\item The geometric morphism $ \flowerstar \colon \fromto{\XX}{\YY} $ is coherent.

		\item The pullback functor $ \fupperstar \colon \fromto{\YY}{\XX} $ carries $ (n-1) $-truncated $ n $-coherent objects of $ \YY $ to $ n $-coherent objects of $ \XX $.
	\end{enumerate}
\end{cor}


\subsection{Examples of coherent \texorpdfstring{$\infty$}{∞}-topoi from algebraic geometry}\label{subsec:examplesofcoherent}

In this section we use \Cref{cor:morsitescoherent} to provide a few examples of coherent \topoi arising from algebraic geometry.

\begin{exm}\label{ex:Spectralspacecoherent} 
	For a spectral topological space $ S $, write $ \Open^{\qc}(S) \subset \Open(S) $\index[notation]{Openqc@$\Open^{\qc}$} for the locale  of quasicompact opens in $ S $.
	Since the quasicompact opens of $ S $ form a basis for the topology on $ S $ that is closed under finite intersections, the \topos $ \Sh(\Open^{\qc}(S)) $ is $ 0 $-localic \Cref{nul:nlocalicintermsofsites}.
	Applying \cite[Proposition B.6.4]{Ultracategories} we see that the inclusion
	\begin{equation*}
		\Open^{\qc}(S) \subset \Open(S)
	\end{equation*}
	induces an equivalence of $ 0 $-localic \topoi
	\begin{equation*}
		\widetilde{S} \equivalent \Sh(\Open^{\qc}(S))
	\end{equation*}
	(see also \Cref{cor:n-localicbasis}).
	Since Grothendieck topology on \smash{$ \Open^{\qc}(S) $} is finitary, the \topos $ \widetilde{S} $ of sheaves on $ S $ is a coherent \topos.
	(Cf. \SAG{Lemma}{2.3.4.1}).

	If $ f \colon \fromto{S}{T} $ is a quasicompact continuous map of spectral topological spaces, the inverse image map $ f^{-1} \colon \fromto{\Open(T)}{\Open(S)} $ restricts to a map
	\begin{equation*}
		f^{-1} \colon \fromto{\Open^{\qc}(T)}{\Open^{\qc}(S)} \period
	\end{equation*}
	\Cref{cor:morsitescoherent} sho\-ws that the induced geometric morphism $ \flowerstar \colon \fromto{\widetilde{S}}{\widetilde{T}} $ is coherent.
	Since spectral topological spaces are sober, a continuous map $ f \colon \fromto{S}{T} $ of spectral topological spaces induces a coherent geometric morphism on the level of \topoi if and only if $ f $ is quasicompact.
\end{exm}

\begin{nul}
	Note that if $ \XX $ is a coherent \topos, then the underlying topological space of $ \XX $ is spectral (\Cref{cor:n-coherenceonlydependsonn-topos}).
\end{nul}

Combining the fact that the Zariski, Nisnevich\footnote{For background on the Nisnevich topology, see \cites[\SAGsec{3.7}]{SAG}{Hoyois:Nisnevichnotes}{Hoyois:Nisnevichagree}{MR1045853}.}, étale, and proétale\footnote{For background on the proétale topology, see \cites[Tags \href{http://stacks.math.columbia.edu/tag/0988}{0988} \& \href{http://stacks.math.columbia.edu/tag/099R}{099R}]{stacksproject}{MR3379634}.} topoi of a scheme all have the same underlying topological space with the fact that if a scheme $ X $ is quasicompact and quasiseparated, then the $ 1 $-topoi of sheaves on $ X $ in each of these topologies is coherent \cites[\SAGthm{Proposition}{2.3.4.2} \& \SAGthm{Remark}{3.7.4.2}]{SAG}[Appendix A]{MotivicNorms:BachmannHoyois}[Example 7.1.7]{Ultracategories}, we deduce the following: 

\begin{prp}\label{prop:coherentschemetopos}
	The following are equivalent for a scheme $ X $:
	\begin{enumerate}[{\upshape (\ref*{prop:coherentschemetopos}.1)}]
		\item The scheme $ X $ is coherent (i.e., quasicompact and quasiseparated).

		\item The Zariski \topos\index[terminology]{Zariski topos@Zariski \topos}\index[terminology]{topos@\topos!Zariski} $ X_{\zar} $\index[notation]{Xzar@$X_{\zar}$} of $ X $ is a coherent \topos.

		\item The Nisnevich \topos\index[terminology]{Nisnevich topos@Nisnevich \topos}\index[terminology]{topos@\topos!Nisnevich} $ X_{\Nis} $\index[notation]{XNis@$X_{\Nis}$}of $ X $ is a coherent \topos.

		\item The étale \topos\index[terminology]{etale topos@étale \topos}\index[terminology]{topos@\topos!etale@étale} $ X_{\et} $\index[notation]{Xet@$X_{\et}$} of $ X $ is a coherent \topos.

		\item The proétale \topos\index[terminology]{proetale topos@proétale \topos}\index[terminology]{topos@\topos!proétale} $ X_{\proet} $\index[notation]{Xproet@$X_{\proet}$} of $ X $ is a coherent \topos.
	\end{enumerate}
\end{prp}

\begin{nul}
	In the case of the étale topology, see also \SAG{Proposition}{2.3.4.2}.
\end{nul}

\begin{exm}
	Let $ f \colon \fromto{X}{Y} $ be a morphism of coherent schemes and let
	\begin{equation*}
		\tau \in \{\zar,\Nis,\et,\proet\} \period
	\end{equation*}
	Then the induced geometric morphism $ \flowerstar \colon \fromto{X_{\tau}}{Y_{\tau}} $ on \topoi of $ \tau $-sheaves is a coherent geometric morphism of coherent \topoi.
	(Cf. \SAG{Proposition}{2.3.5.1}) 
\end{exm}

\begin{exm}
	Let $ X $ be a coherent scheme.
	Then the natural geometric morphisms
	\begin{equation*}
		\fromto{X_{\proet}}{X_{\et}} \comma \qquad \fromto{X_{\et}}{X_{\Nis}} \comma \andeq \fromto{X_{\Nis}}{X_{\zar}}
	\end{equation*}
	are all coherent geometric morphisms of coherent \topoi.
\end{exm}


\subsection{Classification of bounded coherent \texorpdfstring{$\infty$}{∞}-topoi via \texorpdfstring{$\infty$}{∞}-pretopoi}\label{subsec:classificationofbctopoi}

In this section we explain how \atopos that is both bounded and coherent is \emph{determined} by its truncated coherent objects.

\begin{ntn} 
	Write $\Topbc \subset \Topcoh$\index[notation]{Topbc@$ \Topbc $} for the full subcategory spanned by those coherent \topoi that are also bounded, that is, the \defn{bounded coherent}\index[terminology]{bounded coherent!topos@\topos}\index[terminology]{topos@\topos!bounded coherent} \topoi.
\end{ntn}

To a large extent, bounded coherent \topoi function in much the same way as coherent $1$-topoi.
In particular, any bounded coherent \topos is, in a canonical fashion, the \category of sheaves on \asite with excellent formal properties.

\begin{dfn}\label{dfn:pretopos}
	\Acategory $ X $ called an \defn{\pretopos}\index[terminology]{pretopos@\pretopos} if and only if the following conditions are satisfied.
	\begin{itemize}
		\item The \category $ X $ admits finite limits.

		\item The \category $ X $ admits finite coproducts, which are universal and disjoint.

		\item Groupoid objects in $ X $ are effective, and their geometric realizations are universal.
	\end{itemize}

If $ X $ and $ Y $ are \pretopoi, then a functor $ \fupperstar \colon \fromto{Y}{X} $ is a \emph{morphism of $\infty$-pretopoi}\index[terminology]{morphism!of \pretopoi}\index[terminology]{pretopos@\pretopos!morphism} if $ \fupperstar $ preserves finite limits, finite coproducts, and effective epimorphisms.
	We write
	\begin{equation*}
		\preTop \subset \Cat_{\infty,\updelta_1}
	\end{equation*}
	for the subcategory consisting of \pretopoi and morphisms of \pretopoi.
\end{dfn}

\begin{exm}
	If $\XX$ is a coherent \topos, then the full subcategory $\XXcoh\subseteq\XX$ spanned by the coherent objects is \apretopos \SAG{Corollary}{A.6.1.7}.
\end{exm}

\noindent The following two useful facts are immediate from the definitions.

\begin{lem}\label{lem:prodofpretopoi}
	Let $ \{X_i\}_{i \in I} $ be a collection of \pretopoi.
	Then the product $ \prod_{i \in I} X_i $ in $ \Cat_{\infty,\updelta_1} $ is \apretopos and for each $ j \in I $ the projection
	\begin{equation*}
		\pr_{j} \colon \fromto{\prod_{i\in I} X_i}{X_j}
	\end{equation*}
	is a morphism of \pretopoi.
\end{lem}

\begin{lem}\label{lem:pullbackofpretopoi}
	Given morphisms of \pretopoi $ \fromto{X}{Z} $ and $ \fromto{Y}{Z} $, the pullback $ X \cross_Z Y $ in $ \Cat_{\infty,\updelta_1} $ is \apretopos, and the projections
	\begin{equation*}
		\pr_{1} \colon \fromto{X \cross_Z Y}{X} \quad \text{and} \quad \pr_{2} \colon \fromto{X \cross_Z Y}{X}
	\end{equation*}
	are morphisms of \pretopoi.
\end{lem}

\begin{ntn}\label{ntn:effepi}
	Let $ X $ be \apretopos, and write $E\subseteq X$ for the collection of effective epimorphisms in $ X$.
	Then $(X,E)$ is \apresite, and we write $ \eff \coloneq\tau_E$ for the resulting finitary topology on $ X $.
	We call this topology the \defn{effective epimorphism topology} on $ X $ \cite[\SAGsubsec{A.6.2}]{SAG}.
\end{ntn}

\begin{nul}
	The effective epimorphism topology on \apretopos is a subcanonical topology \SAG{Corollary}{A.6.2.6}.
\end{nul}

\begin{dfn}\label{dfn:boundedpretopos}
	\Apretopos $ X $ is \defn{bounded}\index[terminology]{pretopos@\pretopos!bounded}\index[terminology]{bounded!pretopos@\pretopos} if and only if $ X $ is $ \updelta_0 $-small and every object of $ X $ is truncated.
	We write
	\begin{equation*}
		\preTopbdd \subset \preTop
	\end{equation*}
	for the full subcategory spanned by the bounded \pretopoi.
\end{dfn}

\begin{thm}[\SAG{Theorem}{A.7.5.3}]\label{thm:SAG.A.7.5.3}
	The constructions
	\begin{equation*}
		\goesto{\XX}{\XXcohbdd} \andeq \goesto{X}{\Sheff{X}}
	\end{equation*}
	are mutually inverse equivalences of \categories
	\begin{equation*}
		\Topbc \simeq \preTop^{\bdd,\op} \period
	\end{equation*}
\end{thm}

\noindent In light of \Cref{nul:coherencecheckedaftercompletion} we deduce the following variant for Postnikov complete coherent \topoi:

\begin{cor}\label{cor:postSAG.A.7.5.3}
	Write \smash{$ \Top_{\infty}^{\,\post,\coh} \subset \Topcoh $} for the full subcategory spanned by the Postnikov complete coherent \topoi.
	The constructions $\goesto{\XX}{\XXcohbdd}$ and $\goesto{X}{\Sheff{X}^{\post}}$ are mutually inverse equivalences of \categories
	\begin{equation*}
		\Top_{\infty}^{\,\post,\coh} \simeq \preTop^{\bdd,\op} \period
	\end{equation*}
\end{cor}

We finish this section by recording the following bounded analogue of \Cref{lem:prodofpretopoi} that we use later, as well as a similar result for functor \categories

\begin{lem}\label{lem:prodofbddpretopoi}
	Let $ \{X_i\}_{i \in I} $ be a \emph{finite} collection of bounded \pretopoi.
	Then the \pretopos given by the product $ \prod_{i \in I} X_i $ in $ \Cat_{\infty,\updelta_1} $ is a bounded \pretopos.
\end{lem}

\begin{proof}
	For each $ i \in I $ the \category $ X_i $ is $ \updelta_0 $-small, so the product $ \prod_{i \in I} X_i $ is also $ \updelta_0 $-small.
	For any integer $ n \geq -2 $, an object $ F \in \prod_{i \in I} X_i $ is $ n $-truncated if and only if $ \pr_i(F) \in X_i $ is $ n $-truncated for all $ i \in I $.
	Since $ I $ is \textit{finite} and every object of each of the \categories $ \{X_i\}_{i\in I} $ is truncated by assumption, every object of the product $ \prod_{i \in I} X_i $ is truncated.
\end{proof}

\begin{lem}\label{lem:Funpretopoi}
	Let $ C $ be \acategory and $ X $ \apretopos.
	Then:
	\begin{enumerate}[{\upshape (\ref*{lem:Funpretopoi}.1)}]
		\item\label{lem:Funpretopoi.1} The \category $ \Fun(C,X) $ is \apretopos.

		\item\label{lem:Funpretopoi.2} If $ C $ is $ \updelta_0 $-small and has finitely many objects up to equivalence and $ X $ is bounded, then the \pretopos $ \Fun(C,X) $ is bounded.
	\end{enumerate}
\end{lem}

\begin{proof}
	First, \enumref{lem:Funpretopoi}{1} is clear from the definitions and the fact that (co)limits in functor \categories are computed objectwise.

	In the situation of \enumref{lem:Funpretopoi}{2}, note that since $ C $ and $ X $ are $ \updelta_0 $-small, the \pretopos $ \Fun(C,X) $ is $ \updelta_0 $-small.
	Note that for any integer $ n \geq -2 $, an object $ F \in \Fun(C,X) $ is $ n $-truncated if and only if $ F(c) $ is $ n $-truncated for each $ c \in C $.
	So since every object of $ X $ is truncated and $ C $ has finitely many objects up to equivalence, every object of $ \Fun(C,X) $ is also truncated.
	Hence the \pretopos $ \Fun(C,X) $ is bounded.
\end{proof}


\subsection{Coherence of inverse limits}\label{subsec:cohinverselimtis}

We now recall that bounded coherent \topoi and coherent geometric morphisms are stable under inverse limits in $ \Top_{\infty} $.

\begin{prp}[{\SAG{Proposition}{A.8.3.1}}]\label{prop:SAG.A.8.3.1}
	The \category $ \preTopbdd $ admits filtered colimits and the forgetful functor $ \fromto{\preTopbdd}{\Cat_{\infty,\updelta_1}} $ preserves filtered colimits.
\end{prp}

\begin{prp}[{\SAG{Proposition}{A.8.3.2}}]\label{prop:SAG.A.8.3.2}
	Let $ X \colon \fromto{A}{\preTopbdd} $ be a filtered diagram of bounded \pretopoi.
	Then the natural geometric morphism
	\begin{equation*}
		\fromto{\Sheff{\colim\nolimits_{\alpha \in A} X_{\alpha}} }{\lim\nolimits_{\alpha \in A^{\op}} \Sheff{X_{\alpha}}}
	\end{equation*}
	is an equivalence in $ \Top_{\infty} $.
\end{prp}

\begin{nul}
	See \cite[Lemma 3.3]{ClausenMathew:hyperdescent} for a more general statement about filtered colimits of finitary \sites.
\end{nul}

\noindent The following is immediate from the previous two propositions and \Cref{thm:filteredlimsinRTop}=\allowbreak\HTT{Theorem}{6.3.3.1}.

\begin{cor}[{\SAG{Corollary}{A.8.3.3}}]\label{cor:SAG.A.8.3.3}
	The \category $ \Topbc $ admits inverse limits and the inclusion $ \fromto{\Topbc}{\Top_{\infty}} $ and forgetful functor $ \fromto{\Topbc}{\Cat_{\infty,\updelta_1}} $ both preserve inverse limits.
\end{cor}


\subsection{Coherence \& preservation of filtered colimits}\label{subsec:coherenceandfilteredcolims}

The goal of this section is to prove the appropriate \toposic generalization of the fact that a coherent geometric morphism of $ 1 $-topoi preserves filtered colimits (see \Cref{cor:coherentmorphismscommutewithfilteredcolims}).\footnote{We learned how to simplify and generalize the material in this section from its original form through a preprint of Chough \cite[Theorem 3.4]{Chough:Proper}.} 

We begin be recalling a basic fact about filtered colimits of truncated objects and introducing some convenient terminology.

\begin{rec}\label{rec:incpreservesfilteredcolim}
	Since filtered colimits commute with finite limits in \atopos, for any \topos $ \XX $ and integer $ n \geq -2 $, the inclusion $ \incto{\XX_{\leq n}}{\XX} $ preserves filtered colimits.
	Thus $ \XX_{\leq n} $ is an $ \upomega $-accessible localization of $ \XX $.
\end{rec}

\begin{dfn}\label{def:almostcompactness}
	Let $ C $ be a presentable \category.
	We say that an object $ X \in C $ is \defn{almost compact}\index[terminology]{almost compact} if $ \trun_{\leq n}(C) $ is a compact object of the $ n $-category $ C_{\leq n} $.

	We say that functor $ F \colon \fromto{C}{D} $ between presentable \categories \defn{almost preserves filtered colimits}\index[terminology]{almost preserves filtered colimits} if for each integer $ n \geq -2 $, the functor $ F \colon \fromto{C_{\leq n}}{D} $ preserves filtered colimits.
\end{dfn}

\begin{lem}\label{lem:yoxcompactness}
	Let $ (X,\tau) $ be a finitary \site, write $ \XX \coloneq \Sh_{\tau}(X) $, and write $ \yo_{\tau} \colon \fromto{X}{\XX} $ for the sheafified Yoneda embedding.
	Then for any object $ x \in X $, the sheaf $ \yo_{\tau}(x) $ is almost compact.
\end{lem}

\begin{proof}
	Write $ U \coloneq \yo_{\tau}(x) $ and let $ \plowerstar \colon \fromto{\XX_{/U}}{\XX} $ denote the natural étale geometric morphism.
	Let $ V \colon \fromto{A}{\XX_{\leq n}} $ be a filtered diagram.
	Then we have natural equivalences
	\begin{align*}
		\Map_{\XX}(U,\textstyle\colim_{\alpha \in A} V_{\alpha}) &\equivalent \Map_{\XX}(\plowershriek(1_{\XX_{/U}}),\textstyle\colim_{\alpha \in A} V_{\alpha}) \\ 
		&\equivalent \Map_{\XX_{/U}}(1_{\XX_{/U}},\textstyle\colim_{\alpha \in A} \pupperstar(V_{\alpha})) \period
	\end{align*}
	Since $ U \in \XX $ is coherent (\Cref{prop:SAG.A.3.1.3}=\allowbreak\SAG{Proposition}{A.3.1.3}), the global sections functor
	\begin{equation*}
		\Map_{\XX_{/U}}(1_{\XX_{/U}},-) \colon \fromto{\XX_{/U}}{\Space}
	\end{equation*}
	almost preserves filtered colimits \SAG{Proposition}{A.2.3.1}.
	Hence we have natural equivalences
	\begin{align*}
		\Map_{\XX}(U,\textstyle\colim_{\alpha \in A} V_{\alpha}) &\equivalent \colim_{\alpha \in A} \Map_{\XX_{/U}}(1_{\XX_{/U}},\pupperstar(V_{\alpha})) \\ 
		&\equivalent \colim_{\alpha \in A} \Map_{\XX}(\plowershriek(1_{\XX_{/U}}),V_{\alpha}) \\
		&\equivalent \colim_{\alpha \in A} \Map_{\XX}(U,V_{\alpha}) \period \qedhere
	\end{align*}
\end{proof}

\begin{prp}\label{prp:flowerstarcommuteswithfilteredcolim}
	Let $ \fupperstar \colon \fromto{(Y,\tau_Y)}{(X,\tau_X)} $ be a morphism of finitary \sites.
	Then the induced functor $ \flowerstar \colon \fromto{\Sh_{\tau_X}(X)}{\Sh_{\tau_Y}(Y)} $ almost preserves filtered colimits.
\end{prp}

\begin{proof}
	Write $ \XX \coloneq \Sh_{\tau_X}(X) $, $ \YY \coloneq \Sh_{\tau_Y}(Y) $, and $ \yo_{\tau_X} \colon \fromto{X}{\XX} $ and $ \yo_{\tau_Y} \colon \fromto{Y}{\YY} $ for the sheafified Yoneda embeddings.
	Let $ V \colon \fromto{A}{\XX_{\leq n}} $ be a filtered diagram.
	Since the essential image of $ \yo_{\tau_Y} $ generates $ \YY $ under colimits, to see that the natural morphism 
	\begin{equation*}
		\fromto{\textstyle\colim_{\alpha \in A} \flowerstar(V_{\alpha})}{\flowerstar(\textstyle \colim_{\alpha \in A} V_{\alpha})}
	\end{equation*}
	is an equivalence, it suffices to show that for all $ y \in Y $, the induced morphism 
	\begin{equation*}
		\fromto{\Map_{\YY}(\yo_{\tau_Y}(y),\textstyle\colim_{\alpha \in A} \flowerstar(V_{\alpha}))}{\Map_{\YY}(\yo_{\tau_Y}(y),\flowerstar(\textstyle \colim_{\alpha \in A} V_{\alpha}))}
	\end{equation*}
	is an equivalence.
	Applying \Cref{lem:yoxcompactness} to $ \yo_{\tau_Y}(y) $ and $ \fupperstar \yo_{\tau_Y}(y) \equivalent \yo_{\tau_X}(\fupperstar(y)) $ we see that we have equivalences
	\begin{align*}
		\Map_{\YY}(\yo_{\tau_Y}(y),\textstyle\colim_{\alpha \in A} \flowerstar(V_{\alpha})) &\equivalent \colim_{\alpha \in A} \Map_{\YY}(\yo_{\tau_Y}(y),\flowerstar(V_{\alpha})) \\ 
		&\equivalent \colim_{\alpha \in A} \Map_{\YY}(\fupperstar \yo_{\tau_Y}(y),V_{\alpha}) \\ 
		&\equivalent \Map_{\YY}(\fupperstar \yo_{\tau_Y}(y), \textstyle \colim_{\alpha \in A} V_{\alpha}) \\
		&\equivalent \Map_{\YY}(\yo_{\tau_Y}(y), \flowerstar(\textstyle \colim_{\alpha \in A} V_{\alpha})) \period \qedhere
	\end{align*} 
\end{proof}

In light of \Cref{thm:SAG.A.7.5.3}=\allowbreak\SAG{Theorem}{A.7.5.3}, \Cref{prp:flowerstarcommuteswithfilteredcolim} special\-izes to the following.

\begin{cor}\label{cor:coherentmorphismscommutewithfilteredcolims}
	Let $ \flowerstar \colon \fromto{\XX}{\YY} $ be a coherent geometric morphism between bounded coherent \topoi.
	Then the functor $ \flowerstar $ almost preserves filtered colimits.
\end{cor}


\subsection{Points, Conceptual Completeness, \& Deligne Completeness}\label{subsec:conceptual}

In this section we discuss points of \topoi as well as the \toposic generalizations of the Conceptual Completeness Theorem of Makkai--Reyes and Deligne's Completeness Theorem.

\begin{ntn}
	For \atopos $ \XX $, we write
	\begin{equation*}
		\Pt(\XX) \coloneq \Funlowerstar(\Space,\XX)^{\op}\simeq\Funupperstar(\XX,\Space) \index[notation]{Pt@$ \Pt $}
	\end{equation*}
	for the \category of \defn{points}\index[terminology]{category@\category!of points}\index[terminology]{points!\category of} of $ \XX $.

	Note that a morphism $\fromto{\xlowerstar'}{\xlowerstar}$ of $\Pt(\XX)$ is a natural transformation $\fromto{\xlowerstar}{\xlowerstar'}$.
	(The morphisms are the `geometric transformations' usually preferred in $1$-topos theory.)
	This choice is compatible with the direction of posets: for instance, when $P$ is a noetherian poset, one has $\Pt(\widetilde{P})\simeq P$.
\end{ntn}

\noindent For general \topoi, the passage to its \category of points will lose quite a bit of information.
However, the \toposic version of the Conceptual Completeness Theorem of Makkai--Reyes \cite[Theorem 9.2]{MR0505486} tells us that bounded coherent \topoi are to some extent \textit{determined} by their \categories of points.

\begin{thm}[{Conceptual Completeness\index[terminology]{Conceptual Completeness Theorem}; \SAG{Theorem}{A.9.0.6}}]\label{thm:conceptualcompleteness}
	A geometric morphism $ \flowerstar \colon \fromto{\XX}{\YY} $ between bounded coherent \topoi is an equivalence if and only if $ \flowerstar $ is coherent and the induced functor $ \Pt(\flowerstar) \colon \fromto{\Pt(\XX)}{\Pt(\YY)} $ is an equivalence of \categories.
\end{thm}

\begin{dfn}
	\Atopos $ \XX $ \defn{has enough points}\index[terminology]{enough points} if a morphism $ \phi $ in $ \XX $ is an equivalence if and only if for every point $ \xlowerstar \in \Pt(\XX) $ the stalk $ \xupperstar \phi $ is an equivalence.
\end{dfn}

In classical topos theory, the Deligne Completeness Theorem \cite[Exposé VI, Proposition 9.0]{MR50:7131} states that a locally coherent $ 1 $-topos has enough points.
This is no longer true in the setting of \topoi, the main obstruction being that $ \infty $-connective morphisms in \atopos need not be equivalences.
For this reason the \categorical version of Deligne's theorem takes place in the setting of \topoi where $ \infty $-connective morphisms are equivalences, i.e., \topoi in which Whitehead's Theorem is valid.

\begin{rec}
	A morphism $ \phi \colon \fromto{U}{V} $ in \atopos $ \XX $ is \defn{$ \infty $-connective}\index[terminology]{connective@$\infty$-connective} if $ \phi $ is $ n $-connective for each integer $ n \geq -1 $.

	Let $ \flowerstar \colon \fromto{\XX}{\YY} $ be a geometric morphism of \topoi.
	Since the left adjoint $ \fupperstar $ is left exact and preserves effective epimorphisms, if $ \phi $ is an $ \infty $-connective morphism of $ \YY $, then $ \fupperstar(\phi) $ is an $ \infty $-connective morphism of $ \XX $.
\end{rec}

\begin{dfn}\label{def:hypercompleteness}
	Let $ \XX $ be \atopos.
	An object $ U \in \XX $ is \defn{hypercomplete}\index[terminology]{hypercomplete!object}\index[terminology]{object!hypercomplete} if $ U $ is local with respect to the class of $ \infty $-connective morphisms in $ \XX $.
	We write $ \XXhyp \subset \XX $\index[notation]{Xhyp@$ \XXhyp $} for the full subcategory spanned by the hypercomplete objects of $ \XX $.
	We call $\XXhyp$ the \defn{hypercompletion}\index[terminology]{hypercompletion} of $\XX$, and
	we say that $ \XX $ is \defn{hypercomplete}\index[terminology]{hypercomplete!topos@\topos}\index[terminology]{topos@\topos!hypercomplete} if $ \XXhyp = \XX $.
\end{dfn}

\begin{nul}
	The \category $ \XXhyp \subset \XX $ is a left exact localization of $ \XX $, hence \atopos \cite[\HTTpage{699}]{HTT}.
	Moreover, the \topos $ \XXhyp $ is hypercomplete \HTT{Lemma}{6.5.2.12}.
\end{nul}

\begin{cnstr}[functoriality of hypercompletions]\label{nul:lowerstarpreserveshypercomplete}
	Let $ \flowerstar \colon \fromto{\XX}{\YY} $ be a geometric morphism of \topoi.
	Since $ \fupperstar $ preserves $ \infty $-connective morphisms, the pushforward $ \flowerstar \colon \fromto{\XX}{\YY} $ preserves hypercomplete objects, hence restricts to a functor
	\begin{equation*}
		\flowerstar \colon \fromto{\XXhyp}{\YYhyp} \period
	\end{equation*}
	The functor $ \flowerstar \colon \fromto{\XXhyp}{\YYhyp} $ is the right adjoint in a geometric morphism with left exact left adjoint given by the composite 
	\begin{equation*}
		\begin{tikzcd}[sep=1.5em]
			\YYhyp \arrow[r, "\fupperstar"] & \XX \arrow[r] & \XXhyp
		\end{tikzcd}
	\end{equation*} 
	of $ \fupperstar $ with the left adjoint to the inclusion $ \XXhyp \subset \XX $.
	We denote this geometric morphism by \smash{$ \flowerstar^{\hyp} \colon \fromto{\XXhyp}{\YYhyp} $}.
\end{cnstr}

\begin{wrn}
	Let $ \flowerstar \colon \fromto{\XX}{\YY} $ be a geometric morphism of \topoi.
	The pullback functor $ \fupperstar \colon \fromto{\YY}{\XX} $ generally does \textit{not} preserve hypercomplete objects;
	since every \topos is a left exact localization of a presheaf \topos, if this were true then every \topos would be hypercomplete.
\end{wrn}

The hypercompletion $ \XXhyp $ is characterized by the following universal property.

\begin{prp}[{\HTT{Proposition}{6.5.2.13}}]
	Let $ \XX $ be \atopos.
	Then for every hypercomplete \topos $ \HH $, composition with the inclusion $ \XXhyp \subset \XX $ induces an equivalence
	\begin{equation*}
		\equivto{\Funlowerstar(\HH,\XXhyp)}{\Funlowerstar(\HH,\XX)} \period
	\end{equation*}
\end{prp}

\noindent Consequently, the assignment $ \goesto{\XX}{\XXhyp} $ defines a functor right adjoint to the inclusion of hypercomplete \topoi into all \topoi.

\begin{exm}\label{exm:enoughpointsishypercomplete}
	\Atopos with enough points is hypercomplete.
\end{exm}

\begin{exm}
	Let $ \XX $ be a $ 1 $-topos with corresponding $ 1 $-localic \topos $ \XX' $.
	Then $ \XX $ has enough points (in the sense of \cite[Exposé IV, Définition 6.4.1]{MR50:7130}) if and only if the hypercomplete \topos \smash{$ (\XX')^{\hyp} $} has enough points.
\end{exm}

\begin{exm}
	\Atopos $ \XX $ is hypercomplete if and only if the pullback functor $ \pupperstar \colon \fromto{\XX}{\XX^{\post}} $ is conservative.
	In particular, if Postnikov towers converge in $ \XX $ (\Cref{def:Postnikovstuff}), then $ \XX $ is hypercomplete.
	However, the converse in false: 
\end{exm}

\begin{wrn}
	Postnikov towers need not converge in a hypercomplete \topos.
	Morel and Voevodsky provide the following counterexample \cite[\S2.1, Example 1.30]{MR1813224}.
	Let $ G $ denote the profinite group $ \prod_{i \geq 1} \ZZup/2 $.
	Write $ \XX $ for the \topos of hypersheaves on the site of finite continuous $ G $-sets with respect to the topology in which a family of maps is a covering if and only if it is jointly surjective.
	Since $ \Gammaupperstar_{\XX} \colon \fromto{\Space}{\XX} $ preserves connected objects, the constant sheaf $ U \colonequals \Gammaupperstar_{\XX}(\prod_{i \geq 1} \Kup(\ZZup/2,i)) $ at the product of Eilenberg--MacLane spaces $ \Kup(\ZZup/2,i) $ is connected.
	On the other hand, one can show that the limit of the Postnikov tower of $ U $ is the product
	\begin{equation*}
		\lim_{n \geq 0} \trun_{\leq n} U \equivalent \prod_{i \geq 1} \Gammaupperstar_{\XX}(\Kup(\ZZup/2,i)) \comma
	\end{equation*}
	and, moreover, that this object is not connected.
	In particular, the map $ U \to \lim_{n \geq 0} \trun_{\leq n} U $ is not an equivalence.
\end{wrn}

In light of \Cref{exm:enoughpointsishypercomplete}, the following is the correct \toposic generalization of Deligne's Completeness Theorem.

\begin{thm}[{\Categorical Deligne Completeness\index[terminology]{Deligne Completeness Theorem}; \SAG{Proposition}{A.4.0.5}}]\label{thm:Delignecompleteness}
	\Atopos that is locally coherent and hypercomplete has enough points.
\end{thm}

\begin{nul}\label{nul:checkinftyconnonstalks}
	Please observe that for \atopos $ \XX $, the hypercompletion $ \XXhyp $ has enough points if and only if $ \infty $-connectiveness of morphisms in $ \XX $ can be checked on stalks.
	Indeed, a morphism $ \phi $ in $ \XX $ is $ \infty $-connective if and only if for every point $ \xlowerstar $ of $ \XX $, the stalk $ \xupperstar \phi $ is an equivalence in $ \Space $.
	The Deligne Completeness Theorem (\Cref{thm:Delignecompleteness}=\allowbreak\SAG{Proposition}{A.4.0.5}) and \Cref{lem:hypcoherent} thus show that $ \infty $-connective\-ness in a locally coherent \topos can be checked on stalks.
\end{nul}

We have already seen that the coherence of \atopos only depends on its hypercompletion (\Cref{lem:equivalentconditionsforn-coherence}).
The following proposition gives a more refined assertion about the relationship between the coherent objects of \atopos and its hypercompletion.

\begin{prp}[{\SAG{Proposition}{A.2.2.2}}]\label{prop:SAG.A.2.2.2}
	Let $ \XX $ be \atopos, and write $ \Lup^{\!\hyp} \colon \fromto{\XX}{\XXhyp} $ for the left adjoint to the inclusion $ \incto{\XXhyp}{\XX} $.
	If $ \XX $ is locally $ n $-coherent for all $ n \geq 0 $, then:
	\begin{enumerate}[{\upshape (\ref*{prop:SAG.A.2.2.2}.1)}]
		\item The \topos $ \XXhyp $ is locally $ n $-coherent for all $ n \geq 0$.

		\item An object $ U \in \XXhyp $ is coherent if and only if $ U $ is coherent when viewed as an object of $ \XX $.

		\item An object $ U \in \XX $ is coherent if and only if $ \Lup^{\!\hyp}(U) \in \XXhyp $ is coherent.
	\end{enumerate}
\end{prp}

\begin{cor}\label{lem:hypcoherent}
	Let $ \XX $ be \atopos.
	If $ \XX $ is (locally) coherent, then the hypercompletion $ \XXhyp $ of $ \XX $ is (locally) coherent.
\end{cor}


\begin{exm}
	Let $ \XX $ be a bounded coherent \topos.
	Then since $ \XX $ is also locally coherent (\Cref{exm:boundedcohisloccoh}), the hypercompletion $ \XXhyp $ of $ \XX $ is coherent and locally coherent.
\end{exm}





\subsection{Bases for \texorpdfstring{$\infty$}{∞}-topoi}\label{subsec:bases}

Let $ W $ be a topological space and $ B \subset \Open(W) $ a basis for $ W $.
Upon passing to sheaves of sets, right Kan extension defines an equivalence of $ 1 $-topoi
\begin{equation*}\label{equiv:basis}
	\equivto{\Sh(B;\Set)}{\Sh(W;\Set)}
\end{equation*}
with inverse given by restriction of presheaves \cite[Proposition B.6.4]{Ultracategories}.

The analogous statement for sheaves of \textit{spaces} is false:
open subsets of the Hilbert cube $ \prod_{i \in \ZZup} [0,1] $ homeomorphic to a product $ [0,1) \cross \prod_{i \in \ZZup} [0,1] $ form a basis $ B $ for the topology on the Hilbert cube, but sheaves of spaces on $ B $ do not coincide with sheaves on the Hilbert cube \SAG{Counterexample}{20.4.0.1}.
The goal of this section is to show that although right Kan extension need not define an equivalence 
\begin{equation}\label{equiv:basis}
	\fromto{\Sh(B)}{\Sh(W)} \comma
\end{equation}
the failure of \eqref{equiv:basis} to be an equivalence is fundamentally infinitary in nature and \eqref{equiv:basis} \textit{is} an equivalence when we restrict to hypercomplete objects.

We begin by recalling the basics of bases for sites and \sites.

\begin{dfn}\label{def:basis}
	Let $ (C,\tau) $ be \asite.
	A \defn{basis}\index[terminology]{basis} for the topology $ \tau $ on $ C $ is a full subcategory $ B \subset C $ satisfying the following property: for every object $ c \in C $, there exists a set of morphisms $ \{f_i \colon \fromto{b_i}{c}\}_{i\in I} $ such that $ b_i \in B $ for each $ i \in I $ and the set $ \{f_i\}_{i \in I} $ generates a $ \tau $-covering sieve on $ c \in C $.
\end{dfn}

\begin{nul}
	Let $ (C,\tau) $ be \asite and $ B \subset C $ a basis for $ \tau $.
	Then there is a unique topology $ \restrict{\tau}{B} $ on $ B $ satisfying the following property: for each object $ b \in B $, a sieve $ S \subset B_{/b} $ is a $ \restrict{\tau}{B} $-covering sieve if and only if the image of $ S $ under the embedding \smash{$ \incto{B_{/b}}{C_{/b}} $} generates a $ \tau $ covering sieve of $ b \in C $.

	We always regard a basis $ B \subset C $ as \asite equipped the topology \smash{$ \restrict{\tau}{B} $}.
	To simplify notation, we often write $ \tau $ instead of \smash{$ \restrict{\tau}{B} $}.
\end{nul}

\begin{nul}
	Let $ (C,\tau) $ be \asite and $ B \subset C $ a basis for $ \tau $.
	Then for every object $ c \in C $, the full subcategory 
	\begin{equation*}
		B_{/c} \coloneq B \cross_C C_{/c} \subset C_{/c}
	\end{equation*}
	is a basis for the topology on $ C_{/c} $ induced by $ \tau $.
\end{nul}

\begin{exm}
	Let $ W $ be a topological space.
	A full subposet $ B \subset \Open(W) $ is a basis for the standard topology on $ \Open(W) $ (in the sense of \Cref{def:basis}) if and only if $ B $ defines a basis for the topological space $ W $ in the usual sense: every open set of $ W $ can be written as a union of opens belonging to $ B $.
\end{exm}

\begin{exm}\label{exm:posetbasis}
	Let $ P $ be a poset.
	Then the functor $ \fromto{P^{\op}}{\Open(P)} $ defined by $ \goesto{p}{P_{\geq p}} $ is fully faithful and defines a basis for the topology on the Alexandroff topological space $ P $.
	The induced topology on $ P^{\op} \subset \Open(P) $ is the trivial topology.
\end{exm}

The first property of bases for \sites is that a presheaf on $ B $ is a sheaf if and only if its right Kan extension along the inclusion $ \incto{B}{C} $ is a sheaf on $ C $.

\begin{lem}[{\cites[Proposition A.5]{Aoki:tensor}[Proposition B.6.6]{Ultracategories}}]\label{lem:ransheaf}
	Let $ (C,\tau) $ be \asite and $ i \colon \incto{B}{C} $ a basis for the topology $ \tau $ on $ C $.
	Then: 
	\begin{enumerate}[{\upshape (\ref*{lem:ransheaf}.1)}]
		\item\label{lem:ransheaf.1} A presheaf $ F \colon \fromto{B^{\op}}{\Space} $ on $ B $ is a $ \restrict{\tau}{B} $-sheaf if and only if the right Kan extension of $ F $ along $ i \colon \incto{B^{\op}}{C^{\op}} $ is a $ \tau $-sheaf on $ C $.

		\item\label{lem:ransheaf.2} Right Kan extension along $ i $ defines a fully faithful right adjoint
		\begin{equation*}
			\ilowerstar \colon \incto{\Sh_{\tau}(B)}{\Sh_{\tau}(C)}
		\end{equation*}
		with left exact left adjoint given by the composite
		\begin{equation*}
			\begin{tikzcd}
				\Sh_{\tau}(C) \arrow[r, "{\iupperstar}"] & \PSh(B) \arrow[r, "\Lup_{\tau}"] & \Sh_{\tau}(B)
			\end{tikzcd}
		\end{equation*}
		of presheaf restriction followed by $ \restrict{\tau}{B} $-sheafification.
	\end{enumerate}
\end{lem}

\begin{proof}
	Note that \enumref{lem:ransheaf}{2} is an immediate consequence of \enumref{lem:ransheaf}{1}.
	Write $ \ilowerstar F $ for the right Kan extension of $ F $ along $ i \colon \incto{B^{\op}}{C^{\op}} $.

	To prove \enumref{lem:ransheaf}{1}, we first show that if $ \ilowerstar F $ is a $ \tau $-sheaf on $ C $, then $ F $ is a $ \restrict{\tau}{B} $-sheaf on $ B $.
	Let $ b \in B $ and let $ S_B \subset B_{/b} $ be a covering sieve on $ b $; we need to show that the natural map
	\begin{equation*}
		\rho_b^B \colon \fromto{F(b)}{\lim_{b' \in S_B^{\op}} F(b')}
	\end{equation*} 
	is an equivalence.
	Let $ S_C \subset C_{/b} $ denote the sieve generated by $ S_B \subset C_{/b} $.
	Since $ S_B $ is a covering sieve for the topology $ \restrict{\tau}{B} $ on $ B $, the sieve $ S_C $ is a $ \tau $-covering sieve.
	Now notice that the map $ \rho_b^B $ factors as a composite
	\begin{equation*}
		\begin{tikzcd}
			F(b) = \ilowerstar F(b) \arrow[r, "\rho_b^C"] & \displaystyle\lim_{c' \in S_C^{\op}} \ilowerstar F(c') \arrow[r, "\rho'_b"] & \displaystyle\lim_{b' \in S_B^{\op}} \ilowerstar F(b') = \displaystyle\lim_{b' \in S_B^{\op}} F'(b') \period 
		\end{tikzcd}
	\end{equation*}
	The morphism $ \rho_b^C $ is an equivalence because $ \ilowerstar F $ is a $ \tau $-sheaf on $ C $ and $ S_C $ is a $ \tau $-covering sieve.
	The morphism $ \rho'_b $ is an equivalence because $ \ilowerstar F $ is the right Kan extension of $ F $.

	Now we show that if $ F $ is a $ \restrict{\tau}{B} $-sheaf on $ B $, then $ \ilowerstar F $ is a $ \tau $-sheaf on $ C $.
	Let $ c \in C $ and let $ S_C \subset C_{/c} $ be a $ \tau $-covering sieve of $ c $.
	Define
	\begin{equation*}
		S_B \coloneq S \cross_C C_{/c} \subset B_{/c} \period
	\end{equation*}
	We need to show that the top horizontal map in the commutative square
	\begin{equation*}
		\begin{tikzcd}
			\ilowerstar F(c) \arrow[r] \arrow[d] & \displaystyle\lim_{c' \in S_C^{\op}} \ilowerstar F(c') \arrow[d] \\
			\displaystyle\lim_{b \in B_{/c}^{\op}} \ilowerstar F(b) \arrow[r] & \displaystyle\lim_{b \in S_B^{\op}} \ilowerstar F(b)
		\end{tikzcd}
	\end{equation*}
	is an equivalence.
	The vertical maps are equivalences because $ \ilowerstar F $ is the right Kan extension of its restriction to $ B $; hence it suffices to show that the lower horizontal map is an equivalence.
	To see this, first note that the lower horizonal map can be rewritten as a limit of maps
	\begin{equation}\label{eq:limitoflimitsieves}
		\lim_{b \in B_{/c}^{\op}} \lim_{b' \in B_{/b}^{\op}} \ilowerstar F(b') \to \lim_{[f \colon b \to c] \in S_B^{\op}} \lim_{b' \in \fupperstar S_B^{\op}} \ilowerstar F(b') \period
	\end{equation}
	To conclude, note that since $ \iupperstar\ilowerstar F = F $ is a $ \restrict{\tau}{B} $-sheaf on $ B $ and $ B $ is a basis for $ (C,\tau) $, the morphism \eqref{eq:limitoflimitsieves} is an equivalence.
\end{proof}

For sheaves of sets, \Cref{lem:ransheaf} admits a converse: a presheaf of sets $ F $ on $ C $ is a sheaf if and only if the restriction of $ F $ to $ B $ is a sheaf and $ F $ is the right Kan extension of its restriction.
This \textit{fails} for sheaves of spaces in general.
The goal of the remainder of this chapter is to show that this is true if we restrict too \textit{hyper}sheaves.
The technique is to squeeze \smash{$ \Shhyp_{\tau}(C) $} between \smash{$ \Shhyp_{\tau}(B) $} and \smash{$ \Sh_{\tau}(B) $} and apply the following observation.

\begin{lem}\label{lem:hypersqueeze}
	Let $ \XX $ and $ \YY $ be \topoi, and assume that there are fully faithful geometric morphisms
	\begin{equation*}
		\begin{tikzcd}[sep=1.5em]
			\YYhyp \arrow[r, "\flowerstar", hooked] & \XX \arrow[r, "\glowerstar", hooked] & \YY 
		\end{tikzcd}
	\end{equation*} 
	and the composite $ \glowerstar \flowerstar \colon \fromto{\YYhyp}{\YYhyp} $ is the identity.
	Then $ \flowerstar \colon \incto{\YYhyp}{\XXhyp} $ and $ \glowerstar \colon \incto{\XXhyp}{\YYhyp} $ are mutually inverse equivalences.
\end{lem}

\begin{proof}
	Immediate from the fact that $ \flowerstar $ and $ \glowerstar $ preserve hypercomplete objects \Cref{nul:lowerstarpreserveshypercomplete} and the assumption that $ \glowerstar \flowerstar \equivalent \id_{\YYhyp} $.
\end{proof}

The key technical lemma we need to show that a hypersheaf on $ C $ is the right Kan extension of its restriction to $ B $ is a relative version of \SAG{Lemma}{20.4.5.4}.
It seems to have first been noticed by Aoki \cite[Lemma A.10]{Aoki:tensor}.

\begin{lem}\label{lem:SAG.20.4.5.4var}
	Let $ \flowerstar \colon \fromto{\XX}{\YY} $ be a geometric morphism of \topoi and $ B \subset \YY $ a small full subcategory.
	Assume that for each object $ Y \in \YY $, there exists a morphism $ e \colon \fromto{\coprod_{i \in I} U_i}{Y} $ such that $ U_i \in B $ for each $ i \in I $ and $ \fupperstar(e) $ is an effective epimorphism in $ \XX $.
	Then $ \fupperstar(\colim_{U \in B} U) $ is an $ \infty $-connective object of $ \XX $.
\end{lem}

\begin{proof}
	Write $ X \coloneq \fupperstar(\colim_{U \in B} U) $. 
	We prove that $ X $ is $ n $-connective for each $ n \geq 0 $ by induction on $ n $.
	For the base case, note that since $ \fupperstar $ is a left exact left adjoint, the unique morphism $ e \colon \fromto{\coprod_{U \in B} \fupperstar(U)}{1_{\XX}} $ is an effective epimorphism.
	The effective epimorphism $ e $ factors as a composite
	\begin{equation*}
		\begin{tikzcd}[sep=1.5em]
			\displaystyle\coprod_{U \in B} \fupperstar(U) \arrow[r] & \displaystyle\colim_{U \in B} \fupperstar(U) \arrow[r] & 1_{\XX} \comma
		\end{tikzcd} 
	\end{equation*}
	hence the unique morphism $ \fromto{X}{1_{\XX}} $ is an effective epimorphism (i.e., $ X $ is $ 0 $-connective).
	
	For the inductive step, we assume that $ X $ is $ (n-1) $-connective and prove that $ X $ is $ (n-1) $-connective.
	That is, we need to show that the diagonal $ \upDelta_X \colon \fromto{X}{X \cross X} $ is $ (n-1) $-connective.
	Since $ \fupperstar $ is a left exact left adjoint and colimits in \atopos are universal, we can rewrite $ X \cross X $ as the colimit
	\begin{equation*}
		X \cross X \equivalent \colim_{(U,U') \in B \cross B} \fupperstar(U) \cross \fupperstar(U') \period
	\end{equation*}
	Rewriting $ X $ as an iterated colimit
	\begin{equation*}
		X \equivalent \colim_{(U,U') \in B \cross B} \colim_{V \in B_{/(U \cross U')}} \fupperstar(V) \comma
	\end{equation*}
	we see that we can rewrite the diagonal $ \upDelta_{X} $ as a colimit of maps
	\begin{equation*}
		\delta_{U,U'} \colon \colim_{V \in B_{/(U \cross U')}} \fupperstar(V) \to \fupperstar(U) \cross \fupperstar(U') \period
	\end{equation*}
	Thus it suffices to show that each of the maps $ \delta_{U,U'} $ is $ (n-1) $-connective.
	This follows from the inductive hypothesis applied to the geometric morphism
	\begin{equation*}
		\fromto{\XX_{/(\fupperstar(U) \cross \fupperstar(U'))}}{\YY_{/(U \cross U')}}
	\end{equation*}
	whose left exact left adjoint is given by
	\begin{equation*}
		\fupperstar \colon \fromto{\YY_{/(U \cross U')}}{\XX_{/(\fupperstar(U) \cross \fupperstar(U'))}} \period \qedhere
	\end{equation*}
\end{proof}

We are finally ready to prove the main result of this section. 
The following result has appeared in work of Porta--Yue Yu under the additional assumption that representable presheaves are already hypersheaves \cite[Proposition 2.22]{MR3545934}.
We learned of the present proof from Aoki \cite[Appendix A]{Aoki:tensor}.

\begin{prp}\label{prop:hyperbasis}
	Let $ (C,\tau) $ be \asite and $ i \colon \incto{B}{C} $ a basis for the topology $ \tau $.
	Then:
	\begin{enumerate}[{\upshape (\ref*{prop:hyperbasis}.1)}]
		\item\label{prop:hyperbasis.1} If $ F $ is a $ \tau $-hypersheaf on $ C $, then $ F $ is the right Kan extension of its restriction $ \iupperstar F $ to $ B $.

		\item\label{prop:hyperbasis.2} Right Kan extension defines an equivalence of hypercomplete \topoi
		\begin{equation*}
			\ilowerstar \colon \equivto{\Shhyp_{\tau}(B)}{\Shhyp_{\tau}(C)}
		\end{equation*}
		with inverse given by presheaf restriction $ \iupperstar $.

		\item\label{prop:hyperbasis.3} A presheaf $ F \colon \fromto{C^{\op}}{\Space} $ is a $ \tau $-hypersheaf if and only if $ \iupperstar F $ is a $ \restrict{\tau}{B} $-hypersheaf on $ B $ and $ F $ is the right Kan extension of $ \iupperstar F $.
	\end{enumerate}
\end{prp}

\begin{proof}
	Write $ \Lup_\tau \colon \fromto{\PSh(C)}{\Sh_{\tau}(C)} $ for the $ \tau $-sheafification functor and $ \yo \colon \incto{C}{\PSh(C)} $ for the Yoneda embedding.

	First we prove \enumref{prop:hyperbasis}{1}.
	Let $ F $ be a $ \tau $-hypersheaf on $ C $; we prove that the unit $ \fromto{F}{\ilowerstar \iupperstar F} $ is an equivalence.
	By the formula for right Kan extension, for each object $ c \in C $, the unit $ \fromto{F(c)}{\ilowerstar \iupperstar (F)(c)} $ is given by applying the functor $ \Map_{\PSh(C)}(-,F) $ to the natural morphism
	\begin{equation*}
		e_c \colon \fromto{\colim_{b \in B_{/c}^{\op}} \yo(b)}{\yo(c)}
	\end{equation*}
	in $ \PSh(C) $.
	Since $ F $ is a hypercomplete object of $ \Sh_{\tau}(C) $, to prove the claim it suffices to show that the morphism $ \Lup_{\tau}(e_c) $ is $ \infty $-connective for every object $ c \in C $.
	Since $ B \subset C $ is a basis for $ \tau $, this follows from \Cref{lem:SAG.20.4.5.4var} applied to the geometric morphism with left adjoint
	\begin{equation*}
		\begin{tikzcd}[sep=1.5em]
			\PSh(C_{/c}) \equivalent \PSh(C)_{/\yo(c)} \arrow[r, "\Lup_{\tau}"] & \Sh_{\tau}(C)_{/\Lup_{\tau}\!\yo(c)}
		\end{tikzcd} 
	\end{equation*}
	given by $ \tau $-sheafification.

	Now we prove \enumref{prop:hyperbasis}{2}.
	By \enumref{prop:hyperbasis}{1}, \Cref{lem:ransheaf}, and the fact that the right adjoint in a geometric morphism preserves hypercomplete objects \Cref{nul:lowerstarpreserveshypercomplete}, we have fully faithful functors
	\begin{equation*}
		\begin{tikzcd}[sep=1.5em]
			\Shhyp_{\tau}(B) \arrow[r, "\ilowerstar", hooked] & \Shhyp_{\tau}(C) \arrow[r, "\iupperstar", hooked] & \Sh_{\tau}(B) \comma
		\end{tikzcd} 
	\end{equation*}
	where the functor $ \ilowerstar \colon \incto{\Shhyp_{\tau}(B)}{\Shhyp_{\tau}(C)} $ is the right adjoint in a geometric morphism.
	Thus by \Cref{lem:hypersqueeze} it suffices to show that the restriction functor
	\begin{equation}\label{eq:hyperrestrict}
		\iupperstar \colon \incto{\Shhyp_{\tau}(C)}{\Sh_{\tau}(B)}
	\end{equation}
	admits a left exact left adjoint.
	Write $ (-)^{\hyp} \colon \fromto{\Sh_{\tau}(C)}{\Shhyp_{\tau}(C)} $ for the left adjoint to the inclusion $ \fromto{\Shhyp_{\tau}(C)}{\Sh_{\tau}(C)} $.
	We claim that the composite 
	\begin{equation*}
		\begin{tikzcd}[sep=2em]
			\Shhyp_{\tau}(B) \arrow[r, "\ilowerstar", hooked] & \Sh_{\tau}(C) \arrow[r, "(-)^{\hyp}", hooked] & \Shhyp_{\tau}(C) 
		\end{tikzcd} 
	\end{equation*}
	is left adjoint to the restriction \eqref{eq:hyperrestrict}.
	To see this, let $ G \in \Sh_{\tau}(B) $ and $ F \in \Shhyp_{\tau}(C) $ and note that since $ \ilowerstar $ is fully faithful, by \enumref{prop:hyperbasis}{1} and the hypercompleteness of $ F $ we have natural equivalences
	\begin{align*}
		\Map_{\Sh_{\tau}(B)}(G,\iupperstar F) &\equivalent \Map_{\Sh_{\tau}(C)}(\ilowerstar G, \ilowerstar \iupperstar F) \\ 
		&\equivalent \Map_{\Sh_{\tau}(C)}(\ilowerstar G, F) \\
		&\equivalent \Map_{\Shhyp_{\tau}(C)}((\ilowerstar G)^{\hyp}, F) \period
	\end{align*}

	Finally, \enumref{prop:hyperbasis}{3} is immediate from \enumref{prop:hyperbasis}{2} and \Cref{lem:ransheaf}.
\end{proof}

\begin{cor}\label{cor:hypercompletebasis}
	Let $ (C,\tau) $ be \asite and $ i \colon \incto{B}{C} $ a basis for the topology $ \tau $.
	If $ \Sh_{\tau}(C) $ is hypercomplete, then $ \Sh_{\tau}(B) $ is hypercomplete and the geometric morphism $ \ilowerstar \colon \incto{\Sh_{\tau}(B)}{\Sh_{\tau}(C)} $ is an equivalence.
\end{cor}

\begin{cor}\label{cor:n-localicbasis}
	Let $ (C,\tau) $ be \asite, $ i \colon \incto{B}{C} $ a basis for the topology $ \tau $, and $ n \geq 0 $ be an integer.
	If $ \Sh_{\tau}(C) $ and $ \Sh_{\tau}(B) $ are both $ n $-localic, then the geometric morphism $ \ilowerstar \colon \incto{\Sh_{\tau}(B)}{\Sh_{\tau}(C)} $ is an equivalence.
\end{cor}

\begin{exm}\label{exm:Pbasis}
	Let $ P $ be a poset.
	From \Cref{exm:posetbasis} we see that right Kan extension along the inclusion $ P \subset \Open(P)^{\op} $ defined by $ \goesto{p}{P_{\geq p}} $ defines a fully faithful geometric morphism
	\begin{equation*}
		\incto{\Fun(P,\Space)}{\Ptilde}
	\end{equation*}
	that identifies $ \Fun(P,\Space) $ with the hypercompletion of $ \Ptilde $.
	
	In particular, if $ P $ is a \textit{finite} poset, then $ \Ptilde $ is already hypercomplete \cites[\HTTthm{Remark}{7.2.4.18}]{HTT}[Lemma 3.13]{ClausenMathew:hyperdescent}, so right Kan extension defines an equivalence
	\begin{equation*}
		\equivto{\Fun(P,\Space)}{\Ptilde} \period
	\end{equation*}
	See \cite[Corollary 2.4]{AsaiShah} for a direct proof of this fact.
\end{exm}

\begin{wrn}\label{wrn:Ptildenlocalic}
	If $ P $ is an infinite poset, then $ \Ptilde $ need not be hypercomplete; see \cite[Example A.13]{Aoki:tensor} for a counterexample.
	Equivalently, if $ P $ is an infinite poset, then the \topos $ \Fun(P,\Space) $ need not be $ n $-localic for any $ n \geq 0 $.
\end{wrn}

\begin{rmk}
	See \cite[Lemma C.3]{MR3302973} for another very useful criterion for checking that the inclusion of a basis induces an equivalence after passage to sheaves of spaces.
\end{rmk}

\newpage

\section{Shape theory}\label{sec:shapetheory}

This chapter is dedicated to shape theory for \topoi.
The ideas of shape theory for higher topoi go back to Grothendieck's famous letter to Breen \cite{Grothendieck:Breen}, and were later developed by Toën--Vezzosi \cites{MR1949660}[\S3.2]{Toen:HHCSAG}[\S5.3]{TV:SegalTopoi} and Lurie \cites[\HTTsubsec{7.1.6}]{HTT}[\HAsec{A.1}]{HA}[\SAGsec{E.2}]{SAG}.
We refer the reader to these sources for original accounts of the theory.

\Cref{subsec:protruncated} establishes the basic material we'll need on protruncated objects.
\Cref{subsec:shapes} recalls the definition of the shape and explains a few basic properties of the shape; one of the most important of these is that the protruncated shape of \atopos and its hypercompletion agree.  
\Cref{subsec:shapesoflimits} is devoted to proving that the protruncated shape commutes with inverse limits of bouded coherent \topoi (\Cref{cor:protruncshapeinverselim}).
This result provides a computational tool for shapes and will be used repeatedly throughout the text.
\Cref{subsec:profinshape} explains how to regard profinite spaces as \topoi following \cite[\SAGapp{E}]{SAG}.


\subsection{Protruncated objects}\label{subsec:protruncated}

In this section, we recall some facts about protruncated objects that we'll need throughout the text.
We also record an interesting observation which does not seem to be in the literature (\Cref{lem:protruncff}).

\begin{ntn}\label{ntn:pro-n-trun}
	Let $ C $ be a presentable \category.
	For each integer $ n \geq -2 $, the \defn{pro-$ n $-truncation}\index[terminology]{prontruncation@pro-$ n $-truncation} functor $ \trun_{\leq n} \colon \fromto{\Pro(C)}{\Pro(C_{\leq n})} $ is the unique extension of the $ n $-truncation functor $ \trun_{\leq n} \colon \fromto{C}{C_{\leq n}} $ to \proobjects that preserves inverse limits.
\end{ntn}

\begin{nul}\label{nul:protrun}
	Let $ C $ be a presentable \category.
	Then the extension to \proobjects of the functor $ \fromto{C}{\Pro(C_{<\infty})} $ given by sending an object $ X \in C $ to the inverse system given by its Postnikov tower $ \{\trun_{\leq n}(X)\}_{n\geq-2} $ is left adjoint to the inclusion $ \incto{\Pro(C_{<\infty})}{\Pro(C)} $.
	We call this left adjoint
	\begin{equation*}
		\trun_{<\infty} \colon \fromto{\Pro(C)}{\Pro(C_{<\infty})}
	\end{equation*}
	\defn{protruncation}.%
	\index[terminology]{protruncation}\index[notation]{tauless@$\trun_{<\infty}$}
	A morphism of \proobjects $ f \colon \fromto{X}{Y} $, regarded as left exact accessible functors $ \fromto{C}{\Space} $, becomes an equivalence after protuncation if and only if for every truncated object $ K \in C_{<\infty} $, the induced morphism $ f(K) \colon \fromto{X(K)}{Y(K)} $ is an equivalence.

	Since the truncation functors in \atopos preserve finite products \HTT{Lemma}{6.5.1.2}, if $ C $ is \atopos, then the protruncation functor $ \trun_{<\infty} $ also preserves finite products
\end{nul}

\begin{rmk}\label{rmk:ArtinMazurnatural}
	In the terminology of Artin--Mazur \cite[Definition 4.2]{MR0245577}, morphisms in the \category $ \Pro(\Space) $ of prospaces that induce equivalences after protruncation are precisely those morphisms that become \textit{$ \natural $-isomorphisms}\index[terminology]{zzn@$\natural$-isomorphism} in the category $ \Pro(\hoSpace) $.
\end{rmk}

\begin{rmk}\label{rmk:Isaksenmodelstructure}
	Isaksen's \textit{strict model structure} on pro-simplicial sets \cite{MR2039766} presents the \category $ \Pro(\Space) $ of prospaces \cite[Lemma 3.1]{Hoyois:Kunneth}.
	The model structure that Isaksen defines in \cite{MR1828474} is the left Bousfield localization of the strict model structure at the $ \trun_{<\infty} $-equivalences, hence presents the \category $ \Pro(\Space_{<\infty}) $ of protruncated spaces \cite[Remark 3.2]{Hoyois:Kunneth}.
	The latter model structure is what is often used étale homotopy theory, for example in the recent work of Schmidt--Stix \cite{MR3549624} on the étale homotopy type and anabelian geometry.
\end{rmk}

\begin{nul}\label{nul:materialization}
	Let $ C $ be a presentable \category.
	The unique functor
	\begin{equation*}
		\mat \colon \fromto{\Pro(C)}{C}
	\end{equation*}
	that preserves inverse limits and restricts to the identity $ \fromto{C}{C} $ is right adjoint to the Yoneda embedding $ \yo \colon \incto{C}{\Pro(C)} $ \SAG{Example}{A.8.1.7}.
	We call $ \mat $\index[notation]{mat@$\mat$} the \defn{materialization}\index[terminology]{materialization} functor.
	Hence we have adjunctions
	\begin{equation*}
		\begin{tikzcd}
			C \arrow[r, hooked, shift left, "\yo"] & \Pro(C) \arrow[l, shift left, "\mat"] \arrow[r, "\trun_{<\infty}", shift left] & \Pro(C_{<\infty}) \arrow[l, hooked', shift left] \period
		\end{tikzcd}
	\end{equation*}
\end{nul}

If Postnikov towers converge in $ C $ (\Cref{def:Postnikovstuff}), then the composite left adjoint is also fully faithful:

\begin{lem}\label{lem:protruncff}
	Let $ C $ be presentable \category.
	If Postnikov towers converge in $ C $, the protruncation functor
	\begin{equation*}
		\trun_{<\infty} \colon \fromto{C}{\Pro(C_{<\infty})} 
	\end{equation*}
	is fully faithful.
	Moreover, the essential image of $ \trun_{<\infty} \colon \incto{C}{\Pro(C_{<\infty})} $ is the full subcategory spanned by those protruncated objects $ X $ such that for each integer $ n \geq -2 $, the pro-$ n $-truncation $ \trun_{\leq n}(X) \in \Pro(C_{\leq n}) $ is a constant pröbject.
\end{lem}

\begin{proof}
	It suffices to show that for any object $ X \in C $, the unit morphism $ \fromto{X}{\mat \trun_{<\infty}(X)} $ is an equivalence.
	This follows from the equivalence 
	\begin{equation*}
		\mat \trun_{<\infty}(X) \equivalent \lim_{n\geq-2} \trun_{\leq n}(X)
	\end{equation*}
	and the assumption that Postnikov towers converge in $ C $.
\end{proof}

\begin{nul}
	Composing the fully faithful functor $ \trun_{<\infty} \colon \incto{\Space}{\Pro(\Space_{<\infty})} $ with the inclusion $ \incto{\Pro(\Space_{<\infty})}{\Pro(\Space)} $ gives another embedding of spaces into prospaces, different from the Yoneda embedding $\yo \colon \Space \inclusion \Pro(\Space)$:
	for a space $ K $, the natural morphism of prospaces $ \fromto{\yo(K)}{\trun_{<\infty}(K)} $ is an equivalence if and only if $ K $ is truncated.
	Unlike the Yoneda embedding $\yo$, the functor $ \trun_{<\infty} \colon \incto{\Space}{\Pro(\Space)} $ is neither a left nor a right adjoint.
\end{nul}


\subsection{Shape theory}\label{subsec:shapes}

We now recall the basics of \textit{shape theory} for \topoi.
The shape is crucial to the study of \textit{Stone \topoi} presented in \cref{subsec:profinshape}.
Both shape theory and Stone \topoi are key to our development of the \textit{stratified shape} in \Cref{part:strattopoi} and \textit{stratified étale homotopy type} in \Cref{part:stratetale}.

\begin{dfn}\label{dfn:shape}
	The \defn{shape}\index[terminology]{shape} $ \Shape \colon \fromto{\Top_{\infty}}{\Pro(\Space)} $\index[notation]{Pi@$ \Shape $} is the left adjoint to the extension to \proobjects of the fully faithful functor $ \incto{\Space}{\Top_{\infty}} $ given by
	\begin{equation*}
		\goesto{\Pi}{\Space_{/\Pi} \equivalent \Fun(\Pi,\Space)}
	\end{equation*}
	\cite[\SAGsubsec{E.2.2}]{SAG}.
	The shape admits two other useful descriptions:
	\begin{enumerate}[(1)]
		\item Let $ \XX $ be \atopos, and write $ \upGamma_{\XX,!} \colon \fromto{\XX}{\Pro(\Space)} $ for the proëxistent left adjoint of $ \Gammaupperstar_{\XX} \colon \fromto{\Space}{\XX} $.
		The shape of $ \XX $ is equivalent to the prospace $ \upGamma_{\XX,!}(1_{\XX}) $ \cites[\HAappthm{Remark}{A.1.10}]{HA}[\S 2]{MR3763287}.
	
		\item As a left exact accessible functor $ \fromto{\Space}{\Space} $, the prospace $ \Shape(\XX) $ is the composite $ \Gammalowerstar \Gammaupperstar $ \cites[\HTTsubsec{7.1.6}]{HTT}[\S 2]{MR3763287}.
		Under this identification, the shape assigns to a geometric morphism $ \flowerstar \colon \fromto{\XX}{\YY} $ with unit $ \unit \colon \fromto{\id_{\YY}}{\flowerstar \fupperstar} $ the morphism of prospaces corresponding to
		\begin{equation*}
			\Gammaup_{\YY,\ast} \unit \Gammaupperstar_{\YY} \colon \fromto{\Gammaup_{\YY,\ast} \Gammaupperstar_{\YY}}{\Gammaup_{\YY,\ast} \flowerstar \fupperstar \Gammaupperstar_{\YY} \equivalent \Gammaup_{\XX,\ast} \Gammaupperstar_{\XX}}
		\end{equation*}
		in $ \Pro(\Space)^{\op} \subset \Fun(\Space,\Space) $. 
	\end{enumerate}
\end{dfn}

\begin{nul}\label{nul:classifyingtopos}
	The functor $ \uplambda \colon \fromto{\Pro(\Space)}{\Top_{\infty}} $ given by extending the fully faithful functor $ \incto{\Space}{\Top_{\infty}} $ to \proobjects is \textit{not} itself fully faithful.
\end{nul}

For our first example shape computation, please recall that we write $ \invert \colon \fromto{\Cat_{\infty}}{\Space} $ for the left adjoint to the inclusion (\Cref{ntn:inverteverything}).

\begin{exm}\label{exm:shapeofpresheaf}
	If $ C $ is a small \category, then $ \Gammaupperstar \colon \fromto{\Space}{\Fun(C,\Space)} $ admits a genuine left adjoint $ \Gammalowershriek \colon \fromto{\Fun(C,\Space)}{\Space} $ given by taking the colimit of a diagram $ \fromto{C}{\Space} $.%
	\footnote{That is to say, presheaf \topoi are \textit{locally of constant shape} \cite[\HAappthm{Definition}{A.1.5} \& \HAappthm{Proposition}{A.1.8}]{HA}.}
	The shape of the \topos $ \Fun(C,\Space) $ is thus given by the colimit of the constant diagram at the terminal object of $ \Space $:
	\begin{equation*}
		\Shape(\Fun(C,\Space)) = \Gammalowershriek(1_{\Fun(C,\Space)}) = \textstyle\colim_{C} 1_{\Space} \equivalent \invert(C) \period
	\end{equation*}

	Moreover, the functor $ \invert \colon \fromto{\Cat_{\infty}}{\Space} $ is equivalent to the composite 
	\begin{equation*}
		\begin{tikzcd}[sep=1.5em]
			\Cat_{\infty} \arrow[rrr, "{\Fun(-,\Space)}"] & & & \Top_{\infty} \arrow[r, "\Shape"] & \Space \period
		\end{tikzcd}
	\end{equation*}
\end{exm}

\begin{dfn}\label{dfn:shapeequivalence}
	A geometric morphism $ \flowerstar \colon \fromto{\XX}{\YY} $ of \topoi is a \emph{shape equivalence}\index[terminology]{shape!equivalence}\index[terminology]{equivalence!shape} if the induced morphism
	\begin{equation*}
		\Shape(\flowerstar) \colon \fromto{\Shape(\XX)}{\Shape(\YY)}
	\end{equation*}
	is an equivalence in $ \Pro(\Space) $.
	We say that \atopos $ \XX $ has \defn{trivial shape}\index[terminology]{shape!trivial} if $ \Shape(\XX) $ is a terminal object of $ \Pro(\Space) $.
\end{dfn}

\begin{nul}
	Work of Hoyois \cite[Proposition 2.6]{MR3763287} shows that a geometric morphism $ \flowerstar $ is a shape equivalence if and only if $ \flowerstar $ induces an equivalence of \categories of space-valued torsors. 
\end{nul}

\begin{wrn}
	The pullback (in $ \Top_{\infty} $) of a shape equivalence is not generally a shape equivalence, even when both morphisms are shape equivalences. 
	As an example, consider the space $ X \coloneq [0,1] $, and its closed subspace $ Z \coloneq \{ 0 \} $ and open complement $ U \coloneq (0,1] $.
	Then the \topoi $ \Xtilde $, $ \Utilde $, and $ \Ztilde $ all have trivial shape and the natural inclusions
	\begin{equation*}
		\incto{\Ztilde}{\Xtilde} \andeq \incto{\Utilde}{\Xtilde}
	\end{equation*}
	are both shape equivalences \HAa{Example}{A.4.5}, however the pullback $ \Ztilde \cross_{\Xtilde} \Utilde $ is the initial \topos $ \widetilde{\emptyset} $, which has empty shape.
\end{wrn}

\begin{ntn}
	Let $ n \geq -2 $ be an integer.
	We write\index[notation]{Pileqn@$\Shapen$}
	\begin{equation*}
		\Shapen \coloneq \trun_{\leq n} \of \Shape \colon \fromto{\Top_{\infty}}{\Pro(\Space_{\leq n})} 
	\end{equation*}
	for the \defn{pro-$ n $-truncated shape}\index[terminology]{shape!prontruncated@pro-$n$-truncated}\index[terminology]{prontruncated@pro-$n$-truncated!shape} (\Cref{ntn:pro-n-trun}).
	We write\index[notation]{Pilinfty@$\Shapetrun$}
	\begin{equation*}
		\Shapetrun \coloneq \trun_{<\infty} \of \Shape \colon \fromto{\Top_{\infty}}{\Pro(\Space_{<\infty})} 
	\end{equation*}
	for the \emph{protruncated shape}\index[terminology]{shape!protruncated@protruncated}\index[terminology]{protruncated@protruncated!shape} \Cref{nul:protrun}.
\end{ntn}

\begin{exm}\label{exm:protruncandhypercomp}
	Since truncated objects of \atopos are hypercomplete, for any \topos $ \XX $, the natural geometric morphism $ \incto{\XX^{\hyp}}{\XX} $ induces an equivalence
	\begin{equation*}
		\equivto{\Shapetrun(\XX^{\hyp})}{\Shapetrun(\XX)}
	\end{equation*}
	on protruncated shapes.
\end{exm}


\subsection{Shapes of inverse limits}\label{subsec:shapesoflimits}

This section is dedicated to proving that the protruncated shape preserves limits  of inverse systems of bounded coherent \topoi and coherent geometric morphisms (\Cref{cor:protruncshapeinverselim}).\footnote{A proof of this can be found in work of the third-named author \cite[Proposition 2.2]{Haine:recon}, but we present a better proof here.}
This follows from the more general fact that the protruncated shape preserves limits of inverse systems of \topoi and geometric morphisms in which the pushforward preserve filtered colimits of uniformly truncated objects.
We learned this from Chough \cite[\S 3]{Chough:Proper}; though Chough's paper only states this for the profinite shape, his proof works for the protruncated shape.

We first fix some useful notation for the next few results.
Please also recall \Cref{def:almostcompactness}.

\begin{ntn}\label{ntn:inverselimitargument}
	Let $ \XX \colon \fromto{I}{\Top_{\infty}} $ be an inverse diagram of \topoi.
	For each morphism $ \alpha \colon \fromto{j}{i} $ in $ I $, we write
	\begin{equation*}
		f_{\alpha,\ast} \colon \fromto{\XX_j}{\XX_i}
	\end{equation*}
	for the transition morphism.
	For each $ i \in I $, we write 
	\begin{equation*}
		\pi_{i,\ast} \colon \fromto{\lim_{i \in I} \XX_i}{\XX_i}
	\end{equation*}
	for the projection.
	In addition, assume for each morphism $ \alpha \colon \fromto{j}{i} $ the functor $ f_{\alpha,\ast}$ almost preserves filtered colimits.
\end{ntn}

\begin{prp}\label{lem:ChoughsMoerdijkVermeullenlemma}
	Under the assumptions of \Cref{ntn:inverselimitargument}, for each $ i \in I $ and truncated object $ U \in \XX_{<\infty} $ we have
	\begin{equation}\label{eq:ChoughsMoerdijkVermeullenlemma}
		\piupperstar_i(U) \equivalent \left\{ \colim_{(\alpha,\beta) \in (I_{/i} \cross_{I} I_{/j})^{\op}} f_{\beta,\ast} \fupperstar_{\alpha}(U) \right\}_{j \in I} \period
	\end{equation}
\end{prp}

\begin{proof}
	Since inverse limits in $ \Top_{\infty} $ are computed in $ \Cat_{\infty,\updelta_1} $ (\Cref{thm:filteredlimsinRTop}=\allowbreak\HTT{Theorem}{6.3.3.1}), the assumption that each $ f_{\alpha,\ast} $ almost preserves filtered colimits guarantees that the right-hand side of \eqref{eq:ChoughsMoerdijkVermeullenlemma} is a well-defined object of $ \lim_{j \in I} \XX_j $.

	For each $ i \in I $, the forgetful functor $ \fromto{I_{/i}}{I} $ is limit-cofinal \cite[\HTTthm{Example}{5.4.5.9} \& \HTTthm{Lemma}{5.4.5.12}]{HTT}, so we may without loss of generality assume that $ i \in I $ is a terminal object.
	For each $ k \in I $, write $ f_{k,\ast} \colon \fromto{\XX_k}{\XX_i} $ for the geometric morphism induced by the unique morphism $ \fromto{k}{i} $.
	In this case, a simple cofinality argument shows that
	\begin{equation*}
		\colim_{(\alpha,\beta) \in (I_{/i} \cross_{I} I_{/j})^{\op}} f_{\beta,\ast} \fupperstar_{\alpha}(U) \equivalent \colim_{[\beta \colon k \to j] \in (I_{/j})^{\op}} f_{\beta,\ast} \fupperstar_{k}(U) \period
	\end{equation*}
	By definition, for all $ V \in \XX $ we have
	\begin{align*}
		\Map_{\XX}\paren{\left\{ \colim_{[\beta \colon k \to j] \in (I_{/j})^{\op}} f_{\beta,\ast} \fupperstar_{k}(U) \right\}_{j \in I}, V} &\equivalent \lim_{j \in I} \Map_{\XX_j}\paren{\colim_{\beta \in (I_{/j})^{\op}} f_{\beta,\ast} \fupperstar_{k}(U), \pi_{j,\ast}(V)} \\ 
		&\equivalent \lim_{j \in I} \lim_{\beta \in I_{/j}} \Map_{\XX_j}(f_{\beta,\ast} \fupperstar_{k}(U), \pi_{j,\ast}(V)) \\ 
		&\equivalent \lim_{j \in I} \lim_{\beta \in I_{/j}} \Map_{\XX_j}(f_{\beta,\ast} \fupperstar_{k}(U), f_{\beta,\ast} \pi_{k,\ast}(V))
	\end{align*}
	Rewriting the limit as a limit over $ \beta \in \Fun([1],I) $ and using the fact that the constant functor $ \fromto{I}{\Fun([1],I)} $ is limit-cofinal (since it is a left adjoint), we see that 
	\begin{align*}
		\Map_{\XX}\paren{\left\{ \colim_{[\beta \colon k \to j] \in (I_{/j})^{\op}} f_{\beta,\ast} \fupperstar_{k}(U) \right\}_{j \in I}, V} &\equivalent \lim_{\beta \in \Fun([1],I)} \Map_{\XX_j}(f_{\beta,\ast} \fupperstar_{k}(U), f_{\beta,\ast} \pi_{k,\ast}(V)) \\
		&\equivalent \lim_{k \in I} \Map_{\XX_k}(\fupperstar_k(U),\pi_{k,\ast}(V)) \\ 
		&\equivalent \lim_{k \in I} \Map_{\XX}(\piupperstar_k\fupperstar_k(U),V) \\
		&\equivalent \lim_{k \in I} \Map_{\XX}(\piupperstar_i(U),V) \\
		&= \Map_{\XX}(\piupperstar_i(U),V) \period \qedhere
	\end{align*}
\end{proof}

\begin{cor}\label{lem:filteredcolimdescription}
	Keep the assumptions of \Cref{lem:ChoughsMoerdijkVermeullenlemma}.
	Then for each $ i \in I $ and truncated object $ U \in \XX_{i,<\infty} $, we have an equivalence
	\begin{equation*}
		\pi_{i,\ast} \piupperstar_i(U) \equivalent \colim_{\alpha \in (I_{/i})^{\op}} f_{\alpha,\ast} \fupperstar_{\alpha}(U) 
	\end{equation*}
	of objects of $ \XX_i $.
\end{cor}

\begin{proof}
	For each $ i \in I $, the forgetful functor $ \fromto{I_{/i}}{I} $ is limit-cofinal \cite[\HTTthm{Example}{5.4.5.9} \& \HTTthm{Lemma}{5.4.5.12}]{HTT}, so we may without loss of generality assume that $ i \in I $ is a terminal object.
	Then the claim is clear from \Cref{lem:ChoughsMoerdijkVermeullenlemma} and the definition of $ \pi_{i,\ast} $.
\end{proof}

\begin{prp}\label{prop:protruncshapeinverselimgeneral}
	Keep the assumptions of \Cref{lem:ChoughsMoerdijkVermeullenlemma}, and in addition assume that for each $ i \in I $ the global sections functor $ \Gammaup_{\XX_i,\ast} \colon \fromto{\XX_i}{\Space} $ almost preserves filtered colimits.
	Then the natural morphism
	\begin{equation*}
		\fromto{\Shape(\XX)}{\lim_{i \in I} \Shape(\XX_i)}
	\end{equation*}
	becomes an equivalence after protruncation.
\end{prp}

\begin{proof}
	For each $ i \in I $, the forgetful functor $ \fromto{I_{/i}}{I} $ is limit-cofinal \cite[\HTTthm{Example}{5.4.5.9} \& \HTTthm{Lemma}{5.4.5.12}]{HTT}, so we may without loss of generality assume that $ I $ admits a terminal object $ 1 $. 
	Write $ \Gammaup_{i,\ast} \coloneq \Gammaup_{\XX_i,\ast} $, $ f_{i,\ast} \colon \fromto{\XX_i}{\XX_1} $ for the geometric morphism induced by the unique morphism $ \fromto{i}{1} $ in $ I $, and $ \Gammalowerstar \colon \fromto{\lim_{j \in I} \XX_{j}}{\Space} $ for the global sections geometric morphism.

	We want to show that the natural morphism
	\begin{equation*}
		\fromto{\colim_{i \in I^{\op}} \Gammaup_{i,\ast} \Gammaupperstar_i}{\Gammalowerstar \Gammaupperstar}
	\end{equation*}
	in $ \Fun(\Space,\Space) $ is an equivalence when restricted to truncated spaces \cref{nul:protrun}.
	For any truncated space $ K $, we see that we have equivalences
	\begin{align*}
		\colim_{i \in I^{\op}} \Gammaup_{i,\ast} \Gammaupperstar_i(K) &\equivalent \colim_{i \in I^{\op}} \Gammaup_{1,\ast} f_{i,\ast} \fupperstar_i \Gammaupperstar_{1}(K) \\ 
		&\equivalence \Gammaup_{1,\ast}\paren{\colim_{i \in I^{\op}} f_{i,\ast} \fupperstar_i \Gammaupperstar_{1}(K)} && (\text{assumption on } \Gammaup_{i,\ast}) \\
		&\equivalent \Gammaup_{1,\ast} \of \paren{\colim_{i \in I^{\op}} f_{i,\ast} \fupperstar_i} \of \Gammaupperstar_{1}(K) \\ 
		&\equivalence \Gammaup_{1,\ast} \of \pi_{1,\ast}\piupperstar_1 \of \Gammaupperstar_{1}(K) && (\text{\Cref{lem:ChoughsMoerdijkVermeullenlemma}}) \\
		&\equivalent \Gammalowerstar \Gammaupperstar(K) \period && \phantom{(\text{\Cref{prop:inverselimitBC}})} \qedhere
	\end{align*}
\end{proof}

\begin{nul}
	In particular, the assumptions of \Cref{prop:protruncshapeinverselimgeneral} are satisfied for inverse systems of coherent \topoi where the transition morphisms almost preserve filtered colimits \SAG{Theorem}{A.2.3.1}.
\end{nul}

\noindent From \Cref{cor:coherentmorphismscommutewithfilteredcolims,prop:protruncshapeinverselimgeneral} we deduce: 

\begin{cor}\label{cor:protruncshapeinverselim}
	The protruncated shape 
	\begin{equation*}
		\Shapetrun \colon \fromto{\Topbc}{\Pro(\Space_{<\infty})}
	\end{equation*}
	preserves inverse limits.
\end{cor}


\subsection{Profinite spaces \& Stone \texorpdfstring{$\infty$}{∞}-topoi}\label{subsec:profinshape}

In this section we discuss profinite spaces and their relation to \topoi, as developed in \cite[\SAGapp{E}]{SAG}.
Recall that we write $ \Spaceprofin $ for the \category $ \Pro(\Spacefin) $ of profinite spaces (\Cref{rec:profinspace}).

\begin{nul}
	The restriction of the materialization functor $ \mat \colon \fromto{\Pro(\Space)}{\Space} $ to $ \Spaceprofin $ is right adjoint to the composite
	\begin{equation*}
		\begin{tikzcd}[sep=2em]
			\Space \arrow[r, hooked, "\yo"] & \Pro(\Space) \arrow[r, "{(-)\profincomp}"] & \Spaceprofin
		\end{tikzcd}
	\end{equation*}
	of the Yoneda embedding $ \Space \subset \Pro(\Space) $ followed by profinite completion.
\end{nul}

\begin{dfn}\label{def:profinshape}
	The \defn{profinite shape}\index[terminology]{shape!profinite}\index[terminology]{profinite!shape} functor is the composite
	\begin{equation*}
		\Shapeprofin \coloneq (-)\profincomp \of \Shape \colon \fromto{\Top_{\infty}}{\Spaceprofin}
	\end{equation*}
	of the shape functor $ \Shape $ with the profinite completion functor $ (-)\profincomp \colon \fromto{\Pro(\Space)}{\Spaceprofin} $.
\end{dfn}

\begin{thm}[{\SAG{Theorem}{E.2.4.1}}]
	The composite
	\begin{equation*}
		\begin{tikzcd}[sep=1.5em]
			\lambdapi \colon \Spaceprofin \arrow[r, hooked] & \Pro(\Space) \arrow[r, "\uplambda"] & \Top_{\infty}
		\end{tikzcd}
	\end{equation*}
	of the inclusion $ \Spaceprofin \subset \Pro(\Space) $ with the functor $ \uplambda $ of \Cref{nul:classifyingtopos} is fully faithful and right adjoint to the profinite shape functor $ \Shapeprofin $.
\end{thm}

\begin{dfn}\label{def:Stonetopos}
	\Atopos $ \XX $ is \defn{Stone}%
	\footnote{Lurie calls these \topoi \textit{profinite}.}%
	\index[terminology]{topos@\topos!Stone}\index[terminology]{Stone!topos@\topos}
	if $ \XX $ lies in the essential image of the fully faithful functor $ \lambdapi \colon \incto{\Spaceprofin}{\Top_{\infty}} $.
	We write $ \TopStone \subset \Top_{\infty} $\index[notation]{TopStone@$\TopStone$} for the full subcategory spanned by the Stone \topoi.

	Consequently, the inclusion $ \incto{\TopStone}{\Top_{\infty}} $ admits a left adjoint
	\begin{equation*}
		(-)^{\Stone} \colon \fromto{\Top_{\infty}}{\TopStone}
	\end{equation*}
	which we refer to as the \defn{Stone reflection}.\index[terminology]{Stone reflection}\index[terminology]{reflection!Stone}
\end{dfn}

\begin{nul}
	Since bounded coherent \topoi are closed under inverse limits in $ \Top_{\infty} $ (\Cref{cor:SAG.A.8.3.3}=\allowbreak\SAG{Corollary}{A.8.3.3}), \Cref{exm:Spacecoh} shows that Stone \topoi are bound\-ed coherent.
\end{nul}

\begin{prp}[{\SAG{Proposition}{E.3.1.4}}]\label{prp:FunstartoStoneisspace}
	Let $ \XX $ and $ \YY $ be \topoi.
	If $ \YY $ is Stone, then the \category $ \Funlowerstar(\XX,\YY) $ is a (small) \groupoid.
\end{prp}

\begin{nul}\label{nul:matofprofinshape}
	If $ \YY $ is a Stone \topos, then since $ \Space $ is Stone and $ \lambdapi $ is fully faithful with left adjoint given by the profinite shape, we see that
	\begin{equation*}
		\Pt(\YY) \equivalent \Map_{\Top_{\infty}}(\Space,\YY)\equivalent \mat \Shapeprofin(\YY) \period
	\end{equation*}
\end{nul}

Since Stone \topoi are bounded coherent, \Cref{nul:matofprofinshape} combined with Conceptual Completeness (\Cref{thm:conceptualcompleteness}=\allowbreak\SAG{Theorem}{A.9.0.6}) imply the following `Whitehead Theorem' for profinite spaces. 

\begin{thm}[{Whitehead Theorem for profinite spaces\index[terminology]{Whitehead Theore!for profinite spaces}; \SAG{Theorem}{E.3.1.6}}]\label{thm:profiniteWhitehead}
	The materialization functor $ \mat \colon \fromto{\Spaceprofin}{\Space} $ is conservative.
\end{thm}

\begin{prp}[{\SAG{Proposition}{E.4.6.1}}]\label{prop:matandtruncatedness}
	Let $ n \in \NNup $.
	A morphism $ f $ in $ \Spaceprofin $ is $ n $-truncated if and only if $ \mat(f) $ is an $ n $-truncated morphism of $ \Space $.
\end{prp}

Stone \topoi have a number of useful alternative characterizations.
The first is that, under the assumption of bounded coherence, the conclusion of \Cref{prp:FunstartoStoneisspace}=\allowbreak\SAG{Proposition}{E.3.1.4} actually characterizes Stone \topoi.

\begin{thm}[{\SAG{Theorem}{E.3.4.1}}]\label{thm:Stonesaretopoiwithgroupoidsofpoints}
	Let $ \XX $ be \atopos.
	Then $ \XX $ is Stone if and only if both of the following conditions are satisfied:
	\begin{itemize}
		\item The \topos $ \XX $ is bounded and coherent.

		\item The \category of points $ \Pt(\XX) $ of $ \XX $ is \agroupoid.
	\end{itemize}
\end{thm}

The next characterization is that bounded coherent objects are in fact \textit{lisse}.
First we recall the definition of lisse sheaves as well as some material on lisse sheaves that we'll utilize later on.

\begin{rec}\label{rec:localsyslisse}
	Let $\XX$ be \atopos.
	An object $F\in \XX$ is called a \defn{locally constant}\index[terminology]{locally constant} if and only if there exists a cover $\{U_{i}\}_{i \in I}$ of the terminal object of $\XX$, a corresponding family $\{K_{i}\}_{i \in I}$ of spaces, and an equivalence 
	\begin{equation*}
		F \times U_{i}\simeq \Gammaupperstar_{\XX}(K_{i}) \cross U_i
	\end{equation*}
	in $ \XX_{/U_i} $ for each $ i \in I $.

	We say that a locally constant object $F$ as above is \defn{lisse}\index[terminology]{lisse}\index[terminology]{sheaf!lisse}%
	\footnote{Lurie uses the phrase \defn{locally constant constructible}\index[terminology]{locally constant constructible}\index[terminology]{sheaf!locally constant constructible}.}
	if, in addition, the set $ I $ can be chosen to be finite, and the spaces $K_{i}$ can be chosen to be \pifinite.

	We write 
	\begin{equation*}
		\XX^{\locsys}\subseteq \XX \andeq \XX^{\lisse}\subseteq \XX
	\end{equation*}
	for the full subcategories spanned by the locally constant objects and lisse objects, respectively.
	Please note that for any geometric morphism of \topoi $ \flowerstar \colon \fromto{\XX}{\YY} $, the pullback $ \fupperstar \colon \fromto{\YY}{\XX} $ preserves lisse objects.
\end{rec}

Later we'll find the following simple characterization of lisse sheaves as a single pullback very useful:

\begin{lem}[{\SAG{Proposition}{E.2.7.7}}]\label{prp:SAG.E.2.7.7} 
	Let $\XX$ be \atopos.
	Then an object $F$ of $\XX$ is lisse if and only if there exist: a full subcategory $ G \subset \iota \Spacefin $ spanned by finitely many objects, an unique geometric morphism $\glowerstar\colon\fromto{\XX}{\Space_{/G}}$, and an unique equivalence $F\simeq g^{\ast}(I)$, where $I$ classifies the inclusion functor $\fromto{G}{\Space}$.
\end{lem}

\noindent The following useful fact is equivalent to the fact that the profinite shape
\begin{equation*}
	\Shapeprofin \colon \fromto{\Topbc}{\Spaceprofin}
\end{equation*}
preserves inverse limits (see \Cref{cor:protruncshapeinverselim}).

\begin{lem}\label{prp:spaceGiscocompact} 
	For any \pifinite space $G$, the \topos $\Space_{/G}$ is cocompact in $\Topbc$.
	That is, for any inverse system $\{\XX_{\alpha}\}_{\alpha \in A}$ of bounded coherent \topoi with limit $\XX$, the natural functor
	\begin{equation*}
		\Funlowerstar(\XX,\Space_{/G})\to\lim_{\alpha \in A}\Funlowerstar(\XX_{\alpha},\Space_{/G})
	\end{equation*}
	is an equivalence.
\end{lem}

Now we turn to the important characterization of Stone \topoi in terms of lisse sheaves and its consequences.

\begin{prp}[{\SAG{Proposition}{E.3.1.1}}]\label{prop:SAG.E.3.1.1}
	Let $ \XX $ be \topos.
	Then $ \XX $ is Stone if and only if both of the following conditions are satisfied.
	\begin{itemize}
		\item The \topos $ \XX $ is bounded and coherent.

		\item Every truncated coherent object of $ \XX $ is lisse.
	\end{itemize}
\end{prp}

\begin{cor}[{\SAG{Corollary}{E.3.1.2}}]\label{cor:morphismtoStoneiscoherent}
	Let $ \flowerstar \colon \fromto{\XX}{\YY} $ be a geometric morphism between coherent \topoi.
	If $ \YY $ is Stone, then $ \flowerstar $ is coherent.
\end{cor}


By the characterization of Stone \topoi in terms of lisse sheaves, it is not surprising that the Stone reflection of \atopos $ \XX $ can be written as sheaves on $ \XX^{\lisse} $ with respect to the effective epimorphism topology:

\begin{thm}[{\SAG{Theorem}{E.2.3.2}}]\label{thm:SAG.E.2.3.2}
	Let $ \XX $ be \atopos.
	Then:
	\begin{itemize}
		\item The \category $ \XX^{\lisse} $ is a bounded \pretopos and the inclusion $ \incto{\XX^{\lisse}}{\XX} $ is a morphism of \pretopoi.

		\item The inclusion \smash{$ \incto{\XX^{\lisse}}{\XX} $} induces a geometric morphism \smash{$ \fromto{\XX}{\Sheff{\XX^{\lisse}}} $} which exhibits \smash{$ \Sheff{\XX^{\lisse}} $} as the Stone reflection of $ \XX $.
	\end{itemize}
\end{thm}

\noindent Now we assemble all of the ways we can check that a geometric morphism induces an equivalence on Stone reflections.

\begin{cor}[{\SAG{Corollary}{E.2.3.3}}]\label{cor:SAG.E.2.3.3}
	Let $ \flowerstar \colon \fromto{\XX}{\YY} $ be a geometric morphism of \topoi.
	The following are equivalent:
	\begin{itemize}
		\item The induced geometric morphism $ \flowerstar^{\Stone} \colon \fromto{\XX^{\Stone}}{\YY^{\Stone}} $ is an equivalence of \topoi.

		\item The geometric morphism $ \flowerstar $ is a profinite shape equivalence.

		\item The morphism $ \Pt(\flowerstar^{\Stone}) $ is an equivalence of \groupoids.

		\item The pullback functor $ \fupperstar $ restricts to an equivalence of \categories $ \equivto{\YY^{\lisse}}{\XX^{\lisse}} $.
	\end{itemize}
\end{cor}

Putting together the basics about Stone \topoi gives an alternative proof of the monodromy equivalence for lisse sheaves proven by Bachmann and Hoyois \cite[Proposition 10.1]{MotivicNorms:BachmannHoyois}.

\begin{prp}\label{prop:Stonemonodromy}
	Let $ \XX $ be \atopos the unit $ \fromto{\XX}{\XX^{\Stone}} $ of the adjunction to Stone \topoi restricts to an equivalence
	\begin{equation*}
		\Fun(\Shapeprofin(\XX),\Spacefin) \equivalent \XX^{\lisse} \period
	\end{equation*}
\end{prp}

\begin{proof}
	Represent the profinite shape $ \Shapeprofin(\XX) $ by an inverse system $ \{\Pi_{\alpha}\}_{\alpha \in A} $ of \pifinite spaces so that 
	\begin{equation*}
		\Fun(\Shapeprofin(\XX),\Spacefin) = \colim_{\alpha \in A^{\op}} \Fun(\Pi_{\alpha},\Spacefin) \period
	\end{equation*}
	By \Cref{exm:Spacecoh,nul:truncohslice}, for any \pifinite space $ \Pi $ we have $ \Fun(\Pi,\Space)_{<\infty}^{\coh} = \Fun(\Pi,\Spacefin) $, so
	\begin{align*}
		\Fun(\Shapeprofin(\XX),\Spacefin) &= \colim_{\alpha \in A^{\op}} \Fun(\Pi_{\alpha},\Space)_{<\infty}^{\coh} \\
		&\equivalent \paren{\textstyle \lim_{\alpha \in A} \Fun(\Pi_{\alpha},\Space)}_{<\infty}^{\coh} && \text{(\Cref{prop:SAG.A.8.3.2})} \\ 
		&\equivalent (\XX^{\Stone})_{<\infty}^{\coh} && \text{(\Cref{def:Stonetopos})} \\ 
		&\equivalent \XX^{\lisse} && \text{(\Cref{thm:SAG.E.2.3.2})} \period \qedhere
	\end{align*}
\end{proof}

\begin{exm}\label{exm:finiteetalesite}
	Let $ X $ be a coherent scheme, and write $ X^{\fet} $ for the \textit{finite étale site}\index[terminology]{site!finite étale}\index[terminology]{finite étale site} of $ X $:
	the full subcategory of the étale site $ X^{\et} $ spanned by the finite étale $ X $-schemes, with the induced topology (see \cite[§VI.9]{MR3444777}).
	Since the finite étale site is a finitary site, the $ 1 $-localic finite étale \topos $ X_{\fet} \coloneq \Sh(X^{\fet}) $ is coherent (\Cref{prop:SAG.A.3.1.3}=\allowbreak\SAG{Proposition}{A.3.1.3}). 
	The finite étale \topos $ X_{\fet} $ is the classifying \topos of the profinite étale fundamental groupoid of $ X $ (cf. \cites[Exposé V, Proposition 5.8]{MR50:7129}[Lemma VI.9.11]{MR3444777}).
	In particular, the finite étale \topos $ X_{\fet} $ is Stone.
\end{exm}

\begin{ntn}\label{ntn:absGalois}
	Let $ k $ be a field and $ \ksep \supset k $ a separable closure of $ k $.
	We write $ \Gk $ for the absolute Galois group of $ k $ with respect to $ \ksep $.
\end{ntn}

\begin{exm}\label{exm:etaletoposfield}
	Let $ k $ be a field and $ \ksep \supset k $ a separable closure of $ k $.
	This choice of separable closure provides an identification $ (\Spec k)_{\et} \equivalent \widetilde{\BG}_{k} $ of the étale \topos of $ \Spec k $ with the classifying \topos of the profinite group $ \Gk $.
	In particular, $ (\Spec k)_{\et} $ is a Stone \topos.
\end{exm}

\newpage

\section{Oriented pushouts \& oriented fiber products}\label{sec:orientedfiberprod} 

This chapter is dedicated to the study of squares of \topoi that commute up to a natural transformation.
We are particularly interested in the two universal examples of these oriented squares: \textit{recollements} or \textit{oriented pushouts}, and \textit{oriented fiber products}.
Recollements are integral to the theory of stratified higher topoi that we present in \Cref{part:strattopoi}.
In the most basic example, \atopos stratified over the poset $ [1] $ is equivalent to the data of a recollement of \topoi.
Moreover, the whole theory of stratified \topoi is really a generalization of this example.
Oriented fiber products appear in a two ways in this text.
The primary way is in the décollage approach to stratified \topoi that we present in \Cref{part:strattopoi}; this is the topos-theoretic analogue of the approach to stratified spaces presented in \Cref{subsec:spatialdecollages,subsec:stratifiednerve,subsec:profintiedecollages}. 
More precisely, the link between two strata in a stratified \topos (satisfying suitable finiteness hypotheses) is their oriented fiber product.

Since they are the technically more challenging of the two, the majority of the chapter is dedicated to the study of oriented fiber products of \topoi.
Deligne originally considered oriented fiber products of $1$-topoi in order to construct the natural target for the nearby cycles functor \cites[Exposé XIII]{MR0354657}{MR726426}\footnote{The latter text was written by Gérard Laumon.} and work with nearby cycles over more general bases (see also the works of Gabber, Orgogozo, and Saito \cites{Illusie:vanishingcycles}[Exposé XII]{MR3309086}{MR726426}{MR2249998}{MR3595935}).
This target for the nearby cycles functor is known as the \textit{vanishing} topos.
Vanishing \topoi play an important role in this text as well: \Cref{sec:localtopoi} is dedicated to the study of a special class of vanishing \topoi that play the role of local rings in higher topos theory.
These \textit{local \topoi} are a key tool that we use to reduce many questions about \topoi with enough points to questions about local \topoi.

The existence of oriented fiber products of $ 1 $-topoi was first proven by Giraud \cite{MR0344253}; in order to prove properties of oriented fiber products, Deligne provided a description in terms of generating sites.
The bulk of the technical work in this chapter is in showing that Deligne's generating site also works in the setting of \topoi. 
The payoff is that this description allows us to easily see that the oriented fiber product of bounded coherent \topoi is again bounded coherent. 

In \cref{subsec:recollements} we review recollements of \topoi.
The recollement of bounded coherent \topoi is generally neither bounded nor coherent; \Cref{cnstr:orientedfibboundedpretop} explains how to fix this.
\Cref{subsec:orientedsquares} discuesses squares of \topoi that commute up to a natural transformation and the definition of \textit{oriented pushouts}.
\Cref{subsec:internalhoms} discusses internal Homs in $ \Top_{\infty} $ and \textit{path \topoi}.
In \cref{subsec:OFP} we introduce the oriented fiber product of \topoi as an iterated pullback involving path \topoi.
In \Cref{subsec:generatingsites} we give a site-theoretic description of the oriented fiber product and use it to prove that the oriented fiber product of bounded coherent \topoi is again bounded coherent (\Cref{lem:orientedcoherent}).
\Cref{subsec:etalecompat} proves a compatibility between étale geometric morphisms and oriented fiber products (\Cref{prop:orientedslice}) that we'll need to prove a basechange theorem for oriented fiber products in \cref{section:BC}.


\subsection{Recollements of higher topoi}\label{subsec:recollements}

We begin with open and closed subtopoi.

\begin{nul} 
	Let $ \XX$ is \atopos and $U \in \XX $.
	Recall that the overcategory $ \XX_{/U}$ is \atopos, and the forgetful functor $j_!\colon\fromto{\XX_{/U}}{\XX}$ admits a right adjoint $ \jupperstar $, which itself admits a right adjoint $ \jlowerstar$ (\Cref{HTT.6.3.5.6}).
	If $ U $ is an open of $ \XX $, the functor $ \jlowerstar$ is fully faithful.
	
	In this case, we write $ \XX_{\smallsetminus U}$ for the full subcategory of $ \XX$ spanned by those objects $F$ such that the projection $ \pr_2 \colon \fromto{F\times U}{U} $ is an equivalence. 
	The inclusion $ \XX_{\smallsetminus U} \subset \XX $ is accessible and admits a left exact left adjoint, so that $ \XX_{\smallsetminus U} $ is \atopos \HTT{Proposition}{7.3.2.3}.
	We call the \topos $ \XX_{\smallsetminus U} $ the \defn{closed complement}\index[terminology]{complement!closed}\index[terminology]{subtopos!closed}\index[terminology]{closed!subtopos} of $ \XX_{/U}$, and $ \ilowerstar \colon\incto{\XX_{\smallsetminus U}}{\XX}$ for the inclusion.

	In this case, $ \XX$ is a recollement \Cref{rec:recollement} of $ \XX_{\smallsetminus U}$ and $ \XX_{/U}$ with gluing functor $ \iupperstar\jlowerstar$, \textit{viz.},
	\begin{equation*}
		\XX \simeq \XX_{\smallsetminus U}\orientedcup^{\iupperstar\jlowerstar}\XX_{/U} \period
	\end{equation*}
\end{nul}

\begin{nul}\label{nul:usingconservativityofrecoll}
	Let $ \XX $ be \atopos, and let
	\begin{equation*}
		\ilowerstar \colon\incto{\ZZ}{\XX} \andeq \jlowerstar \colon\incto{\UU}{\XX}
	\end{equation*}
	be geometric morphisms of \topoi that exhibit $ \XX$ as the recollement $ \orientedpush{\ZZ}{\iupperstar\jlowerstar}{\UU} $.
	Then since $ \iupperstar $ and $ \jupperstar $ are left exact left adjoints, the natural conservative functor
	\begin{equation*}
		(\iupperstar,\jupperstar) \colon \fromto{\XX}{\ZZ \sqcup \UU}
	\end{equation*} 
	preserves and reflects colimits and finite limits.
	(Here $ \ZZ \sqcup \UU $ denotes the coproduct of $ \ZZ $ and $ \UU $ in $ \Top_{\infty} $, which is the \emph{product} of $ \ZZ $ and $ \UU $ in $ \Cat_{\infty,\updelta_1} $.) 
	In particular, a morphism $ f $ in $ \XX $ is: 
	\begin{enumerate}[(\ref*{nul:usingconservativityofrecoll}.1)]
		\item\label{nul:usingconservativityofrecoll.1} an effective epimorphism if and only if both $ \iupperstar(f) $ and $ \jupperstar(f) $ are effective epimorphisms.

		\item\label{nul:usingconservativityofrecoll.2} $ n $-connective for  $ n \in \NNrhd $ if and only if both $ \iupperstar(f) $ and $ \jupperstar(f) $ are $ n $-connective.

		\item\label{nul:usingconservativityofrecoll.3} $ n $-truncated for some integer $ n \geq -2 $ if and only if both $ \iupperstar(f) $ and $ \jupperstar(f) $ are $ n $-truncated.
	\end{enumerate}
	See \Cref{nul:truncatednessinrecoll}.
\end{nul}

\begin{nul}\label{nul:hypercompleterecollement}
	From \enumref{nul:usingconservativityofrecoll}{2} we see that if \atopos $ \XX $ is the recollement of $ \ZZ $ and $ \UU $ and both $ \ZZ $ and $ \UU $ are hypercomplete, then $ \XX $ is also hypercomplete.
\end{nul}

\begin{nul}\label{nul:recollement1-strat}
	A recollement of \topoi is tantamount to a geometric morphism of \topoi $ \fromto{\XX}{\widetilde{[1]}} $. 
	Indeed, if $ \ZZ$ and $ \UU$ are \topoi, and $ \phi\colon\fromto{\UU}{\ZZ}$ is a left exact accessible functor, then the recollement \smash{$ \XX\coloneq\ZZ\orientedcup^{\phi}\UU$} is \atopos \HAa{Proposition}{A.8.15}, and the global sections geometric morphisms $ \fromto{\ZZ}{\Space}$ and $ \fromto{\UU}{\Space} $ induce a geometric morphism
	\begin{equation*}
		\fromto{\XX}{\Space\orientedcup^{\id_{\Space}}\Space\simeq \widetilde{[1]}} \period
	\end{equation*}
	In the other direction, given a geometric morphism $ \fromto{\XX}{\widetilde{[1]}} $, the closed subtopos
	\begin{align*}
		\XX_{0} &\coloneq \widetilde{\{0\}} \cross_{\widetilde{[1]}} \XX \subset \XX
		\intertext{and the open subtopos}
		\XX_{1} &\coloneq \widetilde{\{1\}} \cross_{\widetilde{[1]}} \XX \subset \XX
	\end{align*}
	form a recollement of $ \XX$. 

	In a strong sense, the entire theory of stratified \topoi (\cref{sec:strattopoi}) is a generalization of this observation.
\end{nul}

Since $ n $-localic and bounded \topoi (\Cref{cnstr:localictopoi} \& \Cref{cnstr:boundedtopoi}) are each closed under limits in $ \Top_{\infty} $, we deduce the following.

\begin{lem}\label{lem:boundedrecollement}
	Let $ n \in \NNup $, let $ \XX $ be \atopos, and let $ \ilowerstar \colon\incto{\ZZ}{\XX}$ and $ \jlowerstar \colon\incto{\UU}{\XX}$ be geometric morphisms of \topoi that exhibit $ \XX $ as the recollement of $ \ZZ $ and $ \UU $ along $ \iupperstar\jupperstar $.
	If $ \XX $ is $ n $-localic (respectively, bounded), then both $ \ZZ $ and $ \UU $ are $ n $-localic (resp., bounded).
\end{lem}


\begin{wrn}\label{wrn:boundednessnotpushedout}
	We caution, however, that there isn't a simple converse to \Cref{lem:boundedrecollement}: the recollement of two bounded \topoi is not necessarily bounded.
	To ensure boundedness, we need a condition on the gluing functor.
\end{wrn}

\begin{dfn}\label{dfn:boundedgluing}
	Let $ \ZZ$ and $ \UU$ be two bounded \topoi, and let $ \phi \colon \fromto{\UU}{\ZZ}$ be a left exact accessible functor. 
	We say that $ \phi$ is a \defn{bounded gluing functor}\index[terminology]{gluing functor!bounded}\index[terminology]{bounded!gluing functor} if and only if the recollement \smash{$ \XX \coloneq \ZZ \orientedcup^{\phi} \UU$} is bounded.
\end{dfn}

\begin{qst}
	Do bounded gluing functors admit a simple or useful intrinsic characterization?
\end{qst}

Let us now turn to the coherence of recollements (\Cref{rec:coherence}). 
We can easily chracterize the coherent objects of a coherent recollement.

\begin{prp}[{\cite[Proposition 2.3.22]{DAGXIII}}]\label{prop:DAGXIII.2.3.22}
	Let $ n \in \NNup $, let $ \XX $ be an $ (n+1) $-coherent \topos, and let $ \ilowerstar \colon \incto{\ZZ}{\XX}$ and $ \jlowerstar \colon\incto{\UU}{\XX}$ be geometric morphisms of \topoi that exhibit $ \XX$ as the recollement of $ \ZZ $ and $ \UU $ along $ \iupperstar\jupperstar$.
	If $ \UU $ is $ 0 $-coherent, then an object $ F \in \XX $ is $ n $-coherent if and only if both $ \iupperstar(F) $ and $ \jupperstar(F) $ are $ n $-coherent.
	In particular, the \topoi $ \ZZ $ and $ \UU $ are $ n $-coherent.
\end{prp}

\begin{wrn}\label{wrn:coherencenotpushedout}
	 We caution again that there isn't a simple converse to \Cref{prop:DAGXIII.2.3.22}: as with boundedness, the recollement of two coherent \topoi is not necessarily coherent.
\end{wrn}

\begin{dfn}\label{dfn:coherentgluing}
	Let $ \ZZ$ and $ \UU$ be coherent \topoi, and let $ \phi\colon\fromto{\UU}{\ZZ}$ be a left exact accessible functor.
	We say that $ \phi$ is a \defn{coherent gluing functor}\index[terminology]{gluing functor!coherent}\index[terminology]{coherent!gluing functor} if and only if the recollement \smash{$ \XX\coloneq\ZZ\orientedcup^{\phi}\UU$} is coherent.
\end{dfn}

\begin{nul}
	Let $ \ZZ$ and $ \UU$ be coherent \topoi, and let $ \phi\colon\fromto{\UU}{\ZZ}$ be a left exact accessible functor. Write $ \ilowerstar \colon\incto{\ZZ}{\XX}$ and $ \jlowerstar \colon\incto{\UU}{\XX}$ for the fully faithful functors defining the recollement.
	Then one can show that the gluing functor $ \phi$ is coherent if the following conditions are satisfied.
	\begin{itemize}
		\item The functor $ \jlowerstar$ is quasicompact in the sense that for every quasicompact object $F\in\XX$, the object $ \jupperstar(F)\in\UU$ is also quasicompact.

		\item For every $n\in\NNup$, every object $F\in\UU$ admits a family $ \{G_{\alpha}\to F\}_{\alpha \in A}$ in which each $G_{\alpha}$ is $n$-coherent, and the family $ \{\phi(G_{\alpha})\to\phi(F)\}_{\alpha \in A}$ is a covering in $ \ZZ$.
	\end{itemize}
\end{nul}

\begin{cnstr}[bounded coherent recollement]\label{cnstr:orientedfibboundedpretop} 
	Let $ \ZZ$ and $ \UU$ be bounded coherent \topoi, and let $ \phi\colon\fromto{\UU}{\ZZ}$ be a left exact accessible functor. Form the recollement
	\begin{equation*}
		\XX'\coloneq\ZZ\orientedcup^{\phi}\UU \comma
	\end{equation*}
	and write $ \ilowerstar \colon\incto{\ZZ}{\XX'}$ and $ \jlowerstar \colon\incto{\UU}{\XX'}$ for the induced closed and open embeddings.
	Consider the full subcategory $X_0\subseteq\XX'$ spanned by those objects $F$ such that both $ \iupperstar (F)$ and  $ \jupperstar(F)$ are each truncated coherent, so that $X_0$ is the oriented fiber product \Cref{nul:orientedfpinCat} in $ \Cat_{\infty,\updelta_1}$:
	\begin{equation*}
		X_0 = \commacat{\ZZcohbdd}{\ZZ}{\UUcohbdd} \period
	\end{equation*}
	Then since $ X_0 \subset \XX' $ is closed under finite limits, finite coproducts, and the formation of geometric realizations of groupoid objects, the \category $ X_0 $ is \apretopos and the inclusion $ \incto{X_0}{\XX'} $ is a morphism of \pretopoi (\Cref{dfn:pretopos}).
	Moreover, by \Cref{nul:usingconservativityofrecoll} every object of $ X_0 $ is truncated and by \Cref{nul:orientedfpinCatesssmall} the \category $ X_0 $ is $ \updelta_0 $-small, hence $ X_0 $ is a bounded \pretopos (\Cref{dfn:boundedpretopos}).
	Consequently, the \topos\index[notation]{ZunionbcU@$\ZZ \orientedcupbc^{\phi} \UU$}
	\begin{equation*}
		\ZZ \orientedcupbc^{\phi} \UU \coloneq\Sheff{X_0} 
	\end{equation*}
	is bounded coherent (\Cref{ntn:effepi}).
	By \SAG{Proposition}{A.6.4.4}, we see that the inclusion $ \incto{X_0}{\XX'}$ extends (uniquely) to a comparison geometric morphism
	\begin{equation*}
		\rlowerstar \colon\fromto{\XX'}{\ZZ\orientedcupbc^{\phi}\UU} \period
	\end{equation*}
	The geometric morphism $ \rlowerstar $ is not generally an equivalence; however, $ \rupperstar $  restricts to an equivalence
	\begin{equation*}
		\rupperstar \colon \equivto{(\ZZ\orientedcupbc^{\phi}\UU)\cohbdd}{X_0} \period
	\end{equation*}
	The geometric morphisms $ \rlowerstar \ilowerstar $ and $ \rlowerstar \jlowerstar$ are both coherent by construction.
	We therefore call \smash{$ \ZZ\orientedcupbc^{\phi}\UU$} the \defn{bounded coherent recollement}%
	\index[terminology]{recollement!bounded coherent}\index[terminology]{bounded coherent!recollement}
	of $ \ZZ $ and $ \UU $ along $ \phi $.
\end{cnstr}

\begin{lem}
	Keep the notations of \Cref{cnstr:orientedfibboundedpretop}.
	Then the natural geometric morphism
	\begin{equation*}
		\ZZ\orientedcup^{\iupperstar \rupperstar \rlowerstar \jlowerstar}\UU\to\ZZ\orientedcupbc^{\phi}\UU
	\end{equation*}
	is an equivalence.
\end{lem}

\begin{proof}
	Write $ \XX\coloneq\ZZ\orientedcupbc^{\phi}\UU$.
	To prove the claim, we show that we have equivalences 
	\begin{equation*}
		\rlowerstar \jlowerstar \colon \equivto{\UU}{\XX_{/j_!(1_{\UU})}} \andeq \ilowerstar \rlowerstar \colon \equivto{\ZZ}{\XX_{\smallsetminus j_!(1_{\UU})}} \period
	\end{equation*}

	The object $j_!(1_{\UU}) \in \XX $, is the object
	\begin{equation*}
		(\varnothing_{\ZZ},1_{\UU},\varnothing_{\ZZ}\to \phi(1_{\UU})) \period
	\end{equation*}
	From this description it is clear that $ j_!(1_{\UU}) $ is an object of the \pretopos $ X_0 \subset \XX $ of \Cref{cnstr:orientedfibboundedpretop} and that $ j_!(1_{\UU}) $ is an open of $ \XX $.
	Thus $ \jupperstar \rupperstar$ restricts to an equivalence
	\begin{equation*}
		\equivto{(\XX_{/j_!(1_{\UU})})\cohbdd}{\UUcohbdd} \period
	\end{equation*}
	Hence the functor $ \rlowerstar \jlowerstar \colon\UU\to\XX_{/j_!(1_{\UU})}$ is an equivalence.
	The truncated coherent objects of the closed subtopos $ \XX_{\smallsetminus j_!(1_{\UU})}$ are precisely those of the form $(F_{\ZZ},1_{\UU},F_{\ZZ}\to \phi(1_{\UU}))$ for some truncated coherent object $F_{\ZZ}$ of $ \ZZ$.
	Hence $ \iupperstar \rupperstar$ restricts to an equivalence
	\begin{equation*}
		\equivto{(\XX_{\smallsetminus j_!(1_{\UU})})\cohbdd}{\ZZcohbdd} \period
	\end{equation*}
	Thus the functor $ \ilowerstar \rlowerstar \colon\ZZ\to\XX_{\smallsetminus j_!(1_{\UU})}$ is an equivalence.
\end{proof}

\begin{lem}\label{lem:whenrecollementisbc}
	Let $ \ZZ$, and $ \UU$ be bounded coherent \topoi, and let $ \phi\colon\fromto{\UU}{\ZZ}$ be a bounded coherent gluing functor.
	Then \smash{$ \ZZ\orientedcup^{\phi}\UU$} is the bounded coherent recollement \smash{$ \ZZ\orientedcupbc^{\phi}\UU $}.
\end{lem}

\begin{proof}
	This follows from \Cref{prop:DAGXIII.2.3.22}=\cite[Proposition 2.3.22]{DAGXIII} combined with \Cref{thm:SAG.A.7.5.3}=\allowbreak\SAG{Theorem}{A.7.5.3}.
\end{proof}

The critical point that we use repeatedly in the sequel is the observation that the bounded coherent recollement depends only upon the restriction of the gluing functor to truncated coherent objects.
More precisely, let $ \ZZ$ and $ \UU$ be bounded coherent \topoi, and let $ \phi\colon\fromto{\UU}{\ZZ}$ and $ \phi'\colon\fromto{\UU}{\ZZ}$ be left exact accessible functors.
Let $ \eta\colon\fromto{\phi}{\phi'}$ be a natural transformation.
Now $ \eta$ induces a functor
\begin{equation*}
	\eta^{\ast}\colon\fromto{\ZZ\orientedcup^{\phi}\UU}{\ZZ\orientedcup^{\phi'}\UU}
\end{equation*}
given by the assignment
\begin{equation*}
	(z,u,\alpha \colon \fromto{z}{\phi(u)}) \mapsto (z,u,\eta_u \alpha \colon \fromto{z}{\phi'(u)}) \period
\end{equation*}
The functor $ \eta^{\ast} $ preserves colimits and finite limits; consquently, $ \eta^{\ast} $ is the left adjoint of a geometric morphism $ \eta_{\ast}$.
Note that $ \eta^{\ast} $ restricts to a functor 
\begin{equation*}
	\eta^{\ast}\colon\fromto{\commacat{\ZZcohbdd}{\ZZ,\phi}{\UUcohbdd}\simeq\XXcohbdd}{(\XX')\cohbdd\simeq\commacat{\ZZcohbdd}{\ZZ,\phi'}{\UUcohbdd}} \comma
\end{equation*}
i.e., $ \eta^{\ast} $ preserves truncated coherent objects.
Hence $ \eta^{\ast} $ induces a geometric morphism
\begin{equation*}
	\eta_{\ast} \colon \fromto{\ZZ\orientedcupbc^{\phi'}\UU}{\ZZ\orientedcupbc^{\phi}\UU}
\end{equation*}
on bounded coherent recollements.

\begin{prp}\label{prp:bcorientedpodependsonbcgluing}
	Let $ \ZZ$ and $ \UU$ be bounded coherent \topoi, and let $ \phi\colon\fromto{\UU}{\ZZ}$ and $ \phi'\colon\fromto{\UU}{\ZZ}$ be left exact accessible functors.
	Let $ \eta\colon\fromto{\phi}{\phi'}$ be a natural transformation.
	If \smash{$ \eta|_{\UUcohbdd} $} is an equivalence, then $ \eta$ induces an equivalence
	\begin{equation*}
		\eta_{\ast} \colon \equivto{\ZZ\orientedcupbc^{\phi'}\UU}{\ZZ\orientedcupbc^{\phi}\UU} \period
	\end{equation*}
\end{prp}

\begin{qst} 
	As a result of \Cref{prp:bcorientedpodependsonbcgluing}, the restriction functor
	\begin{equation*}
		\Funlex(\UU,\ZZ)\to \Funlex(\UUcohbdd,\ZZ)
	\end{equation*}
	is fully faithful on bounded coherent gluing functors.
	What is the essential image of the bounded coherent gluing functors? 
	It might be helpful to have a simple intrinsic characterization.
\end{qst}


\subsection{Oriented squares \& oriented pushouts}\label{subsec:orientedsquares}

To speak of oriented pullbacks of \topoi without finding ourselves buried under a mass of pernicious details (or unproved claims) about double \categories or $(\infty,2)$-cate\-gories, we express the universal property of the oriented pushout in simple terms.
The key kind of square we will have to contemplate is the following.

\begin{ntn}
	We exhibit data of geometric morphisms $ \flowerstar \colon\fromto{\XX}{\ZZ}$, $ \glowerstar\colon\fromto{\YY}{\ZZ}$, $ \plowerstar \colon\fromto{\WW}{\XX}$, and $ \qlowerstar \colon\fromto{\WW}{\YY}$, along with a (not necessarily invertible) natural transformation $ \sigma\colon\fromto{\glowerstar\qlowerstar}{\flowerstar \plowerstar}$ by the single square
	\begin{equation}\label{square:laxsquare}
		\begin{tikzcd}
			\WW \arrow[r, "\qlowerstar" above] \arrow[d, "\plowerstar" left] & \YY \arrow[d, "\glowerstar" right] \arrow[dl, phantom, "\scriptstyle \sigma" below right, "\Longleftarrow" sloped] \\ 
			\XX \arrow[r, "\flowerstar" below] & \ZZ \period
		\end{tikzcd}
	\end{equation}
\end{ntn}

\begin{wrn}
	It seems that this convention for writing $2$-cells is the \textit{opposite} of what's written in some of the $1$-topos theory literature \cites{MR977478}{MR1731050}{MR1787303}, but it agrees with of the algebro-geometric literature \cites[Exposé XIII]{MR0354657}{MR726426}.
	We therefore emphasize that our $2$-morphisms are natural transformations between the \textit{right} adjoints.
\end{wrn}

The oriented fiber product in $ \Cat_{\infty,\updelta_1}$ of a diagram of \topoi does not recover the oriented fiber product in $ \Top_{\infty}$, but rather the \textit{oriented pushout} in $ \Top_{\infty}$. 
We shall also have to contemplate the oriented pushout in \smash{$ \Topbc $}.

\begin{cnstr}[oriented pushout]
	The \category $ \Top_{\infty}$ is \textit{tensored} over the \category $ \Cat_{\infty,\updelta_0}$.
	If $ \WW$ \atopos and $C$ is a $ \updelta_0$-small \category, then the \category $ \Fun(C,\WW)$ is \atopos.
	Moreover, the functor $ \fromto{C}{\Funlowerstar(\WW,\Fun(C,\WW))}$ that carries an object $ c \in C $ to the right adjoint of the functor $ \fromto{\Fun(C,\WW)}{\WW} $ given by evaluation at $ c $ induces an equivalence of \categories
	\begin{equation*}
		\equivto{\Funlowerstar(\Fun(C,\WW),\ZZ)}{\Fun(C,\Funlowerstar(\WW,\ZZ))}
	\end{equation*}
	for any \topos $ \ZZ$.

	Let $ \WW$, $ \ZZ$, and $ \UU$ be \topoi, and let $ \plowerstar \colon\fromto{\WW}{\ZZ}$ and $ \qlowerstar \colon\fromto{\WW}{\UU}$ be geometric morphisms.
	Note that the recollement \smash{$ \ZZ\orientedcup^{\plowerstar \qupperstar}\UU$} is the oriented fiber product \smash{$ \commacat{\ZZ}{\WW}{\UU} $} formed in $ \Cat_{\infty,\updelta_1}$ with respect to the \textit{left} adjoints $ \pupperstar $ and $ \qupperstar $. 
	This \topos enjoys the following universal property: specifying a geometric morphism
	\begin{equation*}
		\omega(f,g,\sigma)_{\ast}\colon\fromto{\ZZ\orientedcup^{\plowerstar \qupperstar}\UU}{\XX}
	\end{equation*}
	is equivalent to specifying an oriented square
	\begin{equation*}
		\begin{tikzcd}
			\WW \arrow[r, "\qlowerstar" above] \arrow[d, "\plowerstar" left] & \UU \arrow[d, "\glowerstar" right] \arrow[dl, phantom, "\scriptstyle \sigma" below right, "\Longleftarrow" sloped] \\ 
			\ZZ \arrow[r, "\flowerstar" below] & \XX \period
		\end{tikzcd}
	\end{equation*}
	This universal property specifies the \topos $ \ZZ\orientedcup^{\plowerstar \qupperstar}\UU$ uniquely.
	We write\index[notation]{ZcupWU@$\ZZ\orientedcup^{\WW}\UU$} 
	\begin{equation*}
		\ZZ\orientedcup^{\WW}\UU\coloneq\ZZ\orientedcup^{\plowerstar \qupperstar}\UU \comma
	\end{equation*}
	and we call the \topos \smash{$ \ZZ\orientedcup^{\WW}\UU $} the \defn{oriented pushout}\index[terminology]{oriented pushout} of $ \plowerstar $ and $ \qlowerstar $.
	In this case, we write
	\begin{equation*}
		\ilowerstar \colon \incto{\ZZ}{\ZZ\orientedcup^{\WW}\UU} \andeq
		\jlowerstar \colon \incto{\UU}{\ZZ\orientedcup^{\WW}\UU}
	\end{equation*}
	for the closed embedding and its open complement, respectively.
\end{cnstr}


\begin{wrn}
	If $ \ZZ$, $ \UU$, and $ \WW$ are all bounded coherent, and if $ \plowerstar $ and $ \qlowerstar $ are both coherent geometric morphisms, \Cref{wrn:boundednessnotpushedout} \& \Cref{wrn:coherencenotpushedout} still apply: we generally cannot ensure that the oriented pushout \smash{$ \ZZ\orientedcup^{\phi}\UU$} is either bounded or coherent (cf. \cite[Exposé VI, \S 4]{MR50:7131}).
\end{wrn}

\begin{cnstr}[bounded coherent oriented pushout]\label{cnstr:bcorientedpushout}
	Consider an oriented square
	\begin{equation*}
		\begin{tikzcd}
			\WW \arrow[r, "\qlowerstar" above] \arrow[d, "\plowerstar" left] & \UU \arrow[d, "\glowerstar" right] \arrow[dl, phantom, "\scriptstyle \sigma" below right, "\Longleftarrow" sloped] \\ 
			\ZZ \arrow[r, "\flowerstar" below] & \XX
		\end{tikzcd}
	\end{equation*}
	where all \topoi are bounded coherent and all geometric morphisms are coherent.
	For any truncated coherent object $G\in\XX$, the object $ \omega(f,g,\sigma)^{\ast}G$ is truncated, and the objects
	\begin{equation*}
		\iupperstar \omega(f,g,\sigma)^{\ast}G\simeq \fupperstar G \andeq \jupperstar \omega(f,g,\sigma)^{\ast}G\simeq \gupperstar G
	\end{equation*}
	are each truncated coherent.
	Hence $ \omega(f,g,\sigma)_{\ast}$ factors through the bounded coherent recollement $ \ZZ\orientedcupbc^{\plowerstar \qupperstar}\UU $ (\Cref{cnstr:orientedfibboundedpretop}) in a unique manner.
	Consequently, we write\index[notation]{ZcupbcWU@$\ZZ\orientedcupbc^{\WW}\UU$}
	\begin{equation*}
		\ZZ\orientedcupbc^{\WW}\UU\coloneq \ZZ \orientedcupbc^{\plowerstar \qupperstar}\UU \comma
	\end{equation*}
	and call this \topos the \defn{bounded coherent oriented pushout}\index[terminology]{oriented pushout!bounded coherent}\index[terminology]{bounded coherent oriented pushout}.
	This is the oriented pushout that is correct in $ \Topbc $.
	Accordingly, one has an equivalence of \pretopoi
	\begin{equation*}
		(\ZZ\orientedcupbc^{\WW}\UU)\cohbdd \simeq \commacat{\ZZcohbdd}{\WWcohbdd}{\UUcohbdd} \period
	\end{equation*}

	Please observe that, by construction, in the square
	\begin{equation*}
		\begin{tikzcd}
			\WW \arrow[r, "\qlowerstar" above] \arrow[d, "\plowerstar" left] & \UU \arrow[d, "\jlowerstar" right, hooked] \arrow[dl, phantom, "\scriptstyle \sigma" below right, "\Longleftarrow" sloped] \\ 
			\ZZ \arrow[r, "\ilowerstar" below, hooked] & \ZZ\orientedcupbc^{\WW}\UU\comma
		\end{tikzcd}
	\end{equation*}
	the natural \textit{\basechange morphism}%
	\index[terminology]{basechange transformation@\basechange transformation}\index[terminology]{transformation@basechange@\basechange}
	\begin{equation*}
		\BC_{\sigma}\colon\fromto{\iupperstar \jlowerstar}{\plowerstar \qupperstar}
	\end{equation*}
	becomes an equivalence after restriction to $ \UUcohbdd$.
	A thorough study of \basechange morphisms will occupy \cref{section:BC}.
\end{cnstr}


\subsection{Internal Homs \& path \texorpdfstring{$ \infty$}{∞}-topoi}\label{subsec:internalhoms}

We now begin to study \textit{oriented fiber products} of \topoi.
Oriented fiber products have the universal property that is dual to that of oriented push\-outs.
In order to define oriented fiber products, we must identify the \textit{cotensor} of $ \Top_{\infty}$ over $ \Cat_{\infty,\updelta_0}$, or at least over $ \Pos$.
Partly in order to define oriented fiber products of \topoi now and partly to define the nerve construction for stratified \topoi later (\Cref{cnstr:nerveofstrattopoi}), we recall some facts about the internal Hom in \topoi.

The first point to be made about the internal Hom is that it doesn't always exist.

\begin{rec}
	Recall \SAG{Theorem}{21.1.6.11} that \atopos $ \WW $ is \emph{exponentiable}\index[terminology]{topos@\topos!exponentiable}\index[terminology]{exponentiable \topos} if and only if the functor $-\times \WW\colon\fromto{\Top_{\infty}}{\Top_{\infty}}$ admits a right adjoint $ \MOR(\WW,-)$.\index[notation]{HOM@$\MOR$}
	If $ \WW$ is exponentiable, then for any \topos $ \ZZ$, we have a natural equivalence
	\begin{equation*}
		\equivto{\Pt(\MOR(\WW,\ZZ))^{\op} \equivalent \Funlowerstar(\Space,\MOR(\WW,\ZZ))}{\Funlowerstar(\WW,\ZZ)}.
	\end{equation*}
	We thus call $ \MOR(\WW,\ZZ)$ the \emph{mapping \topos.}\index[terminology]{topos@\topos!mapping}\index[terminology]{mapping \topos}
	Any compactly generated \topos is exponentiable, and exponentiable \topoi admit several useful characterizations; see \SAG{Theorem}{21.1.6.12}.
\end{rec}

\begin{exm}
	Let $ S $ be spectral topological space $S$.
	Then $ \Stilde $ is compactly generated \cites[\HTTthm{Proposition}{6.5.4.4}]{HTT}[\SAGthm{Proposition}{21.1.7.8}]{SAG}, so for any \topos $ \ZZ$, there exists a mapping \topos $ \MOR(\Stilde,\ZZ)$.
	Each point $ s \in S$ induces a geometric morphism
	\begin{equation*}
		\fromto{\MOR(\Stilde,\ZZ)}{\MOR(\widetilde{s},\ZZ)\simeq\ZZ} \period
	\end{equation*}
\end{exm}

\begin{exm}\label{exm:MorPtilde}
	If $P$ is a finite poset, then the functor $ \MOR(\Ptilde,-) \colon \fromto{\Top_{\infty}}{\Top_{\infty}} $ can be identified with the unique limit-preserving endofunctor of $ \Top_{\infty}$ such that, for any small \category $C$, the natural functor
	\begin{equation*}
		\fromto{\MOR(\Ptilde,\Fun(C,\Space))}{\Fun(\Fun(P,C),\Space)}
	\end{equation*}
	is an equivalence.
	In particular, if $P$ and $Q$ are finite posets, then
	\begin{equation*}
		\MOR(\Ptilde,\widetilde{Q})\simeq\widetilde{\Fun(P,Q)} \period
	\end{equation*}
\end{exm}

\begin{dfn}[{\SAG{Definition}{21.3.2.3}}]
	For any \topos $ \XX$, we call the \topos $ \MOR(\widetilde{[1]},\XX)$ the \defn{path \topos}\index[terminology]{path!topos@\topos}\index[terminology]{topos@\topos!path} of $ \XX$.
	We write
	\begin{equation*}
		\Path{\XX} \coloneq \MOR(\widetilde{[1]},\XX) \index[notation]{Path@$ \Path{\XX} $} \period
	\end{equation*}
\end{dfn}

\begin{exm}
	As a special case of \Cref{exm:MorPtilde}, for any $ \updelta_0 $-small \category $C$, there is a natural equivalence
	\begin{equation*}
		\Path{\Fun(C,\Space)} \equivalent \Fun(\Fun([1],C),\Space) \period
	\end{equation*}
\end{exm}

We often make use of the fact that the path \topos of an $ n $-localic \topos is $ n $-localic:

\begin{lem}\label{prop:pathnlocalic}
	Let $ n \in \NNup $, and let $ \ZZ $ be an $ n $-localic \topos.
	Then the path \topos $ \Path{\ZZ} $ is $ n $-localic.
\end{lem}

\begin{proof}
	This is a special case of \SAG{Lemma}{21.1.7.3}.
\end{proof}

\subsection{Oriented fiber products}\label{subsec:OFP}

We are now ready to construct the oriented fiber product of \topoi and to relate it to the classical oriented fiber product of $ 1 $-topoi (\Cref{lem:orientedproddiscrete}).

\begin{dfn}\label{cnstr:orientedfibprod}
	If $ \flowerstar \colon \fromto{\XX}{\ZZ} $ and $ \glowerstar\colon\fromto{\YY}{\ZZ}$ are two geometric morphisms of \topoi, then the \emph{oriented fiber product}\index[terminology]{oriented fiber product!of topoi@of \topoi}\index[terminology]{topos@\topos!oriented fiber product}\index[notation]{XtimesYZ@$ \XX\orientedtimes_{\ZZ}\YY$} is the pullback
	\begin{equation*}
		\XX\orientedtimes_{\ZZ} \YY \coloneq \XX \crosslimits_{\MOR(\widetilde{\{0\}},\ZZ)} \MOR(\widetilde{[1]},\ZZ) \crosslimits_{\MOR(\widetilde{\{1\}},\ZZ)} \YY
	\end{equation*}
	in $ \Top_{\infty}$.
	We write $ \pr_{1,\ast} \colon \fromto{\XX\orientedtimes_{\ZZ}\YY}{\XX} $ and $ \pr_{2,\ast} \colon \fromto{\XX\orientedtimes_{\ZZ}\YY}{\YY} $ for the natural geometric morphisms.

	Thus a geometric morphism
	\begin{equation*}
		\psi(p,q,\sigma)_{\ast} \colon \fromto{\WW}{\XX\orientedtimes_{\ZZ}\YY}
	\end{equation*}
	determines and is determined by a square \eqref{square:laxsquare}.
	This universal property specifies the \topos $ \XX\orientedtimes_{\ZZ}\YY$ uniquely.
\end{dfn}

\begin{wrn}
	Please note that the oriented fiber product in $ \Top_{\infty} $ is \emph{not} the oriented/lax pullback in $ \Cat_{\infty,\updelta_1}$; we will therefore take pains to express clearly where the oriented fiber product is taking place.

	Additionally, in this paper, the symbol `$ \orientedtimes$' is only ever used for the oriented fiber product in $ \Top_{\infty}$;
	conversely, we only use the notation $ \commacat{X}{Z}{Y}$ for the oriented fiber product in some $ \Cat_{\infty,\delta}$ (see \Cref{nul:orientedfpinCat}).
\end{wrn}

\begin{nul}\label{nul:orientedcommuteswithlimits}
	Please observe that since the exponential functor $ \Path{-} \colon \fromto{\Top_{\infty}}{\Top_{\infty}} $ is a right adjoint and limits in $ \Fun(\horn{2}{2},\Top_{\infty}) $ are computed pointwise, the functor
	\begin{equation*}
		\fromto{\Fun(\horn{2}{2},\Top_{\infty})}{\Top_{\infty}}
	\end{equation*}
	given by the formation of the oriented fiber product preserves limits.
\end{nul}

\begin{exm} 
	When $ \ZZ = \Space $ is the terminal \topos, the oriented fiber product reduces to the product in $ \Top_{\infty}$:
	\begin{equation*}
		\XX\orientedtimes_{\Space}\YY\simeq\XX\times\YY\period
	\end{equation*}
\end{exm}

\begin{nul}\label{nul:projectionsasglobalsec}
	Let $ \flowerstar \colon \fromto{\XX}{\ZZ} $ and $ \glowerstar \colon \fromto{\YY}{\ZZ} $ be geometric morphisms of \topoi. Then under the identifications $ \XX \equivalent \orientedpull{\XX}{\Space}{\Space} $ and $ \YY \equivalent \orientedpull{\Space}{\Space}{\YY} $, the projections
	\begin{equation*}
		\pr_{1,\ast} \colon \fromto{\orientedpull{\XX}{\ZZ}{\YY}}{\XX} \andeq \pr_{2,\ast} \colon \fromto{\orientedpull{\XX}{\ZZ}{\YY}}{\YY}
	\end{equation*}
	are equivalent to $ \orientedpull{\id_{\XX}}{\Gammaup_{\ZZ,\ast}}{\Gammaup_{\YY,\ast}} $ and $ \orientedpull{\Gammaup_{\XX,\ast}}{\Gammaup_{\ZZ,\ast}}{\id_{\YY}} $, respectively (\Cref{ntn:globalsecpoints}).
\end{nul}

\begin{exm}
	For any \topos $ \XX$, the oriented fiber product $ \XX\orientedtimes_{\XX}\XX$ is canonically identified with the path \topos $ \Path{\XX} $.
\end{exm}

The next thing to notice is that the points of an oriented fiber product of \topoi are the oriented fiber product of the corresponding \categories of points.

\begin{nul}\label{nul:Funstarpreservesorientedprod}
	For any \topos $ \EE $, the functor
	\begin{equation*}
		\Funlowerstar(\EE,-)^{\op} \colon \fromto{\Top_{\infty}}{\Cat_{\infty,\updelta_1}}
	\end{equation*}
	commutes with cotensors with $ \Cat_{\infty,\updelta_1} $ (in particular, cotensoring with $ [1] $) and pullbacks of \topoi.
	Hence $ \Funlowerstar(\EE,-)^{\op} $ carries oriented fiber products in $ \Top_{\infty} $ to oriented fiber products in $ \Cat_{\infty,\updelta_1} $.

	Specalising to the case $ \EE = \Space $, we deduce the following.
\end{nul}

\begin{lem}\label{lem:Ptpreservesorientedprod}
	The functor $ \Pt \colon \fromto{\Top_{\infty}}{\Cat_{\infty,\updelta_1}} $ carries oriented fiber products in $ \Top_{\infty} $ to oriented fiber products in $ \Cat_{\infty,\updelta_1} $.
	That is, if $ \flowerstar \colon \fromto{\XX}{\ZZ} $ and $ \glowerstar \colon \fromto{\YY}{\ZZ} $ are geometric morphisms of \topoi, then the natural functor
	\begin{equation*}
		\Pt(\XX\orientedtimes_{\ZZ}\YY)\to\commacat{\Pt(\XX)}{\Pt(\ZZ)}{\Pt(\YY)}
	\end{equation*}
	is an equivalence.
\end{lem}

\begin{exm}
	There is a canonical geometric morphism
	\begin{equation*}
		\psi(\pr_1,\pr_2,\id)_{\ast}\colon\fromto{\XX\times_{\ZZ}\YY}{\XX\orientedtimes_{\ZZ}\YY}\period
	\end{equation*}
\end{exm}

\begin{exm}\label{exm:evanescent}
	The \topos $ \XX\orientedtimes_{\ZZ}\ZZ$ is called the \emph{evanescent}%
	\index[terminology]{topos@\topos!evanescent}\index[terminology]{evanescent \topos}
	(or \emph{vanishing})%
	\index[terminology]{topos@\topos!vanishing}\index[terminology]{vanishing \topos}
	\topos of $ \flowerstar $, and the natural functor
	\begin{equation*}
		\near_{f,\ast}\coloneq\psi(\id_{\XX},f,\id)_{\ast}\colon\fromto{\XX}{\XX\orientedtimes_{\ZZ}\ZZ}
	\end{equation*}
	is called the \defn{nearby cycles functor}.\index[terminology]{nearby cycles}\index[terminology]{functor!nearby cycles}
	Dually, the \topos $ \ZZ\orientedtimes_{\ZZ}\YY$ is called the \defn{coëvanescent}%
	\emph{coëvanescent}\index[terminology]{topos@\topos!coevanescent@coëvanescent}\index[terminology]{coevanescent@coëvanescent \topos} (or \defn{covanishing})%
	\emph{covanishing}\index[terminology]{topos@\topos!covanishing}\index[terminology]{covanishing \topos}
	\topos of $ \glowerstar$, and the natural functor
	\begin{equation*}
		\conearlowerstar^{g}\coloneq\psi(g,\id_{\YY},\id)_{\ast}\colon\fromto{\YY}{\ZZ\orientedtimes_{\ZZ}\YY}
	\end{equation*}
	is called the \defn{conearby cycles functor}.\index[terminology]{conearby cycles}\index[terminology]{functor!conearby cycles}
	 
	The oriented fiber product can be decomposed into fiber products in $ \Top_{\infty}$ involving the evanescent and coëvanescent \topoi as follows: we have equivalences
	\begin{equation*}
		\XX\orientedtimes_{\ZZ}\YY\simeq(\XX\orientedtimes_{\ZZ}\ZZ)\times_{\ZZ}\YY \andeq \XX\orientedtimes_{\ZZ}\YY\simeq\XX\times_{\ZZ}(\ZZ\orientedtimes_{\ZZ}\YY) \comma
	\end{equation*}
	and, more symmetrically,
	\begin{equation*}
		\XX \orientedtimes_{\ZZ} \YY \simeq(\XX\orientedtimes_{\ZZ}\ZZ) \crosslimits_{\Path{\ZZ}}(\ZZ\orientedtimes_{\ZZ}\YY) \period
	\end{equation*}
\end{exm}

\begin{exm}\label{exm:oriendedproduptoloc}
	Keep the notations of \Cref{cnstr:orientedfibprod}, and let $ \plowerstar \colon \incto{\ZZ}{\ZZ'}$ be a fully faithful geometric morphism.
	Then $ \plowerstar $ induces an equivalence of \topoi
	\begin{equation*}
		\equivto{\XX\orientedtimes_{\ZZ}\YY}{\XX\orientedtimes_{\ZZ'}\YY} \period
	\end{equation*}
	To see this, simply note that $ \XX\orientedtimes_{\ZZ}\YY $ and $ \XX\orientedtimes_{\ZZ'}\YY $ have the same universal property since $ \plowerstar $ is fully faithful.
	Hence for the purpose of computing oriented fiber products, we may assume that $ \ZZ$ is a presheaf \topos.
\end{exm}

We're mostly interested in working with $ 1 $-localic \topoi or more generally bounded \topoi.
The following lemma says that taking oriented fiber products doesn't take us out of these subcategories of all \topoi.

\begin{lem}\label{cor:orientedprodnlocalic}
	Let $ \flowerstar \colon \fromto{\XX}{\ZZ} $ and $ \glowerstar \colon \fromto{\YY}{\ZZ} $ be geometric morphisms of \topoi.
	If $ \XX$, $ \YY$, and $ \ZZ$ are $n$-localic (\Cref{cnstr:localictopoi}), so is the oriented fiber product $ \orientedpull{\XX}{\ZZ}{\YY} $.
	Moreover, if $ \XX$, $ \YY$, and $ \ZZ$ are bounded (\Cref{cnstr:boundedtopoi}), so is the oriented fiber product $ \orientedpull{\XX}{\ZZ}{\YY} $.
\end{lem}

\begin{proof}
	For the first assertion, by \Cref{prop:pathnlocalic} the oriented fiber product is a limit of $ n $-localic \topoi, hence $n$-localic.
	The second claim follows from the fact that formation of the oriented fiber product preserves limits \Cref{nul:orientedcommuteswithlimits}.
\end{proof}

\noindent The $ 1 $-toposic oriented fiber product \cites{Illusie:vanishingcycles}{MR3309086}{MR726426}{MR1731050}{MR2249998} is related to the oriented fiber product of corresponding $ 1 $-localic \topoi via the following easy result.

\begin{lem}\label{lem:orientedproddiscrete}
	Let $ \flowerstar \colon \fromto{\XX}{\ZZ} $ and $ \glowerstar \colon \fromto{\YY}{\ZZ} $ be geometric morphisms of $ 1 $-topoi, and write $ \XX' $, $ \YY' $, and $ \ZZ' $ for the corresponding $ 1 $-localic \topoi associated to $ \XX $, $ \YY $, and $ \ZZ $, respectively.
	Then the oriented fiber product of $ 1 $-topoi $ \orientedpull{\XX}{\ZZ}{\YY} $ is canonically equivalent to the $ 1 $-topos of $ 0 $-truncated objects of $ \orientedpull{\XX'}{\ZZ'}{\YY'} $.
\end{lem}

\begin{proof}
	As a consequence of \Cref{cor:orientedprodnlocalic}, the equivalence of \categories
	\begin{equation*}
		(-)_{\leq 0} \colon \equivto{\Top_{\infty}^1}{\Top_1}
	\end{equation*}
	from $ 1 $-localic \topoi to $ 1 $-topoi (\Cref{cnstr:localictopoi}) respects cotensors by the $ 1 $-category $ [1] $.
	The claim now follows from the definitions of the oriented fiber product in the setting of \topoi and $ 1 $-topoi. 
\end{proof}

\subsection{Generating \texorpdfstring{$ \infty$}{∞}-sites for oriented fiber products}\label{subsec:generatingsites}

We now describe a generating \site for the oriented fiber product in the setting of sheaf \topoi.
This description is adapted from Deligne's site-theoretic description \cites[Exposé XI, \S1]{MR3309086}[3.1.3]{MR726426}.
We employ this generating \site to deduce that the oriented fiber product of bounded coherent \topoi and coherent geometric morphisms is bounded coherent (\Cref{lem:orientedcoherent}).

We begin with oriented fiber products of presheaf \topoi.
To do this, we first need to identify the oriented pushout in the \category of \categories with finite limits and left exact functors between them (\Cref{lem:Wsareorientedpushouts}).
The notation below will follow us throughout this section.

\begin{cnstr}\label{exm:presheaforientedfib} 
	Let $X$, $Y$, and $Z$ be $ \updelta_0$-small \categories with finite limits, and let $ \fupperstar\colon\fromto{Z}{X}$ and $ \gupperstar\colon\fromto{Z}{Y}$ be left exact functors.

	Appealing to the straightening/unstraightening equivalence, let $ m \colon \fromto{\Mup(f,g)}{\horn{2}{2}}$ denote the cartesian fibration classified by the span 
	\begin{equation*}
		\begin{tikzcd}
			X & Z \arrow[l, "\fupperstar"'] \arrow[r, "\gupperstar"] & Y \period
		\end{tikzcd}
	\end{equation*}
	Note that the fibers of $ m $ over the vertices $0$, $1$, and $2$ are $X$, $Y$, and $Z$, respectively.
	Write
	\begin{align*}
		\overleftrightW(f,g) &\coloneq \Fun_{\horn{2}{2}}(\horn{2}{2}, \Mup(f,g)) \\
		&\simeq \Fun(\{0 < 2\},X) \crosslimits_{\Fun(\{2\},X)} Z \crosslimits_{\Fun(\{2\},Y)} \Fun(\{1 < 2\},Y)
	\end{align*}
	for the \category of sections of $m$.
	Let us write $ \Kup_Y$ for the set of those morphisms $ \phi \colon \fromto{[1]\times\horn{2}{2}}{\Mup(f,g)} $ in $ \overleftrightW(f,g) $ of the form
	\begin{equation}\label{diag:KYKX}
		\begin{tikzcd}
			v_X \arrow[d, "\phi_X" left] \arrow[r] & v_Z \arrow[d, "\phi_Z" right] & \arrow[l] v_Y \arrow[d, "\phi_Y" right]\\ 
			u_X \arrow[r] & u_Z & \arrow[l] u_Y \comma
		\end{tikzcd}
	\end{equation}
	in which $ \phi_X$ is an equivalence, and the diagram \eqref{diag:KYKX} exhibits $ \phi_Y$ as the pullback of $ \gupperstar\phi_Z$.
	Dually, let us write $ \Kup_X$ for those morphisms $ \phi$ in which $ \phi_Y$ is an equivalence, and the diagram \eqref{diag:KYKX} exhibits $ \phi_X$ as the pullback of $ \fupperstar\phi_Z$. 

	We now define two new \categories by formally inverting these morphisms \Cref{nul:conventionshighercats}:
	\begin{equation*}
		\overleftW(f,g) \coloneq \Kup_Y^{-1} \overleftrightW(f,g) \andeq \Wup(f,g) \coloneq \Kup_X^{-1}\overleftW(f,g) \period
	\end{equation*}
\end{cnstr}

\begin{nul}
	The \category $ \overleftrightW(f,g)$ admits finite limits, which are computed pointwise. The sets $ \Kup_Y$ and $ \Kup_X$ are stable under composition and pullback.
	It follows that the sets $ \Kup_Y$ and $ \Kup_X$ each give rise to right calculi of fractions on $ \overleftrightW(f,g)$ in the sense of Cisinski's book \cite[Theorem 7.2.16]{MR3931682}.

	Consequently, the mapping spaces in $ \overleftW(f,g)$ admit a very simple description: for any objects $u,v\in\overleftrightW(f,g)$, write
	\begin{equation*}
		\Sup_Y(u,v)\subseteq\overleftrightW(f,g)_{/u} \, \crosslimits_{\overleftrightW(f,g)} \, \overleftrightW(f,g)_{/v}
	\end{equation*}
	for the full subcategory spanned by those diagrams $u\ot w\to v$ in which the morphism $u\ot w$ lies in $ \Kup_Y$.
	Then one can compute the mapping spaces in $ \overleftW(f,g)$ as the classifying \groupoids of these \categories:
	\begin{equation*}
		\Map_{\overleftW(f,g)}(u,v) \simeq \invert \Sup_Y(u,v) \period
	\end{equation*}

	Furthermore, the \categories $ \overleftW(f,g)$ and $ \Wup(f,g) $ admit finite limits, and the localizations
	\begin{equation*}
		\fromto{\overleftrightW(f,g)}{\overleftW(f,g)} \andeq \fromto{\overleftW(f,g)}{W}
	\end{equation*}
	each preserve finite limits \cite[Corollary 7.1.16 \& Theorem 7.2.25]{MR3931682}.
\end{nul}

We now work toward showing that the \category $ \overleftW(f,g) $ is the oriented pushout of the span $ X \ot Z \to Y $ in the \category of \categories with finite limits.
This immediately implies that the \topos of presheaves on $ \overleftW(f,g) $ is the oriented fiber product
\begin{equation*}
	\PSh(X) \underset{\PSh(Z)}{\orientedtimes} \PSh(Y)  \period
\end{equation*}

\begin{cnstr}\label{constr:OFPpresheaf}
	Keep the notations of \Cref{exm:presheaforientedfib}.
	Write
	\begin{align*}
		\pupperstar \colon\fromto{X}{\overleftrightW(f,g)} &\andeq \qupperstar \colon\fromto{Y}{\overleftrightW(f,g)} \\
	\intertext{for the left exact functors defined by the assignments}
		\goesto{x}{[x\to 1\ot 1]} &\andeq \goesto{y}{[1\to 1\ot y]} \period
	\end{align*}
	We also regard the left exact functors $ \pupperstar $ and $ \qupperstar $ as landing in $ \overleftW(f,g)$ and $ \Wup(f,g) $ by composing with the relevant localizations.

	Write $ \supperstar \colon \fromto{Z}{\overleftrightW(f,g)} $ for the section of of the natural projection $ \fromto{\overleftrightW(f,g)}{Z} $ defined be sending an object $ z \in Z $ to the cartesian section $ \fupperstar(z)\to z\ot \gupperstar(z)$.
	Define natural transformations $ \theta $ and $ \xi $ fitting into a span
	\begin{equation*}
		\begin{tikzcd}
			\pupperstar \fupperstar & \arrow[l, "\theta" above] \supperstar \arrow[r, "\xi" above] & \qupperstar \gupperstar
		\end{tikzcd}
	\end{equation*}
	as follows: for any $z\in Z$, the components $ \theta(z)$ and $ \xi(z)$ are given by the diagram
	\begin{equation*}
		\begin{tikzcd}
			\fupperstar(z) \arrow[r] & 1 & \arrow[l] 1 \\ 
			\fupperstar(z) \arrow[u, equals] \arrow[d, "!" left] \arrow[r] & z \arrow[u, "!" right] \arrow[d, "!" right] & \arrow[l] \gupperstar(z) \arrow[u, "!" right] \arrow[d, equals]\\
			1 \arrow[r] & 1 & \arrow[l] \gupperstar(z) \period
		\end{tikzcd}
	\end{equation*}
	In particular, note that $ \theta(z)\in \Kup_Y$ and $ \xi(z)\in \Kup_X$.
	Hence the natural transformation $ \theta $ becomes an equivalence after postcomposition with the localization $ \fromto{\overleftrightW(f,g)}{\overleftW(f,g)}$ at $ \Kup_Y $.
	Write
	\begin{equation*}
		\sigmahat \coloneq \xi\theta^{-1} \colon \fromto{\pupperstar \fupperstar}{\qupperstar \gupperstar}
	\end{equation*}
	for the resulting natural transformation of functors $ \fromto{Z}{\overleftrightW(f,g)} $.
	Note that the natural transformation $ \sigmahat $ becomes an equivalence upon postcomposition with the localization to $ \fromto{\overleftrightW(f,g)}{\Wup(f,g)} $ at $ \Kup_X $. 
	
	In other words, these data specify the following three diagrams of \categories with finite limits and left exact functors between them:
	\begin{equation}\label{square:Wdoublehead}
		\begin{tikzcd}[sep=3.5em]
			Z \arrow[r, "\gupperstar" above] \arrow[dr, "\supperstar" {description, near start}] \arrow[d, "\fupperstar" left] & Y \arrow[d, "\qupperstar" right] \\
			X \arrow[r, "\pupperstar" below] \arrow[ur, phantom, "\scriptstyle \xi" {below right, near end}, "\Longrightarrow" {sloped, near end}, "\scriptstyle \theta" {below right, near start}, "\Longleftarrow"{sloped, near start}] & \overleftrightW(f,g) \comma
		\end{tikzcd}
	\end{equation}
	\begin{equation}\label{square:Wleftarrow}
		\begin{tikzcd}[sep=2.5em]
			Z \arrow[r, "\gupperstar" above] \arrow[d, "\fupperstar" left] & Y \arrow[d, "\qupperstar" right] \\ 
			X \arrow[r, "\pupperstar" below] \arrow[ur, phantom, "\scriptstyle \sigmahat" below right, "\Longrightarrow" sloped] & \overleftW(f,g) \comma
		\end{tikzcd}
	\end{equation}
	\begin{equation}\label{square:W}
		\begin{tikzcd}[sep=2.5em]
			Z \arrow[r, "\gupperstar" above] \arrow[d, "\fupperstar" left] & Y \arrow[d, "\qupperstar" right] \\ 
			X \arrow[r, "\pupperstar" below] & \Wup(f,g) \period
		\end{tikzcd}
	\end{equation}
\end{cnstr}

\begin{lem}\label{lem:Wsareorientedpushouts}
	Keep the notations of \Cref{constr:OFPpresheaf}.
	The squares \eqref{square:Wdoublehead}, \eqref{square:Wleftarrow}, and \eqref{square:W} exhibit $ \overleftrightW(f,g)$, $ \overleftW(f,g)$, and $ \Wup(f,g) $ as universal among \categories with finite limits completing these diagrams. 
	More precisely, if $E$ is \acategory with finite limits, then the induced functors
	\begin{equation*}
		\Funlex(\overleftrightW(f,g), E) \to \Funlex(X,E) \crosslimits_{\Funlex(Z,E)} \Fun(\horn{2}{2}, \Funlex(Z,E)) \crosslimits_{\Funlex(Z,E)} \Funlex(Y,E) \comma
	\end{equation*}
	\begin{equation*}
		\Funlex(\overleftW(f,g), E) \to \Funlex(X,E) \crosslimits_{\Funlex(Z,E)} \Fun([1], \Funlex(Z,E)) \crosslimits_{\Funlex(Z,E)} \Funlex(Y,E) \comma
	\end{equation*}
	and
	\begin{equation*}
		\Funlex(\Wup(f,g), E) \to \Funlex(X,E) \crosslimits_{\Funlex(Z,E)} \Funlex(Y,E) 
	\end{equation*}
	are equivalences.
\end{lem}

\begin{proof}
	The last two equivalences follow from the first by the universal properties of the localizations of $ \overleftrightW(f,g) $ at $ \Kup_Y $ and $ \overleftW(f,g) $ at $ \Kup_X$;
	hence we focus on the first equivalence.

	Let
	\begin{equation}\label{diag:leftrightortient}
		\begin{tikzcd}[sep=3.5em]
			Z \arrow[r, "\gupperstar" above] \arrow[dr, "\supperstar" {description, near start}] \arrow[d, "\fupperstar" left] & Y \arrow[d, "q'" right] \\
			X \arrow[r, "p'" below] \arrow[ur, phantom, "\scriptstyle \xi'" {below right, near end}, "\Longrightarrow" {sloped, near end}, "\scriptstyle \theta'" {below right, near start}, "\Longleftarrow"{sloped, near start}] & E
		\end{tikzcd}
	\end{equation}
	be a diagram of \categories with finite limits and left exact functors between them.
	Define a functor $ F \colon \overleftrightW(f,g) \to E $ by the formula
	\begin{equation*}
		F(v_X \ot v_Z \to v_Y) \coloneq \supperstar(v_Z) \crosslimits_{(p'\fupperstar(v_Z) \times q'\gupperstar(v_Z))} (p'(v_X) \times q'(v_Y)) \period
	\end{equation*}
	The functor  
	\begin{equation*}
		\Funlex(X,E) \crosslimits_{\Funlex(Z,E)} \Fun(\horn{2}{2}, \Funlex(Z,E)) \crosslimits_{\Funlex(Z,E)} \Funlex(Y,E) \to \Funlex(\overleftrightW(f,g), E) \comma
	\end{equation*}
	given by sending the object defined by the diagram \eqref{diag:leftrightortient} to the functor $ F $ is then easily seen to be inverse to the induced functor
	\begin{equation*}
		\Funlex(\overleftrightW(f,g), E) \to \Funlex(X,E) \crosslimits_{\Funlex(Z,E)} \Fun(\horn{2}{2}, \Funlex(Z,E)) \crosslimits_{\Funlex(Z,E)} \Funlex(Y,E) \period \qedhere
	\end{equation*}
\end{proof}

\begin{nul}
	Keep the notations of \Cref{constr:OFPpresheaf}.
	Write $ \sigmahat^{\ast} \colon \fromto{\pupperstar\fupperstar}{\qupperstar\gupperstar} $ for the natural transformation of functors $ \fromto{\PSh(Z)}{\PSh(\overleftW(f,g))} $ given by extending the natural transformation $ \sigmahat $ to presheaves.
	We write $ \sigma $ for the natural transformation adjoint to $ \sigmahat^{\ast} $, so that $ \sigma $ fits into an oriented square
	\begin{equation}\label{square:Wfg}
		\begin{tikzcd}
			\PSh(\overleftW(f,g)) \arrow[r, "\qlowerstar" above] \arrow[d, "\plowerstar" left] & \PSh(Y) \arrow[d, "\glowerstar" right] \arrow[dl, phantom, "\scriptstyle \sigma" below right, "\Longleftarrow" sloped] \\ 
			\PSh(X) \arrow[r, "\flowerstar" below] & \PSh(Z) \period
		\end{tikzcd}
	\end{equation}
	We write
	\begin{equation*}
		\clowerstar \colon \PSh(\overleftW(f,g)) \to \PSh(X) \underset{\PSh(Z)}{\orientedtimes} \PSh(Y) 
	\end{equation*}
	for the geometric morphism specified by the square \eqref{square:Wfg}.
\end{nul}

The following is now immediate from \Cref{lem:Wsareorientedpushouts} and the universal property of the \category of presheaves.

\begin{lem}\label{lem:PShWisorientedpullback}
	With the notations of \Cref{constr:OFPpresheaf}, the geometric morphism 
	\begin{equation*}
		\clowerstar \colon \PSh(\overleftW(f,g)) \to \PSh(X) \underset{\PSh(Z)}{\orientedtimes} \PSh(Y)  \period
	\end{equation*}
	is an equivalence.
\end{lem}

\begin{nul}
	In the same manner, we obtain an identification of the fiber product of presheaf \topoi: there is a natural equivalence of \topoi
	\begin{equation*}
		\equivto{\PSh(\Wup(f,g))}{\PSh(X) \crosslimits_{\PSh(Z)} \PSh(Y)} \period
	\end{equation*}
\end{nul}

Now we explain how to introduce a topology into \Cref{constr:OFPpresheaf} to give a generating \site for the oriented fiber product of sheaf \topoi.

\begin{cnstr}[{topology on $ \overleftW(f,g) $}]\label{exm:sheaforientedfib}
	Let $(X,\tau_X)$, $(Y,\tau_Y)$, and $(Z,\tau_Z)$ be $ \updelta_0$-small \sites with finite limits (\Cref{dfn:finitaryinftysite}).
	Write
	\begin{equation*}
		\XX \coloneq \Sh_{\tau_X}(X) \comma \qquad \YY \coloneq \Sh_{\tau_Y}(Y) \comma \andeq \ZZ \coloneq \Sh_{\tau_Z}(Z) \period
	\end{equation*}
	Let $ \fupperstar \colon \fromto{Z}{X} $ and $ \gupperstar\colon\fromto{Z}{Y} $ be left exact morphisms of \sites, so that $ \fupperstar $ and $ \gupperstar $ induce geometric morphisms
	\begin{equation*}
		\flowerstar \colon\fromto{\XX}{\ZZ} \andeq \glowerstar\colon\fromto{\YY}{\ZZ} \period 
	\end{equation*}

	Let $ \overleftarrow{\tau}$ denote the topology on the \category $ \overleftW(f,g) $ (\Cref{exm:presheaforientedfib}) generated by the families $ \{\phi_i\colon\fromto{v_i}{u}\}_{i \in I} $, such that for each $i\in I$, the morphism $ \phi_i$ is the image of a morphism of $ \overleftrightW(f,g)$ of the form
	\begin{equation*}
		\begin{tikzcd}
			v_{i,X} \arrow[d, "\phi_{i,X}" left] \arrow[r] & v_{i,Z} \arrow[d, "\phi_{i,Z}" right] & \arrow[l] v_{i,Y} \arrow[d, "\phi_{i,Y}" right]\\ 
			u_X \arrow[r] & u_Z & \arrow[l] u_Y \comma
		\end{tikzcd}
	\end{equation*}
	such that one of the following conditions holds:
	\begin{itemize}
		\item The family $ \{\phi_{i,X}\colon\fromto{v_{i,X}}{u_X}\}_{i \in I} $ generates a $ \tau_X$-covering sieve, and for all $i\in I$, the morphisms $ \phi_{i,Z}$ and $ \phi_{i,Y}$ are equivalences.

		\item The family $ \{\phi_{i,Y}\colon\fromto{v_{i,Y}}{u_Y}\}_{i \in I} $ generates a $ \tau_Y$-covering sieve, and for all $i\in I$, the morphisms $ \phi_{i,Z}$ and $ \phi_{i,X}$ are equivalences.
	\end{itemize}
\end{cnstr}

The following is now a consequence of \Cref{lem:PShWisorientedpullback} along with the universal property of localizations:

\begin{lem}
	Keep the notations of \Cref{exm:sheaforientedfib}.
	Then there is a natural equivalence of \topoi
	\begin{equation*}
	    \equivto{\Sh_{\overleftarrow{\tau}}(\overleftW(f,g))}{\XX \orientedtimes_{\ZZ} \YY} \period
	\end{equation*}
\end{lem}

\begin{nul}
	Observe that the topology $ \tau_Z$ is irrelevant here, as we should expect, since
	\begin{equation*}
		\XX\orientedtimes_{\ZZ}\YY\simeq\XX\orientedtimes_{\PSh(Z)}\YY
	\end{equation*}
	(\Cref{exm:oriendedproduptoloc}).
\end{nul}

\begin{nul}
	In the same vein, the topology $ \tau$ on $ \Wup(f,g) $ generated by these same families produces the usual (unoriented) fiber product of \topoi: there is a natural equivalence of \topoi
	\begin{equation*}
		\equivto{\Sh_{\tau}(\Wup(f,g))}{\XX \times_{\ZZ} \YY} \period
	\end{equation*}
\end{nul}

\begin{nul}
	If all of the topologies $ \tau_X$, $ \tau_Y $, and $ \tau_Z $ are finitary, then the \sites $ (\overleftW(f,g),\overleftarrow{\tau}) $ and $ (\Wup(f,g),\tau) $ are finitary.
	This lets us deduce that the oriented fiber product preserves coherence properties.
\end{nul}

We conclude this section with the fundamental coherence results for oriented fiber products that we'll need.
 
\begin{lem}\label{lem:orientedcoherent}
	Keep the notations of \Cref{exm:sheaforientedfib}.
	If the topologies $ \tau_{X} $, $ \tau_{Y} $, and $ \tau_{Z} $ are all finitary, then:
	\begin{enumerate}[{\upshape (\ref*{lem:orientedcoherent}.1)}]
		\item\label{lem:orientedcoherent.1} The oriented fiber product $ \orientedpull{\XX}{\ZZ}{\YY} $ is coherent and locally coherent, and the projections
		\begin{equation*}
		 \pr_{1,\ast} \colon \fromto{\orientedpull{\XX}{\ZZ}{\YY}}{\XX} \andeq \pr_{2,\ast} \colon \fromto{\orientedpull{\XX}{\ZZ}{\YY}}{\YY}
		\end{equation*}
		are coherent geometric morphisms.

		\item\label{lem:orientedcoherent.2} The pullback $ \XX \cross_{\ZZ} \YY$ is coherent and locally coherent, and the projections
		\begin{equation*}
			\pr_{1,\ast} \colon \fromto{\XX \cross_{\ZZ} \YY}{\XX} \andeq \pr_{2,\ast} \colon \fromto{\XX \cross_{\ZZ} \YY}{\YY}
		\end{equation*}
		are coherent geometric morphisms.
	\end{enumerate}
\end{lem}

\begin{proof}
	Since the topology $ \overleftarrow{\tau} $ is finitary, \Cref{prop:SAG.A.3.1.3}=\allowbreak\SAG{Proposition}{A.3.1.3} ensures that the \topoi $ \orientedpull{\XX}{\ZZ}{\YY} $ and $ \XX\times_{\ZZ}\YY$ are coherent and locally coherent.
	Since $ \pr_{1,\ast} $ and $ \pr_{2,\ast} $ are induced by the morphisms of finitary \sites
	\begin{equation*}
		\fromto{(X,\tau_X)}{(\overleftW(f,g),\overleftarrow{\tau})} \andeq \fromto{(Y,\tau_Y)}{(\overleftW(f,g),\overleftarrow{\tau})} \comma
	\end{equation*}
	the remainder of \enumref{lem:orientedcoherent}{1} follows from \Cref{cor:morsitescoherent}.
	The proof of \enumref{lem:orientedcoherent}{2} is the same as the proof of \enumref{lem:orientedcoherent}{1}, replacing the finitary \site $ (\overleftW(f,g),\overleftarrow{\tau}) $ by the finitary \site $ (\Wup(f,g),\tau) $.
\end{proof}

Conceptual Completeness now implies that the oriented fiber product of bounded coherent \topoi is \textit{controlled} by its \category of points in the following sense.

\begin{prp}
	An oriented square 
	\begin{equation*}
		\begin{tikzcd}
			\WW \arrow[r, "\qlowerstar" above] \arrow[d, "\plowerstar" left] & \YY \arrow[d, "\glowerstar" right] \arrow[dl, phantom, "\scriptstyle \sigma" below right, "\Longleftarrow" sloped] \\ 
			\XX \arrow[r, "\flowerstar" below] & \ZZ \comma
		\end{tikzcd}
	\end{equation*}
	of bounded coherent \topoi and coherent geometric morphisms is an oriented fiber product square if and only if the induced oriented square
	\begin{equation*}
		\begin{tikzcd}
			\Pt(\WW) \arrow[r, "\qlowerstar" above] \arrow[d, "\plowerstar" left] & \Pt(\YY) \arrow[d, "\glowerstar" right] \arrow[dl, phantom, "\scriptstyle \sigma" below right, "\Longrightarrow" sloped] \\ 
			\Pt(\XX) \arrow[r, "\flowerstar" below] & \Pt(\ZZ) \comma
		\end{tikzcd}
	\end{equation*}
	in $ \Cat_{\infty,\updelta_1} $ exhibits $ \Pt(\WW) $ as the oriented fiber product $ \commacat{\Pt(\XX)}{\Pt(\ZZ)}{\Pt(\YY)} $ \Cref{nul:orientedfpinCat}.
\end{prp}

\begin{proof} 
	This follows from Conceptual Completeness (\Cref{thm:conceptualcompleteness}=\allowbreak\SAG{Theorem}{A.9.0.6}), along with the fact that the functor $ \Pt \colon \fromto{\Top_{\infty}}{\Cat_{\infty,\updelta_1}} $ preserves oriented fiber product squares (\Cref{lem:Ptpreservesorientedprod}).
\end{proof}


\subsection[Compatibility of oriented fiber products \& étale geometric morphisms]{Compatibility of oriented fiber products \& étale geometric \\ morphisms}\label{subsec:etalecompat}

The goal of this section is to prove that oriented fiber products are compatible with étale geometric morphisms in an appropriate sense (\Cref{prop:orientedslice}).
This compatibility is a key technical ingredient in the proof of our \basechange result for oriented fiber products in \Cref{section:BC} (\Cref{thm:BCfororientedfibs}).
Inspired by Illusie's discussion \cite[Exposé XI, 1.10(b)]{MR3309086}, to prove the relevant compatibility, we decompose the oriented fiber product $ \orientedpull{\XX}{\ZZ}{\YY} $ as the iterated pullback
\begin{equation*}
	\orientedpull{\XX}{\ZZ}{\YY} \equivalent \XX \cross_{\ZZ} \Path{\ZZ} \cross_{\ZZ} \YY 
\end{equation*}
and treat the compatibility with pullbacks and path \topoi separately.

We begin with what must be a standard fact about the compatibility of ordinary pullbacks and étale geometric morphisms (\Cref{lem:slicepullback}).
However, we could not locate the following in the existing literature.

\begin{ntn}\label{ntn:pullbackslice}
	Let $ \flowerstar \colon \fromto{\XX}{\ZZ} $ and $ \glowerstar \colon \fromto{\YY}{\ZZ} $ be geometric morphisms of \topoi, and suppose we are given objects $ X \in \XX $, $ Y \in \YY $, and $ Z \in \ZZ $, along with morphisms $ \phi \colon \fromto{X}{\fupperstar(Z)} $ and $ \psi \colon \fromto{Y}{\gupperstar(Z)} $.
	We write\index[notation]{XtimesZY@$X\cross_ZY$} 
	\begin{equation*}
		X \cross_Z Y \coloneq \prupperstar_1(X) \crosslimits_{\prupperstar_1\fupperstar(Z)} \prupperstar_2(Y)  \in \XX \cross_{\ZZ} \YY
	\end{equation*}
	for the pullback of $ \prupperstar_1(X) $ and $ \prupperstar_2(Y) $ over $ \prupperstar_1\fupperstar(Z) \equivalent \prupperstar_2\gupperstar(Z) $ formed in the (unoriented) pullback \topos $ \XX \cross_{\ZZ} \YY $.
\end{ntn}

\begin{lem}\label{lem:slicepullback}
	Keep  the notations of \Cref{ntn:pullbackslice}.
	Then the natural geometric morphism $ \plowerstar \colon \fromto{\XX_{/X} \cross_{\ZZ_{/Z}} \YY_{/Y}}{\XX \cross_{\ZZ} \YY} $ is étale and $ \plowershriek(1) \equivalent X \cross_Z Y $.
\end{lem}

\begin{proof}
	First note that the commutative square 
	\begin{equation*}
      \begin{tikzcd}
      	(\XX \cross_{\ZZ} \YY)_{/(X \cross_Z Y)} \arrow[r] \arrow[d] & (\XX \cross_{\ZZ} \YY)_{/\prupperstar_2(Y)} \arrow[r] & \YY_{/Y} \arrow[dd] \\ 
      	(\XX \cross_{\ZZ} \YY)_{/\prupperstar_1(X)} \arrow[d] & & \\ 
      	\XX_{/X} \arrow[rr] & & \ZZ_{/Z} 
      \end{tikzcd}
    \end{equation*}
	defines a geometric morphism $ \elowerstar \colon \fromto{(\XX \cross_{\ZZ} \YY)_{/(X \cross_Z Y)}}{\XX_{/X} \cross_{\ZZ_{/Z}} \YY_{/Y}} $.
	We claim that $ \elowerstar $ is an equivalence of \topoi.
	To see this, we show that for any \topos $ \EE $, the induced functor 
	\begin{equation*}
		\Funupperstar((\XX \cross_{\ZZ} \YY)_{/(X \cross_Z Y)},\EE) \to \Funupperstar(\XX_{/X} \cross_{\ZZ_{/Z}} \YY_{/Y},\EE)
	\end{equation*}
	is an equivalence of \categories.
	
	Indeed, for any \topos $ \EE $, consider the commutative square
	\begin{equation*}
      \begin{tikzcd}
      	\Funupperstar(\XX_{/X} \cross_{\ZZ_{/Z}} \YY_{/Y},\EE) \arrow[r, "\sim"{yshift=-0.25ex}] \arrow[d] & \displaystyle \Funupperstar(\XX_{/X},\EE) \crosslimits_{\Funupperstar(\ZZ_{/Z},\EE)} \Funupperstar(\YY_{/Y},\EE) \arrow[d] \\ 
      	\Funupperstar(\XX \cross_{\ZZ} \YY,\EE) \arrow[r, "\sim"{yshift=-0.25ex}] & \displaystyle \Funupperstar(\XX,\EE) \crosslimits_{\Funupperstar(\ZZ,\EE)} \Funupperstar(\YY,\EE)\period
      \end{tikzcd}
    \end{equation*}
    It follows from \Cref{HTT.6.3.5.6}=\allowbreak\HTT{Corollary}{6.3.5.6} that the functor
    \begin{equation*}
    	\fromto{\Funupperstar(\XX_{/X} \cross_{\ZZ_{/Z}} \YY_{/Y},\EE)}{\Funupperstar(\XX \cross_{\ZZ} \YY,\EE)}
    \end{equation*}
    is a left fibration whose fiber over a left exact left adjoint $ \hupperstar \colon \fromto{\XX \cross_{\ZZ} \YY}{\EE} $ is the space
    \begin{equation*}
    	\Map_{\EE}(1,\hupperstar\prupperstar_1(X)) \crosslimits_{\Map_{\EE}(1,\hupperstar\prupperstar_1\fupperstar(Z))} \Map_{\EE}(1,\hupperstar\prupperstar_2(Y)) \equivalent \Map_{\EE}(1,\hupperstar(X \cross_{Z} Y)) \period
    \end{equation*}

    On the other hand, again by \Cref{HTT.6.3.5.6}=\allowbreak\HTT{Corollary}{6.3.5.6}, the natural geometric morphism $ \fromto{(\XX \cross_{\ZZ} \YY)_{/(X \cross_Z Y)}}{\XX \cross_{\ZZ} \YY} $ induces a left fibration  
    \begin{equation*}
    	\fromto{\Funupperstar((\XX \cross_{\ZZ} \YY)_{/(X \cross_Z Y)},\EE)}{\Funupperstar(\XX \cross_{\ZZ} \YY,\EE)} 
    \end{equation*}
    whose fiber over $ \hupperstar $ is the space $ \Map_{\EE}(1,\hupperstar(X \cross_Z Y)) $. 
    Thus the geometric morphism $ \elowerstar $ induces a fiberwise equivalence 
    \begin{equation*}\label{eq:leftfiboverpullback}
    	\fromto{\Funupperstar((\XX \cross_{\ZZ} \YY)_{/(X \cross_Z Y)},\EE)}{\Funupperstar(\XX_{/X} \cross_{\ZZ_{/Z}} \YY_{/Y},\EE)} 
    \end{equation*}
    of left fibrations over $ \Funupperstar(\XX \cross_{\ZZ} \YY,\EE) $, hence an equivalence
\end{proof}

Now we turn to the compatibility of oriented fiber products and étale geometric morphisms. 
We treat the path \topos first, and then apply \Cref{lem:slicepullback} to deduce the general result by expressing the oriented fiber product as an iterated pullback.

\begin{lem}\label{lem:slicePath}
	Let $ \ZZ $ be \atopos, and let $ Z \in \ZZ $ be an object.
	Then the natural geometric morphism $ \plowerstar \colon \fromto{\Path{\ZZ_{/Z}}}{\Path{\ZZ}} $ is étale and $ \plowershriek(1) \equivalent \prupperstar_1(Z) $.
\end{lem}

\begin{proof}
	We have two geometric morphisms
	\begin{equation*}
		\plowerstar \colon\Path{\ZZ}_{/\pr_1^{\ast}(Z)}\to\ZZ_{/Z} \andeq \qlowerstar \colon\Path{\ZZ}_{/\pr_1^{\ast}(Z)}\to\Path{\ZZ}_{/\pr_2^{\ast}(Z)}\to\ZZ_{/Z}
	\end{equation*}
	along with a natural transformation $ \gamma\colon\fromto{\qlowerstar}{\plowerstar}$.
	These define a geometric morphism
	\begin{equation*}
		\elowerstar \colon \fromto{\Path{\ZZ}_{/\prupperstar_1(Z)}}{\Path{\ZZ_{/Z}}} 
	\end{equation*} 
	over $ \Path{\ZZ} $.
	We claim that $ \elowerstar $ is an equivalence of \topoi.

	First, for any \topos $ \EE $, consider the commutative square
	\begin{equation*}
      \begin{tikzcd}
      	\Funupperstar(\Path{\ZZ_{/Z}},\EE) \arrow[r, "\sim"{yshift=-0.25ex}] \arrow[d] & \Fun([1],\Funupperstar(\ZZ_{/Z},\EE)) \arrow[d] \\ 
      	\Funupperstar(\Path{\ZZ},\EE) \arrow[r, "\sim"{yshift=-0.25ex}] & \Fun([1],\Funupperstar(\ZZ,\EE)) \period
      \end{tikzcd}
    \end{equation*}
    It follows from \cite[Corollaries \HTTthmlink{2.1.2.9} \& \HTTthmlink{6.3.5.6}]{HTT} that the functor
    \begin{equation*}
    	\fromto{\Funupperstar(\Path{\ZZ_{/Z}},\EE)}{\Funupperstar(\Path{\ZZ},\EE)}
    \end{equation*}
    is a left fibration whose fiber over $ \hupperstar $ is the space
    \begin{equation*}
    	\Map_{\EE}(1,\hupperstar\prupperstar_1(Z)) \crosslimits_{\Map_{\EE}(1,\hupperstar\prupperstar_2(Z))} \Map_{\EE}(1,\hupperstar\prupperstar_2(Z)) \equivalent \Map_{\EE}(1,\hupperstar\prupperstar_1(Z)) \period
    \end{equation*}
    Here the map $ \fromto{\Map_{\EE}(1,\hupperstar\prupperstar_1(Z))}{\Map_{\EE}(1,\hupperstar\prupperstar_2(Z))} $ is induced by the natural transformation $ \sigmahat \colon \fromto{\prupperstar_1}{\prupperstar_2} $ adjoint to the natural transformation $ \sigma \colon \fromto{\pr_{2,\ast}}{\pr_{1,\ast}} $ defining the path \topos $ \Path{\ZZ} $.

    On the other hand, by \Cref{HTT.6.3.5.6}=\allowbreak\HTT{Corollary}{6.3.5.6} for any \topos $ \EE $, the natural geometric morphism $ \fromto{\Path{\ZZ}_{/\prupperstar_1(Z)}}{\Path{\ZZ}} $ induces a left fibration  
    \begin{equation*}
    	\fromto{\Funupperstar(\Path{\ZZ}_{/\prupperstar_1(Z)},\EE)}{\Funupperstar(\Path{\ZZ},\EE)} 
    \end{equation*}
    whose fiber over $ \hupperstar$ is the space $ \Map_{\EE}(1,\hupperstar\prupperstar_1(Z)) $. 
    Thus for any \topos $ \EE $, the geometric morphism $ \elowerstar $ induces a fiberwise equivalence 
    \begin{equation*}\label{eq:leftfiboverPath}
    	\fromto{\Funupperstar(\Path{\ZZ}_{/\prupperstar_1(Z)},\EE)}{\Funupperstar(\Path{\ZZ_{/Z}},\EE)} 
    \end{equation*}
    of left fibrations over $ \Funupperstar(\Path{\ZZ},\EE) $.
\end{proof}

Now we define an object $ \orientedpull{X}{Z}{Y} \in \orientedpull{\XX}{\ZZ}{\YY} $ and prove that 
\begin{equation*}
	\orientedpull{\XX_{/X}}{\ZZ_{/Z}}{\YY_{/Y}} \equivalent (\orientedpull{\XX}{\ZZ}{\YY})_{\orientedpull{X}{Z}{Y}} \period
\end{equation*}

\begin{cnstr}\label{cnstr:orientedslice}
	Let $ \flowerstar \colon \fromto{\XX}{\ZZ} $ and $ \glowerstar \colon \fromto{\YY}{\ZZ} $ be geometric morphisms of \topoi, and let $ X \in \XX $, $ Y \in \YY $, and $ Z \in \ZZ $ be objects, and let $ \phi \colon \fromto{X}{\fupperstar(Z)} $ and $ \psi \colon \fromto{Y}{\gupperstar(Z)} $.
	Form the oriented fiber product
	\begin{equation*}
		\begin{tikzcd}
			\orientedpull{\XX}{\ZZ}{\YY} \arrow[r, "\pr_{2,\ast}" above] \arrow[d, "\pr_{1,\ast}" left] & \YY \arrow[d, "\glowerstar" right] \arrow[dl, phantom, "\scriptstyle \sigma" below right, "\Longleftarrow" sloped] \\ 
			\XX \arrow[r, "\flowerstar" below] & \ZZ \period
		\end{tikzcd}
	\end{equation*}
	Define an object $ \orientedpull{X}{Z}{Y} $\index[notation]{XtimesZY@$\orientedpull{X}{Z}{Y}$} of $ \orientedpull{\XX}{\ZZ}{\YY} $ by the pullback square
	\begin{equation*}
      \begin{tikzcd}[column sep=4em, row sep=2.5em]
	       \orientedpull{X}{Z}{Y}  \arrow[dr, phantom, very near start, "\lrcorner", xshift=-1em, yshift=0.25em] \arrow[d] \arrow[r] & \prupperstar_2(Y) \arrow[d, "\prupperstar_2(\psi)"] \\ 
	       \prupperstar_1(X) \arrow[r, "{\sigmahat(Z) \circ \prupperstar_1(\phi)}"'] & \prupperstar_2 \gupperstar(Z) \comma
      \end{tikzcd}
    \end{equation*}
    where
    \begin{equation*}
    	\sigmahat \colon \fromto{\prupperstar_1 \fupperstar}{\prupperstar_2 \gupperstar} 
    \end{equation*}
    is the natural transformation adjoint to the natural transformation $ \sigma \colon \fromto{\glowerstar \pr_{2,\ast}}{\flowerstar \pr_{1,\ast}} $.
\end{cnstr}

\begin{prp}\label{prop:orientedslice}
	Keep the notations of \Cref{cnstr:orientedslice}.
	Then the natural geometric morphism $ \plowerstar \colon \fromto{\orientedpull{\XX_{/X}}{\ZZ_{/Z}}{\YY_{/Y}}}{\orientedpull{\XX}{\ZZ}{\YY}} $ is étale and $ \plowershriek(1) \equivalent \prupperstar_1(\orientedpull{X}{Z}{Y}) $.
\end{prp}

\begin{proof}
	The claim follows from \Cref{lem:slicePath} along with \Cref{lem:slicepullback} applied to the top right, top left, and bottom left cubes in the diagram
	\begin{equation*}
		\begin{tikzcd}[column sep={11ex,between origins}, row sep={8ex,between origins}]
			& \orientedpull{\XX_{/X}}{\ZZ_{/Z}}{\YY_{/Y}} \arrow[dl] \arrow[rr] \arrow[dd] & & \orientedpull{\ZZ_{/Z}}{\ZZ_{/Z}}{\YY_{/Y}} \arrow[dl] \arrow[rr] \arrow[dd] & & \YY_{/Y} \arrow[dl] \arrow[dd] \\
			\orientedpull{\XX}{\ZZ}{\YY} \arrow[dd] \arrow[rr, crossing over] & & \orientedpull{\ZZ}{\ZZ}{\YY} \arrow[dd] \arrow[rr, crossing over] & & \YY \arrow[dd] \\
			& \orientedpull{\XX_{/X}}{\ZZ_{/Z}}{\ZZ_{/Z}} \arrow[dl] \arrow[rr] \arrow[dd] & & \Path{\ZZ_{/Z}} \arrow[dl] \arrow[rr] \arrow[dd] & & \ZZ_{/Z} \arrow[dl] \arrow[dd, equals] \\
			\orientedpull{\XX}{\ZZ}{\ZZ} \arrow[rr, crossing over] \arrow[dd] & & \Path{\ZZ} \arrow[from=uu, crossing over] \arrow[rr, crossing over] & & \ZZ \arrow[from=uu, crossing over] &  \\
			& \XX_{/X} \arrow[dl] \arrow[rr] & & \ZZ_{/Z} \arrow[dl] \arrow[rr, equals] & & \ZZ_{/Z} \arrow[dl] \\
			\XX \arrow[rr] & & \ZZ \arrow[from=uu, crossing over] \arrow[rr, equals] & & \ZZ \arrow[from=uu, crossing over, phantom] \arrow[from=uu, equals] & \phantom{\ZZ_{Z/}} \comma 
		\end{tikzcd}
	\end{equation*}
	where the front and back faces of the bottom right cube are oriented fiber product squares, all other squares are commutative, and the front and back faces of each of the top right, top left, and bottom left cubes are pullback squares.
\end{proof}


\begin{cor}\label{lem:orientedsliceequiv}
	Keep the notations of \Cref{cnstr:orientedslice}.
	If the morphism
	\begin{equation*}
		\prupperstar_2(\psi) \colon \fromto{\prupperstar_2(Y)}{\prupperstar_2\gupperstar(Z)}
	\end{equation*}
	is an equivalence, then we have a natural equivalence
	\begin{equation*}
		(\orientedpull{\XX}{\ZZ}{\YY})_{/\orientedpull{X}{Z}{Y}} \equivalent (\orientedpull{\XX}{\ZZ}{\YY})_{/\prupperstar_1(X)} \period
	\end{equation*}
\end{cor}

\begin{nul}
	Keep the notations of \Cref{cnstr:orientedslice} and assume, in addition, that $ \XX $, $ \YY $ and $ \ZZ $ are bounded coherent, the geometric morphisms $ \flowerstar $ and $ \glowerstar $ are coherent, and the objects $ X $, $ Y $, and $ Z $ are all truncated coherent.
	Then the object $ \orientedpull{X}{Z}{Y} \in \orientedpull{\XX}{\ZZ}{\YY} $ is the image under the Yoneda embedding $ \yo \colon \incto{\overleftW(f,g)}{\orientedpull{\XX}{\ZZ}{\YY}} $ of the object of $ \overleftW(f,g) $ given by $ X \to Z \ot Y $ (\Cref{exm:sheaforientedfib}) .
\end{nul}

\newpage

\section{Local \texorpdfstring{$\infty$}{∞}-topoi \& localizations}\label{sec:localtopoi}

In this chapter we generalize the basic theory of what are usually called \textit{local geometric morphisms} and \textit{local topoi} to the setting of \topoi \cites[Exposé IV, \S 8]{MR50:7131}[\S C.3.6]{MR2063092}{MR977478}.
Local \topoi play the role of local rings in topos theory: one can \textit{localize} \atopos $ \XX $ at a point $ \xlowerstar \colon \fromto{\Space}{\XX} $ and this construction has the property that the stalk $ \xupperstar U $ of an object $ U \in \XX $ is computed by first pulling back to the localization of $ \XX $ at $ \xlowerstar $ and then taking global sections on the local \topos.
The chief example of a local \topos is the étale \topos of a strictly henselian local ring (\Cref{exm:localizationishenselization}).
As with local rings in algebraic geometry, often questions about \topoi with enough points can be reduced to problems about local \topoi.
This is the main reason we need the theory of local \topoi; to prove a basechange theorem for oriented fiber products (\Cref{thm:BCfororientedfibs}) in \cref{section:BC} by reduction to the local case.

The \toposic theory follows the $ 1 $-toposic story very closely; as such, a number of items in this chapter are likely known to experts.
Notably, Schreiber has studied local \topoi \cite[\S3.2]{Schreiber:cohesive}. 

In \cref{subsec:quasiequiv} we begin by discussing a more general class of geometric morphisms that contains the global sections geometric morphism of a local \topos.
\Cref{subsec:local} then specializes to the study of local \topoi.
\Cref{subsec:nearbyloc} explains how to to use oriented fiber products to localize \atopos at a point.
In \Cref{subsec:compatibilityOFPloc} we prove a compatibility between oriented fiber products and localizations.
In algebraic geometry, the spectrum of the strictly henselian local ring $ \Oup_{X,x}^{\sh} $ of a scheme $ X $ at a geometric point $ \fromto{\bar{x}}{X}$ can be written as a limit over all étale neighborhoods of $ \bar{x} $; \Cref{subsec:GVlocalization} proves that the localization of \atopos at a point can be described in exactly the same way (\Cref{prop:localizationasetale}).
Using this description, in \Cref{subsec:cohloc} we show that the localization of a bounded coherent \topos at a point is coherent (\Cref{lem:localizationiscoherent}).
\Cref{subsec:localexamples} concludes by collecting geometric examples of localizations.


\subsection{Quasi-equivalences}\label{subsec:quasiequiv}

As a precursor, we begin by discussing the \toposic generalization of the notion of a \textit{connected} geometric morphism \cite[p. 525]{MR2063092}.
In the homotopical setting, the term `connected' (and its variants) doesn't seem appropriate.
Instead, we elect for the distinct term \textit{quasi-equivalence}.

\begin{dfn}
	A geometric morphism $ \flowerstar \colon \fromto{\XX}{\YY} $ of \topoi is a \defn{quasi-equivalence}\index[terminology]{quasi-equivalence}\index[terminology]{geometric morphism!quasi-equivalence} if the pullback functor $ \fupperstar $ is fully faithful.
\end{dfn}

\begin{nul}
	Every geometric morphism of \topoi factors as the composite of a quasi-equiv\-alence followed by an \textit{algebraic} geometric morphism \HTT{Proposition}{6.3.6.2}.
	Moreover, this factorization is unique up to (canonical) equivalence .
\end{nul}

If $f_{\ast}$ is a quasi-equivalence, then $f^{\ast}$ is fully faithful, whence we deduce the following.

\begin{lem}\label{lem:globalsecconnected}
	Let $ \flowerstar \colon \fromto{\XX}{\YY} $ be a quasi-equivalence of \topoi.
	Write $ \unit \colon \fromto{\id_{Y}}{\flowerstar \fupperstar} $ for the unit of the adjunction $ \fupperstar \leftadjoint \flowerstar $.
	Then the canonical natural transformation
	\begin{equation*}
		\upGamma_{\YY,\ast}\unit \colon \upGamma_{\YY,\ast} \to \upGamma_{\YY,\ast} \flowerstar \fupperstar \equivalent \upGamma_{\XX,\ast} \fupperstar
	\end{equation*}
	is an equivalence (\Cref{ntn:globalsecpoints}).
\end{lem}

\begin{nul}
	If $ \flowerstar \colon \fromto{\XX}{\YY} $ is a quasi-equivalence of \topoi, then by composing the canonical natural transformation $ \fromto{\Gammaup_{\YY,\ast}}{\Gammaup_{\XX,\ast}\fupperstar} $ with $ \Gammaupperstar_{\YY} $, \Cref{lem:globalsecconnected} ensures that the canonical natural transformation
	\begin{equation*}
		\upGamma_{\YY,\ast} \unit \Gammaupperstar_{\YY} \colon \fromto{\Gammaup_{\YY,\ast} \Gammaupperstar_{\YY}}{\Gammaup_{\YY,\ast} \flowerstar \fupperstar \Gammaupperstar_{\YY} \equivalent \Gammaup_{\XX,\ast} \Gammaupperstar_{\XX}}
	\end{equation*}
	is an equivalence in $ \Pro(\Space)^{\op} \subset \Fun(\Space,\Space) $.
	In particular, $ \flowerstar $ is a shape equivalence (\Cref{dfn:shapeequivalence}).
\end{nul}

\begin{nul}
	As noted in \HTT{Remark}{7.1.6.12}, \atopos $ \XX $ has trivial shape if and only if the geometric morphism $ \fromto{\XX}{\Space} $ is a quasi-equivalence.
	However, in general a shape equivalence of \topoi need not be a quasi-equivalence.
\end{nul}

\begin{exm}
	Let $ X $ be a scheme.
	By \cite[Lemma 5.1.2]{MR3379634}, the natural geometric morphism $ \fromto{X_{\proet}}{X_{\et}} $ from the proétale \topos of $ X $ to the étale \topos of $ X $ is a quasi-equivalence, hence a shape equivalence.
\end{exm}


\subsection{Local \texorpdfstring{$\infty$}{∞}-topoi}\label{subsec:local}

Now we specalize to local \topoi.
Recall that a geometric morphism $ \flowerstar \colon \fromto{\XX}{\YY} $ is \textit{essential}\index[terminology]{essential!geometric morphism}\index[terminology]{geometric morphism!essential} if $ \fupperstar $ admists a left adjoint $ \flowershriek \colon \fromto{\XX}{\YY} $.

\begin{dfn}\label{def:coessentialmorphism}
	We say that a geometric morphism $ \flowerstar \colon \fromto{\XX}{\YY} $ of \topoi is \defn{coëssential}\index[terminology]{coessential@coëssential}\index[terminology]{geometric morphism!coessential@coëssential} if $ \flowerstar $ admits a right adjoint $ \fuppershriek \colon \fromto{\YY}{\XX} $.
	In this case, the functor $ \fuppershriek $ and its left adjoint $ \flowerstar $ define a geometric morphism \smash{$ \fuppershriek \colon \fromto{\YY}{\XX} $} called the \defn{center}\index[terminology]{center!of a coëssential geometric morphism} of $ \flowerstar $.
\end{dfn}

The next definition generalizes what are sometimes called \textit{local geometric morphisms} in the $ 1 $-topos theory literature \cites[\S C.3.6]{MR2063092}{MR977478}.
We instead choose terminology that syncs with the algebro-geometric terminology for local rings and doesn't conflict with other uses of the term `local' in higher category theory.

\begin{dfn}\label{def:localmorphism}
	We say that geometric morphism $ \flowerstar \colon \fromto{\XX}{\YY} $ of \topoi \defn{exhibits $\XX$ as local over $\YY$} if $ \flowerstar $ is both coëssential and a quasi-equivalence.

	\Atopos $ \XX $ is \defn{local}\index[terminology]{local!topos@\topos}\index[terminology]{topos@\topos!local} if $ \XX $ is local over $\Space$.
	In this case we simply call $ \Gammauppershriek_{\XX} \colon \fromto{\Space}{\XX} $ the \defn{center} of $ \XX $.
\end{dfn}

\begin{nul}\label{nul:localequivcond}
	Please observe that a geometric morphism of \topoi $ \flowerstar \colon \fromto{\XX}{\YY} $ exhibits $\XX$ as local over $\YY$ if and only if the functor $ \flowerstar $ admits a \textit{fully faithful} right adjoint $ \fuppershriek $.
	Equivalently, $\XX$ is local over $\YY$ if and only if $ \flowerstar $ admits a section $ \fuppershriek $ in the $(\infty,2)$-category of \topoi.
\end{nul}

\begin{nul}\label{nul:localoverpt}
	Let $ \XX $ be \atopos.
	Note that if the global sections functor $ \Gammalowerstar \colon \fromto{\XX}{\Space} $ admits a right adjoint $ \Gammauppershriek \colon \fromto{\Space}{\XX} $, then $ \Gammauppershriek $ is automatically fully faithful, whence $\XX$ is local.

	Consequently, by the Adjoint Functor Theorem and \Cref{nul:localoverpt}, \atopos $ \XX $ is local if and only if the terminal object $ 1_{\XX} \in \XX $ is completely compact.
\end{nul}

\begin{rmk}\label{rem:1localiclocal}
	Let $ \XX $ be a $ 1 $-topos with corresponding $ 1 $-localic \topos $ \XX' $.
	Then $ \XX' $ is local over $ \Space $ if and only if the global sections functor $ \fromto{(\XX')_{\leq 0} \equivalent \XX}{\Set} $ admits a right adjoint.
\end{rmk}

\begin{lem}
	Let $ \XX $ be a local \topos.
	Then $ \XX $ has homotopy dimension $ \leq 0 $.
	In particular, $ \XX $ has cohomological dimension $ \leq 0 $.
\end{lem}

\begin{proof}
	By \HTT{Lemma}{7.2.1.7}, it suffices to show that $ \Gammaup_{\XX,\ast} \colon \fromto{\XX}{\Space} $ preserves effective epimorphisms; this follows from the assumption that $ \Gammaup_{\XX,\ast} $ is a left adjoint.
	The second statement is a consequence of \HTT{Corollary}{7.2.2.30}.
\end{proof}

\begin{dfn}\label{def:localgeometricmorphism}
	Let $ \XX $ and $ \YY $ be local \topoi with centers $ \Gammauppershriek_{\XX} $ and $ \Gammauppershriek_{\YY} $, respectively.
	A geometric morphism $ \flowerstar \colon \fromto{\XX}{\YY} $ is a \defn{local geometric morphism}\index[terminology]{geometric morphism!local}\index[terminology]{local!geometric morphism} if the canonical natural morphism
	\begin{equation*}
		\fromto{\flowerstar \Gammauppershriek_{\XX}}{\Gammauppershriek_{\YY}}
	\end{equation*}
	adjoint to $ \upGamma_{\YY,\ast}\unit \colon \upGamma_{\YY,\ast} \to \upGamma_{\XX,\ast} \fupperstar $is an equivalence.
	Write $ \Toploc \subset \Top_{\infty} $\index[notation]{Toploc@$\Toploc$} for the (non-full) subcategory whose objects are local \topoi and whose morphisms are local geometric morphisms.
\end{dfn}

If $ \XX $ is a local \topos, then the center of $ \XX $ is an initial object of the \category $ \Pt(\XX) $; in fact, more is true.

\begin{ntn}
	Let $f_{\ast}\colon \fromto{\XX}{\YY} $ and $f'_{\ast}\colon\XX'\to\YY$ be geometric morphisms of \topoi.
	Write\index[notation]{FunYstar@$\Fun_{\YY,\ast}$}
	\begin{equation*}
		\Fun_{\YY,\ast}(\XX,\XX') \coloneq \Funlowerstar(\XX,\XX') \crosslimits_{\Funlowerstar(\XX,\YY)} \{f_{\ast}\}
	\end{equation*}
	for the \category of geometric morphisms $ \fromto{\XX}{\XX'} $ \textit{over} $ \YY $.
\end{ntn}

\begin{lem}
	Let $ \flowerstar \colon \fromto{\XX}{\YY} $ be a geometric morphism that exhibits $\XX$ as local over $\YY$ with center $ \fuppershriek $.
	Then $ \fuppershriek $ is a terminal object of the \category $ \Fun_{\YY,\ast}(\YY,\XX) $.
\end{lem}

\begin{proof}
	Let $ \glowerstar \colon \YY\to\XX $ be a geometric morphism over $ \YY $.
	Then 
	\begin{align*}
		\Map_{\Fun_{\YY,\ast}(\YY,\XX)}(\glowerstar,\fuppershriek) &\equivalent \Map_{\Fun_{\YY,\ast}(\YY,\YY)}(\flowerstar \glowerstar,\id_{\YY}) \\ 
		&\equivalent \Map_{\Fun_{\YY,\ast}(\YY,\YY)}(\id_{\YY},\id_{\YY}) \equivalent \ast \period \qedhere
	\end{align*}
\end{proof}

Like local rings in algebraic geometry, local \topoi provide a convenient way to compute stalks: take global sections after pulling back to an appropriate local \topos.

\begin{lem}\label{lem:stalksasglobalsections}
	Let $ \plowerstar \colon \fromto{\WW}{\XX} $ be a geometric morphism of \topoi.
	Assume that where $ \WW $ is local with center $ \wlowerstar $, and write $ \xlowerstar \coloneq \plowerstar \wlowerstar $.
	Then $ \xupperstar \equivalent \Gammaup_{\WW,\ast} \pupperstar $.
\end{lem}

\begin{proof}
	Since $ \wupperstar \equivalent \Gammaup_{\WW,\ast} $ we see that
	\begin{equation*}
		\xupperstar \equivalent (pw)^{\ast} \equivalent \wupperstar \pupperstar \equivalent \Gammaup_{\WW,\ast} \pupperstar \period \qedhere
	\end{equation*}
\end{proof}

\noindent We shall soon see (\Cref{def:localization,nul:localizationislocal}) that for any \topos $ \XX $ and any point $ \xlowerstar \in \Pt(\XX) $, there is a geometric morphism $ \plowerstar \colon \fromto{\WW}{\XX} $ in which $ \WW $ is local with center $ \wlowerstar $ and $ \xlowerstar \equivalent \plowerstar \wlowerstar $ (and is, moreover, universal with this property).




Local geometric morphisms are also stable under pullback, though we do not use this fact in an integral way in the present paper.

\begin{nul}\label{nul:localstableunderpullback}
	Consider a pullback square of \topoi
	\begin{equation*}
      \begin{tikzcd}[sep=2.25em]
	       \XX \cross_{\ZZ} \YY  \arrow[dr, phantom, very near start, "\lrcorner", xshift=-0.25em, yshift=0.25em] \arrow[d, "\bar{g}_{\ast}"'] \arrow[r, "\bar{f}_{\ast}"] & \YY \arrow[d, "\glowerstar"] \\ 
	       \XX \arrow[r, "\flowerstar"'] & \ZZ \comma
      \end{tikzcd}
    \end{equation*}
    where $ \glowerstar $ exhibits $ \YY $ as local over $ \ZZ $ with center $ \guppershriek $.
    By the universal property of the pullback, the identity on $ \XX $ and the geometric morphism $ \guppershriek \flowerstar \colon \fromto{\XX}{\YY} $ induce a geometric morphism
    \begin{equation*}
    	\gbar^{!} \coloneq (\id_{\XX},\guppershriek \flowerstar) \colon \fromto{\XX}{\XX \cross_{\ZZ} \YY}
    \end{equation*}
    such that $ \gbar_{\ast} \gbar^{!} \equivalent \id_{\XX} $ and $ \fbar_{\ast} \gbar^{!} \equivalent \guppershriek \flowerstar $.
    Using the universal property of the pullback and the fact that $ \glowerstar $ is exhibits $ \YY $ as local over $ \ZZ $, one easily checks that the functor $ \gbar^{!} $ is right adjoint to $ \gbar_{\ast} $, so that $ \gbar_{\ast} $ exhibits $ \XX \cross_{\ZZ} \YY $ as local over $ \XX $ with center $ \gbar^{!} $.
\end{nul}

\subsection{Nearby cycles \& localizations}\label{subsec:nearbyloc}

We now show that the evanescent \topos (\Cref{exm:evanescent}) provides a wealth of local \topoi.
Then, following Deligne as well as Johnstone--Moerdijk \cite[Definition 3.1]{MR977478}, we use the evanescent \topos to construct the localization of \atopos at a point.

A site-theoretic proof of the following result (originally stated without proof by Laumon \cite[3.2]{MR726426}) is given in \cite[Exposé XI, Proposition 4.4]{MR3309086}.
The reliance on sites renders the proof given in \cite[Exposé XI]{MR3309086} inadequate in the context of \topoi; luckily the work of Riehl--Verity \cite{RiehlVerity:elements} permits us to employ simple $ 2 $-categorical arguments.

\begin{prp}\label{prop:nearbycycleslocal}
	Let $ \flowerstar \colon \fromto{\XX}{\ZZ} $ be a geometric morphism of \topoi.
	Then:
	\begin{enumerate}[{\upshape (\ref*{prop:nearbycycleslocal}.1)}]
		\item\label{prop:nearbycycleslocal.1} The nearby cycles functor $ \near_{f,\ast} \colon \fromto{\XX}{\orientedpull{\XX}{\ZZ}{\ZZ}} $ is right adjoint to the projection $ \pr_{1,\ast} \colon \fromto{\orientedpull{\XX}{\ZZ}{\ZZ}}{\XX} $.
		
		\item\label{prop:nearbycycleslocal.2} The functor $ \near_{f,\ast} $ is fully faithful.
		Hence the geometric morphism $ \pr_{1,\ast} $ exhibits $ \orientedpull{\XX}{\ZZ}{\ZZ} $ as local over $ \XX $ with center $ \near_{f,\ast} $. 
	\end{enumerate}
\end{prp}

\begin{proof}
	Recall that for any \topos $ \EE $, the functor
	\begin{equation*}
		\Funlowerstar(\EE,-)^{\op} \colon \fromto{\Top_{\infty}}{\Cat_{\infty,\updelta_1}}
	\end{equation*}
	carries oriented fiber products in $ \Top_{\infty} $ to oriented fiber products in $ \Cat_{\infty,\updelta_1} $ \Cref{nul:Funstarpreservesorientedprod}.
	Thus the proof of \cite[Proposition 3.4.6]{RiehlVerity:elements} works perfectly, giving the oriented fiber product in $ \Top_{\infty} $ the necessary `weak universal property' (as Riehl and Verity call it) to apply \cite[Lemma 3.5.9]{RiehlVerity:elements}, proving both \enumref{prop:nearbycycleslocal}{1} and \enumref{prop:nearbycycleslocal}{2}.
\end{proof}

The dual notion to being local over \atopos naturally appears as the property satisfied by the second projection from the coëvanescent \topos.

\begin{dfn}\label{def:colocalmorphism}
	A geometric morphism $ \flowerstar \colon \fromto{\XX}{\YY} $ of \topoi exhibits $ \XX $ as \emph{colocal over $\YY$}\index[terminology]{topos@\topos!colocal}\index[terminology]{colocal!topos@\topos} if $ \flowerstar $ is a quasi-equivalence and $ \fupperstar $ admits a left exact left adjoint $ \flowershriek \colon \fromto{\XX}{\YY} $.
	In this case, the functor $ \fupperstar $ and its left adjoint $ \flowershriek $ define a geometric morphism $ \fupperstar \colon \fromto{\YY}{\XX} $ called the \emph{cocenter} of $ \flowerstar $.
\end{dfn}

\begin{nul}
	In the setting of $1$-topoi, Johnstone \cite[Theorem C.3.6.16]{MR2063092} uses the term \emph{totally connected} for what we call colocal. Again, such lingo is inapt in our context.
\end{nul}

The following is dual to \Cref{prop:nearbycycleslocal}.

\begin{prp}\label{prop:conearbycyclescolocal}
	Let $ \glowerstar \colon \fromto{\YY}{\ZZ} $ be a geometric morphism of \topoi.
	Then:
	\begin{enumerate}[{\upshape (\ref*{prop:nearbycycleslocal}.1)}]
		\item\label{prop:conearbycyclescolocal.1} The conearby cycles functor $ \conearlowerstar^g \colon \fromto{\YY}{\orientedpull{\ZZ}{\ZZ}{\YY}} $ is left adjoint to the projection $ \pr_{2,\ast} \colon \fromto{\orientedpull{\ZZ}{\ZZ}{\YY}}{\YY} $.
		
		\item\label{prop:conearbycyclescolocal.2} The functor \smash{$ \conearlowerstar^g \equivalent \prupperstar_2 $} is fully faithful.
		Hence the geometric morphism $ \pr_{2,\ast} $ exhibits \smash{$\ZZ\orientedtimes_{\ZZ}\YY$} as colocal over $ \YY $ with cocenter \smash{$ \conearlowerstar^g $}. 
	\end{enumerate}
\end{prp}

\begin{nul}
	A geometric morphism $ \flowerstar \colon \fromto{\XX}{\YY} $ that exhibits $ \XX $ as colocal over $ \YY$ always satisfies the \textit{étale projection formula}
	\begin{equation*}
		\flowershriek(\fupperstar(X) \cross_{\fupperstar(Z)} Y) \equivalent X \cross_{Z} \flowershriek(Y) 
	\end{equation*}
	of \HTT{Proposition}{6.3.5.11}.
	However, the geometric morphism $ \flowerstar $ will almost never be étale; $ \flowershriek $ is conservative if and only if $ \flowerstar $ is an equivalence. 
\end{nul}

\begin{exm}\label{ex:pathtoposislocalandcolocal}
	For any \topos $ \XX $ the diagonal functor
	\begin{equation*}
		\psi(\id_{\XX},\id_{\XX},\id)_{\ast} \colon \fromto{\XX}{\orientedpull{\XX}{\XX}{\XX} \equivalent \Path{\XX}}
	\end{equation*}
	is both the nearby and conearby cycles functor.
	Combining \Cref{prop:nearbycycleslocal,prop:conearbycyclescolocal}, we deduce that we have a chain of (left exact) adjoints
	\begin{equation*}
		\begin{tikzcd}[sep=5.5em]
			\Path{\XX} \arrow[r, shift left=1ex, "\pr_{1,\ast}"{xshift=-2ex,description}] \arrow[r, shift right=2.5ex, "\pr_{2,\ast}"'] & \XX \period \arrow[l, hooked', shift left=1ex, "\prupperstar_2"{xshift=2ex,description}] \arrow[l, hooked', shift right=2.5ex, "\prupperstar_1"']
		\end{tikzcd}
	\end{equation*}
	In particular, the geometric morphisms $ \pr_{1,\ast}, \pr_{2,\ast} \colon \fromto{\Path{\XX}}{\XX} $ are shape equivalences.
\end{exm}

Now we define the \textit{localization} of \atopos at a point as a evanescent \topos; for this please recall \Cref{ntn:globalsecpoints}.

\begin{dfn}\label{def:localization}
	Let $ \XX $ be \atopos and $ \xlowerstar \colon \fromto{\Space}{\XX} $ a point of $ \XX $.
	The \defn{localization}\index[terminology]{localization!of \atopos}\index[terminology]{topos@\topos!localization} of $ \XX $ at $ \xlowerstar $ is the evanescent \topos
	\begin{equation*}
		\Loc{\XX}{x} \coloneq \orientedpull{\xtilde}{\XX}{\XX} \period \index[notation]{Xx@$ \Loc{\XX}{x} $}
	\end{equation*}
	We write $ \ell_{x,\ast} \colon \fromto{\Loc{\XX}{x}}{\XX} $ for the second projection $ \pr_{2,\ast} \colon \fromto{\orientedpull{\xtilde}{\XX}{\XX}}{\XX} $.
\end{dfn}

\begin{nul}\label{nul:localizationislocal}
	Let $ \XX $ be \atopos and $ \xlowerstar $ a point of $ \XX $.
	By \Cref{prop:nearbycycleslocal}, the \topos $ \Loc{\XX}{x} $ is local with center $ \conear_{x,\ast} \colon \fromto{\Space}{\Loc{\XX}{x}} $.
	By \Cref{lem:stalksasglobalsections}, for every object $ F \in \XX $ we can compute the stalk at $ x $ via the familiar formula
	\begin{equation*}
		\xupperstar F \equivalent \upGamma(\Loc{\XX}{x};\ell_x^{\ast} F) \period 
	\end{equation*}
\end{nul}

\begin{ntn}
	Write\index[notation]{Topstar@$\Toppt$} $ \Toppt \coloneq \Top_{\infty,\Space/} $ for the \category of \topoi equipped with a topos-theoretic point.
	The assignment $ \goesto{(\XX,\xlowerstar)}{\Loc{\XX}{x}} $ defines a functor
	\begin{equation*}
		\fromto{\Toppt}{\Toploc}
	\end{equation*}
	In the other direction, the assignment $ \goesto{\XX}{(\XX,\Gammauppershriek_{\XX})} $ defines a fully faithful functor
	\begin{equation*}
		\incto{\Toploc}{\Toppt} \period
	\end{equation*}
\end{ntn}

\begin{prp}\label{prop:localizationofalocaltopos}
	Let $ \XX $ be a local \topos with center $ \xlowerstar $.
	Then the geometric morphism $ \ell_{x,\ast} \colon \fromto{\Loc{\XX}{x}}{\XX} $ is an equivalence.
\end{prp}


\begin{proof}
	Let $ \unit \colon \fromto{\id_{\XX}}{\xlowerstar \Gammaup_{\XX,\ast}} $ be the unit of the adjunction $ \Gammaup_{\XX,\ast} \leftadjoint \xlowerstar $.
	Then the oriented square
	\begin{equation*}
		\begin{tikzcd}
			\XX \arrow[r, equals] \arrow[d, "\upGamma_{\XX,\ast}"'] & \XX \arrow[d, equals] \arrow[dl, phantom, "\scriptstyle \unit" below right, "\Longleftarrow" sloped] \\ 
			\widetilde{x} \arrow[r] & \XX 
		\end{tikzcd}
	\end{equation*}
	exhibits $ \XX $ as the oriented fiber product $ \orientedpull{\widetilde{x}}{\XX}{\XX} $.
\end{proof}


\begin{cor}
	The fully faithful functor $ \incto{\Toploc}{\Toppt} $ admits a right adjoint given by the assignment $ \goesto{(\XX,\xlowerstar)}{\Loc{\XX}{x}} $.
\end{cor}

\subsection{Compatibility of oriented fiber products with localizations}\label{subsec:compatibilityOFPloc}

In this section we prove that the formation oriented fiber products is compatible with localizations of \topoi.
First we note that taking path \topoi commutes with the formation of oriented fiber products.

\begin{lem}\label{lem:pathcompatiblewithofp}
	Let $ \flowerstar \colon \fromto{\XX}{\ZZ} $ and $ \glowerstar \colon \fromto{\YY}{\ZZ} $ be geometric morphisms of \topoi.
	Then we have a natural equivalence
	\begin{equation*}
		\Path{\orientedpull{\XX}{\ZZ}{\YY}} \equivalent \Path{\XX} \underset{\Path{\ZZ}}{\orientedtimes} \Path{\YY} \period
	\end{equation*}
\end{lem}

\begin{proof}
	Since the path \topos construction is a right adjoint $ \fromto{\Top_{\infty}}{\Top_{\infty}} $, we have natural equivalences
	\begin{align*}
		\Path{\orientedpull{\XX}{\ZZ}{\YY}} &= \Path{\XX \cross_{\ZZ} \Path{\ZZ} \cross_{\ZZ} \YY} \\ 
		&\equivalent \Path{\XX} \crosslimits_{\Path{\ZZ}} \Path{\Path{\ZZ}} \crosslimits_{\Path{\ZZ}} \Path{\YY} \\
		&= \Path{\XX} \underset{\Path{\ZZ}}{\orientedtimes} \Path{\YY} \period \qedhere
	\end{align*}
\end{proof}

\begin{prp}
	Let $ \flowerstar \colon \fromto{(\XX,\xlowerstar)}{(\ZZ,\zlowerstar)} $ and $ \glowerstar \colon \fromto{(\YY,\ylowerstar)}{(\ZZ,\zlowerstar)} $ be morphisms of pointed \topoi, so that there is an induced point
	\begin{equation*}
		\orientedpull{\xlowerstar}{\zlowerstar}{\ylowerstar} \colon \fromto{\Space \equivalent \orientedpull{\Space}{\Space}{\Space}}{\orientedpull{\XX}{\ZZ}{\YY}} \period
	\end{equation*}
	Then we have a natural equivalence
	\begin{equation*}
		\Loc{(\orientedpull{\XX}{\ZZ}{\YY})}{\orientedpull{\xlowerstar}{\zlowerstar}{\ylowerstar}} \equivalent \orientedpull{\Loc{\XX}{x}}{\Loc{\ZZ}{z}}{\Loc{\YY}{y}} \period
	\end{equation*}
\end{prp}

\begin{proof}
	Consider the diagram $ \fromto{\horn{2}{2}}{\Fun(\horn{2}{2},\Top_{\infty})} $ defined by the diagram
	\begin{equation}\label{diag:locoforientedprod}
		\begin{tikzcd}[sep=2.75em]
		    \Path{\XX} \arrow[r, "\Path{\flowerstar}"] \arrow[d, "\pr_{1,\ast}"'] & \Path{\ZZ} \arrow[d, "\pr_{1,\ast}" description] & \Path{\YY} \arrow[l, "\Path{\glowerstar}"'] \arrow[d, "\pr_{1,\ast}"] \\
			\XX \arrow[r, "\flowerstar" description] & \ZZ & \YY \arrow[l, "\glowerstar" description] \\
			\Space \arrow[r, equals] \arrow[u, "\xlowerstar"] & \Space \arrow[u, "\zlowerstar" description] & \Space \comma \arrow[l, equals] \arrow[u, "\ylowerstar"']
		\end{tikzcd}
	\end{equation}
	where we have displayed objects of $ \Fun(\horn{2}{2},\Top_{\infty}) $ horizontally, and morphisms in $ \Fun(\horn{2}{2},\Top_{\infty}) $ vertically.
	First taking the (vertical) limit of the diagram \eqref{diag:locoforientedprod} in
	\begin{equation*}
		\Fun(\horn{2}{2},\Top_{\infty})
	\end{equation*}
	we obtain the cospan
	\begin{equation*}
		\begin{tikzcd}[sep=2.75em]
			\Loc{\XX}{x} \arrow[r, "\flowerstar"] & \Loc{\ZZ}{z} & \Loc{\YY}{y} \period \arrow[l, "\glowerstar"']
		\end{tikzcd}
	\end{equation*}
	Then taking the oriented fiber product of the this cospan yields $ \orientedpull{\Loc{\XX}{x}}{\Loc{\ZZ}{z}}{\Loc{\YY}{y}} $.
	On the other hand, by \Cref{lem:pathcompatiblewithofp}, first forming the oriented fiber product horizontally then taking pullbacks vertically yields the localization
	\begin{equation*}
		\Loc{(\orientedpull{\XX}{\ZZ}{\YY})}{\orientedpull{\xlowerstar}{\zlowerstar}{\ylowerstar}} \period
	\end{equation*}
	The claim now follows from the fact that the formation of oriented fiber products commutes with limits \Cref{nul:orientedcommuteswithlimits}.
\end{proof}


\subsection{Localization à la Grothendieck--Verdier}\label{subsec:GVlocalization}

In order to get our hands on geometric examples of localized \topoi, we give another description of $ \Loc{\XX}{x} $ that is akin to Grothendieck and Verdier's original ($1$-toposic) definition of the localization as a limit over étale neighborhoods of $ \xlowerstar $ in $ \XX $ \cite[Exposé VI, 8.4.2]{MR50:7131}.

\begin{dfn}
	Let $ (\XX,\xlowerstar) $ be a pointed \topos.
	The \emph{\category of étale neighborhoods of $ \xlowerstar $}\index[terminology]{etale neighborhood@étale neighborhood} is the pullback
	\begin{equation*}
      \begin{tikzcd}
	      \Nbd(x)  \arrow[dr, phantom, very near start, "\lrcorner", xshift=-0.25em, yshift=0.25em] \arrow[d] \arrow[r] & \Space_{\ast} \arrow[d] \\ 
	       \XX \arrow[r, "\xupperstar"'] & \Space 
      \end{tikzcd}
    \end{equation*}
    formed in $ \Cat_{\infty,\updelta_1} $.

	By \cite[\HTTthm{Corollary}{6.3.5.6} \& \HTTthm{Remark}{6.3.5.7}]{HTT} the \category $ \Nbd(x) $ is equivalent to the full subcategory of $ (\Toppt)_{/(\XX,\xlowerstar)} $ spanned by those objects $ \fromto{(\EE,\elowerstar)}{(\XX,\xlowerstar)} $ with the property that the geometric morphism $ \fromto{\EE}{\XX} $ is étale.

	Please note that $ \Nbd(x) $ is an inverse \category.
\end{dfn}

To provide the limit description of the localization as well as the familiar colimit formula for the stalk $ \xupperstar $, we need to take limits of diagrams indexed by the (not necessarily $\updelta_0$-small) \category $ \Nbd(x) $.
Happily the exact same cofinality argument given in \cite[Exposé IV, 6.8]{MR50:7131} works in the setting of higher topoi, showing that $ \Nbd(x) $ admits a limit-cofinal $\updelta_0$-small subcategory.

\begin{cnstr}\label{cnstr:limloc}
	Let $ \XX $ be a \topos and $ \xlowerstar \in \Pt(\XX) $.
	Then by the Yoneda lemma the stalk functor $ \xupperstar \colon \fromto{\XX}{\Space} $ can be computed as the filtered colimit
	\begin{equation*}
		\xupperstar \equivalent \colim_{(U,u) \in \Nbd(x)^{\op}} \Map_{\XX}(U,-) \period
	\end{equation*}

	The assignment $ \goesto{(U,u)}{\XX_{/U}} $ defines a functor $ E_x \colon \Nbd(x) \to \Top_{\infty,/\XX} $.
	Moreover, the natural forgetful functor
	\begin{equation*}
		\fromto{\Top_{\infty,/E_x}}{\Top_{\infty,/\XX}}
	\end{equation*}
	is a right fibration.
	We write $ \lim_{(U,u) \in \Nbd(x)} \XX_{/U} $ for the limit in $ \Top_{\infty,/\XX} $ (equivalently, in $ \Top_{\infty} $) of the diagram $ E_x $.

	By \Cref{HTT.6.3.5.6}=\allowbreak\HTT{Corollary}{6.3.5.6}, specifying a geometric morphism
	\begin{equation*}
		\fromto{\XX'}{\lim_{(U,u) \in \Nbd(x)} \XX_{/U}}
	\end{equation*}
	is equivalent to specifying a geometric morphism $ \plowerstar \colon \fromto{\XX'}{\XX} $ along with a global section
	\begin{equation*}
		s \in \Gammaup_{\XX',\ast}\left(\lim_{(U,u) \in \Nbd(x)}p^{\ast}U\right)\simeq \lim_{(U,u) \in \Nbd(x)}\Gammaup_{\XX',\ast}p^{\ast}U\period
	\end{equation*}

	Since $ \Loc{\XX}{x} $ is the localization of $ \XX $ at $ \xlowerstar $, we have a natural equivalence $ \xupperstar \equivalent \Gammaup_{\Loc{\XX}{x},\ast} \ell_{x}^{\ast} $ \Cref{nul:localizationislocal}.
	Thus for $ U \in \XX $, we obtain a natural equivalence
	\begin{equation*}
		\lim_{(U,u) \in \Nbd(x)} \xupperstar(U) \equivalent \Gammaup_{\XX_{(x)},\ast}\paren{\lim_{(U,u) \in \Nbd(x)} \ell_{x}^{\ast}(U)} \period
	\end{equation*}
	The global sections $ u \in\xupperstar(U) $ for $ (U,u) \in \Nbd(x) $ together define a global section $ s \in\lim_{(U,u) \in \Nbd(x)} \xupperstar(U) $.
	This provides a comparison geometric morphism 
	\begin{equation*}
		c_{x,\ast} \colon \fromto{\Loc{\XX}{x}}{\lim_{(U,u) \in \Nbd(x)} \XX_{/U}}
	\end{equation*}
	over $ \XX $.
\end{cnstr}

\begin{prp}\label{prop:localizationasetale}
	Let $ \XX $ be \atopos and $ \xlowerstar $ a point of $ \XX $.
	Then the comparison geometric morphism $c_{x,\ast}\colon \fromto{\Loc{\XX}{x}}{\lim_{(U,u) \in \Nbd(x)} \XX_{/U}}$ of \Cref{cnstr:limloc} is an equivalence.
\end{prp}

\begin{proof}
	We wish to show that $ c_{x,\ast} \colon \fromto{\Loc{\XX}{x}}{\lim_{(U,u) \in \Nbd(x)} \XX_{/U}} $ induces an equivalence
	\begin{equation*}
		\equivto{\Top_{\infty,/\Loc{\XX}{x}}}{\Top_{\infty,/E_x}} \period
	\end{equation*}
	Since both projections onto $ \Top_{\infty,/\XX} $ are right fibrations, we are reduced to showing that for every object $ \plowerstar \colon \fromto{\XX'}{\XX} $ of $ \Top_{\infty,/\XX} $ the induced map on fibers of these right fibrations is an equivalence.
	By \Cref{HTT.6.3.5.6}=\allowbreak\HTT{Corollary}{6.3.5.6} the fiber of the right fibration $ \fromto{\Top_{\infty,/E_x}}{\Top_{\infty,/\XX}} $
	over $ \plowerstar \colon \fromto{\XX'}{\XX} $ is given by
	\begin{equation*}
		\{\plowerstar \} \crosslimits_{\Top_{\infty,/\XX}} \Top_{\infty,/E_x} \equivalent \lim_{(U,u) \in \Nbd(x)}\Gammaup_{\XX',\ast}p^{\ast}(U) \comma
	\end{equation*}
	On the other hand, we have equivalences
	\begin{align*}
		\{\plowerstar \} \crosslimits_{\Top_{\infty,/\XX}} \Top_{\infty,/\Loc{\XX}{x}} &\equivalent \Map_{\Funlowerstar(\XX',\XX)}(\plowerstar,\xlowerstar \Gammaup_{\XX',\ast}) \\
		&\equivalent \Map_{\Fun(\XX,\Space)}(\xupperstar,\Gammaup_{\XX',\ast}\pupperstar) \period
	\end{align*}
	By the colimit formula for the stalk (\Cref{cnstr:limloc}), we have natural equivalences
	\begin{align*}
		\Map_{\Fun(\XX,\Space)}(\xupperstar,\Gammaup_{\ast}\pupperstar) &\equivalent \Map_{\Fun(\XX,\Space)}\paren{\colim_{(U,u) \in \Nbd(x)^{\op}} \Map_{\XX}(U,-),\Gammaup_{\XX',\ast}\pupperstar} \\
		&\equivalent \lim_{(U,u) \in \Nbd(x)} \Gammaup_{\XX',\ast}\pupperstar(U) \period
	\end{align*}
	Unwinding definitions, we see that the induced map on fibers 
	\begin{equation*}
		\fromto{\{\plowerstar \} \crosslimits_{\Top_{\infty,/\XX}} \Top_{\infty,/\Loc{\XX}{x}}}{\{\plowerstar \} \crosslimits_{\Top_{\infty,/\XX}} \Top_{\infty,/E_x}}
	\end{equation*}
	is an equivalence.
\end{proof}


\subsection{Coherence of localizations}\label{subsec:cohloc}

In this section we use the Grothendieck–Verdier description of the localization to deduce that $ \Loc{\XX}{x} $ is bounded coherent when $ \XX $ is.
Please note that this is not automatic from \Cref{lem:orientedcoherent}, as points of bounded coherent \topoi need not be coherent in general.

\begin{nul}\label{nul:coherentmorphismofslices}
	Let $ f \colon \fromto{U}{V} $ be a morphism between coherent objects of \atopos $ \XX $.
	Then the geometric morphism $ \flowerstar \colon \fromto{\XX_{/U}}{\XX_{/V}} $ is coherent.
\end{nul}

\begin{lem}\label{lem:boundedcoherentslice}
	Let $ \XX $ be a bounded \topos and $ U \in \XX_{<\infty } $ a truncated object of $ \XX $.
	Then the over \topos $ \XX_{/U} $ is bounded.
\end{lem}

\begin{proof}
	Indeed, if $U$ is $n$-truncated, and if $\XX$ is $N$-localic for some $N\geq n$, then $\XX_{/U}$ in $N$-localic as well (\Cref{exm:nlocalicslive}).
	The claim now follows by exhibiting $\XX$ as an inverse limit of localic \topoi.
\end{proof}

\begin{nul}\label{nul:boundedcohloclimit}
	Let $ \XX $ be a bounded coherent \topos and $ \xlowerstar $ a point of $ \XX $.
	Then the full subcategory
	\begin{equation*}
		\Nbd_{<\infty }^{\coh}(x) \subset \Nbd(x)
	\end{equation*}
	consisting of those neighborhoods $ (U,u) $ such that $ U $ is a truncated coherent object of $ \XX $ is limit-cofinal in $ \Nbd(x) $. 
	Thus \Cref{prop:localizationasetale}, \Cref{nul:coherentmorphismofslices}, and \Cref{lem:boundedcoherentslice} together show that
	\begin{equation*}
		\Loc{\XX}{x} \equivalent \lim_{(U,u) \in \Nbd_{<\infty }^{\coh}(x)} \XX_{/U}
	\end{equation*}
	is an inverse limit in $ \Top_{\infty} $ of bounded coherent \topoi and coherent geometric morphisms.
\end{nul}

From \Cref{cor:SAG.A.8.3.3}=\allowbreak\SAG{Corollary}{A.8.3.3} we deduce the following.

\begin{lem}\label{lem:localizationiscoherent}
	Let $ \XX $ be a bounded coherent \topos and $ \xlowerstar $ a point of $ \XX $.
	Then the localization $ \Loc{\XX}{x} $ is bounded coherent and the geometric morphism $ \ell_{x,\ast} \colon \fromto{\Loc{\XX}{x}}{\XX} $ is coherent.
\end{lem}


\subsection{Geometric examples of localizations}\label{subsec:localexamples}

Now we turn to examples of local \topoi coming from algebraic geometry.
For these examples, please recall \Cref{rem:1localiclocal}.

\begin{exm}[{\cite[Example 1.2(a)]{MR977478}}]
	Let $ W $ be a topological space and $ s \in W $ a \emph{special point} in the sense that the only open set of $ W $ containing $ s $ is $ W $ itself.
	Then it is immediate that the functor $ \fromto{\widetilde{W}}{\Space} $ given by taking the stalk at $ s $ is equivalent to the global sections functor.
	Hence the \topos $ \widetilde{W} $ is local with center $ \xlowerstar \colon \fromto{\Space}{\widetilde{W}} $.
\end{exm}

\begin{subexm}[{\cite[Exposé VI, 8.4.6]{MR50:7131}}]
	Let $ A $ be a local ring with maximal ideal $ \mathfrak{m} $.
	Then the point $ \mathfrak{m} $ of the Zariski space $ \Spec(A)^{\zar} $ is special.
	Hence the Zariski \topos $ (\Spec A)_{\zar} $ is local.
	Moreover, if $ \phi \colon \fromto{A}{A'} $ is a local homomorphism of local rings, then the induced geometric morphism of Zariski \topoi $ \fromto{(\Spec A')_{\zar}}{(\Spec A)_{\zar}} $ is a local geometric morphism.  
\end{subexm}



\begin{exm}[{\cite[Exposé VI, 8.4.4]{MR50:7131}}]
	Let $ X $ be a scheme and $ x \in X $.
	Then the localization of the Zariski \topos of $ X $ at the point $ x $ is the Zariski \topos of the local ring $ \Oup_{X,x} $. 
\end{exm}

\begin{exm}\label{exm:localizationishenselization}
	Let $ X $ be a scheme, and let $ x\to X $ be a point with image $x_0\in X^{\zar}$.
	Suppose $x$ is a \defn{geometric point} in the sense that the residue field $\upkappa(x)$ is a separable closure of $\upkappa(x_0)$.
	Then the localization of the étale \topos of $ X $ at $ x $ is the étale \topos of the \defn{strict localization} $ X_{(x)} \coloneq \Spec \Oup_{X,x_0}^{\sh} $.
	That is,
	\begin{equation*}
		(X_{\et})_{(x)} \simeq (X_{(x)})_{\et} \period
	\end{equation*}

	More generally, for any point $ x \to X $, the evanescent \topos \smash{$ x_{\et}\orientedtimes_{X_{\et}}X_{\et} $} admits an analogous description.
	Write $ \Oup_{X,x_0}^{\hens} $ for the hensilization of the local ring $ \Oup_{X,x_0} $, and let
	\begin{equation*}
		A_x \supset \Oup_{X,x_0}^{\hens}
	\end{equation*}
	denote the unramified extension of $ \Oup_{X,x_0}^{\hens} $ with residue field the separable closure of $ \upkappa(x_0) $ in $ \upkappa(x) $.
	Then there is an equivalence of \topoi
	\begin{equation*}
		x_{\et} \orientedtimes_{X_{\et}} X_{\et} \equivalent (\Spec A_{x})_{\et} \period
	\end{equation*}
\end{exm}

\newpage

\section{\Basechange conditions for oriented fiber products}\label{section:BC}

The goal of this chapter is to prove a basechange result for oriented fiber products of bounded coherent \topoi (\Cref{thm:BCfororientedfibs}).
Our result provides a nonabelian refinement of a basechange result of Gabber \cite[Exposé XI, Théorème 2.4]{MR3309086} as well as one of Moerdijk--Vermeulen \cite[Theorem 2(i)]{MR1731050}.
This basechange result is essential to our \textit{décollage} approach to \textit{stratified higher topoi} in \cref{sec:strattopoi}. 
So that we can first introduce the \basechange theorem in question, a detailed overview of this chapter appears at the end of \cref{subsec:BCcondition}.


\subsection{\Basechange transformations \& \basechange conditions}\label{subsec:BCcondition}

We begin by recalling the \textit{\basechange natural transformation} associated to an oriented square of \topoi.
We are mostly concerned with the `left' \basechange transformation, but have one situation in which we need to consider the `right' \basechange transformation, so we introduce them both here.

\begin{dfn}\label{def:basechangeconditions}
	Consider an oriented square of \categories:
	\begin{equation}\label{square:catBCsquare}
		\begin{tikzcd}
			A \arrow[r, "q_{\ast}" above] \arrow[d, "p_{\ast}" left] & C \arrow[d, "g_{\ast}" right] \arrow[dl, phantom, "\scriptstyle \sigma" below right, "\Longleftarrow" sloped] \\ 
			B \arrow[r, "f_{\ast}" below] & D \period
		\end{tikzcd}
	\end{equation}

	\begin{enumerate}[(\ref*{def:basechangeconditions}.1)]
		\item\label{def:basechangeconditions.1} Assume that the functors $ \flowerstar $ and $ \qlowerstar $ admit left adjoints $ \fupperstar $ and $ \qupperstar $, respectively.
		Write $ \counit_f \colon \fromto{f^{\ast}f_{\ast}}{\id_{B}}$ for the counit and $ \unit_q \colon \fromto{\id_{C}}{q_{\ast}q^{\ast}}$ for the unit.
		The \defn{left \basechange transformation}%
		\index[terminology]{basechange transformation@\basechange transformation}\index[terminology]{transformation@basechange@\basechange}
		associated to the oriented square \eqref{square:catBCsquare} is the composition\index[notation]{BCsigma@$\BC_{\sigma}$}
		\begin{equation*}
			\begin{tikzcd}[sep=2.75em]
				\BC_{\sigma} \colon f^{\ast}g_{\ast} \arrow[r, "f^{\ast}g_{\ast}\unit_q"] & f^{\ast}g_{\ast}q_{\ast}q^{\ast} \arrow[r, "f^{\ast}\sigma q^{\ast}"] & f^{\ast}f_{\ast}p_{\ast}q^{\ast} \arrow[r, "\counit_f p_{\ast}q^{\ast}"] & p_{\ast}q^{\ast} \period 
			\end{tikzcd}
		\end{equation*}	
		We say that the square \eqref{square:catBCsquare} is \defn{left adjointable}\index[terminology]{adjointable}\index[terminology]{left adjointable} or \defn{satisfies the left \basechange condition}\index[terminology]{basechange condition@\basechange condition} if the natural transformation $ \BC_{\sigma} f^{\ast}g_{\ast} \to p_{\ast}q^{\ast} $ is an equivalence.

		\item\label{def:basechangeconditions.2} Assume that the functors $ \plowerstar $ and $ \glowerstar $ admit right adjoints $ \puppershriek $ and $ \guppershriek $, respectively.
		Write $ \counit_p \colon \fromto{\plowerstar \puppershriek}{\id_{B}}$ for the counit and $ \unit_g \colon \fromto{\id_{C}}{\guppershriek \glowerstar}$ for the unit.
		The \defn{right \basechange transformation}\index[terminology]{basechange transformation@\basechange transformation} associated to the oriented square \eqref{square:catBCsquare} is the composition
		\begin{equation*}
			\begin{tikzcd}[sep=2.75em]
				\qlowerstar\puppershriek \arrow[r, "\unit_g \qlowerstar\puppershriek"] & \guppershriek\glowerstar\qlowerstar\puppershriek \arrow[r, "\guppershriek\sigma \puppershriek"] & \guppershriek\flowerstar\plowerstar\puppershriek \arrow[r, "\guppershriek\flowerstar \counit_p"] & \guppershriek\flowerstar \period 
			\end{tikzcd}
		\end{equation*}

	\end{enumerate}
\end{dfn}

\begin{rmk}
	In classical category theory, the adjointability of a commutative square of $ 1 $-categories is often referred to as the \textit{Beck--Chevalley condition}, and the basechange trransformations are often referred  to as \textit{Beck--Chevalley transformations} \cites{MR1731050}[Chapter I, \S3]{MR1787303}.
\end{rmk}

\begin{nul}\label{nul:compositeBC}
	Please observe that given oriented squares of \categories
	\begin{equation*}
		\begin{tikzcd}[sep=2.75em]
			A \arrow[r] \arrow[d] & B \arrow[d] \arrow[r] \arrow[dl, phantom, "\Longleftarrow" sloped, "\scriptstyle \sigma" below right] & C  \arrow[d] \arrow[dl, phantom, "\Longleftarrow" sloped, "\scriptstyle \sigma'" below right]  \\
			A' \arrow[r] & B' \arrow[r]  & C' 
		\end{tikzcd}
	\end{equation*}
	in which the horizontal functors all admit left adjoints, the \basechange morphism of the outer oriented rectangle is equivalent to natural transformation given by the composite of the \basechange morphisms
	\begin{equation*}
		\begin{tikzcd}[sep=2.75em]
			A \arrow[d] & B \arrow[d] \arrow[l] & C  \arrow[d] \arrow[l]  \\
			A' & B' \arrow[l] \arrow[ul, phantom, "\Longleftarrow" sloped, "\scriptstyle \BC_{\sigma}" below left] & C' \period \arrow[ul, phantom, "\Longleftarrow" sloped, "\scriptstyle \BC_{\sigma'}" below left] \arrow[l]
		\end{tikzcd}
	\end{equation*}
\end{nul}

The purpose of this chapter is to generalize the Theorem of Gabber--Illusie and Moer\-dijk--Vermeulen that oriented fiber product squares of coherent ordinary topoi satisfy the \basechange condition.
However, the \toposic generalization is a bit more subtle: exactly because coherent geometric morphisms between bounded coherent \topoi only preserve colimits of uniformly truncated filtered diagrams and not all filtered colimits (\Cref{cor:coherentmorphismscommutewithfilteredcolims}), oriented fiber product squares of bounded coherent \topoi only satisfy the weaker \textit{truncated} \basechange condition.

\begin{dfn}
	We say that an oriented square of \topoi and geometric morphisms
	\begin{equation}\label{square:laxBCsquare}
		\begin{tikzcd}
			\WW \arrow[r, "q_{\ast}" above] \arrow[d, "p_{\ast}" left] & \YY \arrow[d, "g_{\ast}" right] \arrow[dl, phantom, "\scriptstyle \sigma" below right, "\Longleftarrow" sloped] \\ 
			\XX \arrow[r, "f_{\ast}" below] & \ZZ \period
		\end{tikzcd}
	\end{equation}
	satisfies the \defn{truncated \basechange condition}%
	\index[terminology]{basechange condition@\basechange condition!truncated}
	if for every truncated object $ F \in \YY_{<\infty} $, the basechange morphism $ \BC_{\sigma}(F) \colon \fromto{\fupperstar\glowerstar(F)}{\pr_{1,\ast}\prupperstar_2(F)} $ is an equivalence in $ \XX $.
\end{dfn}

The following theorem the main result of this chapter:

\begin{thm}\label{thm:BCfororientedfibs}
	Let $f_{\ast}\colon\fromto{\XX}{\ZZ}$ and $g_{\ast}\colon\fromto{\YY}{\ZZ}$ be coherent geometric morphisms between bounded coherent \topoi.
	Then the oriented fiber product square
	\begin{equation}\label{square:orientedfiber}
		\begin{tikzcd}
			\XX\orientedtimes_{\ZZ}\YY \arrow[r, "\pr_{2,\ast}" above] \arrow[d, "\pr_{1,\ast}" left] & \YY \arrow[d, "g_{\ast}" right] \arrow[dl, phantom, "\scriptstyle \sigma" below right, "\Longleftarrow" sloped] \\ 
			\XX \arrow[r, "f_{\ast}" below] & \ZZ 
		\end{tikzcd}
	\end{equation}
	satisfies the truncated \basechange condition.
\end{thm}

\noindent By passing to $1$-localic \topoi in \Cref{thm:BCfororientedfibs}, we deduce Moerdijk and Vermeulen's $ 1 $-toposic \basechange condition \cite[Theorem 2(i)]{MR1731050}.

\begin{cor}\label{cor:BCfororientedfibsdiscrete}
	Let $f_{\ast}\colon\fromto{\XX}{\ZZ}$ and $g_{\ast}\colon\fromto{\YY}{\ZZ}$ be coherent geometric morphisms between coherent $ 1 $-topoi.
	Then the oriented fiber product square of $ 1 $-topoi
	\begin{equation*}
		\begin{tikzcd}
			\XX\orientedtimes_{\ZZ}\YY \arrow[r, "\pr_{2,\ast}" above] \arrow[d, "\pr_{1,\ast}" left] & \YY \arrow[d, "g_{\ast}" right] \arrow[dl, phantom, "\Longleftarrow" sloped] \\ 
			\XX \arrow[r, "f_{\ast}" below] & \ZZ 
		\end{tikzcd}
	\end{equation*}
	satisfies the left \basechange condition.
	That is, the \basechange natural transformation $ \fromto{\fupperstar \glowerstar}{\pr_{1,\ast} \prupperstar_2} $ is an isomorphism.
\end{cor}

\begin{proof}
	Write $ \XX' $, $ \YY' $, and $ \ZZ' $ for the $ 1 $-localic \topoi associated to $ \XX $, $ \YY $, and $ \ZZ $, respectively.
	Combining the equivalence between coherent $ 1 $-localic \topoi and coherent $ 1 $-topoi (\Cref{prop:coherent1localic}) with \Cref{thm:BCfororientedfibs} shows that the oriented fiber product square of \topoi
	\begin{equation*}
		\begin{tikzcd}
			\XX' \orientedtimes_{\ZZ'} \YY' \arrow[r] \arrow[d] & \YY' \arrow[d] \arrow[dl, phantom, "\Longleftarrow" sloped] \\ 
			\XX' \arrow[r] & \ZZ' 
		\end{tikzcd}
	\end{equation*}
	satisfies the truncated \basechange condition.
	We conclude by restricting to $ 0 $-trun\-cated objects and applying \Cref{lem:orientedproddiscrete}.
\end{proof}

We now give an overview of the rest of the chapter.
The chapter is broken into two parts: \cref{subsec:BCexamples,subsec:BCapplications,subsec:BCstable} provide examples and applications of \Cref{thm:BCfororientedfibs} that do no require understanding its proof, and \cref{subsec:localbBCc,subsec:orientedfunctoriality,subsec:proofofBC} are dedicated to the proof of \Cref{thm:BCfororientedfibs}.

\Cref{subsec:BCexamples} provides some example situations where the basechange condition for oriented fiber products can be easily verified.
\Cref{subsec:BCapplications} provides some example applications of \Cref{thm:BCfororientedfibs}; notably we generalize \cites[Exposé X, Corollaire 1.7]{MR50:7129}[Theorem 5.3]{Chough:Proper} by showing that if $ X $ and $ Y $ are coherent schemes over a separably closed field $ k $ and $ Y $ is proper, then the profinite étale homotopy type of $ X \cross_{\Spec k} Y $ coincides with the product of the profinite étale homotopy types of $ X $ and $ Y $.
In \cref{subsec:BCstable} we investigate the stable consequences of \Cref{thm:BCfororientedfibs} and deduce a generalization of the derived categories basechange theorem for oriented fiber products of Gabber--Illusie \cite[Exposé XI, Théorème 2.4]{MR3309086}.

We then embark on our proof of \Cref{thm:BCfororientedfibs}, which is inspired by the proof of the Gabber--Illusie basechange theorem.
Just like how the proof of the proper basechange theorem in étale cohomology reduces to the case where two of the schemes involved are spectra of local rings, our proof of \Cref{thm:BCfororientedfibs} reduces to the case where the \topoi $ \XX $ and $ \ZZ $ are local.
In \Cref{subsec:localbBCc} we prove that fiber product squares obtained by pulling back along a localization $ \ell_{x,\ast} \colon \fromto{\Loc{\XX}{x}}{\XX} $ satisfy the truncated \basechange condition (\Cref{cor:pullbacklocalizationBC}); this is one of the key ingredients that allows us to reduce the proof of \Cref{thm:BCfororientedfibs} to the case where $ \XX $ and $ \ZZ $ are local.
\Cref{subsec:orientedfunctoriality} discusses the functoriality of oriented fiber products in oriented morphisms of cospans that we need to deduce \Cref{thm:BCfororientedfibs} from the contents of  \Cref{subsec:localbBCc}.
In \Cref{subsec:proofofBC} we put everything together and prove \Cref{thm:BCfororientedfibs}.


\subsection{Examples of the \basechange condition}\label{subsec:BCexamples}

In this section we provide a few examples of (oriented) squares that are easily seen to satisfy the \basechange condition.
None of these examples are used in the sequel.
The first two examples are due to an observation of Gabber \cite[Exposé XI, Remarque 4.9]{MR3309086}.

\begin{exm}\label{ex:vanishingBC}
	Let $ \flowerstar \colon \fromto{\XX}{\ZZ} $ be a geometric morphism of \topoi.
	Then from the equivalence $ \near_f^{\ast} \equivalent \pr_{1,\ast} \colon \fromto{\orientedpull{\XX}{\ZZ}{\ZZ}}{\XX} $ and the fact that $ \pr_{2,\ast} \near_{f,\ast} \equivalent \flowerstar $ (\Cref{prop:nearbycycleslocal}), we have equivalences
	\begin{equation*}
		\pr_{1,\ast}\prupperstar_2 \equivalent \near_f^{\ast} \prupperstar_2 \equivalent \fupperstar \period 
	\end{equation*}
	From this we deduce the left \basechange condition for the defining oriented square of the evanescent \topos:
	\begin{equation*}
		\begin{tikzcd}
			\orientedpull{\XX}{\ZZ}{\ZZ} \arrow[r, "\pr_{2,\ast}"] \arrow[d, "\pr_{1,\ast}"'] & \ZZ \arrow[d, equals] \arrow[dl, phantom, "\Longleftarrow" sloped] \\ 
			\XX \arrow[r, "f_{\ast}"'] & \ZZ \period
		\end{tikzcd}
	\end{equation*}
\end{exm}

\begin{exm}\label{ex:covanishingBC}
	Dually, let $ \glowerstar \colon \fromto{\YY}{\ZZ} $ be a geometric morphism of \topoi.
	From \Cref{prop:conearbycyclescolocal} we see that the defining oriented square of the coëvanescent \topos $ \orientedpull{\ZZ}{\ZZ}{\YY} $ satisfies the left \basechange condition.
\end{exm}

As noted by Johnstone--Moerdijk \cite[Remark 2.5]{MR977478}, pullbacks along local geometric morphisms also satisfy the \basechange condition.

\begin{exm}\label{ex:localBC}
	Consider a pullback square of \topoi 
	\begin{equation}\label{square:localpullbackBC}
      \begin{tikzcd}[sep=2.25em, text height=1.5ex, text depth=0.25ex]
	       \XX \cross_{\ZZ} \YY  \arrow[dr, phantom, very near start, "\lrcorner", xshift=-0.25em, yshift=0.25em] \arrow[d, "\gbar_{\ast}"'] \arrow[r, "\fbar_{\ast}"] & \YY \arrow[d, "\glowerstar"] \\ 
	       \XX \arrow[r, "\flowerstar"'] & \ZZ \comma
      \end{tikzcd}
    \end{equation}
    where $ \glowerstar $ exhibits $ \YY $ as local over $ \ZZ $ with center $ \guppershriek $.
    Then by \Cref{nul:localstableunderpullback} the geometric morphism $ \gbar_{\ast} $ exhibits $ \XX \cross_{\ZZ} \YY $ as local over $ \XX $ and the center $ \gbar^{!} $ of $ \gbar_{\ast} $ satisfies $ \fbar_{\ast} \gbar^{!} \equivalent \guppershriek \flowerstar $.
    We have adjunctions
    \begin{equation*}
    	\fupperstar \glowerstar \leftadjoint \guppershriek \flowerstar \quad \text{and} \quad \gbar_{\ast} \fbar^{\ast} \leftadjoint \fbar_{\ast} \gbar^{!} \comma
    \end{equation*}
    so the equivalence $ \fbar_{\ast} \gbar^{!} \equivalent \guppershriek \flowerstar $ shows that $ \fupperstar \glowerstar \equivalent \gbar_{\ast} \fbar^{\ast} $.
    From this equivalence which we deduce the left \basechange condition for the square \eqref{square:localpullbackBC}.
\end{exm}

\begin{exm}
	Let $ \flowerstar \colon \fromto{\XX}{\ZZ} $ and $ \glowerstar \colon \fromto{\YY}{\ZZ} $ be geometric morphisms of \topoi.
	Decompose the oriented fiber product $ \orientedpull{\XX}{\ZZ}{\YY} $ as an iterated pullback
	\begin{equation}\label{diag:orientedfibdecomp}
		\begin{tikzcd}[sep=2.5em]
			\orientedpull{\XX}{\ZZ}{\YY} \arrow[d] \arrow[r] \arrow[dr, phantom, very near start, "\lrcorner", xshift=-0.25em, yshift=0.25em] & \orientedpull{\ZZ}{\ZZ}{\YY} \arrow[r] \arrow[d] \arrow[dr, phantom, very near start, "\lrcorner", xshift=-0.25em, yshift=0.25em] & \YY \arrow[d, "\glowerstar"] \\
			\orientedpull{\XX}{\ZZ}{\ZZ} \arrow[r] \arrow[d] \arrow[dr, phantom, very near start, "\lrcorner", xshift=-0.25em, yshift=0.25em] & \Path{\ZZ} \arrow[r] \arrow[d] & \ZZ \arrow[d, equals] \arrow[dl, phantom, "\Longleftarrow" sloped] \\ 
			\XX \arrow[r, "\flowerstar"'] & \ZZ \arrow[r, equals] & \ZZ \period
		\end{tikzcd}
	\end{equation}
	It follows from \Cref{ex:localBC} that local geometric morphisms are \textit{proper} \HTT{Definition}{7.3.1.4}.
	Assume that $ \glowerstar $ is a proper geometric morphism. 
	Then by applying \Cref{ex:vanishingBC} to the lower right square of \eqref{diag:orientedfibdecomp}, \Cref{ex:pathtoposislocalandcolocal,ex:localBC} to the lower left square of \eqref{diag:orientedfibdecomp}, and the properness of $ \glowerstar $ to the top squares of \eqref{diag:orientedfibdecomp}, we deduce that the three pullback squares in \eqref{diag:orientedfibdecomp} and the oriented square all satisfy the left \basechange condition, and that $ \pr_{1,\ast} \colon \fromto{\orientedpull{\XX}{\ZZ}{\YY}}{\XX} $ is a proper geometric morphism.
\end{exm}


\subsection[Applications of the basechange theorem for oriented fiber products]{Applications of the basechange theorem for oriented fiber \\ products}\label{subsec:BCapplications}

In this section we give a number of applications of our basechange theorem (\Cref{thm:BCfororientedfibs}).

\begin{exm}\label{exm:BeckChevalleyoverStone}
	Let $ \flowerstar \colon \fromto{\XX}{\ZZ} $ and $ \glowerstar \colon \fromto{\YY}{\ZZ} $ be geometric morphisms of \topoi, and assume that $ \XX $ and $ \YY $ are bounded coherent and $ \ZZ $ is Stone.
	Then by \Cref{cor:morphismtoStoneiscoherent}=\allowbreak\SAG{Corollary}{E.3.1.2}, $ \flowerstar $ and $ \glowerstar $ are automatically coherent.
	Since $ \ZZ $ is Stone, \Cref{prp:orientedfiboverStoneisfib} shows that
	\begin{equation*}
		\orientedpull{\XX}{\ZZ}{\YY} \equivalent \XX \cross_{\ZZ} \YY \period
	\end{equation*}
	Hence by \Cref{thm:BCfororientedfibs} we see that the (unoriented) pullback square
	\begin{equation}\label{sq:pullbackBCStone}
      \begin{tikzcd}[sep=2.25em]
	       \XX \cross_{\ZZ} \YY  \arrow[dr, phantom, very near start, "\lrcorner", xshift=-0.25em, yshift=0.25em] \arrow[d, "\pr_{1,\ast}"'] \arrow[r, "\pr_{2,\ast}"] & \YY \arrow[d, "\glowerstar"] \\ 
	       \XX \arrow[r, "\flowerstar"'] & \ZZ 
      \end{tikzcd}
    \end{equation}
    satisfies the truncated \basechange condition.

\end{exm}

\begin{subexm}\label{subexm:profinshapepreservesproducts}
	Set $ \ZZ = \Space $ in \Cref{exm:BeckChevalleyoverStone}, so that $ \flowerstar = \Gammaup_{\XX,\ast} $ and $ \glowerstar = \Gammaup_{\YY,\ast} $.
	Since left exact functors preserve truncated objects, we see that for any truncated space $ K $ the natural morphism
	\begin{equation*}
		\fromto{\Gammaup_{\XX,\ast} \Gammaupperstar_{\XX} \Gammaup_{\YY,\ast} \Gammaupperstar_{\YY}(K)}{\Gammaup_{\XX,\ast} \pr_{1,\ast}\prupperstar_{2} \Gammaupperstar_{\YY}(K)}
	\end{equation*}
	in $ \Space $ is an equivalence.
	Hence the natural morphism
	\begin{equation*}
		\fromto{\Shape(\XX) \of \Shape(\YY)}{\Shape(\XX \cross \YY)}
	\end{equation*}
	of prospaces becomes an equivalence after protruncation.
	Since the composition monoidal structure and cartesian monoidal structre on $ \Pro(\Space) $ coincide on the full subcategory $ \Space_{\pi}^{\wedge} $ of profinite spaces (\Cref{rec:profinspace}), we deduce that
	\begin{equation*}
		\equivto{\Shapeprofin(\XX \cross \YY)}{\Shapeprofin(\XX) \cross \Shapeprofin(\YY)} \period
	\end{equation*}
	Combining this with \Cref{cor:protruncshapeinverselim} we see that the profinite shape $ \Shapeprofin \colon \fromto{\Topbc}{\Spaceprofin} $ preserves both inverse limits and finite products.
\end{subexm}

\begin{exm}\label{exm:properbasechangeforschemes}
	Let $ k $ be a separably closed field and let $ X $ and $ Y $ be $ k $-schemes.
	Assume that $ X $ is coherent and $ Y $ is proper over $ k $.
	Then combining Chough's work generalizing the proper basechange theorem in étale cohomology to the nonabelian setting \cite[Theorem 5.3]{Chough:Proper} with \Cref{subexm:profinshapepreservesproducts}, we see that the natural geometric morphism
	\begin{equation}\label{eq:geomprodcomp}
		\fromto{(X \cross_{\Spec k} Y)_{\et}}{X_{\et} \crosslimits_{(\Spec k)_{\et}} Y_{\et} \equivalent X_{\et} \cross Y_{\et}}
	\end{equation}
	induces an equivalence on profinite shapes.
	Equivalently, the natural geometric morphism \eqref{eq:geomprodcomp} induces an equivalence on lisse sheaves (\Cref{cor:SAG.E.2.3.3}=\allowbreak\SAG{Corollary}{E.2.3.3}).
\end{exm}


\subsection{Stable consequences of nonabelian basechange}\label{subsec:BCstable}

Let $ R $ be a commutative ring and 
\begin{equation*}
	\begin{tikzcd}
		\XX\orientedtimes_{\ZZ}\YY \arrow[r, "\pr_{2,\ast}" above] \arrow[d, "\pr_{1,\ast}" left] & \YY \arrow[d, "g_{\ast}" right] \arrow[dl, phantom, "\scriptstyle \sigma" below right, "\Longleftarrow" sloped] \\ 
		\XX \arrow[r, "f_{\ast}" below] & \ZZ 
	\end{tikzcd}
\end{equation*}
an oriented fiber product square of coherent $ 1 $-topoi and coherent geometric morphisms.
Gabber and Illusie proved the following stable variant of \Cref{thm:BCfororientedfibs}: for any object $ F \in \Dup(\YY;R) $ that is bounded-above with respect to the natural \tstructure on $ \Dup(\YY;R) $ (\Cref{rec:tstructureonDR}), the basechange morphism  
\begin{equation*}
	\fromto{\fupperstar \glowerstar(F)}{\pr_{1,\ast} \prupperstar_2(F)}
\end{equation*}
is an equivalence \cite[Exposé XI, Théorème 2.4]{MR3309086}.
In this section, we explain how to deduce this result of Gabber--Illusie from \Cref{thm:BCfororientedfibs}.
We also show that the result holds more generally when $ \XX $, $ \YY $, and $ \ZZ $ are bounded coherent \topoi and $ R $ is replaced by a connective $ \Eup_1 $-ring spectrum (\Cref{prop:spectralBC,exm:GabberBC}).

The proof ultimately reduces to the fact that the \basechange morphisms are compatible with the forgetful functors from sheaves of $ R $-module spectra to sheaves of spaces; this fact is elementary, but we could not locate it elsewhere in the literature.
To explain this fact, we begin by recalling the basics of \textit{stabilization} and sheaves of $ R $-module spectra.
The reader familiar with this basic fact or more interested in the stable consequences of \Cref{thm:BCfororientedfibs} but not their proofs is encouraged to skip ahead to \Cref{prop:spectralBC}.

\begin{rec}[{stabilization \HA{Definition}{1.4.2.8}}]
	Write $ \Spacefinite \subset \Space $ for the \category of \defn{finite spaces}: the smallest full subcategory of $ \Space $ containing the terminal object and closed under finite colimits.
	Let $ C $ be \acategory with finite limits.
	Recall that the \defn{stabilization}\index[terminology]{stabilization}\index[notation]{Stab@$\Stab$} of $ C $ is the full subcategory
	\begin{equation*}
		\Stab(C) \subset \Fun(\Spacefinitept,C)
	\end{equation*}
	spanned by those functors that preserve the terminal object and carry pushout squares in $ \Spacefinitept $ to pullback squares in $ C $.
	Also recall that the functor
	\begin{equation*}
		\upOmega_C^{\infty} \colon \fromto{\Stab(C)}{C} 
	\end{equation*}
	is defined by evaluation on the $ 0 $-sphere $ \Sup^0 \in \Spacefinitept $.

	Write $ \Sp \colonequals \Stab(\Space) $\index[notation]{Sp@$\Sp$} for the \category of \defn{spectra}.\index[terminology]{spectra}
	If $ C $ is presentable, then the stabilization $ \Stab(C) $ is equivalent to the tensor product of presentable \categories $ C \tensor \Sp $ \HA{Example}{4.8.1.23}.
\end{rec}

\begin{nul}[functoriality of stabilization]\label{nul:Omegainftycommute}
	Let $ F \colon \fromto{C}{D} $ be a left exact functor between \categories with finite limits.
	Then post-composition with $ F $ defines a functor
	\begin{equation*}
		\Stab(F) \colonequals F \of - \colon \fromto{\Stab(C)}{\Stab(D)}
	\end{equation*}
	on stabilizations.
	When it does not cause confusion, we simply denote this induced functor $ \fromto{\Stab(C)}{\Stab(D)} $ by $ F $.
	It is immediate from the definitions that the square
	\begin{equation*}
      \begin{tikzcd}[sep=2.25em]
	       \Stab(C) \arrow[d, "\upOmega_C^{\infty}"'] \arrow[r, "\Stab(F)"] & \Stab(D) \arrow[d, "\upOmega_D^{\infty}"] \\ 
	       C \arrow[r, "F"'] & D 
      \end{tikzcd}
    \end{equation*}
    canonically commutes.
\end{nul}

\begin{nul}[stabilization of adjunctions]
	Let $ \adjto{F}{C}{D}{G} $ be an adjunction between \categories with finite limits, and assume that the left adjoint $ F $ is left exact.
	Then the functor $ F \colon \fromto{\Stab(C)}{\Stab(D)} $ is left adjoint to $ G \colon \fromto{\Stab(D)}{\Stab(C)} $.
\end{nul}

Now let us explain the sense in which stabilization is compatible with basechange morphisms.

\begin{nul}[stabilization of naural transformations]\label{nul:natonStab}
	Let $ F, F' \colon \fromto{C}{D} $ be left exact functors between \categories with finite limits, and let $ \sigma \colon \fromto{F}{F'} $ be a natural transformation.
	Then pointwise appication of $ \sigma $ defines a natural transformation
	\begin{equation*}
		\Stab(\sigma) \colon \fromto{\Stab(F)}{\Stab(F')} \period
	\end{equation*}
	When it does not cause confusion, we simply denote the natural transformation $ \Stab(\sigma) $ by $ \sigma $.

	It is immediate from the definitions that the natural transformation $ \Stab(\sigma) $ is compatible with $ \sigma $ in the following sense: we have a natural identification $ \upOmega_D^{\infty} \Stab(\sigma) = \sigma \upOmega_C^{\infty} $ of natural tranformations
	\begin{equation*}
		F \upOmega_C^{\infty} = \upOmega_D^{\infty} \Stab(F) \to \upOmega_D^{\infty} \Stab(F') = F' \upOmega_C^{\infty} \period
	\end{equation*}
\end{nul}

\begin{nul}[compatibility of stabilization and \basechange morphisms]\label{nul:BCcompatStab}
	Consider an oriented square of \categories and \textit{left exact} functors:
	\begin{equation*}
		\begin{tikzcd}
			A \arrow[r, "q_{\ast}" above] \arrow[d, "p_{\ast}" left] & C \arrow[d, "g_{\ast}" right] \arrow[dl, phantom, "\scriptstyle \sigma" below right, "\Longleftarrow" sloped] \\ 
			B \arrow[r, "f_{\ast}" below] & D \period
		\end{tikzcd}
	\end{equation*}
	Assume that the functors $ \flowerstar $ and $ \qlowerstar $ admit \textit{left exact} left adjoints $ \fupperstar $ and $ \qupperstar $, respectively.
	From \Cref{nul:natonStab} we see that we have a natural identification
	\begin{equation*}
		\upOmega_B^{\infty} \BC_{\Stab(\sigma)} = \BC_{\sigma} \upOmega_C^{\infty} 
	\end{equation*}
	of natural transformations
	\begin{equation*}
		f^{\ast}g_{\ast} \upOmega_C^{\infty} = \upOmega_B^{\infty} \Stab(f^{\ast}) \Stab(g_{\ast}) \to \upOmega_B^{\infty} \Stab(p_{\ast}) \Stab(q^{\ast}) = p_{\ast}q^{\ast} \upOmega_C^{\infty} \period
	\end{equation*}
\end{nul}

Now we generalize to coefficients in any connective $ \Eup_1 $-ring spectrum.

\begin{ntn}
	Let $ \XX $ be \atopos and $ R $ a connective $ \Eup_1 $-ring spectrum.
	Write:
	\begin{itemize}
		\item $ \LMod(R) $ for the \category of left $ R $-module spectra.
		(Note that if $ R $ is an ordinary associative ring, then $ \LMod(R) $ is the derived \category $ \Dup(R) $ obtained from the category of chain complexes of $ R $-modules by formally inverting the quasi-isomorphisms.)

		\item $ \Dup(\XX;R) \colonequals \XX \tensor \LMod(R) $ for the \category of sheaves of (left) $ R $-modules on $ \XX $.

		\item $ \upU_{\XX} \colon \fromto{\Dup(\XX;R)}{\Stab(\XX)} $ for the forgetful functor.
	\end{itemize}
	
	Given a geometric morphism $ \flowerstar \colon \fromto{\XX}{\ZZ} $, we simply write 
	\begin{equation*}
		\flowerstar \colon \fromto{\Dup(\XX;R)}{\Dup(\ZZ;R)}
	\end{equation*}
	for the induced right adjoint functor with left exact left adjoint.
	Note that the induced functors on sheaves of $ R $-module spectra commute with the forgetful functors in the sense that we have canonical identifications
	\begin{equation*}
		\upU_{\ZZ} \flowerstar = \flowerstar \upU_{\XX} \andeq \upU_{\ZZ} \fupperstar = \fupperstar \upU_{\XX} \period
	\end{equation*}
\end{ntn}

The analagoues of \Cref{nul:natonStab} and \Cref{nul:BCcompatStab} remain true when we forget from sheaves of $ R $-module spectra to sheaves of spectra.
The important point is the following:

\begin{nul}\label{nul:BCcompatModR}
	Given an oriented square of \topoi and geometric morphisms
	\begin{equation*}
		\begin{tikzcd}
			\WW \arrow[r, "q_{\ast}" above] \arrow[d, "p_{\ast}" left] & \YY \arrow[d, "g_{\ast}" right] \arrow[dl, phantom, "\scriptstyle \sigma" below right, "\Longleftarrow" sloped] \\ 
			\XX \arrow[r, "f_{\ast}" below] & \ZZ \comma 
		\end{tikzcd}
	\end{equation*}
	we have a natural identification
	\begin{equation*}
		\upU_{\XX} \BC_{\sigma} = \BC_{\sigma} \upU_{\YY} 
	\end{equation*}
	of natural transformations
	\begin{equation*}
		f^{\ast}g_{\ast} \upU_{\YY} = \upU_{\XX} f^{\ast} g_{\ast} \to \upU_{\XX} p_{\ast} q^{\ast} = p_{\ast}q^{\ast} \upU_{\YY} \period
	\end{equation*}
\end{nul}

Finally, to state the main results of this section, let us recall the natural \tstructure on $ \Dup(\XX;R) $.

\begin{conv}\label{conv:homological}
	We use \textit{homological} indexing conventions for our \tstructures. 
	If $ D $ is a stable \category with a \tstructure, then the shift $ G \mapsto G[1] $ is suspension, and we write $ D_{\geq n} \coloneq D_{\geq 0}[n] $ and $ D_{\leq n} \coloneq D_{\leq  0}[n] $.
\end{conv}

\begin{rec}[{\SAG{Proposition}{1.3.2.7}}]\label{rec:tstructureonDR}
	Let $ \XX $ be \atopos.
	Recall that the stabilization $ \Stab(\XX) $ has a natural \tstructure $ (\Stab(\XX)_{\geq 0}, \Stab(\XX)_{\leq 0}) $ defined by saying that $ F \in \Stab(\XX)_{\leq 0} $ if and only if $ \upOmega_{\XX}^{\infty} F $ is a $ 0 $-truncated object of $ \XX $.
	Consequently, for each integer $ n \geq 0 $, an object $ F $ of $ \Stab(\XX) $ is in $ \Stab(\XX)_{\leq n} $ if and only if $ \upOmega_{\XX}^{\infty} F $ is an $ n $-truncated object of $ \XX $.

	Let $ R $ be a connective $ \Eup_1 $-ring spectrum.
	There is a natural \tstructure on $ \Dup(\XX;R) $ given by setting 
	\begin{equation*}
		\Dup(\XX;R)_{\geq 0} \colonequals \upU_{\XX}^{-1}(\Stab(\XX)_{\geq 0}) \andeq \Dup(\XX;R)_{\leq 0} \colonequals \upU_{\XX}^{-1}(\Stab(\XX)_{\leq 0}) \period
	\end{equation*}
\end{rec}

\begin{ntn}
	Let $ S $ be a stable \category with \tstructure $ (S_{\geq 0},S_{\leq 0}) $.
	We write
	\begin{equation*}
		S_{<\infty} \colonequals \Union_{n \in \ZZup} S_{\leq n}
	\end{equation*}
	for the full subcategory of $ S $ spanned by the \defn{\tboundedabove} objects.
\end{ntn}

We are now ready to prove our refinement of the basechange theorem of Gabber--Illusie.
The proof proceeds in two steps.
First we note that it suffices to check the claim in the `universal' case where $ R $ is the sphere spectrum.
We then show that, in this case, the claim follows from the truncated \basechange condition at the level of \topoi.

\begin{prp}\label{prop:spectralBC}
	Let 
	\begin{equation}\label{square:laxBCsquaretostab}
		\begin{tikzcd}
			\WW \arrow[r, "q_{\ast}" above] \arrow[d, "p_{\ast}" left] & \YY \arrow[d, "g_{\ast}" right] \arrow[dl, phantom, "\scriptstyle \sigma" below right, "\Longleftarrow" sloped] \\ 
			\XX \arrow[r, "f_{\ast}" below] & \ZZ \period
		\end{tikzcd}
	\end{equation}
	be an oriented square of \topoi.
	If \eqref{square:laxBCsquaretostab} satisfies the truncated \basechange condition, then for any $ \Eup_1 $-ring spectrum $ R $, the left basechange morphism associated to the oriented square 
	\begin{equation*}\label{square:laxBCsquarestab}
		\begin{tikzcd}
			\Dup(\WW;R) \arrow[r, "q_{\ast}" above] \arrow[d, "p_{\ast}" left] & \Dup(\YY;R) \arrow[d, "g_{\ast}" right] \arrow[dl, phantom, "\scriptstyle \sigma" below right, "\Longleftarrow" sloped] \\ 
			\Dup(\XX;R) \arrow[r, "f_{\ast}" below] & \Dup(\ZZ;R) \period
		\end{tikzcd}
	\end{equation*}
	of stable \categories is an equivalence when restricted to $ \Dup(\YY;R)_{<\infty} \subset \Dup(\YY;R) $.
\end{prp}

\begin{proof}
	Since the forgetful functor $ \upU_{\XX} \colon \fromto{\Dup(\XX;R)}{\Stab(\XX)} $ is conservative, it suffices to show that for all $ F \in \Dup(\YY;R)_{<\infty} $, the morphism
	\begin{equation*}
		\upU_{\XX}\BC(F) \colon \fromto{\upU_{\XX} \fupperstar \glowerstar(F)}{\upU_{\XX} \plowerstar\qupperstar(F)} 
	\end{equation*}
	is an equivalence.
	By \Cref{nul:BCcompatModR}, we see that the morphism $ \upU_{\XX}\BC(F) $ is equivalent to the morphism
	\begin{equation*}
		\BC(\upU_{\YY}F) \colon \fromto{\fupperstar \glowerstar(\upU_{\YY}F)}{\plowerstar\qupperstar(\upU_{\YY}F)} 
	\end{equation*}
	in $ \Stab(\XX) $.

	To see that $ \BC(\upU_{\YY}F) $ is an equivalence, we need to show that for each integer $ n \in \ZZup $, the morphism
	\begin{equation*}
		\upOmega_{\XX}^{\infty-n}\BC(\upU_{\YY}F) \colon \fromto{\upOmega_{\XX}^{\infty-n} \fupperstar \glowerstar(\upU_{\YY}F)}{\upOmega_{\XX}^{\infty-n} \plowerstar\qupperstar(\upU_{\YY}F)} 
	\end{equation*}
	is an equivalence.
	Since both the left and right adjoint in a geometric morphism of \topoi are left exact, applying \Cref{nul:BCcompatStab} we see that the morphism $ \upOmega_{\XX}^{\infty-n}\BC(\upU_{\YY}F) $ is equivalent to the morphism
	\begin{equation*}
		\BC(\upOmega_{\YY}^{\infty-n}\upU_{\YY}F) \colon \fromto{\fupperstar \glowerstar(\upOmega_{\YY}^{\infty-n}\upU_{\YY}F)}{\plowerstar\qupperstar(\upOmega_{\YY}^{\infty-n}\upU_{\YY}F)} \period
	\end{equation*}
	The assumption that $ F \in \Dup(\YY;R)_{<\infty} $ is \tboundedabove guarantees that for all integers $ n \in \ZZup $, the object $ \upOmega_{\YY}^{\infty-n} \upU_{\YY} F $ is truncated.
	Since the square \eqref{square:laxBCsquaretostab} satisfies the truncated \basechange condition, we see that $ \BC(\upOmega_{\YY}^{\infty-n} \upU_{\YY}F) $ is an equivalence.
	This completes the proof.
\end{proof}

\begin{exm}\label{exm:GabberBC}
	Let $f_{\ast}\colon\fromto{\XX}{\ZZ}$ and $g_{\ast}\colon\fromto{\YY}{\ZZ}$ be coherent geometric morphisms between bounded coherent \topoi and let $ R $ be a connective $ \Eup_1 $-ring spectrum.
	\Cref{thm:BCfororientedfibs,prop:spectralBC} show that the left basechange morphism associated to the oriented square 
	\begin{equation*}
		\begin{tikzcd}
			\Dup(\orientedpull{\XX}{\ZZ}{\YY};R) \arrow[r, "\pr_{2,\ast}" above] \arrow[d, "\pr_{1,\ast}" left] & \Dup(\YY;R) \arrow[d, "g_{\ast}" right] \arrow[dl, phantom, "\scriptstyle \sigma" below right, "\Longleftarrow" sloped] \\ 
			\Dup(\XX;R) \arrow[r, "f_{\ast}" below] & \Dup(\ZZ;R) \period
		\end{tikzcd}
	\end{equation*}
	of stable \categories is an equivalence when restricted to $ \Dup(\YY;R)_{<\infty} \subset \Dup(\YY;R) $.
\end{exm}


\subsection{Localizations \& the truncated \basechange condition}\label{subsec:localbBCc}

The remainder of the chapter is dedicated to actually proving \Cref{thm:BCfororientedfibs}.
In this section we prove the following basechange result, which ultimately allows us to reduce to proving \Cref{thm:BCfororientedfibs} in the case where $ \XX $ and $ \ZZ $ are local and $ \flowerstar $ is a local geometric morphism.

\begin{prp}\label{cor:pullbacklocalizationBC}
	Let $ \plowerstar \colon \fromto{\WW}{\XX} $ be a coherent geometric morphism between bounded coherent \topoi.
	Then for any point $ \xlowerstar $ of $ \XX $, the pullback square
	\begin{equation*}
      \begin{tikzcd}[sep=2.25em]
	       \orientedpull{\xtilde}{\XX}{\WW}  \arrow[dr, phantom, very near start, "\lrcorner", xshift=-0.35em, yshift=0.25em] \arrow[d] \arrow[r] & \WW \arrow[d, "\plowerstar"] \\ 
	       \Loc{\XX}{x} \arrow[r, "\ell_{x,\ast}"'] & \XX
      \end{tikzcd}
    \end{equation*}
    satisfies the truncated \basechange condition.
\end{prp}

To prove \Cref{cor:pullbacklocalizationBC}, we use the Grothendieck--Verdier description of the localization (\Cref{prop:localizationasetale}) and the (obvious) fact that pullbacks along étale geometric morphisms satisfy \basechange condition to reduce to a general result on inverse limits (\Cref{prop:inverselimitBC}).


\begin{lem}\label{lem:etaleBC}
	Let $ \flowerstar \colon \fromto{\EE}{\XX} $ and $ \plowerstar \colon \fromto{\WW}{\XX} $ be geometric morphisms of \topoi.
	If $ \flowerstar $ is étale, then the pullback square 
	\begin{equation*}
      \begin{tikzcd}[sep=2.25em, text height=1.5ex, text depth=0.25ex]
	       \EE \cross_{\XX} \WW  \arrow[dr, phantom, very near start, "\lrcorner", xshift=-0.25em, yshift=0.25em] \arrow[d] \arrow[r] & \WW \arrow[d, "\plowerstar"] \\ 
	       \EE \arrow[r, "\flowerstar"'] & \XX
      \end{tikzcd}
    \end{equation*}
    satisfies the left \basechange condition.
\end{lem}

We fix some useful notation for the result.

\begin{ntn}\label{ntn:inverseBC}
	Let $ \WW,\XX \colon \fromto{I}{\Top_{\infty}} $ be diagrams of \topoi.
	For each morphism $ \alpha \colon \fromto{j}{i} $ in $ I $, we write
	\begin{equation*}
		e_{\alpha,\ast} \colon \fromto{\WW_j}{\WW_i} \quad \text{and} \quad f_{\alpha,\ast} \colon \fromto{\XX_j}{\XX_i}
	\end{equation*}
	for the transition morphisms.
	For each $ i \in I $, we write 
	\begin{equation*}
		\xi_{i,\ast} \colon \fromto{\lim_{i \in I} \WW_i}{\WW_i} \quad \text{and} \quad \pi_{i,\ast} \colon \fromto{\lim_{i \in I} \XX_i}{\XX_i}
	\end{equation*}
	for the projections.
	In addition, we assume that for each morphism $ \alpha \colon \fromto{j}{i} $ in $ I $, the functors
	\begin{equation*}
		e_{\alpha,\ast} \colon \fromto{\WW_{j}}{\WW_i} \andeq f_{\alpha,\ast} \colon \fromto{\XX_{j}}{\XX_i}
	\end{equation*}
	almost preserve filtered colimits (\Cref{def:almostcompactness}).
\end{ntn}

\begin{nul}
	Most importantly, the assumptions of \Cref{ntn:inverseBC} are valid for inverse systems of bounded coherent \topoi and coherent geometric morhisms (\Cref{cor:coherentmorphismscommutewithfilteredcolims}).
\end{nul}

\begin{prp}\label{prop:inverselimitBC}
	Keep the assumptions of \Cref{ntn:inverseBC}.
	Let $ p \colon \fromto{\WW}{\XX} $ be a natural transformation, each of whose components $ p_{i,\ast} \colon \fromto{\WW_i}{\XX_i} $ is almost preserves filtered colimits.
	If for each morphism $ \alpha \colon \fromto{j}{i} $ in $ I $, the square 
	\begin{equation}\label{sq:finiteBCsquare}
		\begin{tikzcd}
		   \WW_j \arrow[d, "p_{j,\ast}"'] \arrow[r, "e_{\alpha,\ast}"] & \WW_{i} \arrow[d, "p_{i,\ast}"] \\ 
		   \XX_j \arrow[r, "f_{\alpha,\ast}"'] & \XX_{i}
		\end{tikzcd}
	\end{equation}
	satisfies the truncated \basechange condition, then for each $ i \in I $ the induced square 
	\begin{equation*}
		\begin{tikzcd}
	       \lim_{i \in I} \WW_i \arrow[d, "\lim_{i \in I} p_{i,\ast}"'] \arrow[r, "\xi_{i,\ast}"] & \WW_{i} \arrow[d, "p_{i,\ast}"] \\ 
	       \lim_{i \in I} \XX_i \arrow[r, "\pi_{i,\ast}"'] & \XX_{i}
      	\end{tikzcd}
	\end{equation*}
	satisfies the truncated \basechange condition.
\end{prp}

\begin{proof}
	Since $ I $ is inverse, for each $ i \in I $, the forgetful functor $ \fromto{I_{/i}}{I} $ is limit-cofinal \cite[\HTTthm{Example}{5.4.5.9} \& \HTTthm{Lemma}{5.4.5.12}]{HTT}.
	Thus we may without loss of generality assume that $ I $ admits a terminal object $ 1 $ and that $ i = 1 $.
	Writing $ q_{\ast} \coloneq \lim_{i \in I} p_{i,\ast} $, we see that we have reduced to showing that the square
	\begin{equation}\label{sq:limitBC}
		\begin{tikzcd}
	       \lim_{i \in I} \WW_i \arrow[d, "q_{\ast}"'] \arrow[r, "\xi_{1,\ast}"] & \WW_{1} \arrow[d, "p_{1,\ast}"] \\ 
	       \lim_{i \in I} \XX_i \arrow[r, "\pi_{1,\ast}"'] & \XX_{1}
      \end{tikzcd}
	\end{equation}
	satisfies the truncated \basechange condition.

	Inverse limits in $ \Top_{\infty} $ are computed in $ \Cat_{\infty,\updelta_1} $ (\Cref{thm:filteredlimsinRTop}=\allowbreak\HTT{Theorem}{6.3.3.1}), so an object of the limit of a diagram $ \YY \colon \fromto{I}{\Top_{\infty}} $ is specified by a compatible system $ \{U_i\}_{i \in I} $ of objects $ U_i \in \YY_i $ along with, for each $ \alpha \colon \fromto{j}{i} $ in $ I $, an equivalence $ \phi_{\alpha}\colon g_{\alpha,\ast}(U_j) \equivalent U_i $, where $ g_{\alpha,\ast} \colon \fromto{\YY_j}{\YY_i} $ is the transition morphism.
	Thus for $ U \in \WW_1 $ we have
	\begin{equation*}
		q_{\ast} \xiupperstar_1(U) \simeq \{p_{i,\ast}\xi_{i,\ast}\xiupperstar_1(U)\}_{i \in I} \comma
	\end{equation*}
	and
	\begin{equation*}
		\piupperstar_1 p_{1,\ast}(U) \simeq \{\pi_{i,\ast}\piupperstar_1 p_{1,\ast}(U)\}_{i \in I} \period
	\end{equation*}
	It therefore suffices to show that for each $ i \in I $, the natural morphism 
	\begin{equation*}
		\pi_{i,\ast}\BC \colon \fromto{\pi_{i,\ast} \piupperstar_1 p_{1,\ast}}{\pi_{i,\ast} q_{\ast} \xiupperstar_1 \equivalent p_{i,\ast} \xi_{i,\ast} \xiupperstar_1}
	\end{equation*}
	induced by the \basechange morphism $ \BC \colon \fromto{\piupperstar_1 p_{1,\ast}}{q_{\ast} \xiupperstar_1} $ is an equivalence when restricted to $ (\WW_{1})_{<\infty} $.

	For $ i \in I $ and $\alpha\colon i\to 1$ the unique morphism, we simply write $ f_{i,\ast} \coloneq f_{\alpha,\ast} $ and $ e_{i,\ast} \coloneq e_{\alpha,\ast} $.
	Note that for every truncated object $ W \in (\WW_{1})_{<\infty} $ we have equivalences
	\begin{align*}
		\pi_{i,\ast} \piupperstar_1 p_{1,\ast}(U) &\equivalent \pi_{i,\ast} \piupperstar_i \fupperstar_i p_{1,\ast}(U)  \\
		&\equivalent \colim_{\alpha \in (I_{/i})^{\op}} f_{\alpha,\ast} \fupperstar_{\alpha} \fupperstar_i p_{1,\ast}(U) && \text{(\Cref{lem:filteredcolimdescription})} \\
		&\equivalence \colim_{[\alpha \colon j \to i ]\in (I_{/i})^{\op}} f_{\alpha,\ast} p_{j,\ast} e_{\alpha}^{\ast} \fupperstar_i(U) \\
		&\equivalent \colim_{\alpha \in (I_{/i})^{\op}} p_{i,\ast} e_{\alpha,\ast} e_{\alpha}^{\ast} e_i^{\ast}(U) \comma
	\end{align*}
	where the third equivalence is by the assumption that the square \eqref{sq:finiteBCsquare} satisfies the truncated basechange condition.
	In addition, \Cref{lem:filteredcolimdescription} and the fact that $ \xiupperstar_i \fupperstar_i \equivalent \xiupperstar_1 $ give equivalences
	\begin{align*}
		p_{i,\ast}\paren{\colim_{\alpha \in (I_{/i})^{\op}} e_{\alpha,\ast} e_{\alpha}^{\ast} e_i^{\ast}(U)} &\equivalent p_{i,\ast} \xi_{i,\ast} \xiupperstar_i \fupperstar_i \equivalent p_{i,\ast} \xi_{i,\ast} \xiupperstar_1(U) 
	\end{align*}
	for every truncated object $ U \in (\WW_{1})_{<\infty} $.
	As left exact functors preserve $ n $-truncatedness for all $ n \geq -2 $, and each $ p_{i,\ast} $ almost preserves filtered colimits by assumption, we see that for every truncated object $ U $ of $ \WW_1 $, the natural morphism
	\begin{align*}
		\fromto{\colim_{\alpha \in (I_{/i})^{\op}} p_{i,\ast} e_{\alpha,\ast} e_{\alpha}^{\ast} e_i^{\ast}(U)}{p_{i,\ast}\paren{\colim_{\alpha \in (I_{/i})^{\op}} e_{\alpha,\ast} e_{\alpha}^{\ast} e_i^{\ast}(U)}}
	\end{align*}
	is an equivalence.
	This provides an equivalence 
	\begin{equation}\label{eq:BCequivontruncatedobjs}
		\equivto{\pi_{i,\ast} \piupperstar_1 p_{1,\ast}(U)}{p_{i,\ast} \xi_{i,\ast} \xiupperstar_1(U)} \period
	\end{equation}
	To conclude, note that the equivalence \eqref{eq:BCequivontruncatedobjs} is homotopic to $ \pi_{1,\ast} \BC(U) $.
\end{proof}

\begin{proof}[Proof of \Cref{cor:pullbacklocalizationBC}]
	Combine \Cref{lem:etaleBC} and \Cref{prop:inverselimitBC}; note that the hypotheses of \Cref{prop:inverselimitBC} are valid by \Cref{nul:boundedcohloclimit,cor:coherentmorphismscommutewithfilteredcolims} (cf. \Cref{cor:SAG.A.8.3.3}=\allowbreak\SAG{Corollary}{A.8.3.3}).
\end{proof}


\subsection{Functoriality of oriented fiber products in oriented diagrams}\label{subsec:orientedfunctoriality}

In this section we discuss the functoriality of the oriented fiber product in oriented diagrams of cospans.
Then we use this additional functoriality to construct some unexpected extra adjoints to the second projection from the oriented fiber product (\Cref{prop:pr2iscoessential}).
In nice cases, this provides a way to check that the \basechange morphism becomes an equivalence after passing to stalks (\Cref{lem:orientedlocalstalk}); this is key to our proof of \Cref{thm:BCfororientedfibs}.

The main results of this section generalize and refine results of Gabber--Illusie \cite[Exposé XI, Proposition 2.3]{MR3309086}.

\begin{nul}\label{nul:laxfunctoriality}
	Suppose that we are given a diagram of \topoi and natural transformations
	\begin{equation*}
		\begin{tikzcd}[sep=2.75em]
			\XX \arrow[r, "\flowerstar"] \arrow[d, "\xlowerstar"'] & \ZZ \arrow[d, "\zlowerstar" description] \arrow[dl, phantom, "\scriptstyle \eta" above left, "\Longleftarrow" sloped] & \YY \arrow[l, "\glowerstar"'] \arrow[d, "\ylowerstar"]  \\
			\XX' \arrow[r, "\flowerstar'"'] & \ZZ' & \YY' \period \arrow[l, "\glowerstar'"] \arrow[ul, phantom, "\scriptstyle \theta" above right, "\Longleftarrow" sloped]
		\end{tikzcd}
	\end{equation*}
	Then by the universal property of the oriented fiber product $ \orientedpull{\XX'}{\ZZ'}{\YY'} $, the diagram
	\begin{equation*}
		\begin{tikzcd}[sep=2em]
			\orientedpull{\XX}{\ZZ}{\YY} \arrow[dd, "\pr_{1,\ast}"'] \arrow[rr, "\pr_{2,\ast}"] & & \YY \arrow[dr, "\ylowerstar"] \arrow[dd, "\glowerstar" description] \arrow[ddll, phantom, "\scriptstyle \sigma" above left, "\Longleftarrow" sloped] & \\
			 & &  & \YY' \arrow[dd, "\glowerstar'"] \arrow[dl, phantom, "\scriptstyle \theta" above left, "\Longleftarrow" sloped] \\ 
			\XX \arrow[rr, "\flowerstar" description] \arrow[dr, "\xlowerstar"'] & & \ZZ \arrow[dr, "\zlowerstar" description] \arrow[dl, phantom, "\scriptstyle \eta" above left, "\Longleftarrow" sloped] & \\
			    & \XX' \arrow[rr, "\flowerstar'"'] & & \ZZ' 
		\end{tikzcd}
	\end{equation*} 
	(functorially) induces a geometric morphism $ \fromto{\orientedpull{\XX}{\ZZ}{\YY}}{\orientedpull{\XX'}{\ZZ'}{\YY'}} $.
	We simply denote this geometric morphism by
	\begin{equation*}
		\orientedpull{\xlowerstar}{\zlowerstar}{\ylowerstar} \colon \fromto{\orientedpull{\XX}{\ZZ}{\YY}}{\orientedpull{\XX'}{\ZZ'}{\YY'}} \comma
	\end{equation*}
	leaving the natural transformations $ \eta $ and $ \theta $ implicit.
	Please note that the geometric morphism $ \orientedpull{\xlowerstar}{\zlowerstar}{\ylowerstar} $ satisfies the obvious relations
	\begin{equation*}
		\pr_{1,\ast} \circ \, (\orientedpull{\xlowerstar}{\zlowerstar}{\ylowerstar}) \equivalent \xlowerstar \pr_{1,\ast} \quad \text{and} \quad \pr_{2,\ast} \circ \, (\orientedpull{\xlowerstar}{\zlowerstar}{\ylowerstar}) \equivalent \ylowerstar \pr_{2,\ast} \period
	\end{equation*}
\end{nul}

\begin{nul}\label{nul:coessentialcospan}
	Suppose that we are given a commutative diagram of \topoi
	\begin{equation}\label{diag:morphismofcospans}
		\begin{tikzcd}[sep=2.75em]
			\XX \arrow[r, "\flowerstar"] \arrow[d, "\xlowerstar"'] & \ZZ \arrow[d, "\zlowerstar" description] & \YY \arrow[l, "\glowerstar"'] \arrow[d, "\ylowerstar"]  \\
			\XX' \arrow[r, "\flowerstar'"'] & \ZZ' & \YY' \comma \arrow[l, "\glowerstar'"]
		\end{tikzcd}
	\end{equation}
	and assume that $ \xlowerstar $, $ \ylowerstar $, and $ \zlowerstar $ are coëssential with centers $ \xuppershriek $, $ \yuppershriek $, and $ \zuppershriek $, respectively. 
	Using the \textit{right} basechange morphisms with respect to the adjunctions $ \xlowerstar \leftadjoint \xuppershriek $, $ \ylowerstar \leftadjoint \yuppershriek $, and $ \zlowerstar \leftadjoint \zuppershriek $ \enumref{def:basechangeconditions}{2} , we obtain a pair of oriented squares 
	\begin{equation}\label{diag:wrongwaylaxsquares}
		\begin{tikzcd}[sep=2.75em]
			\XX' \arrow[r, "\flowerstar'"] \arrow[d, "\xuppershriek"'] & \ZZ' \arrow[d, "\zuppershriek" description] \arrow[dl, phantom, "\Longrightarrow" sloped] & \YY' \arrow[l, "\glowerstar'"'] \arrow[d, "\yuppershriek"]  \\
			\XX \arrow[r, "\flowerstar"'] & \ZZ & \YY \period \arrow[l, "\glowerstar"] \arrow[ul, phantom, "\Longleftarrow" sloped]
		\end{tikzcd}
	\end{equation}
	Note that the natural transformation in the left-hand square of \eqref{diag:wrongwaylaxsquares} points in the \textit{wrong} direction to apply \Cref{nul:laxfunctoriality}.	
\end{nul}

\begin{nul}\label{nul:goodcoessentialcospan}
	Keep the notations of \Cref{nul:coessentialcospan}, and additionally assume that the natural transformation in the left-hand square of \eqref{diag:wrongwaylaxsquares} is an equivalence, so that $ \isomto{\flowerstar \xuppershriek}{\zuppershriek \flowerstar'} $.
	Then by the functoriality of the oriented fiber product in oriented diagrams \Cref{nul:laxfunctoriality}, the diagram \eqref{diag:wrongwaylaxsquares} defines a geometric morphism $ \orientedpull{\xuppershriek}{\zuppershriek}{\yuppershriek} \colon \fromto{\orientedpull{\XX'}{\ZZ'}{\YY'}}{\orientedpull{\XX}{\ZZ}{\YY}} $.	
\end{nul}

The following is now formal.

\begin{prp}\label{prop:pr2iscoessential} 
	With the notations and assumptions of \Cref{nul:goodcoessentialcospan}, the geometric morphism
	\begin{equation*}
		\orientedpull{\xlowerstar}{\zlowerstar}{\ylowerstar} \colon \fromto{\orientedpull{\XX}{\ZZ}{\YY}}{\orientedpull{\XX'}{\ZZ'}{\YY'}}
	\end{equation*}
	is coëssential with center $ \orientedpull{\xuppershriek}{\zuppershriek}{\yuppershriek} \colon \fromto{\orientedpull{\XX'}{\ZZ'}{\YY'}}{\orientedpull{\XX}{\ZZ}{\YY}} $.
\end{prp}


We now explain a particular application of \Cref{prop:pr2iscoessential} that allows us to show that if $ \flowerstar \colon \fromto{\XX}{\ZZ} $ is a local geometric morphism of local \topoi and $ \glowerstar \colon \fromto{\YY}{\ZZ} $ is \textit{any} geometric morphism, then the second projection exhibits $ \orientedpull{\XX}{\ZZ}{\YY} $ as local over $ \YY $.

\begin{nul}\label{nul:definesigma}
	Let $ \flowerstar \colon \fromto{\XX}{\ZZ} $ be a local geometric morphism of local \topoi with centers $ \xlowerstar $ and $ \zlowerstar $, respectively, and let $ \glowerstar \colon \fromto{\YY}{\ZZ} $ be a geometric morphism of \topoi.
	Note that all of the vertical geometric morphisms in the commutative diagram of \topoi 
	\begin{equation*}
		\begin{tikzcd}[sep=2.75em]
			\XX \arrow[r, "\flowerstar"] \arrow[d, "\Gammaup_{\XX,\ast}"'] & \ZZ \arrow[d, "\Gammaup_{\ZZ,\ast}" description] & \YY \arrow[l, "\glowerstar"'] \arrow[d, equals]  \\
			\Space \arrow[r, equals] & \Space & \YY \arrow[l, "\Gammaup_{\YY,\ast}"] 
		\end{tikzcd}
	\end{equation*}
	exhibit the top \topoi as local over the bottom \topoi.
	Since $ \flowerstar $ is a local geometric morphism, applying the discussion of \Cref{nul:coessentialcospan} shows that we are in the situation of \Cref{nul:goodcoessentialcospan}.
	That is to say $ \xlowerstar $, $ \zlowerstar $, and $ \id_{\YY} $ induce a geometric morphism 
	\begin{equation*}
		\orientedpull{\xlowerstar}{\zlowerstar}{\id_{\YY}} \colon \fromto{\YY \equivalent \orientedpull{\Space}{\Space}{\YY}}{\orientedpull{\XX}{\ZZ}{\YY}} \period
	\end{equation*}

\end{nul}

The following is our generalization of \cite[Exposé XI, Proposition 2.3]{MR3309086}.
Note that this generalization is not just \toposic: in our version we don't need to take stalks.

\begin{lem}\label{lem:pr2islocal}
	With the notations of \Cref{nul:definesigma}, the second projection $ \pr_{2,\ast} \colon \fromto{\orientedpull{\XX}{\ZZ}{\YY}}{\YY} $ exhibits $ \orientedpull{\XX}{\ZZ}{\YY} $ as local over $ \YY $ with center
	\begin{equation*}
		\orientedpull{\xlowerstar}{\zlowerstar}{\id_{\YY}} \colon \fromto{\YY \equivalent \orientedpull{\Space}{\Space}{\YY}}{\orientedpull{\XX}{\ZZ}{\YY}} \period
	\end{equation*}
\end{lem}

\begin{proof}
	The fact that $ \pr_{2,\ast} $ is coëssential with center $ \orientedpull{\xlowerstar}{\zlowerstar}{\id_{\YY}} $ is immediate from \Cref{prop:pr2iscoessential}.
	The full faithfulness of $ \orientedpull{\xlowerstar}{\zlowerstar}{\id_{\YY}} $ follows from the equivalence
	\begin{equation*}
		\pr_{2,\ast} \of \, (\orientedpull{\xlowerstar}{\zlowerstar}{\id_{\YY}}) \equivalent \id_{\YY} \period \qedhere
	\end{equation*}
\end{proof}

In the setting of \Cref{lem:pr2islocal}, we deduce that the \basechange morphism becomes an equivalence after taking its stalk at the center of $ \XX $.

\begin{lem}\label{lem:orientedlocalstalk}
	Consider an oriented square of \topoi
	\begin{equation*}
		\begin{tikzcd}
			\WW \arrow[r, "q_{\ast}" above] \arrow[d, "p_{\ast}" left] & \YY \arrow[d, "g_{\ast}" right] \arrow[dl, phantom, "\scriptstyle \sigma" below right, "\Longleftarrow" sloped] \\ 
			\XX \arrow[r, "f_{\ast}" below] & \ZZ \comma
		\end{tikzcd}
	\end{equation*}
	where $ \qlowerstar $ is a quasi-equivalence, $ \XX $ and $ \ZZ $ are local with centers $ \xlowerstar $ and $ \zlowerstar $, respectively, and $ \flowerstar $ is a local geometric morphism.
	Then the natural transformation
	\begin{equation*}
		\xupperstar \BC_{\sigma} \colon \fromto{\xupperstar \fupperstar \glowerstar}{\xupperstar \plowerstar \qupperstar}
	\end{equation*}
	is an equivalence.
\end{lem}

\begin{proof}
	We prove the stronger claim that $ \xupperstar \fupperstar \glowerstar \equivalent \xupperstar \plowerstar \qupperstar $ and the space of natural transformations $ \fromto{\xupperstar \fupperstar \glowerstar}{\xupperstar \plowerstar \qupperstar} $ is contractible.
	Since $ \ZZ $ is local we have equivalences
	\begin{equation*}
		\xupperstar \fupperstar \glowerstar \equivalent \zupperstar \glowerstar \equivalent \Gammaup_{\ZZ,\ast} \glowerstar \equivalent \Gammaup_{\YY,\ast} \period
	\end{equation*} 
	Since $ \XX $ is local and $ \qlowerstar $ is a quasi-equivalence, applying \Cref{lem:globalsecconnected} we have equivalences 
	\begin{equation*}
		\xupperstar \plowerstar \qupperstar \equivalent \Gammaup_{\XX,\ast} \plowerstar \qupperstar \equivalent \Gammaup_{\WW,\ast} \qupperstar \equivalent \Gammaup_{\YY,\ast} \period
	\end{equation*} 
	Thus both $ \xupperstar \fupperstar \glowerstar $ and $ \xupperstar \plowerstar \qupperstar $ are equivalent to the global sections functor on $ \YY $.
	We are now done since $ \Gammaup_{\YY,\ast} $ is corepresented by the terminal object of $ \YY $.
\end{proof}


\subsection{Proof of the \basechange condition for oriented fiber products}\label{subsec:proofofBC}

This section is devoted to the proof of \Cref{thm:BCfororientedfibs}.

\begin{proof}[Proof of \Cref{thm:BCfororientedfibs}]
	Write $ \BC \colon \fromto{\fupperstar \glowerstar}{\pr_{1,\ast}\prupperstar_2} $ for the left \basechange natural transformation of the oriented fiber product square \eqref{square:orientedfiber}.
	Notice that since $ \XX $ is bound\-ed coherent, left exact functors preserve truncated objects, and morphisms between truncated objects are truncated, Deligne Completeness \Cref{nul:checkinftyconnonstalks} shows that to prove the claim it suffices to show that for every point $ \xlowerstar \in \Pt(\XX) $ and truncated object $ F \in \YY_{<\infty} $, the morphism
	\begin{equation*}
		\xupperstar \BC(F) \colon \fromto{\xupperstar \fupperstar \glowerstar(F)}{\xupperstar \pr_{1,\ast}\prupperstar_2(F)}
	\end{equation*}
	is an equivalence in $ \Space $.
	We prove this by localizing $ \XX $ at the point $ \xlowerstar $ and reducing to the case where $ \XX $ and $ \ZZ $ are local and $ \flowerstar $ is a local geometric morphism; the claim then follows from \Cref{lem:orientedlocalstalk}.
		
	To reduce to the local case, fix a point $ \xlowerstar \in \Pt(\XX) $, define $ \zlowerstar \coloneq \flowerstar \xlowerstar $, and let $ \fbar_{\ast} \colon \fromto{\Loc{\XX}{x}}{\Loc{\ZZ}{z}} $ be the induced geometric morphism on localizations.
	To simplify notation we write
	\begin{equation*}
		\WW \coloneq \orientedpull{\XX}{\ZZ}{\YY} \comma \qquad \Loc{\WW}{x} \coloneq \Loc{\XX}{x} \cross_{\XX} \WW \comma \andeq \Loc{\YY}{z} \coloneq \Loc{\ZZ}{z} \cross_{\ZZ} \YY \period
	\end{equation*}
	Consider the cube
	\begin{equation}\label{cube:IllusieFundamentalCube}
		\begin{tikzcd}[sep=2.25em]
			& \WW \arrow[rr, "\pr_{2,\ast}"] \arrow[dd, "\pr_{1,\ast}"' near end] & & \YY \arrow[dd, "\glowerstar"] \\
			\Loc{\WW}{x} \arrow[ur, "\ellbar_{x,\ast}"] \arrow[dd, "\plowerstar"'] \arrow[rr, crossing over, "\qlowerstar" near end] & & \Loc{\YY}{z} \arrow[ur, "\ellbar_{z,\ast}" description] \\
			& \XX \arrow[rr, "\flowerstar" near start] & & \ZZ \\
			\Loc{\XX}{x} \arrow[ur, "\ell_{x,\ast}" description] \arrow[rr, "\fbar_{\ast}"'] & & \Loc{\ZZ}{z} \arrow[ur, "\ell_{z,\ast}"'] \arrow[from=uu, crossing over, "\gbar_{\ast}" near start] & \phantom{\ZZ} \comma 
		\end{tikzcd}
	\end{equation}
	formed by pulling back the back vertical face along the bottom horizontal face.
	In the cube \eqref{cube:IllusieFundamentalCube}, the front vertical face is an oriented square, the back vertical face is an oriented fiber product square, all other vertical faces are commutative, and the side faces are pullback squares.
	Moreover, the cube satisfies the following property:
	\begin{enumerate}
		\item[($ \ast $)]\label{thm:BCfororientedfibs.3} The natural transformation between the right adjoints given by the composite of the back and left faces of \eqref{cube:IllusieFundamentalCube} is equivalent to the natural transformation given by the composite of the front and right faces of \eqref{cube:IllusieFundamentalCube}.
	\end{enumerate}
	
	We claim that the front vertical face of \eqref{cube:IllusieFundamentalCube} is an oriented fiber product square. 
	To see this, note that by \Cref{prop:localizationasetale}, the compatibility of the oriented fiber product with limits \Cref{nul:orientedcommuteswithlimits}, the compatibility of oriented fiber products with étale geometric morphisms (\Cref{prop:orientedslice}), and \Cref{lem:orientedsliceequiv}, we have equivalences
	\begin{align*}
		\orientedpull{\Loc{\XX}{x}}{\Loc{\ZZ}{z}}{\Loc{\YY}{z}} &\equivalent \paren{\lim_{U \in \Nbd(x)} \XX_{/U}} \orientedtimes_{\lim_{V \in \Nbd(z)} \ZZ_{/V}} \paren{\lim_{V \in \Nbd(z)} \YY_{/\gupperstar(V)}} \\ 
		&\equivalent \lim_{U \in \Nbd(x)} \lim_{V \in \Nbd(z)} \paren{ \orientedpull{\XX_{/U}}{\ZZ_{/V}}{\YY_{/\gupperstar(V)}} } \\
		&\equivalent \lim_{U \in \Nbd(x)} \lim_{V \in \Nbd(z)} (\orientedpull{\XX}{\ZZ}{\YY})_{/\orientedpull{U}{V}{\gupperstar(V)}} \\
		&\equivalent \lim_{U \in \Nbd(x)} \lim_{V \in \Nbd(z)} (\orientedpull{\XX}{\ZZ}{\YY})_{/\prupperstar_1(U)} \\ 
		&\equivalent \Loc{\XX}{x} \cross_{\XX} \WW = \Loc{\WW}{x} \period
	\end{align*}
	
	Also note that by applying \Cref{lem:pr2islocal} to the front face of \eqref{cube:IllusieFundamentalCube}, we deduce that $ \qlowerstar \colon \fromto{\Loc{\WW}{x}}{\Loc{\YY}{z}} $ exhibits $ \Loc{\WW}{x} $ as local over $ \Loc{\YY}{z} $.

	Now we define natural transformations
	\begin{equation*}
		\alpha^R \colon \fromto{\xupperstar \fupperstar \glowerstar}{\Gammaup_{\Loc{\XX}{x},\ast}\fbar^{\ast} \gbar_{\ast} \ellbar_z^{\ast}} \quad \text{and} \quad \alpha^L \colon \fromto{\xupperstar \pr_{1,\ast} \prupperstar_2}{\Gammaup_{\Loc{\XX}{x},\ast}\plowerstar \qupperstar \ellbar_z^{\ast}} \comma
	\end{equation*}
	which are both equivalences when restricted to $ \YY_{<\infty} $, as follows.
	Write $ \BC^R $ for the \basechange morphism of the right-hand vertical face of \eqref{cube:IllusieFundamentalCube} and $ \BC^L $ for the \basechange morphism of the left-hand vertical face.
	Since the bottom horizontal face of \eqref{cube:IllusieFundamentalCube} commutes, under identification of left adjoints, $ \BC^R $ defines a natural transformation
	\begin{equation*}
		\fbar^{\ast} \BC^R \colon \fromto{\ell_x^{\ast} \fupperstar \glowerstar \equivalent \fbar^{\ast} \ell_z^{\ast}\glowerstar}{\fbar^{\ast} \gbar_{\ast} \ellbar_z^{\ast}} \period
	\end{equation*}
	Let $ \alpha^R $ be the composite
	\begin{equation*}
		\begin{tikzcd}[sep=1.5em]
			\alpha^R \colon \xupperstar \fupperstar \glowerstar \arrow[r, "\sim"{yshift=-0.2em}] & \Gammaup_{\Loc{\XX}{x},\ast}\fbar^{\ast}\ell_z^{\ast} \glowerstar \arrow[rrr, "\Gammaup_{\Loc{\XX}{x},\ast}\fbar^{\ast} \BC^R"] & & & \Gammaup_{\Loc{\XX}{x},\ast}\fbar^{\ast} \gbar_{\ast} \ellbar_z^{\ast} \comma
		\end{tikzcd}
	\end{equation*}
	where the left-hand equivalence is by \Cref{lem:stalksasglobalsections} and the fact that $ \zupperstar = \xupperstar \fupperstar $.
	By \Cref{cor:pullbacklocalizationBC}, $ \BC^R $ is an equivalence when restricted to $ \YY_{<\infty} $; therefore $ \alpha^R $ is also an equivalence when restricted to $ \YY_{<\infty} $.
	Similarly, since the top horizontal face of \eqref{cube:IllusieFundamentalCube} commutes, under identification of left adjoints, $ \BC^L $ defines a natural transformation
	\begin{equation*}
		\BC^L \prupperstar_2 \colon \fromto{\ell_x^{\ast} \pr_{1,\ast}\prupperstar_2}{\plowerstar \ellbar_x^{\ast} \prupperstar_2 \equivalent \plowerstar \qupperstar \ellbar_z^{\ast}} \period
	\end{equation*}
	Let $ \alpha^L $ be the composite
	\begin{equation*}
		\begin{tikzcd}[sep=1.5em]
			\alpha^L \colon \xupperstar \pr_{1,\ast} \prupperstar_2 \arrow[r, "\sim"{yshift=-0.2em}] & \Gammaup_{\Loc{\XX}{x},\ast}\ell_x^{\ast} \pr_{1,\ast} \prupperstar_2 \arrow[rrr, "\Gammaup_{\Loc{\XX}{x},\ast}\BC^L \prupperstar_2"] & & & \Gammaup_{\Loc{\XX}{x},\ast}\plowerstar \qupperstar \ellbar_z^{\ast} \comma
		\end{tikzcd}
	\end{equation*}
	where the left-hand equivalence is ensured by \Cref{lem:stalksasglobalsections}.
	By \Cref{cor:pullbacklocalizationBC}, the natural transformation $ \BC^L $ is an equivalence when restricted to $ \WW_{<\infty} $.
	Since the functor $ \prupperstar_2 $ is left exact we see that $ \alpha^L $ is an equivalence when restricted to $ \YY_{<\infty} $.
	
	Write $ \BC^F \colon \fromto{\fbar^{\ast} \gbar_{\ast}}{\plowerstar\qupperstar} $ for the \basechange morphism for the front vertical face of the cube \eqref{cube:IllusieFundamentalCube}.
	Since $ \qlowerstar \colon \fromto{\Loc{\WW}{x}}{\Loc{\YY}{z}} $ exhibits $ \Loc{\WW}{x} $ as local over $ \Loc{\YY}{z} $, \Cref{lem:orientedlocalstalk} shows that the natural transformation 
	\begin{equation*}
		\Gammaup_{\Loc{\XX}{x},\ast} \BC^F \colon \fromto{\Gammaup_{\Loc{\XX}{x},\ast}\fbar^{\ast} \gbar_{\ast}}{\Gammaup_{\Loc{\XX}{x},\ast}\plowerstar \qupperstar}
	\end{equation*}
	is an equivalence.
	Since $ \alpha^R $ and $ \alpha^L $ are equivalences when restricted to $ \YY_{<\infty} $, to complete the proof it suffices to show that the square
	\begin{equation*}
		\begin{tikzcd}[sep=2.5em]
			\xupperstar \fupperstar \glowerstar \arrow[d, "\xupperstar \BC"'] \arrow[r, "\alpha^R"] & \Gammaup_{\Loc{\XX}{x},\ast}\fbar^{\ast} \gbar_{\ast} \ellbar_z^{\ast} \arrow[d, "\wr"'{xshift=0.1em}, "\Gammaup_{\Loc{\XX}{x},\ast} \BC^F \ellbar_z^{\ast}"] \\
			\xupperstar \pr_{1,\ast} \prupperstar_2 \arrow[r, "\alpha^L"'] & \Gammaup_{\Loc{\XX}{x},\ast}\plowerstar \qupperstar \ellbar_z^{\ast}
		\end{tikzcd}
	\end{equation*}
	commutes.
	This is immediate from the property (\hyperref[thm:BCfororientedfibs.3]{$ \ast $}) combined with \Cref{nul:compositeBC}.
\end{proof}


\newpage

\part{Stratified higher topos theory}\label{part:strattopoi}

In this part, we import the theory of stratifications into higher topos theory (\cref{sec:strattopoi}).
In \cref{sec:spectraltopoi} we introduce a class of bounded coherent \topoi called \textit{spectral} \topoi.
These are the bounded coherent stratified \topoi all of whose strata are Stone \topoi.
The chief example of a spectral \topos is the étale \topos of a coherent scheme (\Cref{exm:etaleisspectra}).
We then prove our \Categorical Hochster Duality Theorem (\Cref{thm:inftyHochster}) which shows that the \category of profinite stratified spaces is equivalent to the \category of spectral \topoi.
In \cref{sec:profinstratshape} we use \Categorical Hochster Duality to provide a stratified refinement of the profinite shape -- the \textit{profinite stratified shape}, and provide stratified refinement of the main results on the profinite shape discussed in \cref{subsec:profinshape}.


\section{Stratified higher topoi}\label{sec:strattopoi}

In this chapter we introduce the theory of stratifications for higher topoi.
Much of the theory of stratified \topoi introduced here closely resembles the theory of stratified spaces introduced in \Cref{sec:hothystratspaces}.
The main result of this chapter is to make this relationship precise: we prove that the assignment $ \goesto{\Pi}{\Fun(\Pi,\Space)} $ defines a fully faithful embedding of profinite stratified spaces into stratified \topoi (\Cref{prp:fullfaithfulnessoftilde}). 
In the following chapter, we provide an intrinsic characterization of the stratified \topoi that arise via this embedding.

\Cref{sec:Stilde} studies the \topos of sheaves on a spectral topological space.
In \cref{subset:toposstratspectral} we introduce stratified \topoi and explore their basic properties.
\Cref{subsec:naturalstrat} explains why every coherent \topos is naturally stratified by its $ 0 $-localic reflection.
\Cref{subsec:Pitilde} explains how stratified spaces give rise to stratified \topoi.
In particular, we show that if $ \fromto{\Pi}{P} $ is an $ n $-truncated \pifinite $ P $-stratified space, then the \topos $ \Fun(\Pi,\Space) $ is $ n $-localic and coherent (\Cref{cor:Pitildeisbcc}).
\Cref{subsec:BCgluing} introduces a key class of oriented squares of \topoi that are both oriented fiber product squares and oriented pushout squares.
\Cref{sec:toposicdec} uses these \textit{gluing squares} to explain a décollage approach to stratified \topoi.
In \cref{subsec:toposicnerve} we explain how to extract a toposic décollage from a stratified \topos and prove that the resulting nerve functor is an equivalence.
In \cref{subsec:lambdahatff} we conclude the chapter by showing that profinite stratified spaces embed into stratified \topoi.


\subsection{Higher topoi attached to finite posets \& spectral topological spaces}\label{sec:Stilde}

\begin{nul}
	A sheaf on a finite poset $P$ (with its Alexandroff topology of \Cref{dfn:Alexandroff}) is determined by its values on the principal open sets.
	These values coincide with the stalks of the sheaf.
	Precisely, the principal opens form a basis for the topology on $ P $; moreover, and the assignment $\goesto{p}{P_{\geq p}}$ defines a fully faithful functor $ i_P \colon \incto{P}{\Open(P)^{\op}} $, and restriction along $ i_P $ defines an equivalence
	\begin{equation*}
		\equivto{\Ptilde \coloneq \Sh(\Open(P))}{\Fun(P,\Space)}
	\end{equation*}
	(\Cref{exm:Pbasis}).
	The inverse is given by right Kan extension.
	In particular, the \topos $\Ptilde$ is both $0$-localic and Postnikov complete (\Cref{def:Postnikovstuff}).
\end{nul}

\begin{nul}\label{nul:Ptildecoherentobjects}
	If $P$ is a finite poset, then $\Ptilde$ is a coherent \topos (\Cref{ex:Spectralspacecoherent}), and a sheaf $F$ on $P$ is $n$-coherent if and only if all of the stalks of $ F $ have finite homotopy sets in degrees $ m \leq n $.
	In particular, we have a natural identification
	\begin{equation*}
		\Ptilde\cohbdd \equivalent \Fun(P,\Spacefin) \period
	\end{equation*}
\end{nul}

\begin{rmk}
	We work with \textit{finite} posets here for a number of reasons.
	First, to guarantee that the \topos $ \Ptilde $ is coherent, second, to guarantee that $ \Ptilde $ can be expressed as the functor category $ \Fun(P,\Space) $, and finally, because we are most interested in working with stratifications of \topoi over spectral topological spaces (i.e., profinite posets).
\end{rmk}

\begin{nul}\label{nul:Stildespectral}
	The functor $ \widetilde{(-)} \colon \goesto{\Posfin}{\Top_{\infty}} $ extends along inverse limits to a functor
	\begin{equation*}
		\fromto{\TSpcspec \simeq \Pro(\Posfin)}{\Top_{\infty}}
	\end{equation*}
	which we also denote by $\goesto{S}{\Stilde}$.
	Thus if $S \coloneq \{P_{\alpha}\}_{\alpha \in A}$ is an inverse system of finite posets, then
	\begin{equation*}
		\Stilde\simeq\lim_{\alpha \in A}\Ptilde_{\alpha}
	\end{equation*}
	in $\Top_{\infty}$.
	That is, by \Cref{thm:filteredlimsinRTop}=\allowbreak\HTT{Theorem}{6.3.3.1}, $ \Stilde $ is equivalent to the \category with objects collections $\{F_{\alpha}\}_{\alpha \in A}$ of functors \smash{$F_{\alpha}\colon\fromto{\Ptilde_{\alpha}}{\Space}$} along with compatible identifications of $F_{\alpha'}$ with the right Kan extension of $F_{\alpha}$ along \smash{$\fromto{P_{\alpha}}{P_{\alpha'}}$} for any morphism $\fromto{\alpha}{\alpha'}$ in $ A $.
	In particular, the \topos \smash{$\Stilde$} is $0$-localic.

	If we think of $ S $ as a spectral topological space, the $ 0 $-topos (locale) $\Open(S)$ is the limit of the $0$-topoi $\Open(P)$ over the $ 1 $-category $\FC(S)$ of finite constructible stratifications $ \fromto{S}{P} $ of $ S $.
	Thus we have an equivalence of $0$-localic \topoi
	\begin{equation*}
		\Stilde\simeq\lim_{P\in \FC(S)}\Ptilde \period
	\end{equation*}
	Since $\Stilde$ is coherent (\Cref{ex:Spectralspacecoherent}), the \pretopos $\Stilde\cohbdd$ of truncated coherent objects of $\Stilde$ can be identified with the filtered colimit
	\begin{equation*}
		\Stilde\cohbdd \equivalent \colim_{P\in \FC(S)^{\op}}\Ptilde\cohbdd \comma
	\end{equation*}
	along the relevant restriction functors (\cref{subsec:cohinverselimtis}).

	Recall that if $f\colon\fromto{S'}{S}$ is a quasicompact continuous map of spectral topological spaces, then the induced geometric morphism $\flowerstar\colon\fromto{\Stilde'}{\Stilde}$ is coherent (\Cref{ex:Spectralspacecoherent}).
\end{nul}

\begin{nul}
	If $ S $ is a spectral topological space, then the \category of points of $\Stilde$ is equivalent to the materialization of $ S $ (regarded as a profinite poset):
	\begin{equation*}
		\Pt(\Stilde) \simeq \mat(S) \period
	\end{equation*}
	Thus the points of $\Stilde$ are precisely the points of $ S $ equipped with the specialization partial ordering (\Cref{def:specializationpreorder}).
\end{nul}


\subsection{Stratifications over spectral topological spaces}\label{subset:toposstratspectral}

We now begin to study stratified \topoi.
The definition is a straightforward generalization of the notion of a stratified topological space (\Cref{dfn:strattopspace,def:profinstrat}).

\begin{dfn} 
	Let $ S $ be a spectral topological space.
	An \defn{$ S $-stratified \topos}\index[terminology]{topos@\topos!stratified}\index[terminology]{stratified topos@stratified \topos} is a geometric morphism of \topoi $\fromto{\XX}{\Stilde}$.
	We write\index[notation]{StrTop@$\StrTop_{\infty},  \StrTop_{\infty,S} $}
	\begin{equation*}
		\StrTop_{\infty,S} \colonequals \Top_{\infty,/\Stilde}
	\end{equation*}
	for the \category of $ S $-stratified \topoi.

	We define
	\begin{equation*}
		\StrTop_{\infty} \coloneq \Fun([1],\Top_{\infty}) \crosslimits_{\Fun(\{1\},\Top_{\infty})} \TSpcspec \period \index[notation]{StrTop@$\StrTop_{\infty}$}
	\end{equation*}
	The fiber over a spectral topological space $ S $ is identified with the \category $\StrTop_{\infty,S}$.
	Since $\Top_{\infty}$ admits fiber products, the projection \smash{$\fromto{\StrTop_{\infty}}{\TSpcspec}$} is a bicartesian fibration.
\end{dfn}

\begin{nul} 
	Let $ S $ be a spectral topological space.
	Since $\Stilde$ is $0$-localic, it follows that an $ S $-stratification of \atopos $\XX$ is the same data as a morphism of $0$-topoi (=locales)
	\begin{equation*}
		\fromto{\Open(\XX)}{\Open(S)} \period
	\end{equation*}
	Thus there is a natural equivalence of \categories
	\begin{equation*}
		\StrTop_{\infty,S} \simeq \Top_{\infty} \crosslimits_{\Top_0} \Top_{0,/\Open(S)} \period
	\end{equation*}
\end{nul}

\begin{ntn} 
	Let $ S $ be a spectral topological space, and let $ \flowerstar \colon \fromto{\XX}{\Stilde} $ be an $ S $-stratified \topos.
	For any open subset $U\subseteq S$, we abuse notation and write $U$ also for the corresponding open of $\Stilde$, and we write
	\begin{equation*}
		\XX_U\coloneq \XX_{/f^{\ast}U} \simeq \XX \, \crosslimits_{\Stilde} \, \widetilde{U} \subseteq \XX \index[notation]{XXU@$\XX_{U}$}
	\end{equation*}
	for the corresponding open subtopos. 
	Dually, if $Z\subseteq S $ is closed, then we write
	\begin{equation*}
		\XX_Z\coloneq \XX_{\smallsetminus f^{\ast}(S \smallsetminus Z)} \simeq \XX \, \crosslimits_{\Stilde} \, \widetilde{Z} \subseteq \XX \index[notation]{XXZ@$\XX_{Z}$}
	\end{equation*}
	for the corresponding closed subtopos.
	If $U$ and $Z$ are complementary, then the \topos $\XX$ is the recollement of $\XX_Z$ and $\XX_U$ along the gluing functor $ \iupperstar \jlowerstar \colon \fromto{\XX_U}{\XX_Z} $.

	More generally, for any subspace $ W \subset S $, we write 
	\begin{equation*}
		\XX_W \coloneq \XX \, \crosslimits_{\Stilde} \, \widetilde{W} \period \index[notation]{XXW@$\XX_{W}$}
	\end{equation*}
	In particular, for any point $ s \in S $ we define the \defn{$s$-th stratum}\index[terminology]{stratum} as the fiber product in $ \Top_{\infty} $:
	\begin{equation*}
		\XX_s \coloneq \XX \, \crosslimits_{\Stilde} \, \widetilde{\{s\}} \subseteq \XX \period \index[notation]{XXs@$\XX_{s}$}
	\end{equation*}
\end{ntn}

\begin{nul}\label{ntn:inverseofsubposets} 
	Let $P$ be a finite poset, and let $\XX$ be a $P$-stratified \topos.
	Note that for any point $p\in P$, the $p$-th stratum is the fiber product in $ \Top_{\infty} $:
	\begin{equation*}
		\XX_p \coloneq \XX_{P_{\geq p}} \, \crosslimits_{\XX} \, \XX_{P_{\leq p}} \period
	\end{equation*}
	The stratum $ \XX_p $ is an open subtopos of the closed subtopos \smash{$\XX_{P_{\leq p}}\subseteq \XX$} as well as a closed subtopos of the open subtopos \smash{$\XX_{P_{\geq p}} \subseteq \XX$}.
\end{nul}

\begin{exm}
	A $\{0\}$-stratified \topos is nothing more than \atopos.
\end{exm}

\begin{exm}
	Rephrasing \Cref{nul:recollement1-strat}, a $[1]$-stratified \topos $\fromto{\XX}{\widetilde{[1]}}$ is the same data as a recollement of \topoi.
\end{exm}

\begin{rmk} 
	To generalize the previous example, let $P$ be a finite poset.
	It appears that the data of a $P$-stratified \topos determines and is determined by a suitable colax functor from $P^{\op}$ to a double \category of \topoi and left exact functors.

	To make a precise assertion, let us say that a locally cocartesian fibration $\fromto{X}{P^{\op}}$ is \defn{left exact} if each fiber $X_p$ admits all finite limits, and for any $p\leq q$ in $P$, the functor $\fromto{X_q}{X_p}$ is left exact. 
	Left exact locally cocartesian fibrations $\fromto{X}{P^{\op}}$ whose fibers are \topoi organize themselves into a \category \smash{$\LocCocart^{\lex,\toptextit}_{P^{\op}}$}.
	It seems likely that one can produce an equivalence of \categories
	\begin{equation*}
		\LocCocart^{\lex,\toptextit}_{P^{\op}}\simeq\StrTop_{\infty,P} \comma
	\end{equation*}
	natural in $P$.
	To prove this would involve a diversion into a simplicial thicket that is unnecessary for our work here.
\end{rmk}

\begin{exm}
	Let $ S $ be a spectral topological space.
	The \topos $\Stilde$ equipped with the identity stratification is the terminal object of $\StrTop_{\infty,S}$. 
\end{exm}

\begin{exm}
	Let $ S $ be a spectral topological space.
	Write $\TSpc^{\sober} \subset \TSpc $\index[notation]{TSpcsob@$ \TSpc^{\sober} $} for the full subcategory spanned by the sober topological spaces. Then the assignment $\goesto{T}{\widetilde{T}}$ defines a fully faithful functor
	\begin{equation*}
		\incto{\TSpc^{\sober}_{/S}}{\StrTop_{\infty,S}} \period 
	\end{equation*}
\end{exm}

We now begin to investigate the basic properties of stratified \topoi.
For the following result, please recall that the coproduct in $ \Top_{\infty} $ is computed as the \textit{product} of \categories. 

\begin{lem}\label{lem:maptostrataisconservative}
	Let $P$ be a finite poset, and let $ \flowerstar \colon \XX \to \Ptilde $ be a $P$-stratified \topos.
	Then the pullback functor
	\begin{equation*}
		(e_p^{\ast})_{p\in P} \colon \XX \to \coprod_{p\in P} \XX_p
	\end{equation*}
	is conservative. 
	That is, a morphism $ \phi \colon F \to G $ of $\XX$ is an equivalence if and only if, for all $p \in P$, the restriction $ \restrict{\phi}{\XX_p} \colon \restrict{F}{\XX_p} \to \restrict{G}{\XX_p} $ is an equivalence.
\end{lem}

\begin{proof}
	If $P = \varnothing $, there is nothing to prove.
	If $P$ is of rank $0$, then $ P $ is discrete and the pullback functor above is an equivalence.
	Now we induct on the rank of $ P $; assume the claim holds for all finite posets of rank $\leq n$.
	Let $ P $ be a finite poset of rank $ n+1 $, and write $ M \subset P $ for the full subposet spanned by the minimal elements of $ P $.
	Then $ M $ is discrete and closed in $ P $.
	Note that the open complement $U \coloneq P \smallsetminus M $ of $ M $ is of rank $ n $.
	Now if $ \restrict{\phi}{\XX_p} $ is an equivalence for each $ p \in P $, then $\restrict{\phi}{\XX_Z}$  is an equivalence by the rank $ 0 $ case, and $\restrict{\phi}{\XX_U}$ is an equivalence by the rank $n$ case.
	Since $\XX$ is a recollement of $\XX_Z$ and $\XX_U$, the pullback $\XX \to \XX_Z \coproduct \XX_U$ is conservative, so $\phi$ itself is an equivalence.
\end{proof}

\noindent In a similar vein, the following useful hypercompleteness result follows from the fact that the recollement of hypercomplete \topoi is hypercomplete \Cref{nul:hypercompleterecollement} and induction on the rank of the poset:

\begin{lem}\label{lem:hypercompletestratified}
	Let $ P $ be a finite poset and $ \fromto{\XX}{\Ptilde} $ a $ P $-stratified \topos.
	If for each $ p \in P $ the stratum $ \XX_p $ is hypercomplete, then the \topos $ \XX $ is hypercomplete.
\end{lem}

In almost all of the examples of stratified \topoi that we consider in this text, the total \topos is bounded coherent and the stratification is a coherent geometric morphism.
We refer to such a stratified \topos as \textit{bounded coherent constructible}:

\begin{dfn}\label{def:boundedcoherentconstrstratifications}
	Let $\XX$ be \atopos and $ S $ a spectral topological space.
	A stratification $\flowerstar\colon\fromto{\XX}{\Stilde}$ is \defn{constructible}\index[terminology]{stratification!constructible}\index[terminology]{constructible!stratification} if and only if, for any quasicompact open $U\subseteq S$ and any quasicompact open $V\in\Open(\XX)$, the \topos
	\begin{equation*}
		\XX_U \cross_{\XX} \XX_{/V} \equivalent \XX_{/(\fupperstar(U) \cross V)}
	\end{equation*}
	is coherent.
	We say that a constructible stratification $ \flowerstar \colon \fromto{\XX}{\Stilde} $ is \defn{coherent constructible}\index[terminology]{stratification!coherent constructible}\index[terminology]{coherent constructible!stratification} if $ \XX $ is a coherent \topos, and we say that $ \flowerstar $ is \defn{bounded coherent constructible}\index[terminology]{stratification!bounded coherent constructible}\index[terminology]{bounded coherent constructible!stratification} if $ \XX $ is a bounded coherent \topos.
\end{dfn}

\begin{nul}[constructibility for stratifications by finite posets]
	Let $ P $ be a finite poset.
	A $ P $-stratified \topos $\flowerstar\colon\fromto{\XX}{\Ptilde}$ is constructible if and only if for every point $p\in P$ and any quasicompact open $ V \in \Open(\XX) $, the \topos $ \XX_{P_{\geq p}} \cross_{\XX} \XX_{/V} $ is coherent. 
\end{nul}

\begin{exm}
	Let $\fromto{\XX}{\widetilde{[1]}}$ be a $ [1] $-stratified \topos.
	If $\XX$ is coherent, the stratification is constructible if and only if the open subtopos $\XX_1$ is quasicompact (\Cref{prop:DAGXIII.2.3.22}=\cite[Proposition 2.3.22]{DAGXIII}).
\end{exm}

Coherent constructibility can be reformulated as the \textit{a priori} stronger condition that the stratification be a coherent geometric morphism:

\begin{lem}
	Let $ S $ be a spectral topological space and $ \flowerstar \colon \fromto{\XX}{\Stilde} $ be an $ S $-stratified \topos.
	If $ \XX $ is coherent, then the stratification $ \flowerstar $ is constructible if and only if $ \flowerstar $ is a coherent geometric morphism.
\end{lem}

\begin{proof}
	If $ \flowerstar $ is coherent, then since quasicompact opens in $ \XX $ are coherent \SAG{Remark}{A.2.3.5} and coherent objects of $ \XX $ are closed under finite products, $ \flowerstar $ is a constructible stratification

	For the other direction, assume that $ \flowerstar $ is a constructible stratification.
	By \Cref{cor:criterionforcoherence}, to show that $ \flowerstar $ is coherent it suffices to show that $ \fupperstar $ carries truncated coherent objects of $ \Stilde $ to coherent objects of $ \XX $. 
	Let $ F \in \Stilde_{<\infty }^{\coh} $ be a truncated coherent object; then there exists a finite constructible stratification $ \fromto{S}{P} $ such that $ F $ is the pullback of a truncated coherent object of $ \Ptilde $ \Cref{nul:Stildespectral}.
	Thus, for every point $p\in P$, the restriction $ \restrict{\fupperstar(F)}{\XX_p} $ is lisse. 
	By \Cref{prop:DAGXIII.2.3.22}=\cite[Proposition 2.3.22]{DAGXIII} it follows that $ F $ is coherent.
\end{proof}

\begin{ntn}
	Let $ S $ be a spectral topological space.
	We define the \category of \defn{coherent construct\-ible} $ S $-stratified \topoi as the overcategory
	\begin{equation*}
		\StrTopcc_{\infty,S} \coloneq \Top_{\infty,/\Stilde}^{\coh} \period \index[notation]{StrTopcc@$\StrTopcc_{\infty}, \StrTopcc_{\infty,S} $}
	\end{equation*}
	We write $ \StrTopbcc_{\infty,S} \subset \StrTopcc_{\infty,S} $\index[notation]{StrTopbcc@$ \StrTopbcc_{\infty}, \StrTopbcc_{\infty,S}$} for the full subcategory spanned by the \defn{bounded coherent construct\-ible} $ S $-stratified \topoi.

	More generally, we define
	\begin{equation*}
		\StrTopcc_{\infty} \coloneq \Fun([1],\Topcoh) \crosslimits_{\Fun(\{1\},\Topcoh)} \TSpcspec \semicolon
	\end{equation*}
	the fiber over $ S $ is identified with \smash{$\StrTopcc_{\infty,S}$}.
	We write \smash{$ \StrTopbcc_{\infty} \subset \StrTopcc_{\infty} $} for the full subcategory spanned by those objects \smash{$ \fromto{\XX}{\Stilde} $} where $ \XX $ is a bounded \topos.
\end{ntn}


\subsection{The natural stratification of a coherent \texorpdfstring{$\infty$}{∞}-topos}\label{subsec:naturalstrat}

Every coherent \topos $\XX$ has a canonical profinite stratification: the $0$-topos (=locale) $\Open(\XX)$ is the locale of a spectral topological space. This provides a fully faithful embedding of the \category of coherent \topoi into that of coherent constructible stratified \topoi.

To explain this point, let us first recall the equivalence between coherent locales and spectral topological spaces.

\begin{rec}
	Let $ A $ be a locale.
	An object $ a \in A $ is \defn{quasicompact}\footnote{Such elements are sometimes called \defn{finite}; see \cite[Chapter II, \S 3.1]{MR861951}.} if and only if for every subset $ S \subset A $ such that $ \coprod_{s \in S} s = a $, there exists a finite subset $ S_0 \subset S $ such that $ \coprod_{s \in S_0} s = a $.

	The locale $ A $ is \defn{coherent} if and only if $ A $ is coherent in the sense of \Cref{rec:coherence}.
	\Cref{prop:coherenceforn-localic} shows that this is the case if and only if the following conditions are satisfied:
	\begin{enumerate}[(1)]
		\item The quasicompact elements of $ A $ form a \defn{sublattice} of $ A $: the maximal element $ 1_A \in A $ is quasicompact and binary products (=meets) of quasicompact elements are quasicompact.

		\item The quasicompact elements of $ A $ \defn{generate} $ A $: every element $ a \in A $ can be written as a coproduct (=join) $ a = \coprod_{s \in S} s $, where $ S \subset A $ is a subset consisting of quasicompact elements of $ A $.
	\end{enumerate}
	A morphism $ \fromto{A}{A'} $ between coherent locales is \defn{coherent} if and only if the corresponding map of posets $ \fromto{A'}{A} $ sends quasicompact elements to quasicompact elements.

	We write $ \Top_{0}^{\coh} $ for the category of coherent locales and coherent morphisms between them (cf. \Cref{cor:coherent1localicmors}).
\end{rec}

\begin{exm}\label{exm:quasicompactsopensarequasicompact}
	Let $ \XX $ be \atopos.
	Then an open $ U \in \Open(\XX) $ is a quasicompact element of the locale $ \Open(\XX) $ if and only if $ U $ is a quasicompact (i.e., $ 0 $-coherent) object of the \topos $ \XX $.
\end{exm}

The following three results are immediate from the definitions and \Cref{exm:quasicompactsopensarequasicompact}.

\begin{lem}\label{lem:1coherentgivescoherentlocale}
	For any $ 1 $-coherent \topos $ \XX $, the locale $ \Open(\XX) $ is coherent. 
\end{lem}

\begin{lem}\label{lem:functorialityofOpenincoherentmors}
	Let $ \flowerstar \colon \fromto{\XX}{\YY} $ be a coherent geometric morphism between coherent \topoi.
	Then the induced morphism $ \fromto{\Open(\XX)}{\Open(\YY)} $ of coherent locales is coherent.
\end{lem}

\begin{cor}\label{cor:constrstratonlocales}
	Let $ S $ be a spectral topological space and $ \flowerstar \colon \fromto{\XX}{\Stilde} $ an $ S $-stratified \topos.
	If $ \XX $ is coherent, then $ \flowerstar $ is a constructible stratification if and only if the induced morphism of coherent locales $ \fromto{\Open(\XX)}{\Open(S)} $ is coherent.
\end{cor}

The following classical result is an important recognition principle for coherent locales.

\begin{prp}[{\cite[Chapter II, \S\S 3.3--3.4]{MR861951}}]\label{prp:JohnstoneStonerepresentiation}
	The functor $ \Open \colon \fromto{\TSpcspec}{\Top_{0}} $ given by sending a spectral topological space $ S $ to its locale of open subsets factors through $ \Top_{0}^{\coh} $ and defines an equivalence of categories
	\begin{equation*}
		\Open \colon \equivto{\TSpcspec}{\Top_{0}^{\coh}} \period
	\end{equation*}
\end{prp}

\begin{nul}
	The functor $ \Open \colon \equivto{\TSpcspec}{\Top_{0}^{\coh}} $ has an explicit inverse $ \equivto{\Top_{0}^{\coh}}{\TSpcspec} $ given by taking the topological space of \textit{points} of a locale; see \cite[Chapter II, \S 1.3]{MR861951}.
\end{nul}

\begin{ntn}\label{ntn:0localicrefofcoherent}
	\Cref{lem:functorialityofOpenincoherentmors} and \Cref{prp:JohnstoneStonerepresentiation} provide a functor
	\begin{equation*}
		\begin{tikzcd}
			\Sup \colon \Topcoh \arrow[r, "\Open"] & \Top_{0}^{\coh} \arrow[r, "\sim"{yshift=-0.25em}] & \TSpcspec \comma
		\end{tikzcd}
	\end{equation*}
	which we denote by $ \Sup $\index[notation]{S@$ \Sup $}.
	By definition, the $ 0 $-localic reflection of a coherent \topos $ \XX $ is given by the \topos of sheaves on the spectral topological space $ \Sup(\XX) $.
	Thus $ \XX $ comes equipped with a natural $ \Sup(\XX) $-stratification $ \fromto{\XX}{\widetilde{\Sup(\XX)}} $.

	The localization $ \Top_{\infty} \rightleftarrows \Top_{0} $ thus restricts to a localization $ \Topcoh \rightleftarrows \Top_{0}^{\coh} $.
\end{ntn}

\begin{lem}\label{lem:naturalstratconstructible}
	For any coherent \topos $ \XX $, the natural stratification $ \flowerstar \colon \fromto{\XX}{\widetilde{\Sup(\XX)}} $ is constructible (\Cref{def:boundedcoherentconstrstratifications}).
\end{lem}

\begin{proof}
	Clear from \Cref{cor:constrstratonlocales} and the fact that \smash{$ \flowerstar \colon \fromto{\XX}{\widetilde{\Sup(\XX)}} $} induces an equivalence of locales
	\begin{equation*}
		\isomto{\Open(\XX)}{\Open(\widetilde{\Sup(\XX)}) = \Open(\Sup(\XX))} \period \qedhere
	\end{equation*}
\end{proof}

\begin{exm}\label{exm:natstratXet}
	Let $X$ be a coherent scheme.
	Write $X^{\zar}$ for the underlying Zariski spectral topological space of $ X $, and $X_{\et}$ for the étale \topos of $ X $.
	Recall that the \topos $ X_{\et} $ is coherent (\Cref{prop:coherentschemetopos}).
	Since
	\begin{equation*}
		\Open(X_{\et}) \cong \Open(X^{\zar}) \comma
	\end{equation*}
	the natural stratification of the coherent \topos $ X_{\et} $ is given by the natural geometric morphism $ \fromto{X_{\et}}{X_{\zar}} $.
\end{exm}

\begin{nul}\label{nul:naturalstratasadjoint}
	The source functor \smash{$ \fromto{\StrTopcc_{\infty}}{\Topcoh} $} admits a fully faithful left adjoint, given by the assignment
	\begin{equation*}
		\goesto{\XX}{[\fromto{\XX}{\widetilde{\Sup(\XX)}}]} \period
	\end{equation*}
	The essential image of this left adjoint is the full subcategory spanned by those coherent constructible stratified \topoi $ \fromto{\XX}{\Stilde} $ such that the stratification induces an equivalence of locales $ \isomto{\Open(\XX)}{\Open(S)} $.

	The source functor $ \fromto{\StrTopcc_{\infty}}{\Topcoh} $ also admits a fully faithful right adjoint, which carries a coherent \topos $ \XX $ to $ \XX $ equipped with the unique stratification over $ \Space = \widetilde{\{0\}} $.
\end{nul}


\subsection{Stratified \texorpdfstring{$\infty$}{∞}-topoi attached to stratified spaces}\label{subsec:Pitilde}
	
In this section we investigate examples of stratified \topoi that arise from stratified spaces via the following construction.
The main example of interest is when the stratified space is \pifinite.

\begin{cnstr}[the stratified \topos of a stratified space]
	Let $P$ be a finite poset, and $f\colon\fromto{\Pi}{P}$ a $P$-stratified space (\Cref{dfn:stratspaces}); i.e., $f$ is a conservative functor of \categories. 
	In light of the equivalence $\Ptilde\simeq\Fun(P,\Space)$, let us abuse notation slightly and write
	\begin{equation*}
		\Pitilde\coloneq\Fun(\Pi,\Space)
	\end{equation*}
	for the \topos of functors $ \Pi \to \Space$.
	Right Kan extension along $f$ defines a geometric morphism of \topoi
	\begin{equation*}
		\flowerstar\colon\fromto{\Pitilde}{\Ptilde} \comma
	\end{equation*}
	whence $\Pitilde$ is a $P$-stratified \topos.
	For each point $p\in P$, the $p$-th stratum of $\Pitilde$ is canonically identified with the \topos $\widetilde{\Pi_p}=\Fun(\Pi_p,\Space)$. 

	The assignment \smash{$\goesto{\Pi}{\Pitilde}$} defines a functor \smash{$\fromto{\Strat}{\StrTop_{\infty}}$} over \smash{$\Posfin$}.
\end{cnstr}

\begin{exm}[exit-path \categories]\label{subexm:exodromyfortopspaces}
	Let $P$ be a finite poset, and let $ T $ be a conically $P$-stratified topological space in the sense of \HAa{Definition}{A.5.5}.
	The \defn{exit path \category} of $ T $ is the $P$-stratified space
	\begin{equation*}
		\Exit^P(T) \coloneq \Sing^P(T)
	\end{equation*}
	of \cite[\HAappthm{Definition}{A.6.2} \& \HAappthm{Theorem}{A.6.4}]{HA}.
	For each $ p \in P $, the stratum $ \Exit^P(T)_p $ is the fundamental \groupoid $ \Shape(T_p) $ of the topological space $ T_p $.

	Assume that $ T $ is paracompact and the strata of $ \fromto{T}{P} $ are locally of singular shape in the sense of \HAa{Definition}{A.4.15}.
	Consider the $ P $-stratified \topos
	\begin{equation*}
		\fromto{\ExittildeP(T)}{\Ptilde} \period
	\end{equation*}
	In light of \HAa{Theorem}{A.4.19}, for each point $ p \in P $, stratum
	\begin{equation*}
		\ExittildeP(T)_p \equivalent \Fun(\Shape(T_p),\Space)
	\end{equation*}
	is equivalent to the \category of locally constant sheaves on $T_p$.
	Moreover, \cite[\HAappthm{Remark}{A.5.19} \& \HAappthm{Theorem}{A.9.3}]{HA} shows that the \topos \smash{$ \ExittildeP(T) $} is equivalent to the \category of \defn{formally constructible} sheaves on $T$, i.e., those sheaves whose restrictions to each stratum $T_p$ are locally constant.
	(See \Cref{def:Pconstructible}.)
\end{exm}

The remainder of this section is dedicated to showing that if $ f \colon \fromto{\Pi}{P} $ is an $ n $-truncated \pifinite $ P $-stratified space, then the \topos $ \Pitilde $ is $ n $-localic and coherent, and the stratification $ \flowerstar \colon \fromto{\Pitilde}{\Ptilde} $ is constructible (\Cref{cor:Pitildeisbcc}).\footnote{Since the $ n $-category $ \Pi $ need not have finite limits, it is not \textit{a priori} obvious that the \topos $ \Pitilde $ is $ n $-localic (cf. \Cref{wrn:Ptildenlocalic}).}
To do this, we show that $ \Pi^{\op} $ forms a basis for the bounded \pretopos $ \Fun(\Pi,\Spacefin) $, and that the \topos of sheaves on $ \Fun(\Pi,\Spacefin) $ is already hypercomplete.

\begin{cnstr}[{bases for $ \Fun(\Pi,\Spacefin) $}]
	Let $ \Pi $ be a \pifinite stratified space.
	Then the \category $ \Fun(\Pi,\Spacefin) $ is a bounded \pretopos (\Cref{lem:Funpretopoi}).
	Note that $ \Pitilde $ is generated under colimits by the essential image of the Yoneda embedding $ \yo \colon \incto{\Pi^{\op}}{\Pitilde} $, the Yoneda embedding factors through $ \Fun(\Pi,\Spacefin) $, and every object of $ \Fun(\Pi,\Spacefin) $ is quasicompact.
	Hence, for every object $ F \in \Pitilde $ there exists a \textit{finite} set of objects $ \{x_i\}_{i \in I} $ of $ \Pi $ and an effective epimorphism
	\begin{equation*}
		\surjto{\coprod_{i \in I} \yo(x_i)}{F} \period
	\end{equation*}
	That is to say, $ \yo \colon \incto{\Pi^{\op}}{\Pitilde} $ is a basis for the effective epimorphism topology on the bounded \pretopos $ \Fun(\Pi,\Spacefin) $ in the sense of \Cref{def:basis}.

	If $ \Pi $ is an $ n $-category, then the Yoneda embedding factors through $ \Fun(\Pi,\Space_{\uppi,\leq n-1}) $.
	In particular, we have bases
	\begin{equation*}
		\begin{tikzcd}[sep=1.5em]
			\Pi^{\op} \arrow[r, "\yo", hooked]  & \Fun(\Pi,\Space_{\uppi,\leq n-1}) \arrow[r, hooked] & \Fun(\Pi,\Spacefin) \comma
		\end{tikzcd}
	\end{equation*}
	for the effective epimorphism topology on the \pretopos $ \Fun(\Pi,\Spacefin) $.
	Moreover, the middle \category is an $ n $-catego\-ry with finite limits.
	By \Cref{lem:ransheaf,cor:hypercompletebasis} we see that right Kan extension defines fully faithful geometric morphisms 
	\begin{equation}\label{eq:3Pitildes}
		\begin{tikzcd}[sep=1.5em]
			\Sheff{\Pi^{\op}} \arrow[r, hooked]  & \Sheff{\Fun(\Pi,\Space_{\uppi,\leq n-1})} \arrow[r, hooked] & \Sheff{\Fun(\Pi,\Spacefin)} 
		\end{tikzcd}
	\end{equation}
	that become equivalences after hypercompletion.
	To see that $ \Pitilde $ is $ n $-localic and coherent, we show that $ \Sheff{\Pi^{\op}} = \Pitilde $ and prove that the geometric morphisms \eqref{eq:3Pitildes} are equivalences.
	The latter amounts to showing that $ \Sheff{\Fun(\Pi,\Spacefin)} $ is already hypercomplete. 
\end{cnstr}

First we analyze the restriction of the effective epimorphism topology to $ \Pi^{\op} $.
Using the fact that every endomorphism in $ \Pi $ is an equivalence, it is easy to see that all presheaves are sheaves for this topology on $ \Pi^{\op} $.

\begin{dfn}
	Let $ C $ be \acategory.
	The \defn{chaotic topology}\index[terminology]{chaotic topology}\index[terminology]{topology!chaotic} on $ C $ is the Grothendieck topology defined by declaring that a sieve $ S \subset C_{/c} $ is covering if and only if $ S $ contains an object $ s \in S $ such that the structure morphism $ \fromto{s}{c} $ is an equivalence. 

	Note that a every presheaf on $ C $ is a sheaf for the chaotic topology.
\end{dfn}

\begin{lem}\label{lem:chaotictop}
	Let $ \Pi $ be a \pifinite stratified space.
	The restriction of the effective epimorphism topology on $ \Fun(\Pi,\Spacefin) $ to $ \Pi^{\op} \subset \Fun(\Pi,\Spacefin) $ is the chaotic topology.
\end{lem}

\begin{proof}
	Let $ y \in \Pi $ and suppose that we are given a finite set of objects $ \{x_i\}_{i \in I} $ of $ \Pi $ and an effective epimorphism
	\begin{equation*}
		e \colon \surjto{\coprod_{i \in I} \yo(x_i)}{\yo(y)}
	\end{equation*}
	in the \pretopos $ \Fun(\Pi,\Spacefin) $.
	Since the Yoneda embedding is fully faithful, there exist morphisms $ \{e_i \colon \fromto{y}{x_i}\}_{i \in I} $ in $ \Pi $ such that $ e $ is the induced morphism
	\begin{equation*}
		e \equivalent (\eupperstar_i)_{i \in I} \colon \fromto{\coprod_{i \in I} \yo(x_i)}{\yo(y)} \period
	\end{equation*}
	We claim that there exists an index $ i \in I $ such that the morphism $ e_i $ is an equivalence.
	Since $ \Pi $ is layered, it suffices to show that there exists an $ i \in I $ and a morphism $ \fromto{x_i}{y} $ in $ \Pi $.
	To see this, note that since $ e $ is an effective epimorphism, the induced morphism 
	\begin{equation*}
		e(y) \colon \fromto{\coprod_{i \in I} \Map_{\Pi}(x_i,y)}{\Map_{\Pi}(y,y)}
	\end{equation*}
	is a $ \uppi_0 $-surjection of spaces.
	Since $ \uppi_0 \colon \fromto{\Space}{\Set} $ preserves coproducts and $ \uppi_0 \Map_{\Pi}(y,y) $ is nonempty, we deduce that there exists an index $ i \in I $ such that $ \uppi_0 \Map_{\Pi}(x_i,y) $ is nonempty, as desired.
\end{proof}

Now we show that $ \Sheff{\Fun(\Pi,\Spacefin)} $ is hypercomplete.

\begin{nul}\label{nul:natbccstrat}
	Let $ f \colon \fromto{\Pi}{P} $ be a \pifinite stratified space.
	Note that the pullback functor
	\begin{equation*}
		\fupperstar \colon \fromto{\Fun(P,\Spacefin)}{\Fun(\Pi,\Spacefin)}
	\end{equation*}
	is a morphism of \pretopoi, hence induces a natural bounded coherent constructible stratification
	\begin{equation*}
		\flowerstar \colon \fromto{\Sheff{\Fun(\Pi,\Spacefin)}}{\Sheff{\Fun(P,\Spacefin)} \equivalent \Ptilde} \period
	\end{equation*} 
\end{nul}

\begin{prp}\label{prop:SheffFunPiSpacefinhypercomplete}
	Let $ P $ be a finite poset and $ f \colon \fromto{\Pi}{P} $ a \pifinite $ P $-stratified space.
	Then the \topos $ \Sheff{\Fun(\Pi,\Spacefin)} $ is hypercomplete.
\end{prp}

\begin{proof}
	By \Cref{lem:hypercompletestratified} it suffices to show that the strata of the natural stratification
	\begin{equation*}
		\flowerstar \colon \fromto{\Sheff{\Fun(\Pi,\Spacefin)}}{\Ptilde}
	\end{equation*}
	are hypercomplete.
	To see this, note that for each $ p \in P $, we have equivalences
	\begin{align*}
		\Sheff{\Fun(\Pi,\Spacefin)}_p &\equivalent \Sheff{\Fun(\Pi_p,\Spacefin)} \\ 
		&\equivalent \Sheff{(\Spacefin)_{/\Pi_p}} \equivalent \Space_{/\Pi_p} \period \qedhere
	\end{align*}
\end{proof}

Combining \Cref{lem:chaotictop} with \Cref{prop:SheffFunPiSpacefinhypercomplete} and applying \Cref{cor:hypercompletebasis} to the bases
\begin{equation*}
	\Pi^{\op} \subset \Fun(\Pi,\Space_{\uppi,\leq n-1}) \subset \Fun(\Pi,\Spacefin) \comma
\end{equation*}
we conclude the following.

\begin{cor}\label{cor:Pitildeisbcc}
	Let $ P $ be a finite poset, $ n \geq 0 $ be an integer, and $ \Pi \to P $ be an $ n $-truncated \pifinite $ P $-stratified space.
	Then:
	\begin{enumerate}[{\upshape (\ref*{cor:Pitildeisbcc}.1)}]
		\item The \topos $ \Pitilde $ is $ n $-localic and coherent.

		\item Then the stratification $ \fromto{\Pitilde}{\Ptilde} $ is bounded coherent constructible in the sense of \Cref{def:boundedcoherentconstrstratifications}.
		
		\item An object of $ \Pitilde $ is truncated coherent if and only if all of its values are \pifinite spaces.
		That is,
		\begin{equation*}
			\Pitilde\cohbdd = \Fun(\Pi,\Spacefin) \period
		\end{equation*}
	\end{enumerate}
\end{cor}

\Cref{cor:Pitildeisbcc} justifies the following notation:

\begin{ntn}\label{ntn:lambdaP}
	Let $ P $ be a finite poset.
	Denote by
	\begin{equation*}
		\uplambda \colon \fromto{\Stratfin}{\restrict{\StrTopbcc_{\infty}}{\Posfin}} \index[notation]{lambda@$ \uplambda $}
	\end{equation*}
	the functor over $ \Posfin $ defined by the assignment $\goesto{\Pi}{\Pitilde}$.
	For each finite poset $P$, we write $ \uplambda_P \colon \fromto{\Strat_P}{\StrTop_{\infty,P}} $\index[notation]{lambdaP@$ \uplambda_{P} $} for the induced functor on fibers over $ P $.
\end{ntn}


\subsection{Gluing squares}\label{subsec:BCgluing}

In this section we use the truncated \basechange theorem for oriented fiber products (\Cref{thm:BCfororientedfibs}) to study oriented squares of bounded coherent \topoi that are both oriented fiber product squares and oriented pushouts.
These \textit{gluing squares} are essential to our décollage approach to stratified higher topoi in \cref{sec:toposicdec}.

\begin{dfn}\label{def:gluing} 
	A \defn{gluing square}\index[terminology]{gluing square} is an oriented square of \topoi
	\begin{equation*}
		\begin{tikzcd}
			\WW \arrow[r, "q_{\ast}" above] \arrow[d, "p_{\ast}" left] & \UU \arrow[d, "j_{\ast}" right] \arrow[dl, phantom, "\scriptstyle \sigma" below right, "\Longleftarrow" sloped] \\ 
			\ZZ \arrow[r, "i_{\ast}" below] & \XX
		\end{tikzcd}
	\end{equation*}
	satisfying the following properties:
	\begin{enumerate}[(1)]
		\item Every \topos is bounded coherent.

		\item Every geometric morphism is coherent.

		\item The natural geometric morphism $\ZZ\orientedcup^{\WW}_{\bc}\UU\to\XX$ is an equivalence (\Cref{cnstr:bcorientedpushout}).

		\item The natural geometric morphism $\WW\to\ZZ\orientedtimes_{\XX}\UU$ is an equivalence (\Cref{cnstr:orientedfibprod}).
	\end{enumerate}
	We call the oriented fiber product $\WW$ the \defn{link}\index[terminology]{link!of a gluing square} of the gluing square, or the \defn{deleted tubular neighborhood}\index[terminology]{deleted tubular neighborhood} of $\ZZ$ inside $\XX$.
\end{dfn}

\begin{cnstr}[gluing squares from recollements and spans]\label{cnstr:gluingpullback} 
	Let $\XX$ be a bounded coherent \topos, $ \jlowerstar \colon \incto{\UU}{\XX} $ a quasicompact open subtopos, and write $ \ilowerstar \colon \incto{\ZZ}{\XX} $ for the closed complement of $ \UU $.
	Consider the oriented fiber product square
	\begin{equation}\label{square:gluingpulled}
		\begin{tikzcd}
			\ZZ\orientedtimes_{\XX}\UU \arrow[r, "\pr_{2,\ast}" above] \arrow[d, "\pr_{1,\ast}" left] & \UU \arrow[d, "j_{\ast}" right, hooked] \arrow[dl, phantom, "\scriptstyle \sigma" below right, "\Longleftarrow" sloped] \\ 
			\ZZ \arrow[r, "i_{\ast}" below, hooked] & \XX \period
		\end{tikzcd}
	\end{equation}
	The \topos $\XX$ is the bounded coherent recollement $\ZZ\orientedcup^{\pr_{1,\ast}\prupperstar_2}_{\bc}\UU$.
	Indeed, the truncated \basechange theorem (\Cref{thm:BCfororientedfibs}) ensures that the \basechange morphism
	\begin{equation*}
		\BC_{\sigma}\colon i^{\ast}j_{\ast}\to\pr_{1,\ast}\prupperstar_2
	\end{equation*}
	becomes an equivalence after restriction to $\UU^{\coh}_{<\infty}$. So \Cref{prp:bcorientedpodependsonbcgluing} applies, whence \eqref{square:gluingpulled} is a gluing square.

	Dually, let $\WW$, $\ZZ$, and $\UU$ be bounded coherent \topoi, and let $p_{\ast}\colon\fromto{\WW}{\ZZ}$ and $q_{\ast}\colon\fromto{\WW}{\UU}$ be geometric morphisms. Forming the bounded coherent oriented pushout $\XX\coloneq\ZZ\orientedcupbc^{\WW}\UU$, we obtain a square
	\begin{equation}\label{square:gluingpushed}
		\begin{tikzcd}
			\WW \arrow[r, "q_{\ast}" above] \arrow[d, "p_{\ast}" left] & \UU \arrow[d, "j_{\ast}" right, hooked] \arrow[dl, phantom, "\scriptstyle \sigma" below right, "\Longleftarrow" sloped] \\ 
			\ZZ \arrow[r, "i_{\ast}" below, hooked] & \ZZ\orientedcupbc^{\WW}\UU \period
		\end{tikzcd}
	\end{equation}
	We thus obtain a geometric morphism $\psi(p,q,\sigma)_{\ast}\colon\fromto{\WW}{\ZZ\orientedtimes_{\XX}\UU}$, and if $\psi(p,q,\sigma)_{\ast}$ is an equivalence, then the square \eqref{square:gluingpushed} is a gluing square.
\end{cnstr}

\begin{rmk}
	The full subcategory of $ \Fun([1] \cross [1],\Topbc) $ spanned by the gluing squares is equivalent to the (non-full) subcategory of $ \Fun([1],\Cat_{\infty,\updelta_1}) $ whose objects are bounded coherent gluing functors between bounded coherent \topoi and whose morphisms $\fromto{\phi}{\phi'}$ are squares
	\begin{equation*}
		\begin{tikzcd}
			\UU \arrow[r, "\phi" above] \arrow[d, "\flowerstar" left] & \ZZ \arrow[d, "g_{\ast}" right] \\ 
			\UU' \arrow[r, "\phi'" below] & \ZZ'
		\end{tikzcd}
	\end{equation*}
	in which $\flowerstar$ and $g_{\ast}$ are coherent geometric morphisms.
\end{rmk}

\begin{wrn}
	If the coherence assumptions are removed, then \Cref{cnstr:gluingpullback} does not recover $ \XX $ as an oriented pushout of $ \ZZ $ and $ \UU $ along $ \orientedpull{\ZZ}{\XX}{\UU} $.
	To see this, let $ X \coloneq [0,1] $ be the usual closed interval, $ Z \coloneq \{0\} $, and $ U \coloneq X \smallsetminus Z $ the open complement of $ Z $.
	Then the oriented fiber product $ \Ztilde \orientedtimes_{\Xtilde}\Utilde$ is the initial \topos $ \widetilde{\varnothing} $.
	The oriented pushout of $ \Ztilde $ and $ \Utilde $ along $ \widetilde{\varnothing} $ is the coproduct $ \Ztilde \coproduct \Utilde $ in $ \Top_{\infty} $, however the \topos $ \Ztilde \coproduct \Utilde $ is not equivalent to $ \Xtilde $. 
	The main problem here is that the \topoi $ \Utilde $ and $ \Xtilde $ are not coherent.
\end{wrn}

We finish this section with our key example of a gluing square.

\begin{ntn}
	Let $ Y $ be a profinite space.
	We write $ \Ytilde \colonequals \Fun(Y,\Space) $ for the corresponding Stone \topos (\Cref{def:Stonetopos}). 
\end{ntn}

\begin{exm}[{gluing squares of profinite $ [1] $-stratified spaces}]\label{exm:gluingstonetopoi} 
	Let
	\begin{equation}\label{eq:spandecollage}
		\begin{tikzcd}[sep=1.25em]
			Z & W \arrow[l, "p"'] \arrow[r, "q"] & U
		\end{tikzcd}
	\end{equation}
	be a span of profinite spaces, and write $ X $ for the profinite $[1]$-stratified space corresponding to the profinite spatial décollage \eqref{eq:spandecollage}.
	Write $ \XX $ for the form the bounded coherent oriented pushout of Stone \topoi:
	\begin{equation}\label{sq:profinorientedpushout}
		\begin{tikzcd}
			\Wtilde \arrow[r, "q_{\ast}" above] \arrow[d, "p_{\ast}" left] & \Utilde \arrow[d, "j_{\ast}" right, hooked] \arrow[dl, phantom, "\scriptstyle \sigma" below right, "\Longleftarrow" sloped] \\ 
			\Ztilde \arrow[r, "i_{\ast}" below, hooked] & \XX \period
		\end{tikzcd}
	\end{equation}
	Since $\Xtilde$ is the recollement of $\Ztilde$ and $\Utilde$ with gluing functor that agrees with $p_{\ast}q^{\ast}$ when restricted to truncated objects (\Cref{thm:BCfororientedfibs}), and $ \Xtilde $ is bounded coherent, the natural geometric morphism $ \fromto{\XX}{\Xtilde} $ is an equivalence (\Cref{lem:whenrecollementisbc,prp:bcorientedpodependsonbcgluing}).
	Now we compute the link:
	\begin{equation*}
		\Ztilde \orientedtimes_{\XX} \Utilde \simeq \MOR_{\widetilde{[1]}}(\widetilde{[1]},\XX) \simeq \widetilde{\Map_{[1]}([1],X)} \simeq \widetilde{W} \period
	\end{equation*}
	Thus the square \eqref{sq:profinorientedpushout} is in fact a gluing square.
\end{exm}


\subsection{Toposic décollages}\label{sec:toposicdec}

In analogy with the construction of the spatial décollage attached to a stratified space (\Cref{cnstr:nerveofstratifiedspace}), we can attach to a stratified \topos what we call its \textit{(toposic) décollage}.
Whereas a stratified \topos consists of strata that are glued together, its décollage is the result of pulling these strata apart while retaining the linking information necessary to reconstruct the stratified \topos.

\begin{dfn}\label{def:toposicdecollage} 
	Let $P$ be a finite poset.
	We say that a functor $\DD\colon\fromto{\sdop(P)}{\Topbc}$ is a \defn{décollage over $P$}\index[terminology]{decollage@décollage!toposic} if and only if the following conditions are satisfied.
	\begin{enumerate}[(1)]
		\item If $p_0,p_1\in P$ are elements such that $p_0<p_1$, then the square
		\begin{equation*}
			\begin{tikzcd}
				\DD\{p_0 < p_1\} \arrow[r] \arrow[d] & \DD\{p_1\} \arrow[d, "j_{\ast}" right, hooked] \arrow[dl, phantom, "\Longleftarrow" sloped] \\ 
				\DD\{p_0\} \arrow[r, "i_{\ast}" below, hooked] & \DD\{p_0\}\orientedcupbc^{\DD\{p_0 < p_1\}}\DD\{p_1\}
			\end{tikzcd}
		\end{equation*}
		is a gluing square.

		\item For every chain $\{p_0 < \cdots < p_m\}\subseteq P$, the geometric morphism to the fiber product of \topoi
		\begin{equation*}
			\DD\{p_0 < \cdots < p_m\}\to \DD\{p_0 < p_1\} \crosslimits_{\DD\{p_1\}}\DD\{p_1 < p_2\}\crosslimits_{\DD\{p_2\}}\cdots\crosslimits_{\DD\{p_{m-1}\}}\DD\{p_{m-1} < p_m\}
		\end{equation*}
		is an equivalence.
	\end{enumerate}
	We write $ \DecTop{P}\subseteq\Fun(\sdop(P),\Topbc) $\index[notation]{DecTopP@$ \DecTop{P} $} for the full subcategory spanned by the décollages over $P$.
\end{dfn}

\begin{rmk}
	It seems likely that a décollage over $P$ can be thought of as a suitable category internal to \smash{$ \Topbc $} along with a conservative functor to $P$.
	Making such an interpretation precise and helpful is a task that lies outside the scope of this work.
\end{rmk}

\begin{ntn} 
	If $\DD\colon\fromto{\sdop(P)}{\Topbc}$ is a décollage over $P$, and if $p,q\in P$ are points with $p<q$, then for the sake of typographical brevity, we write
	\begin{equation*}
		\DD\{p\}\orientedcup\DD\{q\} \coloneq \DD\{p\} \orientedcupbc^{\DD\{p < q\}} \DD\{q\} \period \index[notation]{cup@$\orientedcup$}
	\end{equation*}
	The two conditions of \Cref{def:toposicdecollage} specify, for each chain $\{p_0 < \cdots < p_m\}\subseteq P$, an equivalence
	\begin{equation*}
		\equivto{\DD\{p_0 < \cdots < p_m\}}{\DD\{p_0\} \underset{\DD\{p_0\}\orientedcup\DD\{p_1\}}{\orientedtimes} \DD\{p_1\} \underset{\DD\{p_1\}\orientedcup\DD\{p_2\}}{\orientedtimes} \cdots \underset{\DD\{p_{m-1}\}\orientedcup\DD\{p_m\}}{\orientedtimes} \DD\{p_m\}} \semicolon
	\end{equation*}
	we call this the \defn{Segal equivalence}\index[terminology]{Segal equivalence} .
\end{ntn}

\begin{exm}
	Let $ P $ be a finite poset.
	The terminal object of $\DecTop{P}$ is the constant functor \smash{$\fromto{\sdop(P)}{\Topbc}$} whose value is the \topos $\Space$.
\end{exm}

Recall that we write $ \lambdapi \colon \incto{\Spaceprofin}{\Topbc} $ for the fully faithful embedding given by the assignment $ \goesto{Y}{\Ytilde} $ (\Cref{def:Stonetopos}).

\begin{exm}[the toposic décollage of a spatial décollage]\label{exm:profinitedecollageastoposic}
	Let $ P $ be a finite poset and let $ D \colon \fromto{\sdop(P)}{\Spaceprofin} $ be a profinite spatial décollage (\Cref{cnstr:profinitespacesanddecollages}).
	Since the functor
	\begin{equation*}
		\lambdapi \colon \incto{\Spaceprofin}{\Topbc}
	\end{equation*}
	is left exact, \Cref{exm:gluingstonetopoi} shows that the functor $ \fromto{\sdop(P)}{\Topbc} $ defined by the assignment
	\begin{equation*}
		\Sigma \mapsto \widetilde{D(\Sigma)}
	\end{equation*}
	is a toposic décollage over $ P $.
	That is to say, objectwise application of $ \lambdapi $ defines a fully faithful functor
	\begin{equation*}
		\lambdapi \of - \colon \fromto{\DecSpaceprofin{P}}{\DecTop{P}} \period
	\end{equation*}
\end{exm}

The following is immediate from \Cref{exm:profinitedecollageastoposic} and the fact that $ \lambdapi $ defines an equivalence between profinite stratified spaces and Stone \topoi.

\begin{prp}\label{prp:decprehochster} 
	Let $P$ be a finite poset.
	Then the essential image of the fully faithful functor
	\begin{equation*}
		\incto{\DecSpaceprofin{P}}{\DecTop{P}}
	\end{equation*}
	given by the objectwise application of $\lambdapi \colon \incto{\Spaceprofin}{\Topbc}$ is the full subcategory
	\begin{equation*}
		\DecTopStone{P} \subset \DecTop{P}
	\end{equation*}
	spanned by those décollages over $P$ that carry each chain to a Stone \topos.
\end{prp}


We finish this section by encoding the functoriality of the décollage construction as in \Cref{cnstr:DecS}.

\begin{cnstr}[functoriality of toposic décollages]
	Write 
	\begin{equation}\label{eq:cartfiboverPosfin}
		\fromto{\textstyle\int_{\Posfin} \Fun(\sdop,\Topbc)}{\Posfin}
	\end{equation}
	for the cartesian fibration classified by the functor $ \fromto{(\Posfin)^{\op}}{\Cat_{\infty}} $ given by the assignment
	\begin{equation*}
		\goesto{P}{\Fun(\sdop(P),\Topbc)}
	\end{equation*}
	with functoriality given by right Kan extension \HTT{Corollary}{3.2.2.13}.
	Thus the objects of $ \int_{\Posfin} \Fun(\sdop,\Topbc) $ consist of pairs $ (P,\FF) $ of a finite poset $ P $ and a functor
	\begin{equation*}
		\FF \colon \fromto{\sdop(P)}{\Topbc} \period
	\end{equation*}
	The fiber of \eqref{eq:cartfiboverPosfin} over a poset $P$ is the \category $ \Fun(\sdop(P),\Topbc) $.

	Let
	\begin{equation*}
		\DecTop{} \subset \textstyle\int_{\Posfin} \Fun(\sdop,\Topbc)\index[notation]{DecTop@$ \DecTop{} $} 
	\end{equation*}
	denote the full subcategory spanned by the pairs $(P,\DD)$ in which $\DD$ is a toposic décollage over $ P$.
	Since $\DecTop{}$ contains all the cartesian edges, the functor
	\begin{equation*}
		\fromto{\DecTop{}}{\Posfin}
	\end{equation*}
	is a cartesian fibration.
\end{cnstr}


\subsection{The nerve of a stratified \texorpdfstring{$\infty$}{∞}-topos}\label{subsec:toposicnerve}

We now explain how to every $ P $-stratified \topos gives rise to a toposic décollage over $ P $ and prove that the resulting \textit{nerve functor}
\begin{equation*}
	\NNerve_{P} \colon \fromto{\StrTopbcc_{\infty,P}}{\DecTop{P}}
\end{equation*}
is an equivalence (\Cref{thm:strattopanddecollages}).

\begin{cnstr}[the nerve of a stratified \topos]\label{cnstr:nerveofstrattopoi} 
	Let $P$ be a finite poset, and let $\flowerstar\colon\fromto{\XX}{\Ptilde}$ be a $P$-stratified \topos.
	Then for any monotonic map $\phi\colon\fromto{Q}{P}$, we define the \defn{\topos of sections of $\XX$ over $Q$} as the pullback of \topoi
	\begin{equation*}
		\MOR_{\Ptilde}(\widetilde{Q},\XX)\coloneq\MOR(\widetilde{Q},\XX) \crosslimits_{\MOR(\widetilde{Q},\Ptilde)} \widetilde{\{\phi\}} \period
	\end{equation*}
	The \topos $\MOR_{\Ptilde}(\widetilde{Q},\XX)$ depends only on the pullback $\XX\times_{\Ptilde}\widetilde{Q}$:
	\begin{equation*}
		\MOR_{\Ptilde}(\widetilde{Q},\XX)\simeq\MOR_{\widetilde{Q}}(\widetilde{Q},\XX\times_{\Ptilde}\widetilde{Q}) \period
	\end{equation*}

	We thus obtain a functor $\NNerve_P(\XX)\colon\fromto{\sdop(P)}{\Top_{\infty}}$ that carries a chain $\Sigma\subseteq P$ to the \topos
	\begin{equation*}
		\NNerve_P(\XX)(\Sigma)\coloneq\MOR_{\Ptilde}(\widetilde{\Sigma},\XX) \period
	\end{equation*}
	For each chain $\{p_0 < \cdots < p_m\} \subseteq P$, we have a natural identification
	\begin{equation*}
		\NNerve_P(\XX)\{p_0 < \cdots < p_m\} \simeq \XX_{p_0}\orientedtimes_{\XX}\XX_{p_1}\orientedtimes_{\XX}\cdots\orientedtimes_{\XX}\XX_{p_m} \period
	\end{equation*}

	In particular, if $\XX$ is bounded coherent constructible (\Cref{def:boundedcoherentconstrstratifications}), then the functor $\NNerve_P(\XX)$ is a décollage over $P$.
	We call $\NNerve_P(\XX)$\index[notation]{NerveP@$\NNerve_{P}$} the \defn{nerve}\index[terminology]{nerve!of a stratified topos@of a stratified \topos} of the $P$-stratified \topos $\XX$, and we call the functor 
	\begin{equation*}
		\NNerve \colon \fromto{\restrict{\StrTopbcc_{\infty}}{\Posfin}}{\DecTop{}} 
	\end{equation*}
	over $\Posfin$ the \defn{nerve functor}.
\end{cnstr}

\begin{exm}[compatibility of nerves]\label{exm:NrhoisrhoN}
	Let $P$ be a finite poset, and $\Pi$ a \pifinite $P$-stratified space.
	We have an identification
	\begin{equation*}
		\NNerve_P(\Pitilde)\simeq\widetilde{\Nerve_P(\Pi)} \comma
	\end{equation*}
	natural in $P$ and $\Pi$.
	To see this, note that for any chain $\Sigma\subseteq P$, the natural morphism 
	\begin{equation*}
		\fromto{\widetilde{\Map_P(\Sigma,\Pi)}}{\MOR_{\Ptilde}(\widetilde{\Sigma},\Pitilde)}
	\end{equation*}
	is an equivalence.
\end{exm}

We now proceed to demonstrate that the nerve is an equivalence of \categories.

\begin{thm}\label{thm:strattopanddecollages}
	For any finite poset $P$, the nerve functor $\NNerve_P\colon\fromto{\StrTopbcc_{\infty,P}}{\DecTop{P}}$ is an equivalence of \categories.
\end{thm}

\begin{proof}
	We begin by reducing to the case in which $P$ is a nonempty, finite, totally ordered set. To make this reduction, we note that $P\simeq\colim_{\Sigma\in\sd(P)}\Sigma$, whence $ \Ptilde $ is the limit $\Ptilde\simeq\lim_{\Sigma\in\sdop(P)}\widetilde{\Sigma}$ in $ \Cat_{\infty,\updelta_1} $ (which is the colimit in $ \Top_{\infty} $).
	Moreover,
	\begin{equation*}
		\sdop(P)\simeq\colim_{\Sigma\in\sd(P)}\sdop(\Sigma) \period
	\end{equation*}
	From this we deduce that
	\begin{equation*}
		\StrTopbcc_{\infty,P}\simeq\colim_{\Sigma\in\sd(P)}\StrTopbcc_{\infty,\Sigma} \andeq \DecTop{P}\simeq\colim_{\Sigma\in\sd(P)}\DecTop{\Sigma}\comma
	\end{equation*}
	which provides our reduction.

	Let $P=[n]\coloneq\{0 < \cdots < n\}$ be a nonempty totally ordered finite set.
	Define an inverse
	\begin{equation*}
		\UUup_n \colon \fromto{\DecTop{[n]}}{\StrTop_{\infty,[n]}^{\bcc}}
	\end{equation*}
	to the nerve functor $\NNerve_{n} \coloneq \NNerve_{[n]} $ as follows.
	Write $ \UUup_n(\DD) $ for the iterated bounded coherent oriented pushout
	\begin{equation*}
		\UUup_n(\DD) \coloneq \DD\{0\}\orientedcupbc^{\DD\{0 < 1\}}\DD\{1\}\orientedcupbc^{\DD\{1 < 2\}}\dots\orientedcupbc^{\DD\{n-1 < n\}}\DD\{n\} \comma
	\end{equation*}
	equipped with its canonical geometric morphism to
	\begin{equation*}
		\widetilde{[n]}\simeq\UUup_n(\Space) \period
	\end{equation*}
	Note that it is immediate from the definition that the geometric morphism $ \fromto{\UUup_n(\DD)}{\widetilde{[n]}} $ is coherent.

	The universal properties of the iterated bounded coherent oriented pushout and the iterated oriented pullback provide natural transformations $ \UUup_n\NNerve_n \to \id $ and $\id\to\NNerve_n\UUup_n$. We aim to show that these natural transformations are equivalences.

	We prove that the natural morphisms $ \UUup_n\NNerve_n \to \id $ and $\id\to\NNerve_n\UUup_n$ are equivalences by induction on $n$.
	The base case $n=0$ is obvious.
	Now assume that $n\geq 1$ and that the natural morphism $ \UUup_{n-1}\NNerve_{n-1} \to \id $ is an equivalence; we prove that the natural morphism $ \UUup_{n}\NNerve_{n} \to \id $ is an equivalence.
	If $\XX$ is a bounded coherent \topos with a constructible stratification $\XX\to\widetilde{[n]}$, then consider the recollement of $\XX$ given by $\XX_{[n-1]}$ and $\XX_n$.
	We thus have a gluing square
	\begin{equation*}
		\begin{tikzcd}
			\XX_{[n-1]}\orientedtimes_{\XX}\XX_n \arrow[r, "q_{\ast}" above] \arrow[d, "p_{\ast}" left] & \XX_n \arrow[d, "j_{\ast}" right, hooked] \arrow[dl, phantom, "\scriptstyle \sigma" below right, "\Longleftarrow" sloped] \\ 
			\XX_{[n-1]} \arrow[r, "i_{\ast}" below, hooked] & \XX \period
		\end{tikzcd}
	\end{equation*}
	As a result, we compute:
	\begin{align*}
		\UUup_n\NNerve_n(\XX) &\simeq \UUup_{n-1}\NNerve_{n-1}(\XX_{[n-1]})\orientedcupbc^{\XX_{[n-1]}\orientedtimes_{\XX}\XX_n}\XX_n \\
		&\simeq \XX_{[n-1]} \orientedcupbc^{\XX_{[n-1]} \orientedtimes_{\XX} \XX_n} \XX_n \\
		&\simeq \XX \comma
	\end{align*}
	as desired.

	Now assume that the natural morphism $\id \to \NNerve_{n-1}\UUup_{n-1}$ is an equivalence; we prove that the natural morphism $\id\to\NNerve_n\UUup_n$ is an equivalence.
	Let $\DD\colon\fromto{\sdop([n])}{\Topbc}$ be a toposic décollage; we need to show that for every chain $ \Sigma \subset [n] $, the natural morphism $ \DD(\Sigma) \to \NNerve_n\UUup_n(\DD)(\Sigma) $ is an equivalence.
	There are two cases to consider: $ \Sigma \neq [n] $ and $ \Sigma = [n] $.
	If $ \Sigma \neq [n] $, then there exists an elment $ k \in [n] $ such that $ k \notin \Sigma $.
	Then applying the inductive hypothesis we see that the map $\DD(\Sigma)\to\NNerve_n\UUup_n(\DD)(\Sigma)$ factors as a composite of equivalences
	\begin{align*}
		\DD(\Sigma) &\simeq (\restrict{\DD}{\sdop([n] \smallsetminus \{k\})})(\Sigma) \\ 
		&\equivalence \NNerve_{[n] \smallsetminus \{k\}} \UUup_{[n] \smallsetminus \{k\}}(\restrict{\DD}{\sdop([n] \smallsetminus \{k\})})(\Sigma) \\ 
		&\simeq \NNerve_n\UUup_n(\DD)(\Sigma) \period
	\end{align*}
	In the case that $ \Sigma=[n] $, note that the morphism $\DD([n])\to\NNerve_n\UUup_n(\DD)([n])$ is homotopic to the Segal equivalence
	\begin{equation*}
		\equivto{\DD\{0 < \cdots < n\}}{\DD\{0\}\underset{\UUup_n(\DD)}{\orientedtimes}\DD\{1\}\underset{\UUup_n(\DD)}{\orientedtimes}\cdots\underset{\UUup_n(\DD)}{\orientedtimes}\DD\{n\}} \comma
	\end{equation*}
	whence our claim.
\end{proof}


\subsection{Profinite stratified spaces as stratified \texorpdfstring{$\infty$}{∞}-topoi}\label{subsec:lambdahatff}

In this section we extend the functor $ \uplambda \colon \fromto{\Strfin}{\restrict{\StrTopbcc_{\infty}}{\Posfin}}$ given by $\goesto{\Pi}{\Pitilde}$ (\Cref{ntn:lambdaP}) to a functor on profinite stratified spaces.
By comparing to the décollage approach to stratified \topoi, we prove that the resulting functor $ \fromto{\Strprofin}{\StrTopbcc_{\infty}} $ is fully faithful (\Cref{prp:fullfaithfulnessoftilde})

\begin{nul}\label{nul:strdecomparisonfinite}
	In light of \Cref{exm:NrhoisrhoN}, for each finite poset $P$, the diagram
	\begin{equation*}
		\begin{tikzcd}
			\StrfinP \arrow[d, "\Nerve_P" left, "\wr"{right, xshift=-0.25ex}] \arrow[r, "\uplambda_P" above] & \StrTopbcc_{\infty,P} \arrow[d, "\NNerve_P" right, "\wr"{left, xshift=0.25ex}]\\ 
			\DecSpacefin{P} \arrow[r, "\lambdapi \circ-" below, hooked] & \DecTop{P}
		\end{tikzcd}
	\end{equation*}
	commutes and the vertical functors are equivalences (\Cref{dfn:finstratspaces}, \Cref{nul:finitedecollages}, and \Cref{thm:strattopanddecollages}).
	Since the functor $ \fromto{\DecSpacefin{P}}{\DecTop{P}} $ given by composition with \smash{$ \lambdapi \colon \incto{\Spaceprofin}{\Topbc} $} is fully faithful (\Cref{exm:profinitedecollageastoposic}), the functor $ \uplambda_P $ is fully faithful.
\end{nul}

\begin{nul}
	The functor $ \uplambda \colon \fromto{\Strfin}{\restrict{\StrTopbcc_{\infty}}{\Posfin}} $ is left exact.
	To see this, we combine two facts.
	First, the functor \smash{$\fromto{\Posfin}{\Topbc}$} given by $\goesto{P}{\Ptilde}$ is left exact.
	Second, for any finite poset $P$, the functor
	\begin{equation*}
		\uplambda_P \colon \fromto{\StrfinP}{\StrTopbcc_{\infty,P}} \comma
	\end{equation*}
	when regarded as a functor $\fromto{\DecSpacefin{P}}{\DecTop{P}}$, is equivalent to composition with $\uplambda_{\{0\}}$, so it too is left exact.
\end{nul}

\begin{cnstr} 
	Since bounded coherent constructible stratified \topoi are closed under the formation of inverse limits in $ \StrTop_{\infty} $ (\Cref{cor:SAG.A.8.3.3}=\allowbreak\SAG{Corollary}{A.8.3.3}), we can now apply \Cref{nul:proexistentadjoint} and extend $\uplambda$ to a functor
	\begin{equation*}
		\lambdahat \colon \fromto{\Strprofin}{\StrTopbcc_{\infty}} \index[notation]{lambdahat@$\lambdahat$}.
	\end{equation*}
	over $\TSpcspec$.
	We denote this functor by the assignment $\goesto{\mbfPi}{\mbfPitilde}$\index[notation]{Pitilde@$\mbfPitilde$}.
\end{cnstr}

\begin{wrn}
	If $ S $ is a spectral topological space and $\Pi$ is a profinite $ S $-stratified space, then although $ S $ determines and is determined by the $\mat(S)$-stratified space $\mat(\Pi)$, the \topoi \smash{$\Pitilde$} and \smash{$\widetilde{\mat(\Pi)}$} are quite different in general.
	The latter is always a presheaf \category, but the former is typically not.
\end{wrn}

\begin{nul}\label{nul:strdecomparisonprofinite}
	Let $ P $ be a finite poset.
	We generalize \Cref{nul:strdecomparisonfinite} as follows.
	Combining \Cref{cnstr:profinitespacesanddecollages,exm:profinitedecollageastoposic,thm:strattopanddecollages}, we see that the square
	\begin{equation*}
		\begin{tikzcd}
			\StrprofinP \arrow[d, "\Nerve_P" left, "\wr"{right, xshift=-0.25ex}] \arrow[r, "\lambdahat_P" above] & \StrTopbcc_{\infty,P} \arrow[d, "\NNerve_P" right, "\wr"{left, xshift=0.25ex}]\\ 
			\DecSpaceprofin{P} \arrow[r, "\lambdapi \circ-" below, hooked] & \DecTop{P} 
		\end{tikzcd}
	\end{equation*}
	commutes.
	Moreover, the vertical functors are equivalences and the bottom horizontal functor is fully faithful. 
	Hence the functor \smash{$ \lambdahat_{P} \colon \incto{\StrprofinP}{\StrTopbcc_{\infty,P}} $} is fully faithful.
\end{nul}

\begin{prp}\label{prp:fullfaithfulnessoftilde} 
	The functor $\lambdahat$ is fully faithful.
	In particular, for every spectral topological space $ S $, the functor
	\begin{equation*}
		\lambdahat_{S} \colon \incto{\StrprofinS}{\StrTopbcc_{\infty,S}}
	\end{equation*}
	is fully faithful.
\end{prp}

\begin{proof}
	Note that we have natural identifications
	\begin{equation*}
		\StrprofinS \simeq \lim_{P\in \FC(S)} \StrprofinP \andeq \StrTopbcc_{\infty,S}\simeq\lim_{P\in \FC(S)}\StrTopbcc_{\infty,P} \comma
	\end{equation*}
	the first of which is \Cref{prp:profinstratSislimoverFS} and the latter of which is obvious.
	The claim now follows from the fact that for each finite poset $ P $, the functor $ \lambdahat_P $ is fully faithful \Cref{nul:strdecomparisonprofinite}.
\end{proof}

\begin{nul}
	Let $ P $ be a finite poset.
	In light of \Cref{prp:decprehochster,prp:fullfaithfulnessoftilde} we see that an $ \fromto{\XX}{\Ptilde} $ is in the essential image of
	\begin{equation*}
		\lambdahat_{P} \colon \incto{\StrprofinP}{\StrTopbcc_{\infty,P}}
	\end{equation*}
	if and only if for every chain $ \{p_0 < \cdots < p_n\} \subset P $, the iterated oriented fiber product
	\begin{equation*}
		\XX_{p_0}\orientedtimes_{\XX}\XX_{p_1}\orientedtimes_{\XX}\cdots\orientedtimes_{\XX}\XX_{p_m}
	\end{equation*}
	is a Stone \topos.
	In the next chapter, we will see that this is equivalent to the \textit{a priori} weaker condition that the strata of $ \XX $ be Stone (\Cref{prp:orientedprodisStone}).
	We will also characterize the essential image of the functor
	\begin{equation*}
		\lambdahat \colon \incto{\Strprofin}{\StrTopbcc_{\infty}} 
	\end{equation*}
	in general, and provide a number of intrinsic descriptions of stratified \topoi in the essential image of $ \lambdahat $ (\Cref{lem:specStoneequiv,thm:inftyHochster,cor:spectralclassificationconstructible}).
\end{nul}

We conclude this chapter with a few remarks describing stratified geometric morphisms $\fromto{\XX}{\Pitilde}$ in a more familiar fashion.
Let us begin with the case in which the base poset is trivial.

\begin{nul}\label{nul:fullyfaithfultilde}
	In light of \Cref{HTT.6.3.5.6}=\allowbreak\HTT{Corollary}{6.3.5.6}, if $\Pi$ is a $\updelta_0$-small space, then for any \topos $ \XX $ there is a natural equivalence
	\begin{equation*}
		\equivto{\Map_{\Pro(\Space)}(\Shape(\XX),\Pi)}{\Funlowerstar(\XX,\Pitilde)} \period
	\end{equation*}
	Here $\Shape(\XX)$ is the shape prospace of \Cref{dfn:shape}.
	In particular, $\Funlowerstar(\XX,\Pitilde)$ is a $\updelta_0$-small \groupoid.

	In particular, if $\Pi,\Pi'$ are $\updelta_0$-small spaces, then the natural map
	\begin{equation*}
		\fromto{\Map_{\Space}(\Pi',\Pi)}{\Map_{\Top_{\infty}}(\Pitilde',\Pitilde)}
	\end{equation*}
	is an equivalence.
\end{nul}

Now we extend this result to the context of $P$-stratified \topoi.

\begin{ntn} 
	Let $P$ be a finite poset, and let $\flowerstar\colon\fromto{\XX}{\Ptilde}$ and $g_{\ast}\colon\fromto{\YY}{\Ptilde}$ be $P$-stratified \topoi. 
	Let us write
	\begin{equation*}
		\Fun_{P,\ast}(\XX,\YY) \coloneq \Funlowerstar(\XX,\YY) \crosslimits_{\Funlowerstar(\XX,\Ptilde)} \{\flowerstar\} \period \index[notation]{FunPlowerstar@$\Fun_{P,\ast}(\XX,\YY)$}
	\end{equation*}
	The mapping space $\Map_{\StrTop_{\infty,P}}(\XX,\YY)$ is the maximal sub-\groupoid of $\Fun_{P,\ast}(\XX,\YY)$.

	If $\XX$ and $\YY$ are bounded coherent and constructibly stratified, then in light of \Cref{thm:strattopanddecollages}, there is a natural equivalence of \categories
	\begin{equation*}
		\Fun_{P,\ast}(\XX,\YY) \simeq \int_{\Sigma\in\sdop(P)}\Funlowerstar(\NNerve_P(\XX)(\Sigma),\NNerve_P(\YY)(\Sigma)) \period
	\end{equation*}
\end{ntn}

\noindent This implies the following.

\begin{prp} 
	Let $P$ be a finite poset and $\XX$ a bounded coherent constructible $ P $-stratified \topos.
	Then for any \pifinite $P$-stratified space $\Pi$, there is a natural equivalence
	\begin{equation*}
		\Fun_{P,\ast}(\XX,\Pitilde) \simeq \int_{\Sigma\in\sdop(P)} \Map_{\Pro(\Space)}(\Shape(\NNerve_P(\XX)(\Sigma)), \Nerve_P(\Pi)(\Sigma)) \period
	\end{equation*}
	In particular, the \category $ \Fun_{P,\ast}(\XX,\Pitilde) $ is \agroupoid.
\end{prp}



\noindent Additionally, \Cref{prp:fullfaithfulnessoftilde} implies the following.

\begin{cor}\label{cor:fullfaithfulnessonstratspaces} 
	For any finite poset $ P $ and \pifinite $P$-stratified spaces $\Pi$ and $\Pi'$, the functor
	\begin{equation*}
		\Map_{P}(\Pi',\Pi)\to\Fun_{P,\ast}(\Pitilde',\Pitilde)
	\end{equation*}
	is an equivalence.
	That is, the functor $\uplambda_P$ is a fully faithful functor $\incto{\StrfinP}{\StrTopbcc_{\infty,P}}$.
\end{cor}

\newpage

\section{Spectral higher topoi}\label{sec:spectraltopoi}

In this chapter, we define the notion of a \textit{spectral \topos.}
The idea is that, on one hand, these are the kinds of \topoi that arise as the étale \topoi of coherent schemes, and on the other, these turn out to be precisely the \topoi that arise as $\mbfPitilde$ for some profinite stratified space $\mbfPi$.

\Cref{subsec:StoneOFP} begins by showing that in an oriented fiber product of bounded coherent \topoi $ \orientedpull{\XX}{\ZZ}{\YY} $, if $ \XX $ and $ \YY $ are Stone, then $ \orientedpull{\XX}{\ZZ}{\YY} $ is Stone; this is key to understand the links in our décollage approach to \textit{spectral \topoi} developed in \cref{subsec:toposicdecollage}.
\Cref{subsec:higherHochster} states and proves our \Categorical Hochster Duality Theorem (\Cref{thm:headlineinftyHochster}).
The \Categorical Hochster Duality Theorem provides an equivalence between profinite stratified spaces and spectral \topoi.
\Cref{subsec:constructiblesheaves} is dedicated to the study of constructible sheaves in the setting of stratified \topoi.
In \Cref{subsec:cohandconstructiblesheaves} we show that spectral \topoi are characterized by the requirement that the constructible sheaves coincide with the truncated coherent objects.


\subsection{Stone \texorpdfstring{$\infty$}{∞}-topoi \& oriented fiber products}\label{subsec:StoneOFP}

In this section we prove two useful facts about oriented fiber products involving Stone \topoi.

\begin{prp}\label{prp:orientedfiboverStoneisfib}
	Let $ \flowerstar \colon \fromto{\XX}{\ZZ} $ and $ \glowerstar \colon \fromto{\YY}{\ZZ} $ be geometric morphisms of \topoi.
	If $ \ZZ $ is Stone, then the natural geometric morphism $ \fromto{\XX \cross_{\ZZ} \YY}{\orientedpull{\XX}{\ZZ}{\YY}} $ is an equivalence.
\end{prp}

\begin{proof}
	It suffices to show that the projections $ \pr_{1,\ast},\pr_{2,\ast} \colon \fromto{\Path{\ZZ}}{\ZZ} $ are equivalences.
	Since $ \ZZ $ is Stone, by \Cref{lem:orientedcoherent} the \topos $ \Path{\ZZ} $ is bounded coherent.
	Moreover, \Cref{thm:Stonesaretopoiwithgroupoidsofpoints}=\allowbreak\SAG{Theorem}{E.3.4.1} shows that the \category $ \Pt(\ZZ) $ is \agroupoid. 
	Thus
	\begin{equation*}
		\Pt(\Path{\ZZ}) \equivalent \Fun([1],\Pt(\ZZ))
	\end{equation*}
	is \agroupoid as well. 
	Again appealing to \Cref{thm:Stonesaretopoiwithgroupoidsofpoints}=\allowbreak\SAG{Theorem}{E.3.4.1} we conclude that $ \Path{\ZZ} $ is Stone.
	The claim now follows from the fact that $ \pr_{1,\ast} $ and $ \pr_{2,\ast} $ are shape equivalences (\Cref{ex:pathtoposislocalandcolocal}).
\end{proof}

\begin{prp}\label{prp:orientedprodisStone}
	Let $ \XX $ and $ \YY $ be Stone \topoi, $ \ZZ $ a bounded coherent \topos, and $ \flowerstar \colon \fromto{\XX}{\ZZ} $ and $ \glowerstar \colon \fromto{\YY}{\ZZ} $ coherent geometric morphisms.
	Then the oriented fiber product $ \orientedpull{\XX}{\ZZ}{\YY} $ is a Stone \topos.
\end{prp}

\begin{proof}
	By \Cref{lem:orientedcoherent} the \topos $ \orientedpull{\XX}{\ZZ}{\YY} $ is bounded coherent, so by \Cref{thm:Stonesaretopoiwithgroupoidsofpoints}=\allowbreak\SAG{Theorem}{E.3.4.1} it suffices to prove that the \category $ \Pt(\orientedpull{\XX}{\ZZ}{\YY}) $ is \agroupoid.
	In light of \Cref{lem:Ptpreservesorientedprod} there is an equivalence
	\begin{equation*}
		\Pt(\orientedpull{\XX}{\ZZ}{\YY}) \equivalent \commacat{\Pt(\XX)}{\Pt(\ZZ)}{\Pt(\YY)} \comma
	\end{equation*}
	so the fact that $ \Pt(\XX) $ and $ \Pt(\YY) $ are \groupoids implies that the \category $\Pt(\orientedpull{\XX}{\ZZ}{\YY})$ is as well.
\end{proof}


\subsection{Spectral \texorpdfstring{$\infty$}{∞}-topoi \& toposic décollages}\label{subsec:toposicdecollage}

In this section we define the \toposic generalization of spectral topological spaces relevant for our \Categorical Hochster Duality Theorem (\Cref{thm:inftyHochster}).

\begin{dfn}\label{def:spectraltopos}
	Let $ S $ be a spectral topological space.
	An $ S $-stratified \topos $ \fromto{\XX}{\Stilde} $ is a \defn{spectral}\index[terminology]{spectral!topos@\topos}\index[terminology]{topos@\topos!spectral} $ S $-stratified \topos if and only if the following conditions are satisfied:
	\begin{itemize}
		\item The \topos $ \XX $ is bounded and coherent.

		\item The stratification by $S$ is constructible.

		\item For every point $ s \in S $, the stratum $ \XX_s \colonequals \stilde \cross_{\Stilde} \XX $ is a Stone \topos.
	\end{itemize}
	We write $ \StrTopspec_{\infty,S} \subset \StrTopbcc_{\infty,S} $ for the full subcategory spanned by the spectral $ S $-stratified \topoi.

	More generally, write
	\begin{equation*}
		\StrTopspec_{\infty} \subset \StrTopbcc_{\infty} \index[notation]{StrTopspec@$\StrTopspec_{\infty}$}
	\end{equation*}
	for the full subcategory whose objects are spectral \topoi and whose morphisms are squares
	\begin{equation*}
		\begin{tikzcd}
			\XX' \arrow[r] \arrow[d] & \XX \arrow[d] \\ 
			\Stilde' \arrow[r] & \Stilde
		\end{tikzcd}
	\end{equation*}
	of coherent geometric morphisms.
	As a consequence of \Cref{lem:orientedcoherent} we observe that the pullback of a spectral \topos along the geometric morphism induced by a quasicompact continuous map is again spectral.
	Hence the functor $\fromto{\StrTopspec_{\infty}}{\TSpcspec}$ is a cartesian fibration.
\end{dfn}

\begin{exm}\label{exm:profinstratsarespectral}
	Let $\fromto{\mbfPi}{S}$ be a profinite stratified space (\Cref{dfn:profinitestratspace}).
	Since the fibers $ \mbfPitilde_s \simeq \widetilde{\mbfPi_s} $ are Stone \topoi, the $ S $-stratified \topos $ \mbfPitilde $ is spectral.
\end{exm}

\begin{nul}
	In \cref{subsec:higherHochster}, we prove the central \Categorical Hochster Duality Theorem, which states that \emph{every} spectral \topos is of the form $\mbfPitilde$ for some profinite stratified space. 
\end{nul}

The key example of a spectral \topos from algebraic geometry is the étale \topos of a coherent scheme.

\begin{exm}\label{exm:etaleisspectra}
	Let $X$ be a coherent scheme.
	We claim that $ X_{\et} $ is spectral with respect to the natural stratification $ \fromto{X_{\et}}{X_{\zar}} $ (\Cref{exm:natstratXet}).

	To see this, we need to show that for any point $ x_0 \in X^{\zar}$, the stratum \smash{$(X_{\et})_{x_0} $} is a Stone \topos.
	Combining the fact that the functor $ \Pt \colon \fromto{\Top_{\infty}}{\Cat_{\infty,\updelta_1}} $ preserves fiber products, with Conceptual Completeness (\Cref{thm:conceptualcompleteness}=\allowbreak\SAG{Theorem}{A.9.0.6}), we see that the natural square
	\begin{equation*}
      \begin{tikzcd}[sep=2.25em]
	       (\Spec \upkappa(x_0))_{\et} \arrow[d] \arrow[r] & X_{\et} \arrow[d] \\ 
	       (\Spec \upkappa(x_0))_{\zar} \arrow[r] & X_{\zar} 
      \end{tikzcd}
    \end{equation*}
    is a pullback square.
    To conclude, recall that a choice of separable closure of the residue field $ \upkappa(x_0) $ provides an identification of $ (\Spec \upkappa(x_0))_{\et} $ with the Stone \topos \smash{$\widetilde{\BG}_{\upkappa(x_0)}$} associated to the absolute Galois group of $ \upkappa(x_0) $ (\Cref{exm:etaletoposfield}).
	Consequently \smash{$ X_{\et} $} is a spectral \topos.
\end{exm}

The following is a convenient reformulation of the condition that a stratified \topos be spectral.

\begin{prp}\label{prop:spectralcheckonpts}
	Let $S$ be a spectral topological space, and let $\XX$ be a bounded coherent constructible $ S $-stratifed \topos.
	Then $\XX$ is spectral if and only if the functor
	\begin{equation*}
		\fromto{\Pt(\XX)}{\Pt(\Stilde)\simeq\mat(S)}
	\end{equation*}
	exhibits $\Pt(\XX)$ as a $\mat(S)$-stratified space.
\end{prp}

\begin{proof}
	This follows directly from \Cref{thm:Stonesaretopoiwithgroupoidsofpoints}=\allowbreak\SAG{Theorem}{E.3.4.1}. 
\end{proof}

\begin{nul}[spectral \topoi as décollages]
	Let $P$ be a finite poset. We now consider the nerve of a spectral $P$-stratified \topos $\fromto{\XX}{\Ptilde}$. Since each stratum $\XX_p$ is Stone, it follows from \Cref{prp:orientedprodisStone} that for any chain $\{p_0 < \cdots < p_n\}\subseteq P$, the value
	\begin{equation*}
		\NNerve_P(\XX)\{p_0 < \cdots < p_n\}\simeq\XX_{p_0}\orientedtimes_{\XX}\XX_{p_1}\orientedtimes_{\XX}\cdots\orientedtimes_{\XX}\XX_{p_n}
	\end{equation*}
	is a Stone \topos. 
	Consequently, we deduce that the equivalence
	\begin{equation*}
		\NNerve_P \colon\equivto{\StrTopbcc_{\infty,P}}{\DecTop{P}}
	\end{equation*}
	restricts to an equivalence between the \category of spectral $P$-stratified \topoi and the full subcategory $\DecTopStone{P}\subset\DecTop{P}$ spanned by those décollages over $P$ that carry each chain to a Stone \topos (\Cref{prp:decprehochster}).
\end{nul}

\begin{lem}\label{lem:specStoneequiv}
	Let $ P $ be a finite poset.
	Then the nerve equivalence 
	\begin{equation*}
		\NNerve_P \colon \equivto{\StrTopbcc_{\infty,P}}{\DecTop{P}}
	\end{equation*}
	restricts to an equivalence $ \equivto{\StrTopspec_{\infty,P}}{\DecTopStone{P}} $.
\end{lem}


\subsection{Hochster duality for higher topoi}\label{subsec:higherHochster}

In \Cref{nul:dualitycube} we described Hochster duality as a cube of dualities: the equivalence of $1$-categories between profinite posets and spectral topological spaces restricts on one hand to an equivalence between profinite sets and Stone spaces, and on the other to an equivalence between finite posets and finite topological spaces.
Our objective now is to exhibit the analogous cube for higher topoi: 
\begin{equation*}
	\begin{tikzcd}[column sep={10ex,between origins}, row sep={8ex,between origins}]
		& \Spacefin \arrow[dl, hooked'] \arrow[rr, "\sim"{yshift=-0.2em}] \arrow[dd, hooked]  & & \Top_{\infty}^{\fin} \arrow[dl, hooked'] \arrow[dd, hooked] \\
		\Spaceprofin \arrow[dd, hooked] \arrow[rr, crossing over, "\sim"{yshift=-0.2em, near start}] & & \Top_{\infty}^{\Stone} \\
		& \Strfin \arrow[dl, hooked'] \arrow[rr, "\sim"{yshift=-0.2em, near start}] & & \StrTop_{\infty}^{\fin} \arrow[dl, hooked'] \\
		\Strprofin \arrow[rr, "\sim"{yshift=-0.2em}] & & \StrTopspec_{\infty} \arrow[from=uu, crossing over, hooked] & \phantom{\StrTop_{\infty}^{\fin}} \period 
	\end{tikzcd}
\end{equation*}
Here the vertical fully faithful functors are given by equipping an object with the trivial stratification.
The top face of this cube was established by Lurie \cite[\SAGapp{E}]{SAG}.
We now address the bottom face, more precisely the equivalence $\Strprofin\simeq\StrTopspec_{\infty}$.

\begin{thm}[{\Categorical Hochster Duality\index[terminology]{Categorical Hochster Duality@\Categorical Hochster Duality}}]\label{thm:inftyHochster}
	Let $ S $ be a spectral topological space.
	Then the functor 
	\begin{equation*}
		\lambdahat_S\colon\fromto{\StrprofinS}{\StrTopspec_{\infty,S}}
	\end{equation*}
	given by the assignment $ \goesto{\mbfPi}{\mbfPitilde} $ is an equivalence of \categories. 
	Consequently, the functor
	\begin{equation*}
		\lambdahat\colon\fromto{\Strprofin}{\StrTopspec_{\infty}}
	\end{equation*}
	is an equivalence of \categories.
\end{thm}

\begin{proof}
	Since $\lambdahat$ is fully faithful (\Cref{prp:fullfaithfulnessoftilde}) and preserves inverse limits, it suffices to prove that for any finite poset $ P $, the fully faithful functor $ \lambdahat \colon \incto{\StrprofinP}{\StrTopspec_{\infty,P}} $ is essentially surjective.
	This follows from the conjunction of \Cref{lem:specStoneequiv} and \Cref{prp:decprehochster}.
\end{proof}

The back face of the cube is just a restriction of the front face: we define $\Top_{\infty}^{\fin}$ as the full subcategory of $ \TopStone $ spanned by the essential image of the fully faithful functor $ \incto{\Spacefin}{\TopStone} $ given by $ \goesto{\Pi}{\Space_{/\Pi} \equivalent \Fun(\Pi,\Space)} $.
Similarly, $\StrTop_{\infty}^{\fin}$ is the \category of bounded coherent constructible \topoi over a finite poset $P$ such that for every point $p\in P$, the \topos $\XX_p$ is in \smash{$\Top_{\infty}^{\fin}$}.


\subsection{Constructible sheaves}\label{subsec:constructiblesheaves}

The truncated coherent objects of a Stone \topos are exactly the \textit{lisse sheaves} (\Cref{rec:localsyslisse}). 
This turns out to be a defining property of Stone \topoi (\Cref{prop:SAG.E.3.1.1}=\allowbreak\SAG{Proposition}{E.3.1.1}). 
In the same manner, the truncated coherent objects of a spectral \topos are exactly the \textit{constructible sheaves}.
In this section we introduce constructible sheaves for stratified \topoi; in the next section, we prove this characterization of spectral \topoi in terms of constructible sheaves.

\begin{ntn}
	Let $P$ be a finite poset and $\XX$ a $P$-stratified \topos.
	For each $p\in P$, we write $e_{p,\ast}\colon\incto{\XX_p}{\XX}$ for the inclusion of the $p$-th stratum.
\end{ntn}

\begin{dfn}\label{def:Pconstructible}
	Let $P$ be a finite poset and $\XX$ a $P$-stratified \topos.
	An object $F\in \XX$ is \defn{formally constructible}\index[terminology]{formally constructible} (or \defn{formally $P$-constructible} if disambiguation is called for) if and only if, for every point $p\in P$, the restriction $ \restrict{F}{\XX_p}  \coloneq e_p^{\ast} F \in \XX_p $ is locally constant.

	We say that $F$ is \defn{constructible}\index[terminology]{constructible} (or \defn{$P$-constructible}) if and only if the following pair of conditions is satisfied:
	\begin{itemize}
		\item The object $F$ is formally constructible.

		\item For any point $p\in P$, the restriction $ \restrict{F}{\XX_p} \in \XX_p$ is lisse.
	\end{itemize}
	We write $ \XX^{\Pcons} \subset \XX $\index[notation]{XPcons@$ \XX^{\Pcons} $} for the full subcategory spanned by the $ P $-constructible sheaves.
\end{dfn}

\begin{nul}
	This notion of constructibility depends upon the whole structure of the stratified \topos, not only upon the underlying \topos.
\end{nul}

\begin{nul}\label{nul:Pconstrpretopos}
	For any finite poset $P$ and $P$-stratified \topos $ \fromto{\XX}{\Ptilde} $, the \category of $ P $-constructible sheaves on $\XX$ is given by the pullback of \categories:
	\begin{equation*}
		\begin{tikzcd}
			\XX^{\Pcons} \arrow[r] \arrow[d, hooked] \arrow[dr, phantom, very near start, "\lrcorner", xshift=-0.45em, yshift=0.25em] & \prod_{p\in P}\XX_p^{\lisse} \arrow[d, hooked] \\ 
			\XX \arrow[r, "(e_p^{\ast})_{p\in P}" below] & \prod_{p\in P}\XX_p\phantom{^{\lisse}} \comma
		\end{tikzcd}
	\end{equation*}
	where \smash{$ \prod_{p\in P}\XX_p $} is the product in $ \Cat_{\infty,\updelta_1} $.
	\Cref{lem:prodofpretopoi,lem:pullbackofpretopoi} now show that \smash{$ \XX^{\Pcons} $} is \apretopos (\Cref{dfn:pretopos}) and the inclusion \smash{$ \incto{\XX^{\Pcons}}{\XX} $} is a morphism of \pretopoi.
\end{nul}

The pullback functor in a geometric morphism of \topoi preserves lisse objects (see \Cref{rec:localsyslisse}); in the same manner, the pullback of a morphism of stratified \topoi preserves constructible objects.

\begin{lem}\label{lem:morspreserveconstr}
	Let $ f \colon \fromto{P}{Q} $ be a morphism of finite posets, and let $ \fromto{\XX}{\Ptilde} $ and $ \fromto{\YY}{\Qtilde} $ be stratified \topoi.
	Then for any geometric morphism $ \qlowerstar \colon \fromto{\XX}{\YY} $ over $ \flowerstar \colon \fromto{\Ptilde}{\Qtilde} $, the pullback $ \qupperstar \colon \fromto{\YY}{\XX} $ sends $ Q $-constructible objects of $ \YY $ to $ P $-constructible objects of $ \XX $.
	Hence $ \qupperstar $ restricts to a morphism of \pretopoi
	\begin{equation*}
		\qupperstar \colon \fromto{\YY^{\Qcons}}{\XX^{\Pcons}} \period
	\end{equation*}
\end{lem}

\begin{proof}
	Let $ F \in \YY^{\Qcons} $ be a $ Q $-constructible object of $ \YY $.
	Then for any point $ p \in P $, the restriction $ \restrict{F}{\YY_{f(p)}} $ is lisse.
	Since the pullback in a geometric morphism preserves lisse objects, we see that the object $ \qupperstar(F)|_{\XX_p} $ is lisse.
	Hence $ \qupperstar(F) $ is $ P $-constructible.

	The fact that \smash{$ \qupperstar \colon \fromto{\YY^{\Qcons}}{\XX^{\Pcons}} $} is a morphism of \pretopoi is immediate from \Cref{nul:Pconstrpretopos}.
\end{proof}

A key feature of \pretopos of constructible sheaves is that it is alway bounded.
We make use of this fact repeatedly throughout the text.

\begin{prp}\label{prop:Pconstrbounded}
	Let $ P $ be a finite poset and $ \fromto{\XX}{\Ptilde} $ a $ P $-stratified \topos.
	Then the \pretopos $ \XX^{\Pcons} $ is bounded (\Cref{dfn:boundedpretopos}).
\end{prp}

\begin{proof}
	If $ P = \emptyset $, then the claim is obvious, so assume that $ P $ is nonempty.
	We prove the claim by induction on the rank of $ P $.

	In the base case where $ P $ has rank $ 0 $, $ P $ is discrete, so $ \XX $ is finite the coproduct of \topoi $ \coprod_{p \in P} \XX_p $ (which is the product $ \prod_{p \in P} \XX_p $ in $ \Cat_{\infty,\updelta_1} $).
	Thus $ \XX^{\Pcons} $ is the product of \categories:
	\begin{equation*}
		\XX^{\Pcons} = \prod_{p \in P} \XX_p^{\lisse} \period
	\end{equation*}
	By \Cref{thm:SAG.E.2.3.2}=\allowbreak\SAG{Theorem}{E.2.3.2}, for all $ p \in P $ the \pretopos \smash{$ \XX_p^{\lisse} $} is bounded; the finiteness of $ P $ and \Cref{lem:prodofbddpretopoi} now show that $ \XX^{\Pcons} $ is also bounded.

	For the induction step, let $ n \geq 0 $ be a natural number and assume that the claim holds for all finite posets $ P $ of rank $ n  $ and $ P $-stratified \topoi $ \fromto{\XX}{\Ptilde} $.
	Let $ P $ be a finite poset of rank $ n+1 $, and write $ M \subset P $ for the full subposet spanned by the minimal elements of $ P $.
	Then $ M $ is discrete and closed in $ P $.
	Write $ U \coloneq P \smallsetminus M $ for the open complement of $ M $ in $ P $.
	Then $ U $ is a poset of rank $ n $.
	Moreover, since $ \Ptilde $ is the recollement of $ \Mtilde $ and $ \Utilde $, the $ P $-stratified \topos $ \XX $ is the recollement of $ \XX_M $ and $ \XX_U $.
	An object $ F \in \XX $ is $ P $-constructible if and only if $ \restrict{F}{\XX_M}  $ and $ \restrict{F}{\XX_U}  $ are both constructible, from which we deduce that $ \XX^{\Pcons} $ is the oriented fiber product of \categories
	\begin{equation*}
		\XX^{\Pcons} \equivalent \commacat{\XX_M^{\Mcons}}{\XX_M}{\XX_U^{\Ucons}} \period 
	\end{equation*}
	Since $ M $ is a poset of rank $ 0 $ and $ U $ is a poset of rank $ n $, by the induction hypothesis both $ \XX_M^{\Mcons} $ and $ \XX_U^{\Ucons} $ are bounded \pretopoi.
	To conclude that the \pretopos $ \XX^{\Pcons} $ is a bounded, note that by \Cref{nul:usingconservativityofrecoll} every object of $ \XX^{\Pcons} $ is truncated and by \Cref{nul:orientedfpinCatesssmall} the \category $ \XX^{\Pcons} $ is $ \updelta_0 $-small.
\end{proof}

Now we extend the definition of constructibility to \topoi stratified over a spectral topological space.

\begin{dfn}
	Let $S$ be a spectral topological space and $\XX$ an $S$-stratified \topos.
	We say that an object $F\in \XX$ is \defn{formally constructible}\index[terminology]{formally constructible} (or \defn{formally $ S $-constructible}) if and only if there exist a finite poset $P$ and a constructible stratification $\fromto{S}{P}$ such that $F$ is formally $P$-constructible.
	We say that $F$ is \defn{constructible}\index[terminology]{constructible} (or \defn{$S$-constructible}) if and only if there exist a poset $P$ and a finite constructible stratification $\fromto{S}{P}$ such that $F$ is $P$-constructible.

	We denote by $\XX^{\Sfcons}\subset \XX$\index[notation]{XXSfcons@$\XX^{\Sfcons}$} (respectively, by $\XX^{\Scons}\subset \XX$\index[notation]{XXScons@$\XX^{\Scons}$}) the full subcategory spanned by the formally constructible objects (respectively, the constructible objects).
\end{dfn}

\begin{nul}
	For any spectral topological space $S$ and $S$-stratified \topos $\fromto{\XX}{\Stilde}$, the \category of constructible sheaves on $\XX$ is thus a filtered colimit of \categories:
	\begin{equation*}
		\XX^{\Scons} \simeq \colim_{P\in \FC(S)^{\op}}\XX^{\Pcons} \period
	\end{equation*}
	Therefore \Cref{lem:morspreserveconstr} and \Cref{prop:Pconstrbounded} combined with \Cref{prop:SAG.A.8.3.1}=\allowbreak\SAG{Proposition}{A.8.3.1} show that $ \XX^{\Scons} $ is a bounded \pretopos.
	Moreover, \Cref{nul:Pconstrpretopos} shows that the inclusion $ \incto{\XX^{\Scons}}{\XX} $ is a morphism of \pretopoi.
\end{nul}

\noindent From \Cref{lem:morspreserveconstr} we immediately deduce the following.

\begin{lem}\label{lem:morspreserveconstrspectral}
	Let $ f \colon \fromto{S}{T} $ be a quasicompact continuous map of spectral topological spaces, and let $ \fromto{\XX}{\Stilde} $ and $ \fromto{\YY}{\Ttilde} $ be stratified \topoi.
	Then for any geometric morphism $ \qlowerstar \colon \fromto{\XX}{\YY} $ over $ \flowerstar \colon \fromto{\Stilde}{\Ttilde} $, the pullback $ \qupperstar \colon \fromto{\YY}{\XX} $ sends $ T $-constructible objects of $ \YY $ to $ S $-constructible objects of $ \XX $.
	Hence $ \qupperstar $ restricts to a morphism of \pretopoi
	\begin{equation*}
		\qupperstar \colon \fromto{\YY^{\Tcons}}{\XX^{\Scons}} \period
	\end{equation*}
\end{lem}


\subsection{Coherence \& constructibility}\label{subsec:cohandconstructiblesheaves}

We now turn to the relationship between coherence and constructibility in \topoi stratified by a spectral topological space.
The main result of this section is that a bounded coherent constructible stratified \topos is spectral if and only all of its truncated coherent objects are constructible (\Cref{cor:spectralclassificationconstructible}).

We begin by giving a characterization of constructibility in terms of local constancy over constructible subsets.

\begin{rec}\label{rec:subsetconstr}
	Let $ S $ be a spectral topological space.
	The collection of \defn{constructible} subsets of $ S $ is the smallest collection of subsets of $ S $ containing all quasicompact open subsets and closed under taking finite intersections and complements.
\end{rec}

\begin{lem}\label{lem:checkconstronnbd} 
	Let $S$ be a spectral topological space, and let $\XX$ be an $S$-stratified \topos.
	Then an object $F$ of $\XX$ is constructible if and only if, for every point $s\in S$, there exists a constructible subset $W\subseteq S$ containing $s$ such that $\restrict{F}{\XX_W} $ is lisse.
\end{lem}

\begin{proof} 
	The `only if' direction is clear.
	Conversely, assume that for every point $s\in S$, there exists a constructible subset $W\subseteq S$ containing $s$ such that $\restrict{F}{\XX_W} $ is lisse.
	Then the collection \smash{$\{W_{\alpha}\}_{\alpha \in A}$} of constructible subsets of $S$ such that \smash{$\restrict{F}{\XX_{W_{\alpha}}}$} is lisse is a cover of $S$ by constructible subsets. 
	Since the constructible topology on $S$ is quasicompact, it follows that there exists a finite subcover $\{W_{\alpha}\}_{\alpha \in A_0}$ of $\{W_{\alpha}\}_{\alpha \in A}$.
	Select a finite constructible stratification $\fromto{S}{P}$ of $S$ such that for every $p\in P$, there exists an $\alpha \in A_0$ such that $S_p\subseteq W_{\alpha}$.
	Then $F$ is $P$-constructible, as desired.
\end{proof}

\begin{lem}\label{lem:everyconstristrunccoh}
	Let $S$ be a spectral topological space, and $\fromto{\XX}{\Stilde}$ a coherent constructible $S$-stratified \topos.
	Then: 
	\begin{enumerate}[{\upshape (\ref*{lem:everyconstristrunccoh}.1)}]
		\item\label{lem:everyconstristrunccoh.1} Every constructible object of $\XX$ is truncated coherent.
	
		\item\label{lem:everyconstristrunccoh.2} If $ \XX $ is also bounded and every truncated coherent object of $\XX$ is constructible, then $\XX$ is spectral.
	\end{enumerate}
\end{lem}

\begin{proof} 
	For the first statement, let $ F\in\XX^{\Scons} $, and let $\fromto{S}{P}$ be a finite constructible stratification such that for every point $p\in P$, the restriction $\restrict{F}{\XX_p}$ is lisse. 
	By \Cref{prop:DAGXIII.2.3.22}=\cite[Proposition 2.3.22]{DAGXIII} it follows that $F$ is coherent.
	Moreover, if each $\restrict{F}{\XX_p}$ is $n$-truncated, then $F$ is $n$-truncated.

	For the second statement, if every truncated coherent object of $\XX$ is constructible and $ \XX $ is bounded, then by \enumref{lem:everyconstristrunccoh}{1} we have that $ \XX^{\Scons} = \XXcohbdd $.
	Hence by the classification of bounded coherent \topoi (\Cref{thm:SAG.A.7.5.3}=\allowbreak\SAG{Theorem}{A.7.5.3}) we see that
	\begin{equation*}
		\XX \simeq \Sheff{\XX^{\Scons}} \period
	\end{equation*}
	For every point $s\in S$, we thus have an equivalence \smash{$\XX_s\simeq\Sheff{\XX_s^{\lisse}}$}.
	This shows that $ \XX_s $ is a Stone \topos, hence $\XX$ is spectral.
\end{proof}

\begin{prp}\label{prp:everytrunccohofspecisconstr}
	If $S$ is a spectral topological space, and $\XX$ is a spectral $S$-stratified \topos $\XX$, then every truncated coherent object of $\XX$ is constructible.
\end{prp}

\begin{proof}
	Let $F$ be a truncated coherent object of $\XX$, and $s\in S$ a point.
	We wish to show that there exists a constructible subset of $ W \subset S$ containing $s$ such that $\restrict{F}{\XX_W}$ is lisse (\Cref{lem:checkconstronnbd}).
	Passing to the closure of $s$, it suffices to assume that $S$ is irreducible, and $s$ is its generic point.

	Since $\restrict{F}{\XX_s}$ is lisse, it follows from \Cref{prp:SAG.E.2.7.7}=\allowbreak\SAG{Proposition}{E.2.7.7} that there exists a full subcategory $E\subset\Spacefin$ spanned by finitely many \pifinite spaces and a unique geometric morphism \smash{$g_{\ast}\colon\fromto{\XX_s}{\Space_{/\interior E}}$} and an equivalence \smash{$\varepsilon_s\colon\equivto{\restrict{F}{\XX_s}}{g^{\ast}(I)}$}, where $I$ is the inclusion functor $\incto{E}{\Space}$.
	Now since $\Space_{/\interior E}$ is cocompact as an object of \smash{$\Topbc$} (\Cref{prp:spaceGiscocompact}) and $\XX_s$ is identified with the limit $\lim_W\XX_W$ over constructible subsets $W \subset S$ containing $s$, it follows that for some such $W$, one may factor $g_{\ast}$ through a geometric morphism $g_{W,\ast}\colon\fromto{\XX_W}{\Space_{/\interior E}}$.
	Now since $\XX_{s,<\infty}^{\coh}\simeq\colim_W\XX_{W,<\infty}^{\coh}$, we shrink $W$ as needed to ensure that there exists an equivalence $\varepsilon_s\colon\equivto{\restrict{F}{\XX_W}}{g_W^{\ast}(I)}$, and conclude that $F$ is lisse on $W$.
\end{proof}

Combining \Cref{prop:spectralcheckonpts,lem:everyconstristrunccoh,prp:everytrunccohofspecisconstr} we arrive at the following equivalent conditions for \atopos to be spectral.

\begin{cor}\label{cor:spectralclassificationconstructible}
	Let $S$ be a spectral topological space.
	The following are equivalent for a bounded coherent constructible $ S $-stratified \topos $ \fromto{\XX}{\Stilde} $: 
	\begin{enumerate}[{\upshape (\ref*{cor:spectralclassificationconstructible}.1)}]
		\item The $ S $-stratified \topos $ \XX $ is spectral.

		\item The functor
		\begin{equation*}
			\fromto{\Pt(\XX)}{\Pt(\Stilde)\simeq\mat(S)}
		\end{equation*}
		exhibits $\Pt(\XX)$ as a $\mat(S)$-stratified space.
	
		\item Every truncated coherent object of $\XX$ is $ S $-constructible.
		That is, $ \XX^{\Scons} = \XXcohbdd $.
	\end{enumerate}
\end{cor}

\begin{exm}
	If $X$ is a coherent scheme, then the truncated coherent objects of $X_{\et}$ are precisely the constructible sheaves of spaces.
	This is the nonabelian analogue of the well-known result that for a finite ring $ R $, the compact objects of the \category $ \Dup(X_{\et};R) $ of étale sheaves of $ R $-complexes on $ X $ coincide with the derived \category of constructible $R$-sheaves \cite[Proposition 2.2.6.2]{MR3887650}.
\end{exm}

We have shown that the \category $\Stratprofin$ of profinite stratified spaces is equivalent to the \category \smash{$\StrTopspec_{\infty}$} (\Cref{thm:inftyHochster}), which is in turn a full subcategory of \smash{$\StrTopbcc_{\infty}$} of bounded coherent constructible stratified \topoi.
However, the \category \smash{$\StrTopbcc_{\infty}$} is a \textit{non-full} subcategory of $\StrTop_{\infty}$.
Just as how every geometric morphism between Stone \topoi is coherent (\Cref{cor:morphismtoStoneiscoherent}=\allowbreak\SAG{Corollary}{E.3.1.2}), the subcategory
\begin{equation*}
	\StrTopspec_{\infty} \subset \StrTop_{\infty}
\end{equation*}
\textit{is} full, as we shall now explain.

\begin{prp}
	Let $f\colon \fromto{S}{T}$ be a quasicompact continuous map of spectral topological spaces, let $ \fromto{\XX}{\Stilde} $ be a coherent constructible stratified \topos, and let $ \fromto{\YY}{\Ttilde} $ be a spectral \topos.
	Then any geometric morphism $ q_{\ast} \colon \XX \to \YY$ over $ f_{\ast} \colon \fromto{\Stilde}{\Ttilde} $ is coherent.
\end{prp}

\begin{proof}
	By \Cref{cor:criterionforcoherence} it suffices to show that if $F\in\XX$ is truncated coherent, then $p^{\ast}F$ is coherent.
	By \Cref{prp:everytrunccohofspecisconstr} we have that
	\begin{equation*}
		\XX^{\Scons} = \XXcohbdd \comma
	\end{equation*}
	so the claim now follows from the facts that $ \qupperstar $ preserves constructibility (\Cref{lem:morspreserveconstrspectral}) and the $ S' $-constructible objects of $ \XX' $ are truncated coherent (\Cref{lem:everyconstristrunccoh}).
\end{proof}

\begin{cor}
	The subcategory $\StrTopspec_{\infty}\subset\StrTop_{\infty}$ is full.
\end{cor}

We finish this section by explaining the generalization of the Stone reflection (\Cref{def:Stonetopos,thm:SAG.E.2.3.2}) to stratified \topoi.

\begin{cnstr}[spectrification]\label{cnstr:spectrify}
	Let $S$ be a spectral topological space, and $\XX$ an $ S $-stratified \topos.
	By \SAG{Proposition}{A.6.4.4}, the inclusion \smash{$\incto{\XX^{\Scons}}{\XX}$} of \pretopoi extends (uniquely) to a geometric morphism \smash{$\fromto{\XX}{\Sheff{\XX^{\Scons}}}$} over $\Stilde$.
	By construction, the $ S $-stratified \topos
	\begin{equation*}
		\XX^{\Sspec} \coloneq \Sheff{\XX^{\Scons}} \index[notation]{XXSspec@$\XX^{\Sspec}$}
	\end{equation*}
	is spectral.
	Furthermore, $\XX^{\Sspec}$ is the universal spectral $S$-stratified \topos receiving a morphism of $ S $-stratified \topoi from $\XX$.
	Thus the assignment
	\begin{equation*}
		\goesto{\XX}{\XX^{\Sspec}}
	\end{equation*}
	defines a relative left adjoint to the inclusion $\incto{\StrTopspec_{\infty}}{\StrTop^{\wedge}_{\infty}}$ over $\TSpcspec$.
	We call $ \XX^{\Sspec} $ the \defn{$ S $-spectrification}\index[terminology]{spectrification}  of $ \XX $.
\end{cnstr}

\begin{exm} 
	When $ S = [n] $, the spectrification of a bounded coherent \topos $\XX$ equipped with a constructible stratification by $[n]$ can be identified as an iterated bounded coherent oriented pushout:
	\begin{equation*}
		\XX^{\nspec}\simeq\XX_0^{\Stone}\orientedcupbc^{(\XX_0\times_{\XX}\XX_1)^{\Stone}}\cdots\orientedcupbc^{(\XX_{n-1}\times_{\XX}\XX_n)^{\Stone}}\XX_n^{\Stone} \period
	\end{equation*}
\end{exm}

\begin{nul}
	Thanks to the existence of the spectrification functor, we deduce the forgetful functor $\fromto{\StrTopspec_{\infty}}{\TSpcspec}$ is a cocartesian fibration (as well as a cartesian fibration): for any quasicompact continuous map $f\colon\fromto{S'}{S}$ and any spectral $S'$-stratified \topos $\XX$, the stratified geometric morphism $\fromto{\XX}{\XX^{\Sspec}}$ is a cocartesian edge over $f$. 
\end{nul}

\begin{lem}
	Let $ S $ be a spectral topological space.
	Then the natural functor
	\begin{equation*}
		\fromto{\StrTopspec_{\infty,S}}{\lim_{P \in \FC(S)} \StrTopspec_{\infty,P}}
	\end{equation*}
	is an equivalence of \categories.
\end{lem}

\begin{proof}
	The formation of the limit is an inverse.
\end{proof}

\newpage

\section{Profinite stratified shape}\label{sec:profinstratshape}

In this chapter we investigate the inverse to the equivalence of \categories
\begin{equation*}
	\lambdahat\colon\equivto{\Stratprofin}{\StrTopspec_{\infty}}
\end{equation*}
provided by \Categorical Hoschster Duality.
This inverse equivalence provides a stratified refinement of the profinite shape (\Cref{exm:profinshapedeloc}).
We will even show that the inverse is a stratified refinement of the \textit{protruncated} shape (\Cref{thm:protruncdeloc}).

In \Cref{subsec:profinstrshapedef} we introduce this inverse, which we call the \textit{profinite stratified shape}.
To justify this language, \Cref{subsec:shapefromprofinstratshape} shows that, up to protruncation, the shape of a spectral \topos can be recovered from its profinite stratified shape by inverting all morphisms in the `pro' sense.
In \cref{subsec:pointsandmat} we show that the materialization of the profinite stratified shape of a spectral \topos $ \XX $ recovers the \category $ \Pt(\XX) $ of points of $ \XX $.
We also prove stratified refinements of the basic results relating profinite spaces and Stone \topoi discussed in \cref{subsec:profinshape}.
\Cref{subsec:stratvanKampen} provides a van Kampen Theorem that expressed the profinite shape of a spectral \topos in terms of the profinite shapes of its strata and links.


\subsection{The definition of the profinite stratified shape}\label{subsec:profinstrshapedef}

\begin{cnstr}[profinite stratified shape]
	We have shown that the functor over $ \TSpcspec $ 
	\begin{equation*}
		\lambdahat \colon \equivto{\Stratprofin}{\StrTopspec_{\infty}}
	\end{equation*}
	given by the assignment $\goesto{\mbfPi}{\widetilde{\mbfPi}}$ is an equivalence of \categories (\Cref{thm:inftyHochster}).
	The further inclusion
	\begin{equation*}
		\incto{\StrTopspec_{\infty}}{\StrTop_{\infty}}
	\end{equation*}
	admits a left adjoint, given by spectrification (\Cref{cnstr:spectrify}).
	We therefore obtain an adjunction
	\begin{equation*}
		\adjto{\StrShape}{\StrTop_{\infty}}{\Stratprofin}{\lambdahat} \period
	\end{equation*}
	The left adjoint $ \StrShape $\index[notation]{Pi@$ \StrShape $} carries a stratified \topos $\fromto{\XX}{\Stilde}$ to the profinite $S$-stratified space that as a left exact accessible functor $\fromto{\Stratfin}{\Space}$ is given by the assignment
	\begin{equation*}
		\goesto{\Pi}{\Map_{\StrTop_{\infty}}(\XX,\widetilde{\Pi})} \period
	\end{equation*}
	Over any spectral topologcial space $S$, this restricts to an adjunction
	\begin{equation*}
		\adjto{\StrShape^{S}}{\StrTop_{\infty,S}}{\StrprofinS}{\lambdahat_{S}}
	\end{equation*}
	on fibers over $S$.
\end{cnstr}

\begin{exm} 
	For any spectral topological space $S$ and any profinite $S$-stratified space $\mbfPi$, we have $ \StrShape^{S}(\widetilde{\mbfPi})\simeq\mbfPi$.
\end{exm}

\begin{exm}\label{exm:strshape0isprofinshape}
	The functor 
	\begin{equation*}
		\StrShape^{\{0\}} \colon \fromto{\Top_{\infty} \equivalent \StrTop_{\infty,\{0\}}}{\Str_{\uppi,\{0\}}^{\wedge} \equivalent \Spaceprofin}
	\end{equation*}
	is the profinite shape $ \Shapeprofin $ of \Cref{def:profinshape}.
\end{exm}

\begin{dfn}
	Let $S$ be a spectral topological space, and let $\fromto{\XX}{\Stilde}$ be an $S$-stratified \topos.
	Then we call the profinite $S$-stratified space $ \StrShape^{S}(\XX)$\index[notation]{stratified shape@$ \StrShape^{S} $} the \defn{profinite $ S $-stratified shape}\index[terminology]{profinite stratified shape@profinite $S$-stratified shape} of $\XX$.
\end{dfn}

\begin{nul}\label{nul:strdecomparisonprofiniteadjoint}
	Let $ P $ be a finite poset.
	In light of \Cref{nul:strdecomparisonprofinite}, note that the square
	\begin{equation*}
		\begin{tikzcd}
			\StrprofinP \arrow[d, "\Nerve_P" left, "\wr"{right, xshift=-0.25ex}] \arrow[r, "\lambdahat_P" above, hooked] & \StrTopbcc_{\infty,P} \arrow[d, "\NNerve_P" right, "\wr"{left, xshift=0.25ex}] \\
			\DecSpaceprofin{P} \arrow[r, "\lambdapi \circ-" below, hooked] & \DecTop{P} 
		\end{tikzcd}
	\end{equation*}
	commutes.
	Moreover, the vertical functors are equivalences and the horizontal functors are fully faithful right adjoints.
\end{nul}

\begin{nul}[pushing forward the profinite stratified shape]\label{nul:compatibilitywithinverteverything} 
	Let $\phi\colon\fromto{S'}{S}$ is a quasicompact continuous map of spectral topological spaces, and write $ \phi_{!} \colon \fromto{\Str_{\uppi,S'}^{\wedge}}{\StrprofinS} $ be the pushforward functor of \Cref{cnstr:pushforwardofprofinstrat}.
	Since left adjoints compose, there is a natural equivalence
	\begin{equation*}
		\equivto{\phi_!\StrShape^{S'}}{\StrShape^{S}} \period
	\end{equation*}
\end{nul}

\begin{exm}\label{exm:profinshapedeloc}
	As a special case of \Cref{nul:compatibilitywithinverteverything}, we see that for any $S$-stratified \topos $\XX$, the profinite shape $\Shapeprofin(\XX)$ is the classifying profinite space of the profinite \category $\StrShape^{S}(\XX)$.
	Thus the stratification on $\XX$ gives rise to a delocalization of its profinite shape.
\end{exm}

Combining \Categorical Hochster Duality (\Cref{thm:inftyHochster}) with \Cref{prp:everytrunccohofspecisconstr} we deduce the \emph{Exodromy Equivalence} stated as \Cref{mainthm:classifyconstr} in the introduction.

\begin{thm}[{\index[terminology]{Exodromy Equivalence for Stratified \Topoi}Exodromy Equivalence for Stratified \Topoi}]\label{thm:exodromyforstrattopoi}
	Let $ S $ be a spectral topological space and $ \XX $ an $ S $-stratified \topos.
	Then the unit
	\begin{equation*}
		\fromto{\XX}{\Fun(\StrShape^{S}(\XX),\Space)}
	\end{equation*}
	of the adjunction to profinite stratified spaces restricts to an equivalence
	\begin{equation*}
		\Fun(\StrShape^{S}(\XX),\Spacefin) \simeq \XX^{\Scons} \period
	\end{equation*}
\end{thm} 

\begin{nul}
	In \Cref{sect:extending} we investigate extensions of \Cref{thm:exodromyforstrattopoi}.
	In particular, we prove stable variant of \Cref{thm:exodromyforstrattopoi} for constructible sheaves with coefficients in a finite ring.
\end{nul}


\subsection{Recovering the protruncated shape from the profinite stratified shape}\label{subsec:shapefromprofinstratshape}

In \Cref{exm:profinshapedeloc} we saw how to recover the profinite shape $ \Shapeprofin(\XX) $ of a spectral stratified \topos $ \XX $ from its profinite stratified shape $ \StrShape(\XX) $ by `inverting all morphisms' in a suitable sense.
This delocalization result comes for free from the functoriality of the profinite stratified shape.
In this section prove a stronger delocalization result (\Cref{thm:protruncdeloc}): the profinite stratified shape is a delocalization of the \textit{protruncated} shape.\footnote{The contents of this section originally appeared in a short preprint by the third-named author \cite{Haine:recon}.} 

The equivalence $ \Strprofin \equivalent \StrTopspec_{\infty} $ provided by \categorical Hochster Duality (\Cref{thm:inftyHochster}) provides a way to recover the shape of a spectral \topos from its profinite stratified shape, via the composite
\begin{equation*}
	\begin{tikzcd}[sep=1.75em]
		\Strprofin \arrow[r, "\sim"{yshift=-0.25em}] & \StrTopspec_{\infty} \arrow[r] & \Topbc \arrow[r, "\Shape"] & \Pro(\Space) \comma
	\end{tikzcd}
\end{equation*}
where the middle functor functor forgets the stratification.
There is another functor $ \invert \colon \fromto{\Strprofin}{\Pro(\Space)} $ that doesn't require the use of \topoi, namely, the extension to \proobjects of the composite 
\begin{equation*}
	\begin{tikzcd}[sep=1.5em]
		\Strfin \arrow[r] & \Cat_{\infty} \arrow[r, "\invert"] & \Space \period
	\end{tikzcd}
\end{equation*}
Here the first functor forgets the stratification and the second functor sends \acategory $ C $ to the \groupoid $ \invert(C) $ obtained by inverting every morphism in $ C $ (\Cref{ntn:inverteverything}).
It follows formally that these two functors agree on $ \Strfin $:

\begin{lem}\label{lem:easyequiv}
	The square
	\begin{equation*}
		\begin{tikzcd}
			\Strfin \arrow[r, hooked, "\lambdahat"] \arrow[d, "\invert"'] & \StrTopspec_{\infty} \arrow[d, "\Shape"] \\ 
			\Space \arrow[r, hooked, "\yo"'] & \Pro(\Space) 
		\end{tikzcd}
	\end{equation*}
	commutes.
\end{lem}

\begin{proof}
	By the definition of the equivalence $ \lambdahat \colon \equivto{\Strprofin}{\StrTopspec_{\infty}} $ (\Cref{thm:inftyHochster}), the following square commutes
	\begin{equation*}
		\begin{tikzcd}
			\Strfin \arrow[r, hooked, "\lambdahat"] \arrow[d] & \StrTopspec_{\infty} \arrow[d] \\ 
			\Cat_{\infty} \arrow[r, "{\Fun(-,\Space)}"'] & \Top_{\infty} \kern0.5em \comma
		\end{tikzcd}
	\end{equation*}
	where the vertical functors forget stratifications.
	Combining this with \Cref{exm:shapeofpresheaf} proves the claim.
\end{proof}

\begin{nul}
	Since the functor $ \invert \colon \fromto{\Strprofin}{\Pro(\Space)} $ preserves inverse limits, \Cref{lem:easyequiv} provides a natural transformation 
	\begin{equation*}
		\theta \colon \fromto{\Shape \of \lambdahat}{\invert} \period
	\end{equation*}
\end{nul}

\begin{thm}[{Homotopy\index[terminology]{Homotopy Theorem}}]\label{thm:protruncdeloc}
	The natural transformation
	\begin{equation*}
		\trun_{<\infty} \theta \colon \fromto{\Shapetrun \of \lambdahat}{\trun_{<\infty}\invert}
	\end{equation*}
	of functors $ \fromto{\Strprofin}{\Pro(\Space_{<\infty})} $ is an equivalence.
\end{thm}

\begin{proof}
	Since the forgetful functor $ \fromto{\StrTopspec_{\infty}}{\Topbc} $ preserves inverse limits, \Cref{cor:protruncshapeinverselim} implies that the protruncated shape $ \Shapetrun \colon \fromto{\StrTopspec_{\infty}}{\Pro(\Space_{<\infty})} $ preserves inverse limits.
	Both $ \trun_{<\infty} $ and $ \invert $ preserve inverse limits, hence their composite
	\begin{equation*}
		\trun_{<\infty} \invert \colon \fromto{\Strprofin}{\Pro(\Space_{<\infty})}
	\end{equation*}
	preserves inverse limits.
	The claim now follows from the fact that $ \theta $ is an equivalence when restricted to $ \Strfin $ (\Cref{lem:easyequiv}) and the universal property of the \category $ \Strprofin $ of profinite stratified spaces.
\end{proof}


\subsection{Points \& materialization}\label{subsec:pointsandmat} 

We now provide a stratified refinement of \Cref{nul:matofprofinshape}.
This allows us to prove a `Whitehead Theorem' for profinite stratified spaces (\Cref{thm:profinitestratifiedWhitehead}), and effectively speak of \textit{$ n $-truncated} profinite stratified spaces via materialization.
In particular, we show that a spectral \topos is $ n $-localic if and only if its \category of points is an $ n $-category (\Cref{prop:nlocalicntruncated}).

\begin{nul}
	Let $S$ be a spectral topological space, and let $\XX$ be an $S$-stratified \topos.
	The \category of points of $ \XX $ is given by
	\begin{equation*}
		\Pt(\XX) = \Funlowerstar(\Space,\XX)^{\op} \equivalent \Fun_{\StrTop_{\infty},\ast}(\widetilde{\{0\}},\XX)^{\op} \period
	\end{equation*}
	Since $ \StrShape(\widetilde{\{0\}}) \equivalent \ast $, applying $ \StrShape $ yields a natural functor
	\begin{equation*}
		\fromto{\Pt(\XX)}{\Fun_{\Strprofin}(\ast,\StrShape(\XX)) \equivalent \mat\StrShape(\XX)} \period
	\end{equation*}

	In the case where $ \XX $ is a spectral \topos, then \Categorical Hochster Duality (\Cref{thm:inftyHochster}) implies the following stratified refinement of \Cref{nul:matofprofinshape}. 
\end{nul}

\begin{lem}\label{lem:matofPiinfisPt}
	If $ \XX $ is a spectral \topos, then the natural morphism
	\begin{equation*}
		 \fromto{\Pt(\XX)}{\mat \StrShape(\XX)}
	\end{equation*}
	of stratified spaces is an equivalence.
\end{lem}

Now we can deduce a stratified refinement of Whitehead's Theorem for profinite spaces (\Cref{thm:profiniteWhitehead}=\allowbreak\SAG{Theorem}{E.3.1.6}).

\begin{thm}[{Whitehead Theorem for profinite stratified spaces\index[terminology]{Whitehead Theorem!for profinite stratified spaces}}]\label{thm:profinitestratifiedWhitehead}
	The materialization functor $ \mat \colon \fromto{\Strprofin}{\Str} $ is conservative.
\end{thm}

\begin{proof}
	Let $ f \colon \fromto{\mbfPi}{\mbfPi'} $ be a morphism in $ \Strprofin $ and assume that $ \mat(f) $ is an equivalence in $ \Str $.
	Write $ \flowerstar \colon \fromto{\widetilde{\mbfPi}}{\widetilde{\mbfPi}'} $ for the induced morphism of spectral \topoi.
	From \Cref{lem:matofPiinfisPt} we deduce that
	\begin{equation*}
		\Pt(\flowerstar) \colon \fromto{\Pt(\widetilde{\mbfPi})}{\Pt(\widetilde{\mbfPi}')}
	\end{equation*}
	is an equivalence of \categories.
	Conceptual Completeness (\Cref{thm:conceptualcompleteness}=\allowbreak\SAG{Theorem}{A.9.0.6}) implies that $ \flowerstar $ is an equivalence of \topoi.
	The full faithfulness of the functor $\goesto{\mbfPi}{\widetilde{\mbfPi}}$ completes the proof.
\end{proof}

We can employ the Whitehead Theorem for profinite stratified spaces to study the Postnikov tower of profinite stratified spaces. 

\begin{dfn}
	Let $ n \in \NNup $.
	A profinite stratified space $ \fromto{\mbfPi}{S} $ is \defn{$ n $-truncated}\index[terminology]{n-truncated@$n$-truncated!profinite stratified space}\index[terminology]{profinite stratified space!n-truncated@$n$-truncated} if and only if $ \mbfPi $ can be exhibited as an inverse limit of finite $n$-truncated \pifinite stratified spaces.
	Equivalently, if we extend $\ho_n\colon\fromto{\Strfin}{\Strfin}$ to an inverse-limit preserving functor $\ho_n\colon\fromto{\Strprofin}{\Strprofin}$, then an $n$-truncated profinite space is one in the essential image of $\ho_n$.

	We write $(\Strprofin)_{\leq n}\subset\Strprofin$ for the full subcategory spanned by the $n$-truncated profinite stratified spaces.
\end{dfn}

\begin{lem}\label{lem:ntruncatedNervemat}
	Let $ n \in \NNup $, and let $S$ be a spectral topological space.
	Then a profinite stratified space $ \fromto{\mbfPi}{S} $ is $ n $-truncated if and only if, for all $ s, t \in \mat(S) $ such that $ s \leq t $, the induced morphism
	\begin{equation*}
		\fromto{\Nerve_{\mat(S)}(\mbfPi)\{s \leq t\}}{\mbfPi_{s} \cross \mbfPi_{t}}
	\end{equation*} 
	is an $ (n-1) $-truncated morphism of $ \Spaceprofin $.
\end{lem}

\begin{proof} 
	Exhibit $\mbfPi$ as an inverse system $\{\Pi_{\alpha}\to P_{\alpha}\}_{\alpha \in A}$ of $n$-truncated \pifinite stratified spaces, and express $s$ and $t$ as inverse systems\pifinite $\{s_{\alpha}\}_{\alpha \in A}$ and $\{t_{\alpha}\}_{\alpha \in A}$ of points.
	Then the inverse system
	\begin{equation*}
		\left\{ \Nerve_{P_{\alpha}}(\Pi_{\alpha})\{s_{\alpha} \leq t_{\alpha}\}\to\Pi_{s_{\alpha}} \cross \Pi_{t_{\alpha}} \right\}_{\alpha \in A} \comma
	\end{equation*}
	exhibiting the morphism $\fromto{\Nerve_{\mat(S)}(\mbfPi)\{s \leq t\}}{\mbfPi_{s} \cross \mbfPi_{t}}$ of $ \Spaceprofin $, is an inverse system of $(n-1)$-truncated morphisms.

	Conversely, assume that $\mbfPi$ is exhibited as an inverse system $\{\Pi_{\alpha}\to P_{\alpha}\}_{\alpha \in A}$ of \pifinite stratified spaces, and that for all $s,t\in\mat(S)$ such that $s\leq t$, the morphism of profinite spaces
	\begin{equation*}
		\fromto{\Nerve_{\mat(S)}(\mbfPi)\{s \leq t\}}{\mbfPi_{s} \cross \mbfPi_{t}}
	\end{equation*}
	is $(n-1)$-truncated.
	Now consider $\ho_n\mbfPi\coloneq\{\ho_n\Pi_{\alpha}\to P_{\alpha}\}_{\alpha \in A}$ and the natural morphism $\fromto{\mbfPi}{\ho_n\mbfPi}$.
	To see that this morphism is an equivalence, we may pass to the materialization by \Cref{thm:profinitestratifiedWhitehead}, where it is obvious.
\end{proof}

\begin{lem}\label{lem:matdetectstruncatedness}
	Let $ n \in \NNup $.
	A profinite stratified space $ \fromto{\mbfPi}{S} $ is $ n $-truncated if and only if $ \mat(\mbfPi) \in \Strat $ is $ n $-truncated in the sense of \Cref{def:truncatednessforstratspaces}.
\end{lem}

\begin{proof}
	For all $ s,t \in \mat(S) $ such that $ s \leq t $, we have a natural identification
	\begin{equation*}
		\mat(\Nerve_{\mat(S)}(\mbfPi)\{s \leq t\}) \equivalent \Nerve_{\mat(S)}(\mat(\mbfPi))\{s \leq t\} \period
	\end{equation*}
	By \Cref{prop:matandtruncatedness}=\allowbreak\SAG{Proposition}{E.4.6.1}, the fact that materialization is a right adjoint, and \Cref{lem:ntruncatedNervemat}, we see that a profinite stratified space $ \mbfPi $ is $ n $-truncated if and only if the morphism 
	\begin{equation*}
		\fromto{\Nerve_{\mat(S)}(\mat(\mbfPi))\{s \leq t\}}{\mat(\mbfPi)_{s} \cross \mat(\mbfPi)_{t}}
	\end{equation*} 
	is an $ (n-1) $-truncated morphism of spaces.
	This is true if and only if $ \mat(\mbfPi) $ is $ n $-truncated in the sense of \Cref{def:truncatednessforstratspaces}.
\end{proof}

Under \Categorical Hochster Duality (\Cref{thm:inftyHochster}) $ n $-localic spectral stratified \topoi correspond to $ n $-truncated profinite stratified spaces:

\begin{prp}\label{prop:nlocalicntruncated}
	Let $ \XX $ be a spectral \topos and $ n \in \NNup $.
	Then the following are equivalent:
	\begin{enumerate}[{\upshape (\ref*{prop:nlocalicntruncated}.1)}]
		\item\label{prop:nlocalicntruncated.1} The \topos $ \XX $ is $ n $-localic.

		\item\label{prop:nlocalicntruncated.2} The \category $ \Pt(\XX) $ of points of $ \XX $ is an $ n $-category.

		\item\label{prop:nlocalicntruncated.3} The profinite stratified shape $ \StrShape(\XX) $ is an $ n $-truncated profinite stratified space.
	\end{enumerate}
\end{prp}

\begin{proof}
	First we show that \enumref{prop:nlocalicntruncated}{1}$ \Rightarrow $\enumref{prop:nlocalicntruncated}{2}$ \Rightarrow $\enumref{prop:nlocalicntruncated}{3}.
	If $ \XX $ is $ n $-localic, then the \category $ \Pt(\XX) $ is an $ n $-category, which shows that  
	\begin{equation*}
		\mat \StrShape(\XX) \equivalent \Pt(\XX)
	\end{equation*}
	is an $ n $-category (\Cref{lem:matofPiinfisPt}).
	Applying \Cref{lem:matdetectstruncatedness} we see that $ \StrShape(\XX) $ is an $ n $-truncated profinite stratified space.

	Now we show that \enumref{prop:nlocalicntruncated}{3}$ \Rightarrow $\enumref{prop:nlocalicntruncated}{1}.
	If $ \StrShape(\XX) $ is an $ n $-truncated profinite stratified space, then $ \StrShape(\XX) $ can be exhibited as an inverse system $ \{\Pi_{\alpha}\}_{\alpha \in A} $ of $ n $-truncated \pifinite stratified spaces.
	Since $ n $-localic \topoi are closed under limits in $ \Top_{\infty} $ and each $ \Pitilde_{\alpha} $ is $ n $-localic (\Cref{cor:Pitildeisbcc}), we see that the \topos
	\begin{equation*}
		\XX \equivalent \lim_{\alpha \in A} \widetilde{\Pi}_{\alpha}
	\end{equation*}
	is $ n $-localic.
\end{proof}

\begin{nul}\label{nul:1truncatedprofinstrat} 
	Combining the preceding with ordinary Stone Duality between profinite sets and Stone topological spaces, we see that the functor
	\begin{equation*}
		\Pt \colon \fromto{\StrTop_{\infty}^{\spec,1}}{\Cat_{1}}
	\end{equation*}
	factors through a fully faithful functor
	\begin{equation*}
		\incto{\StrTop_{\infty}^{\spec,1}}{\CS(\TSpc^{\Stone})}
	\end{equation*}
	from the $2$-category of $1$-truncated profinite stratified spaces to the $2$-category of categories in the category of Stone topological spaces (cf. \Cref{def:COCSS}).
	The essential image of this functor is spanned by the \defn{layered} categories -- i.e., the ones in which every endomorphism is an isomorphism.
\end{nul}


\subsection{Stratified homotopy types via décollages}\label{subsec:strathomotopyviadecollage} 

Let $ P $ be a finite poset and $ \XX $ a spectral $ P $-stratified \topos.
In this section we give an explicit description of the profinite stratified shape $ \StrShape^P(\XX) $ as a décollage involving the profinite shapes of the strat and links of $ \XX $.

\begin{cnstr}[shapes of décollages]
	Let $P$ be a finite poset.
	Recall that the functor
	\begin{equation}
		\lambdahat_P^{\dec} \colonequals \lambdapi \of - \colon \incto{\DecSpaceprofin{P}}{\DecTop{P}}
	\end{equation}
	given by composition with $\uplambda_{\{0\}}$ is fully faithful with essential image is $\DecTopStone{P}$ (\Cref{prp:decprehochster}).

	We construct a left adjoint to $ \lambdahat_P^{\dec} $ as follows.
	Write
	\begin{equation*}
		\Shapeprofin^{\predec,P}\colon\fromto{\DecTop{P}}{\Fun(\sdop(P),\Spaceprofin)} \index[notation]{Pi@$ \Shapeprofin^{\predec,P}, \Shapeprofin^{\dec,P} $}
	\end{equation*}
	for the functor given by composition with the profinite shape $\Shapeprofin$.
	Hence $ \Shapeprofin^{\predec,P} $ sends a toposic décollage $\DD\colon\fromto{\sdop(P)}{\Top_{\infty}^{\bc}}$ to the functor $\goesto{\Sigma}{\Shapeprofin\DD(\Sigma)}$. 
	Composing $ \Shapeprofin^{\predec,P} $ with the left adjoint to the inclusion $\incto{\DecSpaceprofin{P}}{\Fun(\sdop(P),\Spaceprofin)}$, we obtain a functor
	\begin{equation*}
		\Shapeprofin^{\dec,P}\colon\fromto{\DecTop{P}}{\DecSpaceprofin{P}} \index[notation]{Pi@$ \Shapeprofin^{\dec,P} $}
	\end{equation*}
	that is left adjoint to $\lambdahat_P^{\dec}$.
\end{cnstr}

Because it involves applying the left adjoint $ \fromto{\Fun(\sdop(P),\Spaceprofin)}{\DecSpaceprofin{P}} $, the functor $\Shapeprofin^{\dec,P}$ is very inexplicit in general.
However, there are two cases in which applying this left adjoint is superfluous: if the poset has rank $ \leq 1 $, or if the décollage comes from a spectral \topos.

\begin{thm}\label{thm:decollageofspectralisSegal}
	Let $P$ be a finite poset and let $\fromto{\XX}{\Ptilde}$ be a bounded coherent constructible $P$-stratified \topos.
	Assume that either of the following conditions holds:
	\begin{enumerate}[{\upshape (\ref*{thm:decollageofspectralisSegal}.1)}]
		\item\label{thm:decollageofspectralisSegal.1} The poset $ P $ has rank $ \leq 1 $.

		\item\label{thm:decollageofspectralisSegal.2} The $ P $-stratified \topos $ \XX $ is spectral.
	\end{enumerate}
	Then the functor $\goesto{\Sigma}{\Shapeprofin\NNerve_P(\XX)(\Sigma)}$ is already a profinite spatial décollage.
	That is, the natural \emph{Segalification} morphism
	\begin{equation*}
		\Shapeprofin^{\predec,P}\NNerve_P(\XX) \to \Shapeprofin^{\dec,P}\NNerve_P(\XX)
	\end{equation*}
	is an equivalence in $ \Fun(\sdop(P),\Spaceprofin) $.
\end{thm}

\begin{proof}
	The Segal condition is vacuous for posets of rank $ \leq 1 $, so \enumref{thm:decollageofspectralisSegal}{1} is clear.

	For \enumref{thm:decollageofspectralisSegal}{2}, it suffices to prove that for every chain $ \Sigma \coloneq \{p_0 < \cdots < p_n\} \subset P $, the natural morphism of profinite spaces
	\begin{equation*}
		f_{\Sigma} \colon \fromto{\Shapeprofin(\XX_{p_0} \orientedtimes_{\XX} \cdots \orientedtimes_{\XX} \XX_{p_n}) = \Shapeprofin\Mor_{\Ptilde}(\widetilde{\Sigma},\XX)}{\Map_{P}(\Sigma,\StrShape(\XX))}
	\end{equation*}
	is an equivalence.
	By Whitehead's Theorem for profinite spaces (\Cref{thm:profiniteWhitehead}=\allowbreak\SAG{Theorem}{E.3.1.6}), it suffices to prove that the materialization $ \mat(f_{\Sigma}) $ is an equivalence.
	Since $ \XX $ is spectral, we have a natural equivalence 
	\begin{equation*}
		\mat \Shapeprofin(\XX_{p_0} \orientedtimes_{\XX} \cdots \orientedtimes_{\XX} \XX_{p_n}) \equivalent \Pt(\XX_{p_0} \orientedtimes_{\XX} \cdots \orientedtimes_{\XX} \XX_{p_n}) 
	\end{equation*}
	(\Cref{nul:matofprofinshape,prp:orientedprodisStone}).
	Similarly, since $ \Sigma $ is a constant \proobject and $ \XX $ is spectral, by Whitehead's Theorem for profinite stratified spaces (\Cref{thm:profinitestratifiedWhitehead}) we have natural equivalences
	\begin{align*}
		\mat\Map_{P}(\Sigma,\StrShape(\XX)) &\equivalent \Map_{P}(\Sigma,\mat\StrShape(\XX)) \\ 
		&\equivalent \Map_{P}(\Sigma,\Pt(\XX)) \period
	\end{align*}
	By the universal property of the iterated oriented fiber product $ \XX_{p_0} \orientedtimes_{\XX} \cdots \orientedtimes_{\XX} \XX_{p_n} $, we have a natural identification 
	\begin{equation}\label{eq:pointsoflinks}
		\equivto{\Map_{P}(\Sigma,\Pt(\XX))}{\Pt(\XX_{p_0} \orientedtimes_{\XX} \cdots \orientedtimes_{\XX} \XX_{p_n})} \period
	\end{equation}
	To complete the proof, note that the materialization $ \mat(f_{\Sigma}) $ is equivalent to the morphism \eqref{eq:pointsoflinks}.
\end{proof}

\begin{exm}
	Let $P$ be a finite poset, and let $\fromto{\XX}{\Ptilde}$ be a spectral $P$-stratified \topos. It follows from \Cref{thm:decollageofspectralisSegal} that, for any point $p\in P$, the $p$-th stratum $\StrShape^{P}(\XX)_p$ is equivalent to the profinite shape $\Shapeprofin(\XX_p)$ of $ \XX_p $.
\end{exm}

\begin{exm}
	Let $P$ be a finite poset, and let $\fromto{\XX}{\Ptilde}$ be a spectral $P$-stratified \topos. It follows from \Cref{thm:decollageofspectralisSegal} that, for any points $p,q\in P$ with $p<q$, the link $\Map_{P}(\{p < q\},\StrShape^{P}(\XX))$ between the $p$-th and $q$-th strata of $\StrShape^{P}(\XX)$ is equivalent to the profinite shape type $\Shapeprofin(\XX_p\orientedtimes_{\XX}\XX_q)$ of the topos-theoretic link.
\end{exm}

\begin{exm} 
	Let $P$ be a finite poset, and $\XX$ a spectral $P$-stratified \topos.
	For any points $p,q\in P$ such that $p < q$, write
	\begin{equation*}
		i_{pq,\ast}\colon\incto{\XX_{p}}{\XX_{\{p < q\}}} \quad \text{and} \quad j_{pq,\ast}\colon\incto{\XX_{q}}{\XX_{\{p < q\}}}
	\end{equation*}
	for the closed and open immersions of strata, respectively.
	Then the \basechange Theorem for oriented fiber products (\Cref{thm:BCfororientedfibs}) ensures that the décollage
	\begin{equation*}
		\Shapeprofin^{\dec,P}(\XX)\colon\fromto{\sdop(P)}{\Spaceprofin}
	\end{equation*}
	carries any chain $\{p_0 < \cdots < p_n\}\subseteq P$ to the left exact accessible functor $\Spacefin\to\Space$ given by the composite
	\begin{equation*}
		\Gammaup_{\XX_{p_0},\ast} i_{p_0p_1}^{\ast}j_{p_0p_1,\ast}i_{p_1p_2}^{\ast}j_{p_1p_2,\ast}\cdots i_{p_{n-1}p_n}^{\ast}j_{p_{n-1}p_n,\ast}\Gammaup_{\XX_{p_n}}^{\ast} \period
	\end{equation*}
\end{exm}


\subsection{The van Kampen Theorem}\label{subsec:stratvanKampen} 

Let $ P $ be a poset.
The `invert everything' functor $\goesto{\Pi}{\invert(\Pi)}$ from $P$-stratified spaces to spaces, regarded as a functor from $ \fromto{\DecSpace{P}}{\Space} $, is given by the formation of the colimit.
That is, if $\fromto{\Pi}{P}$ is a $P$-stratified space, then there is a natural equivalence
\begin{equation*}
	\invert(\Pi) \simeq \colim_{\Sigma \in \sdop(P)} \Nerve_P(\Pi)(\Sigma) \period
\end{equation*}
There is a variant of this formula for \textit{profinite} stratified spaces.
Let $ \invert\profincomp \colon \fromto{\Pro(\Cat_{\infty})}{\Spaceprofin} $ denote the composite of the `invert everything functor' $ \invert \colon \fromto{\Pro(\Cat_{\infty})}{\Pro(\Space)} $ with profinite completion $ (-)\profincomp \colon \fromto{\Pro(\Space)}{\Spaceprofin} $.
For any profinite $ P $-stratified space $ \mbfPi $, there is a natural equivalence
\begin{equation*}
	\invert\profincomp(\mbfPi) \simeq \colim_{\Sigma \in \sdop(P)} \Nerve_P(\mbfPi)(\Sigma) \period
\end{equation*}
The compatibility \Cref{nul:compatibilitywithinverteverything} combined with \Cref{exm:strshape0isprofinshape,nul:strdecomparisonprofiniteadjoint,thm:decollageofspectralisSegal} provide an analogous colimit formula for the profinite shape type of a spectral \topos:

\begin{prp}[{van Kampen Theorem\index[terminology]{van Kampen Theorem}}]\label{prop:vanKampen}
	Let $P$ be a finite poset and let $\fromto{\XX}{\Ptilde}$ be a bounded coherent constructible $P$-stratified \topos.
	Assume that either of the following conditions holds:
	\begin{enumerate}[{\upshape (\ref*{thm:decollageofspectralisSegal}.1)}]
		\item\label{prop:vanKampen.1} The poset $ P $ has rank $ \leq 1 $.

		\item\label{prop:vanKampen.2} The $ P $-stratified \topos $ \XX $ is spectral.
	\end{enumerate}
	Then the profinite shape of $\XX$ is given by the colimit
	\begin{equation*}
		\Shapeprofin(\XX) \simeq \colim_{\Sigma\in\sdop(P)} \Shapeprofin(\NNerve_P(\XX)(\Sigma))
	\end{equation*}
	in the \category of profinite spaces.
\end{prp}

\begin{exm}\label{exm:vanKampen1strat}
	If $\XX$ is a bounded coherent constructible $ [1] $-stratified \topos exhibited as a recollement \smash{$\ZZ\orientedcup^{\phi}\UU$} of bounded coherent \topoi $\ZZ$ and $\UU$, then the induced square
	\begin{equation*}
      \begin{tikzcd}
	       \Shapeprofin(\ZZ\orientedtimes_{\XX}\UU) \arrow[d] \arrow[r] & \Shapeprofin(\UU) \arrow[d] \\ 
	       \Shapeprofin(\ZZ) \arrow[r] & \Shapeprofin(\XX)
      \end{tikzcd}
    \end{equation*}
	is a pushout square in the \category of profinite spaces.
\end{exm}

\begin{exm}
	Let $n\in\NNup$, and let $\fromto{\XX}{\widetilde{[n]}}$ be a spectral $ [n] $-stratified \topos.
	Then $\Shapeprofin(\XX)$ can be exhibited as the colimit of a punctured $(n+1)$-cube $ \fromto{\sdop([n])}{\Spaceprofin}$ given by the assignment
	\begin{equation*}
		\goesto{\{p_0 < \cdots < p_k\}}{\Shapeprofin(\XX_{p_0}\orientedtimes_{\XX}\XX_{p_1}\orientedtimes_{\XX}\cdots\orientedtimes_{\XX}\XX_{p_k})} \period
	\end{equation*}
\end{exm}


\newpage

\part{Stratified étale homotopy theory}\label{part:stratetale}

In this part we use the profinite stratified shape of \cref{sec:profinstratshape} to give a refinement of the étale homotopy theory of Artin--Mazur--Friedlander.
We first recall how to define the étale homotopy type from the \categorical perspective, as well as the main theorems in étale homotopy theory (\cref{sec:recetalehtpytype}).
We then study the profinite stratified shape of the étale \topos of coherent schemes in detail (\cref{sect:ramification}).
In particular, we provide a concrete description in terms the \textit{profinite Galois categories} introduced in the Introduction (preceding \Cref{lede:reconstruction}).
We conclude with \cref{sect:fullfaithfulness} where we discuss Grothendieck's anabelian conjectures and use a theorem of Voevodsky to prove a strong reconstruction theorem for schemes in characteristic $ 0 $ in terms of profinite Galois categories (\Cref{lede:reconstruction}=\Cref{nul:mainthmonGal}).


\section{Aide-mémoire on étale homotopy types}\label{sec:recetalehtpytype}

In this chapter we recall how to situate the \textit{étale homotopy type} of Artin--Mazur--Fried\-lander in the \categorical setting.
We also provide some example computations of étale homotopy types.

\Cref{subsec:etalehomotopy} recalls the definition of the étale homotopy type via shape theory.
In \Cref{subsec:etalehomotopyexamples} we give some sample computations and uses of the étale homtopy type. 
\Cref{subsec:etalemonodromy} recalls the monodromy equivalence for lisse étale sheaves in terms of the profinite étale homotopy type.
\Cref{subsec:etalehomotopyofsschemes} explains how étale homotopy theory works for simplicial schemes.
\Cref{subsec:ordinaryRiemann} recalls Artin and Mazur's formulation of the Riemann Existence Theorem in terms of the étale homotopy type, and \cref{subsec:etalevanKampen} gives a quick proof of the étale van Kampen Theorem.


\subsection{Artin \& Mazur's étale homotopy types of schemes}\label{subsec:etalehomotopy}

From \acategorical perspective, there are \textit{a priori} four étale shapes to contemplate:

\begin{dfn}
	Let $X$ be a scheme.
	We write:
	\begin{itemize}
		\item $\Shapeet(X) \coloneq \Shape(X_{\et})$\index[notation]{Piet@$ \Shapeet, \Shapeethyp, \Shapeetprofin, \Shapeethypprofin $} for the shape of the $1$-localic étale \topos of $ X $,

		\item $\Shapeethyp(X) \coloneq \Shape(X_{\et}^{\hyp})$ for the shape of the hypercomplete étale \topos of $ X $,

		\item $\Shapeetprofin(X) \coloneq \Shapeprofin(X_{\et})$ for the profinite shape of the $1$-localic \topos of $ X $, and

		\item $\Shapeethypprofin(X) \coloneq \Shapeprofin(X_{\et}^{\hyp})$ for the profinite shape of the hypercomplete étale \topos of $ X $.
	\end{itemize}
\end{dfn}

\begin{nul}\label{nul:equivofetalehomotopyonprotrunc}
	As a special case of \Cref{exm:protruncandhypercomp}, we see that the natural morphism
	\begin{equation*}
		\fromto{\Shapeethyp(X)}{\Shapeet(X)}
	\end{equation*}
	becomes an equivalence after protruncation.
	In particular, the natural morphism
	\begin{equation*}
		\fromto{\Shapeethypprofin(X)}{\Shapeetprofin(X)}
	\end{equation*}
	is an equivalence.
	We simply write 
	\begin{equation*}
		\Shapeettrun(X) \coloneq \trun_{<\infty}\Shapeethyp(X) \equivalent \trun_{<\infty}\Shapeet(X)
	\end{equation*}
	for the protruncated shape of the étale \topos.
\end{nul}

For a locally noetherian scheme $ X $, Artin and Mazur \cite[\S 9]{MR0245577} constructed a \proobject in the homotopy category of spaces called the \textit{étale homotopy type} of $ X $.
Friedlander \cite[\S 4]{MR676809} later refined this construction, producing a \proobject in simplicial sets called the \textit{étale topological type} of $ X $.
The image of the étale topological type in $ \Pro(\hoSpace) $ agrees with the étale homotopy type of Artin--Mazur \cite[Proposition 4.5]{MR676809}.
Hoyois \cite[\S 5]{MR3763287} has shown that Friedlander's étale topological type corepresents the shape of the hypercomplete étale \topos of $ X $:

\begin{thm}[{\cite[Corollary 5.6]{MR3763287}}]\label{thm:Hoyoisetalehomotopycomparison}
	Let $ X $ be a locally noetherian scheme.
	Then the étale topological type of $ X $ corepresents the shape \smash{$\Shapeethyp(X)$} of the hypercomplete étale \topos of $ X $.
\end{thm}

\begin{nul}
	We refer to the shape $ \Shapeet(X) $ of the étale \topos as the \defn{étale shape}\index[terminology]{etale shape@étale shape}\index[terminology]{shape!etale@étale}, call the protruncated shape $ \Shapeettrun(X) $ the \defn{protruncated étale shape}\index[terminology]{etale shape@étale shape!protruncated}\index[terminology]{shape!protruncated étale}\index[terminology]{protruncated!etale shape@étale shape}, and call the profinite shape $ \Shapeetprofin(X) $ the \defn{profinite étale shape}\index[terminology]{etale shape@étale shape!profinite}\index[terminology]{shape!profinite étale}\index[terminology]{profinite!etale shape@étale shape}.
\end{nul}

In many examples, the protruncated étale shape is already profinite:

\begin{thm}[{\cites[Theorem 3.6.5]{DAGXIII}[Theorem 11.1]{MR0245577}[Theorem 7.3]{MR676809}}]\label{thm:profinitenessofetalehomotopy}
	Let $ X $ be a connected noetherian scheme that is geometrically unibranch.
	Then the protruncated étale shape of $ X $ is profinite; that is, the natural morphism
	\begin{equation*}
		\fromto{\Shapeettrun(X)}{\Shapeetprofin(X)} 
	\end{equation*}
	is an equivalence.
\end{thm}

\begin{qst} 
	Let $ X $ be a connected noetherian scheme that is geometrically unibranch. 
	Even in simple cases, we do not at this point have a very good understanding of the kind of information that is contained in the étale shape $\Shapeet(X)$ but not in the other invariants.
	In this paper, we are content to focus our attention on the profinite shapes types and their stratified variants.
\end{qst}

\begin{nul}\label{nul:etalehomotopygroups}
	Let $ X $ be a scheme and $ \fromto{x}{X} $ a geometric point of $ X $.
	Then $ x $ induces a point of the prospace $ \Shapeet(X) $.
	The \textit{$ n $-th extended étale homotopy progroup}\index[terminology]{extended étale homotopy progroup} of $ X $ is the progroup
	\begin{equation*}
		\pietext_n(X,x) \coloneq \uppi_n(\Shapeet(X), x) \index[notation]{pi@$ \pietext_n $} \period
	\end{equation*} 
	In particular, the progroup $ \pietext_1(X,x) $ is the \textit{groupe fondamentale élargi} of \cite[Exposé X, \S 6]{MR43:223b}; see \cite[Corollary 10.7]{MR0245577}.
	The usual étale fundamenal group of \cite[Exposé V, \S 7]{MR50:7129} is the profinite completion of $ \pietext_1(X,x) $.
	Moreover, the usual étale fundamental group of $ X $ is isomorphic to the profinite fundamental group \smash{$ \uppi_1(\Shapeetprofin(X),x) $} \cite[Corollary 3.9]{MR0245577}.
	We denote the usual (profinitely complete) étale fundamental group by $ \piet_1(X,x) $.
\end{nul}


\subsection{Examples}\label{subsec:etalehomotopyexamples}

In this section we provide some example computations of étale shapes.

\begin{exm}[fields] 
	Let $ k $ be a field, and $ \ksep \supset k $ a separable closure of $ k $.
	Recall that we write $ \Gk $ for the absolute Galois group of $ k $ (\Cref{ntn:absGalois}).
	The choice of separable closure of $ k $ provides an identification
	\begin{equation*}
		\Shapeetprofin(\Spec k) \equivalent \BGk \period
	\end{equation*}
	(See \Cref{subsec:examplesofcoherent}.)
\end{exm}

\begin{exm} 
	Since $ \Spec \ZZup $ has no unramified étale covers, the étale fundamental group of the $ \Spec \ZZup $ is trivial.
	Moreover, for all integers $ i \geq 1 $ and $ n \geq 2 $, the étale cohomology group $ \Hup_{\et}^i(\Spec \ZZup; \ZZup/n) $ is trivial (see \cites{MR2261462}{MO:3103}). 
	The Universal Coefficient Theorem and Hurewicz Theorem imply that the profinite étale shape \smash{$ \Shapeetprofin(\Spec \ZZup) $} of $ \Spec \ZZup $ is trivial (cf. \cite[\S4]{MR0245577}).
	Since $ \ZZup $ is a noetherian domain, \Cref{thm:profinitenessofetalehomotopy} applies, hence the protruncated étale shape $ \Shapeettrun(\Spec \ZZup) $ of $ \Spec \ZZup $ is trivial.
\end{exm}

\begin{exm}
	Let $ k $ be an algebraically closed field of characteristic $ 0 $ and
	\begin{equation*}
		C = \Spec(k[x,y]/(y^2 - x^3 - x^2))
	\end{equation*}
	the nodal cubic.
	Then there is a noncanoical identification $ \Shapeettrun(C) \equivalent \BZZ $.

	Since the group $ \ZZup $ is \textit{good} in the sense of Serre \cite[p. 16]{MR1867431}, the profinite étale shape is given by $ \Shapeetprofin(C) \equivalent \BZZhat $ \cites[\SAGthm{Warning}{E.4.3.4}]{SAG}[Theorem 3.14]{MR3051250}.
\end{exm}

\begin{exm}[curves]
	Let $ C $ be a smooth connected curve over a field.
	If $ C $ is affine or of positve genus, then the protruncated étale homotopy type $ \Shapeettrun(X) $ is $ 1 $-truncated \cites[Proposition 15]{MR1383466}[Lemma 2.7(a)]{MR3549624}.
	Thus we have a noncanonical identification
	\begin{equation*}
		\Shapeettrun(C) \equivalent \Bup\piet_1(C) \period
	\end{equation*}
\end{exm}

\begin{exm}[see \Cref{thm:Riemannexistence}]
	Let $ \Sup^2 \in \Space $ denote the $ 2 $-sphere.
	There an equivalence $ \Shapeetprofin(\PPup_{\CCup}^1) \equivalent (\Sup^2)\profincomp $. 
\end{exm}

\begin{exm}[{\cite[Theorem 1]{MR3248993}}]
	Let $ k $ be an algebraically closed field of positive characteristic and let $ X $ be a smooth $ k $-variety.
	Then $ \Shapeettrun(X) \equivalent \ast $ if and only if $ X $ is \textit{isomorphic} to $ \Spec k $.
\end{exm}

\begin{exm}[\Cref{exm:properbasechangeforschemes}]
	Let $ k $ be a separably closed field, and let $ X $ and $ Y $ be coherent $ k $-schemes.
	If $ Y $ is proper, then the natural morphism of profinite spaces
	\begin{equation*}
		\fromto{\Shapeetprofin(X \cross_{\Spec k} Y)}{\Shapeetprofin(X) \cross \Shapeetprofin(Y)}
	\end{equation*}
	is an equivalence.
\end{exm}

The following two examples are from Achinger's remarkable work on $ \Kup(\pi,1) $-schemes in positive characteristic \cite{MR3714509}.

\begin{exm}[{affine \smash{$ \FFp $}-schemes}]\label{exm:AchingerKpi1}
	Let $ p $ be a prime number.
	Achinger showed that if $ X $ is a connected affine \smash{$ \FFp $}-scheme, then the profinite étale homotopy type \smash{$ \Shapeetprofin(X) $} is $ 1 $-truncated \cite[Theorem 1.1]{MR3714509}.
	This is in stark contrast with the case of schemes in characteristic zero. 
\end{exm}

\begin{exm}\label{exm:AchingerAAn}
	Let $ k $ be an algebraically closed field of characteristic $ p > 0 $.
	By Raynaud's proof of Abhyankar's Conjecture \cite{MR1253200}, a finite group $ G $ arises as a quotient of the profinite group $ \pietext_1(\AAup_k^1) $ if and only if $ G $ is a \textit{quasi-$ p $-group} (i.e., $ G $ has no nontrivial quotient of order prime to $ p $). 
	More generally, it follows from Raynaud's work that for $ n \geq 1 $, a finite group $ G $ arises as a quotient of $ \pietext_1(\AAup_k^n) $ if and only if $ G $ is a quasi-$ p $-group.
	Even though the étale fundamental groups of $ \AAup_k^1 $ and $ \AAup_k^n $ abstractly have the same finite quotients, Achinger showed that for positive integers $ m \neq n $, the étale fundamental groups $ \pietext_1(\AAup_k^m) $ and $ \pietext_1(\AAup_k^n) $ are \textit{not} isomorphic as profinite groups \cite[Proposition 7.6]{MR3714509}.
\end{exm}

\Cref{exm:AchingerAAn} demonstrates how `large' étale fundamental groups of $ \FFp $-schemes tend to be.
One might interpret \Cref{exm:AchingerKpi1,exm:AchingerAAn} by saying that the étale fundamental group of a connected affine $ \FFp $-scheme is so large that it contains all of the homotopical information of the scheme.


\subsection{Monodromy}\label{subsec:etalemonodromy}

Specalising \Cref{prop:Stonemonodromy} to the case of the étale \topos of a scheme shows that lisse étale sheaves are the same as representations of the profinite étale shape:

\begin{prp}
	Let $ X $ be a scheme.
	The unit $ \fromto{X_{\et}}{X_{\et}^{\Stone}} $ restricts to an equivalence
	\begin{equation*}
		\Fun(\Shapeetprofin(X),\Spacefin) \equivalent X_{\et}^{\lisse} \period
	\end{equation*}
\end{prp}

\noindent This generalizes the classical fact that the profinite étale fundamental groupoid
\begin{equation*}
	\Shapeetprofinone(X) \equivalent \trun_{\leq 1} \Shapeetprofin(X)
\end{equation*}
classifies lisse étale sheaves of \textit{sets} (see \Cref{exm:finiteetalesite}).


\subsection{Friedlander's étale homotopy of simplicial schemes}\label{subsec:etalehomotopyofsschemes}

Friedlander \cite[\S 4]{MR676809} extended étale homotopy theory from schemes to simplicial schemes; he called the invariant he constructed the \textit{étale topological type}.
Friedlander's work was later refined by Cox \cite{MR550644}, Isaksen \cite{MR2047848}, Barnea--Schlank \cite{MR3459031}, Carchedi \cite{Carchedi:higheretale}, and Chough \cites{MR3553672}{10.1093/imrn/rnz065}.
Thanks to work of Cox \cite[Theorem III.8]{MR550644}, Isaksen \cite[\S3, Theorem 11]{MR2047848}, and Chough \cite[Proposition 3.2.13]{10.1093/imrn/rnz065}, the étale topolocial type has the following simple description: if $ X_{\ast} $ is a simplicial scheme, then the étale topological type of $ X_{\ast} $ is the colimit of the simplicial prospace
\begin{equation*}
	\goesto{[m]}{\Shapeethyp(X_{m})} \period
\end{equation*}
Again, from \acategorical perspective, there are variations on this notion:

\begin{dfn}\label{def:etalehomotopytypesscheme}
	Let $ X_{\ast} $ be a simplicial scheme.
	We define:
	\begin{itemize}
		\item The \defn{étale shape}\index[terminology]{shape!etale@étale}\index[terminology]{etale shape@étale shape} of $ X_{\ast} $ is the geometric realization
		\begin{equation*}
			\Shapeet(X_{\ast}) \coloneq \colim_{[m] \in \Deltaop} \Shapeet(X_{m}) \index[notation]{Piet@$ \Shapeet, \Shapeethyp $}
		\end{equation*}
		of the simplicial prospace $ \goesto{[m]}{\Shapeet(X_{m})} $.

		\item \defn{Friedlander's étale topological type}\index[terminology]{etale topological type@étale topological type} of $ X_{\ast} $ is the geometric realization
		\begin{equation*}
			\Shapeethyp(X_{\ast}) \coloneq \colim_{[m] \in \Deltaop} \Shapeethyp(X_{m}) 
		\end{equation*}
		of the simplicial prospace $ \goesto{[m]}{\Shapeethyp(X_{m})} $.
	\end{itemize}
\end{dfn}

\begin{nul}\label{nul:equivofetaletopologicalonprotrunc}
	Since protuncation is a left adjoint, from \Cref{nul:equivofetalehomotopyonprotrunc} we deduce that the natural morphism of prospaces
	\begin{equation*}
		\fromto{\Shapeethyp(X_{\ast})}{\Shapeet(X_{\ast})}
	\end{equation*}
	becomes an equivalence after protruncation.
	Hence after profinite completion as well.
\end{nul}	

\begin{nul}[the étale \topos of a simplicial scheme]
	We can extend the functor that assigns a scheme its étale \topos to simplicial schemes by left Kan extension; then the étale \topos of a simplicial scheme $ X_{\ast} $ is given by the geometric realization
	\begin{equation*}
		X_{\ast,\et} \coloneq \colim_{[m] \in \Deltaop} X_{m,\et}
	\end{equation*}
	in $ \Top_{\infty} $.
	Since the shape is a left adjoint, we see that the shape of the \topos $ X_{\ast,\et} $ coincides with the étale shape $ \Shapeet(X_{\ast}) $.
	Since hypercomplete \topoi are closed under colimits in $ \Top_{\infty} $, Friedlander's étale topological type coincides with the shape of the hypercomplete \topos given by the geometric realization of the simplicial hypercomplete \topos \smash{$ \goesto{[m]}{X_{m,\et}^{\hyp}} $}.
\end{nul}


\subsection{Riemann Existence Theorem}\label{subsec:ordinaryRiemann}

In this section we recall the Artin--Mazur--Friedlander Riemann Existence Theorem (\Cref{thm:Riemannexistence}); this states that the profinite étale shape of a scheme of finite type over the complex numbers agrees with the profinite completion of the homotopy type of its underlying analytic space.

\begin{ntn}[analytification]
	Write $ \CCup $ for the field of complex numbers and $ \SchftCC $\index[notation]{SchC@$ \SchftCC $} for category of schemes of finite type over $ \CCup $.
	We write
	\begin{equation*}
		(-)^{\an} \colon \fromto{\SchftCC}{\TSpc}\index[notation]{analytification@$ X^{\an} $}
	\end{equation*}
	for the \defn{analytification}\index[terminology]{analytification} functor: this carries a scheme $ X $ of finite type over $ \CCup $ to $ X(\CCup) $ equipped with the complex analytic topology.

	We simply write $ X_{\an} \colonequals \widetilde{X^{\an}} $ for the \topos of sheaves of spaces on the topological space $ X^{\an} $.
\end{ntn}

\begin{rec}\label{rec:morphismfromanalytictoetale}
	Let $ X $ be a scheme finite type over $ \CCup $.
	In \cite[Exposé XI, 4.0]{MR50:7132}, Artin defines a natural geometric morphism of $ 1 $-topoi
	\begin{equation*}
		\varepsilon_{X,\ast} \colon \fromto{X_{\an,\leq 0}}{X_{\et,\leq 0}}
	\end{equation*}
	from the $ 1 $-topos of sheaves of sets on $ X^{\an} $ to the $ 1 $-topos of sheaves of sets on the étale site of $ X $. The geometric morphism $ \varepsilon_{X,\ast} $ extends to a natural geometric morphism of $ 1 $-localic \topoi 
	\begin{equation*}
		\varepsilon_{X,\ast} \colon \fromto{X_{\an}}{X_{\et}} \period
	\end{equation*}
	The naturality can be encoded as a functor \smash{$\fromto{\SchftCC}{\Fun([1],\Top_{\infty})}$}: if $f \colon\fromto{X}{Y}$ is a morphism of finite type $\CCup$-schemes, then there is a natural equivalence
	\begin{equation*}	
		\flowerstar^{\et} \varepsilon_{X,\ast} \simeq \varepsilon_{Y,\ast} \flowerstar^{\an} \period
	\end{equation*}
\end{rec}

In light of \Cref{thm:Hoyoisetalehomotopycomparison}, the Riemann Existence Theorem proved by Artin--Mazur \cite[Theorem 12.9]{MR0245577} and later Friedlander \cite[Theorem 8.6]{MR676809} asserts that $X_{\an} $ and $ X_{\et} $ have the same profinite shape, when regarded as \proobjects of the homotopy category of spaces.
The Artin--Mazur--Friedlander equivalence refines to an equivalence in the \category of profinite spaces (cf. \cites[Proposition 4.12]{Carchedi:higheretale}[\S4]{10.1093/imrn/rnz065}). 
Indeed, the Théorème de Comparaison \cite[Exposé XI, Théorèmes 4.3 \& 4.4]{MR50:7132} can be employed to provide an equivalence between the \category of lisse étale sheaves of spaces on $ X $ and that of lisse sheaves of spaces on $ X^{\an} $, whence we obtain the following.

\begin{thm}[{Riemann Existence\index[terminology]{Riemann Existence Theorem}}]\label{thm:Riemannexistence}
	Let $ X $ be a scheme finite type over $ \CCup $.
	Then $ \varepsilon_{X}^{\ast} $ restricts to an equivalence \smash{$ \equivto{X_{\et}^{\lisse}}{X_{\an}^{\lisse}} $} of \categories of lisse sheaves. 
	Equivalently, $ \varepsilon_{X,\ast} $ induces an equivalence of profinite spaces
	\begin{equation*}
		\equivto{(X^{\an})\profincomp \simeq \Shapeprofin(X_{\an})}{\Shapeprofin(X_{\et})} \period 
	\end{equation*}
\end{thm}


\subsection{A van Kampen Theorem for étale shapes}\label{subsec:etalevanKampen}

In this section we prove a van Kampen Theorem from étale shapes (\Cref{cor:vanKampenetalehomotopy}).
We deduce this from the fact that the functor that sends a scheme to its étale \topos satisfies \textit{Nisnevich excision}\index[terminology]{Nisnevich excision} (\Cref{prop:Nisnevichexcision}). 

\begin{dfn}
	We call a pullback square of schemes
	\begin{equation}\label{sq:elementaryNis}
      \begin{tikzcd}[sep=2.25em]
	       U'  \arrow[dr, phantom, very near start, "\lrcorner", xshift=-0.25em, yshift=0.25em] \arrow[d] \arrow[r, hooked] & X' \arrow[d, "p"] \\ 
	       U \arrow[r, hooked, "j"'] & X 
      \end{tikzcd}
    \end{equation}
    an \defn{elementary Nisnevich square}\index[terminology]{elementary Nisnevich square} if $ j $ is an open immersion, $ p $ is étale, and $ p $ induces an isomorphism $ \isomto{p^{-1}(X \smallsetminus U)}{X \smallsetminus U} $. 
    Here the closed complement $ X \smallsetminus U $ of $ U $ is given the reduced structure.
\end{dfn}

Rydh's general descent theorem \cite[Theorem A]{MR2774654} implies that the formation of the étale $ 1 $-topos sends elementary Nisnevich squares to pushout squares of $ 1 $-topoi. 
The same is true for étale \topoi, though this is not implied by Rydh's result because $ 1 $-localic \topoi are not closed under colimits in $ \Top_{\infty} $.
As in Rydh's theorem, this can be deduced from étale descent (combined with Morel and Voevodsky's theorem characterizing Nisnevich sheaves as presheaves satisfying Nisnevich excision \cites[\SAGthm{Theorem}{3.7.5.1}]{SAG}[\S3, Proposition 1.16]{MR1813224}), but the following proposition provides an elementary proof.

\begin{prp}\label{prop:Nisnevichexcision}
	Given an elementary Nisnevich square of schemes \eqref{sq:elementaryNis}, the induced square of étale \topoi 
	\begin{equation}\label{sq:elementaryNisontopoi}
      \begin{tikzcd}
	       U'_{\et} \arrow[d] \arrow[r, hooked] & X'_{\et} \arrow[d, "p_{\ast}"] \\ 
	       U_{\et} \arrow[r, hooked, "j_{\ast}"'] & X_{\et} 
      \end{tikzcd}
    \end{equation}
    is a pushout square and pullback square in $ \Top_{\infty} $.
	The same is true after passing to hypercomplete étale \topoi.
\end{prp}

\begin{proof}
	The fact that the \eqref{sq:elementaryNisontopoi} is a pullback is immediate from the fact that $ j $ is an open immersion; the same is true for hypercomplete étale \topoi since hypercompletion is a right adjoint.

	Let $ \yo \colon \incto{X^{\et}}{X_{\et}} $ denote the Yoneda embedding of étale site of $ X $ to the étale \topos. 
	Note that if $ Y \in X^{\et} $ is a scheme étale over $ X $, then the natural geometric morphism $ \fromto{Y_{\et}}{(X_{\et})_{/\yo(Y)}} $ is equivalence.
	Since colimits in \atopos are \textit{van Kampen}\footnote{A colimit in \acategory $ C $ with pullbacks is \defn{van Kampen} if the functor $ \fromto{C^{\op}}{\Cat_{\infty}} $ given by $ \goesto{c}{C_{/c}} $ transforms it into a limit in $ \Cat_{\infty} $.
	A presentable \category $ C $ is \atopos if and only if colimits in $ C $ are van Kampen; see \cites[\HTTthm{Proposition}{5.5.3.13}, \href{http://www.math.ias.edu/~lurie/papers/HTT.pdf\#theorem.6.1.3.9}{Theorem 6.1.3.9(3)}, \& \HTTthm{Proposition}{6.3.2.3}]{HTT}{MR3935451}.} and $ \yo(X) $ is the terminal object of $ X_{\et}$, it thus suffices to prove that the pullback square 
	\begin{equation}\label{sq:elementaryNispullbackintopos}
      \begin{tikzcd}[sep=2.25em]
	      \yo(U')  \arrow[dr, phantom, very near start, "\lrcorner", xshift=-0.25em, yshift=0.25em] \arrow[d] \arrow[r] & \yo(X') \arrow[d] \\ 
	       \yo(U) \arrow[r] & \yo(X) 
      \end{tikzcd}
    \end{equation}
    in $ X_{\et} $ is also a pushout.
    In this case, the fact that truncated objects are hypercomplete implies that the same is true in \smash{$ X_{\et}^{\hyp} $}.
	The fact that \eqref{sq:elementaryNispullbackintopos} is a pullback square is immediate from \cite[\href{http://www.math.ias.edu/~lurie/papers/SAG-rootfile.pdf\#theorem.2.5.2.1}{Proposition 2.5.2.1(3)}]{SAG}, the hypotheses of which are valid because \eqref{sq:elementaryNis} is an elementary Nisnevich square.
\end{proof}

\Cref{prop:Nisnevichexcision} immediately implies the classicial `excision' theorem in étale cohomology \cite[Chapter III, Proposition 1.27]{MR559531}.
Since the shape is a left adjoint, the following van Kampen Theorem for the étale shape is immediate.
This generalizes a theorem of Isaksen \cite[\S2, Theorem 8]{MR2047848}.

\begin{cor}[{étale van Kampen Theorem\index[terminology]{etale van Kampen Theorem@étale van Kampen Theorem}\index[terminology]{van Kampen Theorem!etale@étale}}]\label{cor:vanKampenetalehomotopy}
	Given an elementary Nisnevich square of schemes \eqref{sq:elementaryNis}, the induced squares
	\begin{equation*}
      \begin{tikzcd}
	       \Shapeet(U') \arrow[d] \arrow[r] & \Shapeet(X') \arrow[d] \\ 
	       \Shapeet(U) \arrow[r] & \Shapeet(X) 
      \end{tikzcd}
      \andeq
      \begin{tikzcd}
	       \Shapeethyp(U') \arrow[d] \arrow[r] & \Shapeethyp(X') \arrow[d] \\ 
	       \Shapeethyp(U) \arrow[r] & \Shapeethyp(X) 
      \end{tikzcd}
    \end{equation*}
	are pushout squares in $ \Pro(\Space) $. 
\end{cor}

\begin{nul}
	Since protruncation and profinite completion are left adjoints, \Cref{cor:vanKampenetalehomotopy} show that the protruncated and profinite étale shapes send elementary Nisnevich squares to pushout squares in $ \Pro(\Space_{<\infty}) $ and $ \Spaceprofin $, respectively. 
	In particular, \Cref{prop:Nisnevichexcision} (and \SAG{Proposition}{2.5.2.1}) immediately imply Misamore's `étale van Kampen Theorem' \cite[Corollaries 6.5 \& 6.6]{MR2894444} in the case of schemes.
	See also \cites{MR2010804}[\S5]{MR2266996}{MR1914264}.
\end{nul}

\newpage

\section{Galois categories}\label{sect:ramification}

This chapter is dedicated to studying the profinite stratified shape of étale \topoi of coherent schemes; we call this the \textit{stratified étale homotopy type}.
We also show that the stratified étale homotopy type of a coherent scheme $ X $ has a very explicit description as a profinite $ 1 $-category whose topology globalizes the topologies of the absolute Galois groups of the residue fields of the points of $ X $.
The stratified étale homotopy type is the profinite category $ \Gal(X) $ from the Introduction.
For this reason, we also call the profinite $ 1 $-category $ \Gal(X) $ the \textit{Galois category} of $ X $.%
\footnote{The term `Galois category' already has a well-established meaning in Grothendieck's Galois theory \cites[\stackstag{0BMQ}]{stacksproject}[Exposé V]{MR50:7129}{MR2548205}. 
We have chosen to use the term `Galois category' for this distinct notion because $ \Gal(X) $ really is a globalization of the absolute Galois group.}

\Cref{subsec:Galcat} defines the stratified étale homotopy type and shows that it coincides with the `Galois category' from the Introduction.
\Cref{subsec:Galexamples} gives some sample computations of the stratified étale homotopy type.
\Cref{subsec:Galsieves} demonstrates how some properties of schemes can be detected on the level of their Galois categories.
\Cref{subsec:Galundercats} shows that the Galois category of the strict localization of a scheme at a point is an undercategory, and, dually, the Galois category of the strict \textit{normalization} of a scheme at a point is an overcategory.
\Cref{subsec:Galprotruncatedshape} uses the material from \cref{subsec:shapefromprofinstratshape} to show that (up to protruncation), the étale homotopy type of a coherent scheme $ X $ can be recovered by $ \Gal(X) $ by inverting all morphisms.
In the setting of finite type schemes over the complex numbers, \cref{subsec:stratRiemann} provides a stratified refinement of the Riemann Existence Theorem (\Cref{thm:Riemannexistence}).
In \cref{subsec:GalvanKampen} we finish the chapter with a van Kampen Theorem for Galois categories.


\subsection{Galois categories of schemes}\label{subsec:Galcat}

\begin{ntn} 
	Let $ X $ be a coherent scheme.
	We write $\FC(X)$ for the $1$-category of nondegenerate, finite, constructible stratifications of the spectral topological space $ X^{\zar} $ (\Cref{ntn:FC}). 
	Recall that the spectral topological space $ X^{\zar} $ corresponds under Hochster Duality to the profinite poset $ \{P\}_{P\in \FC(X)} $ (\cref{subsec:ordinaryHochsterdual}).
\end{ntn}

\begin{ntn}
	We write $ \Sch $\index[notation]{Sch@$ \Sch $} for the $ 1 $-category of \textit{coherent} schemes \cref{nul:coherentqcqs}.
\end{ntn}

\begin{dfn}
	Let $ X $ be a coherent scheme.
	Then we write
	\begin{equation*}
		\Gal(X)\coloneq \StrShape^{X^{\zar}}(X_{\et}) \period \index[notation]{Gal@$ \Gal(X) $}
	\end{equation*}
	We call $ \Gal(X) $ the \defn{Galois category}\index[terminology]{Galois category} of $ X $.
	
	Since the \topos $ X_{\et} $ is $ 1 $-localic, the profinite stratifed space $ \Gal(X) $ is $ 1 $-truncated (\Cref{prop:nlocalicntruncated}).
	Hence the formation of Galois categories defines a functor
	\begin{equation*}
		\Gal \colon \fromto{\Sch}{(\Strprofin)_{\leq 1}}
	\end{equation*}

	More generally, if $ \fromto{X^{\zar}}{P} $ is a nondegenerate, finite, constructible stratification of $ X $, we define
	\begin{equation*}
		\Gal(X/P) \coloneq \StrShape^{P}(X_{\et}) \period
	\end{equation*} 
\end{dfn}

\begin{nul}
	We obtain a diagram
	\begin{equation*}
		\Gal(X/-)\colon\fromto{\FC(X)}{\Strprofin}
	\end{equation*} 
	of localizations.
\end{nul}

\begin{cnstr}[{explicit description of $ \Gal(X) $}]\label{cnstr:idPiinfty1withGal}
	Let $ X $ be a coherent scheme.
	The $ X^{\zar} $-stratified \topos $X_{\et}$ is spectral.
	By \Cref{lem:matofPiinfisPt} we have a natural equivalence of categories
	\begin{equation*}
		\mat(\Gal(X)) \simeq \Pt(X_{\et}) \period
	\end{equation*}
	The Grothendieck School \cite[Exposé VIII, Théorème 7.9]{MR50:7131} provides the following description of the category $ \Pt(X_{\et}) $ of points of the étale \topos of $ X $: an object is a geometric point $x\to X$, and given geometric points $x\to X$ and $y\to X$, the set $\Map_{\Pt(X_{\et})}(x,y)$ is identified with the set of lifts 
	\begin{equation*}
		\begin{tikzcd}
			X_{(x)} \arrow[d] & \\ 
			X & y \arrow[l] \arrow[ul, dotted] 
		\end{tikzcd}
	\end{equation*}
	of the geometric point $ y $ to the strict localization $ X_{(x)} $ of $ X $ at $ x $.
	In other words, the $ 1 $-category $\mat(\Gal(X))$ agrees with the underlying $1$-category denoted $\Gal(X)$ in \Cref{cnstr:GalXcat}.

	Appealing to \Cref{nul:1truncatedprofinstrat}, we regard $ \Gal(X) $ as a category object in the category of Stone topological spaces.
	Unwinding the definitions, we see that topology on $\Gal(X)$ is precisely the one described in \Cref{cnstr:GalXtop}
\end{cnstr}

Our \Categorical Hochster Duality Theorem (\Cref{thm:inftyHochster}) implies the following \defn{Exodromy Equivalence} for schemes:

\begin{thm}
	Let $ X $ be a coherent scheme.
	Then there is a natural equivalence
	\begin{equation*}
		X_{\et} \simeq \widetilde{\Gal(X)} \period
	\end{equation*}
	Equivalently, there is a natural equivalence
	\begin{equation*}
		X_{\et}^{\cons} \simeq \Fun(\Gal(X),\Spacefin) \period
	\end{equation*}
\end{thm}

\begin{nul}
	If $ X $ and $Y$ are coherent schemes, then the natural map
	\begin{equation*}
		\fromto{\Map_{\Top_{\infty}}(X_{\et},Y_{\et})}{\Map_{\Strprofin}(\Gal(X),\Gal(Y))}
	\end{equation*}
	is an equivalence.
\end{nul}

\begin{nul}
	In \Cref{sect:extending} we prove a variant of the Exodromy Equivalence
	\begin{equation*}
		X_{\et}^{\cons} \simeq \Fun(\Gal(X),\Spacefin) 
	\end{equation*}
	for constructible sheaves with coefficients in a finite ring as well as $ \el $-adic sheaves.
	We also extend the Exodromy Equivalence from coherent schemes to a large class of stacks.
\end{nul}


\subsection{Examples}\label{subsec:Galexamples}

We now provide some example computations of profinite Galois categories.

\begin{exm}
	Let $ X $ be a coherent scheme, and  consider $ X $ with its trivial $\{0\}$-stratifi\-cation. 
	As a special case of \Cref{nul:compatibilitywithinverteverything}, $ \Gal(X/\{0\}) $ recovers the usual profinite étale shape of $ X $: there is a canonical equivalence
	\begin{equation*}
		\equivto{\Gal(X/\{0\})}{\Shapeetprofin(X)} \period
	\end{equation*}
\end{exm}

\begin{exm}[DVRs] 
	Let $ A $ be a discrete valuation ring, write $ K $ for the fraction field of $ A $, and $ k $ for  the residue field of $ A $.
	Write $S=\Spec A$, $ s = \Spec k $ for the closed point of $ S $, and $ \eta = \Spec K $ for the generic point of $ S $.
	Then $S^{\zar}\cong[1]$, so $S_{\et}$ is a spectral \topos that is naturally $[1]$-stratified, with closed stratum $ s_{\et} $ and open stratum $ \eta_{\et} $.

	Fix a separable closure $ k^{\sep} \supset k $.
	Write $S^{\hens} = \Spec A^{\hens} $ and write $S^{\sh}$ for the spectrum of the strict henselization $A^{\sh} $ of $A$ with respect to $ k^{\sep} $.
	Write
	\begin{equation*}
		\eta^{\hens} \colonequals \Spec(\Frac A^{\hens}) \andeq \eta^{\sh} \colonequals \Spec(\Frac A^{\sh}) \period
	\end{equation*}
	In this case, please observe that the evanescent \topos \smash{$ s_{\et}\orientedtimes_{S_{\et}}S_{\et} $} can be identified with the étale \topos \smash{$S^{\hens}_{\et}$} (\Cref{exm:localizationishenselization}), and the oriented fiber product \smash{$s_{\et}\orientedtimes_{S_{\et}}\eta_{\et}$} can be identified with the étale \topos \smash{$\eta^{\hens}_{\et}$}.

	Fix a separable closure $ K^{\sep} \supset K $.
	Write $\Gk$ and $ \GK $ for the absolute Galois groups of $ k $ and $ K $, respectively.
	Write $\Dup_A\subseteq \GK$ for the decomposition group\index[terminology]{decomposition group}\index[notation]{DA@$\Dup_A$} of $ A $, and $\Iup_A\subseteq \Dup_A$ for the inertia group\index[terminology]{inertia group}\index[notation]{IA@$\Iup_A$} of $ A $ \stacks{0BSD}.
	Recall that there is a canonical isomorphism $ \isomto{\Dup_A/\Iup_A}{\Gk} $ \cites[\stackstag{0BSW}]{stacksproject}[\S2.3, Proposition 11]{MR1045822}.
	The choices of separable closures $ k^{\sep} \supset k $ and $ K^{\sep} \supset K $ provide the following identifications of profinite spaces:
	\begin{equation*}
		\Shapeetprofin(\eta)\simeq \BGK \comma \quad \Shapeetprofin(\eta^{\hens})\simeq \BD_A \comma \quad \Shapeetprofin(\eta^{\sh})\simeq \BI_A \comma \quad\text{and} \quad \Shapeetprofin(S^{\hens})\simeq \BGk \period
	\end{equation*}

	The choices of separable closures of $ k $ and $ K $ along with the quotient map $ \fromto{\Dup_A}{\Gk} $ allow us to noncanonically identify the profinite décollage $\Nerve_{[1]}(\Gal(S))$ as the functor \smash{$\fromto{\sdop([1])}{\Spaceprofin}$} given by the diagram
	\begin{equation*}
		\BGk\ot \BD_A\to \BGK \period
	\end{equation*}
\end{exm}

\begin{exm}[{knots and primes\index[terminology]{knots and primes}}]
	Let $ K $ be a number field, and write $ \Oup_K \subset K $ for the ring of integers of $ K $.
	The profinite Galois category $\Gal(\Spec \Oup_K)$ has (isomorphism classes of) objects the prime ideals of $\Oup_K$.
	For each nonzero prime ideal $\pfrak\in\Spec \Oup_K$, the automorphisms of $\pfrak$ can be identified with the absolute Galois group $\Gup_{\upkappa(\pfrak)}$ of the finite field $\upkappa(\pfrak)$.
	Thus the profinite étale shape of $\Spec \Oup_K$ is stratified by the various closed strata, each of which is an embedded profinite circle -- i.e., a knot \smash{$\BG_{\upkappa(\pfrak)} \equivalent \BZZhat $}.
	The open complement of each $\BG_{\upkappa(\pfrak)}$ is the profinite classifying space of the profinite group
	\begin{equation*}
		\Gup_{\pfrak} \coloneq \piet_1(\Spec(\Oup_K) \smallsetminus  \pfrak) \period
	\end{equation*}
	Note that $ \Gup_{\pfrak}  $ is the automorphism group of the maximal Galois extension of $K$ that is ramified at most only at $\pfrak$ and the infinite primes.
	Enveloping each knot is a tubular neighborhood, given by $\Gal(\Spec A_{\pfrak}^{\sh})$, so that the deleted tubular neighborhood of $\BG_{\upkappa(\pfrak)}$ is a $\BG_{K_{\pfrak}}$.
\end{exm}


\subsection{Sieves \& cosieves of Galois categories}\label{subsec:Galsieves}

One can read off various facts about schemes from their Galois categories.
In this section and the next, we begin to propose a dictionary between schemes and their profinite Galois categories.\footnote{This dictionary first appeared in a preprint of the first-named author \cite{Barwick:galperf}.}
We continue this endeavour in \cref{sect:fullfaithfulness}, as the dictionary is strongest between profinite Galois categories and schemes that admit no nontrivial universal homeomorphisms (\Cref{dfn:toprigid}).

The following is immediate.

\begin{prp}
	A monomorphism $U\inclusion X$ of coherent schemes is an open immersion if and only if the induced functor $\Gal(U)\to\Gal(X)$ is equivalent to the inclusion of a cosieve.

	Dually, a monomorphism $Z\inclusion X$ of coherent schemes is a closed immersion if and only if $\Gal(Z)\to\Gal(X)$ is equivalent to the inclusion of a sieve.
\end{prp}

\begin{rec}
	An \defn{interval}\index[terminology]{interval} in \acategory $ C $ is a full subcategory $ D \subseteq C$ such that a morphism $ d \to d' $ of $D$ factors through an object $ c $ of $C$ only if $ c $ lies in $D$.
\end{rec}

\begin{cor}
	A monomorphism $W\inclusion X$ of coherent schemes is a locally closed immersion if and only if the induced functor $\Gal(W)\to\Gal(X)$ is equivalent to the inclusion of an interval.
\end{cor}

\begin{rec}
	Let $ C $ be \acategory.
	An object $ c \in C $ is \defn{weakly initial}\index[terminology]{weakly initial}\index[terminology]{object!weakly initial} if for every object $ c' \in C $, the space $ \Map_{C}(c,c') $ is nonempty.
	An object $ c' \in C $ is \defn{weakly terminal}\index[terminology]{weakly terminal}\index[terminology]{object!weakly terminal} if $ c $ is weakly initial in $ C^{\op} $.
\end{rec}

\begin{cor}
	A coherent scheme $ X $ is local if and only if $\Gal(X)$ contains a weakly initial object.
	Dually, a coherent scheme $ X $ is irreducible if and only if $\Gal(X)$ contains a weakly terminal object.
\end{cor}

\begin{nul}
	For any coherent scheme $ X $ and any point $x_0\in X^{\zar}$, the Galois category of the localization is the fiber product
	\begin{equation*}
		\Gal(X_{(x_0)})\simeq\Gal(X) \crosslimits_{X^{\zar}}X^{\zar}_{x_0/} \period
	\end{equation*}
	Dually, for any point $y_0\in X^{\zar}$, the Galois category of the closure $X^{(y_0)}$ of $y_0$ (with the reduced subscheme structure, say) is the fiber product
	\begin{equation*}
		\Gal(X^{(y_0)})\simeq\Gal(X) \crosslimits_{X^{\zar}}X^{\zar}_{/y_0} \period
	\end{equation*}
\end{nul}


\subsection{Undercategories \& overcategories of Galois categories}\label{subsec:Galundercats} 

In this section we extend our dictionary by showing that undercategories correspond to localizations, while overcategories correspond to normalizations (\Cref{cor:Galoverunder}).

\begin{ntn}[strict localization]
	Let $ X $ be a scheme and $x\to X$ a point of $ X $.
	We write $ x_0 \in X^{\zar} $ for the image of $ x $ in the Zariski topological space of $ X $ 
	Let $ \upkappa(x_0)^{\sep} \supseteq \upkappa(x) $ denote the separable closure of $ \upkappa(x_0) $ in $ \upkappa(x) $.
	We write $\Ohens_{X,x_0}$\index[notation]{Oh@$\Ohens_{X,x_0}$} for the henselization\index[terminology]{henselization} of the local ring $\Oup_{X,x_0}$, and write
	\begin{equation*}
		\Ohens_{X,x}\supseteq \Ohens_{X,x_0}
	\end{equation*}
	for the unique extension of henselian local rings that on residue fields reduces to the field extension $\upkappa(x_0)^{\sep} \supseteq \upkappa(x_0)$.
	We also write
	\begin{equation*}
		X_{(x)}	\coloneq \Spec(\Ohens_{X,x}) \period \index[notation]{Xx@$X_{(x)}$}
	\end{equation*}
	We call $X_{(x)}$ the \defn{localization of $ X $ at $x$}\index[terminology]{localization!of a scheme} (\Cref{exm:localizationishenselization}).
	The scheme $ X_{(x)} $ is the limit of the diagram of étale $ X $-schemes $ U \to X $ equipped with a lift $x\to U$ of $ x \to X $ to $U $.

	If $x\to X$ is a geometric point, then $\Ohens_{X,x}$ is the strict henselization of $\Oup_{X,x_0}$, and $X_{(x)}$ is the strict localization of $ X $ at $x$.
\end{ntn}

\begin{ntn}[strict normalization]
	Let $ X $ be a scheme, let $y\to X$ is a geometric point, and let $ \upkappa(y)^{\alg} \supset \upkappa(y) $ be an algebraic closure of $ \upkappa(y) $. 
	We write $ X^{(y_0)} \subset X $ for the integral subscheme defined by the Zariski closure of $ \{y_0\} $ with the reduced subscheme structure.
	We write $X^{(y)}$\index[notation]{Xy@$X^{(y)}$} for the normalization of $X^{(y_0)}$ under $\Spec\upkappa(y)^{\alg}$.
	We call $X^{(y)}$ the \defn{strict normalization of $ X $ at $y$}\index[terminology]{strict normalization}.

	The strict normalization $ X^{(y)} $ can be expressed as an inverse limit of schemes over $ X^{(y_0)} $ as follows. 
	For each intermediate field extension $ \upkappa(y_0) \subset k \subset \upkappa(y)^{\alg} $ finite over $ \upkappa(y_0) $, write $ \fromto{X^{(y_0),k}}{X^{(y_0)}} $ for the normalization of $ X^{(y_0)} $ under $ \Spec k $.
	Then $ X^{(y)} $ is the inverse limit
	\begin{equation*}
		\isomto{X^{(y)}}{\lim_{\upkappa(y_0) \subset k \subset \upkappa(y)^{\alg}} X^{(y_0),k}}   
	\end{equation*}
	of the resulting diagram of normalizations and finite surjective transition morphisms \cite[\S3]{MR3649361}.
\end{ntn}

\begin{rmk}
	\defn{Absolutely integrally closed}\index[terminology]{absolutely integrally closed}\index[terminology]{scheme!absolutely integrally closed} schemes are integral normal schemes whose function field is algebraically closed.
	In other words, an absolutely integrally closed scheme is a strict normalization $X^{(y)}$ of a scheme at a geometric point $y\to X$ \cite{MR0289501}.
	This class of schemes has a number of curious properties:
	\begin{itemize}
		\item If $Z$ is absolutely integrally closed, then for any point $z_0\in Z^{\zar}$, the local ring $\Oup_{Z,z_0}$ is strictly henselian \cite[Proposition 2.6]{MR3649361}.

		\item If $Z$ is absolutely integrally closed, then the étale topos and the Zariski topos of $Z$ coincide \cite[Corollary 2.5]{MR3649361}.
		Hence $\Gal(Z)\simeq Z^{\zar}$. 
		In other words, $\Gal(Z)$ is a profinite poset with a terminal object.

		\item If $Z$ is absolutely integrally closed and $W$ is irreducible, then any integral morphism $W\to Z$ is radicial \cite[Lemma 2.3]{MR3649361}. 
		Thus any integral surjection $W\to Z$ is a universal homeomorphism (\Cref{rec:radicialUH}).

		\item If $Z$ is absolutely integrally closed, then the poset $\Gal(Z)\simeq Z^{\zar}$ has all finite nonempty joins \cite[Theorem 2.1]{tsccontractions}.
	\end{itemize}
\end{rmk}

Here now is the basic observation, which follows more or less immediately from the limit descriptions of the strict localization and the strict normalization:

\begin{prp}
	Let $ X $ be a coherent scheme, and let $x^{\alg} \to X$ and $y^{\alg} \to X$ be points of $ X $ that correspond to algebraic closures of the residue fields of their image points.
	Write $ x \to X $ and $ y \to Y $ for the underlying geometric points of $x^{\alg} \to X$ and $y^{\alg} \to X$, respectively.
	Then the following profinite sets are in (canonical) bijection:
	\begin{itemize}
		\item The profinite set $\Map_{\Gal(X)}(x,y)$ of morphisms $x\to y$ in $\Gal(X)$.

		\item The profinite set $\Mor_X(y,X_{(x)})$ of lifts of $y$ to the strict localization $X_{(x)}$.

		\item The profinite set $\Mor_X(x^{\alg},X^{(y)})$ of lifts of $x$ to the strict normalization $X^{(y)}$.
	\end{itemize}
\end{prp}

\noindent We may thus describe the over- and undercategories of Galois categories.
The expression of the undercategory in terms of the strict localization is originally due to Grothendieck \cite[Exposé VIII, Corollaire 7.6]{MR50:7131}.

\begin{cor}\label{cor:Galoverunder}
	Let $ X $ be a coherent scheme, and let $x\to X$ and $y\to X$ be geometric points of $ X $.
	Then we have
	\begin{equation*}
		\Gal(X)_{x/} \simeq \Gal(X_{(x)}) \andeq \Gal(X)_{/y}\simeq\Gal(X^{(y)}) \period
	\end{equation*}
\end{cor}

\begin{cor}
	Let $ X $ be a coherent scheme. 
	Then $\Gal(X)$ is equivalent to both of the following full subcategories of $ X $-schemes:
	\begin{itemize}
		\item The full subcategory spanned by the strict localizations of $ X $.

		\item The full subcategory spanned by the strict normalizations of $ X $.
	\end{itemize}
\end{cor}

\noindent Since $\Gal(X^{(y)})\simeq X^{(y),\zar}$, it follows that Galois categories are of a very particular sort:

\begin{cor} 
	Let $ X $ be a coherent scheme.
	For any geometric point $y\to X$, the overcategory $\Gal(X)_{/y}$ is a profinite poset with all finite nonempty joins.
	In particular, every morphism of $\Gal(X)$ is a monomorphism. 
\end{cor}


\subsection{Recovering the protruncated étale shape}\label{subsec:Galprotruncatedshape}

Since $ \Gal(X) $ is the profinite stratified shape of the spectral \topos $ X_{\et} $, the fact that the profinite stratified shape is a delocalization of the protruncated shape (\Cref{thm:protruncdeloc}) immediately implies the following:

\begin{thm}[Homotopy]\label{thm:mainAG}
	Let $ X $ be a coherent scheme.
	Then there is a natural natural map of prospaces
	\begin{equation*}
		\theta_{X} \colon \fromto{\Shapeet(X)}{\invert(\Gal(X))} \period
	\end{equation*}
	Moreover, $ \theta_{X} $ induces an equivalence on protruncations.

	As a consequence, for each integer $ n \geq 1 $ and geometric point $ \fromto{x}{X} $, there is a natural isomorphisms of progroups
	\begin{equation*}
		\isomto{\pietext_n(X,x)}{\uppi_{n}(\invert(\Gal(X)),x)} \period
	\end{equation*}
	Here $ \pietext_n(X,x) $ is the $ n $-th homotopy progroup of the étale shape $ \Shapeet(X) $ of $ X $ \cref{nul:etalehomotopygroups}.
\end{thm}


\subsection{Stratified Riemann Existence Theorem}\label{subsec:stratRiemann}

We prove a stratified refinement of the Riemann Existence Theorem of Artin--Mazur--Friedlander (\Cref{thm:Riemannexistence}), giving a comparison between étale and analytic stratified homotopy types for schemes of finite type over the field $\CCup$ of complex numbers (\Cref{cor:stratRiemann}).

\begin{cnstr}
	Let $ X $ be a scheme of finite type over $ \CCup $.
	Let us recall the geometric morphism $ \epsilonlowerstar \colon \fromto{X_{\an}}{X_{\et}} $ of \Cref{rec:morphismfromanalytictoetale}.
	We give $X^{\an}$ the profinite stratification $\fromto{X^{\an}}{X^{\zar}}$ of \Cref{exm:Xanprofinstrat}.
	With respect to this stratification, $ \epsilonlowerstar$ is an $ X^{\zar} $-stratified geometric morphism.

	Let $ S $ be a spectral topological space and let $ X^{\zar} \to S $ be a quasicompact continuous map.
	We write
	\begin{equation*}
		\XSan \coloneq \Sheff{X_{\an}^{\Scons}} \andeq \XSet \coloneq \Sheff{X_{\et}^{\Scons}}
	\end{equation*}
	for the $ S $-spectrifications of $ X_{\an} $ and $ X_{\et} $, respectively. 
	Since the geometric morphism $ \epsilonlowerstar \colon \fromto{X_{\an}}{X_{\et}} $ is $ X^{\zar} $-stratified, the pullback functor $ \epsilonupperstar $ restricts to a morphism of \pretopoi
	\begin{equation*}
		\fromto{X_{\et}^{\Scons}}{X_{\an}^{\Scons}}
	\end{equation*}
	(\Cref{lem:morspreserveconstrspectral}), hence induces an $ S $-stratified geometric morphism 
	\begin{equation*}
		\epsilonlowerstar \colon \fromto{\XSan}{\XSet} \period
	\end{equation*}

	The inclusions $ X_{\an}^{\Scons} \inclusion X_{\an} $ and $ X_{\et}^{\Scons} \inclusion X_{\et} $ induce natural geometric morphisms
	\begin{equation*}
		v_{\ast}^{S,\an} \colon X_{\an} \to \XSan \andeq v_{\ast}^{S,\et} \colon X_{\et} \to \XSet
	\end{equation*}
	that exhibit $\XSan$ and $\XSet$ as the $S$-spectrifications of $X_{\an}$ and $X_{\et}$.
	These functors are compatible with $\varepsilon$ in the sense that the squares
	\begin{equation*}\label{eqn:epsilonandv}
		\begin{tikzcd}
			X_{\an} \arrow[d, "v_{\ast}^{S,\an}"'] \arrow[r, " \epsilonlowerstar"] & X_{\et} \arrow[d, "v_{\ast}^{S,\et}"] \\
			\XSan \arrow[r, "\varepsilon_{\ast}"'] & \XSet 
		\end{tikzcd}
	\end{equation*}
	commute, by construction.
\end{cnstr}

\begin{nul}
	Let $ \fromto{X^{\zar}}{P} $ be a finite constructible stratification.
	Then the topological space $X^{\an}$ also inherits a stratification $\fromto{X^{\an}}{P}$, which is conical.
	The Exodromy Equivalence for stratified topological spaces (\Cref{subexm:exodromyfortopspaces}), provides an equivalence
	\begin{equation*}
		\Exit^P(X^{\an})\profincomp \equivalent \StrShape^P(\XPan)
	\end{equation*}
	between the profinite completion (in the stratified sense) of the exit-path \category of $ X^{\an} $ and the profinite stratified shape of $ \XPan $.
\end{nul}

We now discuss a general nonabelian form of the results identified in \cite[6.1.2]{MR751966}.
The key result, which seems to be well-known to experts, is the following `nonabelian derived' form of the Artin Comparison Theorem \cite[Exposé XVI, Théorème 4.1]{MR50:7132}.

\begin{thm}\label{thm:Artincomparison}
	Let $ f \colon X \to Y $ be a morphism of finite type $\CCup$-schemes.
	Then the square
	\begin{equation*}\label{eqn:Artin}
		\begin{tikzcd}
			\XXan \arrow[d, "f^{\an}_{\ast}"'] \arrow[r, " \epsilonlowerstar"] & X_{\et} \arrow[d, "f^{\et}_{\ast}"] \\ 
			\YYan \arrow[r, "\varepsilon_{\ast}"'] & Y_{\et} 
		\end{tikzcd}
	\end{equation*}
	of \topoi satisfies the left \basechange condition when restricted to constructible sheaves on $X_{\et}$.
	That is, for any constructible sheaf $F \in X_{\et}$, the natural \basechange morphism
	\begin{equation*}
		\varepsilon^{\ast} f^{\et}_{\ast} F \to f^{\an}_{\ast} \varepsilon^{\ast} F
	\end{equation*}
	is an equivalence.
\end{thm}

\begin{prp}\label{prop:constructibleanalyticetalesheavesofsets}
	Let $ X $ be a scheme of finite type over $ \CCup $.
	Then the pullback functor $ \epsilonupperstar \colon \fromto{X_{\et}}{X_{\an}} $ restricts to an equivalence on constructible sheaves.
\end{prp}

\begin{proof}
	We induct on the Krull dimension of $ X $; the claim in obvious for dimension $0$.
	
	Let us assume that $ X $ is of dimension $n$, and that the claim is known for schemes of dimension $<n$.
	We may also assume that $ X $ is irreducible.
	In this case, we may write $ X^{\zar} $ as the limit of stratifications of the form $X^{\zar} \to S$,
	where
	\begin{equation*}
		S = Z^{\zar} \cup \{\infty\} \simeq \lim_{P \in \FC(Z)} P^{\rhd}
	\end{equation*}
	for some closed subscheme $ Z \subset X $ of dimension $<n$.
	Hence it will suffice to show that for any such stratification $X^{\zar} \to S$, the $S$-stratified geometric morphism
	\begin{equation*}
		\epsilonlowerstar \colon \XSan \to \XSet
	\end{equation*}
	is an equivalence.
	The source is a recollement of $(Z/Z^{\zar})_{\an}$ and $(U/\infty)_{\an}$:
	\begin{equation*}
		\begin{tikzcd}
			(Z/Z^{\zar})_{\an} \arrow[r, "i^{\an}_{\ast}", hooked] & \XSan & (U/\infty)_{\an} \arrow[l, "j^{\an}_{\ast}"', hooked'] \period
		\end{tikzcd}
	\end{equation*}
	Accordingly, the target is a recollement of $Z_{\et}$ and $(U/\infty)_{\et}$:
	\begin{equation*}
		\begin{tikzcd}
			Z_{\et} \arrow[r, "i^{\et}_{\ast}", hooked] & \XSet & (U/\infty)_{\et} \arrow[l, "j^{\et}_{\ast}"', hooked'] \period
		\end{tikzcd}
	\end{equation*}
	Restricted to the closed piece of this recollement, $\epsilonlowerstar$ is an equivalence by the induction hypothesis.
	Restricted to the open piece, $\epsilonlowerstar$ is an equivalence by the \textit{unstratified} Riemann Existence Theorem (\Cref{thm:Riemannexistence}).

	It therefore suffices to prove that the gluing functors agree under these equivalences.
	In other words, it suffices to prove that for any lisse sheaf $F$ on $U_{\et}$, the morphism
	\begin{equation*}
		\epsilonupperstar i_{\et}^{\ast} j^{\et}_{\ast} F \to i_{\an}^{\ast} j^{\an}_{\ast} \epsilonupperstar F
	\end{equation*}
	is an equivalence (\Cref{prp:bcorientedpodependsonbcgluing}).
	Now \Cref{thm:Artincomparison} applies, completing the proof.
\end{proof}

\begin{ntn}
	Let $ X $ be a scheme of finite type over $ \CCup $.
	We write $ \Gal_{\an}(X) $ for the profinite $ X^{\zar} $-stratified space
	\begin{equation*}
		\Gal_{\an}(X) \coloneq \StrShape^{X^{\zar}}(X_{\an}) \equivalent \StrShape^{X^{\zar}}(\XXan) \period \index[notation]{Galan@$ \Gal_{\an}(X) $}
	\end{equation*}
	We write
	\begin{equation*}
		\varepsilon \colon \Gal_{\an}(X) \to \Gal(X) \period
	\end{equation*}
	for the morphism of profinite $ X^{\zar} $-stratified spaces induced by the geometric morphism $ \epsilonlowerstar $.
\end{ntn}

\begin{cor}[Stratified Riemann Existence]\label{cor:stratRiemann}
	Let $ X $ be a scheme of finite type over $ \CCup $.
	Then the natural morphism $ \varepsilon \colon \Gal_{\an}(X) \to \Gal(X) $ is an equivalence.
\end{cor}

As a corollary, we also obtain:

\begin{cor}
	Let $ X $ be a scheme of finite type over $ \CCup $ and let $X^{\zar} \to P$ be a finite constructible stratification.
	Then the natural functor
	\begin{equation*}
		\Exit^P(X^{\an}) \to \Gal(X/P) 
	\end{equation*}
	exhibits the profinite stratified space $\Gal(X/P)$ as the profinite completion of the exit-path \category of $X^{\an} \to P$.
\end{cor}


\subsection{The van Kampen Theorem for Galois categories}\label{subsec:GalvanKampen}

For this section, we fix a coherent scheme $ X $ and a nondegenerate constructible stratification $ X^{\zar}\to[1] $.
We write $ Z \subset X $ for the closed stratum and $ U \subset X $ for the quasicompact open complement of $ Z $.
The décollage associated to the profinite $ [1] $-stratified space $ \Gal(X/[1]) $ is the functor $\fromto{\sdop([1])}{\Spaceprofin}$ given by the diagram
\begin{equation*}
	\Shapeetprofin(Z) \ot \Shapeprofin(Z_{\et}\orientedtimes_{X_{\et}}U_{\et}) \to \Shapeetprofin(U) \period
\end{equation*}
(Note that any subscheme structure on $Z$ will do, as nilimmersions don't affect the étale \topos.) 
The profinite space $\Shapeprofin(Z_{\et}\orientedtimes_{X_{\et}}U_{\et})$ is the \textit{deleted tubular neighborhood}\index[terminology]{deleted tubular neighborhood} of $\Shapeetprofin(Z)$ in $\Shapeetprofin(X)$.

\begin{exm}
	Suppose that $ Z = \{z\} $ and $ \upkappa(z) $ is separably closed.
	Then the deleted tubular neighborhood of $ \Shapeetprofin(Z) \simeq \ast $ in $ \Shapeetprofin(X) $ is the profinite étale shape of the punctured Milnor tube $ X_{(z)} \smallsetminus \{z\} $.
\end{exm}

\begin{exm}
	Suppose that $ X $ is a curve over a field $k$ and $Z=\{z\}$ is a rational point.
	Then we obtain an identification of the deleted tubular neighborhood with the profinite classifying space of `the' profinite decomposition group $ \Dup_z \subseteq \piet_1(U)$.
	In more general situations, this justifies thinking of the deleted tubular neighborhood $\Shapeprofin(Z_{\et}\orientedtimes_{X_{\et}}U_{\et})$ as a kind of `decomposition homotopy type'.
\end{exm}

Specializing \Cref{exm:vanKampen1strat} to the étale topology yields the following \textit{étale van Kampen Theorem}\index[terminology]{etale van Kampen Theorem@étale van Kampen Theorem}\index[terminology]{van Kampen Theorem!etale@étale}:

\begin{prp}\label{prp:etalevanKampen}
	Let $ X $ be a coherent scheme and let $ X^{\zar}\to[1] $ be a nondegenerate constructible stratification with closed stratum $ Z $ and open stratum $ U $.
	Then the square
	\begin{equation*}
      \begin{tikzcd}
	       \Shapeetprofin(Z_{\et} \orientedtimes_{X_{\et}} U_{\et}) \arrow[d] \arrow[r] & \Shapeetprofin(U_{\et}) \arrow[d] \\ 
	       \Shapeetprofin(Z_{\et}) \arrow[r] & \Shapeetprofin(X_{\et})
      \end{tikzcd}
    \end{equation*}
	is a pushout square in the \category of profinite spaces.
\end{prp}

\begin{nul}
	\Cref{prp:etalevanKampen} functions in the same manner as Friedlander's van Kampen Theorem \cite[Proposition 15.6]{MR676809}.

	Cox \cites{MR512271I}{MR512271II} also developed a deleted tubular neighborhood for schemes, which is what appears in Friedlander's formulation of the van Kampen Theorem.
	It can be shown that Cox's deleted tubular neighborhood and our toposic version have equivalent protruncated shapes types.
\end{nul}

\newpage

\section[Extending exodromy: coefficients, stacks, \& \texorpdfstring{$\el$}{ℓ}-adic sheaves]{Extending exodromy: coefficients, stacks, \& \texorpdfstring{$\el$}{ℓ}-adic \\ sheaves}\label{sect:extending}

Let $X$ be a coherent scheme.
One limitation of our study of constructible sheaves thus far is that we have effectively restricted our attention to constructible sheaves on $X$ with \textit{nonabelian, \pifinite} coefficients.
The first goal of this chapter is to use \Cref{thm:exodromyforstrattopoi} along with some basic facts about the compactness of certain \categories to extend our Exodromy Theorem to coefficients in a \textit{finite} ring.
Specifically, for any \textit{finite} ring $R$, we provide an equivalence between the constructible derived \category of étale sheaves of $ R $-modules on $ X $ and the \category of continuous representations of the profinite $ 1 $-category $ \Gal(X) $ with values in $ \Perf(R) $:

\begin{thm}\label{thm:exodromyfinitering}
	Let $ X $ be a coherent scheme and let $ R $ be a finite ring.
	Then there is an equivalence of \categories
	\begin{equation}\label{eq:exodromyfinitering}
		\Dcons(X_{\et};R) \simeq \Fun(\Gal(X),\Perf(R)) \period
	\end{equation}
\end{thm}

\noindent See \Cref{thm:stableexodromyforstrattopoi} for an even more general statement.

However, many of the serious applications of constructible sheaves and their cohomology arise with coefficients a more general \textit{topological} ring, such as $\ZZell$, $\QQell$, or $\QQellbar$.
Consequently, it is desirable to extend our Exodromy Theorem to include these topological rings.
To do this, we have to solve two problems:

\begin{enumerate}[(1)]
	\item Given a topological ring $ \Lambda $, it is not \textit{a priori} clear how to incorporate the topology on $\Lambda$ in the \category $\Perf(\Lambda)$.
	We address this by introducing \textit{pyknotic} structures -- called \textit{condensed} structures by Clausen and Scholze.
	Pyknotic structures generalize common topological structures in a way that plays well with algebraic structures.
	Using this formalism, if $C$ and $D$ are pyknotic \categories, we can meaningfully speak of the the \category $\Functs(C,D)$ of \textit{continuous functors} $C \to D$.
	Critically, this extends gives us a way of extending the right-hand side of the equivalence \eqref{eq:exodromyfinitering} to topological rings; see \Cref{cor:FunproisFuncts}.

	\item Constructible sheaves of $\Lambda$-modules are not generally sheaves for the étale topology, but the \textit{proétale} topology of Bhatt and Scholze \cite{MR3379634}.
	One expects a version of our Exodromy Theorem for \topoi that closely resemble the proétale \topos of a scheme; we have not been able to obtain such a result.
	Instead, we do something more modest: we identify situations in which the theorem can be extended along limits and filtered colimits.
\end{enumerate}

The formulation and proof of the following pair of results occupies most of this chapter. 

\begin{ntn}
	Let $ X $ be a coherent scheme.
	\begin{itemize}
		\item Let $ \Lambda $ be a noetherian ring that is complete with respect to the topology defined by an ideal $ I \subset \Lambda $.
		We write $ \Dcons(X_{\proet};\Lambda) $\index[notation]{DconsproetLambda@$ \Dcons(X_{\proet};\Lambda) $} for the constructible derived \category of proétale sheaves of $ \Lambda $-modules on $ X $ \cite[Definition 6.5.1]{MR3379634}.

		\item Let $ \el $ be a prime number and $ E $ an algebraic extension of $ \QQell $.
		If $ X $ is topologically noetherian (thus automatically coherent), we write $ \Dcons(X_{\proet};E) $\index[notation]{DconsproetE@$ \Dcons(X_{\proet};E) $} for the constructible derived \category of proétale sheaves of $ E $-modules on $ X $ \cite[Definition 6.8.8]{MR3379634}.
	\end{itemize}
\end{ntn}

\begin{thm}\label{thm:exodromyZell}
	Let $ X $ be a coherent scheme, $\Lambda$ be a noetherian ring, and $ I \subset \Lambda $ an ideal.
	Assume that $ \Lambda $ is complete with respect to the $I$-adic topology and that for each integer $ n \geq 1 $, the quotient ring $ \Lambda/I^n $ is finite.
	Then there is an equivalence of \categories
	\begin{equation*}
		\Dcons(X_{\proet};\Lambda) \simeq \Functs(\Gal(X),\Perf(\Lambda)) \period
	\end{equation*}
\end{thm}

For topologically noetherian schemes, we extend \Cref{thm:exodromyZell} to coefficients in \smash{$ \QQell $} or \smash{$ \QQellbar $}:

\begin{thm}\label{thm:exodromyQellbar}
	Let $ X $ be a topologically noetherian scheme, $ \el $ be a prime number, and $ E $ be an algebraic field extension of $\QQell$.
	Then there is an equivalence of \categories
	\begin{equation*}
		\Dcons(X_{\proet};E) \simeq \Functs(\Gal(X),\Perf(E)) \period
	\end{equation*}
\end{thm}

\begin{rmk}
	We need $ X $ to be topologically noetherian in order to ensure that the standard notion of a constructible complex of \smash{$ \QQell $}-sheaves is equivalent to the requirement that the sheaf is lisse over a finte stratification (see \cite[\S6.6]{MR3379634}).
\end{rmk}

At end end of the chapter, we also explain how to extend these Exodromy Equivalences to a large class of stacks.

\Cref{subsec:catobjectsoftopoi} sets up the background about category objects in \topoi that we need to prove our extensions of the Exodromy Equivalence.
\Cref{subsec:exodromydiscrete} uses the compactness results of category objects that we prove in \cref{subsec:catobjectsoftopoi} to prove \Cref{thm:exodromyfinitering}.
\Cref{subsec:pyknoticspaces} sets up the basics of pyknotic spaces and \categories and explains what we mean by `$ \Functs $'.
Armed with this, in \Cref{subsec:profinitepyknotic,subsec:categoryobjs} we explain how to embed profinite spaces into pyknotic spaces, and profinite stratified spaces into pyknotic \categories.
\Cref{subsec:functorcomparison} shows that functors out of $ \Gal(X) $ in the `pro' sense agree with continuous functors in the pyknotic sense.
This gives a reinterpretation of the Exodromy Theorem with finite coefficients in terms of the pyknotic formalism.
With all of this in place, in \Cref{subsec:profiniteexodromy} we explain how the Exodromy Theorem with profinite coefficients (\Cref{thm:exodromyZell}) follows from Exodromy with finite coefficients (see \Cref{thm:exodromyforOEcoefficients}).
\Cref{subsec:compactnessofpyknoticcats} is dedicated to extending Exodromy with profinite coefficients to coefficients in an algebraic extension of $ \QQell $ (\Cref{thm:exodromyQellbar}); this step is not entirely formal and involves an analysis of \tstructures in the pyknotic world (see \Cref{thm:Qellexodromy}).
\Cref{subsec:exodromyforstacks} ends the chapter by extending the Exodromy Theorem to a large class of stacks (\Cref{prop:GaldeltaConstr}).


\subsection{Category objects of higher topoi}\label{subsec:catobjectsoftopoi}

In preparation for the results of this chapter, it is necessary to develop a little background on the (co)compactness of category objects in \topoi.
First, the cocompactness results that we prove in this section are what allow us to extend the Exodromy Theorem to coefficients in a finite discrete ring (\Cref{thm:exodromyfinitering}).
Second, the compactness results that we prove are what allow us to regard $ \Gal(X) $ as a pyknotic \category and extend the coefficients in the Exodromy Theorem from a finite ring to more general topological rings.

In order to prove the relevant (co)compactness results, it is technically convenient to use the formalism of \textit{complete Segal objects}.
The reason for this is that we can often formulate criteria for (co)compactness of simplicial objects in terms of the (co)compactness of their simplices; see \Cref{prop:Catcocompactness,prop:Catcompactness}. 

\begin{dfn}\label{def:COCSS}
	Let $ D $ be \acategory with finite limits.
	We say that a simplicial object $ F \colon \fromto{\Deltaop}{D} $ is a \defn{category object}\index[terminology]{category object} of $ D $ if $ F $ satisfies the \defn{Segal condition}\index[terminology]{Segal condition}:
	\begin{enumerate}[(\ref*{def:COCSS}.1)]
		\item For every integer $ k \geq 1 $, the natural morphism
		\begin{equation*}
			\fromto{F_k}{F\{0 < 1\} \crosslimits_{F\{1\}} F\{1 < 2\} \crosslimits_{F\{2\}} \cdots \crosslimits_{F\{k-1\}} F\{k-1 < k\}}
		\end{equation*}
		is an equivalence in $ D $.
	\end{enumerate}
	We say that a category object $ F \colon \fromto{\Deltaop}{D} $ is a \defn{complete Segal object}\index[terminology]{complete Segal object} if, in addition, $ F $ satisfies the following \defn{completeness condition}\index[terminology]{completeness condition}:
	\begin{enumerate}[(\ref*{def:COCSS}.1)]
		\setcounter{enumi}{1}
		\item The natural morphism
		\begin{equation*}
			\fromto{F_0}{F_3 \crosslimits_{F\{0 < 2\} \cross F\{1 < 3\}} F_1}
		\end{equation*}
		is an equivalence in $ D $.
	\end{enumerate}

	Write $ \CO(D) \subset \Fun(\Deltaop,D) $\index[notation]{CO@$ \CO(D) $} for the full subcategory spanned by the category objects and $ \CS(D) \subset \CO(D) $\index[notation]{CS@$ \CS(D) $} for the full subcategory spanned by the complete Segal objects.
\end{dfn}

\begin{nul}\label{nul:JoyalTiereny}
	Joyal and Tierney \cite{MR2342834} showed that the nerve construction defines an equivalence 
	\begin{equation*}
		\Nerve \colon \equivto{\Cat_{\infty}}{\CS(\Space)}
	\end{equation*}
	from the \category of \categories to the \category of complete Segal spaces.
	For each integer $ n \geq 0 $, the nerve restricts to an equivalence 
	\begin{equation*}
		\Nerve \colon \equivto{\Cat_{n}}{\CS(\Space_{\leq n-1})}
	\end{equation*}
	between $ n $-categories and complete Segal objects in $ (n-1) $-truncated spaces.
\end{nul}

\begin{nul}
	See \cite[\SAGsubsec{A.8.2}]{SAG} for more on category objects.
\end{nul}

\begin{nul}
	Let $ D $ be \acategory with finite limits.
	Then the full subcategories
	\begin{equation*}
		\CS(D) \subset \CO(D) \subset \Fun(\Deltaop,D)
	\end{equation*}
	are closed under all limits that exist in $ D $.
\end{nul}

\begin{nul}
	Let $ \XX $ be \topos.
	Since finite limits commute with filtered colimits in \topoi, the full subcategories
	\begin{equation*}
		\CS(\XX) \subset \CO(\XX) \subset \Fun(\Deltaop,\XX)
	\end{equation*}
	are closed under limits and filtered colimits. 
	In particular $ \CS(\XX) $ and $ \CO(\XX) $ are presentable and the incluions into $ \Fun(\Deltaop,\XX) $ admit left adjoints.
\end{nul}

For future use, we'll record a few facts about the interaction between \categories of product-preserving functors and complete Segal objects.
We later use this to see that formation of pyknotic objects and complete Segal objects commute.
All are immediate from the definitions.

\begin{ntn}
	Let $ B $ and $ D $ be \categories with finite products.
	We write
	\begin{equation*}
		\Funcross(B,D) \subset \Fun(B,D)
	\end{equation*}
	for the full subcategory spanned by the functors that preserve finite products.
\end{ntn}

\begin{lem}\label{lem:prodFuncommuteswithFun}
	Let $ B $, $ C $, and $ D $ be \categories, and assume that $ B $ and $ D $ have finite products.
	Then the natural equivalence of \categories
	\begin{equation*}
		\Fun(B,\Fun(C,D)) \equivalent \Fun(C,\Fun(B,D))
	\end{equation*}
	restricts to an equivalence
	\begin{equation*}
		\Funcross(B,\Fun(C,D)) \equivalent \Fun(C,\Fun^{\times}(B,D)) \period
	\end{equation*}
\end{lem}

\begin{cor}\label{lem:pykCS}
	Let $ B $ be \acategory with products and $ D $ \acategory with finite limits.
	Then the natural equivalence of \categories 
	\begin{equation*}
		\Funcross(B,\Fun(\Deltaop,D)) \equivalent \Fun(\Deltaop,\Fun^{\times}(B,D))
	\end{equation*}
	restrict to equivalences
	\begin{equation*}
		\Funcross(B,\CO(D)) \equivalent \CO(\Funcross(B,D)) \andeq \Funcross(B,\CS(D)) \equivalent \CS(\Funcross(B,D)) \period
	\end{equation*}
\end{cor}

We now prove the relevant cocompactness result that we use in \cref{subsec:categoryobjs} in order to see that profinite stratified spaces embed into pyknotic \categories.
The main result we need is that category objects in \pifinite spaces are cocompact as category objects in profinite spaces.
We show this by combining the fact that $ \Spacefin \subset \Spaceprofin $ consists of cocompact objects with the fact that a category object in truncated spaces is always right Kan-extended from a finite subcategory of $ \Deltaop $.

\begin{ntn}
	Let $ m $ be a positive integer.
	We write $ \DDelta_{\leq m} \subset \DDelta $\index[notation]{Deltam@$ \DDelta_{\leq m} $} for the full subcategory spanned by the objects $ [0], [1], \ldots, [m] $.
\end{ntn}

\begin{dfn}\label{ntn:COtruc}
	Let $ D $ be \acategory with finite limits and $ m \geq 1 $ an integer.
	An \defn{$ m $-skeletal category object}\index[terminology]{category object!skeletal@category object!$m$-skeletal} of $ D $ is a functor $ F \colon \fromto{\Deltaop_{\leq m}}{D} $ satisfying the \defn{Segal condition}\index[terminology]{Segal condition}: for every integer $ 0 \leq k \leq m $, the natural morphism
	\begin{equation*}
		\fromto{F_k}{F\{0 < 1\} \crosslimits_{F\{1\}} F\{1 < 2\} \crosslimits_{F\{2\}} \cdots \crosslimits_{F\{k-1\}} F\{k-1 < k\}}
	\end{equation*}
	is an equivalence in $ D $.
	We write 
	\begin{equation*}
		\CO_{\leq m}(D) \subset \Fun(\Deltaop_{\leq m},D) \index[notation]{COm@$ \CO_{\leq m}(D) $}
	\end{equation*}	
	for the full subcategory spanned by the $ m $-skeletal category objects.
\end{dfn}

\begin{nul}
	Let $ \XX $ be \atopos and $ m \geq 1 $ an integer.
	Then the restriction functor
	\begin{equation*}
		\restrict{(-)}{\Deltaop_{\leq m}} \colon \fromto{\CO(\XX)}{\CO_{\leq m}(\XX)}
	\end{equation*}
	commutes with all limits and filtered colimits.
\end{nul}

\begin{prp}\label{prop:Catcocompactness}
	Let $ \XX $ be \acategory with finite limits, $ n \geq -2 $ an integer, and $ C \in \CO(\XX) $ a category object of $ \XX $.
	If the natural map
	\begin{equation*}
		\fromto{C\{0<1\}}{C\{0\} \cross C\{1\}}
	\end{equation*}
	is $ n $-truncated and for each $ I \in \DDelta $, the object $ C(I) $ of $ \XX $ is cocompact, then $ C $ is a cocompact object of $ \CO(\XX) $. 
\end{prp}

\begin{proof}[Proof of \Cref{prop:Catcompactness}]
	Let $ D \colon \fromto{A}{\CO(\XX)} $ be a cofiltered diagram.
	By \SAG{Proposition}{A.8.2.6}, the object $ C $ is right Kan extended from $ \Deltaop_{\leq n+2} $, hence we see that
	\begin{equation*}
		\colim_{\alpha \in A^{\op}} \Map_{\CO(\XX)}(D_{\alpha},C) \equivalent \colim_{\alpha \in A^{\op}} \Map_{\CO_{\leq n+2}(\XX)}(\restrict{D_{\alpha}}{\Deltaop_{\leq n+2}}, \restrict{C}{\Deltaop_{\leq n+2}}) \period
	\end{equation*}
	From the end description of mapping spaces in a functor category, the fact that finite limits commute with filtered colimits in $ \Space $, and the cocompactness of $ C(I) $ for each $ I \in \Deltaop $, we see that we have equivalences
	\begin{align*}
		\colim_{\alpha \in A^{\op}} \Map_{\CO(\XX)}(D_{\alpha},C) &\equivalent \colim_{\alpha \in A^{\op}} \int_{I \in \DDelta_{\leq n+2}} \Map_{\XX}(D_{\alpha}(I),C(I)) \\ 
		&\equivalence \int_{I \in \DDelta_{\leq n+2}} \colim_{\alpha \in A^{\op}} \Map_{\XX}(D_{\alpha}(I),C(I)) \\ 
		&\equivalence \int_{I \in \DDelta_{\leq n+2}} \Map_{\XX}(\textstyle \lim_{\alpha \in A^{\op}} D_{\alpha}(I),C(I)) \\
		&\equivalent \Map_{\CO(\XX)}(\textstyle \lim_{\alpha \in A^{\op}} D_{\alpha}, C) \period \qedhere
	\end{align*} 
\end{proof}	

\begin{cor}\label{cor:COspacefiniscocompact}
	Every object in the image of the Yoneda embedding $ \incto{\CO(\Spacefin)}{\CO(\Spaceprofin)} $ is cocompact.
\end{cor}

We now turn to the compactness of category objects.
The formal manipulations in the proof of the following proposition are very similar to those used in the proof of \Cref{prop:Catcocompactness}.

\begin{prp}\label{prop:Catcompactness}
	Let $ \XX $ be \atopos and $ C \in \CO(\XX) $ a category object of $ \XX $.
	If for each $ I \in \DDelta $, the object $ C(I) $ is almost compact (\Cref{def:almostcompactness}), then for each integer $ n \geq -2 $ the functor
	\begin{equation*}
		\Map_{\CO(\XX)}(C,-) \colon \fromto{\CO(\XX_{\leq n})}{\Space}
	\end{equation*}
	preserves filtered colimits.
\end{prp}	

\begin{nul}\label{nul:compactnesseasyhypotheses}
	The hypotheses of \Cref{prop:Catcompactness} are satisfied if $ \XX $ is a locally coherent \topos and for each $ I \in \DDelta $, the object $ C(I) $ is coherent \SAG{Corollary}{A.2.3.2}.
\end{nul}

\begin{proof}[Proof of \Cref{prop:Catcompactness}]
	Let $ D \colon \fromto{A}{\CO(\XX_{\leq n})} $ be a filtered diagram.
	First note that by \SAG{Proposition}{A.8.2.6}, for each $ \alpha \in A $, the category object $ D_{\alpha} $ is right Kan-extended from $ \Deltaop_{\leq n+2} $.
	Moreover, since $ \XX_{\leq n} \subset \XX $ is closed under filtered colimits, the colimit $ \colim_{\alpha \in A} D_{\alpha} $ is also right Kan extended from $ \Deltaop_{\leq n+2} $. 
	Thus we see that
	\begin{equation*}
		\colim_{\alpha \in A} \Map_{\CO(\XX)}(C,D_{\alpha}) \equivalent \colim_{\alpha \in A} \Map_{\CO_{\leq n+2}(\XX)}(\restrict{C}{\Deltaop_{\leq n+2}},\restrict{D_{\alpha}}{\Deltaop_{\leq n+2}})
	\end{equation*}

	From the end description of mapping spaces in a functor category and the fact that finite limits commute with filtered colimits in $ \Space $, we see that we have equivalences
	\begin{align*}
		\colim_{\alpha \in A} \Map_{\CO(\XX)}(C,D_{\alpha}) &\equivalent \colim_{\alpha \in A} \int_{I \in \DDelta_{\leq n+2}} \Map_{\XX}(C(I),D_{\alpha}(I)) \\ 
		&\equivalence \int_{I \in \DDelta_{\leq n+2}} \colim_{\alpha \in A} \Map_{\XX}(C(I),D_{\alpha}(I))
	\end{align*}
	Since the object $ C(I) $ is almost compact for each $ I \in \DDelta $ and $ \colim_{\alpha \in A} D_{\alpha} $ is right Kan extended from $ \Deltaop_{\leq n+2} $, we see that
	\begin{align*}
		\colim_{\alpha \in A} \Map_{\CO(\XX)}(C,D_{\alpha})  &\equivalence \int_{I \in \DDelta_{\leq n+2}} \Map_{\XX}(C(I),\textstyle \colim_{\alpha \in A} D_{\alpha}(I)) \\
		&\equivalent \Map_{\CO(\XX)}(C,\textstyle \colim_{\alpha \in A} D_{\alpha}) \period \qedhere
	\end{align*} 
\end{proof}	


\subsection{Exodromy with discrete coefficients}\label{subsec:exodromydiscrete}

Let $ X $ be a coherent scheme.
The original formulation of the Exodromy Theorem for schemes says that functors $ \Gal(X) \to \Spacefin $ are the same things as constructible sheaves of spaces on $ X $ (\Cref{cnstr:idPiinfty1withGal}).
We now prove the analagous claim where $ \Spacefin $ is replaced by the \category of perfect complexes over a finite ring $ R $.
The important property shared by $ \Spacefin $ and $ \Perf(R) $ that allows us to reduce the claim to \Cref{prop:Catcocompactness} is that all of the mapping spaces in these \categories are \pifinite. 

\begin{dfn}
	We say that \acategory $ C $ is \defn{locally \pifinite}\index[terminology]{category@\category!locally \pifinite}\index[terminology]{locally \pifinite!\category} if for all objects $ X,Y \in C $, the mapping space $ \Map_C(X,Y) $ is \pifinite.
	We say that a locally \pifinite \category $ C $ is \defn{\pifinite}\index[terminology]{category@\category!\pifinite}\index[terminology]{pifinite@\pifinite!\category} if $ C $ has finitely many objects up to equivalence.
\end{dfn}

\begin{exms}
	\hfill
	\begin{itemize}
		\item The \category $ \Spacefin $ of \pifinite spaces is locally \pifinite \SAG{Remark}{E.2.6.4}.

		\item For any finite ring $ R $, the \category $ \Perf(R) $ of perfect complexes over $ R $ is locally \pifinite.

		\item If $ \Pi $ is a \pifinite stratified space, then the \category $ \Pi $ is \pifinite
	\end{itemize}
\end{exms}

\begin{ntn}\label{ntn:subpi}
	Let $ C $ be a locally \pifinite \category.
	We denote by $ \Subpi(C) $\index[notation]{Subpi@$ \Subpi(C) $} the filtered poset of \pifinite full subcategories of $ C $.
	Note that $ C $ is the filtered union
	\begin{equation*}
		C = \colim_{C_0 \in \Subpi(C)} C_0 \period
	\end{equation*}
\end{ntn}

\begin{nul}\label{nul:canassumelocallypifiniteispifinite}
	Note that if $\mbfPi $ is a profinite stratified space and $ C $ is a locally \pifinite \category, then any functor $ \mbfPi \to C $ lands in a \pifinite full subcategory of $ C $.
	This is because the profinite set $\uppi_0(\interior \mbfPi) $ is quasicompact, so its image in the discrete set $ \uppi_0(\interior C) $ is finite.
	In particular, the natural functor 
	\begin{equation*}
		\colim_{C_0 \in \Subpi(C)} \Fun(\mbfPi,C_0) \to \Fun(\mbfPi,C)
	\end{equation*}
	is an equivalence.
\end{nul}

We now introduce the notion of a constructible sheaf with values in a general \category.

\begin{rec}
	Let $ \YY $ be \atopos and $ D $ a presentable \category.
	A \defn{$ D $-valued sheaf}\index[terminology]{sheaf} on $ \YY $ is a limit-preserving functor $ \fromto{\YY^{\op}}{D} $.
	We write $ \Sh(\YY;D) \subset \Fun(\YY^{\op},D) $ for the full subcategory spanned by the $ D $-valued sheaves.
	Note that $ \Sh(\YY;D) $ is equivalent to the tensor product of presentable \categories $ \YY \tensor D $ \cite[\SAGsubsec{1.3.1}]{SAG}.
\end{rec}

\begin{dfn}\label{def:consgencoeff}
	Let $\YY$ be \atopos, and let $ C $ be a $ \updelta_0 $-small \category.
	Then a sheaf $F$ on $\YY$ with values in $ \PSh(C) $ is said to be a \defn{lisse sheaf valued in $C$}\index[terminology]{sheaf!lisse}\index[terminology]{lisse} if and only if there exist a finite covering $ \coprod_{i\in I} U_i \surjection 1_{\YY} $, a collection $\{K_i\}_{i\in I}$ of objects of $ C $, and equivalences between $ \restrict{F}{U_i} $ and the constant sheaf on $\YY_{/U_i}$ at $K_i$.
	
	Let $ S $ be a spectral topological space, and let $ \XX $ be an $S$-stratified \topos.
	A sheaf $F$ on $\XX $ valued in $\PSh(C)$ is said to be a \defn{constructible sheaf valued in $ C $}\index[terminology]{sheaf!constructible}\index[terminology]{constructible} if and only if there exists a finite constructible stratification $ S \to P $ such that for each element $ p \in P$, the restriction $ \restrict{F}{\XX_p} $ is a lisse sheaf valued in $C$.
	We write
	\begin{equation*}
		\Cons^S(\XX;C) \subset \Sh(\XX; \PSh(C)) \index[notation]{cons@$\Cons^S(\XX;C)$}
	\end{equation*}
	for the full subcategory spanned by the constructible sheaves valued in $ C $.
\end{dfn}

\begin{nul}
	By definition, the \category $ \Cons^S(\XX;C) $ is the filtered colimit
	\begin{equation*}
		\Cons^S(\XX; C) \simeq \colim_{P \in \FC(S)} \Cons^P(\XX;C)
	\end{equation*}
	over the finite constructible stratifications $ S \to P $.
	
	Let $P$ be a finite poset, $ Z \subseteq P $ a closed subset, and $ U \colonequals P \smallsetminus Z $ the open complement of $ Z $.
	Then $\Cons^P(\XX;C)$ is a recollement of $\Cons^Z(\XX_Z;C)$ and $\Cons^U(\XX_U;C)$:
	\begin{equation*}
		\Cons^P(\XX;C) \simeq \commacatdisplay{\Cons^Z(\XX_Z;C)}{\Cons^Z(\XX_Z;C)}{\Cons^U(\XX_U;C)} \period
	\end{equation*}
\end{nul}

\begin{exm}
	Let $ S $ be a spectral topological space, and let $ \XX $ be an $S$-stratified \topos.
	Using the identification of the \category $ \Space $ of spaces as sheaves on $ \Spacefin $ with respect to the effective epimorphism topology, we see that we have a natural identification
	\begin{equation*}
		\XX^{\Scons} = \Cons^{S}(\XX;\Spacefin) \period
	\end{equation*}
\end{exm}

\begin{exm}
	Let $ X $ be a coherent scheme and $ R $ a finite ring.
	Note that the derived \category $ \Dup(R) $ is compactly generated by $ \Perf(R) $ and $ \Dup(R) \subset \PSh(\Perf(R)) $ is closed under limits.
	Hence we have a natural identification
	\begin{equation*}
		\Dcons(X_{\et};R) = \Cons^{X^{\zar}}(X_{\et};\Perf(R))
	\end{equation*}
	of the constructible derived \category of étale sheaves of $ R $-modules on $ X $ with the \category of constructible sheaves valued in $ \Perf(R) $ introduced in \Cref{def:consgencoeff}.
\end{exm}

Now we turn to extending the Exodromy Theorem from constructible sheaves valued in \pifinite spaces, to constructible sheaves valued in any locally \pifinite \category.

\begin{lem}\label{lem:canassumelocallypifiniteispifinite}
	Let $ C $ be a $\updelta_0$-small, locally \pifinite \category.
	Let $S$ be a spectral topological space, and let $\XX $ be an $S$-stratified \topos.
	Then the natural functor
	\begin{equation*}
		\colim_{C_0 \in \Subpi(C)} \Cons^{S}(\XX; C_0) \to \Cons^{S}(\XX;C)
	\end{equation*}
	is an equivalence.
\end{lem}

\begin{proof}
	Since both source and target are filtered colimits over finite constructible stratifications of $S$, we are free to assume that $S=P$ is a finite poset.
	By inducting on the rank of $P$ and using the recollement decomposition of the source and the target, we may also assume that $P$ is the trivial poset.
	Now if $F$ is a lisse sheaf in $ C $ on $\XX$, then select a finite covering $ \coprod_{i\in I} U_i \surjection 1_{\YY} $, a collection $\{K_i\}_{i\in I}$ of objects of $ C $, and equivalences between $ \restrict{F}{U_i} $ and the constant sheaf on $\XX_{/U_i}$ at $K_i$.
	If $C_0\subseteq C$ denotes the full subcategory spanned by the objects $K_i$, then the sheaf takes values in $\PSh(C_0)$ (embedded in $\PSh(C)$ via right Kan extension).
\end{proof}

\begin{thm}\label{thm:stableexodromyforstrattopoi}
	Let $ C $ be a $\updelta_0$-small, locally \pifinite category.
	Let $ S $ be a spectral topological space, and $ \XX $ an $ S $-stratified \topos.
	Then there is a natural equivalence
	\begin{equation*}
		\Fun(\StrShape^{S}(\XX),C) \simeq \Cons^S(\XX;C) \period
	\end{equation*}
\end{thm}

\begin{proof}
	In light of \Cref{nul:canassumelocallypifiniteispifinite} and \Cref{lem:canassumelocallypifiniteispifinite}, we may assume that $ C $ is \pifinite.
	
	Let $ \mbfPi \coloneq \StrShape^S(\XX) $.
	By \Cref{prop:Catcompactness} and the fact that $ C $ is \pifinite, it follows that for any integer $n\geq 0$, the natural functor
	\begin{equation*}
		\Fun(\mbfPi,\Fun(C^{\op},\Space_{\uppi,\leq n})) \to \Fun(C^{\op},\Fun(\mbfPi,\Space_{\uppi,\leq n}))
	\end{equation*}
	is an equivalence.
	Passing to colimits over $n$, we conclude that the natural functor
	\begin{equation*}
		\Fun(\mbfPi,\Fun(C^{\op},\Spacefin)) \to \Fun(C^{\op},\Fun(\mbfPi,\Spacefin))
	\end{equation*}
	is an equivalence as well.
	Applying $\Fun(C^{\op},-) $ to the Exodromy Equivalence of \Cref{thm:exodromyforstrattopoi} to obtain an equivalence
	\begin{equation*}
		\Fun(\mbfPi,\Fun(C^{\op},\Spacefin)) \simeq \Cons^S(\XX;\Fun(C^{\op},\Spacefin)) \period
	\end{equation*}
	Now we conclude by observing that the functors $ \StrShape^S(\XX) \to \Fun(C^{\op},\Spacefin) $ that land in the essential image of the Yoneda embedding correspond under this equivalence to the constructible sheaves valued in $ C $.
\end{proof}

\begin{cor}\label{cor:stableexodromyforschemes}
	Let $ X $ be a coherent scheme and $ R $ a finite ring.
	Then there is a natural equivalence
	\begin{equation*}
		\Dcons(X_{\et};R) \equivalent \Fun(\Gal(X),\Perf(R)) \period
	\end{equation*}
\end{cor}

\begin{nul}
	Attached to any constructible sheaf $F$ of $R$-complexes on $X$, we have an associated \defn{exodromy representation}
	\begin{equation*}
		\exrep_F \colon \Gal(X) \to \Perf(R)
	\end{equation*}
	that is sufficient to reconstruct $F$.
\end{nul}


\subsection{Pyknotic spaces \& pyknotic higher categories}\label{subsec:pyknoticspaces}

In order to upgrade the coefficients in \Cref{cor:stableexodromyforschemes} from a finite discrete ring to a more general topological ring, we make use of the pyknotic formalism.
In this section, we briefly describe the describe elements of the pyknotic formalism that we need to extend our Exodromy Theorem to $ \el $-adic sheaves.
For more details on the pyknotic formalism, we refer the reader to \cite{pyknoticI,Scholze:Condensedtalk,Scholze:condensedtalknotes,Scholze:condensednotes}.

\begin{cnstr}
	Stone duality identifies the category $\Stn$\index[notation]{Stn@$\Stn$} of Stone topological spaces with the category \smash{$\Pro(\Setfin)$} of profinite sets.
	The subcategory $E \subseteq \Stn$ consisting of effective epimorphisms is an $1$-presite structure on $\Stn$.
	We write $\eff \coloneq \tau_E$\index[notation]{eff@$\eff$} for the resulting finitary topology, the \defn{effective epimorphism topology}\index[terminology]{effective epimorphism topology}\index[terminology]{topology!effective epimorphism}.
\end{cnstr}

\begin{dfn}
	A \defn{pyknotic space}\index[terminology]{pyknotic space} is a hypersheaf $\Stn^{\op} \to \Space_{\updelta_1}$ for the effective epimorphism topology.
	We write
	\begin{equation*}
		\Pyk(\Space) \coloneq \Sheffhyp{\Stn;\Space_{\updelta_1}} \index[notation]{PykS@$\Pyk(\Space)$} 
	\end{equation*}
	for the \category of pyknotic spaces.
\end{dfn}

\begin{wrn}
	The category $\Stn$ is $\updelta_1$-small, but it is not $\updelta_0$-small.
	Moreover, there does \textit{not} exist a cofinal $\updelta_0$-small set of covering sieves of an object for the effective epimorphism topology.
	Consequently, when we speak of hypersheaves on $\Stn$ for the effective epimorphism topology, we have to consider hypersheaves valued in $\updelta_1$-small spaces.
	The result will be a \defn{large \topos}, that is, a left exact $\updelta_1$-accessible localization of \acategory $\Fun(C^{\op},\Space_{\updelta_1})$ of presheaves of $\updelta_1$-small spaces on a $\updelta_1$-small \category $C$.

	Large \topoi work exactly as do \topoi, except that everything has to be shifted one universe up.
	For example, if $\XX$ is bounded and coherent as a large \topos, then $\XXcohbdd$ is only $\updelta_1$-small.

	An alternative to working with large \topoi is considering instead only the accessible hypersheaves $\Stn^{\op} \to \Space_{\updelta_0}$.
	These are what Clausen and Scholze call \defn{condensed spaces}.
	These do not form a \topos (large or small), but in many ways the \category of accessible sheaves is relatively well behaved.
\end{wrn}

\begin{nul}
	The large \topos $\Pyk(\Space)$ is hypercomplete, coherent, and locally coherent \cite[Propositions \SAGthmlink{A.2.2.2} \& \SAGthmlink{A.3.1.3}]{SAG}.
\end{nul}

Let us identify two further generating sites for $\Pyk(\Space)$ -- one larger and one smaller.

\begin{ntn}
	Define full subcategories
	\begin{equation*}
		\Proj \subset \Comp \subset \TSpc
	\end{equation*}
	as follows:
	\begin{itemize}
		\item $ \Comp $\index[notation]{Comp@$\Comp$}  is spanned by the \defn{compacta} -- i.e., compact Hausdorff topological spaces.

		\item $ \Proj $\index[notation]{Proj@$\Proj$} is spanned by the \defn{projective compacta} -- i.e., compact Hausdorff topological spaces that are extremally disconnected \cites{MR0121775}[Chapter III, \S3.7]{MR861951}.
	\end{itemize}

	Equivalently, a topological space is a projective compactum if and only if it can be exhibited as the retract of the Stone--Čech compactification $\SC(S)$ of some set $S$.
	In particular, $ \Proj \subset \Stn $.
\end{ntn}

\begin{nul}\label{nul:descriptionsofPykS}
	For every compactum $ K $, there is a natural surjection $ \surjto{\SC(K^{\disc})}{K} $ from the Stone--Čech compactification of the discrete topological space $ K^{\disc} $ with underlying set $ K $ to $ K $ (cf. \cite[Remark  2.8]{Scholze:diamonds}).
	Hence the subcategories $ \Stn \subset \Comp $ and $ \Proj \subset \Comp $ are bases for the effective epimorphism topology on $ \Comp $ (\Cref{def:basis}).
	Therefore, restriction of presheaves defines equivalences of \categories
	\begin{equation*}
		\Sheffhyp{\Comp;\Space_{\updelta_1}} \equivalence \Sheffhyp{\Stn;\Space_{\updelta_1}} \equivalence \Sheffhyp{\Proj;\Space_{\updelta_1}}
	\end{equation*}
	with inverses given by right Kan extension (\Cref{prop:hyperbasis}).
\end{nul}

\begin{wrn}\label{wrn:PykS}
	Since the $ 1 $-sites $ \Comp $ and $ \Stn $ have finite limits and the inclusion $ \incto{\Stn}{\Comp} $ preserves finite limits, \Cref{cor:n-localicbasis} shows that restriction defines an equivalence of $ 1 $-localic topoi
	\begin{equation*}
		\equivto{\Sheff{\Comp;\Space_{\updelta_1}}}{\Sheff{\Stn;\Space_{\updelta_1}}} \period
	\end{equation*}

	However, as pointed out to us by Clausen and Scholze, since the $ 1 $-site $ \Proj $ of projective compacta does not have finite limits, and restriction only defines an equivalence
	\begin{equation*}
		\equivto{\Sheffhyp{\Comp;\Space_{\updelta_1}}}{\Sheffhyp{\Proj;\Space_{\updelta_1}}}
	\end{equation*}
	on topoi of \textit{hypersheaves}.
\end{wrn}

\begin{nul}\label{nul:PykSasnonablianderived}
	Since $ \Proj \subset \Comp $ consists of projective objects of $ \Comp $, the \v{C}ech nerve of any surjection in $ \Proj $ is a split simplicial object.
	Hence by \SAG{Proposition}{A.3.3.1} we see that a functor
	\begin{equation*}
		F \colon \fromto{\Projop}{\Space_{\updelta_1}} 
	\end{equation*}
	is a sheaf with respect to the effective epimorphism topology if and only if $ F $ carries coproducts in $ \Proj $ to products in $ \Space $.
	That is to say, the category $ \Pyk(\Space) $ is equivalent to the \category of functors $ \Projop \to \Space_{\updelta_1} $ that carry finite coproducts of projective compacta to products.

	From this description, it is essentially immediate that the \topos $ \Sheff{\Proj;\Space_{\updelta_1}} $ is Postnikov complete, whence we obtain an equivalence
	\begin{equation*}
		\Sheffhyp{\Comp;\Space_{\updelta_1}} \simeq \Sheff{\Proj;\Space_{\updelta_1}} 
	\end{equation*}
	(cf. \cite[\S2.4]{pyknoticI}).
	That is to say, the \category $ \Pyk(\Space) $ is the \textit{nonabelian derived \category}\index[terminology]{nonabelian derived category!nonabelian derived \category} or \textit{animation}\index[terminology]{animation} of the category $ \Proj $ \cites[\HTTsubsec{5.5.8}]{HTT}[\S5.1]{CesnaviciusScholze}
\end{nul}

This last description of the \category of pyknotic spaces lets us define pyknotic objects in any \category with finite products.

\begin{dfn}
	Let $C$ be a \category with all finite products.
	The \category $\Pyk(C)$\index[notation]{PykC@$\Pyk(C)$} of \defn{pyknotic objects}\index[terminology]{pyknotic object} of $C$ is the full subcategory of $\Fun(\Projop,C)$ spanned by those functors that carry finite coproducts of projective compacta to products in $ C $.
\end{dfn}

\begin{cnstr}[discrete \& indiscrete objects]\label{nul:PykSislocal}
	Let $C$ be a $\updelta_1$-small presentable \category.
	The global sections functor $ \fromto{\Pyk(C)}{C} $ is given by evaluation at the one-point compactum $ \pt $.
	For any pyknotic object $Y$, we write $Y^{\und} \coloneq Y(\pt)$.
	We call $ Y^{\und} $ the \defn{underlying object} of $Y$.

	Left adjoint to this is the constant sheaf functor $ \fromto{C}{\Pyk(C)}$ that carries an object $X \in C$ to what we will call the \defn{discrete pyknotic object}\index[terminology]{pyknotic object!discrete}\index[terminology]{discrete pyknotic object} $X^{\disc}$\index[notation]{discrete@$ X^{\disc} $} attached to $X$.
	This pyknotic object can be described explicitly: $X^{\disc}$ is given by the assignment 
	\begin{equation*}
		K \mapsto X^K \coloneq \colim_{I \in \Setfin_{K/}} \prod_{i \in I} X \period
	\end{equation*}
	
	The underlying space functor also admits a \textit{right} adjoint $X \mapsto X^{\indisc}$\index[notation]{Xindisc@$ X^{\indisc} $}.
	For an object $ X \in C $, the sheaf $ X^{\indisc} \colon \fromto{\Projop}{C} $ is given by the assignment
	\begin{equation*}
		K \mapsto X^{|K|} \coloneq \prod_{k \in |K|} X \comma
	\end{equation*}
	i.e., the product of copies of $ X $ indexed by the underlying \textit{set} $ |K| $ of the topological space $ K $.
	We call $X^{\indisc}$ the \defn{indiscrete pyknotic object}\index[terminology]{pyknotic object!indiscrete}\index[terminology]{indiscrete pyknotic object} attached to $X$.

	Both the discrete and indiscrete functors are fully faithful, so that
	\begin{equation*}
		(X^{\disc})^{\und} = (X^{\indisc})^{\und} = X \period
	\end{equation*}
	Accordingly, we say that a pyknotic object in the essential image of $X \mapsto X^{\disc}$ is \defn{discrete}, and a pyknotic object in the essential image of $X \mapsto X^{\indisc}$ is \defn{indiscrete}.
\end{cnstr}

\begin{exm}\label{exm:pifinitestratspacesaspyknoticcategories}
	Let $ \Pi \to P $ be a \pifinite stratified space.
	Then the pyknotic \category $\Pi^{\disc}$ can be identified as above: $\Pi^{\disc}(K) \simeq \Pi^K$.
	Since profinite stratified spaces embed fully faithfully into \topoi (\Cref{prp:fullfaithfulnessoftilde} \& \Cref{nul:naturalstratasadjoint}), one obtains an equivalence
	\begin{equation*}
		\Pi^{\disc}(K) \simeq \Pi^K \simeq \Funlowerstar(\Ktilde, \Pitilde) \comma
	\end{equation*}
	natural in $K$ and $\Pi$.
\end{exm}

\begin{wrn}
	The center $ X \mapsto X^{\indisc} $ is \textit{not} the only point of the \topos $ \Pyk(\Space) $.
	Let $ T $ be a topological space.
	Define pyknotic set $ \Pup_T $ by sending $K$ to the quotient of the set continuous maps $K \to T$ by the locally constant maps:
	\begin{equation*}
		\Pup_T(K) \coloneq \Map_{\TSpc}(K, T)/\Map_{\TSpc}^{\lc}(K, T) \period
	\end{equation*}
	If $T$ is nonempty, then the pyknotic set $ \Pup_T$ has underlying set $\pt$;
	thus if $T$ is neither empty nor $\pt$, then $ \Pup_T$ is a nontrivial pyknotic structure on the point.
	See \stacks{0991}.
\end{wrn}

\begin{exm}
	For any finite set $J$, the discrete pyknotic set $J^{\disc}$ is the sheaf
	\begin{equation*}
		K \mapsto \Map_{\Proj}(J, K)
	\end{equation*}
	represented by $J$.
	If $\{J_{\alpha}\}_{\alpha\in\Lambda}$ is an inverse system of finite sets, then the limit
	\begin{equation*}
		\lim_{\alpha\in\Lambda} J_{\alpha}^{\disc}
	\end{equation*}
	is the sheaf represented by the Stone topological space $\lim_{\alpha\in\Lambda} J_{\alpha}$;
	this is not discrete.
	Cf. \Cref{nul:Yonedaisdiscrete}.
\end{exm}

\begin{exm}\label{exm:cgembed}
	Write $ \TSpc^{\cg} \subset \TSpc $ for the full subcategory spanned by the \textit{compactly generated} topological spaces.
	Then the functor $ \fromto{\TSpc^{\cg}}{\Pyk(\Set)} $ given by the assignment
	\begin{equation*}
		T \mapsto [K \mapsto \Map_{\TSpc}(K,T)]
	\end{equation*}
	is a fully faithful right adjoint.
	See \cites[Example 2.1.6]{pyknoticI}[Proposition 1.7]{Scholze:condensedtalknotes}.
\end{exm}

Now we may speak of pyknotic \emph{\categories}, which are nothing more than functors $ \Proj^{\op} \to \Cat_{\infty} $ that carry finite coproduts of projective compacta to the corresponding finite products of \categories.
We can also interpret pyknotic \categories as complete Segal objects in pyknotic spaces:

\begin{exm}\label{exm:PykCatasCSPykS}
	\Cref{lem:pykCS} provides equivalences
	\begin{equation*}
		\Pyk(\CO(\Space)) \equivalent \CO(\Pyk(\Space)) \andeq \Pyk(\Cat_{\infty}) \equivalent \CS(\Pyk(\Space)) \period
	\end{equation*}
	In light of \Cref{nul:JoyalTiereny}, for each integer $ n \geq 0 $, \Cref{lem:pykCS} provides an equivalence
	\begin{equation*}
		\Pyk(\Cat_{n}) \equivalent \CS(\Pyk(\Space_{\leq n-1})) \period
	\end{equation*} 
\end{exm}

\begin{dfn}\label{ntn:Functs}
	Let $C$ and $D$ be pyknotic \categories.
	The \category of \defn{continuous functors}\index[terminology]{continuous functor} $C \to D$ is the end
	\begin{equation*}
		\Functs(C, D) \coloneq \int_{K \in \Proj} \Fun(C(K), D(K)) \period \index[notation]{Functs@$\Functs(C,D)$}
	\end{equation*}
	This is part of the natural enrichment of $\Pyk(\Cat_{\infty})$ over $\Cat_{\infty}$.
\end{dfn}

\begin{exm}\label{exm:FunctsasCSS}
	Let $ C $ and $ D $ be pyknotic \categories.
	As a complete Segal space, the \category $ \Functs(C,D) $ of continuous functors $ \fromto{C}{D} $ is given by the assignment
	\begin{equation*}
		[n] \mapsto \Map_{\Pyk(\Cat_{\infty})}(C \cross [n]^{\disc},D) = \int_{K \in \Proj} \Map_{\Cat_{\infty}}((C \cross [n]^{\disc})(K),D(K)) \period
	\end{equation*}
\end{exm}


\subsection{Profinite spaces as pyknotic spaces}\label{subsec:profinitepyknotic}

The goal of this section is to show that not only profinite \textit{sets} embed into $ \Pyk(\Space) $, but profinite spaces actually embed into $ \Pyk(\Space) $. 
The first approach one might have to showing this is to try to show that \pifinite spaces are cocompact when regarded as discrete pyknotic spaces, as this implies that the proëxtension of the discrete functor
\begin{equation*}
	(-)^{\disc} \colon \incto{\Spacefin}{\Pyk(\Space)}
\end{equation*}
is fully faithful.
Unfortunately, this approach is destined for failure: finite \textit{sets} aren't even cocompact objects of $ \Pyk(\Space) $, as the following counterexample shows.

\begin{ctrexm}
	The finite set $ \{0,1\} $ with two elements is not cocompact when regarded as a discrete pyknotic space.
	Since the embedding of compactly generated topological spaces into $ \Pyk(\Set) $ preserves limits, and the image of a finite set $ S $ under the Yoneda embedding $ \incto{\Pro(\Setfin)}{\Pyk(\Set)} $ coincides with its image under the embedding $ \incto{\TSpccg}{\Pyk(\Set)} $, to see that $ \{0,1\} $ is not cocompact in $ \Pyk(\Set) $, it suffices to prove that the discrete topological space $ \{0,1\} $ is not cocompact in $ \TSpccg $. 
	To see this, let $ s \colon \incto{\NNup}{\NNup} $ be the successor function $ \goesto{n}{n+1} $ and consider the diagram of discrete topological spaces
	\begin{equation*}
		\begin{tikzcd}[sep=1.5em]
			\cdots \arrow[r, "s", hooked] & \NNup \arrow[r, "s", hooked] &  \NNup \arrow[r, "s", hooked] & \NNup \period
		\end{tikzcd}
	\end{equation*}
	We claim that the map of sets
	\begin{equation}\label{eq:sucessor}
		\textstyle\colim_{n} \Map(\NNup,\{0,1\}) \to \Map(\textstyle\lim_n \NNup,\{0,1\})
	\end{equation}
	is not a bijection.
	To see this, note that the limit $ \lim_{n} \NNup $ empty, so $ \Map(\lim_n \NNup,\{0,1\}) $ has cardinality $ 1 $.
	On the other hand, $ \Map(\NNup,\{0,1\}) $ is the powerset $ \Pup(\NNup) $ of $ \NNup $, and the colimit $ \colim_{n} \Pup(\NNup) $ along the inverse image maps $ s^{-1} \colon \surjto{\Pup(\NNup)}{\Pup(\NNup)} $ has infinite cardinality.
\end{ctrexm}

The approach we take to show that profinite spaces embed into $ \Pyk(\Space) $ is somewhat indirect: we show that profinite sets form a basis for the effective epimorphism topology on $ \Spaceprofin $, so that hypersheaves on $ \Pro(\Setfin) $ and $ \Spaceprofin $ coincide (\Cref{prop:hyperbasis}).
The Yoneda embedding provides the desired embedding $ \incto{\Spaceprofin}{\Pyk(\Space)} $.

In order to get this approach off the ground, we first need to talk about the effective epimorphism topology on $ \Spaceprofin $.
The existence of this topology is immediate from \cite[{\SAGthm{Proposition}{A.3.2.1} \& \SAGthm{Theorem}{E.6.3.1}}]{SAG}.

\begin{prp}
	Write $ E \subseteq \Spaceprofin $ for the subcategory of those morphisms $ e \colon \fromto{X}{Y} $ in $ \Spaceprofin $ that can be written as an inverse limit of morphisms $ e_{\alpha} \colon \fromto{X_{\alpha}}{Y_{\alpha}} $ where each $ e_{\alpha} $ is an effective epimorphism of \pifinite spaces.
	Then $ E $ defines a \presite structure on $ \Spaceprofin $.
	We write $ \eff \coloneq \tau_E $\index[notation]{eff@$\eff$} for the resulting finitary topology, the \defn{effective epimorphism topology}\index[terminology]{topology!effective epimorphism}\index[terminology]{effective epimorphism topology}.
\end{prp}


\begin{nul}\label{exm:effepionProsubcanon}
	From \SAG{Proposition}{A.3.3.1} it follows that the effective epimorphism topology on $ \Spaceprofin $ is subcanonical.
	Moreover, since the Yoneda embedding
	\begin{equation*}
		\incto{\Spaceprofin}{\Sheff{\Spaceprofin;\Space_{\updelta_1}}}
	\end{equation*}
	preserves $\updelta_0$-small limits, truncated objects of a \topos are hypercomplete, and hypercomplete objects are closed under limits, the Yoneda embedding factors through $ \Sheffhyp{\Spaceprofin;\Space_{\updelta_1}} $.
\end{nul}

Now we show that $ \Pro(\Setfin) \subset \Spaceprofin $ is a basis for the effective epimorphism topology (in the sense of \Cref{def:basis}).
In fact, we show that every object of $ \Spaceprofin $ admits an effective epimorphism from a profinite set.
This requires a number of preliminaries.

\begin{nul}\label{nul:n-truncaedcohbasis}
	Note that for every space $ U $ there exists an effective epimorphism $ \surjto{\uppi_{0}(U)}{U} $.
	In particular, $ \Setfin \subset \Spacefin $ is a basis for the effective epimorphism topology and every object admits a cover by a single object of $ \Setfin $.
\end{nul}

Since we must contend with \proobjects, it isn't immediate from \Cref{nul:n-truncaedcohbasis} that every profinite space admits an effective epimorphism from a profinite set.
To show this, we'll use the fact that we can always arrange to index a \proobject by a particularly nice poset.

\begin{dfn}
	We say that a poset $ A $ is \defn{down-finite}\index[terminology]{down-finite}\index[terminology]{poset!down-finite} if for every element $ \alpha \in A $, the set $ \{ \beta \in A \,|\, \beta \leq \alpha \} $ is finite.
\end{dfn}

\begin{lem}[{\SAG{Lemma}{E.1.6.4}}]\label{lem:reducetoresidfinite}
	Let $ A' $ be a filtered poset.
	Then there exists a colimt-cofinal map of posets $ \fromto{A}{A'} $, where $ A $ is a down-finite filtered poset.
\end{lem}

\begin{cnstr}[rank function]
	If $ A $ is a down-finite poset, then there exists a map of posets
	\begin{equation*}
		\rank \colon \fromto{A}{\NNup}\index[notation]{rank@$\rank$}
	\end{equation*}
	called the \defn{rank}\index[terminology]{rank} which is determined by the following requirement: for each $ \alpha \in A $, the number $ \rk(\alpha) $ is the smallest natural number not equal to $ \rk(\beta) $ for $ \beta < \alpha $ (cf. \HA{Remark}{A.5.17}).
	In particular, $ \rk(\alpha) = 0 $ if and only if $ \alpha $ is a minimal element of $ A $.
\end{cnstr}

\begin{prp}\label{prp:probasis}
	For every object $ X \in \Spaceprofin $, there exists an cover $ \surjto{Y}{X} $ for the effective epimorphism topology on $ \Spaceprofin $ for which $ Y \in \Pro(\Setfin) $.
	In particular, $ \Pro(\Setfin) \subset \Spaceprofin $ is a basis for the effective epimorphism topology on $ \Spaceprofin $.
\end{prp}

\begin{proof}
	To simplify notation, write $ C \coloneq \Spacefin $ and $ D \coloneq \Setfin $.
	Let $ \{X_\alpha\}_{\alpha \in A^{\op}} $ be an object of $ \Pro(C) $.
	We without loss of generality assume that $ A $ is a down-finite filtered poset (\Cref{lem:reducetoresidfinite}).
	We construct a morphism $ e \colon \fromto{\{Y_\alpha\}_{\alpha \in A^{\op}}}{\{X_\alpha\}_{\alpha \in A^{\op}}} $ in $ \Pro(C) $ such that for each $ \alpha \in A $, the morphism $ e_\alpha \colon \fromto{Y_\alpha}{X_\alpha} $ is an effective epimorphism and $ Y_{\alpha} \in D $.
	We construct this inductively on the rank of elements of $ A $.
	For each $ n \in \NNup $, write 
	\begin{equation*}
		A_{\leq n} \coloneq \{ \alpha \in A \, | \, \rk(\alpha) \leq n \} \period
	\end{equation*}

	First, for each element $ \alpha \in A $ with $ \rk(\alpha) = 0 $ (i.e., minimal element of $ A $), appealing to \Cref{nul:n-truncaedcohbasis}, choose an effective epimorphism $ e_\alpha \colon \surjto{Y_{\alpha}}{X_{\alpha}} $ where $ Y_\alpha \in D $.
	
	For the induction step, suppose that we have defined a functor $ Y \colon \fromto{A_{\leq n}^{\op}}{D} $ along with a natural effective epimorphism $ e \colon \surjto{Y}{\restrict{X}{A_{\leq n}^{\op}}} $; we now extend $ Y $ to $ A_{\leq n+1}^{\op} $ as follows.
	For each $ \alpha \in A $ with $ \rank(\alpha) = n+1 $, consider the pulled-back effective epimorphism
	\begin{equation*}
		\surjto{\coprod_{\substack{\beta < \alpha \\ \rank(\beta) = n}} X_{\alpha} \crosslimits_{X_\beta} Y_{\beta}}{X_{\alpha}} \period
	\end{equation*}
	For each $ \beta < \alpha $ with $ \rk(\beta) = n $, appealing to \Cref{nul:n-truncaedcohbasis} we choose an effective epimorphism $ e'_\beta \colon \surjto{Y'_{\beta}}{X_{\alpha} \cross_{X_{\beta}} Y_{\beta}} $, and define the effective epimorphism $ e_{\alpha} \colon \surjto{Y_\alpha}{X_{\alpha}} $ as the composite
	\begin{equation*}
		\begin{tikzcd}
			e_\alpha \colon Y_{\alpha} \coloneq \displaystyle \coprod_{\substack{\beta < \alpha \\ \rank(\beta) = n}} Y'_\beta \arrow[r, ->>, "\coprod_\beta e'_\beta"] & \displaystyle \coprod_{\substack{\beta < \alpha \\ \rank(\beta) = n}} X_{\alpha} \crosslimits_{X_\beta} Y_{\beta} \arrow[r, ->>] & X_\alpha \period
		\end{tikzcd}
	\end{equation*}
	Then by construction the functor $ Y \colon \fromto{A_{\leq n}^{\op}}{D} $ extends to a functor $ Y \colon \fromto{A_{\leq n+1}^{\op}}{D} $ equipped with a natural effective epimorphism $ e \colon \surjto{Y}{\restrict{X}{A_{\leq n+1}^{\op}}} $, as desired.
\end{proof} 

As an immediate consequence of \Cref{prop:hyperbasis}, we obtain the desired equivalence on hypersheaves:

\begin{cor}\label{cor:proembeds}
	Restriction of presheaves along the inclusion $ \incto{\Pro(\Setfin)}{\Spaceprofin} $ defines an equivalence of large \topoi
	\begin{equation*}
		\Sheffhyp{\Spaceprofin;\Space_{\updelta_1}} \equivalence \Pyk(\Space) 
	\end{equation*}
	with inverse given by right Kan extension.
\end{cor}

\begin{nul}\label{nul:Yonedaisdiscrete}
	We finish this section by showing that the restricted Yoneda embedding
	\begin{equation*}
		 \yo \colon \incto{\Spacefin}{\Pyk(\Space)}
	\end{equation*}
	agrees with the discrete functor $ \Gammaupperstar \colon \incto{\Spacefin}{\Pyk(\Space)} $. 
	To see this, first note that the global sections functor $ \Gammalowerstar \colon \fromto{\Pyk(\Space)}{\Space} $ is given by the composite
	\begin{equation*}
		\begin{tikzcd}
			\Pyk(\Space) \arrow[r, "\restrict{(-)}{\Spacefin}"] & \Sheff{\Spacefin} \arrow[r, "\sim"{yshift=-0.2em}] & \Space
		\end{tikzcd}
	\end{equation*}
	of restriction along the inclusion $ \incto{\Spacefin}{\Spaceprofin} $ with evaluation at the terminal object.
	The inverse equivalence $ \equivto{\Space}{\Sheff{\Spacefin}} $ is given by sending a space $ Y $ to the functor
	\begin{equation*}
		\Map_{\Space}(-,Y) \colon \fromto{\Spacefin^{\op}}{\Space} \period
	\end{equation*}
	Hence we have natural equivalences 
	\begin{align*}
		\Map_{\Pyk(\Space)}(\yo(K),X) &\equivalent X(K) \equivalent \restrict{X}{\Spacefin^{\op}}(K) \\ 
		&\equivalent \Map_{\Space}(K,\Gammalowerstar(X)) \\ 
		&\equivalent \Map_{\Pyk(\Space)}(\Gammaupperstar(K),X) \period
	\end{align*}

	Since the Yoneda embedding $ \yo \colon \incto{\Spaceprofin}{\Pyk(\Space)} $ preserves inverse limits, we see that the Yoneda embedding is the extension of the discrete functor $ \Gammaupperstar \colon \incto{\Spacefin}{\Pyk(\Space)} $ to \proobjects.
\end{nul}


\subsection{Profinite stratified spaces as pyknotic \texorpdfstring{$\infty$}{∞}-categories}\label{subsec:categoryobjs}

Our next goal is to show that profinite layered \categories embed into $ \Pyk(\Cat_{\infty}) $.
We'll deduce this from the fact that profinite spaces embed into $ \Pyk(\Space) $ by regarding profinite layered \categories as \textit{complete Segal objects} of $ \Spaceprofin $.

We consider a composite of three functors:
\begin{enumerate}[(1)]
	\item The nerve provides an embedding $ \Nerve \colon \incto{\Layfin}{\CS(\Spacefin)} $ of \pifinite layered \categories into complete Segal objects in \pifinite spaces.
	Extending the nerve to \proobjects provides an embedding $ \Nerve \colon \incto{\Pro(\Layfin)}{\Pro(\CS(\Spacefin))} $.

	\item The proëxtension of the Yoneda embedding $ \incto{\CS(\Spacefin)}{\CS(\Spaceprofin)} $ defines a functor
	\begin{equation*}
		\fromto{\Pro(\CS(\Spacefin))}{\CS(\Spaceprofin)} \period
	\end{equation*}
	Since the image of Yoneda embedding $ \incto{\CS(\Spacefin)}{\CS(\Spaceprofin)} $ consists of cocompact objects (\Cref{cor:COspacefiniscocompact}), applying \HTT{Proposition}{5.3.5.11} we see that this functor is fully faithful.
	
	\item \Cref{cor:proembeds} proves an embedding $ \incto{\Spaceprofin}{\Pyk(\Space)} $.
	Passing to complete Segal objects, we obtain an embedding
	\begin{equation*}
		\incto{\CS(\Spaceprofin)}{\CS(\Pyk(\Space)) \equivalent \Pyk(\Cat_{\infty})} \period
	\end{equation*}
\end{enumerate}

We thus have fully faithful functors
\begin{equation}\label{eq:profinlayeredembed}
	\begin{tikzcd}[sep=1.5em]
		\Pro(\Layfin) \arrow[r, hooked, "\Nerve"] & \Pro(\CS(\Spacefin)) \arrow[r, hooked] & \CS(\Spaceprofin) \arrow[r, hooked] & \Pyk(\Cat_{\infty}) \period
	\end{tikzcd}
\end{equation}
Since the embedding \eqref{eq:profinlayeredembed} preserves inverse limits and agrees with the discrete functor $ (-)^{\disc} \colon \incto{\Layfin}{\Pyk(\Cat_{\infty})} $ when restricted to \pifinite layered \categories \Cref{nul:Yonedaisdiscrete}, it is the extension to \proobjects of the discrete functor.
Thus we have shown:

\begin{prp}\label{exm:profinstratspaceembedintoPyk}
	The functor $ \fromto{\Pro(\Layfin)}{\Pyk(\Cat_{\infty})} $ defined by extending the discrete functor $ (-)^{\disc} \colon \incto{\Layfin}{\Pyk(\Cat_{\infty})} $ to \proobjects is fully faithful.
\end{prp}

In particular, we may now deduce a key almost-cocompactness result for profinite stratified spaces, when regarded as pyknotic \categories.
The following result plays an important role in the proof of \Cref{thm:Qellexodromy}.

\begin{cor}\label{cor:profincompactuniform}
	Let $ \mbfPi $ be a profinite layered \category, let $ n \geq 0 $ be an integer, and let $ D \colon \fromto{A}{\Pyk(\Cat_{n})} $ be a filtered diagram.
	Then the natural functor
	\begin{equation*}
		\colim_{\alpha \in A} \Functs(\mbfPi,D_{\alpha}) \to \Functs(\mbfPi,\textstyle \colim_{\alpha \in A} D_{\alpha})
	\end{equation*}
	is an equivalence.
\end{cor}

\begin{proof}
	In light of \Cref{exm:FunctsasCSS,nul:compactnesseasyhypotheses}, note that both $ \mbfPi $ and $ \mbfPi \cross [1] $ satisfy the hypotheses of \Cref{prop:Catcompactness}.
\end{proof}

\begin{exm}
	Let $\mbfPi$ be a profinite stratified space.
	We may now generalize \Cref{exm:pifinitestratspacesaspyknoticcategories} as follows.
	Again since profinite stratified spaces embed fully faithfully into \topoi,
	it follows that we obtain the following formula for $\mbfPi$ as a pyknotic \category:
	\begin{equation*}
		\mbfPi(K) \simeq \Funlowerstar(\Ktilde, \widetilde{\mbfPi}) \comma
	\end{equation*}
	for any projective compactum $K$, functorially in $K$ and $\mbfPi$.
	Thus, if $\XX \to \widetilde{S}$ is a spectral \topos, then one has
	\begin{equation*}
		\StrShape^S(\XX)(K) \simeq \Funlowerstar(\Ktilde, \XX) \comma
	\end{equation*}
	for any projective compactum $K$.

	More generally, we may use this formula to define a more refined shape for any bounded coherent \topos $\XX$:
	the assignment
	\begin{equation*}
		K \mapsto \Funlowerstar(\Ktilde, \XX)
	\end{equation*}
	defines a pyknotic \category \smash{$\upPi_{(\infty,1)}^{\pyk}(\XX)$} that we call the \defn{pyknotic shape} of $ \XX $.
	Following Lurie's work on ultracategories \cite{Ultracategories}, it is reasonable to expect that the \topos $\XX$ can itself be reconstructed from the \category of continuous functors
	\begin{equation*}
		\upPi_{(\infty,1)}^{\pyk}(\XX) \to \Spaceu \comma
	\end{equation*}
	where \smash{$ \Spaceu $} denotes the pyknotic \category \smash{$ K \mapsto \Ktilde $}.
	We will not pursue such questions here, however.
\end{exm}


\subsection{Exodromy with discrete coefficients, revisited}\label{subsec:functorcomparison}

Let $ X $ be a coherent scheme.
The original formulation of the Exodromy Theorem for schemes says that if $ C $ is a locally \pifinite \category, then functors $ \Gal(X) \to C $ (in the `pro' sense) are the same things as constructible sheaves of spaces on $ X $ (\Cref{thm:stableexodromyforstrattopoi}).
The goal of this section is to give a pyknotic reformulation of this theorem.
To do this, we show that functors $ \fromto{\Gal(X)}{C} $ in the `pro' sense are the same as continuous functors $ \fromto{\Gal(X)}{C^{\disc}} $ in the pyknotic sense; here we regarded $ \Gal(X) $ as a pyknotic category under the embedding of \Cref{exm:profinstratspaceembedintoPyk}.

\begin{lem}\label{cor:FunproisFuncts}
	Let $ C $ be a locally \pifinite \category and $ \mbfPi = \{\Pi_{\alpha}\}_{\alpha \in A} $ a profinite layered \category.
	Then the natural functor
	\begin{equation*}
		\Fun(\mbfPi, C) = \colim_{\alpha \in A^{\op}} \Fun(\Pi_{\alpha},C) \to \Functs(\mbfPi,C^{\disc})
	\end{equation*} 
	is an equivalence.
\end{lem}

\begin{proof}
	Since the \category $ C $ is locally \pifinite, $ C $ is the filtered union
	\begin{equation*}
		C = \colim_{C_0 \in \Subpi(C)} C_{0}
	\end{equation*}
	over the poset of \pifinite full subcategories of $ C $ ordered by inlusion (\Cref{ntn:subpi}).
	Since the pyknotic set $ \uppi_0(\interior \mbfPi) $ is quasicompact and $ \uppi_0(\interior C) $ is discrete, we see that every continuous functor $ \fromto{\mbfPi}{C^{\disc}} $ factors through $ C_0^{\disc} \subset C^{\disc} $ for some \pifinite full subcategory $ C_0 \subset C $.
	Hence we have an identification
	\begin{equation*}
		\Functs(\mbfPi,C^{\disc}) = \colim_{C_0 \in \Subpi(C)} \Functs(\mbfPi,C_0^{\disc}) \period
	\end{equation*}
	Since each \category $ C_0 $ is \pifinite, each $ \Pi_{\alpha} $ has finitely many objects up to equivalence, and colimits commute, from \Cref{prop:Catcocompactness} we see that 
	\begin{align*}
		\colim_{C_0 \in \Subpi(C)} \Functs(\mbfPi,C_0^{\disc}) &\equivalent \colim_{C_0 \in \Subpi(C)} \colim_{\alpha \in A^{\op}} \Functs(\Pi_{\alpha}^{\disc},C_0^{\disc}) \\ 
		&\equivalent \colim_{C_0 \in \Subpi(C)} \colim_{\alpha \in A^{\op}} \Fun(\Pi_{\alpha},C_0) \\
		&\equivalent \colim_{\alpha \in A^{\op}} \Fun(\Pi_{\alpha},C) \\
		&= \Fun(\mbfPi,C) \period \qedhere
	\end{align*}
\end{proof}

\begin{exm}
	Let $ \mbfPi $ be a profinite layered \category.
	\Cref{cor:FunproisFuncts} provides an equivalence
	\begin{equation*}
		\Fun(\mbfPi,\Spacefin) \equivalent \Functs(\mbfPi,\Spacefin^{\disc}) \period
	\end{equation*}
	Moreover, for any finite ring $ R $, \Cref{cor:FunproisFuncts} provides an equivalence
	\begin{equation*}
		\Fun(\mbfPi,\Perf(R)) \equivalent \Functs(\mbfPi,\Perf(R)^{\disc}) \period
	\end{equation*}
\end{exm}

We thus obtain the following reformulation of \Cref{thm:stableexodromyforstrattopoi}:

\begin{cor}\label{cor:pyknoticdiscreteexodromy}
	Let $ C $ be a $\updelta_0$-small, locally \pifinite category.
	Let $ S $ be a spectral topological space, and $ \XX $ an $ S $-stratified \topos.
	Then there is a natural equivalence
	\begin{equation*}
		\Functs(\StrShape^{S}(\XX),C^{\disc}) \simeq \Cons^S(\XX;C) \period
	\end{equation*}
\end{cor}


\subsection{Exodromy with profinite coefficients}\label{subsec:profiniteexodromy}

In this section we extend the Exodromy Theorem from finite coefficients to coefficients in the ring of integers in a nonarchimedean local field.
First we isolate the class of profinite rings that we're interested in.

\begin{dfn}\label{def:I-complete}
	Let $ \Lambda $ be ring and $ I \subset \Lambda $ an ideal.
	We say that $ \Lambda $ is \defn{$ I $-profinite}\index[terminology]{ring!profinite@$I$-profinite}\index[terminology]{profinite ring@$I$-profinite ring} if: 
	\begin{enumerate}[(\ref*{def:I-complete}.1)]
		\item\label{def:I-complete.1} the ring $ \Lambda $ is noetherian,

		\item\label{def:I-complete.2} the ring $ \Lambda $ is complete with respect to the topology defined by the ideal $ I $,

		\item\label{def:I-complete.3} and for each integer $ n \geq 1 $, the quotient ring $ \Lambda/I^n $ is finite.
	\end{enumerate}

	We simply say `let $ \Lambda $ be an $ I $-profinite ring' to mean  `let $ \Lambda $ be a ring with ideal $ I \subset \Lambda $ satisfying \enumref{def:I-complete}{1}--\enumref{def:I-complete}{3}'.  
\end{dfn}

\begin{nul}
	The reason for the noetherian hypothesis is to apply results from \cite[\S\S6.5-6.8]{MR3379634}; these results use the noetherianity of $ \Lambda $ to apply the Artin--Rees Lemma. 
	The requirement that the quotients $ \Lambda/I^n $ be finite is so that we have access to the Exodromy Theorem with coefficients in a finite discrete ring (\Cref{cor:stableexodromyforschemes,cor:pyknoticdiscreteexodromy}).
\end{nul}

\begin{exm}
	Let $ E $ be a nonarchimedean local field with ring of integers $ \Oup_E $.
	Write $ \mfrak_{E} \subset \Oup_E $ for the maximal ideal.
	Then $ \Oup_E $ is an $ \mfrak_{E} $-profinite ring.

	In particular, for any prime number $ \el $ and prime power $ q $, the ring $ \ZZell $ is $ (\el) $-profinite and the ring $ \FFq\llbracket t \rrbracket $ is $ (t) $-profinite.
\end{exm}

\begin{dfn}
	Let $ \Lambda $ be an $ I $-profinite ring.
	The pyknotic \category of \defn{perfect $\Lambda$-complexes}\index[terminology]{perfect complex} is the limit
	\begin{equation*}
		\Perf(\Lambda) \coloneq \lim_{n \geq 1} \Perf(\Lambda/I^n)^{\disc} \index[notation]{PerfLambda@$\Perf(\Lambda)$}
	\end{equation*}
	in $\Pyk(\Cat_{\infty})$.
\end{dfn}

\begin{nul}
	Please observe that if $ \Lambda $ is an $ I $-profinite ring, since the underlying functor preserves limits, the underlying \category $\Perf(\Lambda)^{\und}$ coincides with the usual \category of perfect complexes on $\Lambda$.

	More generally, by \Cref{nul:PykSislocal} we see that for any projective compactum $K$ exhibited as a profinite set $\{K_{\alpha}\}_{\alpha \in A^{\op}}$, we have
	\begin{equation*}
		\Perf(\Lambda)(K) \simeq \lim_{n \geq 1} \colim_{\alpha \in A} \Perf(\Lambda/I^n)^{K_{\alpha}} \period
	\end{equation*}
\end{nul}

\begin{rmk}
	It is not necessary to give such an \textit{ad hoc} definition of the pyknotic \category of perfect complexes on a profinite ring.
	There is an intrinsic definition for a general pyknotic ring, but to develop this material here would take us too far afield.

	In the case of a $ I $-profinite ring, the more intrinsic definition recovers the definition given here.
\end{rmk}

\begin{rec}
	Let $X$ be a coherent scheme, and let $ \Lambda $ be an $ I $-profinite ring.
	Recall \cite[Definition 6.5.1]{MR3379634} that the \category of constructible $\Lambda$-complexes on $X$ can be identified as the limit of \categories
	\begin{equation*}
		\Dcons(X_{\proet};\Lambda) \simeq \lim_{n \geq 1} \Dcons(X_{\et};\Lambda/I^n) \period
	\end{equation*}
	In other words, a constructible $\Lambda$-complex on $X$ is a prosystem $\{F_n\}_{n \geq 1}$ consisting of constructible $(\Lambda/I^n)$-complexes $F_n$ on $X$ along with coherent identifications
	\begin{equation*}
		F_m \simeq F_n \otimes_{\Lambda/I^n} (\Lambda/I^m)
	\end{equation*}
	for all $m \leq n$.
\end{rec}

\begin{thm}\label{thm:exodromyforOEcoefficients}
	Let $X$ be a coherent scheme, and let $ \Lambda $ be an $ I $-profinite ring.
	Then there is a natural equivalence of \categories
	\begin{equation*}
		\Dcons(X_{\proet};\Lambda) \simeq \Functs(\Gal(X),\Perf(\Lambda)) \period
	\end{equation*}
\end{thm}

\begin{proof}
	This follows by taking limits over $ n $ of the equivalences
	\begin{equation*}
		\Dcons(X_{\proet};\Lambda/I^n) \simeq \Functs(\Gal(X),\Perf(\Lambda/I^n)) \period
	\end{equation*}
	provided by \Cref{cor:stableexodromyforschemes}.
\end{proof}

\begin{nul}
	Let $X$ be a coherent scheme, and let $ \Lambda $ be an $ I $-profinite ring.
	Attached to any constructible sheaf $F$ of $\Lambda$-complexes on $X$, we have an associated \defn{exodromy representation}
	\begin{equation*}
		\exrep_F \colon \Gal(X) \to \Perf(\Lambda)
	\end{equation*}
	that is sufficient to reconstruct $F$.
\end{nul}

\begin{wrn}\label{wrn:BSexample6.6.12}
	For an $ I $-profinite ring $ \Lambda $, the pyknotic \category $\Perf(\Lambda)$ is \textit{not} generally discrete.
	Moreover, it is not generally the case that an exodromy representation $\exrep_F \colon \Gal(X) \to \Perf(\Lambda)$ factors through a quotient of $\Gal(X)$ with only finitely many isomorphism types.
	For instance, Bhatt and Scholze give an example of a constructible $\ZZell$-complex on a nonnoetherian scheme that is not lisse on the strata of any finite stratification \cite[Example 6.6.12]{MR3379634}.
\end{wrn}

If $ X $ is a topologically noetherian scheme, Bhatt and Scholze demonstrate that this problem does not arise.
To explain, this, let us briefly recall the basics of the \textit{constructible topology} on a spectral topological space.

\begin{rec}
	Let $ S $ be a spectral topological space.
	The \defn{constructible topology} on $ S $ is the topology on the underlying set of $ S $ generated by the constructible subsets of $ S $ (\Cref{rec:subsetconstr}).
	We write $ S^{\ctop} $ for the set $ S $ equipped with the constructible topology.
	The topological space $ S^{\ctop} $ is a Stone topological space.

	The constructible topology admits a description in terms of \proobjects: if we exhibit $ S $ as a profinite poset $ \{P_{\alpha}\}_{\alpha \in A} $, then the profinite set $ S^c $ is given by the inverse system $ \{\interior P_{\alpha}\}_{\alpha \in A} $ of the underlying sets of the posets $ P_{\alpha} $.
	In particular, the assignment $ \goesto{S}{S^{\ctop}} $ is a right adjoint to the inclusion $ \incto{\Stn}{\TSpcspec} $.

	If $ X $ is a coherent scheme, then notice that the profinite set $ \uppi_0 \interior \Gal(X) $ coincides with the set $ X^{\zar} $ equipped with the constructible topology.
	If $\Lambda$ is an $I$-profinite ring, then the exodromy representation $\exrep_F$ attached to a constructible $\Lambda$-complex on $X$ induces a continuous map of pyknotic sets
	\begin{equation*}
		\uppi_0 \interior \exrep_F \colon X^{\zar,\ctop} = \uppi_0 \interior \Gal(X) \to \uppi_0 \interior \Perf(\Lambda) \period
	\end{equation*}
\end{rec}

\begin{lem}[{\cite[Proposition 6.6.11]{MR3379634}}] \label{lem:BS6.6.11}
	Let $X$ be a topologically noetherian scheme, and let $\Lambda$ be an $I$-profinite ring.
	Then for any constructible sheaf $F$ of $\Lambda$-complexes on $X$, the continuous map
	\begin{equation*}
		\uppi_0 \interior \exrep_F \colon X^{\zar,\ctop} \to \uppi_0 \interior \Perf(\Lambda)
	\end{equation*}
	factors through a finite (discrete) quotient of $X^{\zar,\ctop}$.
\end{lem}


\subsection{Exodromy with \texorpdfstring{$\el$}{ℓ}-adic coefficients}\label{subsec:compactnessofpyknoticcats}

The goal of this section is to use the Exodromy Theorem with $ \ZZell $-coefficients (\Cref{thm:exodromyforOEcoefficients}) to prove the Exodromy Theorem with \smash{$ \QQell $}- or \smash{$ \QQellbar $}-coefficients.
For this we begin by developing the basics of how to regard $ \Perf(\ZZell) $ as a pyknotic object in stable \categories with \tstructure and \texact functors. 
Recall that we use \textit{homological} indexing conventions for our \tstructures (\Cref{conv:homological}).

\begin{rec}[\tstructure on filtered colimits]\label{rec:tstructurefilteredcolim}
	Let $ \{C_{\alpha}\}_{\alpha \in A} $ be a filtered diagram of stable \categories with \tstructures and \texact transition morphisms. 
	Then the colimit $ C \coloneq \colim_{\alpha \in A} C_{\alpha} $ in $ \Cat_{\infty} $ admits a natural \tstructure defined by
	\begin{equation*}
		C_{\geq 0} \coloneq \colim_{\alpha \in A} C_{\alpha,\geq 0} \andeq C_{\leq 0} \coloneq \colim_{\alpha \in A} C_{\alpha,\leq 0} \period
	\end{equation*}
	Moreover, if for each $ \alpha \in A $ the \tstructure on $ C_{\alpha} $ is bounded, then the natural \tstructure on $ C $ is bounded as well.
\end{rec}

\begin{rec}[{\tstructure on $ \Perf(\Lambda)^{\und} $}]
	Let $ \Lambda $ be an $ I $-profinite ring.
	The stable \category $\Perf(\Lambda)^{\und}$ carries its usual a bounded \tstructure, defined as follows.
	Let $F = \{F_n\}_{n \geq 1}$ be a perfect $\Lambda$-complex.
	Then $F \in \Perf(\Lambda)^{\und}_{\geq 0}$ if and only if, for any $n \geq 1$, the perfect $(\Lambda/I^n)$-complex $F_n$ lies in $\Perf(\Lambda/I^n)_{\geq 0}$.
	On the other hand, $F \in \Perf(\Lambda)^{\und}_{\geq 0}$ if and only if, for all $k > 0$, the prosystem $\{\Hup_k(F_n) \}_{n \geq 1}$ vanishes.
\end{rec}

\begin{cnstr}[{\tstructure on $ \Perf(\Lambda) $}]
	Now we extend this to a \tstructure on all the values of the pyknotic \category $\Perf(\Lambda)$ on a projective compactum $ K $.
	To this end, for any $n \geq 1$, and any finite set $ J $, we endow the $ J $-fold product $ \Perf(\Lambda/I^n)^J $ with the product \tstructure induced from the \tstructure on $ \Perf(\Lambda/I^n) $; this \tstructure is bounded because $ J $ is finite.
	For any projective compactum $ K = \{K_{\alpha}\}_{\alpha \in A^{\op}} $, we also endow the filtered colimit
	\begin{equation*}
		\Perf(\Lambda/I^n)^K = \colim_{\alpha \in A} \Perf(\Lambda/I^n)^{K_{\alpha}}
	\end{equation*}
	with its natural bounded \tstructure (\Cref{rec:tstructurefilteredcolim}).
	With this definition, the assignment
	\begin{equation*}
		K \mapsto \Perf(\Lambda/I^n)^K
	\end{equation*}
	is a pyknotic object in stable \categories with bounded \tstructures and \texact functors.

	For any projective compactum $K$, the limit $\lim_{n\geq 1} \Perf(\Lambda/I^n)^K$ is an inverse system of right \texact functors.
	Hence we may follow the example of $\Perf(\Lambda)^{\und}$ and define
	\begin{equation*}
		\Perf(\Lambda)(K)_{\geq 0} \coloneq \lim_{n \geq 1} \,(\Perf(\Lambda/I^n)^{K})_{\geq 0} \comma
	\end{equation*} 
	and $\Perf(\Lambda)(K)_{\leq 0}$ as the full subcategory of $\Perf(\Lambda)(K)$ spanned by those $F = \{F_n\}_{n \geq 1}$ such that for each $k > 0$, the prosystem $\{ \Hup_k(F_n) \}_{n \geq 1}$ vanishes.
	With this definition, the assignment $K \mapsto \Perf(\Lambda)(K)$ is a pyknotic object in stable \categories with \tstructures and \texact functors.

	Given integers $ a \leq b $, we write $ \Perf(\Lambda)_{[a,b]} $ for the subfunctor of $ \Perf(\Lambda) $ given by the assignment
	\begin{equation*}
		K \mapsto (\Perf(\Lambda)(K))_{[a,b]} \period
	\end{equation*}
\end{cnstr}

The following is now immediate from \cite[Lemma 6.5.3]{MR3379634} and \Cref{lem:BS6.6.11}:

\begin{lem}\label{lem:uniformboundednessofexodromyreps}
	Let $X$ be a topologically noetherian scheme, let $ \Lambda $ be an $ I $-profinite ring, and let $ F $ be a constructible $\Lambda$-complex on $X$.
	Then the exodromy representation
	\begin{equation*}
		\exrep_F \colon \Gal(X) \to \Perf(\Lambda)
	\end{equation*}
	is \emph{bounded} in the sense that there exists integers $a \leq b$ for which $\exrep_F $ factors through $\Perf(\Lambda)_{[a,b]} \subset \Perf(\Lambda)$.
\end{lem}

Now we are ready to extend to $ \QQell $-coefficients.

\begin{dfn}
	Let $ \el $ be a prime number and $E$ an algebraic extension of $\QQell $.
	We define the pyknotic \category $\Perf(E)$ as the filtered colimit of pyknotic \categories
	\begin{equation*}
		\Perf(E) \coloneq \colim_{E' \subset E} \Perf(\Oup_{E'})[\el^{-1}] \comma
	\end{equation*}
	over finite subextensions $ \QQell \subset E' \subset E$.
	Here, $C[\el^{-1}]$ is shorthand for the filtered colimit of pyknotic \categories
	\begin{equation*}
		C[\el^{-1}] \coloneq \colim\big(
		\begin{tikzcd}[sep=1.5em]
			C \arrow[r, "\cdot \el"] & C \arrow[r, "\cdot \el"] &  C \arrow[r, "\cdot \el"] & \cdots 
		\end{tikzcd}
		\big) \comma
	\end{equation*}
	where $ \cdot \el \colon \fromto{C}{C} $ denotes multiplication by $ \el $.
\end{dfn}

\begin{rec}
	Let $X$ be a topologically noetherian scheme, and let $E$ be an algebraic extension of $\QQell $.
	Recall \cite[Proposition 6.8.14]{MR3379634} that the \category of constructible $E$-complexes on $X$ can be identified as the filtered colimit
	\begin{equation*}
		\Dcons(X_{\proet};E) \simeq \colim_{E' \subset E} \Dcons(X_{\et};\Oup_{E'})[\el^{-1}] \comma
	\end{equation*}
	over finite subextensions $\QQell \subset E' \subset E$.

	Accordingly, we observe that $\Perf(E)^{\und}$ is the usual \category of perfect $E$-complex\-es.
	In particular, if $ \xi $ is a geometric point, then $\Dcons(\xi;E) \equivalent \Perf(E)^{\und} $.
\end{rec}

\begin{nul}
	More generally, we can be unpack the value of the pyknotic \category $\Perf(E)$ on a projective compactum $ K = \{K_{\alpha}\}_{\alpha \in A^{\op}}$ quite explicitly:
	\begin{equation*}
		\Perf(E)(K) \simeq \colim_{E' \subseteq E} \left( \lim_{n\geq 1} \colim_{\alpha \in A} \Perf(\Oup_{E'}/\mfrak_{E'}^n)^{K_{\alpha}} \right) [\el^{-1}] \period
	\end{equation*}
\end{nul}

We finish this section by proving the Exodromy Theorem for \smash{$ \QQell $}- and \smash{$ \QQellbar $}-coefficients.

\begin{thm}[{Exodromy for $ \el $-adic sheaves\index[terminology]{Exodromy Theorem!for elladic sheaves@for $\el$-adic sheaves}}]\label{thm:Qellexodromy}
	Let $X$ be a topologically noetherian scheme, $ \el $ a prime number, and $E$ an algebraic extension of $\QQell$.
	Then there is a natural equivalence of \categories
	\begin{equation*}
		\Dcons(X_{\proet};E) \simeq \Functs(\Gal(X),\Perf(E)) \period
	\end{equation*}
\end{thm}

\noindent To prove this theorem, we appeal to the work of \Cref{prop:Catcompactness} on the the \textit{almost compactness} of profinite layered \categories.

\begin{proof}[Proof of \Cref{thm:Qellexodromy}]
	Note that for each projective compactum $ K $, the value $\Perf(E)(K)$ is a filtered colimit over \texact functors.
	Hence combining \Cref{thm:exodromyforOEcoefficients,lem:uniformboundednessofexodromyreps,cor:profincompactuniform} we see that we have equivalences
	\begin{align*}
		\Dcons(X_{\proet}; E) &\equivalent \colim_{E' \subseteq E} \Dcons(X_{\et}; \Oup_{E'})[\el^{-1}] \\ 
		&\equivalent \colim_{E' \subseteq E} \Functs(\Gal(X),\Perf(\Oup_{E'}))[\el^{-1}] \\
		&\equivalent \colim_{E' \subseteq E} \colim_{N\geq 0} \Functs(\Gal(X),\Perf(\Oup_{E'})_{[-N,N]})[\el^{-1}] \\
		&\equivalent \colim_{N\geq 0} \Functs(\Gal(X), \colim_{E' \subseteq E} \Perf(\Oup_{E'})[\el^{-1}]_{[-N,N]}) \\
		&\equivalent \Functs(\Gal(X), \Perf(E)) \period \qedhere
	\end{align*}
\end{proof}

\begin{nul}
	Let $X$ be a topologically noetherian scheme, and let $E$ be an algebraic extension of $\QQell$.
	Attached to any constructible sheaf $F$ of $E$-complexes on $X$, we have an associated \defn{exodromy representation}
	\begin{equation*}
		\exrep_F \colon \Gal(X) \to \Perf(E)
	\end{equation*}
	that is sufficient to reconstruct $F$.

	For more general pyknotic rings $\Lambda$, once one has a good pyknotic \category $\Perf(\Lambda)$ of perfect $\Lambda$-complexes, it seems sensible simply to \textit{define} constructible sheaves as continuous representations $\Gal(X) \to \Perf(\Lambda)$ (even if $X$ does not satisfy any noetherian hypotheses).
\end{nul}

\begin{nul}
	Recall the following well-known example of Deligne \cite[Example 7.4.9]{MR3379634}:
	let $C$ be a smooth complete curve of genus at least $1$ over an algebraically closed field, with two points identified.
	Let $E$ be an algebraic extension of $\QQell$.
	The usual protruncated étale homotopy type $\Shapetrun^{\et}(C)$ is insufficient to reconstruct $E$-local systems on $ C $.

	However, the profinite \textit{stratified} étale homotopy type does suffice to recover these local systems: the \category of local systems of $E$-complexes on $C$ is equivalent to the \category of continuous functors $\Gal(C) \to \Perf(E)$ that carry all morphisms of $\Gal(X)$ to equivalences.

	Since the classifying space functor $ \invert \colon \fromto{\Cat_{\infty}}{\Space} $ preserves finite products, post-composition with $ \invert $ defines a left adjoint $\invert^{\pyk} $ to the inclusion $\Pyk(\Space) \inclusion \Pyk(\Cat_{\infty})$.
	For any coherent scheme $X$, one can form the \defn{pyknotic étale homotopy type}
	\begin{equation*}
		\Shape^{\pyk,\et}(X) \coloneq \invert^{\pyk}(\Gal(X)) \period
	\end{equation*}
	This pyknotic space suffices to reconstruct local systems with all the coefficient types discussed in this section.
	Its homotopy groups are the \defn{pyknotic étale homotopy groups} \smash{$\uppi_{\ast}^{\pyk,\et}(X)$}.
	
	One can even show that the Bhatt--Scholze proétale fundamental group corepresents the pyknotic fundamental group \smash{$\uppi_1^{\pyk,\et}(X)$};
	we intend to explore this and related points in future work.
\end{nul}


\subsection{Fibered Galois categories \& exodromy for simplicial schemes and stacks}\label{subsec:exodromyforstacks}

We finish this chapter by extending our notion of Galois categories to simplicial schemes and stacks, and proving our Exodromy Theorem in this context.\footnote{The material from this section first appeared – with a slightly different set-up – in a preprint of the first- and third-named authors \cite{exodromyforstacks}.}

\begin{rec}[{\HTT{Definition}{6.3.1.6}}]
	Let $ M $ be \acategory.
	A functor $ p \colon \XX \to M $ is a \defn{topos fibration}\index[terminology]{topos fibration}\index[terminology]{fibration!topos} if $ p $ is a bicartesian fibration, for each $ m \in M $ the fiber $ \XX_{m} $ is \atopos, and for each morphism  $f \colon n \to m$ of $ M $, the induced pullback functor $ \fupperstar \colon \XX_n \to \XX_m $ is left exact.
\end{rec}

\begin{dfn}
	Let $M$ be \acategory.
	A \defn{bounded coherent topos fibration}\index[terminology]{topos fibration!bounded coherent} $ \XX \to M $ is a topos fibration in which each fiber $\XX_m$ is a bounded coherent \topos, and for every morphism $f \colon n \to m$ of $ M $, the induced geometric morphism $ \flowerstar \colon \XX_m \to \XX_n$ is coherent.
	A \defn{spectral topos fibration}\index[terminology]{topos fibration!spectral} $ \XX \to S $ is a bounded coherent topos fibration in which each fiber $\XX_m$ is a spectral \topos (for the canonical profinite stratification of \cref{subsec:naturalstrat}).
\end{dfn}

\begin{nul}
	The usual straightening/unstraightening equivalence restricts to an equivalence between the \category of bounded coherent (respectively, spectral) topos fibrations $\XX \to M$ and the \category of functors from $M^{\op}$ to the \category of bounded coherent (resp., spectral) \topoi (cf. \HTT{Proposition}{6.3.1.7}).

	For a bounded coherent topos fibration $ \fromto{\XX}{M} $ we write $\XXcohbdd \subseteq \XX$ for the full subcategory spanned by the objects that are truncated and coherent in their fiber.
	Then $\XXcohbdd \to M$ is a cocartesian fibration that is classified by a functor from $ M $ to the category of bounded \pretopoi \cite[\SAGthm{Definition}{A.7.4.1} \& \SAGthm{Theorem}{A.7.5.3}]{SAG}.
\end{nul}

\begin{exm}
	If $X_{\ast}$ is a simplicial coherent scheme, then the fibered topos $X_{\ast,\et} \to \DDelta$ is a spectral topos fibration.
\end{exm}

A fibered form of \Categorical Hochster Duality is what allows us to construct fibered Galois categories.
To prove this fibered form of \Categorical Hochster Duality, we need to make sense of \categories fibered in profinite stratified spaces.

\begin{dfn}\label{dfn:fiberedinprofinite}
	Let $M$ be \acategory.
	We say that a functor $ f \colon \Pi \to M$ is an \defn{\category over $M$ fibered in layered \categories}\index[terminology]{category@\category!fibered in layered categories@\category fibered in layered \categories} if $ f $ is a catesian fibration whose fibers are layered \categories.
	We write $\Laycart_{/M}$\index[notation]{Laycart@$ \Laycart_{/M} $} for the \category of \categories over $M$ fibered in layered \categories.

	More generally, an \defn{\category over $ M $ fibered in profinite layered \categories} is a pyknotic object $ \mbfPi $ in \categories over $ M $ such that:
	\begin{itemize}
		\item for every projective compactum $K$, the functor $ \mbfPi(K) \to M $ is a cartesian fibration, and 

		\item for every object $m \in M$, the pyknotic \category $ \mbfPi_m \coloneq \mbfPi \times_M \{m\} $ is a profinite layered \category.
	\end{itemize}
	We write $\Lay_{\uppi,/M}^{\cart,\wedge}$\index[notation]{Laycartprofin@$ \Lay_{\uppi,/M}^{\cart,\wedge} $} for the \category of \categories over $M$ fibered in profinite layered \categories.
\end{dfn}

\begin{wrn}
	One might also contemplate the \category $\Pro(\Laycart_{\uppi,/M})$ of \proobjects in the full subcategory
	\begin{equation*}
		\Laycart_{\uppi,/M} \subseteq \Laycart_{/M}
	\end{equation*}
	spanned by those cartesian fibrations whose fibers are \pifinite layered \categories.
	This is generally \emph{not} equivalent to the \category of categories over $M$ fibered in profinite layered \categories.
	Under straightening/unstraightening, provides equivalences of \categories
	\begin{align*}
		\Lay_{\uppi,/M}^{\cart,\wedge} &\equivalent \Fun(M^{\op},\Layfin^{\wedge}) \\
		\intertext{and}
		\Pro(\Laycart_{\uppi,/M}) &\equivalent \Pro(\Fun(M^{\op},\Layfin)) \period
	\end{align*}
	Thus the \categories $ \Lay_{\uppi,/M}^{\cart,\wedge} $ and $ \Pro(\Laycart_{\uppi,/M}) $ coincide when $M$ is a finite poset \HTT{Proposition}{5.3.5.15}, but otherwise typically do not coincide.
\end{wrn}

\begin{nul}[Fibered \Categorical Hochster Duality]\label{nul:descriptionoffiberedHochster}
	Let $M$ be \acategory.
	By \Categorical Hochster Duality, the \category of spectral topos fibrations over $M$ is equivalent to the \category $\Lay_{\uppi,/M}^{\cart,\wedge}$.
	Let us make the equivalence explicit.
	If $\XX \to M$ is a spectral topos fibration, then we define \acategory over $M$ fibered in layered \categories
	\begin{equation*}
		\StrShape^{M}(\XX) \to M
	\end{equation*}
	as follows.
	An object of $\StrShape^{M}(\XX)$ is a pair $(m,\nu)$, where $m\in M$ and $\nu_{\ast} \colon \Space \to \XX_m$ is a point.
	A morphism
	\begin{equation*}
		(m,\nu) \to (n,\xi)
	\end{equation*}
	consists of morphism $f \colon m \to n$ of $M$ along with a natural transformation $\nu_{\ast} \to f_{\ast}\xi_{\ast}$.
	The \category $\StrShape^{M}(\XX)$ fibered in layered \categories admits a canonical fiberwise profinite structure;
	for each object $ m \in M $, the fiber \smash{$\StrShape^{M}(\XX)_m$} is the profinite stratified shape \smash{$\StrShape(\XX_m)$}.

	In the other direction, if $\mbfPi \to M$ is \acategory over $M$ fibered in profinite layered \categories, then let $X_0 \to M$ denote the cocartesian fibration in which the objects are pairs $(m,F)$ consisting of an object $m\in M$ and a functor $F \colon \mbfPi_m \to \Spacefin$, and a morphism
	\begin{equation*}
		(f, \phi) \colon (m, F) \to (n, G)
	\end{equation*}
	consists of a morphism $f \colon m \to n$ of $M$ and a natural transformation $\phi \colon f_!F \to G$.
	Then $(\mbfPitilde)\cohbdd$ is equivalent to the subcategory of $X_0$ whose objects are those pairs $(m,F)$ in which $F$ is continuous and whose morphisms are those pairs $(f, \phi)$ in which $\phi$ is continuous.
\end{nul}

\begin{cnstr}
	If $M$ is \acategory and $\YY$ is a bounded coherent topos, then the projection $\YY \times M \to M$ is a bounded coherent topos fibration.
	The assignment $ \goesto{\YY}{\YY \times M} $ defines a functor from the \category of bounded coherent topoi to the \category of bounded coherent topos fibrations over $M$.
	This functor admits a left adjoint, which we denote by $|\!-\!|_M$.
	At the level of \pretopoi, $(|\XX|_M)\cohbdd$ is equivalent to the \category of cocartesian sections of $\XXcohbdd \to M$, i.e., the limit of the corresponding functor $ \fromto{M}{\preTopbdd} $.
\end{cnstr}

Now we arrive at the main topos-theoretic result.
\begin{prp}\label{prop:maintopos}
	Let $M$ be \acategory, and let $\XX \to M$ be a spectral topos fibration.
	Then the \pretopos $(|\XX|_M)\cohbdd$ is equivalent to the \category of continuous functors
	\begin{equation*}
		F \colon \StrShape^{M}(\XX) \to \Spacefin^{\disc}
	\end{equation*}
	with the following properties.
	\begin{enumerate}[{\upshape (\ref*{prop:maintopos}.1)}]
	 	\item\label{prop:maintopos.1} The functor $F$ carries every cartesian edge to an equivalence.

	 	\item\label{prop:maintopos.2} The functor $F$ is \defn{uniformly truncated} in the following sense: there exists an $N\in\NNup$ such that for each object $(m,\nu) \in  \StrShape^{S}(\XX)$, the space $F(m,\nu)$ is $N$-truncated.
	\end{enumerate}
\end{prp}

\begin{proof}
	The \pretopos $(|\XX|_M)\cohbdd$ can be identified with the \category of cocartesian sections of the cocartesian fibration $\XXcohbdd \to M$.
	The description of \Cref{nul:descriptionoffiberedHochster} completes the proof.
\end{proof}

Please note that the condition \enumref{prop:maintopos}{2} is automatic if $M$ has only finitely many connected components (e.g., $ M = \DDelta $).

Finally, since the profinite stratified shape is a delocalization of the protruncated shape (\Cref{thm:protruncdeloc}) we deduce the following:

\begin{prp}\label{prop:protruncated}
 	Let $M$ be \acategory, and let $\XX \to M$ be a spectral topos fibration.
 	Then the protruncated shape of the \topos $|\XX|_M$ is equivalent to the protruncation of the classifying prospace of $ \StrShape^{M}(\XX)$.
\end{prp} 

\begin{cnstr}[{$ \GalDelta(X_{\ast}) $}]\label{cnstr:GalDelta}
	Let $X_{\ast}$ be a simplicial coherent scheme.
	Denote by $\GalDelta(X_{\ast})$\index[notation]{GalDelta@$\GalDelta(X_{\ast})$} the following $1$-category.
	The objects are pairs $(m,\nu)$ consisting of an object $ [m] \in \DDelta$ and a geometric point $\nu \to Y_m$.
	A morphism
	\begin{equation*}
		(m,\nu) \to (n,\xi)
	\end{equation*}
	of $\GalDelta(X_{\ast})$ consists of a morphism $\sigma \colon m \to n$ of $\DDelta$ and a specialization $\nu \leftsquigarrow \sigma^{\ast}(\xi)$.
	There is an obvious forgetful functor $\GalDelta(X_{\ast}) \to \DDelta$, which is a cartesian fibration.
	A morphism $(m,\nu) \to (n,\xi)$ is cartesian over $\sigma \colon m \to n$ in $\DDelta$ if and only if the specialization $\nu \leftsquigarrow \sigma^{\ast}(\xi)$ is an isomorphism.

	The $ 1 $-category $\GalDelta(X_{\ast})$ over $\DDelta$ is naturally fibered in profinite layered $ 1 $-categories.
	Moreover, the \category over $ \DDelta $ fibered in profinite layered \categories $ \StrShape^{\DDelta}(X_{\ast,\et}) $ associated to the spectral topos fibration
	\begin{equation*}
		X_{\ast,\et} \to \DDelta
	\end{equation*}
	is identified with $\GalDelta(X_{\ast})$.

	In this case, \Cref{prop:maintopos} implies that $(|X_{\ast,\et}|_{\DDelta})\cohbdd$ is equivalent to the \category of continuous functors \smash{$ \GalDelta(X_{\ast}) \to \Spacefin^{\disc} $} that carry cartesian edges to equivalences.
\end{cnstr}

\begin{exm}
	If $X_{\ast}$ is a simplicial coherent scheme, then classifying prospace of the fiberwise profinite category $\GalDelta(X_{\ast})$ is equivalent to the protruncation of the étale homotopy type of $X_{\ast}$ (\Cref{thm:mainAG}).
\end{exm}

Now let us use this formalism to extend the \textit{Exodromy Equivalence} of \Cref{mainthm:classifyconstr} to the context of simplicial schemes and thus stacks.

\begin{cnstr}
	Write $ \Aff $\index[notation]{Aff@$ \Aff $} for the $ 1 $-category of affine schemes.
	We employ \HTT{Corollary}{3.2.2.13} to construct \acategory $ \PShet $\index[notation]{PShet@$ \PShet $} and a cocartesian fibration
	\begin{equation*}
		\PShet \to \Aff^{\op}
	\end{equation*}
	in which:
	\begin{itemize}
		\item The objects of $\PShet$ are pairs $(S, F)$ consisting of an affine scheme $S$ and a presheaf $ F $ on the small étale site of $S$.

		\item A morphism $(S, F) \to (T, G)$ is a pair $(f, \phi)$ consisting of a morphism $f \colon T \to S$ and a morphism of presheaves $\phi \colon f^{-1}F \to G$ on the small étale site of $T$.
	\end{itemize}

	Define $\Shet \subset \PShet$\index[notation]{Shet@$ \Shet $} to be the full subcategory spanned by those pairs $(S, F)$ in which $F$ is a sheaf; then $\Shet \to \Aff^{\op}$ is a topos fibration.
	Define $\Conset \subset \Shet$\index[notation]{Conset@$ \Conset $} to be the further full subcategory spanned by those pairs $(S, F)$ in which $F$ is a constructible sheaf (\Cref{def:Pconstructible}); then $ \Conset \to \Aff^{\op}$ is a cocartesian fibration.
\end{cnstr}

\begin{dfn}
	Let $X \to \Aff$ be a stack, i.e., a right fibration that is classified by an accessible fpqc sheaf $\Aff^{\op} \to \Space$.
	A \defn{constructible sheaf} on $X$ is a cocartesian section
	\begin{equation*}
		F \colon X^{\op} \to \Conset
	\end{equation*}
	over $\Aff^{\op}$.
	We write $ \Conset(X) $ for the \category of constructible sheaves on $ X $.
\end{dfn}

\begin{wrn}
	This can only be expected to be a reasonable definition for coherent stacks.	
\end{wrn}

\begin{nul}
	Informally, a constructible sheaf $F$ on $X$ assigns to every affine scheme $S$ over $X$ a constructible sheaf $F_S$ on $ S $ and to every morphism $f \colon S \to T$ of affine schemes an equivalence $F_S \simeq f^{\ast}F_T$.
	In other words, the \category of constructible sheaves on $X$ is the limit of the diagram $X^{\op} \to \Cat_{\infty} $ given by the assignment $S \mapsto \Conset(S) = S_{\et}^{\cons} $.

	Since $X$ is not generally a $ \updelta_0 $-small category, it is not obvious that this limit exists in $\Cat_{\infty} $.
	However, if $X$ contains a $ \updelta_0 $-small limit-cofinal full subcategory $Y$, then the desired limit exists.
\end{nul}

We thus conclude:

\begin{prp}\label{prop:GaldeltaConstr}
	Let $p \colon X \to \Aff$ is a stack.
	If $X_{\ast}$ is a simplicial coherent scheme presenting $ X $, then there is an equivalence between the \category $\Conset(X)$ and the \category of continuous functors
	\begin{equation*}
		\GalDelta(X_{\ast}) \to \Spacefin
	\end{equation*}
	that carry cartesian edges to equivalences (cf. \Cref{ntn:Functs}).
\end{prp}

Recall that the protruncated étale homotopy type of a simplicial scheme $ X_{\ast} $ can be identified with the colimit in protruncated spaces of the simplicial object that carries $ [m] \in \DDelta$ to the protruncated étale shape of the fibers of the cartesian fibration \smash{$\GalDelta(X_{\ast}) \to \DDelta$} agree with the protruncated étale shape of the schemes $Y_m$, it follows from \Cref{prop:protruncated} that the protruncated shape of the total category $\GalDelta(X_{\ast}) $ is the colimit of this simplicial diagram.
In other words:

\begin{thm}
	Let $ X_{\ast} $ be a simplicial coherent scheme.
	The classifying protruncated space of $\GalDelta(X_{\ast})$ recovers the protruncated étale homotopy type of $X_{\ast}$.
\end{thm}

Combining this with \Cref{prop:GaldeltaConstr} we obtain:

\begin{cor}\label{cor:exodromyforArtinn-stacks}
	Let $ n \in \NNup $ and let $ X $ be an Artin $ n $-stack.
	If $X_{\ast}$ is a simplicial coherent scheme presenting $X$, then the localization of $\GalDelta(X_{\ast})$ at the cartesian edges classifies constructible sheaves on $ X $.
\end{cor}

\noindent \Cref{cor:exodromyforArtinn-stacks} speaks only of Artin $n$-stacks, but of course applies just as well to any coherent fpqc stack with a presentation by a simplicial coherent scheme.

\begin{exm}
	Let $ k $ be a ring, $ G $ be an affine $ k $-group, and $ X $ be a $k$-scheme with an action of $G$.
	Recall that the simplicial $k$-scheme $\Bar_{k,\ast}(X,G,k)$ whose $n$-simplices are $X \times_{\Spec k} G^{n}$ presents the quotient stack $X/G$.

	By \Cref{cor:exodromyforArtinn-stacks}, the category of $G$-equivariant constructible sheaves on $X$ is equivalent to the category of continuous functors
	\begin{equation*}
		\GalDelta(\Bar_{k,\ast}(X,G,k)) \to \Spacefin
	\end{equation*}
	that carry the cartesian edges to equivalences.
	If $R$ is a finite ring, then the derived category of $G$-equivariant constructible sheaves of $R$-modules on $X$ is equivalent to the category of continuous functors
	\begin{equation*}
		\GalDelta(\Bar_{k,\ast}(X,G,k)) \to \Perf(R)
	\end{equation*}
	that carry cartesian edges to equivalences.

	The objects of the category $\GalDelta(\Bar_{k,\ast}(X,G,k))$ can be thought of as tuples
	\begin{equation*}
		([m], \Omega, x_0, g_1, \dots, g_m)
	\end{equation*}
	where $ [m] \in \DDelta $, $\Omega$ is a separably closed field, and
	\begin{equation*}
		x_0 \colon \Spec \Omega \to X \andeq g_1,\ldots,g_m \colon \Spec \Omega \to G
	\end{equation*}
	are points with the property that $(x_0, g_1, \dots, g_m)$ is a geometric point of $X \times_{\Spec k} G^m$ such that that $\Omega$ is the separable closure of the residue field of the image of $(x_0, g_1, \dots, g_m)$ in the Zariski space of $ X \times_{\Spec k} G^m $.
\end{exm}

\newpage

\section[Perfectly reduced schemes \& reconstruction of absolute schemes]{Perfectly reduced schemes \& reconstruction of \\ absolute schemes}\label{sect:fullfaithfulness} 

We have shown that the étale \topos $X_{\et}$ of a coherent scheme $X$ can be reconstructed from the profinite \category $\Gal(X)$. 
Following Grothendieck's \textit{Brief an Faltings} \cite[(8)]{MR1483108}, we can ask to what extent $X$ itself can be recovered from $X_{\et}$. We first note that there are three easily-spotted obstacles to the \textit{conservativity} of the functor $\goesto{X}{X_{\et}}$.
\begin{enumerate}[(1)]
	\item One must restrict attention to schemes over a base with suitable finiteness conditions: for example, a nontrivial extension $ K \subset L $ of separably closed fields induces an equivalence of étale \topoi (since both are the terminal \topos).

	\item The base must be sufficiently small: over $\CCup$, for example, any two smooth proper curves of the same genus have equivalent étale \topoi.

	\item One must account for universal homeomorphisms: for example, the normalization of the cuspidal cubic induces an equivalence of étale \topoi.
	In fact, any universal homeomorphism induces an equivalence of étale \topoi; this is the \emph{invariance topologique} of the étale \topos \cites[Exposé IX, 4.10]{MR50:7129}[Exposé VIII, 1.1]{MR50:7131}.
\end{enumerate}

The first two points compel us to impose serious finiteness conditions on our schemes, and this last point compels us to consider the \category obtained from the $1$-category $ \Sch $\index[notation]{Sch@$ \Sch $} of coherent schemes by inverting universal homeomorphisms.
Fortunately, it is not necessary to do something excessively abstract: there is a $1$-categorical colocalization that performs this function.

\Cref{subsec:UHGalois} analyzes the effect of universal homeomorphisms on Galois categories.
\Cref{subsec:perfred} recalls how to characterize schemes that admit no nontrivial universal homeomorphisms.
\Cref{subsec:perfection} shows that the subcategory of such schemes can be obtained from the category of all schemes by formally inverting the universal homeomorphisms.
\Cref{subsec:Grothendieckconjecture} discusses Grothendieck's anabelian conjectures and proves \Cref{lede:reconstruction}; that the Galois category is a complete invariant of normal schemes over a finitely generated field of characteristic $ 0 $ (\Cref{nul:mainthmonGal}).
\Cref{subsec:examplecurves} illustrates our main theorem by making explicit how to reconstruct curves from a combination of stratified-homotopy-theoretic and Galois-theoretic data.


\subsection[Universal homeomorphisms and equivalences of Galois categories]{Universal homeomorphisms and equivalences of Galois \\ categories}\label{subsec:UHGalois}

Now we arrive at a sensitive question: under which circumstances does a morphism of schemes induce an equivalence of étale topoi or, equivalently, of Galois categories?
The well-known theorem here is Grothendieck's \emph{invariance topologique} of the étale topos \cite[Exposé VIII, 1.1]{MR50:7131}, which states that a universal homeomorphism induces an equivalence on étale topoi.
Let us reprove this result with the aid of Galois categories.

\begin{rec}[{\stacks{01S2}}]\label{rec:radicialUH}
	A morphism of schemes $ f \colon X \to Y $ is \defn{radicial}\index[terminology]{radicial}  if the underlying map of Zariski topological spaces is injective and for every $ x_0 \in X $, the field extension $ \upkappa(f(x_0))\subseteq\upkappa(x_0) $ is purely inseparable.
	Equivalently, $ f $ is radicial if and only if $ f $ is universally injective.

	A morphism of schemes is a \defn{universal homeomorphism}\index[terminology]{universal homeomorphism}  if and only if it is radicial, surjective, and universally closed. 
\end{rec}

\begin{prp}\label{prp:radicialintermsofGal} 
	Let $f\colon X\to Y$ be a morphism of coherent schemes.
	If $f$ is radicial, then every fiber of $\Gal(X)\to\Gal(Y)$ is either empty or a contractible groupoid.
	Conversely, if $f$ is of finite type, and if every fiber of $\Gal(X)\to\Gal(Y)$ is either empty or a contractible groupoid, then $f$ is radicial.
\end{prp}

\begin{proof}
	If $f$ is radicial, then the map $X^{\zar}\to Y^{\zar}$ is an injection.
	Moreover, for every point $x_0\in X^{\zar}$, the field extension $\upkappa(f(x_0))\subseteq\upkappa(x_0)$ is purely inseparable, so the map $\BG_{\upkappa(x_0)}\to \BG_{\upkappa(f(x_0))}$ on fibers is an equivalence.
	Hence for every geometric point $ y \to Y $ with image $y_0$, the fiber over $y$ is a contractible groupoid.

	Conversely, if $f$ is of finite type, and if every fiber of $\Gal(X)\to\Gal(Y)$ is either empty or a contractible groupoid, then the map $X^{\zar}\to Y^{\zar}$ is an injection.
	In particular, $ f $ is quasifinite.
	For every point $x_0\in X^{\zar}$, the fibers of the map $\BG_{\upkappa(x_0)}\to \BG_{\upkappa(f(x_0))}$ are all contractible, hence the map $\BG_{\upkappa(x_0)}\to \BG_{\upkappa(f(x_0))}$ is an equivalence.
	Now since the extnesion $\upkappa(f(x_0))\subseteq\upkappa(x_0)$ is finite, it is purely inseparable.
\end{proof}

\begin{exm} 
	The finite type hypothesis in the second half of \Cref{prp:radicialintermsofGal} is necessary: any nontrivial extension $ K \subset L $ of separably closed fields induces the identity on trivial Galois categories.
\end{exm}

\begin{cor}\label{cor:radicialsurjintermsofGal}
	Let $f\colon X\to Y$ be a morphism of coherent schemes.
	If $f$ is radicial and surjective, then every fiber of $\Gal(X)\to\Gal(Y)$ is a contractible groupoid.
	Conversely, if $f$ is of finite type, and if every fiber of $\Gal(X)\to\Gal(Y)$ is a contractible groupoid, then $f$ is radicial and surjective.
\end{cor}

Now we turn to describing integral morphisms in terms of fibrations of Galois categories.

\begin{rec} 
	A functor between \categories $f \colon C\to D $ is a \defn{right fibration}\index[terminology]{fibration!right}\index[terminology]{right fibration} if and only if, for every object $ c \in C$, the induced functor $C_{/c} \to D_{/f(c)}$ is an equivalence of \categories.
	Dually, $f$ is a \defn{left fibration}\index[terminology]{fibration!left}\index[terminology]{left fibration} if and only if $f^{\op}$ is a right fibration, so that for every object $ c \in C$, the induced functor $C_{c/}\to D_{f(c)/}$ is an equivalence of \categories.
\end{rec}

\begin{prp}\label{prp:integralintermsofGal} 
	Let $f\colon X\to Y$ be a morphism of coherent schemes. 
	If $f$ is an integral morphism, then $\Gal(X)\to\Gal(Y)$ is a right fibration.
\end{prp}

\begin{proof}
	Assume that $f$ is integral.
	Then for every geometric point $x\to X$, the induced morphism $X^{(x)}\to Y^{(f(x))}$ is also integral, and by Schröer's result \cite[Lemma 2.3]{MR3649361}, it is radicial as well. 
	Hence at the level of Zariski topological spaces, $X^{(x),\zar}\to Y^{(f(x)),\zar}$ is an inclusion of a closed subset; since the source and target are each irreducible, and the inclusion carries the generic point to the generic point, it is a homeomorphism. 
	(In fact, $X^{(x)}\to Y^{(f(x))}$ is a universal homeomorphism.) 
	Thus
	\begin{equation*}
		\Gal(X)_{/x}\simeq\Gal(X^{(x)})\simeq X^{(x),\zar}\to Y^{(f(x)),\zar}\simeq\Gal(Y^{(f(x))})\simeq\Gal(Y)_{/f(x)}
	\end{equation*}
	is an equivalence.
	That is, $\Gal(X)\to\Gal(Y)$ is a right fibration.
\end{proof}

An equivalence of categories is a right fibration with fibers contractible groupoids.
We thus deduce:

\begin{prp} 
	Let $f\colon X\to Y$ be a morphism of coherent schemes.
	If $f$ is a universal homeomorphism, then $\Gal(X)\to\Gal(Y)$ is an equivalence.
	In particular, since the étale \topoi of $X$ and $Y$ depend only on $\Gal(X)$ and $\Gal(Y)$, it follows that $f_{\ast} \colon X_{\et} \to Y_{\et} $ is an equivalence of \topoi.
\end{prp}


\subsection{Perfectly reduced schemes}\label{subsec:perfred}

The notion of a perfect scheme is usually defined only for $ \FFp $-schemes.
Here, we extend this notion to arbitrary reduced schemes in a way that restricts to the usual notion for $ \FFp $-schemes.

Just as a reduced scheme receives no nontrivial nilimmersions, a perfect scheme receives no nontrivial universal homeomorphisms.
This is in fact a local condition that can be expressed in very concrete terms:

\begin{lem}\label{lem:trigequiv}
	The following are equivalent for a coherent scheme $X$.
	\begin{enumerate}[{\upshape (\ref*{lem:trigequiv}.1)}]
		\item If $ f \colon \fromto{X'}{X}$ is a universal homeomorphism and $X'$ is reduced, then $ f $ is an isomorphism.

		\item Every universal homeomorphism $\fromto{X'}{X}$ admits a section.

		\item There exists an affine open covering $\{\Spec A_i\}_{i\in I}$ of $X$ such that for every $i\in I$, the following conditions hold:
		\begin{itemize}
			\item For all $a,b\in A_i$ such that $ a^2 = b^3 $, there exists a unique element $ c \in A_i$ such that $ a = c^3 $ and $ b = c^2 $.

			\item For each prime number $p$ and all $ a,b \in A_i $ such that $ a^p = p^p b$, there exists a unique element $ c \in A_i $ such that $ a = pc $ and $ b = c^p $.
		\end{itemize} 
	\end{enumerate}
\end{lem}

\noindent This is discussed in \stacks{0EUK}. See also \cites[1.4 and 1.7]{MR618807}[Appendix B]{MR2679038}[Theorem 1]{MR714467}.

\begin{dfn}\label{dfn:toprigid}
	We say that a coherent scheme $X$ is \defn{perfectly reduced}\index[terminology]{perfectly reduced}\index[terminology]{scheme!perfectly reduced} if $ X $ satisfies the equivalent conditions of \Cref{lem:trigequiv}.
	Denote by $\Schperf\subset\Sch$\index[notation]{Schperf@$ \Schperf $} the full subcategory spanned by the perfectly reduced schemes.

	We say that a coherent scheme $ X $ is \defn{seminormal}\index[terminology]{seminormal}\index[terminology]{scheme!seminormal} if and only if there exists an affine open covering $\{\Spec A_i\}_{i\in I}$ of $X$ such that for each $ i \in I $ and all $ a,b \in A_i $ such that $ a^2 = b^3 $, there exists a unique element $ c \in A_i $ such that $ a = c^3 $ and $ b = c^2 $.
\end{dfn}

\begin{rmk}
	Rydh has also studied perfectly reduced schemes under the name \defn{absolutely weakly normal}\index[terminology]{absolutely weakly normal}\index[terminology]{scheme!absolutely weakly normal} schemes \cite[Definition B.1]{MR2679038}.
\end{rmk}

\begin{exm}\label{exm:twnsch}
	A $ \QQup $-scheme is perfectly reduced if and only if it is seminormal.

	Let $p$ be a prime number.
	A reduced $ \FFp $-scheme is perfectly reduced if and only if it is perfect.
\end{exm}


\subsection{Perfection}\label{subsec:perfection}

We now show that $\Schperf$ is the result of inverting the universal homeomorphisms in $\Sch$.
More precisely, we show that the inclusion $\incto{\Schperf}{\Sch}$ admits a right adjoint $\goesto{X}{X_{\perf}}$ and the counit $\fromto{X_{\perf}}{X}$ is a universal homeomorphism. We first check that inverse limits of universal homeomorphisms are universal homeomorphisms. 

\begin{lem}\label{prp:cofiltlimuhuh}
	Let $X$ be a scheme.
	Let $ A $ be an inverse category, and $ W \colon\fromto{A}{\Sch_{/X}}$ a diagram of $X$-schemes.
	Assume that for each object $\alpha \in A$, the structure morphism $p_{\alpha}\colon\fromto{W_{\alpha}}{X}$ is a universal homeomorphism.
	Then the natural morphism
	\begin{equation*}
		p \colon \fromto{W' \coloneq \lim_{\alpha \in A^{\op}}W_{\alpha}}{X}
	\end{equation*}
	is a universal homeomorphism.
\end{lem}

\begin{proof}
	All the transition morphisms $\fromto{W_{\alpha}}{W_{\alpha'}}$ are universal homeomorphisms. 
	It follows from \cite[8.3.8(i)]{MR36:178} that $p$ is surjective. 
	For any field $k$, the diagram
	\begin{equation*}
		W(k)\colon\fromto{A^{\op}}{\Set}
	\end{equation*}
	is a diagram of injections, hence for each $ \alpha \in A^{\op}$, the map $\fromto{W'(k)}{W_{\alpha}(k)}$ is an injection.
	Thus $p$ is a universal injection. 
	It remains to show that $p$ is integral.
	Since $ W $ is a diagram of affine $ X $-schemes, it is enough to observe that the filtered colimit $\colim_{\alpha \in A}p_{\alpha,\ast} \Oup_{W_{\alpha}}$ is an integral $ \Oup_X $-algebra.
\end{proof}

\begin{prp} 
	The inclusion $\incto{\Schperf}{\Sch}$ admits a right adjoint, and the counit $\fromto{X_{\perf}}{X}$ is a universal homeomorphism.
\end{prp}

\begin{proof} 
	For any coherent scheme $X$, let $\UH_X\subset\Sch_{/X}$ be the full subcategory spanned by the universal homeomorphisms $p\colon\fromto{Y}{X}$.
	The full subcategory of $ \UH_X $ spanned by the finite universal homeomorphisms is limit-cofinal in $ \UH_X $.
	Hence the limit of $X$-schemes 
	\begin{equation*}
		X_{\perf} \coloneq \lim_{Y \in \UH_X} Y
	\end{equation*}
	exists and defines a universal homeomorphism $ \counit_{X} \colon\fromto{X_{\perf}}{X}$.
	Any universal homeomorphism $\fromto{Y}{X_{\perf}}$ admits a section, which proves that $X_{\perf}$ is perfect. 
	Moreover, if $Z$ is perfect, then for any morphism $f\colon\fromto{Z}{X}$, the pullback $\fromto{Z\cong Z\times_XX_{\perf}}{X_{\perf}}$ provides an inverse to the natural map
	\begin{equation*}
		\fromto{\Mor_{\Sch}(Z,X_{\perf})}{\Mor_{\Sch}(Z,X)} \period
	\end{equation*}
	This proves that $ \counit $ is a counit morphism exhibiting $ \Schperf $ as a colocalization of $ \Sch $.
\end{proof}

\begin{cor}
	The \category obtained from the $1$-category $\Sch$ by inverting universal homeomorphisms is equivalent to $\Schperf$.
\end{cor}

\begin{dfn}
	We call the right adjoint $\goesto{X}{X_{\perf}}$\index[notation]{Xperf@$ X_{\perf} $} the \defn{perfection functor}\index[terminology]{perfection}.
\end{dfn}

\begin{nul}[absolute weak normalization]
	Rydh presented an alternative description of perfection under the name \defn{absolute weak normalization}\index[terminology]{absolute weak normalization} \cite[Appendix B]{MR2679038}.
	Let $X$ be a reduced coherent scheme, and assume that $ X $ satisfies one of the following properties: 
	\begin{enumerate}[(1)]
		\item The set of irreducible components of $ X $ is finite. 
		In this case, write $\overline{X}$ for `the' \textit{absolute integral closure} of $X$ \cite[\S1]{MR0289501}.

		\item The scheme $ X $ is affine.
		In this case, write $\overline{X}$ for `the' \textit{total integral closure} of $X$ \cites{MR0224600,MR0257064}
	\end{enumerate}
	In each of these cases, one can show that $X_{\perf}$ is isomorphic to the weak normalization of $X$ (in the sense of Andreotti--Bombieri \cite[Teorema 2]{MR0266923}) under \smash{$\fromto{\overline{X}}{X}$}.
\end{nul}

\begin{exm}
	For reduced $ \QQup $-schemes, perfection agrees with seminormalization \stacks{0EUT}.
\end{exm}

\begin{exm}
	Let $ p $ be a prime number.
	If $ X $ is a reduced $ \FFp $-scheme then by \cite[Lemma 3.8]{MR3674218} we have a natural isomorphism
	\begin{equation*}
		X_{\perf} \isomorphic \lim \bigg(
		\begin{tikzcd}[sep=2.5em]
			\cdots \arrow[r, "\Frob_X"] & X \arrow[r, "\Frob_X"] & X 
		\end{tikzcd}\bigg) \comma
	\end{equation*}
	where $ \Frob_X $ is the absolute Frobenius.
\end{exm}

In preparation for discussing Grothendieck's Conjecture in the next section, the remainder of this section is concerned with studying properties of schemes up to universal homeomorphism.

\begin{dfn}
	Let $ X $ and $ Y $ be coherent schemes.
	A \defn{topological morphism}\index[terminology]{topological morphism} from $X$ to $Y$ is an morphism $f\colon X_{\perf}\to Y$.
	If $f$ induces an isomorphism $X_{\perf}\equivalence Y_{\!\perf}$, then we call $ f $ a \defn{topological equivalence}\index[terminology]{topological equivalence} from $X$ to $Y$.
\end{dfn}

\begin{dfn} 
	Let $P$ be a property of morphisms of schemes that is stable under base change and composition.
	We say that a morphism $f\colon X\to Y$ is \defn{topologically $P$}\index[terminology]{topologically P@topologically $ P $} if and only if $ f $ is topologically equivalent to a morphism of schemes $f'\colon X'\to Y'$ with property $P$.
\end{dfn}

\begin{nul} 
	Let $P$ be a property of morphisms of schemes that is stable under base change and composition. The class of topologically $P$ morphisms is the smallest class of morphisms $P^t$ that contains $P$ and satisfies the following condition: for any commutative diagram
	\begin{equation*}
		\begin{tikzcd}
			X \arrow[r, "f" above] \arrow[d, "\phi" left] & Y \arrow[d, "\psi" right]\\ 
			X' \arrow[r, "f'" below] & Y'
		\end{tikzcd}
	\end{equation*}
	in which $\phi$ and $\psi$ are universal homeomorphisms, the morphism $f$ lies in $P^t$ if and only if $f'$ does.

	A morphism $f\colon X\to Y$ of perfectly reduced schemes is topologically $P$ if and only if $ f $ factors as a universal homeomorphism $X\to X'$ followed by a morphism $X'\to Y$ with property $P$.
\end{nul}

\begin{exm}
	A morphism $f\colon X\to Y$ of perfectly reduced schemes is topologically radicial, surjective, universally closed, or integral if and only if $ f $ is radicial, surjective, universally closed, or integral (respectively).
\end{exm}

\begin{exm}
	A morphism $f\colon X\to Y$ of perfectly reduced schemes is topologically étale if and only if $ f $ is étale.
	Indeed, if $f'\colon X'\to Y$ is étale, then $X'$ is perfectly reduced \cite[Proposition B.6(ii)]{MR2679038}.
\end{exm}


\subsection[Grothendieck's Conjecture \& the proof of the Reconstruction Theorem]{Grothendieck's Conjecture \& the proof of the Reconstruction Theorem}\label{subsec:Grothendieckconjecture}

The étale fundamental group is an information-dense invariant, and Grothendieck's \textit{Anabelian Conjectures} are roughly an investigation of the extent to which the étale fundamental group is a \textit{complete} invariant for certain classes of schemes.
In dimension $0$, the classical Neukirch--Uchida Theorem \cites{MR0244211}{MR0258804}{MR0432593} ensures that two number fields are isomorphic if and only if their absolute Galois groups are isomorphic as profinite groups.
In dimension $1$, Tamagawa \cite{MR1478817} and Mochizuki \cite{MR1432110} show that dominant morphisms between smooth hyperbolic curves over suitable fields of characteristic zero can be detected at the level of fundamental groups.
Work of Pop \cite[Theorem 1]{MR1259367} shows that an isomorphism between two function fields over finitely generated fields can be detected at the level of Galois groups. 

If the étale fundamental group is information-dense, then the étale homotopy type must be even more so.
Indeed, Schmidt and Stix \cite[Theorem 1.2]{MR3549624} show that over a finitely generated field $k$ of characteristic $0$, if $ X $ and $Y$ are smooth, geometrically connected varieties that can be embedded as locally closed subschemes of a product of hyperbolic curves, then the map
\begin{equation*}
	\fromto{\Isom_k(X,Y)}{\Isom_{\BG_k}(\Shapeetprofin(X),\Shapeetprofin(Y))}
\end{equation*}
is a split injection with a natural retraction.
Here $\Isom_{\BG_k}$ denotes the set of homotopy classes of equivalences of profinite spaces over $\BG_k$.

In this section we discuss the relationship between the Galois category of a coherent scheme and Grothendieck's anabelian program.
We begin by isolating the class of schemes that appear in Grothendieck's Conjecture.

\begin{dfn}\label{dfn:absolute} 
	We call a scheme $ X $ \defn{absolute}\index[terminology]{scheme!absolute}\index[terminology]{absolute scheme} if $ X $ is perfectly reduced and the morphism $X\to \Spec \ZZup$ is topologically essentially of finite type. 
	Write $ \Sch_{\abs}\subset\Schperf $\index[notation]{Schabs@$ \Sch_{\abs} $} for the subcategory whose objects are absolute schemes and whose morphisms are of finite type.
\end{dfn}

In order to understand the extent to which the formation of the étale \topos of an absolute scheme is fully faithful, we need to isolate a basic property that morphisms of étale \topoi induced by morphisms of schemes have that general geometric morphisms do not have.
The things to notice is that Chevalley's Theorem ensures that any morphism of finite presentation between coherent schemes carries constructible sets to constructible sets.  

\begin{dfn}[admissible morphisms]
	Let $S$ and $T$ be spectral topological spaces.
	We say that a quasicompact continuous map $f\colon S\to T$ is \defn{admissible}\index[terminology]{admissible} if and only if the $ f $ sends constructible subset of $S$ to constructible subsets of $ T $.

	Accordingly, we say that a morphism $ \fromto{\Pi}{\Pi'} $ of profinite stratified spaces is \defn{admissible} if and only if the induced quasicompact continuous map of spectral topological spaces $ \fromto{\ho_0(\Pi)}{\ho_0(\Pi')} $ is admissible. 
	We write $\Strprofinadm \subset \Strprofin$ for the subcategory containing all objects with morphisms the admissible morphisms.

	Likewise, if $\XX$ and $\YY$ are bounded coherent \topoi, we say that a coherent geometric morphism $ f_{\ast} \colon \fromto{\XX}{\YY} $ is \defn{admissible}\index[terminology]{admissible} if and only if the induced quasicompact continuous map of spectral topological spaces $ \fromto{\Sup(\XX)}{\Sup(\YY)} $ is admissible (\Cref{ntn:0localicrefofcoherent}).
	We write
	\begin{equation*}
		\Topbcadm \subset \Topbc \index[notation]{Topbcadm@$ \Topbcadm $}
	\end{equation*}
	for the subcategory containing all objects with morphisms the admissible geometric morphisms.
\end{dfn}

\begin{nul}[admissible morphisms of Jacobson spaces]
	Recall that a topological space $ S $ is \defn{Jacobson} if every nonempty locally closed subset of $ S $ contains a closed point of $ S $ \stacks{005T}.
	If $S$ and $T$ are Jacobson spectral topological spaces, then a quasicompact continuous map $ f \colon \fromto{S}{T} $ is admissible if and only if $ f $ carries closed points of $ S $ to closed points of $ T $.
	Similarly, if $\Pi$ and $\Pi'$ are profinite stratified spaces such that $ \ho_0(\Pi) $ and $ \ho_0(\Pi) $ are Jacobson spectral topological spaces, then a morphism $ f \colon \fromto{\Pi}{\Pi'} $ is admissible if and only if $ f $ carries minimal objects to minimal objects.
\end{nul}

Here is the `tantalising conjecture' of Grothendieck in his letter to Faltings \cite[p. 7]{MR1483108}:

\begin{cnj}\label{cnj:main} 
	The functor
	\begin{equation*}
		(-)_{\et} \colon \fromto{\Sch_{\abs}}{(\Topbcadm)_{/(\Spec\ZZup)_{\et}}}
	\end{equation*}
	is fully faithful.
	In particular, if $X$ and $Y$ are absolute schemes, then any admissible geometric morphism $ X_{\et}\to Y_{\et}$ is induced by a finite type morphism $X\to Y$.
\end{cnj}

\Cref{cnj:main} implies the following stratified anabelian result:

\begin{cor}\label{thm:stratifiedanabelian}
	Assume \Cref{cnj:main}.
	Then the functor
	\begin{equation*}
		\Gal \colon \fromto{\Sch_{\abs}}{(\Strprofinadm)_{/\Gal(\Spec\ZZup)}}
	\end{equation*}
	is fully faithful.
	In particular, if $X$ and $Y$ are absolute schemes, then any admissible morphism $\Gal(X) \to \Gal(Y)$ is induced by a finite type morphism $X\to Y$.
\end{cor}

An early paper of Voevodsky \cite{MR1098621} provides a proof of \Cref{cnj:main} when restricted to normal absolute schemes of characteristic $0$.

\begin{thm}[{\cite[Theorem 3.1]{MR1098621}}]
	Let $k$ be a finitely generated field of characteristic $0$, and write $\Sch_{k}^{\norm}$ for the category of normal schemes of finite type over $k$.
	Then the functor
	\begin{equation*}
		(-)_{\et} \colon \fromto{\Sch_{k}^{\norm}}{(\Topbcadm)_{/(\Spec k)_{\et}}}
	\end{equation*}
	is fully faithful.
\end{thm}

\noindent Voevodsky also claims that, with some modifications, his proof will work when $k$ is a finitely generated field of positive characteristic and of transcendence degree $ \geq 1$.

Voevodsky's result combined with Conceptual Completeness (\Cref{thm:conceptualcompleteness}=\allowbreak\SAG{Theorem}{A.9.0.6}) show that a morphism $ f \colon \fromto{X}{Y} $ of reduced normal schemes of finite type over a finitely generated filed of characteristic $ 0 $ is an isomorphism if and only if $ f $ induces and equivalence on categories of points $ \equivto{\Pt(X_{\et})}{\Pt(Y_{\et})} $.
Combining our \categorical Hochster Duality Theorem with Voevodsky's Theorem and our identification of $\StrShapeet(X)$ with the topological category $\Gal(X) $ (\Cref{cnstr:idPiinfty1withGal}), we can upgrade this conservativity result to the following strong Reconstruction Theorem for these schemes:

\begin{thm}[Reconstruction]\label{nul:mainthmonGal}
	Let $k$ be a finitely generated field of characteristic $ 0 $ and let \smash{$ \kbar \supset k $} be an algebraic closure of $ k $.
	Then for any normal $k$-schemes $X$ and $Y$ of finite type, the natural map
	\begin{equation*}
		\Mor_k(X,Y) \to \Mor_{\BG_k}(\Gal(X),\Gal(Y))
	\end{equation*}
	identifies $ \Mor_k(X,Y) $ with the subgroupoid of continuous functors $ \fromto{\Gal(X)}{\Gal(Y)} $ that carry minimal objects to minimal objects.

	In particular, if $X$ and $Y$ are normal $k$-schemes of finite type, and $ \Gal(X) $ and $ \Gal(Y)$ are equivalent as topological categories over $\BG_k$, then $X$ and $Y$ are isomorphic as $ k $-schemes.
\end{thm}

Thus the category of normal $k$-schemes of finite type can be embedded as a subcategory of profinite categories with an action of $\Gk$, as asserted in \Cref{lede:reconstruction}.

\begin{nul}
	The data of the map \smash{$\Gal(X) \to \BG_k$} is the same as a continuous \smash{$ \Gk $} action on the fiber over the point of \smash{$ \BG_k $} specified by the algebraic closure \smash{$ \kbar \supset k $}.
	This fiber can be identified with the Galois category \smash{$ \Gal(X_{\kbar}) $}.
\end{nul}

\begin{exm}\label{exm:StratRiemannfgfield}
	Let $ k $ be a finitely generated field of characteristic $ 0 $ and choose a complex embedding $ \incto{\kbar}{\CCup} $.
	Then by Stratified Riemann Existence (\Cref{cor:stratRiemann}), a normal $k$-variety $X$ can be reconstructed from the profinite stratified space
	\begin{equation*}
		\StrShape(\widetilde{X_{\kbar}^{\an}};X_{\kbar}^{\zar})
	\end{equation*}
	with its $\Gk$-action.
\end{exm}


\subsection{Example: Curves}\label{subsec:examplecurves}

Let $ k $ be a finitely generated field of characteristic $ 0 $.
In this section, we illustrate how to use \Cref{nul:mainthmonGal} to reconstruct a connected, smooth, complete curve over $k$ from a combination of stratified-homotopy-theoretic and Galois-theoretic data.

\begin{ntn}
	Let $n \geq 2$ be an integer.
	Let $\fan_n$\index[notation]{zzzn@$ \fan_n $} be the poset with underlying set $ \{0,1,\dots,n-1,\infty\} $, where $0,1,\dots,n-1$ are pairwise incomparable, and for each element $i\in\{0,1,\dots,n-1\}$ we have $ i < \infty$.
	Let $p_n \colon \fan_{n+1} \to \fan_n$ be the map of posets defined by
	\begin{equation*}
		p_n(i) \coloneq \begin{cases}
			i\comma & i\in\{0,1,\dots,n-1\}  \\
			\infty\comma & i\in\{n,\infty\} \period
		\end{cases}
	\end{equation*}
	Write $ \fan $\index[notation]{zzz@$ \fan $} for the profinite poset definied by the inverse system
	\begin{equation*}
		\begin{tikzcd}[sep=1.5em]
			\cdots \arrow[r] & \fan_4 \arrow[r, "p_3"] & \fan_3 \arrow[r, "p_2"] & \fan_2 \period
		\end{tikzcd}
	\end{equation*}
	As a spectral topological space, $ \fan $ is isomorphic to the underlying Zariski topological space of any connected, normal curve.
\end{ntn}

For each integer $ g \geq 0 $ we construct a profinite stratified space $ \Chat_g $ that abstractly plays the role of a connected, smooth, complete curve of genus $ g $.
To do this, we first identify the décollage over $ \fan_n $ associated to the exit-path \category of a closed smooth surface of genus $ g $ with closed strata given by $ n $ marked points.

\begin{cnstr}[{the profinite stratified space \smash{$ \Chat_g $}}]
	Let $g\geq 0$ be an integer.
	Define a spatial décollage
	\begin{equation*}
		\upC_{g,n} \colon \sdop(\fan_n) \to \Space \index[notation]{Cgn@$ \upC_{g,n} $}
	\end{equation*}
	over $\fan_n$ as follows.
	For each $i\in \{0,1,\dots,n-1\}$, set $\upC_{g,n}\{i\} \coloneq \{i\}$, and let $\upC_{g,n}\{\infty\}$ be the classifying space of the free group on generators
	\begin{equation*}
		a_1,\ b_1,\ a_2,\ b_2,\ \dots,\ a_g,\ b_g,\ c_1,\ c_2,\ \dots,\ c_{n-1} \period
	\end{equation*}
	For each $i\in \{0,1,\dots,n-1\}$, let $\upC_{g,n}\{i<\infty\}$ be the classifying space of the free group on a single generator $\xi_i$.
	The morphisms $\upC_{g,n}\{i<\infty\} \to \upC_{g,n}\{\infty\}$ carry the generator $\xi_i$ to
	\begin{equation*}
		\begin{cases}
			\left( [a_1, b_1][a_2, b_2] \cdots [a_g, b_g] \right)( c_1 c_2 \dots c_{n-1})^{-1} \comma & i=0  \\
			c_i \comma & i \neq 0 \period
		\end{cases}
	\end{equation*}
	
	Define a morphism of spatial décollages $ \upC_{g,n+1} \to \upC_{g,n}$ over $p_n$ by carrying $a_j \mapsto a_j$, $b_j \mapsto b_j$, $c_i \mapsto c_i$ for $i\in \{0,1,\dots,n-1\}$, and $ c_n $ to the identity.
	We abuse notation and also write $ \upC_{g,n} $ for the $ \fan_{n} $-stratified space associated to the spatial décollage $ \upC_{g,n} $.
	With this notation, we have defined an inverse system of stratified spaces
	\begin{equation*}
		\cdots \to \upC_{g,4} \to \upC_{g,3} \to \upC_{g,2} \period
	\end{equation*}

	For each integer $ n \geq 2 $, let \smash{$ \Chat_{g,n} $}\index[notation]{Chatgn@$ \Chat_{g,n} $} be the profinite completion of the $ \fan_n $-stratified space \smash{$ \upC_{g,n} $}.
	We write \smash{$ \Chat_g $}\index[notation]{Chatg@$ \Chat_{g} $} for the profinite $ \fan $-stratified space defined by the inverse system of profinite stratified spaces
	\begin{equation*}
		\cdots \to \Chat_{g,4} \to \Chat_{g,3} \to \Chat_{g,2} \period
	\end{equation*}
\end{cnstr}

For the remainder of this section we fix a finitely generated field $k$ of characteristic $0$ and an algebraic closure \smash{$ \kbar \supset k $} of $ k $.
The following is immediate from \Cref{cor:stratRiemann}.

\begin{prp}
	Let $C$ be a connected, smooth, complete curve over $k$ of genus $g$.
	Then $\Gal(C_{\kbar})$ is equivalent to the profinite stratified space $\Chat_g$.
\end{prp}

\Cref{nul:mainthmonGal} says that the curve $C$ can be reconstructed from the profinite stratified space $\Gal(C_{\kbar}) \simeq \Chat_g$ with its action of $\Gk$.

To explain this point in more detail, let us make the following slightly tongue-in-cheek definition.

\begin{dfn}
	Let $ k $ be a field.
	An \defn{incorporeal field extension}\index[terminology]{incorporeal!field extension} of $k$ is a finite transitive $\Gk$-set.
\end{dfn}

Galois theory shows that the assignment $ E \mapsto \Gal(E \otimes_k \kbar) $ defines an equivalence from the category of finite extensions of $ k $ to the category of incorporeal field extensions of $k$.
We partially extend this to curves.

\begin{dfn}
	Let $ k $ be a field and $ g \geq 0 $ an integer.
	An \defn{incorporeal curve over $k$ of genus $g$}\index[terminology]{incorporeal!curve} is a continuous action $ \alpha $ of $\Gk$ on $\Chat_g$.

	Let $ (\Chat_{g_1},\alpha_1) $ and $ (\Chat_{g_2},\alpha_2) $ be incorporeal curves over $ k $.
	A \defn{$k$-morphism}
	\begin{equation*}
		(\Chat_{g_1},\alpha_1) \to (\Chat_{g_2},\alpha_2)
	\end{equation*}
	is a $\Gk$-equivariant continuous functor $ \Chat_{g_1} \to \Chat_{g_2} $.

	Let $S$ be an incorporeal field extension of $k$.
	An \defn{$S$-point} of an incorporeal curve $(\Chat_g,\alpha)$ over $k$ is a $\Gk$-equivariant functor $ S \to \Chat_g $.
\end{dfn}

Incorporeal curves are completely group-theoretic objects.
They amount to inverse families of free profinite groups along with actions of $\Gk$.

\begin{nul}
	\Cref{nul:mainthmonGal} implies that the assignment $C \mapsto \Gal(C_{\kbar})$ defines a fully faithful functor from connected, smooth, complete curves over $k$ to incorporeal curves over $k$.

	Additionally, it allows one to reconstruct the points of $C$ from the corresponding incorporeal curve.
	For any finite extension $ k' \supset k$, we have a natural bijection between the set of $ k' $-points of $C$ and the set of \smash{$\Gal(k' \otimes_k \kbar)$}-points of $\Gal(C_{\kbar})$.
\end{nul}


\newpage

\DeclareFieldFormat{labelnumberwidth}{#1}
\printbibliography[keyword=alph, heading=references]
\DeclareFieldFormat{labelnumberwidth}{{#1\adddot\midsentence}}
\printbibliography[heading=none, notkeyword=alph]


\newpage
\printindex[notation]

\newpage
\printindex[terminology]

\end{document}